\documentclass[12pt]{article}
\usepackage{fullpage, amsmath, amssymb, amsfonts, amsthm, stmaryrd, url, hyperref}
\usepackage[all]{xy}

\newtheorem{theorem}{Theorem}[subsection]
\newtheorem{lemma}[theorem]{Lemma}
\newtheorem{cor}[theorem]{Corollary}
\newtheorem{conj}[theorem]{Conjecture}
\newtheorem{prop}[theorem]{Proposition}
\theoremstyle{definition}
\newtheorem{defn}[theorem]{Definition}
\newtheorem{hypothesis}[theorem]{Hypothesis}
\newtheorem{example}[theorem]{Example}
\newtheorem{remark}[theorem]{Remark}
\newtheorem{convention}[theorem]{Convention}

\numberwithin{equation}{theorem}

\newcommand{\CC}{\mathbb{C}}
\newcommand{\Cp}{\mathbb{C}_p}

\newcommand{\FF}{\mathbb{F}}
\newcommand{\Fp}{\mathbb{F}_p}

\newcommand{\Qp}{\mathbb{Q}_p}
\newcommand{\PP}{\mathbb{P}}
\newcommand{\QQ}{\mathbb{Q}}
\newcommand{\RR}{\mathbb{R}}
\newcommand{\Zp}{\mathbb{Z}_p}
\newcommand{\ZZ}{\mathbb{Z}}
\newcommand{\bA}{\mathbf{A}}
\newcommand{\bB}{\mathbf{B}}
\newcommand{\bC}{\mathbf{C}}

\newcommand{\be}{\mathbf{e}}
\newcommand{\bv}{\mathbf{v}}
\newcommand{\bw}{\mathbf{w}}
\newcommand{\bx}{\mathbf{x}}
\newcommand{\calB}{\mathcal{B}}
\newcommand{\calC}{\mathcal{C}}
\newcommand{\calE}{\mathcal{E}}
\newcommand{\calF}{\mathcal{F}}
\newcommand{\calG}{\mathcal{G}}
\newcommand{\calH}{\mathcal{H}}
\newcommand{\calI}{\mathcal{I}}
\newcommand{\calJ}{\mathcal{J}}
\newcommand{\calL}{\mathcal{L}}
\newcommand{\calM}{\mathcal{M}}
\newcommand{\calO}{\mathcal{O}}
\newcommand{\calP}{\mathcal{P}}
\newcommand{\calR}{\mathcal{R}}
\newcommand{\gothg}{\mathfrak{g}}
\newcommand{\gothm}{\mathfrak{m}}
\newcommand{\gothp}{\mathfrak{p}}
\newcommand{\gotho}{\mathfrak{o}}
\newcommand{\gothq}{\mathfrak{q}}
\newcommand{\gothU}{\mathfrak{U}}
\newcommand{\gothV}{\mathfrak{V}}
\newcommand{\dual}{\vee}

\newcommand{\QpLoc}{\mathbb{Q}_p\text{-}\mathbf{Loc}}
\newcommand{\QpdLoc}{\mathbb{Q}_{p^d}\text{-}\mathbf{Loc}}
\newcommand{\ZpLoc}{\mathbb{Z}_{p}\text{-}\mathbf{Loc}}
\newcommand{\ZpILoc}{\mathbb{Z}_{p}\text{-}\mathbf{ILoc}}
\newcommand{\ZpdLoc}{\mathbb{Z}_{p^d}\text{-}\mathbf{Loc}}
\newcommand{\ZpdILoc}{\mathbb{Z}_{p^d}\text{-}\mathbf{ILoc}}

\DeclareMathOperator{\AdBan}{\mathbf{AdBan}}

\DeclareMathOperator{\bd}{bd}

\DeclareMathOperator{\coker}{coker}
\DeclareMathOperator{\cont}{cont}

\DeclareMathOperator{\et}{\acute{e}t}

\DeclareMathOperator{\Ext}{Ext}
\DeclareMathOperator{\FEt}{\mathbf{F\acute{E}t}}
\DeclareMathOperator{\fet}{f\acute{e}t}
\DeclareMathOperator{\FFC}{FF}

\DeclareMathOperator{\Fitt}{Fitt}
\DeclareMathOperator{\Frac}{Frac}
\DeclareMathOperator{\frep}{frep}
\DeclareMathOperator{\Gal}{Gal}
\DeclareMathOperator{\GL}{GL}

\DeclareMathOperator{\Hom}{Hom}
\DeclareMathOperator{\image}{image}
\DeclareMathOperator{\id}{id}

\DeclareMathOperator{\inte}{int}

\DeclareMathOperator{\Map}{Map}
\DeclareMathOperator{\Maxspec}{Maxspec}

\DeclareMathOperator{\Mor}{Mor}
\DeclareMathOperator{\Norm}{Norm}
\DeclareMathOperator{\op}{op}
\DeclareMathOperator{\perf}{perf}
\DeclareMathOperator{\Pic}{Pic}
\DeclareMathOperator{\pr}{pr}
\DeclareMathOperator{\proet}{pro\acute{e}t}

\DeclareMathOperator{\Proj}{Proj}
\DeclareMathOperator{\rank}{rank}

\DeclareMathOperator{\Spa}{Spa}
\DeclareMathOperator{\Spec}{Spec}
\DeclareMathOperator{\spect}{sp}

\DeclareMathOperator{\Spv}{Spv}

\DeclareMathOperator{\Trace}{Trace}

\DeclareMathOperator{\unr}{unr}

\title{Relative $p$-adic Hodge theory: Foundations}
\author{Kiran S. Kedlaya and Ruochuan Liu}
\date{May 2, 2015}

\begin{document}

\maketitle

\begin{abstract}
We describe a new approach to relative $p$-adic Hodge theory based on systematic use of Witt vector constructions and nonarchimedean analytic geometry in the style of both Berkovich and Huber.
We give a thorough development of $\varphi$-modules over a relative Robba ring
associated to a perfect Banach ring of characteristic $p$, including the relationship between
these objects and \'etale $\Zp$-local systems and $\Qp$-local systems on the algebraic and analytic
spaces associated to the base ring, and the relationship between (pro-)\'etale cohomology and
$\varphi$-cohomology. We also make a critical link to mixed characteristic
by exhibiting an equivalence of tensor categories between the finite \'etale algebras
over an arbitrary perfect Banach algebra over a nontrivially normed complete field of characteristic $p$
and the finite \'etale algebras over a corresponding Banach $\Qp$-algebra.
This recovers the homeomorphism between the absolute Galois groups of
$\Fp((\pi))$ and $\Qp(\mu_{p^\infty})$ given by the
field of norms construction of Fontaine and Wintenberger, as well as
generalizations considered by Andreatta, Brinon, Faltings, Gabber, Ramero, Scholl, and most recently Scholze.
Using Huber's formalism of adic spaces and Scholze's formalism of perfectoid spaces, we globalize the constructions to give several descriptions of the \'etale
local systems on analytic spaces over $p$-adic fields.
One of these descriptions uses a relative version of the Fargues-Fontaine curve.
\end{abstract}

\tableofcontents

\setcounter{section}{-1}
\section{Introduction}

After its formalization by Deligne \cite{deligne-hodge2},
the subject of \emph{Hodge theory} may be viewed as the study of
the interrelationship among different cohomology theories
associated to algebraic varieties over $\CC$, most notably singular
(Betti) cohomology and the cohomology of differential forms (de Rham
cohomology).
{}From work of Fontaine and others, there emerged a
parallel subject of \emph{$p$-adic Hodge theory} concerning the
interrelationship among different cohomology theories
associated to algebraic varieties over a finite extension $K$
of $\Qp$, most notably \'etale cohomology with coefficients in $\Qp$
and algebraic de Rham cohomology.

In ordinary Hodge theory, the relationship between Betti and de Rham
cohomologies is forged using the \emph{Riemann-Hilbert correspondence},
which relates topological data (local systems) to analytic data
(integrable connections).
In $p$-adic Hodge theory, one needs a similar correspondence relating
de Rham data to \'etale $\Qp$-local systems, which arise from
the \'etale cohomology functor on schemes of finite type over $K$.
However, in this case the local systems turn out to be far
more plentiful, so it is helpful to first build a correspondence relating
them to some sort of intermediate algebraic objects. This is achieved by the
theory of \emph{$(\varphi, \Gamma)$-modules},
which gives some Morita-type dualities relating
\'etale $\Qp$-local systems over $K$ (i.e., continuous
representations of the absolute Galois group $G_K$ on
finite-dimensional $\Qp$-vector spaces)
with modules over certain mildly noncommutative topological algebras.
The latter appear as topological monoid algebras over certain commutative
\emph{period rings}
for certain continuous operators (the eponymous $\varphi$ and $\Gamma$).

One of the key features of Hodge theory is that it provides information
not just about individual varieties, but also about families of varieties
through the mechanism of \emph{variations of Hodge structures}.
Only recently has much progress been made in developing any analogous
constructions in $p$-adic Hodge theory; part of the difficulty is that
there are two very different directions in which relative $p$-adic Hodge theory
can be developed. In the remainder of this introduction, we first give a bit
more background about $(\varphi, \Gamma)$-modules, and contrast the \emph{arithmetic}
and \emph{geometric} forms
of relative $p$-adic Hodge theory. We then describe the results of this paper in detail, indicate some points of contact with recent work of
Scholze \cite{scholze1, scholze2}, and describe some future goals.

\subsection{Artin-Schreier theory and \texorpdfstring{$(\varphi, \Gamma)$}{(phi, Gamma)}-modules}

For $K$ a perfect field of characteristic $p$, the discrete
representations of the
absolute Galois group $G_K$ of $K$ on finite dimensional $\Fp$-vector
spaces form a category equivalent to the category of
\emph{$\varphi$-modules} over $K$, i.e., finite-dimensional $K$-vector spaces
equipped with isomorphisms with their $\varphi$-pullbacks.
This amounts to
a nonabelian
generalization of the Artin-Schreier description of $(\ZZ/p\ZZ)$-extensions
of fields of characteristic $p$ \cite[Expos\'e~XXII, Proposition~1.1]{sga7-2}.

A related result is that the continuous representations of $G_K$
on finite free $\Zp$-modules
form a category equivalent to the category of
finite free $W(K)$-modules equipped with
isomorphisms with their $\varphi$-pullbacks.
Here $W(K)$ denotes the ring of Witt vectors over $K$, and $\varphi$ denotes
the unique lift to $W(K)$ of the absolute Frobenius on $K$.
One can further globalize this result to arbitrary smooth schemes
over $K$, in which the corresponding category becomes
a category of \emph{unit-root $F$-crystals}; see
\cite[Theorem~2.2]{crew-F} or \cite[Proposition~4.1.1]{katz-modular}.

Fontaine's theory of $(\varphi, \Gamma)$-modules
\cite{fontaine-phigamma} provides a way to extend such
results to mixed-characteristic local fields.
The observation underpinning the theory is that a sufficiently
wildly ramified extension of a mixed-characteristic local field
behaves Galois-theoretically just like a local field of positive
characteristic. For example, for $K_0$ a finite unramified
extension of $\Qp$ (or the completion of an infinite algebraic unramified
extension of $\Qp$) with residue field $k_0$, the fields $K_0(\mu_{p^\infty})$
and $k_0((T))$ have homeomorphic Galois groups. One can describe
representations of the absolute Galois group of a local field by
restricting to a suitably deeply ramified extension, applying an
Artin-Schreier construction, then adding appropriate descent data to
get back to the original group.
This assertion is formalized in the theory of \emph{fields of norms}
introduced by Fontaine and Wintenberger \cite{fontaine-wintenberger,
wintenberger}; some of the analysis depends on Sen's calculation of ramification numbers in
$p$-adic Lie extensions \cite{sen-lie}.
See \cite[Part~4]{brinon-conrad} for a detailed but readable exposition.

\subsection{Arithmetic vs. geometric}

As noted earlier, there are two different directions in which one
can develop relative forms of $p$-adic Hodge theory. We distinguish
these as \emph{arithmetic} and \emph{geometric}.

In arithmetic relative $p$-adic Hodge theory,
one still treats
continuous representations of the absolute Galois group of a finite
extension of $\Qp$. However, instead of taking representations
simply on vector spaces,
one allows finite locally free modules over affinoid algebras over $\Qp$.
Interest in the arithmetic theory arose originally from the consideration of
$p$-adic analytic families of automorphic forms and their associated
families of Galois representations; such families include Hida's
$p$-adic interpolation of ordinary cusp forms \cite{hida},
the eigencurve of Coleman--Mazur \cite{coleman-mazur},
and further generalizations.
Additional interest has come from the prospect of a $p$-adic local Langlands
correspondence which would be compatible with formation of analytic families
on both the Galois and automorphic sides. For the group $\GL_2(\Qp)$,
such a correspondence has recently emerged from the work of Breuil,
Colmez, Emerton, Pa\v{s}k\={u}nas, et al. (see for instance \cite{colmez-langlands})
and has led to important advances concerning modularity of Galois
representations, in the direction of the Fontaine-Mazur conjecture
\cite{kisin-fmc}, \cite{emerton-fmc}.

In the arithmetic setting, there is a
functor from Galois representations
to families of $(\varphi, \Gamma)$-modules, constructed by Berger and Colmez
\cite{berger-colmez}. However, this functor is not an equivalence of categories;
rather,
it can only be inverted locally \cite{dee}, \cite{kedlaya-liu}. It seems that
in this setting, one is forced to study families of
Galois representations in the context of the larger category of
families of $(\varphi, \Gamma)$-modules. For instance,
one sees this distinction
in the relative study of Colmez's \emph{trianguline} Galois representations
\cite{colmez-trianguline},
as in the work of Bella\"iche \cite{bellaiche} and Pottharst
\cite{pottharst2}.

By contrast, in geometric relative $p$-adic Hodge theory,
one continues to consider continuous representations
acting on finite-dimensional
$\Qp$-vector spaces. However, instead of the absolute Galois
group of a finite extension of $\Qp$,
one allows \'etale fundamental groups of affinoid spaces over finite
extensions of $\Qp$. The possibility of developing an analogue of
$(\varphi, \Gamma)$-module theory in this setting emerged from the work
of Faltings, particularly his \emph{almost purity theorem} \cite{faltings-purity1,
faltings-almost}, and prior to this paper had been carried out most thoroughly by
Andreatta and Brinon \cite{andreatta-field-of-norms, andreatta-brinon}.  A similar construction was described
by Scholl \cite{scholl}. (See also the exposition by Olsson \cite{olsson-almost}.)

\subsection{Analytic spaces associated to Banach algebras}

We now turn to the topics addressed by this particular paper.
The first substantial chunk of the paper concerns some geometric spaces associated to nonarchimedean commutative Banach rings, including the
\emph{Gel'fand spectrum} in the sense of Berkovich and the \emph{adic spectrum} considered by Huber. These constructions are somewhat more exotic than the spaces of maximal ideals occurring in Tate's theory of rigid analytic spaces, but Tate's construction starts to behave poorly when one considers Banach rings other than affinoid algebras, and breaks down completely if one considers nonnoetherian rings. Since such rings play a crucial role in our work, we are forced to use the alternate constructions. 

One important feature of our work is the dialogue between the Gel'fand and adic spectra. The latter is topos-theoretically complete, whereas the Gel'fand spectrum is the maximal Hausdorff quotient of the adic spectrum. While the Gel'fand spectrum is a natural dwelling place for most of our local arguments, we transfer back to adic spectra in order to globalize.

When we globalize (by glueing together adic spectra), we do not end up with the most general sort of spaces considered by Huber: we only encounter spaces which are \emph{analytic}, meaning that the residue fields associated to all points carry nontrivial valuations. By contrast, Huber's theory also includes spaces more closely related to ordinary schemes and formal schemes. For another, we are mostly interested only in spaces which have a certain finiteness property (that of being \emph{taut}) which roughly means they are approximated well by their maximal Hausdorff quotients. (For instance, Berkovich's \emph{strictly analytic spaces} can be promoted to taut adic spaces in such a way that the original spaces occur as the maximal Hausdorff quotients.)
However, one cannot hope to banish nontaut spaces entirely from $p$-adic Hodge theory:
for instance, they appear in the work of Hellmann
\cite{hellmann1, hellmann2} on the moduli spaces for Breuil-Kisin modules described by
Pappas and Rapoport \cite{pappas-rapoport}. The study of these spaces seems to
include features of both arithmetic and geometric relative $p$-adic Hodge theory,
which appears to render both Tate and Berkovich spaces insufficient.

\subsection{Perfectoid fields and algebras}

As noted earlier, one of the main techniques of $p$-adic Hodge theory is the relationship between the absolute Galois groups
of certain fields of mixed and positive characteristic, such as $\Qp(\mu_{p^\infty})$ and $\Fp((T))$. For relative
$p$-adic Hodge theory, it is necessary to extend this correspondence somewhat further.
However, instead of an approach dependent on ramification theory, we use a construction
based on analysis of Witt vectors.

Suppose first that $L$ is a perfect field of characteristic $p$ complete for a multiplicative norm,
with valuation ring $\gotho_L$. Let $W(\gotho_L)$ be the ring of $p$-typical Witt vectors over $\gotho_L$.
One can generate certain fields of characteristic $0$ by quotienting $W(\gotho_L)$ by certain principal ideals
and then inverting $p$. For instance, for $L$ the completed perfect closure of $\Fp((\overline{\pi}))$ and
\[
z = \sum_{i=0}^{p-1} [\overline{\pi}+1]^{i/p} \in W(\gotho_L),
\]
we may identify $W(\gotho_L)/(z)$ with the ring of integers of the completion of $\Qp(\mu_{p^\infty})$. The relevant condition
(that of being \emph{primitive of degree $1$} in the sense of Fargues
and Fontaine \cite{fargues-fontaine}) is
a Witt vector analogue of the property of an element of $\Zp\llbracket T \rrbracket$ being associated to a monic linear
polynomial whose constant term is not invertible in $\Zp$ (which allows use of the division algorithm to identify
the quotient by this element with $\Zp$). See Definition~\ref{D:primitive}.

Using this construction, we obtain a correspondence between certain complete fields of mixed characteristic
(which we call \emph{perfectoid fields}, following \cite{scholze1}) and perfect fields of characteristic $p$
together with appropriate principal ideals in the ring of Witt vectors over the valuation ring
(see Theorem~\ref{T:perfectoid field}). It is not immediate from the construction that the former category
is closed under formation of finite extensions, but this turns out to be true and not too difficult to check
(see Theorem~\ref{T:mixed lift field}). In particular, we recover the field of norms correspondence.

To extend this correspondence to more general Banach algebras, we exploit a relationship developed in
\cite{kedlaya-witt} between the Berkovich space of a perfect uniform Banach ring $R$ of characteristic $p$
and the Berkovich space of the ring $W(R)$. (For instance, any subspace of $\calM(R)$ has the same homotopy
type as its inverse image under $\mu$ \cite[Corollary~7.9]{kedlaya-witt}.)
This leads to a correspondence between
certain Banach algebras\footnote{The definition of perfectoid algebras can be made at several levels of generality; we work in more generality than in \cite{scholze1},
where perfectoid algebras must be defined over a perfectoid field, but less than in
\cite{fontaine-bourbaki} or \cite{gabber-ramero-arxiv}, where perfectoid algebras need not contain $1/p$.} in characteristic $0$
(which we call \emph{perfectoid algebras}, again following \cite{scholze1}) and perfect Banach algebras in characteristic $p$, which is compatible with formation of both rational subspaces and finite \'etale covers
(see Theorem~\ref{T:perfectoid ring} and Theorem~\ref{T:mixed lift ring}).
We also obtain a result in the style of Faltings's \emph{almost purity theorem} \cite{faltings-purity1, faltings-almost}
(see also \cite{gabber-ramero, gabber-ramero-arxiv}), which underlies the aforementioned
generalization of $(\varphi, \Gamma)$-modules introduced by
Andreatta and Brinon \cite{andreatta-field-of-norms, andreatta-brinon}.

\subsection{Robba rings and slope theory}

Another important technical device in usual $p$-adic Hodge theory is the classification of Frobenius-semilinear
transformation on modules over certain power series by \emph{slopes}, in rough analogy with the classification
of vector bundles on curves. This originated in work of the first author \cite{kedlaya-annals};
we extend this work here to the relative setting.
(The relevance of such results to $p$-adic Hodge theory largely factors through
the work of Berger \cite{berger-cst, berger-adm}, to which we will return in a later paper.)

In \cite{kedlaya-annals}, one starts with the \emph{Robba ring} of germs of analytic functions on open annuli
of outer radius 1 over a $p$-adic field, and then passes to a certain ``algebraic closure'' thereof. The
latter can be constructed from the ring of Witt vectors over the completed algebraic closure of a power series
field. One is thus led naturally to consider similar constructions starting from the ring of Witt vectors
over a general analytic field; the analogues of the results of \cite{kedlaya-annals} were worked out by the first
author in \cite{kedlaya-revisited}.

Using the previously described work largely as a black box, we are able to introduce analogues of Robba rings
starting from the ring of Witt vectors of a perfect uniform $\Fp$-algebra (and obtain some weak analogues
of the theorems of Tate and Kiehl), and to study slopes of Frobenius modules
thereof. We obtain semicontinuity of the slope polygon as a function on the spectrum of the base ring
(Theorem~\ref{T:open locus}) as well as a slope filtration theorem when this polygon is constant
(Theorem~\ref{T:split polygon at point}). (Similar results in the arithmetic relative setting have recently
been obtained by the second author \cite{liu}.)

We also obtain a description of Frobenius modules over relative Robba rings in the style of Fargues and Fontaine, using vector bundles
over a certain scheme (Theorem~\ref{T:vector bundles}).
When the base ring is an analytic field, the scheme in question
is connected, regular, separated, and noetherian of dimension 1; it might thus be considered to be a
\emph{complete absolute curve}. (The adjective \emph{absolute} means that the curve cannot be seen
as having relative dimension 1 over a point; it is a scheme over $\Qp$, but not of finite type.) However, for more general base rings, the resulting scheme is not even noetherian.
In any case, one obtains a $p$-adic picture with strong resemblance to the correspondence between stable vector bundles on compact Riemann surfaces and irreducible unitary fundamental group
representations, as constructed by Narasimhan and Seshadri \cite{narasimhan-seshadri}.

\subsection{\texorpdfstring{$\varphi$}{phi}-modules and local systems}

By combining the preceding results,
we obtain a link between \'etale local systems
and $\varphi$-modules (and cohomology thereof), in what amounts to a broad nonabelian generalization of Artin-Schreier
theory as well as a generalization of the field of norms correspondence. This link is most naturally described on the category of \emph{perfectoid spaces}, obtained by glueing the adic spectra of perfectoid Banach algebras.

To describe \'etale local systems and their cohomology on more general spaces, including rigid and Berkovich analytic spaces, one must combine the previous theory with a descent construction from some local perfectoid covers of the space. For foundational purposes, an especially convenient mechanism for this is provided by the \emph{pro-\'etale topology} proposed by Scholze \cite{scholze2}.
In this framework, the rings of $p$-adic periods (such as the extended Robba ring) become sheaves for the pro-\'etale topology, and the analogue of a $(\varphi, \Gamma)$-module is a sheaf of $\varphi$-modules over a ring of period sheaves. Note that there is no explicit analogue of $\Gamma$; its role is instead played by the sheaf axiom for the pro-\'etale topology.

In this language, we obtain $\varphi$-module-theoretic descriptions of \'etale $\Zp$-local systems (Theorem~\ref{T:proetale equivalence1 global}) and their pro-\'etale cohomology (Theorem~\ref{T:proetale cohomology1}), as well as $\Qp$-local systems (Theorem~\ref{T:proetale equivalence2 global}) and their pro-\'etale cohomology
(Theorem~\ref{T:proetale cohomology2b}). We also see that when the base space is reduced to a point, our categories of $\varphi$-modules are equivalent to the corresponding categories of $(\varphi, \Gamma)$-modules arising in classical $p$-adic Hodge theory
(\S\ref{subsec:compare classical}).

\subsection{Contact with the work of Scholze}

After preparing the initial version of this paper, we discovered that some closely related work had been carried out by Peter Scholze, which ultimately has appeared in the papers \cite{scholze1, scholze2}.
The ensuing rapid dissemination of Scholze's work has had the benefit of providing an additional entry point into the circle of ideas underlying our work. However, it also necessitates a discussion of the extent to which the two bodies of work interact and overlap, which we now provide (amplifying the brief discussion appearing in the introduction to \cite{scholze1}).

We begin with some historical remarks.
The genesis of our work lies in the first author's paper
\cite{kedlaya-witt}, the first version of which appeared on arXiv in April 2010. There one first finds the homeomorphism of topological spaces which now underlies the perfectoid correspondence. This homeomorphism is again described in the first author's 2010 ICM lecture \cite{kedlaya-icm}, together with some preliminary discussion of how it could be used to construct tautological local systems on Rapoport-Zink period spaces. In late 2010, we began to prepare the present paper so as to provide foundations for carrying out the program advanced in \cite{kedlaya-icm}. In early 2011, we learned that Scholze had used similar ideas for a totally different (and rather spectacular) purpose: to resolve some new cases of the weight-monodromy conjecture in $\ell$-adic \'etale cohomology (which subsequently appeared as \cite{scholze1}).
Draft versions of papers were exchanged in both directions, which in Scholze's case included his work on the de Rham-\'etale comparison isomorphism (which subsequently appeared as \cite{scholze2}).

Based on this exchange, we elected to adopt some key formal ideas from Scholze's work but to retain independent derivations of all of our results, even in cases of overlap. This decision was dictated in part by some minor but nonnegligible foundational differences between the two works. We now describe some points of agreement and disagreement with \cite{scholze1, scholze2}.
\begin{itemize}
\item
As noted previously, we make heavy use of the Gel'fand spectrum associated to a Banach ring, translating into the language of adic spectra for glueing constructions; by contrast, Scholze works exclusively with adic spectra.
\item
The term \emph{perfectoid} is adopted from Scholze; we had not initially assigned a word to this concept.
\item
The \emph{perfectoid correspondence} for fields
(Theorem~\ref{T:perfectoid field})
is described in \cite{scholze1} without reference to Witt vectors,
as is the compatibility with finite extensions
(Theorem~\ref{T:mixed lift field});
this necessitates the use of some almost ring theory in the form of \cite[Theorem~5.3.27]{gabber-ramero}.
Besides being necessary for the construction of period sheaves, we
find the arguments using Witt vectors somewhat more transparent;
see \cite{kedlaya-new-phigamma}
for a demonstration of this point in the form of an exposition of the
 classical theory of $(\varphi, \Gamma)$-modules.
\item
Scholze considers only perfectoid algebras over perfectoid fields, whereas our definition of a perfectoid algebra
(Definition~\ref{D:perfectoid Banach})
does not include this restriction. This is partly because our description of the perfectoid correspondence for algebras
(Theorem~\ref{T:perfectoid ring}) includes enough extra data (in terms of Witt vectors) to enable lifting from characteristic $p$ back to characteristic $0$ without reference to an underlying field. This more general approach has also been adopted by Gabber and Ramero \cite{gabber-ramero-arxiv}.
\item
The compatibility of the perfectoid correspondence with finite \'etale covers (Theorem~\ref{T:mixed lift ring}) is established by reduction to the field case, much as in \cite{scholze1}, but again the use of Witt vectors takes the place of almost ring theory. As a result, instead of proving almost purity in the course of proving Theorem~\ref{T:mixed lift ring}, we deduce it as a corollary
(Theorem~\ref{T:almost purity}).
\item
Our study of perfectoid algebras includes some results with no analogues in \cite{scholze1}, including compatibility with strict morphisms (Proposition~\ref{P:perfectoid uniform strict})
and with morphisms of dense image (Theorem~\ref{T:perfectoid dense image}). The latter implies that the uniform completed tensor product of two perfectoid algebras (over a not necessarily perfectoid base) is again perfectoid
(Corollary~\ref{C:perfectoid tensor product over any base}).
\item
Our definition of a \emph{perfectoid space} is essentially that of Scholze, except that we do not insist on working over a perfectoid base field. Our derivations of the basic properties are as in \cite{scholze1} except that we make internal references in place of the equivalent cross-references to \cite{scholze1}.
\item
For passing from $\varphi$-modules to $(\varphi, \Gamma)$-modules,
we adopt Scholze's \emph{pro-\'etale topology} essentially unchanged from \cite{scholze2}. Note that this is not simply the pro-category associated to the \'etale topology, but requires an extra Mittag-Leffler condition; by contrast, in the category of schemes, Bhatt and Scholze \cite{bhatt-scholze} have shown that the pro-category associated to the \'etale topology can be used with similar effect.
\end{itemize}

\subsection{Further goals}

To conclude, we indicate some questions we intend to address in subsequent work.
\begin{itemize}
\item
The usual Robba ring in $p$-adic Hodge theory is \emph{imperfect}, that is, its Frobenius endomorphism is not surjective. By contrast, our analogue of the Robba ring corresponding to a perfectoid algebra
has bijective Frobenius. For certain purposes (e.g., approximations of the $p$-adic Langlands correspondence), it is desirable to
descend the theory of $(\varphi, \Gamma)$-modules from the perfect Robba ring to an imperfect version when possible. This
occurs in classical $p$-adic Hodge theory via the theorem of
Cherbonnier and Colmez \cite{cherbonnier-colmez}; some cases in relative $p$-adic Hodge theory are covered by the generalization of
Cherbonnier--Colmez given by Andreatta and Brinon
\cite{andreatta-brinon}. However, it should be possible to embed
these constructions into a more general framework; some first steps
in this direction are taken in \cite{kedlaya-multi}.

\item
Another goal is to construct certain ``tautological'' local systems on period spaces
of $p$-adic Hodge structures (filtered $(\varphi, N)$-modules). Such period spaces arise,
for instance, in the work of Rapoport and Zink on period mappings on deformation spaces
of $p$-divisible groups \cite{rapoport-zink}. The construction we have in mind is
outlined in \cite{kedlaya-icm};
it is likely to be greatly assisted by the work of Scholze and Weinstein on the (perfectoid) moduli space of $p$-divisible groups
\cite{scholze-weinstein}. The non-minuscule case is particularly intriguing, as there the natural parameter space is no longer a period domain (as suggested somewhat cavalierly in \cite{kedlaya-icm}); rather, it is a presently hypothetical space analogous to Hartl's moduli space of Hodge-Pink structures in the equal characteristic case \cite{hartl-equi}.

\item
Yet another goal is the integration of Scholze's approach to the de Rham--\'etale comparison isomorphism into the framework of relative $(\varphi, \Gamma)$-modules, and its extension to the crystalline and semistable cases. For instance, if $X \to Y$ is a proper morphism of adic spaces, there should be higher direct image functors from de Rham relative $(\varphi, \Gamma)$-modules on $X$ to the corresponding objects on $Y$. A closely related issue is to study the analogue of Fontaine's de Rham, crystalline, and semistable conditions and in particular to generalize the ``de Rham implies potentially semistable'' theorem (originally established by Berger using the Andr\'e--Kedlaya--Mebkhout monodromy theorem for $p$-adic differential equations).
\end{itemize}

\subsection*{Acknowledgments}
Thanks to Fabrizio Andreatta, Max Bender,
Vladimir Berkovich, Brian Conrad, Chris Davis, Laurent Fargues, Jean-Marc Fontaine, Ofer Gabber,
Arthur Ogus, Peter Scholze,
Michael Temkin, and Liang Xiao for helpful discussions.
Thanks also to Scholze for providing early versions of his papers \cite{scholze1, scholze2}
and to the anonymous referees for highly instructive feedback.
Kedlaya was supported by NSF CAREER grant DMS-0545904, NSF grant DMS-1101343, DARPA grant HR0011-09-1-0048, MIT (NEC Fund,
Cecil and Ida Green professorship), and UC San Diego
(Stefan E. Warschawski professorship).
Liu was partially supported by IAS under NSF grant DMS-0635607 and the Recruitment Program of Global Experts of China.
Additionally, both authors were supported by NSF grant DMS-0932078 while in residence at MSRI during fall 2014.

\section{Algebro-geometric preliminaries}

Before proceeding to analytic geometry,
we start with some background facts from algebraic geometry.

\setcounter{theorem}{0}
\begin{hypothesis}
Throughout this paper, fix a prime number $p$.
\end{hypothesis}

\begin{convention}
When we refer to a \emph{tensor category}, we will always assume it is equipped not just with the usual
monoidal category structure, but also with a \emph{rank function} into some abelian group.
Equivalences of tensor categories (or for short \emph{tensor equivalences}) will be assumed to respect rank.
\end{convention}

\subsection{Finite, flat, and projective modules}

\begin{convention}
Throughout this paper, all rings are assumed to be commutative and unital unless otherwise indicated.
\end{convention}

\begin{defn} \label{D:projective}
Let $M$ be a module over a ring $R$.
We say that $M$ is \emph{pointwise free} if
$M \otimes_R R_{\gothp}$ is a free module over $R_{\gothp}$
for each maximal ideal $\gothp$ of $R$
(and hence for each prime ideal).
The term \emph{locally free} is sometimes used for this, but it is better to make
this term match its usual meaning in sheaf theory
by saying that $M$ is \emph{locally free}
if there exist $f_1,\dots,f_n \in R$ generating the unit ideal such that
$M \otimes_R R[f_i^{-1}]$ is a free module over $R[f_i^{-1}]$ for $i=1,\dots,n$.

We say that $M$ is \emph{projective} if it is a direct summand of a free module.
The following conditions are equivalent \cite[\S II.5.2, Th\'eor\`eme~1]{bourbaki-ac}.
\begin{enumerate}
\item[(a)]
$M$ is finitely generated and projective.
\item[(b)]
$M$ is a direct summand of a finite free module.
\item[(c)]
$M$ is finitely presented and pointwise free.
\item[(d)]
$M$ is finitely generated and pointwise free of locally constant rank. (The \emph{rank} of $M$
at $\gothp \in \Spec(R)$ is defined as $\dim_{R_\gothp/\gothp R_\gothp} (M_\gothp/\gothp M_\gothp)$.)
\item[(e)]
$M$ is finitely generated and locally free.
\end{enumerate}
For $R$ reduced, it is enough to check that $M$ is finitely generated of locally constant rank; see
\cite[Exercise~20.13]{eisenbud}. For an analogous argument
for Banach rings, see 
Proposition~\ref{P:finite generation2}.

For $M$ finitely presented, we may define the \emph{Fitting ideals}
$\Fitt_i(M)$ as in \cite[\S 20.2]{eisenbud};
these are finitely generated ideals of $R$ satisfying
$\Fitt_0(M) \subseteq \Fitt_1(M) \subseteq \cdots$
and $\Fitt_i(M) = R$ for $i$ sufficiently large. The construction commutes with base change: for any ring homomorphism $R \to S$, we have $\Fitt_i(M \otimes_R S) = \Fitt_i(M) S$
\cite[Corollary~20.5]{eisenbud}.
The module $M$ is finite projective of constant rank $n$ if and only if $\Fitt_i(M) = 0$ for $i=0,\dots,n-1$ and $\Fitt_n(M) = R$ \cite[Proposition~20.8]{eisenbud}. 
\end{defn}

\begin{defn} \label{D:faithfully flat}
Let $M$ be a module over a ring $R$. We say $M$ is \emph{faithfully flat} if
$M$ is flat and $M \otimes_R N \neq 0$ for every nonzero $R$-module $N$. Since tensor
products commute with direct limits, $M$ is faithfully flat if and only if
it satisfies the following conditions.
\begin{enumerate}
\item[(a)]
For any injective homomorphism $N \to P$ of finite $R$-modules, $M \otimes_R N \to M \otimes_R P$
is injective.
\item[(b)]
For any nonzero finite $R$-module $N$, $M \otimes_R N$ is nonzero.
\end{enumerate}
For other characterizations, see \cite[\S I.3.1, Proposition~1]{bourbaki-ac}.
\end{defn}

\begin{lemma} \label{L:faithfully flat by maximal ideals}
A flat ring homomorphism $R \to S$ is faithfully flat if and only if
every maximal ideal of $R$ is the contraction of a maximal ideal of $S$.
\end{lemma}
\begin{proof}
See \cite[\S I.3.5, Proposition~9]{bourbaki-ac}.
\end{proof}

\begin{remark} \label{R:finite type}
Recall that for any ring $R$, quasicoherent sheaves on $\Spec(R)$
correspond to $R$-modules via the global sections functor.
Under this correspondence, the property of a sheaf being finitely generated,
finitely presented, or finite projective implies the corresponding property for its module of global sections \cite[Tags~01PB, 01PC, 05JM]{stacks-project}.
\end{remark}

\subsection{Comparing \'etale algebras}

We will expend a great deal of effort comparing finite \'etale algebras over different rings.
A key case is given by base change from a ring to a quotient ring, in which the henselian
property plays a key role.
(A rather good explanation of this material can be obtained from \cite{gabber-ramero} by specializing from
almost ring theory to ordinary ring theory.)

\begin{defn} \label{D:finite etale category}
As in \cite[D\'efinition~17.3.1]{ega4-4}, we say a morphism of schemes is \emph{\'etale}
if it is locally of finite presentation and formally \'etale.
(The latter condition is essentially a unique infinitesimal lifting property;
see \cite[D\'efinition~17.1.1]{ega4-4}.)
A morphism of rings is \'etale if the corresponding morphism of affine schemes is \'etale.

For any ring $R$, let $\FEt(R)$ denote the tensor category of finite \'etale algebras
over the ring $R$, with morphisms being arbitrary morphisms of $R$-algebras.
Such morphisms are themselves finite and \'etale (e.g., by \cite[Proposition~17.3.4]{ega4-4}).
Any $S \in \FEt(R)$ is finite as an $R$-module
and finitely presented as an $R$-algebra, and hence finitely presented as an $R$-module
by \cite[Proposition~1.4.7]{ega4-1}. Also, $S$ is flat over $R$ by
\cite[Th\'eor\`eme~17.6.1]{ega4-4}. By the criteria described in Definition~\ref{D:projective},
$S$ is finite projective as an $R$-module.
Conversely, an $R$-algebra $S$ which is finite projective
as an $R$-module is finite \'etale if and only if the $R$-module homomorphism $S \to \Hom_R(S,R)$
taking $x$ to $y \mapsto \Trace_{S/R}(xy)$ is an isomorphism. (Namely, since \'etaleness
is an open condition \cite[Remarques~17.3.2(iii)]{ega4-4}, this reduces to the case where $R$ is a field,
which is straightforward.)

For short, we will describe a
finite \'etale $R$-algebra which is faithfully flat over $R$
(or equivalently of positive rank everywhere over $R$, by Lemma~\ref{L:faithfully flat by maximal ideals} and the going-up theorem) as
a \emph{faithfully finite \'etale} $R$-algebra.
\end{defn}

\begin{defn}
Let $R$ be a ring and let $U$ be an element of $\FEt(R)$. Then there exists an idempotent element
$e_{U/R} \in U \otimes_R U$ mapping to 1 via the multiplication map $\mu: U \otimes_R U \to U$
and killing the kernel of $\mu$; see for instance \cite[Proposition~3.1.4]{gabber-ramero}.
For $V$ another $R$-algebra and
$e \in U \otimes_R V$ an idempotent element, let $\Gamma(e): V \to e(U \otimes_R V)$ be the morphism
of $R$-algebras sending $x \in V$ to $e(1 \otimes x)$.
In case $\Gamma(e)$ is an isomorphism, we obtain a morphism $\psi_e: U \to V$ of $R$-algebras by
applying
the natural map $\Delta(e): U \to e(U \otimes_R V)$ sending $x$ to $e(x \otimes 1)$
followed by $\Gamma(e)^{-1}$.
Conversely, for $\psi: U \to V$ a morphism of $R$-algebras,
the idempotent $e_\psi = (1 \otimes \psi)(e_{U/R}) \in U \otimes_R V$ has the property
that $\Gamma(e_\psi)$ is an isomorphism (see \cite[Proposition~5.2.19]{gabber-ramero}).
\end{defn}

\begin{lemma} \label{L:idempotents to morphisms}
For $R$ a ring, $U \in \FEt(R)$, and $V$ an $R$-algebra, the function
$\psi \mapsto e_\psi$ defines a bijection from
the set of $R$-algebra morphisms from $U$ to $V$ to
the set of idempotent elements $e \in U \otimes_R V$ for which $\Gamma(e)$ is an isomorphism.
The inverse map is $e \mapsto \psi_e$.
\end{lemma}
\begin{proof}
See \cite[Lemma~5.2.20]{gabber-ramero}.
\end{proof}

\begin{lemma} \label{L:unique idempotent lifting}
Let $R$ be a ring and let $I$ be an ideal contained in the Jacobson radical of $R$.
\begin{enumerate}
\item[(a)]
No two distinct idempotents of $R$ are congruent modulo $I$.
\item[(b)]
For any integral extension $S$ of $R$,
the ideal $IS$ is contained in the Jacobson radical of $S$.
\end{enumerate}
\end{lemma}
\begin{proof}
For (a), let $e,e'\in R$ are idempotents with $e-e' \in I$.
Then $(e+e'-1)^2 = 1 - (e-e')^2 \in 1+I$, so $e+e'-1$ is a unit in $R$. Since $(e-e')(e+e'-1) = 0$,
we have $e=e'$.

For (b), note that the set of $y \in S$ which are roots of monic polynomials over $R$ whose nonleading coefficients belong to $I$
is an ideal. Consequently, if $x \in 1 + IS$, it is a root of a monic polynomial $P \in R[T]$  for
which $P(T-1)$ has all nonleading coefficients in $I$. It follows that the constant coefficient
$P_0$ of $P$ belongs to $\pm 1 + I$ and so is a unit; since $P_0 = P_0 - P(x)$
is divisible by $x$ (being the evaluation at $x$ of the polynomial $P_0-P$), $x$ is also a unit.
\end{proof}

\begin{prop} \label{P:lift idempotents}
The following statements hold.
\begin{enumerate}
\item[(a)]
Let $R \to R'$ be a homomorphism of rings such that for any invertible module $M$ over $R$, every element of $M$ which generates $M \otimes_R R'$ also generates $M$. 
(For example, this holds if $\Pic(R) \to \Pic(R')$ is injective and every element of $R$ which maps to a unit in $R'$ is itself a unit.)
Suppose that for each $S \in \FEt(R)$, every idempotent element of
$S \otimes_R R'$ is the image of a unique idempotent element of $S$. Then the base change functor
$\FEt(R) \to \FEt(R')$ is rank-preserving and fully faithful.
\item[(b)]
Let $R$ be a ring, and let $I$ be an ideal contained in the Jacobson radical of $R$.
Suppose that for each $S \in \FEt(R)$, every idempotent element of $S/IS$ lifts
to $S$. Then the base change functor $\FEt(R) \to \FEt(R/I)$ is rank-preserving and fully faithful.
\end{enumerate}
\end{prop}
Note that in case (a), if $R \to R'$ is injective, then $S$ injects into $S \otimes_R R'$
because $S$ is locally free as an $R$-module (see Definition~\ref{D:finite etale category}),
so the condition simply becomes that every idempotent element of
$S \otimes_R R'$ belongs to $S$.
\begin{proof}
The rank-preserving property is evident in both cases.
To check full faithfulness in case (a), note that
the hypothesis on the homomorphism $R \to R'$ implies that a map between finite projective
$R$-modules is an isomorphism if and only if its base extension to $R'$ is an isomorphism
(because the isomorphism condition amounts to invertibility of the determinant).
Consequently,
for $U,V \in \FEt(R)$ and $e \in U \otimes_R V$ an idempotent element,
the map $\Gamma(e): V \to e(U \otimes_R V)$ is an isomorphism if and only
if its base extension to $R'$ is an isomorphism.
Lemma~\ref{L:idempotents to morphisms} then implies that
every morphism $U \otimes_R R' \to V \otimes_R R'$ of $R'$-algebras descends to a morphism $U \to V$
of $R$-algebras, as desired.

To check full faithfulness in case (b), note that for $U,V \in \FEt(R)$ and
$\overline{e}\in (U/IU) \otimes_{R/IR} (V/IV)$ an idempotent element,
by Lemma~\ref{L:unique idempotent lifting} there is at most one
idempotent $e \in U \otimes_R V$ lifting $\overline{e}$.
Moreover, because $I$ is contained in the Jacobson radical of $R$,
for any invertible module $M$ over $R$, every element of $M$ which generates $M/I$ also generates $M$ by Nakayama's lemma.
We may thus apply (a) to deduce the claim.
\end{proof}

\begin{defn} \label{D:henselian}
A pair $(R,I)$ consisting of a ring $R$ and an ideal $I \subseteq R$
is said to be \emph{henselian} if the following conditions hold.
\begin{enumerate}
\item[(a)]
The ideal $I$ is contained in the Jacobson radical of
$R$.
\item[(b)]
For any monic $f \in R[T]$,
any factorization $\overline{f} = \overline{g} \overline{h}$
in $(R/I)[T]$ with $\overline{g}, \overline{h}$ monic and coprime lifts to
a factorization $f = gh$ in $R[T]$.
\end{enumerate}
For example, if $R$ is $I$-adically complete, then $(R,I)$ is henselian
by the usual proof of Hensel's lemma.
A local ring $R$ with maximal ideal $\gothm$ is \emph{henselian} if
the pair $(R,\gothm)$ is henselian.
\end{defn}

\begin{remark} \label{R:henselian criteria}
There are a number of equivalent formulations of the definition
of a henselian pair.
For instance, let $(R,I)$ be a pair consisting of a ring $R$ and an ideal $I$
contained in the Jacobson radical of $R$.
By \cite[Theorem~5.11]{greco-tams},
$(R,I)$ is henselian if and only if
every monic polynomial $f  = \sum_i f_i T^i \in R[T]$
with $f_0 \in I, f_1 \in R^\times$ has a root in $I$.
(In other words, it suffices to check the lifting condition for
$\overline{g} = T$.)
See \cite[Expos\'e XI, \S 2]{raynaud-henselian} for some other
formulations.
\end{remark}

\begin{theorem} \label{T:henselian}
Let $(R,I)$ be a henselian pair. Then the base change functor
$\FEt(R) \to \FEt(R/I)$ is a tensor equivalence.
\end{theorem}
\begin{proof}
See  \cite[Satz~4.4.7]{kurke}, \cite[Satz~4.5.1]{kurke-etc}, or
\cite{gruson}.
See also \cite[Theorem~5.5.7]{gabber-ramero} for a more general assertion
in the context of almost ring theory.
\end{proof}

\begin{remark} \label{R:fet direct limit}
Let $\{R_i\}_{i \in I}$ be a direct system in the category of rings.
We may then define the direct 2-limit $\varinjlim_i \FEt(R_i)$.
For $R = \varinjlim_i R_i$, there is a natural functor $\varinjlim_i \FEt(R_i) \to \FEt(R)$ given by base extension
to $R$. This functor is fully faithful by \cite[Proposition~17.7.8(ii)]{ega4-4}
(since affine schemes are quasicompact).
To see that it is essentially surjective, start with $S \in \FEt(R)$.
Since $S$ is finitely presented as an $R$-algebra, by \cite[Lemme~1.8.4.2]{ega4-1}
it has the form $S_i \otimes_{R_i} R$ for some $i \in I$ and some finitely presented
$R_i$-algebra $S_i$. By \cite[Proposition~17.7.8(ii)]{ega4-4} again, there exists
$j \geq i$ such that $S_j = S_i \otimes_{R_i} R_j$ is finite \'etale over $R_j$.
We conclude that $\varinjlim_i \FEt(R_i) \to \FEt(R)$ is a tensor equivalence.
\end{remark}

\subsection{Descent formalism}
\label{subsec:descent formalism}

We will make frequent use of faithfully flat descent for modules, as well as variations thereof (e.g.,
for Banach rings).
It is convenient to frame this sort of argument in standard abstract descent formalism, since this
language can also be used to discuss glueing of modules
(see Example~\ref{exa:quasicoherent}). We set up in terms of cofibred categories rather than fibred
categories, as appropriate for studying modules over rings rather than sheaves over schemes; in the latter context we will use sheaf-theoretic language instead.

\begin{defn}
Let $F: \calF \to \calC$ be a covariant functor between categories.
For $X$ an object (resp.\ $f$ a morphism) in $\calC$,
let $F^{-1}(X)$ (resp.\ $F^{-1}(f)$) denote the class of objects
(resp.\ morphisms) in $\calC$ carried to $X$ (resp.\ $f$) via $F$.

For $f: X \to Y$ a morphism in $\calC$ and $E \in F^{-1}(X)$,
a \emph{pushforward} of $E$ along $f$ is a morphism $\tilde{f}: E \to f_* E \in F^{-1}(f)$
such that any $g \in F^{-1}(f)$ with source $E$ factors uniquely through $\tilde{f}$.
(We sometimes call the target $f_* E$ a pushforward of $E$ as well, understanding that
it comes equipped with a fixed morphism from $E$.)
We say $\calF$ is a \emph{cofibred category} over $\calC$, or that
$F: \calF \to \calC$ defines a cofibred category, if
pushforwards always exist and the composition of two pushforwards (when defined)
is always a pushforward.
\end{defn}

\begin{defn}
Let $\calC$ be a category in which pushouts exist.
Let $F: \calF \to \calC$ be a functor defining a cofibred category.
Let $f: X\to Y$ be a morphism in $\calC$. Let $\pi_1, \pi_2: Y \to Y \sqcup_X Y$
be the coprojection maps. Let $\pi_{12}, \pi_{13}, \pi_{23}: Y \sqcup_X Y \to
Y \sqcup_X Y \sqcup_X Y$ be the coprojections such that $\pi_{ij}$ carries the first and second
factors of the source into the $i$-th and
$j$-th factors in the triple coproduct (in that order).
A \emph{descent datum} in $\calF$ along $f$ consists of an object
$M \in F^{-1}(Y)$ and an isomorphism
$\iota: \pi_{1*} M \to \pi_{2*} M$ between some choices of pushforwards,
satisfying the following \emph{cocycle condition}.
Let $M_1, M_2, M_3$ be some pushforwards of $M$ along the three coprojections
$Y \to Y \sqcup_X Y \sqcup_X Y$.
Then $\iota$ induces a map
$\iota_{ij}: M_i \to M_j$ via $\pi_{ij}$;
the condition is that $\iota_{23} \circ \iota_{12} = \iota_{13}$.

For example, any object $N\in F^{-1}(X)$
induces a descent datum by taking $M$ to be a
pushforward of $N$ along $f$ and taking $\iota$ to be the map identifying
$\pi_{1*} M$ and $\pi_{2*} M$ with a single pushforward of $M$ along
$X \to Y \sqcup_X Y$.
Any such descent datum is said to be \emph{effective}.
We say that $f$ is an \emph{effective descent morphism} for $\calF$ if the following
conditions hold.
\begin{enumerate}
\item[(a)]
Every descent datum along $f$ is effective.
\item[(b)]
For any $M, N \in \calF$ with $F(M) = F(N)$, the morphisms $M \to N$ in $\calF$
lifting the identity morphism are in bijection with morphisms between the corresponding
descent data. (We leave the definition of the latter to the reader.)
\end{enumerate}
\end{defn}

\begin{example} \label{exa:quasicoherent}
Let $\calC$ be the category of rings.
Let $\calF$ be the category of modules over rings,
with morphisms defined as follows: for $R_1, R_2 \in \calC$ and $M_i \in \calF$
a module over $R_i$, morphisms $M_1 \to M_2$ consist of pairs $(f,g)$
with $f: R_1 \to R_2$ a morphism in $\calC$ and $g: f_* M_1 \to M_2$ a morphism
of modules over $R_2$. Let $F: \calF \to \calC$ be the functor taking each module
to its underlying ring; this functor defines a cofibred category with pushforwards
defined as expected.

Let $R \to R_1, \dots, R \to R_n$ be ring homomorphisms corresponding to open
immersions of schemes which cover $\Spec R$, and put
$S = R_1 \oplus \cdots \oplus R_n$.
Then $R \to S$ is an effective descent morphism for $\calF$;
this is another way of stating the standard fact that
any quasicoherent sheaf on an affine scheme is represented uniquely by a module over the ring
of global sections \cite[Th\'eor\`eme~1.4.1]{ega1}. This fact is generalized
by Theorem~\ref{T:descent modules}.
\end{example}

\begin{theorem} \label{T:descent modules}
Any faithfully flat morphism of rings is an effective descent morphism
for the category of modules over rings (Example~\ref{exa:quasicoherent}).
\end{theorem}
\begin{proof}
See \cite[Expos\'e~VIII, Th\'eor\`eme~1.1]{sga1}.
\end{proof}

\begin{theorem} \label{T:descent finite locally free}
For $f: R \to S$ a faithfully flat morphism of rings,
an $R$-module $U$ is finite (resp.\ finite projective)
if and only if $f^* U = U \otimes_R S$ is a finite (resp.\ finite projective) $S$-module.
An $R$-algebra $U$ is finite \'etale
if and only if $f^* U $ is a finite \'etale $S$-algebra.
\end{theorem}
\begin{proof}
For the first assertion, see \cite[Expos\'e~VIII, Proposition~1.10]{sga1}.
For the second assertion, see \cite[Expos\'e~IX, Proposition~4.1]{sga1}.
\end{proof}

For a morphism of rings which is faithful but not flat (e.g., a typical adic completion of a nonnoetherian ring),
it is difficult to carry out descent except for modules
which are themselves flat. Here is a useful example due to Beauville and Laszlo \cite{beauville-laszlo}.
(Note that even the noetherian case of this result is not an immediate corollary of faithfully flat descent, because we do not specify a descent datum on $\widehat{R}$ itself; see \cite[\S 2]{artin}.)
\begin{prop} \label{P:reduced descent}
Let $R$ be a ring. Suppose that $t \in R$ is not a zero divisor and that $R$ is $t$-adically separated.
Let $\widehat{R}$ be the $t$-adic completion of $R$.
\begin{enumerate}
\item[(a)]
For any flat $R$-module $M$, the sequence
\[
0 \to M \to (M \otimes_R R[t^{-1}]) \oplus (M \otimes_R \widehat{R}) \to M \otimes_R \widehat{R}[t^{-1}] \to 0,
\]
in which the last nontrivial arrow is the difference between the two base extension maps, is exact.
\item[(b)]
Let $M_1$ be a finite projective module over $R[t^{-1}]$, let
$M_2$ be a finite projective module over $\widehat{R}$,
and let $\psi_{12}: M_1 \otimes_{R[t^{-1}]} \widehat{R}[t^{-1}]
\cong M_2 \otimes_{\widehat{R}} \widehat{R}[t^{-1}]$
be an isomorphism of $\widehat{R}[t^{-1}]$-modules.
Then there exist a finite projective $R$-module $M$, an isomorphism $\psi_1: M \otimes_R R[t^{-1}] \cong M_1$, and an isomorphism
$\psi_2: M \otimes_R \widehat{R} \cong M_2$ such that
$\psi_{12} \circ \psi_1 = \psi_2$; moreover, these data are unique up to unique isomorphism.
In particular, the morphism $R \to R[t^{-1}] \oplus \widehat{R}$ is an effective descent morphism for
the category of finite projective modules over rings.
\end{enumerate}
\end{prop}

In order to carry out analogous arguments in other contexts (as in Proposition~\ref{P:glue projective}),
it is helpful to introduce some formalism.
We will see later how to recover the Beauville-Laszlo theorem in this framework
(Remark~\ref{R:new Beauville-Laszlo}).

\begin{defn} \label{D:glueing datum}
Let
\[
\xymatrix{
R \ar[r] \ar[d] & R_1 \ar[d] \\
R_2 \ar[r] & R_{12}
}
\]
be a commuting diagram of ring homomorphisms such that the sequence
\begin{equation} \label{eq:glueing datum sequence}
0 \to R \to R_1 \oplus R_2 \to R_{12} \to 0
\end{equation}
of $R$-modules,
in which the last nontrivial arrow is the difference between the given homomorphisms, is exact.
By a \emph{glueing datum} over this diagram, we will mean a datum consisting of
modules $M_1, M_2, M_{12}$
over $R_1, R_2, R_{12}$, respectively, equipped with isomorphisms $\psi_1: M_1 \otimes_{R_1} R_{12} \cong M_{12}$,
$\psi_2: M_2 \otimes_{R_2} R_{12} \cong M_{12}$.
We say such a glueing datum is \emph{finite} or \emph{finite projective}
if the modules are finite or finite projective over their corresponding rings.

When considering a glueing datum, it is natural to consider the kernel $M$ of the map
$\psi_1 - \psi_2: M_1 \oplus M_2 \to M_{12}$. There are natural maps $M \to M_1$, $M \to M_2$ of $R$-modules,
which by adjunction correspond to maps $M \otimes_R R_1 \to M_1$, $M \otimes_R R_2 \to M_2$.
\end{defn}

\begin{lemma} \label{L:Kiehl generic1}
Consider a finite glueing datum for which $M \otimes_R R_1 \to M_1$ is surjective.
Then we have the following.
\begin{enumerate}
\item[(a)]
The map $\psi_1 - \psi_2: M_1 \oplus M_2 \to M_{12}$ is surjective.
\item[(b)]
The map $M \otimes_R R_2 \to M_2$ is also surjective.
\item[(c)]
There exists a finitely generated $R$-submodule $M_0$ of $M$ such that
for $i=1,2$, $M_0 \otimes_R R_i \to M_i$ is surjective.
\end{enumerate}
\end{lemma}
\begin{proof}
The surjection $M \otimes_R R_1 \to M_1$ induces a surjection $M \otimes_R R_{12} \to M_{12}$,
and hence a surjection $M \otimes_R (R_1 \oplus R_2) \to M_{12}$. Since this map factors through
$\psi_1 - \psi_2$, the latter is surjective.  This yields (a).

For each $\bv \in M_2$, $\psi_2(\bv)$ lifts to $M \otimes_R (R_1 \oplus R_2)$;
we can thus find $\bw_{i}$ in the image of $M \otimes_R R_i \to M_i$ such that $\psi_1(\bw_{1}) - \psi_2(\bw_{2}) = \psi_2(\bv)$.
Put $\bv' = (\bw_{1}, \bv + \bw_{2}) \in M_1 \oplus M_2$; note that $\bv' \in M$ by construction. Consequently, the image of $M \otimes_R R_2 \to M_2$ contains both $\bw_2$ and $\bv + \bw_2$, and hence also $\bv$. This yields (b), from which (c) is immediate since $M_i$ is a finite $R_i$-module.
\end{proof}

\begin{lemma} \label{L:Kiehl generic2}
Suppose that for every finite projective glueing datum,
the map $M \otimes_R R_1 \to M_1$ is surjective.
\begin{enumerate}
\item[(a)]
For any finite projective glueing datum,
$M$ is a finitely presented $R$-module and $M \otimes_R R_1 \to M_1$, $M \otimes_R R_2 \to M_2$ are bijective.
\item[(b)]
Suppose in addition that the image of 
$\Spec(R_1 \oplus R_2) \to \Spec(R)$ contains $\Maxspec(R)$. 
Then with notation as in (a), $M$ is a finite projective $R$-module.
\end{enumerate}
\end{lemma}
\begin{proof}
Choose $M_0$ as in Lemma~\ref{L:Kiehl generic1}(c).
Choose a surjection $F \to M_0$ of $R$-modules
with $F$ finite free, and put $F_1 = F \otimes_R R_1$,
$F_2 = F \otimes_R R_2$, $F_{12} = F \otimes_R R_{12}$,
$N = \ker(F \to M)$, $N_1 = \ker(F_1 \to M_1)$, $N_2 = \ker(F_2 \to M_2)$, $N_{12} = \ker(F_{12} \to M_{12})$.
{}From Lemma~\ref{L:Kiehl generic1}, we have a commutative diagram
\begin{equation} \label{eq:kiehl2}
\xymatrix{
& 0 \ar[d] & 0 \ar[d] & 0 \ar[d] & \\
0 \ar[r] & N \ar[r] \ar[d] & N_1 \oplus N_2 \ar[r] \ar[d] & N_{12} \ar@{-->}[r] \ar[d] & 0 \\
0 \ar[r] & F \ar[r] \ar[d] & F_1 \oplus F_2 \ar[r] \ar[d] & F_{12} \ar[r]\ar[d] & 0 \\
0 \ar[r] & M \ar[r] \ar@{-->}[d] & M_1 \oplus M_2 \ar[r] \ar[d] & M_{12} \ar[d] \ar[r] & 0 \\
 & 0 & 0 & 0 &
}
\end{equation}
with exact rows and columns, excluding the dashed arrows.
Since $M_i$ is projective, the exact sequence
\[
0 \to N_i \to F_i \to M_i \to 0
\]
splits, so
\[
0 \to N_i \otimes_{R_i} R_{12} \to F_{12} \to M_{12} \to 0
\]
is again exact. Thus $N_i$ is finite projective over $R_i$ and admits an isomorphism
$N_i \otimes_{R_i} R_{12} \cong N_{12}$.
By Lemma~\ref{L:Kiehl generic1} again, the
dashed horizontal arrow in \eqref{eq:kiehl2} is surjective. By diagram chasing,
the dashed vertical arrow in \eqref{eq:kiehl2} is also surjective; that is, we may add the dashed arrows
to \eqref{eq:kiehl2} while preserving exactness of the rows and columns. In particular,
$M$ is a finitely generated $R$-module; we may repeat the argument with $M$ replaced by $N$ to deduce that $M$ is
finitely presented.

For $i=1,2$, we obtain a commutative diagram
\[
\xymatrix{
& N \otimes_R R_i \ar[r] \ar[d] & F_i \ar[r] \ar@{=}[d] & M \otimes_R R_i \ar[r] \ar[d] & 0 \\
0 \ar[r] & N_i \ar[r] & F_i \ar[r]  & M_i \ar[r] & 0
}
\]
with exact rows: the first row is derived from
the left column of \eqref{eq:kiehl2} by tensoring over $R$ with $R_i$,
while the second row is derived from the middle column of \eqref{eq:kiehl2},
Since the left vertical arrow is surjective (Lemma~\ref{L:Kiehl generic1} once more),
by the five lemma, the right vertical arrow is injective.
We thus conclude that the map $M \otimes_R R_i \to M_i$, which was previously shown
(Lemma~\ref{L:Kiehl generic1}) to be surjective, is in fact a bijection.
This yields (a).

For $n$ a nonnegative integer and $i \in \{1,2,12\}$, let $U_{n,i}$ be the closed-open subset of $\Spec(R_i)$ on which $M_i$ has rank $n$. This set is the nonzero locus of some idempotent $e_{n,i} \in R_{n,i}$. Since $M_{12} \cong M_1 \otimes_{R_1} R_{12} \cong M_2 \otimes_{R_2} R_{12}$, $U_{n,12}$ can be characterized as the inverse image of either $U_{n,1}$ or $U_{n,2}$; this means that the images of $e_1$ and $e_2$ in $R_{12}$ are both equal to $e_{12}$. It follows that $e = e_1 \oplus e_2$ is an idempotent in $R$ mapping to $e_i$ in $R_i$; its nonzero locus is an open subset $U_n$
of $\Spec(R)$ whose inverse image in $\Spec(R_i)$ is $U_{n,i}$. This means that to
prove that $M$ is projective, we may reduce to the case where $M_1$ and $M_2$ are finite projective of some constant rank $n$.

Since $M$ is finitely presented, we may define its Fitting ideals $\Fitt_i(M)$ as in Definition~\ref{D:projective}.
Since $M_1 \oplus M_2$ is finite projective over $R_1 \oplus R_2$ of constant rank $n$,
$\Fitt_i(M)(R_1 \oplus R_2) = \Fitt_i(M_1 \oplus M_2)$ equals 0 for $i=0,\dots,n-1$ and
$R_1 \oplus R_2$ for $i=n$.
Since $R \to R_1 \oplus R_2$ is injective, this immediately implies that
$\Fitt_i(M) = 0$ for $i=0,\dots,n-1$.

Now assume that the image of $\Spec(R_1 \oplus R_2) \to \Spec(R)$ contains $\Maxspec(R)$. 
Then for each $\gothp \in \Maxspec(R)$, $M$ must have rank $n$ at $\gothp$ by comparison with some point in $\Spec(R_1 \oplus R_2)$, so $\Fitt_n(M)_\gothp = \Fitt_n(M_\gothp) = R_\gothp$. It follows that the inclusion $\Fitt_n(M) \subseteq R$ is an equality, yielding (b).
\end{proof}

\begin{cor} \label{C:fet square}
Suppose that the hypotheses of Lemma~\ref{L:Kiehl generic2}(b) are satisfied.
Then the natural functor
\[
\FEt(R) \to \FEt(R_1) \times_{\FEt(R_{12})} \FEt(R_2)
\]
is an equivalence of categories.
\end{cor}

\begin{proof}
Choose $A_1 \in \FEt(R_1), A_2 \in \FEt(R_2), A_{12} \in \FEt(R_{12})$
equipped with isomorphisms $A_1 \otimes_{R_1} R_{12} \cong A_2 \otimes_{R_2} R_{12} \cong A_{12}$,
and view this package as a finite projective glueing datum.
By Lemma~\ref{L:Kiehl generic2} plus our extra assumptions,
the kernel $A$ of $A_1 \oplus A_2 \to A_{12}$ is a finite projective $R$-module
and the natural maps $A \otimes_R R_1 \to A_1$,
$A \otimes_R R_2 \to A_2$ are isomorphisms.
Using the exact sequence
\begin{equation} \label{eq:etale glueing1}
0 \to A \to A_1 \oplus A_2 \to A_{12} \to 0,
\end{equation}
the multiplication maps on $A_1, A_2, A_{12}$ define a multiplication map on $A$, making it a flat $R$-algebra.
By Lemma~\ref{L:Kiehl generic2} again, we also have an exact sequence
\begin{equation} \label{eq:etale glueing2}
0 \to \Hom_R(A,R) \to \Hom_{R_1}(A_1,R_1) \oplus \Hom_{R_2}(A_2,R_2) \to \Hom_{R_{12}}(A_{12}, R_{12}) \to 0.
\end{equation}
Using \eqref{eq:etale glueing1}, \eqref{eq:etale glueing2}, and the snake lemma, we see that the
the trace pairing on $A$ defines an isomorphism
$A \to \Hom_R(A,R)$. This proves the claim.
\end{proof}

\subsection{\'Etale local systems}

The \'etale topology on schemes only gives rise to a useful notion of locally constant sheaves  if one restricts attention to torsion coefficients. In order to consider \'etale local systems of $\Zp$-modules or $\Qp$-vector spaces, the traditional approach is to keep track of $\Zp$-modules as inverse systems, then arrive at $\Qp$-vector spaces by formally inverting $p$ and performing \'etale descent. 
One complicating feature of this approach is that the inversion of $p$ only happens locally; that is, an \'etale $\Qp$-local system does not generally admit a $\Zp$-lattice. Another inconvenience is that the objects in these categories are not literally defined as sheaves. After first introducing the usual definitions, we describe the alternate point of view introduced recently by Bhatt and Scholze \cite{bhatt-scholze}, in which \'etale local systems are reinterpreted as genuine locally constant sheaves for a modified topology called the \emph{pro-\'etale topology}. The construction is inspired by (but somewhat simpler than) the similarly named construction for adic spaces introduced by Scholze in \cite{scholze2}, which we will also make use of in \S\ref{sec:perfect}.

\begin{defn} \label{D:etale site}
For $X$ a scheme, the \emph{small \'etale site} $X_{\et}$ of $X$ is the category of
\'etale $X$-schemes and \'etale morphisms, equipped with the Grothendieck topology
generated by set-theoretically surjective families of morphisms.

For $n$ a positive integer, a \emph{lisse sheaf} of $\ZZ/p^n \ZZ$-modules on $X_{\et}$  is a sheaf of flat $\ZZ/p^n \ZZ$-modules which is represented by
a finite \'etale $X$-scheme.
A \emph{lisse $\Zp$-sheaf} on $X_{\et}$ is an inverse system
$T = \{\cdots \to T_1 \to T_0\}$ in which each
$T_n$ is a lisse sheaf of $\ZZ/p^n \ZZ$-modules and each arrow $T_{n+1} \to T_n$
identifies $T_n$ with the cokernel of multiplication by $p^n$ on $T_{n+1}$.
A lisse $\Zp$-sheaf on $X_{\et}$ is also called an \emph{(\'etale) $\Zp$-local system} on $X$.
Such objects may be constructed using faithfully flat descent
(Theorem~\ref{T:descent modules} and Theorem~\ref{T:descent finite locally free}).
Let $\ZpLoc(X)$ denote the
category of $\Zp$-local systems on $X$.

An \emph{isogeny $\Zp$-local system} on $X$ is an element of the isogeny category of $\Zp$-local systems on $X$.
Let $\ZpILoc(X)$ denote the category of isogeny $\Zp$-local systems on $X$.

A \emph{$\Qp$-lisse sheaf} on $X_{\et}$, also called an
\emph{(\'etale) $\Qp$-local system} on $X$, is an element of the stack
associated to the fibred category of isogeny $\Zp$-local systems,
i.e., a descent datum in isogeny $\Zp$-local systems for the \'etale topology (compare \cite[Definition~4.1]{dejong-etale}).
Let $\QpLoc(X)$ denote the category of $\Qp$-local systems on $X$.
\end{defn}

\begin{remark} \label{R:local systems}
The \emph{small finite \'etale site} $X_{\fet}$ of $X$ is the subcategory of $X_{\et}$ 
in which all internal morphisms and all structure morphisms are finite \'etale, with the induced topology. One may similarly define categories of $\Zp$-local systems, isogeny $\Zp$-local systems, and $\Qp$-local systems with respect to $X_{\fet}$.
To distinguish the two sets of definitions, let us temporarily write $\ZpLoc(X_{\fet})$
and $\ZpLoc(X_{\et})$ for the two resulting categories of \'etale $\Zp$-local systems, and similarly for $\ZpILoc$ and $\QpLoc$.

The fact that lisse sheaves of $\ZZ/p^n \ZZ$-modules are represented by finite \'etale schemes means that the restriction functor $\ZpLoc(X_{\fet}) \to \ZpLoc(X_{\et})$ is a tensor equivalence. The same then holds for $\ZpILoc(X_{\fet}) \to \ZpILoc(X_{\et})$ but not in general for $\QpLoc(X_{\fet}) \to \QpLoc(X_{\et})$
unless $X = \Spec(K)$ for $K$ a field (in which case $X_{\fet} = X_{\et}$).
\end{remark}

\begin{remark} \label{R:local systems rings equivalence}
By Remark~\ref{R:local systems},
for any rings $A,B$, any tensor equivalence $\FEt(A) \cong \FEt(B)$
induces tensor equivalences $\ZpLoc(\Spec(A)) \cong \ZpLoc(\Spec(B))$,
$\ZpILoc(\Spec(A)) \cong \ZpILoc(\Spec(B))$.
However, this is not sufficient to produce an equivalence
$\QpLoc(\Spec(A)) \cong \QpLoc(\Spec(B))$; for this, it would suffice to have an isomorphism of the \'etale topoi associated to $\Spec(A)$ and $\Spec(B)$.
\end{remark}

\begin{remark} \label{R:local systems fundamental group}
Let $X$ be a connected scheme.
Choose a geometric point $\overline{x}$ of $X$ and use it as the base point to define the \'etale fundamental group $\pi_1^{\et}(X, \overline{x})$ in the sense of \cite[Expos\'e~V, \S 7]{sga1}. Then the category of \'etale $\Zp$-local systems on $X$
is equivalent to the category of continuous representations
of $\pi_1^{\et}(X, \overline{x})$ on finite free $\Zp$-modules.
(See also \cite[Expos\'e VI, \S 1.2.4]{sga5}.)

Similarly, the category of isogeny $\Zp$-local systems on $X$
is equivalent to the category of continuous representations
of $\pi_1^{\et}(X, \overline{x})$ on finite-dimensional $\Qp$-vector spaces. Underlying this statement is the fact that $\pi_1^{\et}(X, \overline{x})$ is by definition profinite, since it is defined in terms of the category of finite \'etale covering spaces. Consequently, any continuous map from $\pi_1^{\et}(X, \overline{x})$ into $\GL_n(\Qp)$ has compact image, and so factors through some conjugate of $\GL_n(\Zp)$.
That is, there is a lattice in $\Qp^n$ stable under the action of $\pi_1^{\et}(X, \overline{x})$, which can even be constructed explicitly as follows. Define a sequence of lattices $T_0, T_1, \dots$ in $\Qp^n$ by taking $T_0$ to be arbitrary and $T_{m+1}$ to be the lattice generated by the image of $T_m$ under $\pi_1^{\et}(X, \overline{x})$.
Then $T_{1} = T_{2} = \cdots$.
\end{remark}

The obvious functor from isogeny $\Zp$-local systems to $\Qp$-local systems is fully faithful but not always essentially surjective, as in the following examples suggested by the referee.
\begin{example} \label{exa:no Zariski descent}
Let $k$ be a field. Let $X$ be the union of two copies of $\PP^1_k$ (for $k$ an arbitrary field) glued along $\{0, \infty\}$.
Form an \'etale $\Qp$-local system $V$ of rank $1$ by glueing the trivial rank 1 local systems on $X - \{0\}$ and $X - \{\infty\}$ via the morphism which is the identity on one copy of $\PP^1_k - \{0, \infty\}$ and multiplication by $p$ on the other copy. 
Using Remark~\ref{R:local systems fundamental group},
we may see that $V$ is not an isogeny $\Zp$-local system: if it were, the corresponding representation into $\Qp^\times$ would have noncompact image (because the image would have to contain $p^\ZZ$).
\end{example}

While Example~\ref{exa:no Zariski descent} involves a scheme which is not irreducible, by using \'etale descent rather than just Zariski descent we can produce an example involving an irreducible (but not normal) scheme.

\begin{example}
With notation as in Example~\ref{exa:no Zariski descent}, let $Y$ be a copy of $\PP^1_k$
with the points $0$ and $\infty$ glued together. The scheme $X$ then arises as a finite \'etale cover of $Y$ of degree 2, induced by the map $\PP^1_k \cup \PP^1_k \to \PP^1_k$
acting as the identity on the first factor and the map $x \mapsto 1/x$ on the second factor. Let $\tau: X \to X$ be the nontrivial involution of $X$ over $Y$. Then $V \oplus \tau^* V$ descends to an \'etale $\Qp$-local system of rank 2 on $Y$.
\end{example}

\begin{remark} \label{R:lattice space}
For $T \in \ZpLoc(X)$, let $T \otimes \Qp$ denote the corresponding object in $\ZpILoc(X)$. For $Y$ an $X$-scheme, let $T_Y$ be the pullback of $T$ to $Y$.

For $m$ a nonnegative integer, let $F_m$ be the functor 
taking each $X$-scheme $Y$ to the pairs $(T', \iota)$ in which $T' \in \ZpLoc(Y)$ and $\iota: T_Y \otimes \Qp \to T' \otimes \Qp$ is an isomorphism such that $p^m \iota \in \Mor(T_Y, T')$, $p^m \iota^{-1} \in \Mor(T', T_Y)$.
Then $F_m$ is representable by a finite \'etale $X$-scheme $L_m(T)$; we identify $L_m(T_Y)$ with $L_m(T) \times_X Y$.
In the context of Remark~\ref{R:local systems fundamental group}, this construction corresponds to identifying lattices in $\Qp^n$ between $p^{-m} \Zp^n$ and $p^m \Zp^n$.

The value of this construction is that it allows for certain statements about lattices in isogeny $\Zp$-local systems to be translated into assertions about finite \'etale $X$-schemes, despite the fact that a $\Zp$-local system over $X$ is defined in terms of an infinite sequence of finite \'etale covers. 
Here are some examples pertinent to the proof of Lemma~\ref{L:isogeny descent}.
\begin{enumerate}
\item[(a)]
There is a closed subscheme $I_m(T)$ of $L_m(T) \times_X L_m(T)$ which is finite \'etale over $X$ with the following property: for every $X$-scheme $Y$, $I_m(T)(Y)$ is the subset of $(L_m(T) \times_X L_m(T))(Y)$ corresponding to pairs $((T'_1, \iota_1), (T'_2, \iota_2))$
for which $\iota_2\circ \iota_1^{-1} \in \Mor(T'_1, T'_2)$. 
In the context of Remark~\ref{R:local systems fundamental group}, this construction corresponds to the inclusion relation on lattices in $\Qp^n$ between $p^{-m} \Zp^n$ and $p^m \Zp^n$. (To construct $I_m(T)$, by faithfully flat descent we may reduce to the case where $T$ is constant modulo $p^{4m}$, in which case the argument is straightforward.)
\item[(b)]
Let $s_1,\dots,s_k: X \to L_m(T)$ be sections of the map $L_m(T) \to X$.
Then there exists a unique section $s: X \to L_m(T)$ with the following property:
for every $X$-scheme $Y$ and every section $s': Y \to L_m(T_Y)$,
$(s \times_X Y) \times_Y s': Y \to L_m(T_Y) \times_Y L_m(T_Y)$ factors through $I_m(T_Y)$
if and only if $(s_i \times_X Y) \times_Y s': Y \to L_m(T_Y) \times_Y L_m(T_Y)$ factors through $I_m(T_Y)$ for $i=1,\dots,k$.
In the context of Remark~\ref{R:local systems fundamental group}, this construction corresponds to forming the lattice of $\Qp^n$ generated by a finite number of other lattices. (Again, the construction proceeds by faithful flat descent to reduce to the case where $T$ is constant modulo $p^{2m}$.)
\item[(c)]
Let $f: Y \to L_m(T)$ be a morphism of faithfully finite \'etale $X$-schemes.
Then there exists a faithfully finite \'etale $X$-scheme $Y'$ such that 
$Y \times_X Y'$ splits over $Y'$ as a finite disjoint union of copies of $Y'$ (by induction on the degree of $Y \to X$). The pullback of $f$ then gives rise to a collection of sections of $L_m(T_{Y'}) \to Y'$, to which we may apply the construction of (b); since the latter is canonical, it acquires a descent datum back to $X$, so we end up with a section $s: X \to L_m(T)$.
\end{enumerate}
\end{remark}

Whereas $\Zp$-local systems and $\Qp$-local systems descend along surjective \'etale morphisms of schemes, isogeny $\Zp$-local systems do not in general.
However, one does obtain descent for finite \'etale morphisms.
\begin{lemma} \label{L:isogeny descent}
Any faithfully finite \'etale morphism $R \to R'$ of rings is an effective descent morphism for isogeny $\Zp$-local systems.
\end{lemma}
\begin{proof}
Suppose that $V' \in \ZpILoc(\Spec(R'))$ carries a descent datum relative to $\Spec(R)$.
Choose $T' \in \ZpLoc(\Spec(R'))$ giving rise to $V'$.
We construct a sequence of morphisms $T' = T'_0 \to T'_1 \to \cdots$ in $\ZpLoc(\Spec(R'))$ which are isomorphisms in $\ZpILoc(\Spec(R'))$, as follows.

Put $R'' = R' \otimes_R R'$ and let $\pi_1, \pi_2: \Spec(R'') \to \Spec(R')$ be the two projections.
Given $T'_i$, put $T''_i = \pi_1^*(T'_i)$; for each $m$, we identify $L_m(T''_i)$ with $L_m(T'_i) \times_{\Spec(R'), \pi_1} \Spec(R'')$.
We may fix a sufficiently large $m$ at the beginning; the descent datum on $V'$ defines a section $s$ of the projection $L_m(T_i'') \to \Spec(R_i'')$.
By applying Remark~\ref{R:lattice space}(c) to the composition $\pi_1 \circ s: \Spec(R'') \to L_m(T''_i) \to L_m(T'_i)$, we obtain a new object $T'_{i+1} \in \ZpLoc(\Spec(R'))$ and an isomorphism $T'_i \otimes \Qp \to T'_{i+1} \otimes \Qp$. Moreover, this morphism descends to a morphism $T'_i \to T'_{i+1}$ because the pullback of $s$ along the diagonal morphism $\Spec(R') \to \Spec(R'')$ corresponds to the identity morphism on $T'$.

We now see that $T'_1 = T'_2 = \cdots$ by reducing to the case where $R$ and $R'$ are local (and hence connected) and applying Remark~\ref{R:local systems fundamental group}.
This means that $T'_1$ acquires a descent datum from $V$ and thus descends to an object in $\ZpLoc(\Spec(R))$.
\end{proof}

\begin{remark} \label{R:normal noetherian local systems}
Suppose that $X$ is normal and noetherian. Then $X$ is the disjoint union of finitely many irreducible components. 
On each component, \'etale $\Qp$-local systems correspond precisely to continuous representations of the \'etale fundamental group (e.g., see \cite[Lemma~7.4.7]{bhatt-scholze}). Consequently, the natural functor $\ZpILoc(X) \to \QpLoc(X)$ is an equivalence of categories.
\end{remark}

We now describe the alternate approach to local systems given in \cite{bhatt-scholze}, in which local systems become genuine sheaves for an alternate topology.
\begin{defn} \label{D:pro-etale schemes}
A morphism $f: Y \to X$ of schemes is \emph{weakly \'etale} if both $f$ and
$\Delta_f: Y \to Y \times_X Y$ are flat. For example, if $X = \Spec(A)$, $Y = \Spec(B)$ and $B$ is a direct limit of \'etale $A$-algebras, then $f$ is weakly \'etale.

The \emph{pro-\'etale site} of a scheme $X$,
denoted by $X_{\proet}$, is the site consisting of weakly \'etale $X$-schemes and fpqc coverings.
For any topological space $T$ and any scheme $X$, the presheaf
\[
\calF_T: U \mapsto \Map_{\cont}(U, T)
\]
on $X_{\proet}$ is a sheaf \cite[Lemma~4.2.12]{bhatt-scholze}, called the \emph{constant sheaf} with values in $T$. A sheaf which is locally of this form is said to be \emph{locally constant}. (In \cite{bhatt-scholze} one finds also a discussion of \emph{constructible} sheaves, which we do not consider here.)

For any $n$, given a lisse sheaf of $\ZZ/p^n \ZZ$-modules on $X_{\et}$, we may pull back from $X_{\et}$ to $X_{\proet}$ to obtain a
$\calF_{\ZZ/p^n \ZZ}$-module on $X_{\proet}$ which is locally free of finite rank. 
By taking inverse limits, we obtain a functor from $\ZpLoc(X)$ to sheaves of $\calF_{\Zp}$-modules on $X_{\proet}$ which are locally free of finite rank.
 We also obtain a functor from $\QpLoc(X)$ to sheaves of $\calF_{\Qp}$-modules on $X_{\proet}$ which are locally free of finite rank.
\end{defn}

\begin{theorem} \label{T:proetale local systems schemes}
For any scheme $X$, the category $\ZpLoc(X)$ (resp. $\QpLoc(X)$) is naturally equivalent to the category of sheaves of $\calF_{\Zp}$-modules (resp.\ $\calF_{\Qp}$-modules) on $X_{\proet}$ which are locally free of finite rank (and in particular locally constant).
\end{theorem}
\begin{proof}
See \cite[\S 6.8]{bhatt-scholze}.
\end{proof}

\begin{remark} \label{R:unramified coefficients}
For $F$ a finite extension of $\Qp$, one can define \'etale $\gotho_F$-local systems, isogeny \'etale $\gotho_F$-local systems, and
\'etale $F$-local systems on a scheme $X$ by analogy with the case $F = \Qp$. In fact, these can be interpreted as objects of the corresponding category over $\Qp$ plus the extra structure of an action of $\gotho_F$.
We will only need this
observation in the case where $F = \QQ_{p^d}$ is a finite unramified extension of $\Qp$ of degree $d$, in which case we
label the resulting categories $\ZpdLoc(X)$, $\ZpdILoc(X)$, and $\QpdLoc(X)$.
\end{remark}

\subsection{Semilinear actions}

\begin{convention}
For $S$ a ring equipped with an endomorphism $\varphi$ and $M$ an $S$-module, a \emph{semilinear $\varphi$-action}
on $M$
will always mean an \emph{isomorphism} (not just an endomorphism) $\varphi^*M \to M$ of $S$-modules.
We will commonly interpret such an action as a $\varphi$-semilinear map $M \to M$.
\end{convention}

Although we have not found a precise reference, we believe that the following is a standard lemma in algebraic $K$-theory, specifically from the study of polynomial extensions (as in \cite[Chapter~XII]{bass}).
\begin{lemma} \label{L:cover with free}
Let $S$ be a ring equipped with an endomorphism $\varphi$.
Let $M$ be a finitely generated $S$-module equipped with a semilinear $\varphi$-action.
Then there exists a finite free $S$-module $F$ equipped with a semilinear $\varphi$-action
and a $\varphi$-equivariant surjection $F \to M$.
\end{lemma}
\begin{proof}
Choose generators $\bv_1,\dots,\bv_n$ of $M$, and use them to define a surjection
$E = S^n \to M$ of $S$-modules.
Let $T: \varphi^* M \to M$ be the given isomorphism.
Choose $A_{ij}, B_{ij} \in S$
so that $T(\bv_j \otimes 1) = \sum_i A_{ij}\bv_i$, $T^{-1}(\bv_j) = \sum_i B_{ij} (\bv_i \otimes 1)$;
by writing
\[
\bv_k= T(T^{-1}(\bv_k))
= T \left( \sum_j B_{jk} (\bv_j \otimes 1)\right) = \sum_{i,j} A_{ij} B_{jk} \bv_i,
\]
we see that the columns of the matrix $C = A B - 1$ are elements of $N = \ker(E \to M)$.

Let $D$ be the block matrix
$\begin{pmatrix} A & C \\ 1 & B \end{pmatrix}$.
By using row operations to clear the bottom left block, we find that $\det(D) = \det(A B - C) = 1$.
Consequently, $D$ is invertible over $S$, so we may use it to define an isomorphism
$\varphi^* F \to F$ for $F = E \oplus E$. This isomorphism carries
$\varphi^*(N \oplus E)$ into
$N \oplus E$, so we obtain a $\varphi$-equivariant surjection $F \to M$ as desired.
\end{proof}

\begin{cor} \label{C:cover with free}
Let $S$ be a ring equipped with an endomorphism $\varphi$. 
Let $M$ be a finite projective $S$-module equipped with a semilinear $\varphi$-action. Then there exists another finite projective $S$-module $N$ admitting a semilinear $\varphi$-action such that $M \oplus N$ is a free $S$-module.
\end{cor}
\begin{proof}
Apply Lemma~\ref{L:cover with free} to construct a finite free $S$-module $F$ equipped with a semilinear $\varphi$-action and a $\varphi$-equivariant $S$-linear surjection $F \to M$, then put $N = \ker(F \to M)$.
\end{proof}

\begin{defn} \label{D:varphi cohomology}
Let $S$ be a ring equipped with an endomorphism $\varphi$.
Let $M$ be a module over $S$ equipped with a semilinear $\varphi$-action.
We then write
\[
H^0_{\varphi}(M) = \ker(\varphi-1, M), \qquad H^1_{\varphi}(M) = \coker(\varphi-1, M),
\]
and $H^i_{\varphi}(M) = 0$ for $i \geq 2$.
The groups $H^i_{\varphi}(M)$ may be interpreted as the Yoneda extension groups
$\Ext^i(S, M)$ in the category of left modules over the twisted polynomial ring $S\{\varphi\}$,
by tensoring $M$ over $S$ with the free resolution
\[
0 \to S\{\varphi\} \stackrel{\varphi-1}{\to} S\{\varphi\} \to S \to 0
\]
of $S$. (For a detailed development of Yoneda extension groups, see for instance \cite[\S IV.9]{hilton-stammbach}.)
\end{defn}

\begin{remark} \label{R:cohomology restriction}
In Definition~\ref{D:varphi cohomology},
if $M$ is a module over $S$ equipped with a semilinear $\varphi^d$-action for some positive integer $d$,
we may identify $H^i_{\varphi^d}(M)$ with $H^i_{\varphi}(N)$
for $N = M \oplus \varphi^* M \oplus \cdots \oplus (\varphi^{d-1})^* M$.
\end{remark}

\begin{remark} \label{R:dual semilinear action}
Let $S$ be a ring equipped with an endomorphism $\varphi$.
Let $M$ be a finite projective module over $S$ equipped with a semilinear $\varphi$-action.
Then there is a natural way to equip the dual module $M^\dual = \Hom_R(M,R)$ with a $\varphi$-module structure so that the pairing map
$M \otimes_R M^\dual \to R$ is $\varphi$-equivariant.
\end{remark}

\section{Spectra of nonarchimedean Banach rings}
\label{sec:berkovich}

We set notation and terminology concerning
spectra of nonarchimedean (commutative) Banach rings.
We will consider two separate but related notions of spectrum, the \emph{Gel'fand spectrum} of Berkovich \cite{berkovich1, berkovich2}
and the \emph{adic spectrum} of Huber \cite{huber1, huber2, huber}.

\setcounter{theorem}{0}
\begin{convention}
For $M$ a matrix over a ring equipped with a submultiplicative seminorm
$\alpha$, we write $\alpha(M)$ for $\sup_{i,j} \{\alpha(M_{ij})\}$.
\end{convention}

\subsection{Seminorms on groups and rings}

We begin by setting notation regarding seminorms. We will later have to consider also \emph{semivaluations}, which take values not in the real numbers but in more general ordered abelian groups; see \S\ref{subsec:adic spectrum}.

\begin{defn}
Consider the following conditions on an abelian group $G$ and a
function $\alpha: G \to [0, +\infty)$.
\begin{enumerate}
\item[(a)]
For all $g,h \in G$, we have $\alpha(g-h) \leq \max\{\alpha(g), \alpha(h)\}$.
\item[(b)]
We have $\alpha(0) = 0$.
\item[(b$'$)]
For all $g \in G$, we have $\alpha(g) = 0$ if and only if $g=0$.
\end{enumerate}
We say $\alpha$ is a \emph{(nonarchimedean) seminorm}
if it satisfies (a) and (b), and a
\emph{(nonarchimedean) norm} if it satisfies (a) and (b$'$). Any seminorm $\alpha$ induces a norm on $G/\ker(\alpha)$.

If $\alpha, \alpha'$
are two seminorms on the same abelian group $G$,
we say $\alpha$ \emph{dominates} $\alpha'$,
and write $\alpha \geq \alpha'$ or $\alpha' \leq \alpha$, if there exists
$c>0$ for which $\alpha'(g) \leq c \alpha(g)$ for all $g \in G$.
If $\alpha$ and $\alpha'$ dominate each other, we say they are
\emph{equivalent}; in this case, $\alpha$ is a norm if and only if
$\alpha'$ is.
\end{defn}

\begin{defn}
Let $G,H$ be two abelian groups equipped with nonarchimedean seminorms
$\alpha, \beta$, and let $\varphi: G \to H$ be a homomorphism.
We say $\varphi$ is \emph{bounded} if $\alpha \geq \beta \circ \varphi$,
and \emph{isometric} if $\alpha = \beta \circ \varphi$.
(An intermediate condition is that $\varphi$ is \emph{submetric}, meaning that
$\alpha(g) \geq \beta(\varphi(g))$ for all $g \in G$.)

The \emph{quotient seminorm} induced by $\alpha$
is the seminorm $\overline{\alpha}$
on $\image(\varphi)$
defined by
\[
\overline{\alpha}(h) = \inf\{\alpha(g): g \in G, \varphi(g) = h\}.
\]
If $H$ is also equipped with a seminorm $\beta$, we say $\varphi$ is
\emph{strict} if the two seminorms $\overline{\alpha}$ and $\beta$
on $\image(\varphi)$ are equivalent; this implies in particular
that $\varphi$ is bounded.
We say $\varphi$ is \emph{almost optimal} if $\overline{\alpha}$ and $\beta$ coincide.
We say $\varphi$ is \emph{optimal} if every $h \in \image(\varphi)$ admits a lift $g \in G$
with $\alpha(g) = \overline{\alpha}(h)$.
Any optimal homomorphism is almost optimal, and any optimal homomorphism is strict, but not conversely. Also beware that a composition $g \circ f$ of strict morphisms is not guaranteed to be strict unless $f$ is surjective or $g$ is injective.
\end{defn}

\begin{remark}
Berkovich uses the term
\emph{admissible} in place of \emph{strict}, but the latter is
well-established in the context of topological groups,
as in \cite[\S III.2.8]{bourbaki-top}. However, there is no perfect choice of terminology;
our convention will create some uncomfortable linguistic proximity
during the discussion of \emph{strict $p$-rings}.
\end{remark}

\begin{defn}
For $G$ an abelian group equipped with a nonarchimedean seminorm $\alpha$,
equip the group of Cauchy sequences in $G$ with the seminorm
whose value on the sequence $g_0,g_1,\dots$ is $\lim_{i \to \infty} \alpha(g_i)$.
The quotient by the kernel of this seminorm is the
\emph{separated completion} $\widehat{G}$ of $G$ under $\alpha$.
For the unique continuous extension of $\alpha$ to
$\widehat{G}$, the homomorphism $G \to \widehat{G}$ is isometric,
and injective if and only if $\alpha$ is a norm (in which case we call
$\widehat{G}$ simply the \emph{completion} of $G$).
\end{defn}

\begin{defn} \label{D:submultiplicative}
Let $A$ be a ring. Consider the following conditions
on a (semi)norm $\alpha$ on the additive group of $A$.
\begin{enumerate}
\item[(c)]
For all $g,h \in A$, we have $\alpha(gh) \leq \alpha(g) \alpha(h)$.
\item[(c$'$)]
We have (c), and for all $g \in A$, we have $\alpha(g^2) = \alpha(g)^2$.
(Equivalently, $\alpha(g^n) = \alpha(g)^n$ for
all $g \in A$ and all positive integers $n$. In particular, $\alpha(1) \in \{0,1\}$.)
\item[(c$''$)]
We have (c$'$), $\alpha(1) = 1$, and for all $g,h \in A$, we have $\alpha(gh) = \alpha(g) \alpha(h)$.
\end{enumerate}
We say $\alpha$ is
\emph{submultiplicative} if it satisfies (c),
\emph{power-multiplicative} if it satisfies (c$'$),
and \emph{multiplicative} if it satisfies (c$''$).
Note that if $\alpha$ is a submultiplicative seminorm and $\alpha'$
is a power-multiplicative seminorm, then $\alpha$ dominates $\alpha'$ if and only
if $\alpha(a) \geq \alpha'(a)$ for all $a \in A$.
\end{defn}

\begin{example}
For any abelian group $G$, the \emph{trivial norm} on $G$ sends
$0$ to $0$ and any nonzero $g \in G$ to 1. For any ring $A$, the trivial
norm on $A$ is submultiplicative in all cases,
power-multiplicative if and only if $A$ is reduced,
and multiplicative if and only if $A$ is an integral domain. (As usual, the zero ring is not considered to be a domain.)
\end{example}

\begin{defn}
For $A$ a ring equipped with a submultiplicative seminorm $\alpha$, define
\begin{align*}
\gotho_A &= \{a \in A: \alpha(a) \leq 1 \} \\
\gothm_A &= \{a \in A: \alpha(a) < 1 \} \\
\kappa_A &= \gotho_A/\gothm_A.
\end{align*}
If $\alpha(1) \leq 1$, then $\gotho_A$ is a ring and $\gothm_A$ is an ideal of $\gotho_A$.
If $A$ is a field equipped with a multiplicative norm, then $\gotho_A$ is a valuation ring with maximal ideal
$\gothm_A$ and residue field $\kappa_A$.
\end{defn}

\begin{example} \label{exa:power-multiplicative product}
Let $I$ be an arbitrary index set. For each $i \in I$, specify a ring $A_i$ and a power-multiplicative
seminorm $\alpha_i$ on $A_i$. Put $A = \prod_{i \in I} A_i$, and define the function
$\alpha: A \to [0, +\infty]$ by setting $\alpha((a_i)_{i \in I}) = \sup_i\{\alpha_i(a_i)\}$.
Then the subset $A_0$ of $A$ on which $\alpha$ takes finite values is a subring
on which $\alpha$ restricts to a power-multiplicative seminorm. This example is in some sense universal;
see Theorem~\ref{T:transform}.
\end{example}

\begin{defn}
Let $A$ be a ring equipped with a submultiplicative seminorm $\alpha$.
The \emph{spectral seminorm} on $A$ is the power-multiplicative
seminorm $\alpha_{\spect}$ defined by $\alpha_{\spect}(a) = \lim_{s \to \infty}
\alpha(a^s)^{1/s}$. (The existence of the
limit is an exercise in real analysis known as \emph{Fekete's lemma}.)
Note that equivalent choices of $\alpha$ define the same spectral seminorm.
\end{defn}

\begin{defn}
Let $A,B,C$ be rings equipped with submultiplicative seminorms
$\alpha, \beta, \gamma$,
and let $A \to B, A \to C$ be bounded homomorphisms.
The \emph{product seminorm} on the ring $B \otimes_A C$ is defined
by taking $f \in B \otimes_A C$
to the infimum of $\max_i \{\beta(b_i) \gamma(c_i)\}$
over all presentations $f = \sum_i b_i \otimes c_i$.
The separated completion of $B \otimes_A C$ for the product
seminorm is denoted $B \widehat{\otimes}_A C$ and called the
\emph{completed tensor product} of $B$ and $C$ over $A$.
\end{defn}

\begin{remark} \label{R:operator norm}
The definition of a submultiplicative seminorm $\alpha$ on $A$ does not include the condition
that $\alpha(1) \leq 1$. However, if we define the \emph{operator seminorm} $\alpha'$ by the formula
\begin{equation} \label{eq:operator seminorm}
\alpha'(a) = \inf\{c \geq 0: \alpha(ab) \leq c \alpha(b) \mbox{\,for all $b \in A$}\},
\end{equation}
then $\alpha'$ is a submultiplicative norm, $\alpha'(1) \leq 1$, and $\alpha'(a) \leq \alpha(a)$ for all $a \in A$.
Moreover, if $\alpha(1) > 0$, then we may take $b=1$ in \eqref{eq:operator seminorm}
to deduce that $\alpha'(a) \geq \alpha(1)^{-1} \alpha(a)$.
Consequently, in all cases $\alpha'$ is equivalent to $\alpha$ (this being trivially true if $\alpha(1) = 0$).
\end{remark}

\subsection{Banach rings and modules}

\begin{defn} \label{D:Banach ring}
Throughout this paper, an \emph{analytic field} is a field equipped with a nontrivial multiplicative nonarchimedean norm under which it is complete. For $K$ an analytic field, any finite
extension of $K$ admits a unique structure of an analytic field extending $K$ \cite[Theorem~3.2.3/2]{bgr}. The inclusion of the nontriviality condition is a convention which is not universal: it is notably absent in Berkovich's work. However, this condition will be needed in order to work with adic spectra; it also shows up as a hypothesis in some other key results, such as the open mapping theorem (Theorem~\ref{T:open mapping}).

A \emph{Banach ring} is a commutative ring $R$ equipped with a submultiplicative norm under which it is complete.
We allow the zero ring as a Banach ring, so that the completed tensor product is defined on the category of Banach rings.
(What we call Banach rings would more commonly be called \emph{commutative Banach rings},
but we will not use noncommutative Banach rings in this paper.)

A \emph{Banach algebra} over a Banach ring $R$ is a Banach ring $S$ equipped with the structure of an $R$-algebra in such a way that the map $R \to S$ is bounded.

From \S\ref{subsec:adic spectrum} on, we will only consider Banach rings containing a topologically nilpotent unit. For more discussion of this condition, see Remark~\ref{R:small norm elements}.
\end{defn}

\begin{lemma} \label{L:closure trivial}
Let $I$ be a nontrivial ideal in a Banach ring $A$. Then the closure of $I$
is also a nontrivial ideal. In particular, any maximal ideal in $A$ is closed.
\end{lemma}
\begin{proof}
If the closure were trivial, then $I$ would contain an
element $x$ for which $1-x \in \gothm_A$. But then the series $\sum_{i=0}^\infty (1-x)^i$
would converge in $A$ to an inverse of $x$, contradicting the assumption that $I$ is a nontrivial
ideal.
\end{proof}

For $A$ a Banach ring, it is easy to check
(using Remark~\ref{R:henselian criteria}) that the pair
$(\gotho_A, \gothm_A)$ is henselian. The following refinement
of this observation will also prove to be useful.
See also Proposition~\ref{P:henselian direct limit2}.
\begin{lemma} \label{L:henselian direct limit}
Let $\{(A_i, \alpha_i)\}_{i \in I}$ be a direct system in the category of Banach rings and submetric homomorphisms.
Equip the direct limit $A$ of the $A_i$ in the category of rings with
the infimum $\alpha$ of the quotient seminorms induced by the $\alpha_i$.
\begin{enumerate}
\item[(a)]
The pair $(A, \ker(\alpha))$ is henselian.
\item[(b)]
The pair $(\gotho_A, \gothm_A)$ is henselian.
\end{enumerate}
\end{lemma}
\begin{proof}
In both cases, we check the criterion of Remark~\ref{R:henselian criteria}.
To check (a),
note first that $\ker(\alpha)$ is contained in the Jacobson radical of $A$:
if $a-1 \in \ker(\alpha)$, then there exists some index $i$ for which
$a-1$ is an element of $A_i$ of norm less than 1. Since $A_i$ is complete, this forces
$a$ to be invertible. With that in mind,
let $f = \sum_i f_i T^i \in A[T]$ be a monic polynomial with $f_0 \in \ker(\alpha),
f_1 \in A^\times$. We construct a root of $f$ using the Newton-Raphson iteration as follows.
Put $x_0 = 0$. Given $x_l \in A$ for some nonnegative integer $l$
such that $x_l \in \ker(\alpha)$, $f'(x_l)$
is invertible modulo $\ker(\alpha)$ and hence is a unit. We may thus define
$x_{l+1} = x_l - f(x_l)/f'(x_l)$ and note that $x_{l+1} \in \ker(\alpha)$.
For any sufficiently large $i \in I$, the sequence
$\{x_l\}$ is Cauchy in $A_i$, and so has a limit which is a root of $f$.

To check (b),
let $f = \sum_i f_i T^i \in \gotho_A[T]$ be a monic polynomial with $f_0 \in \gothm_A,
f_1 \in \gotho_A^\times$; then $f$ admits a root $r$ in $\gothm_{\widehat{A}}$.
Choose any $s \in A$ with $\alpha(r-s) < 1$,
and put $x_0 = s$, $x_{l+1} = x_l - f(x_l)/f'(x_l)$.
For any sufficiently large $i \in I$, the sequence
$\{x_l\}$ is Cauchy in $A_i$, and so has a limit which is a root of $f$.
\end{proof}

\begin{lemma} \label{L:descend etale on field}
Retain notation as in Lemma~\ref{L:henselian direct limit}.
\begin{enumerate}
\item[(a)]
The base change functor $\FEt(A) \to \FEt(\widehat{A})$
is rank-preserving and fully faithful.
\item[(b)]
Suppose that $\alpha$ is a multiplicative seminorm and $K = A/\ker(\alpha)$ is a field.
Then the base change functor $\FEt(A) \to \FEt(\widehat{K})$
is an equivalence of categories.
\end{enumerate}
\end{lemma}
\begin{proof}
To check (a), we first observe that by Lemma~\ref{L:henselian direct limit}(a),
$\ker(\alpha)$ is contained in the Jacobson radical of $A$. Next,
for any $x \in A$ which becomes a unit in $\widehat{A}$, we can find
$y \in A$ for which $\alpha(xy-1) < 1$, so $xy$ is a unit in some $A_i$
and so $x$ is a unit in $A$. Next, any invertible module $M$ over $A$ is the base extension of some invertible module $M_i$ over some $A_i$;
if $M_i \otimes_{A_i} \widehat{A}$ admits a generator, then so does
$M_i \otimes_{A_i} A_j$ for any sufficiently large $j$ by Lemma~\ref{L:nearby generators} below.
Finally, for $S \in \FEt(A)$, note that
any idempotent $S \otimes_A \widehat{A}$ can have at most one preimage in
$S \otimes_A A/\ker(\alpha)$ (since $A/\ker(\alpha)$ injects into $\widehat{A}$ and $S$
is a projective $A$-module) and hence
at most one preimage in $S$ (by Lemma~\ref{L:unique idempotent lifting}).
We conclude by Proposition~\ref{P:lift idempotents} that to check (a),
it suffices to verify that for each $S \in \FEt(A)$, every idempotent of $S \otimes_A \widehat{A}$
arises from some idempotent of $S$.

Since $S \in \FEt(A)$, by Remark~\ref{R:fet direct limit} we can choose an index $i \in I$ for which $S = S_i \otimes_{A_i} A$ for some
$S_i \in \FEt(A_i)$.
Since $S_i$ is a finite locally free $A_i$-module (see Definition~\ref{D:finite etale category}),
we can choose a finite free $A_i$-module $F_i$ admitting a direct sum decomposition $F_i \cong S_i \oplus T_i$.
Choose a basis $x_1,\dots,x_n$ of $F_i$ and let $y_1,\dots,y_n$ be the projections of $x_1,\dots,x_n$ onto $S_i$.
For $h,k \in \{1,\dots,n\}$, write $y_h y_k$ in $F_i$ as
$\sum_{l} c_{hkl} x_l$ with $c_{hkl} \in A_i$, so that in $S_i$ we have $y_h y_k = \sum_l c_{hkl} y_l$.
Put $c = \max\{1, \sup_{h,k,l} \{\alpha_i(c_{hkl})\}\}$.

For each $j \in I$ with $i \leq j$,
let $\beta_j$ be the restriction to $S_j = S_i \otimes_{A_i} A_j$
of the supremum norm on $F_j = F_i \otimes_{A_i} A_j$ defined by the basis $x_1,\dots,x_n$.
Note that $\beta_j(xy) \leq c \beta_j(x) \beta_j(y)$
for all $j$ and all $x,y \in S_j$. Similarly, let $\beta$ be the supremum seminorm on
$S \otimes_A \widehat{A}$ defined by the basis $x_1,\dots,x_n$,
so that $\beta(xy) \leq c \beta(x) \beta(y)$
for all $x,y \in S \otimes_A \widehat{A}$. In particular, any nonzero idempotent element
$e \in S \otimes_A \widehat{A}$ satisfies $\beta(e) \geq c^{-1}$.

Let $e \in S \otimes_A \widehat{A}$ be an idempotent element.
Choose $\epsilon > 0$ with $\epsilon \max\{\beta(e),1\} < 1$.
Since $e^2 = e$ in $S \otimes_A \widehat{A}$, we can choose
$j \in I$ and $x \in S_j$ with $\beta(x-e) < c^{-1}$
and $\beta_j(x^2 - x) \leq c^{-1} \epsilon$.
Define the sequence $x_0, x_1, \dots$ by $x_0 = x$ and $x_{l+1} = 3x_l^2 - 2x_l^3$.
We then have $\beta_j(x_l^2 - x_l) \leq c^{-1} \epsilon^{l+1}$ by induction on $l$,
by writing
\[
x_{l+1}^2 - x_{l+1} = 4(x_l^2 - x_l)^3  - 3 (x_l^2 - x_l)^2.
\]
Also, $x_{l+1} - x_l = (x_l^2 - x_l)(1 - 2x_l)$, so
by induction on $l$, $\beta_j(x_l) \leq \max\{\beta_j(x),1\}$
and $\beta(x_l) \leq \max\{\beta(x),1\}$.
Using the equation $x_{l+1} - x_l = (x_l^2 - x_l)(1 - 2x_l)$ again,
we see that the $x_l$ form a Cauchy sequence, whose limit $y$ in $S_j$
must satisfy $y^2 = y$. In addition, $\beta(y - x) \leq c^{-1} \epsilon \max\{\beta(x),1\}$,
so $\beta(y-e) < c^{-1}$. Since $(y-e)^2$ is an idempotent element of $S \otimes_A \widehat{A}$,
this is only possible if $(y-e)^2 = 0$; since $y(y-e)^2 = y - ey$ and $e(y-e)^2 = e - ey$, this yields $y = e$. This completes the proof of (a).

To check (b), note that the hypotheses ensure that the completion $\widehat{K}$ of $K$ is an analytic field.
It suffices to show that an arbitrary finite separable field extension $\widehat{L}$ of $\widehat{K}$ occurs in the essential image of the base change functor. By the primitive element theorem, we can write
$\widehat{L} \cong \widehat{K}[T]/(P)$ for some monic separable polynomial
$P \in \widehat{K}[T]$. By Hensel's lemma (or more precisely Krasner's lemma), we also have
$\widehat{L} \cong \widehat{K}[T]/(Q)$ for any monic polynomial $Q \in \widehat{K}[T]$
whose coefficients are sufficiently close to those of $P$.
In particular, we may choose $Q \in K[T]$, in which case we may write
$\widehat{L} = L \otimes_K \widehat{K}$ for $L = K[T]/(Q) \in \FEt(K)$.
Since $\FEt(A) \to \FEt(K)$ is essentially surjective by
Lemma~\ref{L:henselian direct limit} plus Theorem~\ref{T:henselian}, this proves the claim.
\end{proof}

\begin{lemma} \label{L:analytic fields submultiplicative}
Let $K$ be an analytic field with norm $\alpha$, and let $L$ be a finite extension of $K$.
Then the unique multiplicative extension of $\alpha$ to $L$ (Definition~\ref{D:Banach ring})
is also the unique power-multiplicative extension of $\alpha$ to $L$.
\end{lemma}
\begin{proof}
Let $\beta$ be the multiplicative extension of $\alpha$ to $L$, and let $\gamma$ be a power-multiplicative
extension of $\alpha$ to $L$. Note that for $x \in K^\times, y \in L$,
we have
\[
\gamma(xy) \leq \gamma(x) \gamma(y) =\gamma(x^{-1})^{-1} \gamma(y) \leq \gamma(xy),
\]
so $\gamma(xy) = \gamma(x) \gamma(y)$.

Given $x \in L^\times$, let $P \in K[T]$ be the minimal polynomial of $x$
over $K$; since $K$ is complete, the Newton polygon of $P$ consists of a single segment.
In other words, if we write $P(T) = \sum_{i=0}^n P_i T^i$ with $P_n = 1$, then
$|P_{n-i}|^{1/i} \leq |P_0|^{1/n} = \beta(x)$ for $i=1,\dots,n$.
(See \cite[\S 2.1]{kedlaya-course} for more discussion of Newton polygons.)

If $\gamma(x) > |P_0|^{1/n}$,
then under $\gamma$ the sum $0 = \sum_{i=0}^n P_i x^i$ would be dominated by
the term $P_n x^n$, a contradiction. Hence $\gamma(x) \leq |P_0|^{1/n} = \beta(x)$
and similarly $\gamma(x^{-1}) \leq \beta(x^{-1})$; by writing
\[
1 = \gamma(x \cdot x^{-1}) \leq \gamma(x) \gamma(x^{-1}) \leq \beta(x) \beta(x^{-1}) = 1,
\]
we see that $\gamma(x) = \beta(x)$ as desired.
\end{proof}

Before moving on to Banach modules, we make one observation about modules over a Banach ring.
\begin{lemma} \label{L:finitely generated Banach}
Let $R$ be a Banach ring.
\begin{enumerate}
\item[(a)]
For any finite $R$-module $M$, the quotient seminorm
defined by a surjection $\pi: R^n \to M$ of $R$-modules does not depend, up to equivalence, on the choice of the
surjection.
\item[(b)]
Let $R \to S$ be a bounded homomorphism of Banach rings. Let $M$ be a finite $R$-module,
let $N$ be a finite $S$-module, and let $M \to N$ be an additive $R$-linear map.
Then this map becomes bounded if we equip $M$ and $N$ with seminorms
as described in (a).
\end{enumerate}
\end{lemma}
\begin{proof}
To prove (a), let $\pi': R^m \to M$ be a second surjection, and combine $\pi$ and $\pi'$ to obtain a third surjection
$\pi'': R^{n+m} \to M$.
It is enough to check that the quotient seminorms $|\cdot|, |\cdot|''$ induced by $\pi, \pi''$ are equivalent,
as then the same argument will apply with $\pi$ and $\pi'$ interchanged.

Let $\be_1,\dots,\be_{n+m}$ be the standard basis of $R^{n+m}$.
On one hand, we clearly have $|\cdot|'' \leq |\cdot|$ because lifting an element of $M$ to $R^n$
also gives a lift to $R^{n+m}$. On the other hand, for $j=n+1,\dots,n+m$,
we can write $\pi(\be_j) = \sum_{i=1}^n A_{ij} \pi(\be_i)$ for some $A_{ij} \in R$.
If an element of $M$ lifts to $\sum_{i=1}^{n+m} c_i \be_i \in R^{n+m}$, it also lifts to
\[
\sum_{i=1}^n \left( c_i + \sum_{j=n+1}^{n+m} A_{ij} c_j \right) \be_i \in R^n.
\]
Consequently, we have $|\cdot| \leq \max\{1,|A|\} |\cdot|''$. This yields (a).

To prove (b), choose surjections $R^m \to M$, $S^n \to N$ of $R$-modules. We may then lift the composition
$R^m \to M \to N$ to a homomorphism $R^m \to S^n$ which is evidently bounded. This proves the claim.
\end{proof}

\begin{defn}
Let $R$ be a Banach ring. A \emph{Banach module} over $R$
is an $R$-module $M$ whose additive group is complete for a norm
$|\cdot|_M$ for which for some $c>0$, we have
$|r\bv|_M \leq c|r| |\bv|_M$ for all $r \in R, \bv \in M$.
In particular, any Banach algebra over $R$ is a Banach module over $R$.
\end{defn}

One has an analogue of the Banach-Schauder
open mapping theorem in the nonarchimedean setting.
(Note that this result fails completely without a restriction on the base ring.)
\begin{theorem} \label{T:open mapping}
Let $R$ be a Banach ring containing a topologically nilpotent unit.
Let $\varphi: V \to W$ be a bounded surjective
homomorphism of Banach modules over $R$.
Then $\varphi$ is open and strict.
\end{theorem}
\begin{proof}
For $R$ an analytic field, see \cite[\S I.3.3, Th\'eor\`eme~1]{bourbaki-evt}
or \cite[Proposition~8.6]{schneider}.
For the more general case, see \cite{henkel}.
\end{proof}

\begin{lemma} \label{L:inject tensor}
Let $V,W,X$ be Banach modules over an analytic field $K$.
\begin{enumerate}
\item[(a)]
The map
$V \otimes_K W \to V \widehat{\otimes}_K W$ is injective.
\item[(b)]
Let $f: V \to W$ be a bounded homomorphism
and let $f_X: V \widehat{\otimes}_K X \to W \widehat{\otimes}_K X$
be the induced map. Then the natural map $\ker(f) \widehat{\otimes}_K X \to \ker(f_X)$ is a bijection.
\item[(c)]
In (b), if $f$ is strict, then so is $f_X$.
\end{enumerate}
\end{lemma}
\begin{proof}
All three parts reduce immediately to the case where $V,W,X$ contain dense $K$-vector subspaces of at most
countable dimension (i.e., they are \emph{separable} Banach modules).
In this setting, (a) follows from the existence of Schauder bases for $V$ and $W$; see for instance \cite[Lemma~1.3.11]{kedlaya-course}. Similarly, (b) and (c) follow from the existence of a Schauder basis for $X$.
\end{proof}

\begin{defn}
Let $R$ be a Banach ring.
A \emph{finite Banach module/algebra} over $R$ is a Banach module/algebra $M$ over $R$ admitting
a strict surjection $R^n \to M$ of Banach modules over $R$ for some nonnegative integer $n$
(for the supremum norm on $R^n$ defined by the canonical basis).
By Lemma~\ref{L:finitely generated Banach}, the equivalence class of the norm on $M$ is determined
by the underlying $R$-module.
\end{defn}

\begin{remark} \label{R:Noetherian Banach}
Let $R$ be a Banach ring and let $M$ be a finite $R$-module. 
Lemma~\ref{L:finitely generated Banach} equips $M$ with a distinguished equivalence class of seminorms, but $M$ need not be separated or complete under such a seminorm. In fact, by Theorem~\ref{T:open mapping}, $M$ is separated and complete if and only if it is a finite Banach module over $R$. Moreover, in the case when $R$ contains a topologically nilpotent unit, the following
conditions on $R$ are equivalent
(see \cite[Propositions~3.7.3/2, 3.7.3/3]{bgr} for the case where $R$ is a Banach algebra over an analytic field, then modify the arguments using Theorem~\ref{T:open mapping}).
\begin{enumerate}
\item[(a)]
The ring $R$ is noetherian. (Note that it does not suffice to exhibit a dense noetherian subring of $R$; see \cite[Proposition~12]{buzzard-verberkmoes}.)
\item[(b)]
Every ideal of $R$ is closed.
\item[(c)]
The forgetful functor from finite Banach $R$-modules to finite $R$-modules is an
equivalence of categories.
\end{enumerate}
\end{remark}

For Banach rings which are not noetherian, as noted in Remark~\ref{R:Noetherian Banach},
we cannot equip arbitrary finite modules over $R$ with
natural Banach module structures. However, we can do so for finite projective $R$-modules.
\begin{lemma} \label{L:finite projective Banach}
Let $R$ be a Banach ring.
Let $P$ be a finite projective $R$-module. Choose a finite projective $R$-module $Q$ and an isomorphism
$P \oplus Q \cong R^n$ of $R$-modules, for $n$ a suitable nonnegative integer. Equip $R^n$ with the supremum
norm defined by the canonical basis.
\begin{enumerate}
\item[(a)]
The subspace norm on $P$ for the inclusion into $R^n$ is equivalent to the quotient norm for the projection
from $R^n$, and gives $P$ the structure of a finite Banach module over $R$.
\item[(b)]
The equivalence class of the norms described in (a) is independent of the choice of $Q$ and of the presentation
$P \oplus Q \cong R^n$.
\item[(c)]
The above construction defines  a fully faithful functor from finite projective $R$-modules to finite Banach modules over $R$
whose underlying $R$-modules are projective, which is a section of the forgetful functor.
\end{enumerate}
\end{lemma}
\begin{proof}
Let $P', Q'$ be copies of $P,Q$, respectively. Note that the supremum norms $|\cdot|_1$, $|\cdot|_2$
on $P \oplus P' \oplus Q \oplus Q'$
defined by the presentations
\[
(P \oplus Q) \oplus (P' \oplus Q') \cong R^n \oplus R^n, \qquad
(P \oplus Q') \oplus (P' \oplus Q) \cong R^n \oplus R^n
\]
are equivalent by Lemma~\ref{L:finitely generated Banach}.

It is clear that the subspace and quotient norms on $P \oplus Q$ induced by $|\cdot|_1$ are identical,
and that $P \oplus Q$ is complete under these norms.
Consequently, the subspace and quotient norms on $P \oplus Q$ induced by $|\cdot|_2$ are equivalent,
and $P \oplus Q$ is complete under these norms.
Restricting to $P$ yields the subspace and quotient norms induced by the original presentation,
so these two are equivalent.
Moreover, $P$ is the intersection of the closed subspaces $P \oplus Q$ and $P \oplus Q'$ of
$P \oplus P' \oplus Q \oplus Q'$. This proves (a). Parts (b) and (c) follow from
Lemma~\ref{L:finitely generated Banach}.
\end{proof}

\begin{lemma} \label{L:nearby generators}
Let $P$ be a finite projective module over a Banach ring $R$, and choose a norm on $P$ as in
Lemma~\ref{L:finite projective Banach}. Let $\be_1,\dots,\be_n$ be a finite set of generators of $P$
as an $R$-module.
Then there exists $c>0$ such that any $\be'_1,\dots,\be'_n \in P$ with
$|\be'_i - \be_i| < c$ for $i=1,\dots,n$ also form a set of generators of $P$ as an $R$-module.
\end{lemma}
\begin{proof}
The conclusion does not depend on the choice of the norm (only the constant $c$ does),
so we may use the restriction of the supremum norm on $R^n$ along the homomorphism $R^n \to P$
defined by $\be_1,\dots,\be_n$. In this case, the claim is evident with $c=1$, as then the matrix
$A$ defined by $\be'_j = \sum_i A_{ij} \be_i$ satisfies $|A-1| < 1$ and hence is invertible.
\end{proof}

\begin{defn} \label{D:etale algebra Banach structure}
For $A$ a Banach ring and $B \in \FEt(A)$, view $B$ as a finite Banach module over $A$
via Lemma~\ref{L:finite projective Banach}. The multiplication map $\mu: B \otimes_A B \to B$ is then bounded by Lemma~\ref{L:finite projective Banach} again; consequently, we can find an equivalent norm on $B$ which is submultiplicative, and thus view $B$ as a finite Banach algebra over $A$. We will frequently do so without further comment.
\end{defn}

The analogue of a polynomial extension for Banach rings is the following construction.
\begin{defn} \label{D:Gauss norm}
For $r_1, \dots, r_n > 0$,
define the \emph{Tate algebra} over the Banach ring $A$ with radii $r_1,\dots,r_n$ to be the ring
\[
A\{T_1/r_1,\dots,T_n/r_n\} = \left\{f = \sum_I
a_I T^I: a_I \in A, \lim_{I \to \infty} |a_I| r^I = 0 \right \},
\]
where $I = (i_1,\dots,i_n)$ runs over $n$-tuples of nonnegative integers,
$T^I = T_1^{i_1} \cdots T_n^{i_n}$, and $r^I = r_1^{i_1} \cdots r_n^{i_n}$.
(That is, the series in question converge on the closed polydisc
defined by the conditions $|T_i| \leq r_i$ for $i=1,\dots,n$.)
The set $A\{T_1/r_1,\dots,T_n/r_n\}$ is a subring
of $A \llbracket T_1, \dots, T_n \rrbracket$ complete for the
\emph{Gauss norm}
\[
\left| \sum_I a_I T^I \right|_r = \sup_I \{|a_I| r^I\},
\]
which is easily seen to be submultiplicative (resp.\ power-multiplicative, multiplicative)
if the seminorm on $A$ is;
see \cite[Lemma~1.7]{kedlaya-witt}.
In case $r_1 = \cdots = r_n = 1$, we contract the notation to $A\{T_1,\dots,T_n\}$.

A bounded homomorphism $A \to B$ of Banach rings
is \emph{affinoid} if it factors as $A \to A\{T_1,\dots,T_n\} \to B$
for some positive integer $n$ and some strict surjection
$A\{T_1,\dots,T_n\} \to B$.
We also say that $B$ is an \emph{affinoid algebra} over $A$.

We say that $A$ is \emph{strongly noetherian} if every affinoid algebra over $A$ is noetherian, or equivalently the rings $A\{T_1,\dots,T_n\}$ are noetherian for all $n \geq 0$. It appears to be unknown whether every noetherian Banach algebra is strongly noetherian; that is, there is no known analogue of the Hilbert basis theorem for Banach algebras.
\end{defn}

\begin{remark} \label{R:strictly affinoid homomorphism}
In Berkovich's theory, what we call an \emph{affinoid homomorphism} is more commonly called a \emph{strictly affinoid homomorphism}; by contrast, an \emph{affinoid homomorphism} would be allowed to have the form 
$A \to A\{T_1/r_1,\dots,T_n/r_n\} \to B$
for some positive integer $n$, some $r_1,\dots,r_n > 0$, and some strict surjection
$A\{T_1/r_1,\dots,T_n/r_n\} \to B$. This extra generality is important in Berkovich's theory especially when $A$ carries the trivial norm, but is incompatible with Huber's adic constructions.
\end{remark}

\subsection{The Gel'fand spectrum of a Banach ring}
\label{subsec:spectra}

We now introduce one type of topological space corresponding to a Banach ring, as considered by Berkovich \cite{berkovich1, berkovich2}.
\begin{hypothesis}
Throughout \S\ref{subsec:spectra},
let $A$ be a Banach ring with norm denoted by $|\cdot|$. Note that we do not yet impose any extra conditions on $A$, but see Remark~\ref{R:small norm elements}.
\end{hypothesis}

\begin{defn} \label{D:Gelfand}
The \emph{Gel'fand spectrum}
$\calM(A)$ of $A$ is the set of multiplicative seminorms $\alpha$ on $A$
bounded above by $|\cdot|$ (or equivalently, dominated by $|\cdot|$).
We topologize $\calM(A)$
as a closed subspace of the product $\prod_{a \in A} [0, |a|]$;
hence $\calM(A)$ is compact by Tikhonov's theorem
\cite[\S 1.9.5, Th\'eor\`eme~3]{bourbaki-top} (see also
\cite[Theorem~1.2.1]{berkovich1}).
A subbasis of the topology on $\calM(A)$
is given by the sets $\{\alpha \in \calM(A): \alpha(f) \in I\}$
for each $f \in A$ and each open interval $I \subseteq \RR$.
Any bounded homomorphism $\varphi: A \to B$ between Banach rings
induces a continuous map
$\varphi^*: \calM(B) \to \calM(A)$ by restriction.
\end{defn}

\begin{remark} \label{R:not Banach}
One can use Definition~\ref{D:Gelfand} to define the spectrum $\calM(A)$
more generally for any ring $A$ equipped with a submultiplicative seminorm. However,
this will provide no useful additional generality, because the map $A \to \widehat{A}$
always induces a homeomorphism $\calM(\widehat{A}) \to \calM(A)$.
\end{remark}

Berkovich's first main theorem about the spectrum is the following.

\begin{theorem}[Berkovich] \label{T:nonempty spectrum}
For $A$ nonzero, $\calM(A) \neq \emptyset$.
\end{theorem}
\begin{proof}
See \cite[Theorem~1.2.1]{berkovich1}.
\end{proof}
\begin{cor} \label{C:vanishing norm}
For any nontrivial ideal $I$ of $A$, there exists $\alpha \in \calM(A)$ such that
$\alpha(f) = 0$ for all $f \in I$.
\end{cor}
\begin{proof}
Let $J$ be the closure of $I$. By Lemma~\ref{L:closure trivial},
$A/J$ is nonzero, so $\calM(A/J) \neq \emptyset$
by Theorem~\ref{T:nonempty spectrum}. Any element of $\calM(A/J)$ restricts to an
element $\alpha \in \calM(A)$ of the desired form.
(Compare \cite[Corollary~1.2.4]{berkovich1}.)
\end{proof}
\begin{cor} \label{C:ideal from spectrum}
A finite set $f_1,\dots,f_n$ of elements of $A$ generates the unit ideal if and only
if for each $\alpha \in \calM(A)$, there exists an index $i \in \{1,\dots,n\}$
for which $\alpha(f_i) > 0$.
\end{cor}
\begin{proof}
If there exist $u_1,\dots,u_n \in A$ for which $u_1 f_1 + \cdots + u_n f_n = 1$,
then for each $\alpha \in \calM(A)$,
we have $\max_i\{\alpha(u_i) \alpha(f_i)\} \geq 1$ and so $\alpha(f_i) > 0$ for some index $i$.
Conversely, suppose that $f_1,\dots,f_n$ generate a nontrivial ideal $I$;
then by Corollary~\ref{C:vanishing norm}, we can choose $\alpha \in \calM(A)$
such that $\alpha(f) = 0$ for all $f \in I$.
\end{proof}
\begin{cor} \label{C:unit from spectrum}
An element $f \in A$ is a unit if and only if $\alpha(f) > 0$ for all $\alpha \in \calM(A)$.
\end{cor}

\begin{defn} \label{D:residue field}
For $\alpha \in \calM(A)$, define the prime ideal
$\gothp_\alpha = \alpha^{-1}(0)$;
then $\alpha \in \calM(A)$ induces a multiplicative
norm on $A/\gothp_\alpha$. The completion of $\Frac(A/\gothp_\alpha)$
for the unique multiplicative extension of this norm is called the
\emph{residue field} of $\alpha$, and denoted $\calH(\alpha)$.
The image of the map $\calM(A) \to \Spec(A)$ taking $\alpha$ to $\gothp_\alpha$ contains all maximal ideals, by
Corollary~\ref{C:vanishing norm}; see Lemma~\ref{L:finite generation} for some consequences of this
observation.

The \emph{Gel'fand transform} of $A$ is the map
$A \to \prod_{\alpha \in \calM(A)} \calH(\alpha)$; it is bounded
for the supremum norm on the
product (or more precisely, on the subring of the product on which the supremum is finite).
\end{defn}

\begin{remark} \label{R:small norm elements}
Starting in \S\ref{subsec:adic spectrum}, we will require $A$ to contain a
topologically nilpotent unit $z$.
To put this in context, consider the following conditions on $A$.
\begin{enumerate}
\item[(a)]
We may view $A$ as a Banach algebra over some analytic field.
\item[(b)]
There exists a topologically nilpotent unit $z \in A$ such that 
$\left|z \right|_{\spect} \left|z^{-1} \right|_{\spect} = 1$.
We will refer to any such $z$ as a \emph{uniform unit} in $A$.
Note that for any $\alpha \in \calM(A)$,
\[
1 = \alpha(z) \alpha(z^{-1}) \leq \left|z \right|_{\spect} \left|z^{-1} \right|_{\spect} = 1
\]
and so $\alpha(z) = \left|z \right|_{\spect}$. 
\item[(c)]
There exists a topologically nilpotent unit $z \in A$.
\item[(d)]
The ring $A$ is \emph{free of trivial spectrum}: there exists no $\alpha \in \calM(A)$ such that the norm on $\calH(\alpha)$ is trivial, or equivalently (thanks to Corollary~\ref{C:ideal from spectrum}) that the ideal generated by $\gothm_A$ is trivial.
\end{enumerate}
We record the following observations.
\begin{itemize}
\item
These conditions occur in increasingly weaker order: (a) implies (b) implies (c) implies (d). To see that (a) implies (b), note that if $A$ is a Banach algebra over an analytic field $K$, then any $z \in K$ with $\left| z \right| \in (0,1)$ has the desired form.
\item
If $A$ is of characteristic $p$, then conditions (a) and (b) are equivalent:
the existence of a uniform unit $z$ forces $A$ to be a Banach algebra over the analytic field $\Fp((z))$. This will be important in the study of perfectoid algebras, and is ultimately the reason why we can avoid allowing analytic fields to carry the trivial norm.
\item
If $A$ is not of characteristic $p$, then (b) does not imply (a). 
See Remark~\ref{R:power series} (in the case $rs=1$) and Remark~\ref{R:Robba is Tate} for examples in the context of this paper.
\item
Conditions (b) and (c) are not equivalent. 
For example, let $k$ be any field, 
choose $c_1, c_2 \in (0, +\infty)$ which are linearly independent over $\QQ$, normalize the $z_i$-adic norm on $k((z_i))$ by putting $\left| z_i \right| = e^{-c_i}$,
and put $A = k((z_1)) \oplus k((z_2))$ with the supremum norm (which is power-multiplicative). Then $(z_1, z_2)$ is a topologically nilpotent unit, but there is no element $z \in A$ satisfying $\left| z \right|_{\spect} \left| z^{-1} \right|_{\spect} = 1$ and $\left| z \right|_{\spect} \neq 1$.
However, when $A$ is uniform (as in this example) this problem can be remedied by changing the norm without changing the norm topology; see Remark~\ref{R:warp norm}.
\item
We do not know whether (c) and (d) are equivalent.
\end{itemize}
\end{remark}

Berkovich's second main theorem about the spectrum is the following result.
\begin{theorem}[Berkovich] \label{T:transform}
The restriction of the supremum norm on
$\prod_{\alpha \in \calM(A)} \calH(\alpha)$ along
the Gel'fand transform is the spectral seminorm on $A$.
\end{theorem}
\begin{proof}
See \cite[Corollary~1.3.2]{berkovich1}.
\end{proof}

\begin{remark} \label{R:transform}
We collect several remarks about Theorem~\ref{T:transform}.
\begin{enumerate}
\item[(a)]
Theorem~\ref{T:transform} implies Theorem~\ref{T:nonempty spectrum}: if $A$ is nonzero, the spectral seminorm
of $1 \in A$ equals 1.
\item[(b)]
The supremum norm in Theorem~\ref{T:transform} is always achieved if $A$ is nonzero: for each $f \in A$,
the map $f \mapsto \alpha(f)$ is continuous on the compact space $\calM(A)$, and so achieves its maximum.
Consequently, Theorem~\ref{T:transform} may be viewed as a form of the \emph{maximum modulus principle}
in nonarchimedean analytic geometry. For an analogous result in rigid analytic geometry,
see \cite[Proposition~6.2.1/4]{bgr}.
\item[(c)]
It is not generally true that $A$ is complete under its spectral
seminorm even when the latter is a norm; this observation is related to the definition of \emph{uniformization} (see Definition~\ref{D:uniformization}).
One exception is for affinoid algebras over an analytic field; see Corollary~\ref{C:spectral is norm}.
\item[(d)]
For any function $g: \calM(A) \to \RR^+$ whose image is bounded away from 0 and $\infty$, the norm
$\sup\{\alpha^{g(\alpha)}: \alpha \in \calM(A)\}$ defines the same topology on $A$ as the spectral seminorm.
\end{enumerate}
\end{remark}

\begin{lemma} \label{L:finite generation}
A homomorphism $M \to N$ of $A$-modules, with $N$ a finite $A$-module, is surjective if and only if
$M \otimes_A \calH(\alpha) \to N \otimes_A \calH(\alpha)$
is surjective for all $\alpha \in \calM(A)$.
\end{lemma}
\begin{proof}
Suppose that $M \otimes_A \calH(\alpha) \to N \otimes_A \calH(\alpha)$
is surjective for all $\alpha \in \calM(A)$. For each maximal ideal $\gothp$ of $A$,
choose $\alpha \in \calM(A)$ with $\gothp_\alpha = \gothp$.
Then $A/\gothp \to \calH(\alpha)$ is an extension of fields, so
surjectivity of $M \otimes_A \calH(\alpha) \to N \otimes_A \calH(\alpha)$
implies surjectivity of $M \otimes_A A/\gothp \to N \otimes_A A/\gothp$.
This in turn implies surjectivity of $M \otimes_A A_\gothp \to N \otimes_A A_\gothp$
by Nakayama's lemma, and hence surjectivity of $M \to N$.
\end{proof}

\begin{lemma} \label{L:fibre product}
For $A \to B$, $A \to C$ homomorphisms of Banach rings, the map
$\calM(B \widehat{\otimes}_A C) \to \calM(B) \times_{\calM(A)} \calM(C)$
is surjective.
\end{lemma}
\begin{proof}
This reduces to the case where $A,B,C$ are all analytic fields, for which we
may apply Lemma~\ref{L:inject tensor}(a) and Theorem~\ref{T:nonempty spectrum}. See also \cite[Lemma~1.20]{kedlaya-witt}.
\end{proof}

\begin{lemma} \label{L:finite etale surjective spectrum}
For $A$ a Banach ring and $B$ a faithfully finite \'etale $A$-algebra viewed as a Banach algebra over $A$ as per 
Definition~\ref{D:etale algebra Banach structure}, the map $\calM(B) \to \calM(A)$ is surjective. (It is also open; see Lemma~\ref{L:henselian local ring}(c).)
\end{lemma}
\begin{proof}
The hypothesis on $B$ ensures that for each $\alpha \in \calM(A)$, $B \otimes_A \calH(\alpha) = B \widehat{\otimes}_A \calH(\alpha)$ is a nonzero direct sum of finite extensions of $\calH(\alpha)$, and so $\calM(B \otimes_A \calH(\alpha))$ is nonempty.
This proves the claim.
\end{proof}

\begin{remark} \label{R:compact spaces}
When studying spectra, it is helpful to use general facts about compact topological spaces.
Here are a few that we will need.
\begin{enumerate}
\item[(a)]
The image of a quasicompact topological space under a continuous map is quasicompact
\cite[\S I.9.4, Th\'eor\`eme~2]{bourbaki-top}.
Consequently, any continuous map $f: Y \to X$ from a quasicompact topological space to a Hausdorff topological space
is closed \cite[\S I.9.4, Corollaire~2]{bourbaki-top}.
\item[(b)]
With notation as in (a), if $V$ is open in $Y$, then $W = X \setminus f(Y \setminus V)$ is open.
One consequence is that if $Z$ is closed in $X$ and $V$ is an open neighborhood of $f^{-1}(Z)$,
then $W$ is an open neighborhood of $Z$ and $f^{-1}(W) \subseteq V$.
Another consequence is that the quotient and subspace topologies on $\image(f)$ coincide:
if $U \subseteq \image(f)$ and $V = f^{-1}(U)$ is open in $Y$, then $U = \image(f) \cap W$ is open in $\image(f)$.
That is, any continuous surjection (resp.\ bijection)
from a quasicompact space to a Hausdorff space is a quotient map (resp.\ a homeomorphism).
\item[(c)]
If $X$ is the inverse limit of an inverse system $\{X_i\}_{i \in I}$ of nonempty compact spaces,
then $X$ is nonempty and compact. This follows from Tikhonov's theorem,
or see \cite[\S 1.9.6, Proposition~8]{bourbaki-top}.
As a corollary, for $i \in I$ and $Z$ a closed subset of $X_i$,
$Z$ has empty inverse image in $X$ if and only if there exists an index
$j \geq i$ for which $Z$ has empty inverse image in $X_j$.
\item[(d)]
With notation as in (c),
for any $i \in I$ and any open subsets $V_{1,i},\dots,V_{n,i}$ of $X_i$ whose inverse images
in $X$ form a covering, there exists an index $j \geq i$ for which the inverse images
$V_{1,j},\dots,V_{n,j}$ of $V_{1,i},\dots,V_{n,i}$ in $X_j$ form a covering of $X_j$ itself:
apply (c) to the closed set $X_i \setminus (V_{1,i} \cup \dots \cup V_{n,i})$.
As a corollary,
any finite open covering of $X$ is refined by the pullback of a finite open covering of some $X_i$.
\item[(e)]
With notation as in (c),
any disconnection of $X$ (i.e., any partition of $X$ into two disjoint closed-open subsets $U_1, U_2$)
is the inverse image of a disconnection of some $X_j$, by the following argument.
Choose any $i \in I$. By (a),
the images $V_{1,i}, V_{2,i}$ of $U_1, U_2$ in $X_i$ are closed and disjoint; they may thus be covered
by disjoint open neighborhoods $W_{1,i}, W_{2,i}$. By (d), we can find an index $j \geq i$ such that $X_j$ is covered by the inverse images $W_{1,j}, W_{2,j}$ of $W_{1,i}, W_{2,i}$ in $X_j$. Since $W_{1,j}, W_{2,j}$ are open and disjoint, they
form a disconnection of $X_j$ which pulls back to
the given disconnection of $X$.

\end{enumerate}
\end{remark}

\subsection{The adic spectrum of an adic Banach ring}
\label{subsec:adic spectrum}

We next introduce a second type of topological space corresponding to a Banach ring, as considered by Huber \cite{huber1, huber2, huber}. The natural levels of generality of the  Berkovich and Huber constructions are incompatible; we work at reduced levels of generality where the two constructions can be compared.
We begin with the base algebraic objects of Huber's construction.

Before proceeding, we recall that from now on, we only consider Banach rings containing a topologically nilpotent unit (see Remark~\ref{R:small norm elements}).

\begin{defn}
For $A$ a ring equipped with a submultiplicative norm, let $A^{\circ}$ denote the subring of power-bounded elements of $A$. Note that $A^\circ \neq \gotho_A$ in general unless the norm on $A$ is power-multiplicative.

An \emph{adic Banach ring} is a pair $(A, A^+)$ in which $A$ is a Banach ring (which from now on must be a Banach algebra containing a topologically nilpotent unit) and $A^+$ is a subring of $A^\circ$ which is open and integrally closed in $A$. These conditions ensure that every topologically nilpotent element of $A$ must belong to $A^+$. 

A \emph{morphism} of adic Banach rings $(A,A^+) \to (B,B^+)$ is a bounded homomorphism $\varphi: A \to B$ of Banach rings such that $\varphi(A^+) \subseteq B^+$. With this definition, the correspondence $A \mapsto (A, A^\circ)$ defines a functor from the category of Banach rings to the category of adic Banach rings.

For $(A,A^+) \to (B, B^+), (A,A^+) \to (C,C^+)$ two morphisms of adic Banach rings,
their coproduct in the category of adic Banach rings will be denoted 
by $(B,B^+) \widehat{\otimes}_{(A,A^+)} (C,C^+)$. It consists of $(D,D^+)$ where $D = B \widehat{\otimes}_A C$ and $D^+$ is the completion of
the integral closure of $B^+ \otimes_{A^+} C^+$ in $D$.
\end{defn}

\begin{remark} \label{R:etale adic extension}
For $(A,A^+)$ an adic Banach ring and $B \in \FEt(A)$, view $B$ as a finite Banach $A$-algebra as in Definition~\ref{D:etale algebra Banach structure}. Then let $B^+$ be the integral closure of $A^+$ in $B$; in this way, we obtain a morphism $(A,A^+) \to (B,B^+)$. This construction will be used in the definition of \'etale morphisms on adic spaces in \S\ref{sec:adic}.
\end{remark}

We now associate topological spaces to adic Banach rings.
We will discuss the special topological properties of these spaces in more detail in 
\S\ref{sec:adic}.

\begin{defn}
Let $\Gamma$ be a totally ordered abelian group, and let $\Gamma_0$ be the pointed monoid $\Gamma \cup \{0\}$ with $0 \cdot \Gamma_0 = 0$ ordered so that $0 < \gamma$ for all $\gamma \in \Gamma$.
A \emph{semivaluation} on a ring $A$ with values in $\Gamma$ is a function $v: A \to \Gamma_0$ satisfying the following conditions.
\begin{enumerate}
\item[(a)]
For all $a,b \in A$, we have $v(a-b) \leq \max\{v(a), v(b)\}$.
\item[(b)]
For all $a,b \in A$, we have $v(ab) = v(a) v(b)$.
\item[(c)]
We have $v(0) = 0$ and $v(1) = 1$. If moreover $v^{-1}(0) = \{0\}$, we say that $v$ is a \emph{valuation}.
\end{enumerate}
For example, if $\Gamma = \RR^+$, then a (semi)valuation is the same as a multiplicative (semi)norm.

For $A$ a Banach ring, we declare two semivaluations on $A$ (possibly valued in different ordered groups) to be \emph{equivalent} if they define the same order relation on $A$. It is clear that this defines an equivalence relation and that the equivalence classes form a set (rather than a larger class). Denote the latter set by $\Spv(A)$. For linguistic convenience, we identify each equivalence class in $\Spv(A)$ with a particular representative in an arbitrary but fixed manner.

A semivaluation $v$ on $A$ is \emph{continuous} if for every nonzero $\gamma$ in the value group of $v$ (i.e., the subgroup of $\Gamma$ generated by the nonzero images of $v$) there is a neighborhood $U$ of $0$ in $A$ such that $v(u) < \gamma$ for all $u \in U$.

The \emph{adic spectrum} of $(A,A^+)$ is the subset $\Spa(A,A^+)$ of $\Spv(A)$ consisting of the equivalence classes of continuous semivaluations on $A$ bounded by 1 on $A^+$.
Since $A^+$ is integrally closed, we have the following equality analogous to
Theorem~\ref{T:transform}:
\begin{equation} \label{eq:adic transform}
A^{+} = \{x \in A: v(x) \leq 1 \quad (v \in \Spa(A,A^+)\}.
\end{equation}
(See \cite[Proposition~1.6]{huber2} for more details.)
We equip $\Spa(A,A^+)$ with the topology generated by sets of the form
\[
\{v \in \Spv(A,A^+): v(a) \leq v(b) \neq 0 \} \qquad (a,b \in A).
\]
A \emph{rational subspace} of $\Spa(A,A^+)$ is one of the form
\begin{equation} \label{eq:adic rational subspace}
\{v \in \Spa(A,A^+): v(f_i) \leq v(g) \neq 0 \quad (i=1,\dots,n)\}
\end{equation}
for some $f_1,\dots,f_n, g \in A$ generating the unit ideal.
One gets the same definition if one only requires that $f_1,\dots,f_n$ generate the unit ideal, since it is harmless to append $g$ as an extra generator.
One may also drop the condition $v(g) \neq 0$; see Remark~\ref{R:approximate rational}.

Note that any morphism $\psi: (A,A^+) \to (B,B^+)$ induces a continuous map $\psi^*: \Spa(B,B^+) \to \Spa(A,A^+)$. Under this map, the inverse image of any rational subspace is again a rational subspace.
\end{defn}

\begin{remark} \label{R:Tate f-adic}
Huber's definition of $\Spa(A,A^+)$ applies to more general topological rings than Banach rings. Namely, Huber defines an \emph{f-adic ring} to be a topological ring $A$ containing an open subring $A_0$ which is adic with a finitely generated ideal of definition (called a \emph{ring of definition} of $A$). He then says that $A$ is \emph{Tate} if it contains a topologically nilpotent unit $z$. In this case, for any $c \in (0,1)$, the norm
\[
\alpha(x) = \inf\{c^n: n \in \ZZ, z^{-n} x \in A_0\}
\]
gives $A$ the structure of a Banach ring. 
One can show in addition (see \cite[\S 1]{huber1})
that the category of Banach rings containing topologically nilpotent units is equivalent to the category of Tate f-adic rings, except that one must allow morphisms of Banach rings which are continuous but not necessarily bounded
(e.g., see Remark~\ref{R:warp norm}).

For $z$ a topologically nilpotent unit in $A$, 
a semivaluation $v$ on $A$ bounded by 1 on $A^+$ is continuous if and only if
for every $x \in A$ with $v(x) \neq 0$, there exists $n \in \ZZ$ with $v(z^n) < v(x)$.
In case $A$ is a Banach ring over some analytic field $K$,
one may say additionally that a semivaluation on $A$ bounded by 1 on $A^+$ is continuous if and only if its restriction to $K$ is equivalent to the norm on $K$.
This implies that any continuous morphism of Banach rings over $K$ is bounded;
it also arises
in the comparison with Gel'fand spectra in Definition~\ref{D:Berkovich to Huber}.
\end{remark}

For the remainder of \S\ref{subsec:adic spectrum}, let $(A,A^+)$ be any adic Banach ring. The analogue of the compactness of the Gel'fand spectrum is the following result.
A more precise statement is that adic spectra are \emph{spectral spaces}; see
\S\ref{subsec:adic topological properties}.

\begin{theorem}[Huber] \label{T:spectra are spectral1}
The space $\Spa(A,A^+)$ is quasicompact and the rational subspaces form a topological basis consisting of quasicompact open subsets.
\end{theorem}
\begin{proof}
See \cite[Theorem~3.5(i,ii)]{huber1}.
\end{proof}

We now relate this construction back to the Gel'fand spectrum.
\begin{defn} \label{D:Berkovich to Huber}
There is a natural map $\calM(A) \to \Spa(A,A^+)$ taking each 
$\alpha \in \calM(A)$ to the equivalence class of $\alpha$ as a semivaluation.
Beware that this map is not continuous.

Now suppose that $A$ contains a uniform unit $z$;
then there is a map $\Spa(A,A^+) \to \calM(A)$ defined as follows.
Given a semivaluation $v \in \Spa(A,A^+)$, define the multiplicative seminorm $\alpha = \alpha(v) \in \calM(A)$ by the formula
\[
\alpha(x) = \inf\{\left| z \right|_{\spect}^{r/s}: r \in \ZZ, s \in \ZZ_{>0}, v(z^r) > v(x^s) \}.
\]
The composition $\calM(A) \to \Spa(A,A^+) \to \calM(A)$ is the identity.
In particular, the map $\calM(A) \to \Spa(A,A^+)$ is injective,
and by Theorem~\ref{T:nonempty spectrum}, $\Spa(A,A^+) \neq \emptyset$ whenever $A \neq 0$. 
\end{defn}

\begin{remark} \label{R:approximate rational}
Given a rational subspace $U$ of $\Spa(A,A^+)$ as in \eqref{eq:adic rational subspace},
let $\overline{U}$ be the image of $U$ under the projection $\Spa(A,A^+) \to \calM(A)$;
since $U$ is quasicompact, $\overline{U}$ is compact by Remark~\ref{R:compact spaces}(a).
One has $\alpha(g) > 0$ for all $\alpha \in \overline{U}$, so by compactness
$c = \inf\{\alpha(g): \alpha \in \overline{U} \}$ is positive. For $0 < \epsilon <c$, any $f'_1,\dots,f'_n,g' \in A$
satisfying $|f'_i-f_i| < \epsilon, |g'-g| <\epsilon$ generate the unit ideal
and satisfy
\[
U = \{v \in \Spa(A,A^+): v(f'_i) \leq v (g')
\quad (i=1,\dots,n)\}.
\]
Compare \cite[Remark~1.15]{kedlaya-witt},
\cite[Proposition~7.2.4/1]{bgr}.
\end{remark}

\begin{defn} \label{D:Berkovich rational subspace}
We define a \emph{rational subspace} of $\calM(A)$ as the intersection of $\calM(A)$ with a rational subspace of $\Spa(A,A^+)$.
For the rational subspace $U$ of $\Spa(A,A^+)$ defined in \eqref{eq:adic rational subspace},
the corresponding rational subspace of $\calM(A)$ is
\begin{equation} \label{eq:Berkovich rational subspace}
\{\alpha \in \calM(A):\alpha(f_i) \leq  \alpha(g) \quad (i=1,\dots,n)\}
\end{equation}
and the image of $U$ in $\calM(A)$ is equal to the intersection $U \cap \calM(A)$.
As a corollary, we see that every nonempty rational subspace of $\Spa(A,A^+)$ meets $\calM(A)$, so $\calM(A)$ is dense in $\Spa(A,A^+)$. (See however Remark~\ref{R:Berkovich rational subspace} below.)

Rational subspaces of $\calM(A)$ are closed, not open; as a result, not every rational subspace containing some $\alpha \in \calM(A)$ is a neighborhood of $\alpha$. However, those which are neighborhoods form a neighborhood basis of $\alpha$ in $\calM(A)$; we say that such rational subspaces \emph{encircle} $\alpha$. 

Now assume $A$ contains a uniform unit.
By the previous paragraph, $\Spa(A,A^+) \to \calM(A)$ is continuous
and hence a quotient map by Remark~\ref{R:compact spaces}(b). In fact,
$\calM(A)$ is the maximal Hausdorff quotient of $\Spa(A,A^+)$: for any continuous map $\Spa(A,A^+) \to U$ with $U$ Hausdorff, any $v \in \Spa(A,A^+)$ projecting to $\alpha \in \calM(A)$ is a specialization of $\alpha$, so $v$ and $\alpha$ must have the same image in $U$. An immediate consequence is that disconnections of $\Spa(A,A^+)$ and $\calM(A)$ correspond, since they define maps to a two-point space.
\end{defn}

\begin{remark} \label{R:Berkovich rational subspace}
A rational subspace of $\Spa(A,A^+)$ need not be determined by its intersection with $\calM(A)$ except in some restricted circumstances (e.g., see Corollary~\ref{C:naive rational covering}).
A typical example is $(A,A^+) = (K\{T\}, \gotho_K + \gothm_K \cdot \gotho_K\{T\})$ for $K$ an analytic field: for
\begin{align*}
U &= \{v \in \Spa(A,A^+): v(T) \geq 1 \} \\
V &= \{v \in \Spa(A,A^+): v(T) \leq 1 \},
\end{align*}
one has $U \cap \calM(A) = U \cap V \cap \calM(A)$ but $U \not\subseteq V$ in general. See
\cite[Example~2.20]{scholze1} for a pictorial representation of this example.
\end{remark}

\begin{remark}  \label{R:Berkovich rational subspaces}
As per Remark~\ref{R:strictly affinoid homomorphism}, what we are calling a \emph{rational subspace} of $\calM(A)$ would be called a \emph{strictly rational subspace} in Berkovich's setup. An arbitrary rational subspace would have the form
\[
\{\alpha \in \calM(A): \alpha(f_1) \leq p_1 \alpha(g), \dots, \alpha(f_n) \leq p_n \alpha(g)\}
\]
for some $f_1,\dots,f_n,g \in A$ generating the unit ideal and some $p_1,\dots,p_n > 0$; such subspaces are needed to obtain a neighborhood basis when $A$ is not required to contain a topologically nilpotent unit.
\end{remark}

The analogue of the residue field $\calH(\alpha)$ of a point $\alpha$ in a Gel'fand spectrum is the following construction.
\begin{defn} \label{D:adic field}
An \emph{adic field} is an adic Banach ring $(K, K^+)$ in which $K$ is an analytic field and $K^+$ is a valuation ring in $K$ (i.e., a subring containing either $x$ or $1/x$ for each $x \in K^\times$). The space $\Spa(K,K^+)$ is not a point unless $K^+ = \gotho_K$; however, the valuation corresponding to $K^+$ defines the generic point of $\Spa(K,K^+)$.

Given $v \in \Spa(A,A^+)$, 
let $(\calH(v), \calH(v)^+)$ be the adic field with $\calH(v) = \calH(\alpha(v))$
and $\calH(v)^+$ equal to the valuation ring of the continuous multiplicative extension of $v$ to $\calH(\alpha(v))$. 
By construction, there is a canonical morphism $(A,A^+) \to (\calH(v), \calH(v)^+)$
under which the generic point of $\Spa(\calH(v), \calH(v)^+)$ maps to $v$.
\end{defn}

\begin{defn}
Let $U$ be a quasicompact open subset of $\Spa(A,A^+)$.
We say that $U$ is an \emph{affinoid subdomain} of $\Spa(A,A^+)$ if there exists an affinoid homomorphism $\varphi: (A,A^+) \to (B,B^+)$ which is initial among morphisms
$\psi: (A,A^+) \to (C,C^+)$ of adic Banach rings for which 
$\psi^*(\Spa(C,C^+)) \subseteq U$. We refer to the representing morphism $(A,A^+) \to (B,B^+)$ as an \emph{affinoid localization}.
 
In general, the structure of affinoid subdomains is quite mysterious
(see Theorem~\ref{T:Gerritzen-Grauert} for an exception). 
However, every rational subspace $U$ is an affinoid subdomain and the map 
$\Spa(B,B^+) \cong U$ is a homeomorphism
(see Lemma~\ref{L:rational subdomain} below). We thus refer to $U$ also as a \emph{rational subdomain} and to the corresponding affinoid localization also as a
\emph{rational localization}.

Note that the completed tensor product of two affinoid (resp.\ rational) localizations is again such a localization, corresponding to the intersection of affinoid (resp.\ rational) subdomains.
\end{defn}

\begin{lemma} \label{L:rational subdomain}
Let $U$ be a rational subspace of $\Spa(A,A^+)$ defined as in \eqref{eq:adic rational subspace}.
\begin{enumerate}
\item[(a)]
The subspace $U$ is an affinoid subdomain represented by $\varphi: (A,A^+) \to (B,B^+)$, where $B$ is the quotient of $A\{T_1,\dots,T_n\}$ for the closure of the ideal $(gT_1 - f_1,\dots,gT_n-f_n)$, equipped with the quotient norm,
and $B^+$ is the completion of the integral closure of the image of $A^+[T_1,\dots,T_n]$ in $B$.
\item[(b)]
The map $\varphi^*: \Spa(B,B^+) \to \Spa(A,A^+)$ induces a homeomorphism $\Spa(B,B^+) \cong U$. More precisely, the rational subspaces of $\Spa(B,B^+)$ correspond to the rational subspaces of $\Spa(A,A^+)$ contained in $U$.
\end{enumerate}
\end{lemma}
\begin{proof}
For (a), see \cite[Proposition~1.3]{huber2}. To check (b),
note that $\varphi^*$ by definition gives a continuous map from 
$\Spa(B,B^+)$ to $U$.
To see that the map is bijective,
choose any $v \in U$.
The map $(A,A^+) \to (\calH(v), \calH(v)^+)$ factors uniquely through a bounded homomorphism
$(B,B^+) \to (\calH(v), \calH(v)^+)$;
the generic point of $(\calH(v), \calH(v)^+)$ maps to the unique point of $\Spa(B,B^+)$ in the preimage of $v$.

To see that the induced morphism $\Spa(B,B^+) \to U$ is a homeomorphism,
it suffices to check the final assertion, i.e.,
that any rational subspace of $\Spa(B,B^+)$ is also a rational subspace of $\Spa(A,A^+)$.
This follows from Remark~\ref{R:approximate rational}: any rational subspace of
$\Spa(B,B^+)$ can be described using generators in $A[f_1/g,\dots,f_n/g]$, and such a description can be translated into a description using generators in $A$.
(See also \cite[Theorem~7.2.4/2]{bgr} and \cite[Lemma~1.5]{huber1}.)
\end{proof}

To obtain building blocks for the theory of adic spaces, we must define structure sheaves on adic Banach rings. We postpone the globalization step until \S\ref{sec:adic}.
\begin{defn} \label{D:rational covering}
By a \emph{rational covering} (resp.\ \emph{affinoid covering}) of $\Spa(A,A^+)$,
we will mean either a finite collection $\{U_i\}_i$ of rational (resp.\ affinoid) subdomains of $\Spa(A,A^+)$ forming a set-theoretic covering, or the corresponding collection $\{\Spa(A,A^+) \to \Spa(B_i,B_i^+)\}_i$ of rational (resp.\ affinoid) localizations, depending on context. 

Note that a rational covering of $\Spa(A,A^+)$ induces a set-theoretic covering of $\calM(A)$ by rational subspaces, but not conversely in general. However, a finite collection of rational subspaces whose relative interiors cover $\calM(A)$ does induce a rational covering of $\Spa(A,A^+)$; we call such a covering a \emph{strong rational covering} of $\Spa(A,A^+)$ (or of $\calM(A)$).
\end{defn}

\begin{defn} \label{D:adic structure presheaf}
Define the \emph{structure presheaf} $\calO$ on $\Spa(A,A^+)$ as the functor taking each open subset $U$ to the inverse limit of $B$ over all rational localizations $(A,A^+) \to (B,B^+)$ for which $\Spa(B,B^+) \subseteq U$.
In particular, for any rational localization $(A,A^+) \to (B,B^+)$, we have $\Gamma(\Spa(B,B^+), \calO) = B$.
The stalks of $\calO$ are henselian local rings (see Lemma~\ref{L:henselian local ring} below).

We say that $(A,A^+)$ is \emph{sheafy} if the structure presheaf is a sheaf; an equivalent condition (e.g., see Proposition~\ref{P:acyclicity template}) is that for any rational localization $(A,A^+) \to (B,B^+)$, the map $B \to H^0(\Spa(B,B^+), \calO)$ is an isomorphism.  In this case, $(\Spa(A,A^+), \calO)$ is a locally ringed space.
This is not true in general; see \cite[\S 1]{huber1} or \cite[\S 4.1]{buzzard-verberkmoes} for failures of injectivity,
and Example~\ref{exa:nonuniform rational localization} for a failure of surjectivity.
\end{defn}

\begin{prop}[Huber] \label{P:strongly noetherian}
Let $(A,A^+)$ be an adic Banach ring such that $A$ is strongly noetherian
(see Definition~\ref{D:Gauss norm}).
Then $(A,A^+)$ is sheafy.
\end{prop}
\begin{proof}
See \cite[Theorem~2.2]{huber2}.
\end{proof}

\begin{lemma} \label{L:henselian local ring}
The following statements hold.
\begin{enumerate}
\item[(a)]
For $v \in \Spa(A,A^+)$, the stalk $\calO_v$ is a henselian local ring whose residue field is dense in $\calH(v)$.
\item[(b)]
For $\alpha \in \calM(A)$, let $A_\alpha$ be the direct limit of $B$ over all rational localizations $(A,A^+) \to (B,B^+)$ encircling $\alpha$. (We call this ring the \emph{Hausdorff localization} at $\alpha$ to distinguish it from the stalk $\calO_\alpha$.) Then $A_\alpha$ is a henselian local ring
whose residue field is dense in $\calH(\alpha)$.
\item[(c)]
With notation as in Lemma~\ref{L:finite etale surjective spectrum}, the map
$\calM(B) \to \calM(A)$ is open.
\end{enumerate}
\end{lemma}
\begin{proof}
In both (a) and (b), the local property follows from Corollary~\ref{C:unit from spectrum}
and the henselian property follows from Lemma~\ref{L:henselian direct limit}(a).
To check (c), we may work locally around $\alpha \in \calM(A)$; by (b) and Theorem~\ref{T:henselian}, we reduce to the case where $\calM(B)$ contains a unique point $\beta$ lifting $\alpha$, $\calH(\beta)$ is a Galois extension of $\calH(\alpha)$ with group $G$, and $G$ acts on $B$. In this case, for any open subset $V$ of $\calM(B)$ with image $U$ in $\calM(A)$, the inverse image of $U$ in $\calM(B)$ is the open set $\cup_{g \in G} g(V)$; since $\calM(B) \to \calM(A)$ is a quotient map by Remark~\ref{R:compact spaces}(b), $U$ is open.
\end{proof}

We conclude this section by introducing the key formal arguments in the proofs of the theorems of Tate and Kiehl (Theorem~\ref{T:Tate-Kiehl}), which allow us to reduce certain questions about coverings (namely sheaf, acyclicity, and glueing properties) to coverings of a very simple form.

\begin{defn} \label{D:standard Laurent}
For $f_1,\dots,f_n \in A$ generating the unit ideal, the \emph{standard rational covering} of $\Spa(A,A^+)$ generated by $f_1,\dots,f_n$ is the covering by the rational subspaces
\[
U_i = \{v \in \Spa(A,A^+): v(f_j) \leq v(f_i) \quad (j=1,\dots,n)\} \qquad (i=1,\dots,n).
\]
For $f_1,\dots,f_n \in A$ arbitrary, the \emph{standard Laurent covering} generated by $f_1,\dots,f_n$ is the covering by the rational subspaces
\[
S_e = \bigcap_{i=1}^n S_{i,e_i} \qquad (e = (e_1,\dots,e_n) \in \{-, +\}^n),
\]
where
\[
S_{i,-} =\{v \in \Spa(A,A^+): v(f_i) \leq 1\},
\qquad
S_{i,+} =\{v \in \Spa(A,A^+): v(f_i) \geq 1\}.
\]
A standard Laurent covering with $n=1$ is also called a
\emph{simple Laurent covering}. 
\end{defn}

We will use the following observations.
\begin{lemma} \label{L:rational and Laurent coverings}
The following statements hold.
\begin{enumerate}
\item[(a)]
Any rational covering can be refined by a standard rational covering. 
\item[(b)]
For any standard rational covering $\gothU$ of $X$, there exists a standard Laurent covering $\gothV$ of $X$ such that for each $V = \Spa(B,B^+) \in \gothV$, the restriction of $\gothU$ to $V$ (omitting empty intersections) is a standard rational covering generated by units in $B$.
\item[(c)]
Any standard rational covering generated by units can be refined by a standard Laurent covering generated by units.
\end{enumerate}
\end{lemma}
\begin{proof}
To prove (a), we follow \cite[Lemma~8.2.2/2]{bgr}. Given a rational covering $U_1,\dots,U_n$ where $U_i$ is generated by the parameter set $S_i = \{f_{i1},\dots,f_{in_i},g_i\}$, let $S$ be the set of products of the form $s_1\cdots s_n$ where $s_i \in S_i$ for all $i$. Let $S'$ be the subset of $S$ consisting of products $s_1 \cdots s_n$ for which $s_i = g_i$ for at least one $i$. 
Note that $S'$ generates the unit ideal: for any $v \in \Spa(A,A^+)$, for each $i$ we can find $s_i \in S_i$ not vanishing at $v$, taking $s_i = g_i$ for any $i$ for which $v \in U_i$. Thus the parameter set $S'$ defines a standard rational covering. To see that this refines the original covering, note that the rational subspace with final parameter $s_1\cdots s_n$ does not change if we add $S \setminus S'$ to the set of parameters (again because the $U_i$ form a covering), which makes it clear that this subspace is contained in $U_i$ for any index $i$ for which $s_i = g_i$ (because we now have parameters obtained from $s_1,\dots,s_n$ by replacing $s_i$ with each of the other elements of $S_i$).

To prove (b), we follow \cite[Lemma~8.2.2/3]{bgr}.
Let $\gothU$ be the standard rational covering defined by the parameters $f_1,\dots,f_n$.
We argue as in Remark~\ref{R:approximate rational}:
since $f_1,\dots,f_n$ generate the unit ideal, by Corollary~\ref{C:ideal from spectrum}
the quantity
\[
c = \inf\{\max_i\{\alpha(f_i)\}: \alpha \in \calM(A)\}
\]
is positive. Since $A$ contains a topologically nilpotent unit, we may rescale $f_1,\dots,f_n$ to reduce to the case $c> 1$. In this case, the standard Laurent covering $\gothV$ defined by $f_1,\dots, f_n$ has the desired property: on the subspace where $\left| f_1 \right|,\dots,\left| f_s \right| \leq 1$ and
$\left| f_{s+1} \right|,\dots,\left| f_{n} \right| \geq 1$, the restriction of $\gothU$ is the standard rational covering generated by $f_{s+1},\dots,f_n$ plus some empty intersections.

To prove (c), we follow \cite[Lemma~8.2.2/4]{bgr}. Consider the standard rational covering generated by the units $f_1,\dots,f_n$. This cover is refined by the standard Laurent covering generated by $f_i f_j^{-1}$ for $1 \leq i < j \leq n$, by an elementary combinatorics argument (any total ordering on a finite set has a maximal element). 
\end{proof}

Using these observations, we obtain the following criterion.
\begin{prop} \label{P:Tate reduction}
Let $\calP$ be a property of rational coverings of rational subdomains of $\Spa(A,A^+)$ satisfying the following conditions.
\begin{enumerate}
\item[(a)]
The property $\calP$ is local: if it holds for a refinement of a given covering, it also holds for the original covering.
\item[(b)]
The property $\calP$ is transitive: if it holds for a covering $\{(B,B^+) \to (C_i, C_i^+)\}_i$ and for 
some coverings $\{(C_i,C_i^+) \to (D_{ij},D_{ij}^+)\}_j$ for each $i$, then it holds for the composite covering $\{(B,B^+) \to (D_{ij}, D_{ij}^+)\}_{i,j}$.
\item[(c)]
The property $\calP$ holds for any simple Laurent covering.
\end{enumerate}
Then the property $\calP$ holds for any rational covering of any rational subdomain
of $\Spa(A,A^+)$.
\end{prop}
\begin{proof}
We make the following observations.
\begin{enumerate}
\item[(i)]
We may deduce $\calP$ for any standard Laurent covering generated by units by writing it as a composition of simple Laurent coverings generated by units, then invoking (b) and (c).
\item[(ii)]
We may deduce $\calP$ for any standard rational covering generated by units by applying  Lemma~\ref{L:rational and Laurent coverings}(c)
to refine the covering by a standard Laurent covering generated by units, then invoking (a) and (i).
\item[(iii)]
Given a standard rational covering $\{(B,B^+) \to (C_i, C_i^+)\}_i$, 
using Lemma~\ref{L:rational and Laurent coverings}(b)
we obtain a standard Laurent covering $\{(B,B^+) \to (D_j, D_j^+)\}_j$ such that for each $j$, the covering $\{(D_j, D_j^+) \to (D_j, D_j^+) \widehat{\otimes}_{(B,B^+)} (C_i, C_i^+)\}_i$ is a standard rational covering generated by units in $D_j$.
We may thus deduce $\calP$ for the covering $\{(B,B^+) \to (D_j, D_j^+) \widehat{\otimes}_{(B,B^+)} (C_i, C_i^+)\}_{i,j}$
by invoking (ii) and (b), and then deduce $\calP$ for the original covering by invoking (a).
\item[(iv)]
We may deduce $\calP$ for any covering by applying
Lemma~\ref{L:rational and Laurent coverings}(b) to refine the covering by a standard rational covering, then invoking (i).
\end{enumerate}
These observations prove the claim.
\end{proof}

We will apply Proposition~\ref{P:Tate reduction} to two general purposes: 
construction of acyclic sheaves of rings, and comparison of certain modules over such sheaves with their global sections. In the latter case, the conditions of Proposition~\ref{P:Tate reduction} will be easy to verify directly. In the former case, one must be slightly more careful; we  package an extra argument into the following proposition modeled on \cite[Proposition~8.2.2/5]{bgr}.

\begin{prop} \label{P:acyclicity template}
Let $\calF$ be a presheaf of abelian groups on $\Spa(A,A^+)$.
Suppose that for every rational subdomain $U = \Spa(B,B^+)$ of $\Spa(A,A^+)$ and every simple Laurent covering $V_1, V_2$ of $U$, we have 
\begin{equation} \label{eq:acyclicity template}
\check{H}^0(U, \calF; \gothV) = \calF(U), 
\quad \text{resp.} \quad
\check{H}^i(U, \calF; \{V_1,V_2\}) = \begin{cases} \calF(U) & i=0 \\ 0 & i=1.
\end{cases}
\end{equation}
Then for every rational subdomain $U$ of $\Spa(A,A^+)$ and every rational covering $\gothV$ of $U$,
\[
H^0(U, \calF) = \check{H}^0(U, \calF; \gothV) = \calF(U), 
\quad \text{resp.} \quad
H^i(U, \calF) = \check{H}^i(U, \calF; \gothV) = \begin{cases} \calF(U) & i=0 \\ 0 & i>0. \end{cases}
\]
\end{prop}
\begin{proof}
Throughout this argument, let $U$ be an arbitrary rational subdomain of $\Spa(A,A^+)$ and let $\gothV$ be a rational covering of $U$. We identify a series of properties of $\gothV$ which satisfy the criteria of Proposition~\ref{P:Tate reduction}, and hence hold for all $U$ and $\gothV$.

First, the property that $\calF(U) \to \check{H}^0(U, \calF; \gothV)$ is injective satisfies (a) and (b) formally and (c) by \eqref{eq:acyclicity template}.

Next, the property that $\calF(U) \to \check{H}^0(U, \calF; \gothV)$ is bijective satisfies (b) formally and (c) by \eqref{eq:acyclicity template}.
To check (a), let $\gothV'$ be a refinement of $\gothV$ such that $\calF(U) \to \check{H}^0(U, \calF; \gothV')$ is bijective.
The map $\check{H}^0(U, \calF; \gothV) \to \check{H}^0(U, \calF; \gothV')$ is then surjective, but it is also injective by the previous paragraph (applied to each element of $\gothV$).

From now on, assume we are in the second situation.
Next, we say that $\gothV$ is \emph{\v{C}ech-acyclic} if $\check{H}^i(U, \calF; \gothV) = 0$ for all $i>0$.
This property satisfies criteria (b) formally and (c) by \eqref{eq:acyclicity template}, but not (a).

Instead, we say that $\gothV$ is \emph{universally \v{C}ech-acyclic} if its pullback to any rational subdomain of $U$ is \v{C}ech-acyclic.
This property formally also satisfies criteria (b) and (c) of Proposition~\ref{P:Tate reduction}. However, it also satisfies (a) by a spectral sequence argument; see \cite[Corollary~8.1.4/3]{bgr}.

We thus deduce that every rational covering of every rational subdomain is \v{C}ech-acyclic. Acyclicity for sheaf cohomology then follows by a standard homological algebra argument
(see \cite[Tag~01EW]{stacks-project}).
\end{proof}

As a first application of this argument, we have the following result which asserts that  for an arbitrary adic Banach ring,
the only obstruction to the analogue of Tate's acyclicity theorem is 
the failure of the structure presheaf to be a sheaf.
\begin{lemma} \label{L:H1 for simple Laurent covering}
Let $S_-, S_+$ be the simple Laurent covering of $\Spa(A,A^+)$ defined by some $f \in A$. Let $(A, A^+) \to (B_1, B_1^+), (A,A^+) \to (B_2, B_2^+), (A,A^+) \to (B_{12}, B_{12}^+)$ be the rational localizations corresponding to $S_-, S_+, S_- \cap S_+$. 
Then the map $B_1 \oplus B_2 \to B_{12}$ taking $(b_1, b_2)$ to $b_1 - b_2$ is surjective.
\end{lemma}
\begin{proof}
By Lemma~\ref{L:rational subdomain}, we obtain strict surjections
\[
A\{T\} \to B_1, \quad A\{U\} \to B_2, \quad A\{T,U\} \to B_{12}
\]
taking $T$ to $f$ and $U$ to $f^{-1}$. In particular, any $b \in B_{12}$ can be lifted to some $\sum_{i,j=0}^\infty a_{ij} T^i U^j \in A\{T,U\}$. 
Let $a'_n$ be the sum of $a_{ij}$ over all $i,j \geq 0$ with $i-j=n$; note that this sum converges in $A$. Let $b_1$ be the image of $\sum_{n=0}^\infty a'_n T^n$ in $B_1$.
Let $b_2$ be the image of $-\sum_{n=1}^\infty a'_{-n} U^n$ in $B_2$. Then
$(b_1, b_2) \in B_1 \oplus B_2$ maps to $b \in B_{12}$, proving the desired exactness.\end{proof}

\begin{theorem} \label{T:Tate sheaf property for structure sheaf}
Suppose that $(A,A^+)$ is sheafy. Then for every rational covering $\gothU$ of $\Spa(A,A^+)$,
\[
H^i(\Spa(A,A^+), \calO) = \check{H}^i(\Spa(A,A^+), \calO; \gothU) = 
\begin{cases} A & i=0 \\ 0 & i>0. \end{cases}
\]
\end{theorem}
\begin{proof}
By Proposition~\ref{P:acyclicity template}, it suffices to check \v{C}ech-acyclicity for simple Laurent coverings. Since the sheafy condition propagates to rational subspaces, we may as well consider only simple Laurent coverings of $\Spa(A,A^+)$ itself. In the notation of Lemma~\ref{L:H1 for simple Laurent covering}, the sequence
\[
0 \to A \to B_1 \oplus B_2 \to B_{12} \to 0
\]
is exact at $B_{12}$; by the sheafy hypothesis, it is also exact at $A$ and $B_1 \oplus B_2$.
Thus Proposition~\ref{P:acyclicity template} yields the claim.
\end{proof}

In some cases, one can apply Proposition~\ref{P:Tate reduction} to prove properties of individual inclusions of rational subdomains, by taking $\calP$ to be the condition that every subdomain in a covering has the desired property. However, in some cases it is not straightforward to verify locality, in which case the following alternate reduction process may be preferable.
\begin{prop} \label{P:Tate reduction single}
Let $\calP$ be a property of inclusions $V \subseteq U$ of rational subdomains of $\Spa(A,A^+)$ satisfying the following conditions.
\begin{enumerate}
\item[(a)]
The property $\calP$ is transitive: if it holds for $V \subseteq U$ and $W \subseteq V$, then it holds for $W \subseteq U$.
\item[(b)]
The property $\calP$ holds for any inclusion $V \subseteq U$ which is part of a simple Laurent covering of $U$.
\end{enumerate}
Then the property $\calP$ holds for any inclusion of rational subdomains
of $\Spa(A,A^+)$.
\end{prop}
\begin{proof}
To check that $\calP$ holds for $V \subseteq U$, write $U = \Spa(B,B^+)$
and suppose that $V$ is defined by the parameters $f_1,\dots,f_n,g \in B$.
Let $z \in B$ be a topologically nilpotent unit.
By Remark~\ref{R:approximate rational}, for any sufficiently large $m$ the set
\[
V_0 = \{v \in U: v(gz^{-m}) \geq 1\}
\]
is contained in $V$. For $i=1,\dots,n$ in turn, define
\[
V_i = \{v \in V_{i-1}: v(f_i g^{-1}) \leq 1\}.
\]
Since each of the inclusions $V = V_n \subseteq \cdots \subseteq V_0 \subseteq U$
is part of a simple Laurent covering, we may deduce the claim from (a) and (b).
\end{proof}

\subsection{Coherent sheaves on affinoid spaces}
\label{subsec:spectra locally ringed}

We now restrict to the setting of affinoid spaces over an analytic field,
where a good theory of coherent sheaves is available thanks to the work of Tate and Kiehl. 
However, since we are working in the framework of adic spectra, we must be a bit careful to ensure that our statements do indeed follow from the classical ones.
(If one returns to the classical results, these can mostly be extended to the most general setting of Berkovich; see Remark~\ref{R:Berkovich affinoid spaces}.)

\begin{hypothesis} \label{H:noetherian}
Throughout \S\ref{subsec:spectra locally ringed},
let $(A,A^+)$ be an adic Banach ring in which $A$ is an affinoid algebra over an analytic field $K$. We refer to any such object as an \emph{adic affinoid algebra} over $K$. Unless specified we do not assume $A^+ = A^\circ$; when this does occur, we get an affinoid $K$-algebra of tft (topologically finite type) in the terminology of \cite{scholze1}.
\end{hypothesis}

\begin{lemma} \label{L:ideals closed}
The ring $A$ is noetherian, so (by Remark~\ref{R:Noetherian Banach})
any ideal of $A$ is closed. Moreover, any finite $A$-module may be viewed as a finite Banach
$A$-module in a canonical way, under which any $A$-linear homomorphism of finite $A$-modules
is continuous and strict.
\end{lemma}
\begin{proof}
For the first assertion, see \cite[Proposition~6.1.1/3]{bgr}. For the other assertions
(which apply to any noetherian Banach ring), see \cite[\S 3.7.3]{bgr}.
\end{proof}

\begin{remark}
A refinement of Lemma~\ref{L:ideals closed} is that
$A$ is \emph{excellent} in the sense of Grothendieck
\cite[Th\'eor\`eme~2.6]{ducros}, and hence catenary.
\end{remark}

\begin{lemma}[Noether normalization] \label{L:Noether normalization}
For $A \neq 0$, there exists a
finite strict monomorphism $K\{T_1,\dots,T_n\} \to A$ for some $n \geq 0$.
\end{lemma}
\begin{proof}
See \cite[Corollary~6.1.2/2]{bgr}.
\end{proof}
\begin{cor} \label{C:nullstellensatz}
The following statement are true.
\begin{enumerate}
\item[(a)]
Every maximal ideal of $A$
has residue field finite over $K$.
\item[(b)]
The formula $\alpha \mapsto \gothp_\alpha$ defines a bijection from
$\Maxspec(A)$ to the points of $\calM(A)$ with residue field finite over $K$. (We will hereafter identify $\Maxspec(A)$ with a subspace of $\calM(A)$ and of $\Spa(A,A^+)$.)
\item[(c)]
If $A$ is nonzero, then the sets in (b) are nonempty.
\end{enumerate}
\end{cor}
\begin{proof}
Assertion (a) follows from Lemma~\ref{L:Noether normalization}.
For (b), note that on one hand, if $\calH(\alpha)$ is finite over $K$,
then $A \to \calH(\alpha)$ is surjective because its image generates a dense subfield containing $K$.
Consequently, $\gothp_\alpha$ is a maximal ideal of $A$.
Conversely, if $\gothm$ is a maximal ideal of $A$, then $A/\gothm$ is a finite extension
of $K$, and so is complete for the unique multiplicative extension of the norm on $K$. It thus
may be identified with $\calH(\alpha)$ for some $\alpha \in \calM(A)$.
This proves (b); since any nonzero ring has a maximal ideal, (b) implies (c).
\end{proof}

\begin{cor} \label{C:spectral is norm}
For $A$ reduced, the spectral seminorm on $A$
is a norm equivalent to the given norm.
\end{cor}
\begin{proof}
See \cite[Theorem~6.2.4/1]{bgr}.
\end{proof}

\begin{cor} \label{C:finite reduction}
Let $A \to B$ be a finite morphism. Then for the spectral seminorms on $A$ and $B$, the induced map $\kappa_A \to \kappa_B$ is finite.
\end{cor}
\begin{proof}
See \cite[Theorem~6.3.4/2]{bgr}.
\end{proof}

We next relate affinoid subdomains of adic spectra with the corresponding notion in rigid analytic geometry.
\begin{defn}
An \emph{affinoid subdomain} of $\Maxspec(A)$ is a subset $U$ of $\Maxspec(A)$ for which there exists a morphism $\varphi: A \to B$ of affinoid algebras over $K$ which is initial for the property that $\varphi^*: \Maxspec(B) \to \Maxspec(A)$ factors through $U$. For example, the intersection of $\Maxspec(A)$ with a rational subspace of $\Spa(A,A^+)$ is an affinoid subdomain \cite[Proposition~7.2.3/4]{bgr}; any such subspace is called a \emph{rational subdomain} of $\Maxspec(A)$.
\end{defn}

\begin{lemma} \label{L:rational subspaces match}
Let $U$ be a rational subspace of $\Spa(A,A^+)$ defined as in 
\eqref{eq:adic rational subspace}. Let 
$(A,A^+) \to (B,B^+)$ be the representing morphism. 
\begin{enumerate}
\item[(a)]
We have $B \cong A\{T_1,\dots,T_n\}/(g T_1 - f_1, \dots, g T_n - f_n)$.
\item[(b)]
The space $U \cap \Maxspec(A)$ is a rational subdomain of $\Maxspec(A)$ represented by the morphism $A \to B$.
\item[(c)]
If $A$ is reduced, then so is $B$.
\item[(d)]
If $A^+ = A^\circ$, then $B^+ = B^\circ$.
\end{enumerate}
\end{lemma}
\begin{proof}
By Lemma~\ref{L:ideals closed} the ideal $(g T_1 - f_1, \dots, g T_n - f_n)$ in
$A\{T_1,\dots,T_n\}$ is closed,
so (a) follows from Lemma~\ref{L:rational subdomain}.
We deduce (b) from (a) plus \cite[Proposition~7.2.3/4]{bgr}.
We deduce (c) from (b) plus \cite[Corollary~7.3.2/10]{bgr}.
To deduce (d), apply Corollary~\ref{C:finite reduction} to the finite morphism
$A\{T_1,\dots,T_n\} \to B$ to deduce that $\kappa_B$ is integral over $\kappa_A[T_1,\dots,T_n]$ via the map taking $T_i$ to $f_i/g$.
\end{proof}

\begin{lemma} \label{L:rational subspaces covering match1}
Any finite covering of $\Maxspec(A)$ by rational subdomains is refined by a covering induced by a rational covering of $\Spa(A,A^+)$.
\end{lemma}
\begin{proof}
By \cite[Lemma~8.2.2/2]{bgr}, any finite covering of $\Maxspec(A)$ by rational subdomains is refined by standard rational covering in the sense of Definition~\ref{D:standard Laurent}.
\end{proof}

The Gerritzen-Grauert theorem in rigid analytic geometry can be interpreted as follows. (See also Temkin's proof in the context of Berkovich's theory \cite[Theorem~3.1]{temkin-gg}.)
\begin{theorem} \label{T:Gerritzen-Grauert}
Let $U$ be an affinoid subdomain of $\Maxspec(A)$. 
Then there exists a rational covering $V_1,\dots,V_n$ of $\Spa(A,A^+)$ such that 
$U \cap V_i$ is a rational subdomain in $V_i \cap \Maxspec(A)$ for $i=1,\dots,n$. In particular, $U$ can be written as a finite union of rational subdomains of $\Maxspec(A)$ (but not conversely: not every finite union of rational subdomains is an affinoid subdomain).
\end{theorem}
\begin{proof}
By \cite[Theorem~7.3.5/1]{bgr}, there exists a finite collection of rational subspaces $W_1,\dots,W_m$ of $\Maxspec(A)$ with the property that $U \cap W_i$ is a rational subdomain in $W_i$ for $i=1,\dots,m$. 
By Lemma~\ref{L:rational subspaces covering match1}, this covering can be refined to a covering induced by a rational covering of $\Spa(A,A^+)$.
\end{proof}

\begin{lemma} \label{L:naive rational covering}
Suppose that $A^+ = A^\circ$ (and hence $A$ is reduced; see Definition~\ref{D:uniform Banach ring}).
For any rational subdomain $\Spa(B,B^+)$ of $\Spa(A,A^+)$,
a finite collection of rational subdomains $\{U_i\}_i$ of $\Spa(B,B^+)$ is a rational covering if and only if $\{U_i \cap \Maxspec(B)\}_i$ is a covering of $\Maxspec(B)$.
\end{lemma}
\begin{proof}
We first check the special case of a one-element covering.
Suppose that $V \subseteq U$ are rational subspaces of $\Spa(A,A^\circ)$ such that $U$ is rational and $V \cap \Maxspec(A) = U \cap \Maxspec(A)$.
Let $(A,A^\circ) \to (B,B^+) \to (C,C^+)$ be the corresponding rational localizations.
By Lemma~\ref{L:rational subspaces match}, $B^+ = B^\circ$.
Since $B \to C$ is a rational localization, 
it induces isomorphisms at completed (algebraic) local rings,
and hence is an open immersion of rigid analytic spaces \cite[Proposition~7.3.3/5]{bgr}. But by assumption $\Maxspec(B) = \Maxspec(C)$,
so we obtain an isomorphism of rigid analytic spaces
(see the discussion after \cite[Corollary~7.3.3/6]{bgr})
and $B \to C$ is itself an isomorphism of rings. Since $B^\circ = C^\circ \subseteq C^+ \subseteq C^\circ$, we have $C^+ = C^\circ$ and hence $V = U$.

To prove the statement of the lemma, note that $B^+ = B^\circ$ by
Lemma~\ref{L:rational subspaces match}; we may thus assume $(B,B^+) = (A,A^\circ)$.
Let $\{U_i\}_i$ be a finite collection of rational subdomains of $\Spa(A,A^\circ)$ such that $\{U_i \cap \Maxspec(A)\}_i$ is a covering of $\Maxspec(A)$. By
Lemma~\ref{L:rational subspaces covering match1}, there exists a rational covering $\{V_j\}_j$ of $\Spa(A,A^\circ)$ such that the covering $\{V_j \cap \Maxspec(A)\}_{j=1}^n$ of $\Maxspec(A)$ refines the covering $\{U_i \cap \Maxspec(A)\}_i$.
That is, for each $j=1,\dots,n$, there exists some $i$ such that $V_j \cap \Maxspec(A) \subseteq U_i \cap \Maxspec(A)$, or in other words $(V_j \cap U_i) \cap \Maxspec(A) = V_j \cap \Maxspec(A)$. By the previous paragraph, this implies that
$V_j \cap U_i = V_j$, or in other words $V_j \subseteq U_i$. Hence $\{U_i\}_i$ is a rational covering as claimed.
\end{proof}
\begin{cor}\label{C:naive rational covering}
Suppose that $A$ is reduced.
Then no two distinct rational subspaces of $\Spa(A,A^+)$ have the same intersection with $\Maxspec(A)$.
\end{cor}
\begin{proof}
We may reduce to the case where $A$ is reduced. By Corollary~\ref{C:spectral is norm}, we may further reduce to the case 
$A^+ = A^\circ$.
For $U = \Spa(B,B^+)$ a rational subspace of $\Spa(A,A^\circ)$,
$U = \emptyset$ iff $B = 0$ (Theorem~\ref{T:nonempty spectrum}) iff $\Maxspec(B) = 0$.
We may thus deduce the claim from Lemma~\ref{L:naive rational covering}.
\end{proof}

\begin{prop} \label{P:affinoid subspaces match}
Let $U$ be an affinoid subdomain of $\Spa(A,A^+)$
represented by $(A,A^+) \to (B,B^+)$.
\begin{enumerate}
\item[(a)]
The space $U \cap \Maxspec(A)$ is an affinoid subdomain of $\Maxspec(A)$ represented by the morphism $A \to B$.
\item[(b)]
If $A$ is reduced, then so is $B$.
\item[(c)]
If $A^+ = A^\circ$, then $B^+ = B^\circ$.
\end{enumerate}
\end{prop}
\begin{proof}
To check (a), let $A \to C$ be a morphism of affinoid algebras over $K$ such that $\Maxspec(C)$ maps into $U \cap \Maxspec(A)$; we must show that this morphism factors through $B$. Form a morphism $(A, A^+) \to (C,C^+)$ of adic Banach rings by taking $C^+ = C^\circ$.
To factor $A \to C$ through $B$, it suffices to check that $\Spa(C,C^+)$ maps into $U$;
for this purpose, we may assume that $A$ is reduced, then pass from $(A,A^+)$ to $(A,A^\circ)$ by base extension (using Corollary~\ref{C:spectral is norm}) and thus assume that $A^+ = A^\circ$. Apply Theorem~\ref{T:Gerritzen-Grauert} to obtain a rational covering $\{V_i\}_i$ of $\Spa(A,A^\circ)$ such that 
for each $i$, $U \cap V_i \cap \Maxspec(A)$ is a rational subspace of $V_i \cap \Maxspec(A)$. Let $(A,A^+) \to (D_i, D_i^+)$ be the rational localization representing $V_i$; by Lemma~\ref{L:rational subspaces match}, $D_i^+ = D_i^\circ$.
Put $(E_i, E_i^+) = (B,B^+) \widehat{\otimes}_{(A,A^+)} (D_i, D_i^+)$
and $(F_i, F_i^+) = (C,C^+) \widehat{\otimes}_{(A,A^+)} (D_i, D_i^+)$;
it now suffices to check that the image of $\Spa(F_i, F_i^+) \to \Spa(D_i, D_i^+)$
is contained in $U \cap V_i$. Since $(D_i, D_i^+) \to (E_i, E_i^+)$ 
and $(C,C^+) \to (F_i, F_i^+)$ are rational localizations, Lemma~\ref{L:rational subspaces match} implies that 
$E_i^+ = E_i^\circ$, $F_i^+ = F_i^\circ$,
and
$U \cap V_i \cap \Maxspec(A)$ is a rational subdomain of $\Maxspec(D_i) = V_i \cap \Maxspec(A)$.
Since $\Maxspec(F_i)$ maps into $U \cap V_i \cap \Maxspec(A)$,
it follows that $D_i \to F_i$ factors through $E_i$
and so $(D_i, D_i^\circ) \to (F_i, F_i^\circ)$ factors through $(E_i, E_i^\circ)$.
Hence the image of $\Spa(F_i, F_i^+) \to \Spa(D_i, D_i^+)$
is contained in $U \cap V_i$, yielding (a).

Given (a), we may deduce (b) and (c) by the same arguments as in the proofs of parts (c) and (d) of Lemma~\ref{L:rational subspaces match}.
\end{proof}

\begin{prop} \label{P:naive rational covering}
Suppose that $A^+ = A^\circ$ (and hence $A$ is reduced).
For any affinoid subdomain $\Spa(B,B^+)$ of $\Spa(A,A^\circ)$,
a finite collection of affinoid subdomains $\{U_i\}_i$ of $\Spa(B,B^+)$ is an affinoid covering if and only if $\{U_i \cap \Maxspec(B)\}_i$ is a covering of $\Maxspec(B)$.
\end{prop}
\begin{proof}
This follows by the same proof as Lemma~\ref{L:naive rational covering}, but using
 Proposition~\ref{P:affinoid subspaces match} in lieu of
 Lemma~\ref{L:rational subspaces match}.
\end{proof}

\begin{remark} \label{R:compatibility of affinoid subdomains}
Proposition~\ref{P:affinoid subspaces match} and 
Proposition~\ref{P:naive rational covering} allow us 
to assert statements about affinoid subdomains of $\Spa(A,A^+)$ by invoking the corresponding statements about affinoid subdomains of $\Maxspec(A)$ from rigid analytic geometry. We will do this without further comment in what follows.
\end{remark}

\begin{lemma} \label{L:affinoid subdomain is flat}
For any affinoid localization $(A,A^+) \to (B,B^+)$, the morphism $A \to B$ is flat.
\end{lemma}
\begin{proof}
See \cite[Corollary~7.3.2/6]{bgr}.
\end{proof}

\begin{prop} \label{P:affinoid covering is faithfully flat}
Let $\{(A,A^+) \to (B_i, B_i^+)\}_{i=1}^n$ be an affinoid covering.
Then the ring homomorphism $A \to B_1 \oplus \cdots \oplus B_n$ is faithfully flat.
\end{prop}
\begin{proof}
The homomorphism is flat by Lemma~\ref{L:affinoid subdomain is flat}.
It is faithful by Lemma~\ref{L:faithfully flat by maximal ideals}
and the fact that every maximal ideal of $A$ is closed
(by Corollary~\ref{C:vanishing norm}).
\end{proof}

\begin{cor} \label{C:etale local}
Let $\{(A,A^+) \to (B_i, B_i^+)\}_{i=1}^n$ be an affinoid covering.
\begin{enumerate}
\item[(a)]
A finite $A$-module $M$ is locally free if and only if
$M \otimes_A B_i$ is a locally free $B_i$-module for $i=1,\dots,n$.
\item[(b)]
A finite $A$-algebra $R$ is \'etale if and only if
$R \otimes_A B_i$ is an \'etale $B_i$-algebra for $i=1,\dots,n$.
\end{enumerate}
\end{cor}
\begin{proof}
This follows from Proposition~\ref{P:affinoid covering is faithfully flat} plus
Theorem~\ref{T:descent finite locally free}.
\end{proof}

\begin{theorem} \label{T:Tate-Kiehl}
Let $\gothU$ be an affinoid covering.
\begin{enumerate}
\item[(a)]
For any finite $A$-module $M$,
let $\tilde{M}$ be the sheaf of $\calO$-modules
on $\Spa(A,A^+)$ induced by $M$.
Then $H^i(\Spa(A,A^+), \tilde{M}; \gothU) = M$ for $i=0$ and $0$ for $i>0$.
In particular, $(A,A^+)$ is sheafy
and $H^i(\Spa(A,A^+), \tilde{M}) = M$ for $i=0$ and $0$ for $i>0$.
\item[(b)]
The functor $M \mapsto \tilde{M}$ defines a tensor equivalence between 
finite $A$-modules and coherent sheaves of $\calO$-modules on $\Spa(A,A^+)$.
In particular, for $\{(A,A^+) \to (B_i, B_i^+)\}_{i=1}^n$ the morphisms representing
$\gothU$, the homomorphism
$A \to B_1 \oplus \cdots \oplus B_n$ is an effective descent morphism for finite Banach modules over Banach rings.
\end{enumerate}
\end{theorem}
\begin{proof}
Part (a) is due to Tate; see \cite[Corollary~8.2.1/5]{bgr}.
Part (b) is due to Kiehl; see \cite[Theorem~9.4.3/3]{bgr}.
\end{proof}

\begin{cor} \label{C:lift disconnection}
Let $U$ be a closed-open subset of $\Spa(A,A^+)$.
Then there exists a unique idempotent element $e \in A$
whose image in $\calH(v)$ is $1$ if $v \in U$ and
$0$ if $v \notin U$.
In particular, the projection $A \to eA$
taking $x \in A$ to $ex$ induces a homeomorphism $\Spa(eA, eA^+) \cong U$.
\end{cor}
Note that the analogous result for $\calM(A)$ also holds because
the closed-open subsets of $\calM(A)$ and $\Spa(A,A^+)$ correspond;
see Definition~\ref{D:Berkovich rational subspace}.
Note also that this result generalizes to arbitrary adic Banach rings;
see Proposition~\ref{P:affinoid system disconnection}.
\begin{proof}
Cover $U$ with finitely many rational subdomains $U_1,\dots,U_m$ and the complement of $U$ with finitely many rational subdomains $V_1,\dots,V_n$. Let
$(A,A^+) \to (B_i,B_i^+)$ and $(A,A^+) \to (C_j,C_j^+)$ be the morphisms representing 
$U_i$ and $V_j$, respectively.
By Theorem~\ref{T:Tate-Kiehl}(a) applied with $M = A$, the element
$((1,\dots,1),(0,\dots,0))$
of $(B_1 \oplus \cdots \oplus B_m) \oplus (C_1 \oplus \cdots \oplus C_n)$
determines an idempotent element $e \in A$ with the desired property.

To verify uniqueness, let $e' \in A$ be another idempotent of the desired form.
Then $1 - e - e'$ maps to $1$ or $-1$ in $\calH(v)$ for each $v \in \Spa(A,A^+)$,
and so is a unit
in $A$ by Corollary~\ref{C:unit from spectrum}. Now $(e-e')(1-e-e') = 0$,
so $e-e' = 0$ as desired.
\end{proof}

\begin{remark}  \label{R:Berkovich affinoid spaces}
As described in \cite[\S 2]{berkovich1}, Berkovich extends the preceding results 
(excluding Lemma~\ref{L:Noether normalization} and Corollary~\ref{C:nullstellensatz})
in three ways: the base field $K$ is permitted to carry the trivial norm; affinoid algebras are defined
as in Remark~\ref{R:strictly affinoid homomorphism}; and rational subspaces are defined as in Remark~\ref{R:Berkovich rational subspaces}. 
The basic idea is that for any affinoid algebra $A$ over $K$ in the sense of Berkovich,
one can construct a nontrivially normed analytic field $L$ containing $K$ such that
$A_L = A \widehat{\otimes}_K L$ is an affinoid algebra over $L$ in the classical sense (by adjoining some transcendentals with prescribed norms). The most nontrivial points are that the map $A \to A_L$ is strict 
(Lemma~\ref{L:inject tensor}) and faithfully flat \cite[Lemma~2.1.2]{berkovich2}
and the corresponding restriction map $\calM(A_L) \to \calM(A)$ is surjective
(Lemma~\ref{L:fibre product}).
\end{remark}

\begin{remark} \label{R:power series}
In a different direction, note that for any $r \in (0,1), s \in (1, +\infty)$ with $rs \geq 1$, the ring $\ZZ((z))$ equipped with the norm
\[
\left| \sum_{i \in \ZZ} c_i z^i \right| = \max\{\max\{r^i: i \geq 0, c_i \neq 0\},
\max\{s^i: i > 0, c_{-i} \neq 0\}\}
\]
is strongly noetherian: we may construct $\ZZ((z))\{T_1,\dots,T_n\}$ by taking the $z$-adic completion of $\ZZ[z,T_1,\dots,T_n]$ and then inverting $z$. This means that many of the preceding results apply also to affinoid algebras over $\ZZ((z))$, including Lemma~\ref{L:affinoid subdomain is flat}
(see \cite[Lemma~1.7.6]{huber}), 
Proposition~\ref{P:affinoid covering is faithfully flat}
(as a corollary of Lemma~\ref{L:affinoid subdomain is flat}),
Corollary~\ref{C:etale local}
(as a corollary of Proposition~\ref{P:affinoid covering is faithfully flat}),
Theorem~\ref{T:Tate-Kiehl}(a)
(see \cite[Theorem~2.5]{huber2}), and Corollary~\ref{C:lift disconnection}
(as a corollary of Theorem~\ref{T:Tate-Kiehl}(a)).
One may also extend Theorem~\ref{T:Tate-Kiehl}(b): we are unaware of a precise reference, but given the previous results one may directly emulate the proof of
\cite[Theorem~9.4.3/3]{bgr}. In this paper we will only apply Theorem~\ref{T:Tate-Kiehl}(b) in the case of finite projective modules, in which case one may instead appeal to Theorem~\ref{T:tate to Kiehl}; note that the proof of that theorem does not depend on any results from \S\ref{subsec:affinoid systems}, so there is no vicious circle.
\end{remark}

\subsection{Affinoid systems}
\label{subsec:affinoid systems}

To get some handle on Banach algebras which are not affinoid algebras over a field, 
we use an analogue of the observation that every ring is a direct limit of
noetherian subrings (namely its finitely generated $\ZZ$-subalgebras).
As usual, we restrict to classical affinoid algebras and note in passing that similar arguments can be derived in Berkovich's framework.

\begin{defn} \label{D:affinoid system}
By an \emph{affinoid system}, we will mean a directed system
$\{((A_i,A_i^+), \alpha_i)\}_{i \in I}$ in the category of adic affinoid algebras over $\ZZ((z))$ and submetric (not just bounded) morphisms.
Note that each ring $A_i$ is strongly noetherian (Remark~\ref{R:power series}).

Given an affinoid system, equip the direct limit $A$ of the $A_i$ in the category of rings
with the submultiplicative seminorm $\alpha$ given by taking the infimum of the quotient norms induced by the $\alpha_i$, and let $A^+$ be the direct limit of the $A_i^+$. We will refer to the completion of $(A,A^+)$ with respect to $\alpha$ as the \emph{completed direct limit} of the system.
\end{defn}

\begin{lemma} \label{L:construct affinoid system}
Let $(R,R^+)$ be an adic Banach algebra.
Then there exists an affinoid system with completed direct limit $(R,R^+)$.
\end{lemma}
\begin{proof}
By choosing a topologically nilpotent unit $z \in R$, we may view $R$ as a Banach algebra over $\ZZ((z))$ for a suitable norm as in Remark~\ref{R:power series}.
Let $I$ be the set of finite subsets of $R^+$.
For each $S \in I$, let $A_i$ be the quotient of $\ZZ((z))\{S\}$ by the kernel of the map to
$R$ taking $s$ (as a generator of the ring) to $s$
(as an element of $R$); this is an affinoid algebra over $\ZZ((z))$.
Equip $A_i$ with the supremum of the quotient norm and the subspace norm;
since this is again a norm under which $A_i$ is a Banach algebra over $\ZZ((z))$, it is equivalent to the quotient norm by the open mapping theorem (Theorem~\ref{T:open mapping}).
Let $A_i^+$ be the image of $\ZZ\llbracket z \rrbracket\{S\}$ in $A_i$.
This gives the desired affinoid system.
\end{proof}

The previous observation has some strong consequences for Banach algebras.

\begin{remark} \label{R:affinoid system disconnection}
Let $\{((A_i,A_i^+), \alpha_i)\}_{i \in I}$ be an affinoid system.
For $((A, A^+), \alpha)$ the direct limit and $(R,R^+)$ the completion,
the restriction map $\Spa(R,R^+) \to \varprojlim_i \Spa(A_i,A_i^+)$ is continuous,
and also bijective because specifying a compatible system of semivaluations on each $A_i$ bounded by $\alpha_i$
is equivalent to specifying a semivaluation on $A$ bounded by $\alpha$.
Moreover, every rational subspace of $\Spa(R,R^+)$ arises from some $\Spa(A_i,A_i^+)$
by Remark~\ref{R:approximate rational}.
We thus obtain a homeomorphism $\Spa(R,R^+) \cong \varprojlim_i \Spa(A_i,A_i^+)$.
\end{remark}

We obtain the following extension of Corollary~\ref{C:lift disconnection}. 
For an alternate approach that also includes the case of a Banach algebra over a trivially normed field (without a topologically nilpotent unit), see \cite[Theorem~7.4.1]{berkovich1}.
\begin{prop} \label{P:affinoid system disconnection}
Let $(R,R^+)$ be an adic Banach algebra, and let 
$U$ be a closed-open subset of $\Spa(R,R^+)$.
Then there exists a unique idempotent element $e \in R$
whose image in $\calH(v)$ is $1$ if $v \in U$ and
$0$ if $v \notin U$.
In particular, the projection $R \to eR$
taking $x \in R$ to $ex$ induces a homeomorphism $\Spa(eR, eR^+) \cong U$.
\end{prop}
\begin{proof}
By Lemma~\ref{L:construct affinoid system}, we can ensure that there exists an affinoid system $\{((A_i, A_i^+), \alpha_i)\}_{i \in I}$ with completed direct limit $\Spa(R,R^+)$.
By Remark~\ref{R:compact spaces}(e) and Remark~\ref{R:affinoid system disconnection},
$U$ is the inverse image of a closed-open subset of some $\Spa(A_i,A_i^+)$,
and hence is induced by some idempotent element of some $A_i$ by
Corollary~\ref{C:lift disconnection}
(as extended by Remark~\ref{R:power series}).
\end{proof}

We next relate rational localizations of the completed direct limit of an affinoid system
with the corresponding objects defined on individual terms of the system.

\begin{lemma} \label{L:affinoid system rational}
Let $\{((A_i, A_i^+), \alpha_i)\}_{i \in I}$ be an affinoid system with direct limit $((A,A^+), \alpha)$.
Let $(R,R^+)$ be the completion of $(A,A^+)$.
\begin{enumerate}
\item[(a)]
For any rational localization $(R,R^+) \to (S,S^+)$,
there exist an index $i \in I$ and a rational localization $(A_i,A_i^+) \to (B_i,B_i^+)$ such that $(S,S^+) \cong (B_i, B_i^+) \widehat{\otimes}_{(A_i,A_i^+)} (R,R^+)$. The same is then true for each $j \geq i$ for
$(B_j,B_j^+) = (B_i,B_i^+) \widehat{\otimes}_{(A_i,A_i^+)} (A_j,A_j^+)$; in fact, the $(B_j,B_j^+)$ form another affinoid system with completed
direct limit $(S,S^+)$.
\item[(b)]
With notation as in (a), for any $v \in \Spa(R,R^+)$ restricting to $v_i \in \Spa(A_i,A_i^+)$,
$v$ belongs to $\Spa(S,S^+)$ if and only
if $v_i$ belongs to $\Spa(B_i,B_i^+)$.
\item[(c)]
With notation as in (a), for any $\beta \in \calM(R)$ restricting to $\beta_i \in \calM(A_i)$,
$(R,R^+) \to (S,S^+)$ encircles $\beta$ if and only if there exists an index $j \geq i$ for which
$(A_j,A_j^+) \to (B_j,B_j^+)$ encircles $\beta_j$.
\end{enumerate}
\end{lemma}
\begin{proof}
Part (a) is immediate from Remark~\ref{R:approximate rational}.
Part (b) is immediate from Remark~\ref{R:affinoid system disconnection}.
Part (c) follows by taking maximal Hausdorff quotients in Remark~\ref{R:affinoid system disconnection} to view $\calM(R)$ as the inverse limit of the $\calM(A_i)$.
\end{proof}
\begin{cor} \label{C:affinoid system localization}
Let $\{((A_i, A_i^+), \alpha_i)\}_{i \in I}$ be an affinoid system with direct limit $((A, A^+), \alpha)$.
Let $(R,R^+)$ be the completion of $(A,A^+)$.
\begin{enumerate}
\item[(a)]
For $v \in \Spa(A,A_+)$ and $i \in I$, let $\calO_{i,v}$ be the stalk of the structure sheaf of $\Spa(A_i,A_i^+)$ at $v$. Then $\varinjlim_{i \in I} \calO_{i,v}$ is a local ring whose residue field is dense in $\calH(v)$.
\item[(b)]
For $\beta \in \calM(R)$ and $i \in I$,
let $A_{i,\beta}$ denote the Hausdorff localization of $A_i$ at the restriction
of $\beta$. Then $\varinjlim_{i \in I} A_{i,\beta}$ is a local ring
whose residue field is dense in $\calH(\beta)$.
\end{enumerate}
\end{cor}

\begin{remark} \label{R:lift simple covering}
With notation as in Lemma~\ref{L:affinoid system rational}, note that by Remark~\ref{R:compact spaces}(d), any rational covering of $\Spa(A,A^+)$
is defined over some $(A_i,A_i^+)$.
\end{remark}

We have the following extension of Lemma~\ref{L:descend etale on field}.
\begin{prop} \label{P:henselian direct limit2}
Let $\{((A_i, A_i^+), \alpha_i)\}_{i \in I}$ be an affinoid system with direct limit $((A,A^+),\alpha)$.
Put $I = \ker(\alpha)$, $\overline{A} = A/I$, and $R = \widehat{A}$.
Then the base change functors
$\FEt(A) \to \FEt(\overline{A}) \to \FEt(R)$ are tensor equivalences.
\end{prop}
\begin{proof}
The base change functor $\FEt(A) \to \FEt(\overline{A})$ is a tensor equivalence
by Lemma~\ref{L:henselian direct limit}(a) and Theorem~\ref{T:henselian}.
The functor $\FEt(\overline{A}) \to \FEt(R)$ is rank-preserving and fully faithful by
Lemma~\ref{L:descend etale on field}(a).
It is thus enough to check that $\FEt(A) \to \FEt(R)$ is essentially surjective.

Choose any $V \in \FEt(R)$.
For each $\beta \in \calM(R)$, for $A_{i,\beta}$ the Hausdorff localization of $A_i$ at the restriction
of $\beta$,
the functor $\FEt(\varinjlim_{i \in I} A_{i,\beta}) \to \FEt(R_{\beta})$ is an equivalence
by Corollary~\ref{C:affinoid system localization}
(to see that both $\varinjlim_{i \in I} A_{i,\beta}$ and $R_\beta$ have dense images in $\calH(\beta)$)
and Lemma~\ref{L:descend etale on field}(b).
We can thus choose an index $i \in I$ and a rational localization $(A_i,A_i^+) \to (B_i,B_i^+)$ encircling $\beta$ such that for $S = R \widehat{\otimes}_{A_i} B_i$,
the object $V \otimes_R S$ in $\FEt(S)$ descends to an object in $\FEt(B_i)$.
By the compactness of $\calM(R)$, we can find an index $i \in I$
and a strong rational covering $\{(A_i, A_i^+) \to (B_{i,j}, B_{i,j}^+)\}_{j=1}^n$
such that for $S_j = R \widehat{\otimes}_{A_i} B_{i,j}$, 
the object $V \otimes_R S_j$ in $\FEt(S_j)$ descends to an object $U_{i,j}$ in $\FEt(B_{i,j})$.
If write $B_{i,jl}$ and $S_{jl}$ for
$B_{i,j} \widehat{\otimes}_{A_i} B_{i,l}$ and $S_j \widehat{\otimes}_R S_l$, the functor
$\FEt(\varinjlim_{i \in I} B_{i,jl}) \to \FEt(S_{jl})$
is fully faithful; we thus obtain (after suitably increasing $i$)
isomorphisms among the $U_{i,j}$ on overlaps satisfying the cocycle condition.
By Theorem~\ref{T:Tate-Kiehl}(b) and Corollary~\ref{C:etale local}
(as extended by Remark~\ref{R:power series}),
the $U_{i,j}$ glue to an object in $\FEt(A_i)$, and hence in $\FEt(A)$.
This proves the claim.
\end{proof}

Using Lemma~\ref{L:affinoid system rational}, we obtain a weak extension of
Theorem~\ref{T:Tate-Kiehl} to arbitrary Banach algebras.
A better result would be to glue finite projective modules, but this is more difficult; see
\S\ref{subsec:glueing finite}.

\begin{theorem} \label{T:henselian direct limit3}
Let $(R,R^+)$ be an adic Banach algebra.
Let $\{(R,R^+) \to (R_i,R_i^+)\}_{i=1}^n$ be a rational covering.
Then the
homomorphism $R \to R_1 \oplus \cdots \oplus R_n$ is an effective descent morphism
for finite \'etale algebras over Banach rings.
\end{theorem}
\begin{proof}
By Lemma~\ref{L:construct affinoid system},
we can construct an affinoid system $\{((A_i,A_i^+), \alpha_i)\}_{i \in I}$
with completed direct limit $(R,R^+)$. By Remark~\ref{R:lift simple covering},
for each sufficiently large $j$, the given covering family is induced by a covering family $\{(A_j, A_j^+) \to (B_{j,i}, B_{j,i}^+)\}_{i=1}^n$.
By Proposition~\ref{P:henselian direct limit2}, any descent datum for finite \'etale algebras over Banach rings with respect to $R \to S_1 \oplus \cdots \oplus S_n$ arises from
a descent datum with respect to $A_j \to B_{j,1} \oplus \cdots \oplus B_{j,n}$ for some $j$.
This descent datum is effective by Theorem~\ref{T:Tate-Kiehl}(b) (as extended by Remark~\ref{R:power series}, to uniquely glue the underlying finite flat algebras)
and Corollary~\ref{C:etale local} (to show that the resulting algebra is finite \'etale).
\end{proof}

\begin{cor} \label{C:glue local systems}
With notation as in Theorem~\ref{T:henselian direct limit3},
the morphism $R \to R_1 \oplus \cdots \oplus R_n$ is an effective descent morphism
for \'etale $\Zp$-local systems over Banach rings.
\end{cor}
\begin{proof}
This follows from Theorem~\ref{T:henselian direct limit3}
and Remark~\ref{R:local systems rings equivalence}.
\end{proof}

\subsection{Glueing of finite projective modules}
\label{subsec:glueing finite}

We now turn to the problem of glueing finite modules over Banach rings,
using the formalism of \S\ref{subsec:descent formalism} as a starting point. We begin with a cautionary note.

\begin{remark} \label{R:rational flat}
For $(R, R^+) \to (S,S^+)$ a rational localization of adic Banach rings,
the map $R \to S$ is flat when $R$ is an affinoid algebra over an analytic field by
Lemma~\ref{L:affinoid subdomain is flat}, but need not be flat in general.
For instance, flatness almost always fails for perfectoid algebras. 
Guided by this observation and by the analogy with the Beauville-Laszlo theorem
(Proposition~\ref{P:reduced descent}),
we limit
our glueing ambitions to cases where the modules being glued are themselves flat.
\end{remark}

Taking Remark~\ref{R:rational flat} into account, we will be interested in the categories of sheaves of locally free modules of finite rank over various sheaves of rings on adic spectra; in particular, we will want to know when these categories are equivalent to the categories of finite projective modules over the ring of global sections. The guiding principle at work is that the only obstructions to obtaining such results (analogous to Kiehl's theorem) are failures of acyclicity of the base rings
(analogous to Tate's theorem).

\begin{lemma} \label{L:Cartan factorization}
Let $R_1 \to S$, $R_2 \to S$ be bounded homomorphisms of Banach rings (not necessarily
containing topologically nilpotent units) such that the
sum homomorphism $\psi: R_1 \oplus R_2 \to S$ of groups is strict surjective. Then
there exists a constant $c>0$ such that for every positive integer $n$,
every matrix $U \in \GL_n(S)$ with $|U-1| < c$ can be written in the form $\psi(U_1) \psi(U_2)$
with $U_i \in \GL_n(R_i)$. Moreover, if $\psi$ is almost optimal, this holds with $c = 1$.
\end{lemma}
\begin{proof}
By hypothesis, there exists a constant $d \geq 1$ such that every $x \in S$ lifts to some
pair $(y_1, y_2) \in R_1 \oplus R_2$ with $|y_1|, |y_2| \leq d |x|$.
It will suffice to prove the claim for $c = d^{-2}$.

Given $U \in \GL_n(S)$ with $|U-1| < c$, put $V = U-1$, and lift each entry $V_{ij}$ to a pair
$(X_{ij}, Y_{ij}) \in R_1 \oplus R_2$ with $|X_{ij}|, |Y_{ij}| \leq d |V_{ij}|$.
Then the matrix $U' = \psi(1-X) U \psi(1-Y)$ satisfies $|U'-1| \leq d |U-1|^2$.
If $|U-1| \leq d^{-l}$ for some integer $l \geq 2$, then $|U'-1| \leq d^{-l-1}$, so we may construct
the desired matrices by iterating the construction.
(See \cite[Lemma~4.5.3]{fvdp} for a similar argument or \cite[Theorem~2.2.2]{kedlaya-course} for a
more general result.)
\end{proof}

\begin{defn} \label{D:glueing pair}
Let
\[
\xymatrix{
R \ar[r] \ar[d] & R_1 \ar[d] \\
R_2 \ar[r] & R_{12}
}
\]
be a commutative diagram of Banach rings. (For the purposes of this definition, it is
not necessary to assume the presence of topologically nilpotent units.)
We call this diagram a
\emph{glueing square} if the following conditions hold.
\begin{enumerate}
\item[(a)]
The sequence
\[
0 \to R \to R_1 \oplus R_2 \to R_{12} \to 0
\]
of $R$-modules, in which the last nontrivial arrow takes $(s_1,s_2)$ to $s_1 - s_2$, is strict exact.
\item[(b)]
The map $R_2 \to R_{12}$ has dense image.
\item[(c)]
The map $\calM(R_1 \oplus R_2) \to \calM(R)$ is surjective.
\end{enumerate}
We define \emph{glueing data} on a glueing square as in Definition~\ref{D:glueing datum}.
\end{defn}

The following argument is a variant of Lemma~\ref{L:nearby generators}.
\begin{lemma} \label{L:Kiehl lemma}
Consider a glueing square as in Definition~\ref{D:glueing pair},
and let $M_1, M_2, M_{12}$ be a finite glueing datum.
Let $M$ be the kernel of the map $M_1 \oplus M_2 \to M_{12}$ taking $(m_1,m_2)$ to $\psi_1(m_1) - \psi_2(m_2)$.
\begin{enumerate}
\item[(a)]
For $i=1,2$, the natural map $M \otimes_R R_i \to M_i$ is surjective.
\item[(b)]
The map $M_1 \oplus M_2 \to M_{12}$ is surjective.
\end{enumerate}
\end{lemma}
\begin{proof}
We follow \cite[Lemmas~4.5.4 and~4.5.5]{fvdp}.
Choose generating sets $\bv_1,\dots,\bv_n$ and $\bw_1,\dots,\bw_n$ of $M_1$ and $M_2$, respectively,
of the same cardinality. We may then choose $n \times n$ matrices $A,B$ over $R_{12}$ such that
$\psi_2(\bw_j) = \sum_i A_{ij} \psi_1(\bv_i)$ and $\psi_1(\bv_j) = \sum_i B_{ij} \psi_2(\bw_i)$.

By hypothesis, the map $R_2 \to R_{12}$ has dense image.
We may thus choose an $n \times n$ matrix $B'$ over $R_2$ so that
$A(B'-B)$ has norm less than the constant $c$ of Lemma~\ref{L:Cartan factorization}.
We may then write $1 + A(B'-B) = C_1 C_2^{-1}$ with $C_i \in \GL_n(R_i)$.

We now may define elements $\bx_j \in M_1 \oplus M_2$ by the formula
\[
\bx_j = (\bx_{j,1}, \bx_{j,2}) =
\left(\sum_i (C_1)_{ij} \bv_i, \sum_i (B' C_2)_{ij} \bw_i \right) \qquad (j=1,\dots,n).
\]
Then
\[
\psi_1(\bx_{j,1}) - \psi_2(\bx_{j,2}) = \sum_i (C_1 - A B' C_2)_{ij} \psi_1(\bv_i) =
\sum_i ((1 - AB)C_2)_{ij} \psi_1(\bv_i) = 0,
\]
so $\bx_j \in M$.
Since $C_1 \in \GL_n(R_1)$, the $\bx_{i,1}$ generate $M_1$ over $R_1$,
so the map $M \otimes_R R_1 \to M_1$ is surjective.
We may now apply Lemma~\ref{L:Kiehl generic1} to deduce (a) and (b).
\end{proof}

\begin{prop} \label{P:glue projective}
Consider a glueing square as in Definition~\ref{D:glueing pair},
and let $M_1, M_2, M_{12}$ be a finite projective glueing datum.
Let $M$ be the kernel of the map $M_1 \oplus M_2 \to M_{12}$ taking $(m_1,m_2)$ to $\psi_1(m_1) - \psi_2(m_2)$.
Then $M$ is a finite projective $R$-module and the natural maps $M \otimes_R R_i \to M_i$ are isomorphisms.
\end{prop}
\begin{proof}
By Lemma~\ref{L:Kiehl lemma}, the hypotheses of Lemma~\ref{L:Kiehl generic2}(a)
are satisfied. It thus suffices to check that the additional hypothesis of Lemma~\ref{L:Kiehl generic2}(b) is satisfied, i.e., that
the image of $\Spec(R_1 \oplus R_2) \to \Spec(R)$ contains $\Maxspec(R)$.
Given $\gothp \in \Maxspec(R)$, choose $\alpha \in \calM(R)$ with $\gothp_\alpha = \gothp$ (see Definition~\ref{D:residue field}). By assumption, $\alpha$ lifts to some $\beta \in \calM(R_1 \oplus R_2)$; then $\gothp_\beta$ is a prime ideal of $\Spec(R_1 \oplus R_2)$ lifting $\gothp$.
\end{proof}

\begin{defn} \label{D:Tate-Kiehl properties}
Let $(A,A^+)$ be an adic Banach ring.
Let $\calR$ be a presheaf of topological rings on $\Spa(A,A^+)$.
We say that $\calR$ satisfies the \emph{Tate sheaf property} if 
for every rational localization $(A, A^+) \to (B,B^+)$ and every rational covering $\gothV$ of $U = \Spa(B,B^+)$,
\begin{equation} \label{eq:Tate-Kiehl properties}
H^i(U, \calR) = \check{H}^i(U, \calR; \gothV) = \begin{cases} \calR(U) & i=0 \\ 0 & i>0. \end{cases}
\end{equation}
In particular, this implies that $\calR$ is a sheaf.
By Proposition~\ref{P:acyclicity template}, it suffices to check \eqref{eq:Tate-Kiehl properties} for simple Laurent coverings.

We say that $\calR$ satisfies the \emph{Kiehl glueing property} if for every rational subdomain $U$ of $\Spa(A,A^+)$, the functor from the category of finite projective $\calR(U)$-modules to the category of sheaves of $\calR$-modules over $U$ which are locally free of finite rank is an equivalence of categories.
\end{defn}

\begin{theorem} \label{T:tate to Kiehl}
Let $(A,A^+)$ be a sheafy adic Banach ring. Then the structure sheaf on $\Spa(A,A^+)$ satisfies the Tate sheaf property and the Kiehl glueing property.
\end{theorem}
\begin{proof}
The Tate sheaf property is a consequence of Theorem~\ref{T:Tate sheaf property for structure sheaf}. The Kiehl glueing property follows from acyclicity plus Proposition~\ref{P:Tate reduction} and Proposition~\ref{P:glue projective}.
\end{proof}

\begin{remark}
By Theorem~\ref{T:Tate-Kiehl}, any adic affinoid algebra over an analytic field
is sheafy. Some additional cases in which we will establish the sheafy property, and hence the Tate and Kiehl properties, will be
perfect uniform Banach $\Fp$-algebras (Theorem~\ref{T:Tate-Kiehl analogue1}),
perfectoid algebras (Theorem~\ref{T:Tate-Kiehl analogue2}),
and preperfectoid algebras (Theorem~\ref{T:Kiehl for preperfectoid}).
For some additional examples of presheaves of rings satisfying the Tate and Kiehl properties, see \S\ref{subsec:sheaf properties}.
\end{remark}

\begin{remark} \label{R:new Beauville-Laszlo}
As an aside, we use the glueing formalism to produce a new proof of the Beauville-Laszlo theorem as formulated in Proposition~\ref{P:reduced descent}. (Note that \cite{beauville-laszlo} also includes a somewhat stronger result which we do not treat here.)

Set
\[
R_1 = \widehat{R}, \qquad R_2 = R[t^{-1}], \qquad R_{12} = \widehat{R}[t^{-1}].
\] 
Since $t$ is not a zero divisor in $R$, the maps $R \to R_2, R_1 \to R_{12}$ are both injective. It is clear that $R_1 \oplus R_2 \to R_{12}$ is surjective, so we obtain a glueing square.

Given a finite glueing datum, set notation as in the first paragraph of the proof of Lemma~\ref{L:Kiehl lemma}. Since $R_2 \to R_{12}$ has dense image for the $t$-adic topology, we may choose a matrix $B'$ over $R_2$ so that $A(B'-B)$ has entries in $t R_1$. Put $C_1 = 1 + A(B'-B) \in \GL_n(R_1)$ and $C_2 = 1 \in \GL_n(R_2)$; we may then continue as in Lemma~\ref{L:Kiehl lemma} to conclude that $M \otimes_R R_1 \to M_1$ is surjective.

Since $\Spec(R_1 \oplus R_2) \to \Spec(R)$ is surjective, the hypotheses of Lemma~\ref{L:Kiehl generic2} are satisfied. Proposition~\ref{P:reduced descent} follows at once.
\end{remark}

\subsection{Uniform Banach rings}

In the classical theory of Banach algebras over $\RR$ or $\CC$, an important role is played by the class of \emph{uniform function algebras} (i.e., algebras of continuous functions on compact spaces topologized using the supremum norm). We now introduce the analogous objects in the nonarchimedean setting.

\begin{defn} \label{D:uniform Banach ring}
The following conditions on a Banach ring $A$ are equivalent.
\begin{enumerate}
\item[(a)]
The norm on $A$ is equivalent to some power-multiplicative norm.
\item[(b)]
The norm on $A$ is equivalent to its spectral seminorm (which we therefore also call the \emph{spectral norm}).
\item[(c)]
There exists $c>0$ such that $\left| x^2 \right| \geq c \left|x\right|^2$ for all $x \in A$. (One gets another equivalent condition by replacing 2 with any larger integer.)
\item[(d)]
The subring $A^{\circ}$ of $A$ is bounded.
\end{enumerate}
If these conditions hold, we say $A$ is \emph{uniform}. Any uniform Banach ring is reduced; the converse is false in general, but any reduced affinoid algebra over an analytic field is uniform by Corollary~\ref{C:spectral is norm}.
We say that an adic Banach ring $(A,A^+)$ is \emph{uniform} if $A$ is a uniform Banach ring.
\end{defn}

\begin{example} \label{exa:Tate algebra}
For $A$ a uniform Banach ring and $r_1,\dots,r_n > 0$, the Tate algebra
\[
A\{T_1/r_1,\dots,T_n/r_n\}
\] 
is again uniform; see Definition~\ref{D:Gauss norm}.
\end{example}

\begin{remark} \label{R:transform uniform}
For $A$ a uniform Banach ring, by Theorem~\ref{T:transform} the spectral norm on $A$ is equal to the restriction of the supremum norm on $\prod_{\alpha \in \calM(A)} \calH(\alpha)$ along
the Gel'fand transform. Here are some notable consequences.
\begin{enumerate}
\item[(a)]
Any bounded homomorphism $A \to B$ of uniform Banach rings equipped with their spectral norms is submetric in general, and isometric if and only if
$\calM(B) \to \calM(A)$ is surjective.
\item[(b)]
For $(A,A^+)$ a uniform adic Banach ring, the map $A \to H^0(\Spa(A,A^+), \calO)$ is injective. That is, $(A,A^+)$ can only fail to be sheafy if local sections fail to glue.
\end{enumerate}
\end{remark}

For uniform Banach rings, we have the following criterion for projectivity of finitely generated modules, analogous to criterion (d) in Definition~\ref{D:projective} for reduced rings.

\begin{prop} \label{P:finite generation2}
Let $A$ be a uniform Banach ring
and let $M$ be a finitely generated $A$-module. Then the following conditions are equivalent.
\begin{enumerate}
\item[(a)]
The module $M$ is projective.
\item[(b)]
The rank function
$\beta \mapsto \dim_{\calH(\beta)} (M \otimes_A \calH(\beta))$ on $\calM(A)$
is continuous.
\item[(c)]
There exists a bounded homomorphism $A \to B$ of uniform Banach rings such that
$\calM(B) \to \calM(A)$ is surjective and $M \otimes_A B$ is a projective $B$-module.
\end{enumerate}
\end{prop}
\begin{proof}
If (a) holds, then the function $\gothp \to \dim_{A/\gothp} (M \otimes_A (A/\gothp))$
on $\Spec(A)$ is continuous. By restricting along the map $\calM(A) \to \Spec(A)$,
we obtain (b). 

If (b) holds, then the function $\beta \mapsto \dim_{\calH(\beta)} (M \otimes_A \calH(\beta))$ is constant on each set in some finite disconnection of $\calM(A)$.
By Proposition~\ref{P:affinoid system disconnection} and the relationship between disconnections of $\calM(A)$ and $\Spa(A,A^\circ)$ (Definition~\ref{D:Berkovich rational subspace}), this disconnection descends to $\Spec(A)$,
so we may reduce to the case where $\dim_{\calH(\beta)} (M \otimes_A \calH(\beta))$
is equal to a constant value $n$. For each maximal ideal $\gothp$ of $A$, we may
choose $\alpha \in \calM(A)$ with $\gothp_\alpha = \gothp$ (see Definition~\ref{D:residue field}).
Choose elements $\bv_1,\dots,\bv_n$ of $M$ whose images in $M \otimes_A \calH(\alpha)$
are linearly independent. Then $\bv_1,\dots,\bv_n$ form a basis of $M \otimes_A A/\gothp$, so
they also generate $M \otimes_A A_\gothp$ by Nakayama's lemma.
We may then choose $f \in A \setminus \gothp$ so that $\bv_1,\dots,\bv_n$ generate
$M \otimes_A A[f^{-1}]$.
Suppose that $\bv_1,\dots,\bv_n$ fail to form a basis of $M \otimes_A A[f^{-1}]$; then there must
exist $a_1,\dots,a_n \in A$ not all mapping to zero in $A[f^{-1}]$
and a nonnegative integer $m$ such that $f^m a_1 \bv_1 + \cdots + f^m a_n \bv_n = 0$.
For each $\beta \in \calM(A)$, if $\beta(f) = 0$, then obviously $\beta(f^m a_i) = 0$ for $i=1,\dots,n$.
Otherwise, $\gothp_\beta \in \Spec(A[f^{-1}])$ and so $\bv_1,\dots,\bv_n$ generate
$M \otimes_A A/{\gothp_\beta}$, again without relations because
$\dim_{\calH(\beta)} (M \otimes_A \calH(\beta)) = n$. Hence $\beta(f^m a_i) = 0$ for $i=1,\dots,n$ again.
By Theorem~\ref{T:transform}, we deduce that $f^m a_i = 0$ for $i=1,\dots,n$,
a contradiction. We conclude that
$M \otimes_A A[f^{-1}]$ is a free $A[f^{-1}]$-module;
in other words, $M$ is free over a distinguished open subset of $\Spec(A)$ containing the
original maximal ideal $\gothp$ as well as all other prime ideals contained in $\gothp$.
We may thus cover $\Spec(A)$ by such open subsets, so (a) holds.

If (a) holds, then (c) is evident. Conversely, if (c) holds, then the function
$\gamma \mapsto \dim_{\calH(\gamma)} (M \otimes_A \calH(\gamma))$ on $\calM(B)$ is continuous by the
previous paragraph.
This function factors through the function $\beta \mapsto \dim_{\calH(\beta)} (M \otimes_A \calH(\beta))$ on $\calM(A)$;
the latter is forced to be continuous because $\calM(B) \to \calM(A)$ is a surjective continuous map of
compact spaces and hence a quotient map (Remark~\ref{R:compact spaces}(b)). We thus deduce (b). Hence
all three conditions are equivalent.
\end{proof}

\begin{remark} \label{R:uniformity unstable}
Unfortunately, the class of uniform Banach rings is not stable under some key operations.
\begin{itemize}
\item
For $A \to B, A \to C$ morphisms of uniform Banach rings, the completed tensor product 
$B \widehat{\otimes}_A C$ need not be uniform. A simple example is $A = \QQ_p$ and $B = C = \CC_p$ (the completion of an algebraic closure of $\QQ_p$). For a special case where uniformity is preserved, see Lemma~\ref{L:unramified uniform extension}.
\item
For $(A,A^+)$ a uniform adic Banach ring, a rational localization of $\Spa(A,A^+)$ need not be uniform; see Example~\ref{exa:nonuniform rational localization}.
When this is always true, we say that $(A,A^+)$ is \emph{stably uniform}; see
Theorem~\ref{T:uniform rational is sheafy} for an example of this condition.
\end{itemize}
One operation under which the class of uniform Banach rings does turn out to be stable is the formation of finite \'etale extensions; see Proposition~\ref{P:finite etale Banach norm} and Remark~\ref{R:finite etale plus ring}.
\end{remark}

\begin{lemma} \label{L:unramified uniform extension}
Let $k \subseteq \ell$ be an extension of perfect fields of characteristic $p$.
Put $K = W(k)[p^{-1}]$ and $L = W(\ell)[p^{-1}]$.
Let $A$ be a uniform Banach algebra over $K$ equipped with the spectral norm.
Then the tensor product seminorm on $A \otimes_K L$ is power-multiplicative;
in particular, $A \widehat{\otimes}_K L$ is again uniform.
\end{lemma}
\begin{proof}
Let $S$ be a basis of $\ell$ over $k$; we can then write each element $b \in A \otimes_K L$ uniquely in the form $\sum_{s \in S} a_s \otimes [s]$
for some $a_s \in A$, all but finitely of which are zero.
In terms of such a representation, we have
\[
\left| b \right| = \max\{\left|a_s\right|: s \in S\}
\]
and hence
\[
\left| b^p - \sum_{s \in S} a_s^p \otimes [s^p] \right|
\leq p^{-1} \left| b \right|^p.
\]
Since $k$ and $\ell$ are perfect, $\{s^p: s \in S\}$
is also a basis of $\ell$ over $k$, 
so
\[
\left| \sum_{s \in S} a_s^p \otimes [s^p] \right|
= \max\{\left| a_s^p \right|: s \in S\}.
\]
It follows that $|b^p| = c^p$, proving the claim.
\end{proof}

The following example is due to Mihara \cite{mihara}.
\begin{example} \label{exa:nonuniform rational localization}
Let $K$ be an analytic field and pick any $r>0$.
Let $A$ be the closure of the $K$-subalgebra of $K\{X/r, U\}$ generated by
$U^n X^{\lceil \log_2 n \rceil}$ for $n=1,2,\dots$. Then
$A$ is uniform because it is a subring of the Tate algebra $K\{X/r,U\}$
(see Example~\ref{exa:Tate algebra}).
However, for $(A,A^\circ) \to (B,B^+)$ the rational localization corresponding to the set $\{v \in \Spa(A,A^\circ): v(X) \leq 1\}$, the Banach ring $B$ is not uniform
\cite[Theorem~3.11]{mihara}.

It is not known whether $(A,A^\circ)$ is sheafy. However, for $B$ the subset of the infinite product $A \times A \times \cdots$ on which the supremum norm is bounded,
Mihara shows that $\Spa(B,B^\circ)$ is not sheafy \cite[Theorem~3.15]{mihara}.
For another example, see \cite[\S 4.6]{buzzard-verberkmoes}.
\end{example}

The following lemma is \cite[Proposition~2.3]{mihara}.
\begin{lemma} \label{L:principal ideal closed}
For any uniform Banach ring $A$ and any $f \in A$, the ideals $(T-f)$  and $(1-fT)$ in $A\{T\}$ and $A\{T,T^{-1}\}$ are closed.
\end{lemma}
\begin{proof}
Equip $A$ with the spectral norm.
For each $\alpha \in \calM(A)$, let $\tilde{\alpha} \in \calM(A\{T\})$ be the Gauss norm relative to $\alpha$. 
For any $g \in A\{T\}$, we then have $\tilde{\alpha}(T-f) = \max\{1, \alpha(f)\} \geq 1$, so
\[
\tilde{\alpha}(g) \leq \tilde{\alpha}(T-f) \tilde{\alpha}(g) = \tilde{\alpha}((T-f)g).
\]
Since $\tilde{\alpha}$ equals the spectral norm on $\calH(\alpha)\{T\}$,  
taking the supremum of the $\tilde{\alpha}$ computes the spectral norm on $A\{T\}$ by Theorem~\ref{T:transform}. We thus deduce that
\[
\left| g \right| \leq \left| (T-f) g \right|,
\]
so multiplication by $T-f$ defines a strict endomorphism of $A\{T\}$. This proves the claim for the ideal $(T-f)$ in $A\{T\}$; the other cases are similar.
\end{proof}

The following consequence is a special case of \cite[Corollary~4]{buzzard-verberkmoes},
but our proof is slightly different.
\begin{cor} \label{C:principal ideal closed}
Let $(A,A^+)$ be a uniform adic Banach ring. Let $\{(A,A^+) \to (B_i,B_i^+)\}_{i=1}^2$ be the standard Laurent covering defined by some $f \in A$, and put
$B_{12} = B_1 \widehat{\otimes}_A B_2$. Then the sequence
\[
0 \to A \to B_1 \oplus B_2 \to B_{12} \to 0
\]
is exact.
\end{cor}
\begin{proof}
In the diagram
\begin{equation} \label{eq:principal ideal closed diagram}
\xymatrix{
& & 0 \ar[d] & 0 \ar[d] & \\
& & (T-f)A\{T\} \oplus (1-fU) A\{U\} \ar[r]\ar[d] & (T-f)A\{T,T^{-1} \} \ar[r] \ar[d] & 0 \\
0 \ar[r] & A \ar[r] \ar[d] & A\{T\} \oplus A\{U\} \ar[r] \ar[d] & A\{T,T^{-1} \} \ar[r]  \ar[d] & 0 \\
0 \ar[r] & A \ar[r] \ar[d]  & B_1 \oplus B_2 \ar[d] \ar[r] & B_{12} \ar[r] \ar[d] & 0 \\
 & 0 & 0 & 0 &
}
\end{equation}
the first and second rows and the first column are evidently exact, while the second and third columns are exact thanks to Lemma~\ref{L:rational subdomain} and 
Corollary~\ref{C:principal ideal closed}. 
By Theorem~\ref{T:transform}, the third row is exact at $A$; by
diagram chasing, it is also exact at the other positions.
\end{proof}

The following theorem is due to Buzzard and Verberkmoes \cite[Theorem~7]{buzzard-verberkmoes}.
\begin{theorem}[Buzzard-Verberkmoes] \label{T:uniform rational is sheafy}
A stably uniform adic Banach ring is sheafy.
\end{theorem}
\begin{proof}
Let $(A,A^+)$ be a uniform adic Banach ring. If $(A,A^+)$ is stably uniform, then by Corollary~\ref{C:principal ideal closed}, the structure presheaf satisfies the criterion of Proposition~\ref{P:acyclicity template}. Hence $(A,A^+$) is sheafy.
\end{proof}

\begin{remark} \label{R:sheafy to stably uniform}
We do not know whether conversely to Theorem~\ref{T:uniform rational is sheafy}, any sheafy uniform Banach ring is necessarily stably uniform. To check this, it would suffice by Proposition~\ref{P:Tate reduction single} to check that for any uniform sheafy adic Banach ring $(A,A^+)$ and any $f \in A$, the quotient $A\{T\}/(T-f)$ is again uniform. 

One easy but important special case is that if $(A,A^+)$ is sheafy and admits a rational covering by stably uniform adic Banach rings, then $A$ is stably uniform.
This is of particular interest when the covering spaces are perfect (Proposition~\ref{P:locally perfect is perfect}) or perfectoid (Remark~\ref{R:locally perfectoid is perfectoid}).
\end{remark}

\begin{remark} \label{R:check uniform by analytic}
Let $A$ be a Banach algebra over an analytic field $K$, let $K \to L$ be a morphism of analytic fields, and put $A_L = A \widehat{\otimes}_K L$.
Since $K \to L$ is an isometric inclusion, it is strict; by
Lemma~\ref{L:inject tensor}(c), the map $A \to A_L$ is also strict. Using
criterion (c) of Definition~\ref{D:uniform Banach ring}, we see that
if $A_L$ is uniform, then so is $A$ (but not conversely; see Remark~\ref{R:uniformity unstable}). Moreover, if $A_L$ is stably uniform, then any rational localization $(A,A^+) \to (B,B^+)$ gives rise to a rational localization $(A_L, A_L^+) \to (B_L, B_L^+)$, and so $A$ is also stably uniform.

A closely related observation is that if $(A, A^+) \to (B,B^+)$ is a bounded morphism of adic Banach rings such that $B$ is stably uniform and $A \to B$ splits in the category of Banach modules over $A$, then $A$ is stably uniform. Namely, the splitting persists under rational localizations, so we need only check that $A$ is uniform; this again follows from criterion (c) of Definition~\ref{D:uniform Banach ring}.
\end{remark}

In light of Remark~\ref{R:uniformity unstable}, we are compelled to introduce the following construction.
\begin{defn} \label{D:uniformization}
For any Banach ring $A$, the separated completion of $A$ for the spectral seminorm
is a uniform Banach ring, called the \emph{uniformization} of $A$ and denoted $A^u$.
Note that the natural map $A \to A^u$ induces a continuous bijection $\calM(A^u) \to \calM(A)$, which is thus a homeomorphism by Remark~\ref{R:compact spaces}(b).
Using this observation plus Theorem~\ref{T:transform} (as in Remark~\ref{R:transform uniform}),
we see that 
uniformization defines a left adjoint to the forgetful functor from uniform Banach rings to arbitrary Banach rings. 

For $(A,A^+)$ an adic Banach ring, the \emph{uniformization} of $(A,A^+)$ is defined as $(A^u, A^{u+})$, where $A^{u+}$ is the completion of the image of $A^+$ in $A^u$. Again, this construction defines a left adjoint to the forgetful functor from uniform adic Banach rings to arbitrary adic Banach rings,
and the map $\Spa(A^u, A^{u+}) \to \Spa(A,A^+)$ is a homeomorphism which identifies rational subspaces on both sides. 
\end{defn}

\begin{lemma} \label{L:maximal uniform extension}
Let $A$ be a uniform Banach ring. 
Then for $B \in \FEt(A)$, there is a maximal power-multiplicative seminorm 
on $B$ for which the homomorphism $A \to B$ is bounded.
\end{lemma}
\begin{proof}
Let $\alpha$ be the spectral norm on $A$.
When $A$ and $B$ are analytic fields, the claim is clear:
since the homomorphism $A \to B$ is nonzero, it is bounded if and only if it is isometric,
and the only power-multiplicative extension of $\alpha$ to $S$ is the multiplicative extension
(Lemma~\ref{L:analytic fields submultiplicative}).

When $A$ is an analytic field and $B$ is arbitrary, we may split $B$ as a direct sum $B_1 \oplus \cdots \oplus B_n$ of finite separable field extensions of $A$. The maximal power-multiplicative seminorm in this case is
the supremum of the maximal seminorms on the $B_i$.

For general $A$, we may write $\alpha = \sup\{\gamma: \gamma \in \calM(A)\}$
by Theorem~\ref{T:transform}. We then take the supremum of the restrictions to $B$ of
the maximal power-multiplicative seminorms on the rings $B \otimes_R \calH(\alpha)$;
this is maximal by Theorem~\ref{T:transform} again.
\end{proof}

\begin{lemma} \label{L:locally monogenic}
For any Banach ring $A$ and any $B \in \FEt(A)$, there exists a strong rational covering $((A,A^\circ) \to (A_i,A_i^+))_i$ such that for each $i$, $B \otimes_A A_i$ splits as a direct sum of monogenic $A$-algebras (i.e., algebras of the form $A[T]/(P)$ for some monic polynomial $P$).
\end{lemma}
\begin{proof}
By compactness, it suffices to check that for each $\alpha \in \calM(A)$, there exists a rational localization encircling $\alpha$ with the desired property. Using Lemma~\ref{L:henselian local ring} and compactness, we may reduce to the case where $B \otimes_A \calH(\alpha)$ is connected. By the primitive element theorem
and Lemma~\ref{L:nearby generators}, the primitive elements of
$B \otimes_A \calH(\alpha)$ form an open subset; we may thus choose a primitive element $f$ belonging to $\Frac(B)$.
By Lemma~\ref{L:henselian local ring} again, we may reduce to the case where the minimal polynomial of $f$ over $\Frac(A)$ has coefficients in $A$. Let $P \in A[T]$ be this polynomial; its resultant is nonzero at $\alpha$, so we may reduce to the case where the resultant has nonzero norm on all of $\calM(A)$, and hence is a unit by
Corollary~\ref{C:unit from spectrum}. Hence $A[T]/(P)$ is a finite \'etale $A$-algebra which becomes isomorphic to $B$ upon base extension from $A$ to $\calH(\alpha)$; by  Lemma~\ref{L:henselian local ring} again, after replacing $A$ by some localization we obtain an isomorphism $B \cong A[T]/(P)$ as desired.
\end{proof}

\begin{prop} \label{P:finite etale Banach norm}
Let $A$ be a Banach ring.
\begin{enumerate}
\item[(a)]
The functor $\FEt(A) \to \FEt(A^u)$ is a tensor equivalence.
\item[(b)]
For $B \in \FEt(A^u)$, the power-multiplicative seminorm given by Lemma~\ref{L:maximal uniform extension} is equivalent to any norm on $B$
provided by Definition~\ref{D:etale algebra Banach structure}, and thus provides $B$ with the structure of a uniform Banach algebra over $A^u$. 
\end{enumerate}
\end{prop}
\begin{proof}
To prove (a), apply Lemma~\ref{L:construct affinoid system} to construct an affinoid system $\{((A_i, A_i^+), \alpha_i)\}_{i \in I}$ with completed direct limit $A$.
Define another affinoid system $\{((B_{i}, B_{i}^+), \beta_{i})\}_{i \in I}$ in which
$B_i$ is the reduced quotient of $A_i$ and $\beta_i$ is the spectral norm, under which $B_i$ is complete by Corollary~\ref{C:spectral is norm}.
We have $\FEt(A_i) \cong \FEt(B_i)$ by  Theorem~\ref{T:henselian}.
The completed direct limit of the new affinoid system is 
$A^u$, so by Proposition~\ref{P:henselian direct limit2} we have tensor equivalences
\[
\FEt(A) \cong \FEt(\varinjlim_i A_i) \cong \FEt(\varinjlim_i B_{i}) \cong \FEt(A^u).
\]

To prove (b), we may assume $A$ is uniform with norm $\alpha$.
By Theorem~\ref{T:transform}, we have $\alpha = \sup\{\gamma: \gamma \in \calM(A)\}$ and $\beta \in \sup\{\gamma: \gamma \in \calM(B)\}$.
Suppose first that $A$ is an analytic field; then $B$ is a direct sum of analytic fields containing $A$ (see Definition~\ref{D:Banach ring}), so in particular $B$ is complete under $\beta$. The equivalence of norms follows from the open mapping theorem (Theorem~\ref{T:open mapping}).

Suppose next that $B$ is a direct sum of monogenic extensions; we immediately reduce to the case where $B = A[T]/(P)$ is monogenic of degree $d>0$.
In this case, we compare $\beta$ to the supremum norm $\beta'$ defined by the basis $1, T, \dots, T^{d-1}$.
Since $\beta$ is equal to the spectral seminorm of $\beta'$, we need only check that $\beta' \leq c \beta$ for some $c>0$.
For $\gamma \in \calM(A)$, choose an algebraic closure $L$ of $\calH(\gamma)$ and let $z_1,\dots,z_d$ be the roots of $P$ in $L$. For $a_0,\dots,a_{d-1} \in A$, the supremum of $Q =  \sum_{i=0}^{d-1} a_i T^i \in A[T]$ over those $\delta \in \calM(B)$ lying over $\gamma$ may be computed as the supremum of $\left| Q(z_j) \right|$ for $j=1,\dots,d$.
Let $V$ be the Vandermonde matrix in $z_1,\dots,z_d$; then
\[
\max\{\gamma(a_i): i=0,\dots,d-1\} \leq \left| V^{-1} \right|
\max\{\left| Q(z_j) \right|: j=1,\dots,d\}.
\]
For $P = T^d + \sum_{i=0}^{d-1} P_i T^i$, we have
\[
\left| z_i \right| \leq c_1, \qquad c_1 = \max\{\alpha(P_{d-i})^{1/i}: i=1,\dots,d\}.
\]
Let $f$ be the resultant of $P$; since $B \in \FEt(A)$, $f$ is a unit in $A$. By writing $V^{-1}$ as $\det(V)^{-1}$ times the cofactor matrix of $V$, we compute that
\[
\beta' \leq c \beta, \qquad c = c_1^{d-1} \alpha(f^{-1}),
\]
proving the claim.

For general $A$ and $B$, apply Lemma~\ref{L:locally monogenic} to construct a strong covering family $\{(A,A^\circ) \to (A_i, A_i^+)\}_{i=1}^n$ such that for $i=1,\dots,n$, $B \otimes_A A_i$ splits a direct sum of monogenic extensions of $A_i$. 
Let $\beta'$ be any norm derived as in Lemma~\ref{L:finite projective Banach}.
To show that $\beta' \leq c \beta$ for some $c>0$,
we reduce to the previous paragraph:
by Theorem~\ref{T:transform}, both $\beta$ and $\beta'$ may be computed by taking suprema over the corresponding norms on $B \otimes_A A^u_i$ for $i=1,\dots,n$.
\end{proof}

\begin{remark} \label{R:finite etale Banach norm}
The conclusions of Proposition~\ref{P:finite etale Banach norm} continue to hold without assuming that $A$ contains a topologically nilpotent unit.
The primary modification needed is to redefine
the Hausdorff localization $A_\alpha$ in terms of Berkovich rational subspaces
(Remark~\ref{R:Berkovich rational subspaces}).
\end{remark}

\begin{remark} \label{R:warp norm}
For $A$ a uniform Banach ring, recall 
(Remark~\ref{R:transform}) that for any function $g: \calM(A) \to \RR^+$ whose image is bounded away from 0 and $\infty$, 
the function $\left| \cdot \right|_g = \sup\{\alpha^{g(\alpha)}: \alpha \in \calM(A)\}$ is a norm defining the topology on $A$. In particular, for any topologically nilpotent $z$ in $A$, we can choose $g$ so that $\left| z \right|_g \left| z^{-1} \right|_g = 1$.
That is, for uniform Banach rings, conditions (b) and (c) of Remark~\ref{R:small norm elements} become equivalent if we only keep track of the norm topology, rather than the equivalence class of the norm.
\end{remark}

\begin{remark} \label{R:finite etale plus ring}
With notation as in Remark~\ref{R:etale adic extension}, if $A^+ = A^\circ$, then $B^+ = B^\circ$: for $b \in B^\circ$, the characteristic polynomial of multiplication by $b$ as an $A$-linear endomorphism of $B$ is monic, has $b$ as a root (by Cayley-Hamilton), and has coefficients in $A^\circ$ (by Theorem~\ref{T:transform} to reduce to the case of an analytic field, plus usual properties of Newton polygons as in \cite[Chapter~2]{kedlaya-course}).
\end{remark}

\section{Perfect rings and strict \texorpdfstring{$p$}{p}-rings}

Recall that there is a natural way to lift perfect rings of characteristic $p$ to \emph{strict $p$-rings}
of characteristic $0$ (see Definition~\ref{D:strict p-ring}), and that these can be used to describe
\'etale local systems on perfect rings of characteristic $p$ using a nonabelian generalization of
Artin-Schreier-Witt theory (see Proposition~\ref{P:phi-modules}).
We take a first step towards exploiting this description in $p$-adic Hodge theory by setting up a
correspondence between certain highly ramified analytic fields of mixed characteristics
and perfect analytic fields of characteristic $p$. This
correspondence, which has also been described recently by Scholze \cite{scholze1}, generalizes
the \emph{field of norms} construction of Fontaine-Wintenberger \cite{fontaine-wintenberger} but with
a rather different proof.
We then extend the correspondence to Banach algebras, in the direction of generalizing
Faltings's \emph{almost purity theorem}; however, this will not be completed until we
introduce extended Robba rings (see \S\ref{subsec:almost purity}).

\setcounter{theorem}{0}
\begin{convention}
We will refer frequently to \emph{(adic) Banach $\Fp$-algebras} even though $\Fp$ cannot be viewed as an analytic field. What we will mean are (adic) Banach rings whose underlying rings are of characteristic $p$.
\end{convention}

\subsection{Perfect \texorpdfstring{$\Fp$}{Fp}-algebras}

We begin with some observations about perfect rings of characteristic $p$, which we may more briefly
characterize as \emph{perfect $\Fp$-algebras}.

\begin{defn}
For $R$ an $\Fp$-algebra, let $\overline{\varphi}: R \to R$ denote the $p$-th power map,
i.e., the \emph{Frobenius endomorphism}. We say $R$ is \emph{perfect} if $\overline{\varphi}$
is a bijection; this forces $R$ to be
reduced. Note that any localization of a perfect $\Fp$-algebra is also perfect.

By a \emph{perfect uniform Banach $\Fp$-algebra} (resp. a \emph{perfect uniform adic Banach $\Fp$-algebra}), we will mean a uniform Banach algebra $R$ (resp.\ a uniform adic Banach algebra $(R,R^+)$) over $\Fp$ such that $R$ is a perfect ring. In the adic case, the fact that $R^+$ is integrally closed forces it to also be perfect.
We will see shortly that any perfect uniform adic Banach $\Fp$-algebra is stably uniform (Proposition~\ref{P:perfect uniform localization}).
\end{defn}

We will sometimes have need to pass from an $\Fp$-algebra to an associated perfect $\Fp$-algebra.
\begin{defn}
Let $R$ be an $\Fp$-algebra.
The \emph{perfect closure} of $R$ is the limit $R^{\perf}$
of the direct system
\[
R \stackrel{\overline{\varphi}}{\to} R \stackrel{\overline{\varphi}}{\to} \cdots,
\]
viewed as an $R$-algebra via the map to the first factor. (The map $R \to R^{\perf}$ induces an
injection of the reduced quotient of $R$ into $R^{\perf}$.)
Any power-multiplicative seminorm on $R$ extends uniquely to a
power-multiplicative seminorm on $R^{\perf}$; in particular, given a power-multiplicative norm on $R$,
we may extend it to $R^{\perf}$ and then obtain homeomorphisms $\calM(R^{\perf}) \to \calM(R)$ and $\Spa(R^{\perf}, R^{+,\perf}) \to \Spa(R,R^+)$.
We will also call $R^{\perf}$ the \emph{direct perfection} of $R$, to distinguish it from the \emph{inverse perfection}
in which one takes the arrows in the opposite direction; we will have more use for the latter construction in
\S\ref{subsec:perfection}.
\end{defn}

\begin{lemma} \label{L:idempotents}
Let $R$ be a perfect $\Fp$-algebra.
If $e \in R$ satisfies $e^p = e$, then
\[
e_i = \prod_{j \in \Fp \setminus \{i\}} \frac{e-j}{i-j} \qquad (i \in \Fp)
\]
is an idempotent in $R$ and $\sum_{i \in \Fp} i e_i = e$.
\end{lemma}
\begin{proof}
Note that $e_i(e-i)$ is divisible by $e^p - e$ and thus equals 0. Hence
\[
e_i^2 = e_i \prod_{j \in \Fp \setminus \{i\}} \frac{e-j}{i-j}
= e_i \prod_{j \in \Fp \setminus \{i\}} \frac{i-j}{i-j} = e_i.
\]
The identity $\sum_{i \in \Fp} i e_i = e$ arises from Lagrange interpolation of the
polynomial $T \in \Fp[T]$ at the points of $\Fp$.
\end{proof}
\begin{cor} \label{C:idempotents}
The ring $R^{\overline{\varphi}}$ is the $\Fp$-algebra generated by the idempotents of $R$.
In particular, this ring equals $\Fp$ if and only if $R$ is connected.
\end{cor}

\begin{lemma} \label{L:perfect etale extension}
For $R$ a perfect $\Fp$-algebra, any $S \in \FEt(R)$ is also perfect.
\end{lemma}
\begin{proof}
Since $S$ is \'etale over the reduced ring $R$, it is also reduced
\cite[Proposition~17.5.7]{ega4-4}; hence $\overline{\varphi}: S \to S$ is injective.
Since $\Omega_{S/R} = 0$ by \cite[Proposition~17.2.1]{ega4-4},
$S$ is generated over $R$ by $S^p$ by \cite[Proposition~0.21.1.7]{ega4-1}.
Combining this with the surjectivity of $\overline{\varphi}: R \to R$ yields surjectivity of
$\overline{\varphi}:S \to S$. Hence $S$ is perfect, as desired.
\end{proof}

The study of perfect uniform Banach $\Fp$-algebras is greatly simplified by the following observations.
For some related results in characteristic 0, see \S\ref{subsec:perfectoid2}.
\begin{remark} \label{R:perfect uniform strict}
Let $R,S,T$ be perfect uniform Banach $\Fp$-algebras.
\begin{enumerate}
\item[(a)]
Any strict homomorphism $f: R \to S$ is almost optimal (but not necessarily optimal).
\item[(b)]
If $f_1, f_2: R \to S$ are homomorphisms such that $f_1 - f_2$ is strict,
then $f_1 - f_2$ is almost optimal.
\item[(c)]
For any bounded homomorphisms $T \to R$, $T \to S$, the completed tensor product $R \widehat{\otimes}_T S$
is again a perfect uniform Banach $\Fp$-algebra.
(By contrast, the completed tensor product of uniform Banach rings is not guaranteed to
be uniform in general.)
\item[(d)] If $I$ is a closed and perfect ideal of $R$, then $R/I$ is also a perfect and uniform $\Fp$-algebra. 
\end{enumerate}
The proofs of (a)--(c) are all similar, so we only describe (a) in detail.
Choose $c>0$ such that every $x \in \image(f)$ lifts to some $y \in R$
with $|y| \leq c |x|$. For any positive integer $n$, $x^{p^n}$ then lifts to some $y_n \in R$
with $|y_n| \leq c |x^{p^n}|$; we may then lift $x$ to $y_n^{p^{-n}}$ and note that
\[
|y_n^{p^{-n}}| = |y_n|^{p^{-n}} \leq c^{p^{-n}} |x^{p^n}|^{p^{-n}} = c^{p^{-n}} |x|.
\]
Since $c^{p^{-n}}$ can be made arbitrarily close to 1, this proves the claim.

To prove (d), note that the perfectness is obvious; it remains to show that the spectral seminorm on $R/I$ is a norm. Suppose the spectral seminorm for some $\bar{y}\in R/I$ is 0. That is, for any $\epsilon>0$,  there exist some $n\in\mathbb{N}$ and $y_n\in R$ lifting $\bar{y}^{p^n}$ such that $|y_n|^{p^{-n}}=|y_n^{p^{-n}}|\leq \epsilon$. Since $I$ is perfect, $y_n^{p^{-n}}$ lifts $\bar{y}$. Hence the quotient norm for $\bar{y}$ is less than $\epsilon$, yielding that $\bar{y}=0$. 
\end{remark}

\begin{prop} \label{P:perfect uniform localization}
Let $(R,R^+)$ be a perfect uniform adic Banach $\Fp$-algebra.
Then any rational localization of $(R,R^+)$ is again perfect uniform; in particular,
$(R,R^+)$ is stably uniform.
 \end{prop}
\begin{proof}
Suppose $(R,R^+) \to (S,S^+)$ corresponds to a rational subdomain $U$ as in
\eqref{eq:adic rational subspace}. 
By Lemma~\ref{L:rational subdomain}, we may view $S$  as the quotient of 
$R\{T_1,\dots,T_n\}$ by the closure of the ideal $(gT_1 - f_1, \dots, gT_n - f_n)$.
Equip $S^{1/p}$ with the norm given by $|x|_{S^{1/p}} = |x^p|_S^{1/p}$.
By applying $\overline{\varphi}^{-1}$, raising norms to the $p$-th power, and using that $R$ is perfect uniform (so its norm remains unchanged), we obtain another rational localization $(R,R^+) \to (S^{1/p},(S^+)^{1/p})$ representing $U$.
By Lemma~\ref{L:rational subdomain} again, we may view $S^{1/p}$ as the quotient of $R\{T_1^{1/p},\dots,T_n^{1/p}\}$ by the closure of the ideal $(g^{1/p} T_1^{1/p} - f_1^{1/p}, \dots, g^{1/p} T_n^{1/p} - f_n^{1/p})$.
The inclusion $R\{T_1,\dots,T_n\} \to R\{T_1^{1/p},\dots,T_n^{1/p}\}$ then induces a morphism $(S,S^+) \to (S^{1/p}, (S^+)^{1/p})$ of adic Banach algebras over $(R,R^+)$,
which must be an isomorphism by the universal property of rational localizations.
It follows that $S$ is perfect and uniform.
\end{proof}

\begin{remark} \label{R:perfect uniform localization1}
Retain notation as in Proposition~\ref{P:perfect uniform localization}.
Let $R'$ be the completed perfect closure of $R\{T_1,\dots,T_n\}$.
Let $I'$ be the closure of the ideal of $R'$ generated by $(gT_i - f_i)^{p^{-h}}$ for $i=1,\dots,n$ and $h=0,1,\dots$.
Put $S' = R'/I'$; by Remark~\ref{R:perfect uniform strict}(d), $S'$ is perfect uniform.
By the universal property of rational localizations, we obtain a morphism $S \to S'$
inducing a bijection $\calM(S') \cong \calM(S)$.
Since $S$ is uniform, the map $S \to S'$ is injective by Theorem~\ref{T:transform}.

Choose $h_1,\dots,h_n,k \in R$ such that $h_1 f_1 + \cdots + h_n f_n + kg = 1$, then note that any element 
\[
y = \sum_{i_1,\dots,i_n} y_{i_1,\dots,i_n} T_1^{i_1} \cdots T_n^{i_n} \in R\{T_1,\dots,T_n\}^{\perf}
\]
represents the same element of the quotient as 
\[
z = (k + h_1 T_1 + \cdots + h_n T_n)^n \sum_{i_1,\dots,i_n} y_{i_1,\dots,i_n} f_1^{i_1 - \lfloor i_1 \rfloor}
\cdots f_n^{i_n - \lfloor i_n \rfloor} g^{n - (i_1 - \lfloor i_1\rfloor + \cdots + i_n - \lfloor i_n\rfloor)} T_1^{\lfloor i_1 \rfloor} \cdots T_n^{\lfloor i_n \rfloor},
\]
which satisfies
\[
\left| z \right| \leq c \left| y \right|, \qquad c  = \max\{\left| h_1 \right|,\dots, \left| h_n \right|, \left| k \right|\}^n \left|f_1 \right| \cdots \left| f_n \right| \left| g \right|^n.
\]
Thus the map $S \to S'$ is also surjective, hence an isomorphism.

An additional consequence of this calculation is that while the map $R\{T_1,\dots,T_n\} \to S$ is not almost optimal, for any $c>1$ we can arrange for the quotient norm to be at most $c$ times the spectral norm on $S$,
by running the construction with $f_1,\dots,f_n,g$ replaced by suitable $p$-th power roots.
\end{remark}

For perfect uniform Banach $\Fp$-algebras, we have the following refinement of
Proposition~\ref{P:finite etale Banach norm}.
\begin{lemma} \label{L:finite uniform Banach}
Let $R$ be a perfect uniform Banach $\Fp$-algebra with norm $\alpha$.
Let $S$ be a finite perfect $R$-algebra admitting the structure of a finite Banach module over $R$
for some norm $\beta$. (Such $\beta$ exists when $S$ is projective as an $R$-module
by Lemma~\ref{L:finite projective Banach}, and hence when $S \in \FEt(R)$.)
Then $S$ is a perfect uniform Banach algebra.
\end{lemma}
\begin{proof}
Equip $S \otimes_R S$ with the product seminorm induced by $\beta$.
By Lemma~\ref{L:finitely generated Banach}, the multiplication map $\mu: S \otimes_R S \to S$
is bounded. Consequently, there exists $c>0$ such that
\begin{equation} \label{eq:finite uniform Banach}
\beta(xy) \leq c \beta(x) \beta(y) \qquad (x,y \in S).
\end{equation}
Rewrite \eqref{eq:finite uniform Banach} as $c \beta(xy) \leq (c \beta(x))(c \beta(y))$,
then apply Fekete's lemma to deduce that the limit $\gamma(x) = \lim_{n \to \infty} (c \beta(x^n))^{1/n}$
exists. From \eqref{eq:finite uniform Banach} again, we see that $\gamma$ is a power-multiplicative seminorm on $S$
and that $\gamma(x) \leq c \beta(x)$.

Let $R'$ be a copy of $R$ equipped with the norm $\alpha^p$;
the homomorphism $\overline{\varphi}^{-1}: R \to R'$ is isometric because $R$ is uniform.
Let $S'$ be a copy of $S$ equipped with the norm $\beta^p$; then $S'$ is a finite Banach module over $R'$
and the map $\overline{\varphi}^{-1}: S \to S'$ is semilinear with respect to
$\overline{\varphi}^{-1}: R \to R'$. By Lemma~\ref{L:finitely generated Banach} again,
$\overline{\varphi}^{-1}: S \to S'$ is bounded; that is, there exists $d>0$ such that for all $x \in S$,
$\beta(x^{1/p})^p \leq d \beta(x)$.
Equivalently, for all $x \in S$, $\beta(x^p) \geq d^{-1} \beta(x)^p$.
By induction on the positive integer $n$, we have $c \beta(x^{p^n}) \geq c d^{-1-p-\cdots-p^{n-1}} \beta(x)^{p^n}$;
by taking $p^n$-th roots and then taking the limit as $n \to \infty$, we deduce that
$\gamma(x) \geq d^{-1/(p-1)} \beta(x)$.

{}From the preceding paragraphs, $\gamma$ is equivalent to $\beta$;
it is thus a norm on $S$ under which $S$ is a perfect uniform Banach algebra.
\end{proof}

\begin{lemma} \label{L:isometric extension}
Let $R$ be a perfect uniform Banach $\Fp$-algebra,
and let $\gamma$ be an isometric automorphism of $R$ extending to an automorphism of $S \in \FEt(R)$.
Then $\gamma$ is also isometric on $S$ for the norm provided by Lemma~\ref{L:finite uniform Banach}
(or equivalently Proposition~\ref{P:finite etale Banach norm}).
\end{lemma}
\begin{proof}
Suppose first that $R=L$ is an analytic field.
Given $y \in S$, let $P  = \sum_i P_i T^i \in L[T]$ be the minimal polynomial of $y$.
As in the proof of Lemma~\ref{L:analytic fields submultiplicative},
we have $|y| = |P_0|^{1/d}$ for $d = \deg(P)$
and $|\gamma(y)| = |\gamma(P_0)|^{1/d} = |y|$ as desired.

We reduce the general case to the case of an analytic field using Theorem~\ref{T:transform}.
More precisely, for each $\beta \in \calM(R)$, we may use $\gamma$ to identify $\calH(\beta)$ with $\calH(\gamma^*(\beta))$,
then use the extended action of $\gamma$ to define an automorphism of $S \widehat{\otimes}_R \calH(\beta)$.
Since this automorphism is isometric by the previous paragraph, we may apply Theorem~\ref{T:transform}
to deduce that the action of $\gamma$ on $S$ is isometric.
\end{proof}

\begin{remark} \label{R:almost perfect}
Suppose that $R$ is a perfect uniform $\Fp$-algebra over a perfect analytic field $L$.
For $S \in \FEt(R)$, $\gotho_S$ is perfect because $S$ is, so $\Omega_{\gotho_S/\gotho_R} = 0$. This does not
imply that $\gotho_S$ is finite \'etale over $\gotho_R$, because $\gotho_S$ need not be a finite $\gotho_R$-module.
However, we can say that the quotient of $\gotho_S$ by the sum of its finitely generated projective $\gotho_R$-submodules
is killed by $\gothm_L$: the quotient by a single submodule is killed by a nonzero element $\overline{z} \in \gothm_L$, then use perfectness to replace $\overline{z}$ by $\overline{z}^{p^{-n}}$ for any nonnegative integer $n$.
A related statement in the language of almost ring theory is that $\gotho_S$ is \emph{almost finite \'etale} over $\gotho_R$;
see Theorem~\ref{T:almost purity} for a similar statement and derivation.
\end{remark}

As a consequence of these observations, we obtain analogues of the Tate-Kiehl theorems for perfect uniform
Banach algebras. For extensions of these results, see \S\ref{subsec:sheaf properties};
for an analogue for perfectoid algebras, see Theorem~\ref{T:Tate-Kiehl analogue2}.

\begin{lemma} \label{L:Tate lemma1}
Let $(R,R^+)$ be a perfect uniform adic Banach $\Fp$-algebra.
Choose $f \in R$ and let $\{(R,R^+) \to (R_i, R_i^+)\}_{i=1,2,12}$ represent the rational subdomains
\[
\{v \in \Spa(R,R^+): v(f) \leq 1\}, \quad
\{v \in \Spa(R,R^+): v(f) \geq 1\}, \quad
\{v \in \Spa(R,R^+): v(f) = 1\}
\]
of $\Spa(R,R^+)$.
Then the sequence
\begin{equation} \label{eq:strict exact sequence1}
0 \to R \to R_1 \oplus R_2 \to R_{12} \to 0
\end{equation}
is almost optimal exact (i.e., exact with each morphism being almost optimal).
\end{lemma}
\begin{proof}
Strict exactness follows from Corollary~\ref{C:principal ideal closed}; almost optimality then follows from Remark~\ref{R:perfect uniform strict}. One can also give a more direct proof using affinoid systems and Tate's theorem; we leave this as an exercise.
\end{proof}
\begin{theorem} \label{T:Tate-Kiehl analogue1}
Any perfect uniform adic Banach $\Fp$-algebra $(R,R^+)$ is sheafy,
and the structure sheaf on $\Spa(R,R^+)$
satisfies the Tate sheaf and Kiehl glueing properties
(see Definition~\ref{D:Tate-Kiehl properties}).
\end{theorem}
\begin{proof}
We may deduce sheafiness either from Theorem~\ref{T:uniform rational is sheafy}
and Proposition~\ref{P:perfect uniform localization} or from
Lemma~\ref{L:Tate lemma1} and Proposition~\ref{P:acyclicity template}.
The properties of the structure sheaf then follow from
Theorem~\ref{T:tate to Kiehl}.
\end{proof}

\begin{example}
For $X$ an arbitrary (possibly infinite) set and $R$ a ring, let $R[X]$ denote the free commutative $R$-algebra
generated by $X$, and write $R[X^{p^{-\infty}}]$ for $\cup_{n=1}^\infty R[X^{p^{-n}}]$.
Then for any $\Fp$-algebra $R$, we have a natural (in $R$) identification $R[X]^{\perf} \cong  R^{\perf}[X^{p^{-\infty}}]$.
\end{example}

The operation of forming the perfect closure, or the completed perfect closure in case we have a power-multiplicative
norm, does not change the \'etale fundamental group.

\begin{theorem} \label{T:perfect etale}
Let $R$ be an $\Fp$-algebra.
\begin{enumerate}
\item[(a)]
The base change functor
$\FEt(R) \to \FEt(R^{\perf})$ is a tensor equivalence.
\item[(b)]
Suppose that $R$ is a uniform Banach $\Fp$-algebra,
and let $S$ be the completion of $R^{\perf}$. Then
the base change functor $\FEt(R) \to \FEt(S)$ is a tensor equivalence.
\end{enumerate}
\end{theorem}
\begin{proof}
The morphism $\Spec(R^{\perf}) \to \Spec(R)$ is surjective, integral,
and radicial, so by \cite[Corollaire~18.12.11]{ega4-4}, it is
a universal homeomorphism. In particular, it is universally submersive
\cite[Expos\`e~IX, D\'efinition~2.1]{sga1}, so
$\FEt(R) \to \FEt(R^{\perf})$ is fully faithful by
\cite[Expos\'e~IX, Corollaire~3.3]{sga1}.
On the other hand, by Remark~\ref{R:fet direct limit},
$\FEt(R^{\perf})$ is the direct
2-limit of $\FEt(T)$ as $T$ runs over all $R$-subalgebras of $R^{\perf}$
of the form $R[x_1^{1/p^m},\dots,x_n^{1/p^m}]$ for some
nonnegative integer $m$ and some $x_1,\dots,x_n \in R$.
For each such $T$, the morphism $\Spec(T) \to \Spec(R)$ is finite,
radicial, surjective, and of finite presentation,
so $\FEt(R) \to \FEt(T)$ is essentially surjective by
\cite[Expos\'e~IX, Th\'eor\`eme~4.10]{sga1}.
We deduce that $\FEt(R) \to \FEt(R^{\perf})$ is also essentially surjective.
This proves (a).

To prove (b), note that by Lemma~\ref{L:construct affinoid system},
we can write $R$ as the completion of the direct limit of some affinoid system $\{A_i\}_{i \in I}$. Form a new affinoid system $\{B_i\}_{i \in I}$ 
by taking $B_i = A_i^u$ (note that by Corollary~\ref{C:spectral is norm}, $B_i$ is isomorphic to the reduced quotient of $A_i$).
Let $B$ be the direct limit of the second affinoid system; its completion is again $R$.
From the second affinoid system, form a third one by adding $\overline{\varphi}^{-j}(B_i)$ for all nonnegative integers $j$; the completed directed limit becomes $S$. Applying $\FEt$ to the commutative diagram
\[
\xymatrix{
B \ar[r] \ar[d] & R \ar[d] \\
B^{\perf} \ar[r] & S
}
\]
yields tensor equivalences along the horizontal arrows by Proposition~\ref{P:henselian direct limit2}
and along the left vertical arrow by (a). Hence the right vertical arrow
also becomes a tensor equivalence; this proves (b).
\end{proof}

For attribution of the following result, see Remark~\ref{R:locally perfect is perfect}.
\begin{prop} \label{P:locally perfect is perfect}
Let $(A,A^+)$ be a sheafy adic Banach $\Fp$-algebra.
Suppose that there exists
a rational covering $\{(A, A^+) \to (B_i, B_i^+)\}_i$ such that each $B_i$ is perfect uniform. Then $A$ is perfect.
\end{prop}
\begin{proof}
Treating the desired result as a property of rational coverings, we verify the criteria of Proposition~\ref{P:Tate reduction}: (a) follows from 
Proposition~\ref{P:perfect uniform localization}, (b) follows because $(A,A^+)$ is stably uniform, and (c) follows from
Corollary~\ref{C:principal ideal closed}. Explicitly, if $A$ is uniform and $B_1$ and $B_2$ are perfect, then $B_{12}$ is also perfect by Proposition~\ref{P:perfect uniform localization}; consequently, any $f \in A$ has a unique $p$-th root in each of $B_1, B_2, B_{12}$, so these $p$-th roots define an element of $A$ via the exact sequence
in Corollary~\ref{C:principal ideal closed}.
\end{proof}

\begin{remark} \label{R:locally perfect is perfect}
Proposition~\ref{P:locally perfect is perfect} is essentially \cite[Corollary~10]{buzzard-verberkmoes}, except that we assume that $(A,A^+)$ is sheafy rather than stably uniform. This implies \cite[Corollary~10]{buzzard-verberkmoes} because stably uniform adic Banach rings are sheafy (Theorem~\ref{T:uniform rational is sheafy}),
and indeed this is how the proof of \cite[Corollary~10]{buzzard-verberkmoes} proceeds.
On the other hand, one may easily deduce Proposition~\ref{P:locally perfect is perfect} from \cite[Corollary~10]{buzzard-verberkmoes} because an adic Banach algebra which is both sheafy and locally stably uniform is stably uniform
(Remark~\ref{R:sheafy to stably uniform}).

Either of these results may be viewed as a partial resolution of 
\cite[Conjecture~2.16]{scholze1}. As originally stated, that conjecture asserts
that an adic Banach $\Fp$-algebra $(A,A^+)$ which admits a covering by perfect uniform algebras is itself perfect uniform; however, a counterexample against this conjecture is given in \cite[Proposition~13]{buzzard-verberkmoes}. According to \cite{buzzard-verberkmoes}, Scholze has proposed to amend this conjecture by adding the hypothesis that $(A,A^+)$ be uniform, which avoids the counterexample from \cite[Proposition~13]{buzzard-verberkmoes}. If we further assume that $(A,A^+)$ is stably uniform, we get precisely \cite[Corollary~10]{buzzard-verberkmoes}.
The corresponding question for perfectoid algebras is somewhat subtler; see 
Remark~\ref{R:locally perfectoid is perfectoid}.
\end{remark}

\subsection{Strict \texorpdfstring{$p$}{p}-rings}

Perfect $\Fp$-algebras lift naturally to characteristic zero, as follows. (Our derivations
follow \cite[\S 5]{serre-local-fields}; see \cite[\S 1]{illusie} for a discussion more explicitly in terms
of Witt vectors.)
\begin{defn} \label{D:strict p-ring}
A \emph{strict $p$-ring} is a $p$-torsion-free, $p$-adically complete ring $S$ for which $S/(p)$ is perfect.
Given such a ring, for any $p$-adically complete ring $U$ and
any ring homomorphism $\overline{t}: S/(p) \to U/(p)$, $\overline{t}$ lifts uniquely to a multiplicative map
$t: S/(p) \to U$; more precisely, for any $\overline{x} \in S/(p)$ and any $y \in U$ lifting
$\overline{t}(x^{p^{-n}})$, we have $t(\overline{x}) \equiv y^{p^n} \pmod{p^{n+1}}$.
In particular, the projection $S \to S/(p)$ admits a multiplicative section $[\cdot]: S/(p) \to S$, called the
\emph{Teichm\"uller map}; each $x \in S$ can be written uniquely as $\sum_{n=0}^\infty p^n [\overline{x}_n]$
with $\overline{x}_n \in S/(p)$.
\end{defn}

\begin{lemma} \label{L:Teichmuller2}
Let $S$ be a strict $p$-ring, let $U$ be a $p$-adically complete ring, and let $\pi: U \to U/(p)$ be the natural projection.
Let $\overline{t}: S/(p) \to U/(p)$ be a ring homomorphism, and lift $\overline{t}$ to a multiplicative map $t: S/(p) \to U$
as in Definition~\ref{D:strict p-ring}. Then the formula
\begin{equation} \label{eq:Teichmuller homomorphism}
T\left( \sum_{n=0}^\infty p^n [\overline{x_n}] \right) = \sum_{n=0}^\infty p^n t(\overline{x}_n) \qquad (\overline{x}_0, \overline{x}_1, \dots \in S/(p))
\end{equation}
defines a (necessarily unique) homomorphism $T: S\to U$ such that
$T \circ \left[\cdot \right] = t$.
\end{lemma}
\begin{proof}
We check by induction that for each positive integer $n$, $T$ induces an additive map $S/(p^n) \to U/(p^n)$.
This holds for $n=1$ because $\pi \circ t$ is a homomorphism. Suppose the claim holds for some $n \geq 1$.
For $x = [\overline{x}] + p x_1, y = [\overline{y}] + py_1, z = [\overline{z}] + pz_1 \in S$
with $x+y=z$,
\begin{align*}
[\overline{z}] &\equiv ([\overline{x}^{p^{-n}}] + [\overline{y}^{p^{-n}}])^{p^{n}} \pmod{p^{n+1}} \\
t(\overline{z}) &\equiv (t(\overline{x}^{p^{-n}}) + t(\overline{y}^{p^{-n}}))^{p^{n}} \pmod{p^{n+1}}
\end{align*}
as in Definition~\ref{D:strict p-ring}. In particular,
\begin{equation} \label{eq:teichmuller1}
T([\overline{z}]) - T([\overline{x}]) - T([\overline{y}]) \equiv
\sum_{i=1}^{p^{n}-1}
 \binom{p^{n}}{i} T([\overline{x}^{ip^{-n}} \overline{y}^{1-ip^{-n}}]) \pmod{p^{n+1}}.
\end{equation}
On the other hand, since $\frac{1}{p} \binom{p^{n}}{i} \in \ZZ$ for $i=1,\dots,p^{n}-1$,
we may write
\[
z_1 - x_1 - y_1
= \frac{[\overline{x}] + [\overline{y}] - [\overline{z}]}{p} \equiv - \sum_{i=1}^{p^{n}-1}
\frac{1}{p} \binom{p^{n}}{i} [\overline{x}^{ip^{-n}} \overline{y}^{1-ip^{-n}}] \pmod{p^n},
\]
apply $T$, invoke the induction hypothesis on both sides, and multiply by $p$ to obtain
\begin{equation} \label{eq:teichmuller2}
pT(z_1) - pT(x_1) - pT(y_1)
\equiv - \sum_{i=1}^{p^{n}-1}
 \binom{p^{n}}{i} T([\overline{x}^{ip^{-n}} \overline{y}^{1-ip^{-n}}]) \pmod{p^{n+1}}.
\end{equation}
Since $T(x) = T([\overline{x}]) + p T(x_1)$ and so on, we may add \eqref{eq:teichmuller1}
and \eqref{eq:teichmuller2} to deduce that $T(z) - T(x) - T(y) \equiv 0  \pmod{p^{n+1}}$,
completing the induction.
Hence $T$ is additive; multiplicativity of $t$ forces $T$ to also be multiplicative, as desired.
\end{proof}

\begin{remark} \label{R:addition formula}
For $X$ an arbitrary set, the $p$-adic completion $S$ of $\ZZ[X^{p^{-\infty}}]$ is a
strict $p$-ring with $S/(p) \cong \Fp[X^{p^{-\infty}}]$. If we take $X = \{\overline{x}, \overline{y}\}$,
then
\begin{equation} \label{eq:witt formulas}
[\overline{x}] + [\overline{y}] = \sum_{n=0}^\infty p^n [P_n(\overline{x}, \overline{y})]
\end{equation}
for some $P_n(\overline{x}, \overline{y})$ in the ideal $(\overline{x}^{p^{-\infty}}, \overline{y}^{p^{-\infty}})
\subset \Fp[\overline{x}^{p^{-\infty}}, \overline{y}^{p^{-\infty}}]$
and homogeneous of degree 1.
For instance, $P_0(\overline{x}, \overline{y}) = \overline{x} + \overline{y}$ and
$P_1(\overline{x}, \overline{y}) = -\sum_{i=1}^{p-1} p^{-1} \binom{p}{i} \overline{x}^{i/p} \overline{y}^{1 - i/p}$.
By Lemma~\ref{L:Teichmuller2}, \eqref{eq:witt formulas}
is also valid for any strict $p$-ring $S$ and any $\overline{x}, \overline{y} \in S/(p)$.
One can similarly derive formulas for arithmetic in a strict $p$-ring in terms
of Teichm\"uller coordinates; these can also be obtained using Witt vectors (Definition~\ref{D:Witt vectors}).
\end{remark}

\begin{theorem} \label{T:Witt}
The functor $S \rightsquigarrow S/(p)$ from strict $p$-rings to perfect $\Fp$-algebras
is an equivalence of categories.
\end{theorem}
\begin{proof}
Full faithfulness follows from Lemma~\ref{L:Teichmuller2}.
To prove essential surjectivity,
let $R$ be a perfect $\Fp$-algebra,
choose a surjection $\psi: \Fp[X^{p^{-\infty}}] \to R$ for some set $X$, and put $\overline{I} = \ker(\psi)$.
Let $S_0$ be the $p$-adic completion of $\ZZ[X^{p^{-\infty}}]$; this is a strict $p$-ring
with $S_0/(p) \cong \Fp[X^{p^{-\infty}}]$.
Put $I = \{\sum_{n=0}^\infty p^n [\overline{x}_n] \in S_0: \overline{x}_0, \overline{x}_1,\dots \in \overline{I}\}$;
this forms an ideal in $S_0$ by Remark~\ref{R:addition formula}.
Then $S = S_0/I$ is a strict $p$-ring with $S/(p) \cong R$.
\end{proof}

\begin{defn} \label{D:Witt vectors}
For $R$ a perfect $\Fp$-algebra, let $W(R)$ denote the strict $p$-ring with $W(R)/(p) \cong R$;
this object is unique up to unique isomorphism by Theorem~\ref{T:Witt}.
More concretely, we may identify $W(R)$ with the set of infinite sequences over $R$
so that the sequence $(\overline{x}_0, \overline{x}_1, \dots)$ corresponds to the ring element
$\sum_{n=0}^\infty p^n [\overline{x}_n]$. This is a special case of the construction of the
ring of \emph{$p$-typical Witt vectors} associated to a ring $R$, hence the notation.
The construction of $W(R)$ is functorial in $R$, so for instance $\overline{\varphi}$ lifts functorially
to an endomorphism $\varphi$ of $W(R)$. It is common shorthand to write $W_n(R)$ for $W(R)/(p^n)$.
\end{defn}

One of the key roles that strict $p$-rings play in our work is in the classification of local systems
over rings of positive characteristic. The central point is a nonabelian version of
Artin-Schreier-Witt theory, for which we follow
\cite[Proposition~4.1.1]{katz-modular} (see also \cite[Theorem~2.2]{crew-F}).

\begin{lemma} \label{L:DM relative}
Let $R$ be a perfect $\Fp$-algebra, and let $n$ be a positive integer.
Let $M$ be a finite projective $W_n(R)$-module of everywhere positive rank,
equipped with a semilinear $\varphi^a$-action
for some positive integer $a$. Then there exists a
faithfully finite \'etale $R$-algebra $S$
such that $M \otimes_{W_n(R)} W_n(S)$ admits a basis fixed by $\varphi^a$.
More precisely, if $m<n$ is another positive integer
and $M \otimes_{W_n(R)} W_m(R)$ admits a $\varphi^a$-fixed basis, then $S$ can be chosen
so that this basis lifts to a $\varphi^a$-fixed basis of $M \otimes_{W_n(R)} W_n(S)$.
\end{lemma}
\begin{proof}
We treat the case $n=1$ first.
Suppose first that $M$ is free; choose a basis $\be_1,\dots,\be_m$ of $M$
on which $\varphi^a$ acts via the invertible matrix $A$ over $W_1(R) \cong R$.
Let $X$ be the closed subscheme of $\Spec(R[U_{ij}: i,j=1,\dots,m])$
defined by the matrix equation $\overline{\varphi}^a(U) = A^{-1} U$.
The morphism $X \to \Spec(R)$ is finite (evidently) and \'etale (by the Jacobian criterion),
so $X = \Spec(S)$ for some finite \'etale $R$-algebra $S$.
The elements $\bv_1,\dots,\bv_m$ of $M \otimes_{R} S$ defined by
$\bv_j = \sum_i U_{ij} \be_i$ form a basis fixed by $\varphi^a$.
Since the construction is naturally independent of the choice of the original basis,
for general $M$ we can glue to obtain a finite \'etale $R$-algebra $S$ and a fixed
basis of $M \otimes_R S$.

What is left to check is that $S$ has positive rank everywhere as an $R$-module.
This can be checked pointwise on $R$, and also may be checked after faithfully flat descent,
so we may reduce to the case where $R$ is an algebraically closed field.
It is enough to check that the map $U \mapsto U^{-1} \overline{\varphi}^a(U)$ on $\GL_m(R)$
is surjective; this observation is due to Lang
and is proved as follows (following
\cite[\S VI.1, Proposition~4]{serre-class},
\cite[Expos\'e~XXII, Proposition~1.1]{sga7-2}).
For each $A \in \GL_m(R)$, the map $L_A: U \mapsto U^{-1} A \overline{\varphi}^a(U)$
induces a bijective map from the tangent space at $1$, so the image of $L_A$ contains a nonempty
Zariski open subset $V_A$ of $\GL_m(R)$. Since $\GL_m$ is a connected group scheme, the open sets $V_A$ and $V_1$ must intersect in some matrix $B$,
for which\[
B = U_1^{-1} \overline{\varphi}^a(U_1) = U_2^{-1} A \overline{\varphi}^a(U_2)
\]
for some $U_1, U_2 \in \GL_m(R)$. We then have $A = U^{-1} \overline{\varphi}^a(U)$ for $U = U_1 U_2^{-1}$.

The case $n=1$ is now complete; we treat the case $n>1$ by induction on $n$.
We may assume that there exists a basis $\be_1,\dots,\be_m$ of $M$ on which
$\varphi^a$ acts via a matrix $A$ congruent to $1$ modulo $p^{n-1}$.
We may then take $\Spec(S)$ to be the closed subscheme of $\Spec(R[U_{ij}: i,j=1,\dots,m])$
defined by the matrix equation $\overline{\varphi}^a(U) - U + p^{1-n}(A - 1) = 0$:
this subscheme is again finite \'etale (and hence affine) over $R$,
and the elements $\bv_1,\dots,\bv_m$ of $M \otimes_{W(R)} W(S)$
defined by $\bv_j = \be_j + \sum_i p^{n-1} U_{ij} \be_i$
form a basis fixed by $\varphi^a$ modulo $p^n$.
\end{proof}

\begin{prop} \label{P:phi-modules}
For $R$ a perfect $\Fp$-algebra,
for each positive integer $n$,
there is a natural (in $R$ and $n$) tensor equivalence between
lisse sheaves of $\ZZ/p^n \ZZ$-modules on $R$ and finite projective modules over $W_n(R)$ equipped with semilinear
$\varphi$-actions.
\end{prop}
One can also weaken the condition on the modules over $W_n(R)$; see
Proposition~\ref{P:phi-modules are projective}.
\begin{proof}
Let $T$ be a lisse sheaf of $\ZZ/p^n \ZZ$-modules on $R$.
Let $\Spec(R_n)$ be the finite \'etale $R$-scheme parametrizing trivializations of $T$.
In case $T$ is of constant rank $d$, $R_n$ carries an action of the group $G = \GL_d(\ZZ/p^n \ZZ)$, so we may define
$M(T) = W_n(R_n)^G$; by faithfully flat descent
(Theorems~\ref{T:descent modules} and~\ref{T:descent finite locally free}), $M(T)$ is projective of constant rank $d$ over $W_n(R)$.
The construction extends naturally to general $T$.

Let $M$ be a finite projective module over $W_n(R)$ equipped with a semilinear
$\varphi$-action. The assignment
\[
S \mapsto (M \otimes_{W_n(R)} W_n(S))^{\varphi}
\]
defines an \'etale sheaf $T(M)$ on $\Spec(R)$. It is easy to check thanks to Lemma~\ref{L:DM relative} that
the functors $T \rightsquigarrow M(T)$ and $M \rightsquigarrow T(M)$ form an equivalence.
\end{proof}

\begin{remark}
One might like to assert Proposition~\ref{P:phi-modules} with $\GL_d$ replaced by other group schemes.
The main difficulty is that the analogue of Hilbert's Theorem 90 is not always valid; this is
related to the classification of \emph{special groups}
by Serre \cite{serre-chevalley} and Grothendieck \cite{grothendieck-chevalley}.
One tractable special case is that of a unipotent group scheme; see Proposition~\ref{P:Lang torsor}.
\end{remark}

\begin{prop} \label{P:Lang torsor}
Let $d,m,n$ be integers with $d,m\geq 1$ and $n \geq 2$ (we may also take $n=1$ in case $p>2$).
Let $\gothg$ be an algebraic Lie subalgebra of the Lie algebra of $d \times d$ matrices over $\QQ_p$.
Let $\gothg_n$ be the intersection of $\gothg$ with the Lie algebra of
$d \times d$ matrices over $p^n \ZZ_p$.
Let $G_{n,m}$ be the unipotent group scheme defined by the Lie algebra $\gothg_n
\otimes_{\ZZ_p} \ZZ_p/(p^m)$, viewed over $\mathbb{F}_p$ by Greenberg realization
(i.e., identifying $\ZZ_p$ with the Witt vectors of $\Fp$).
For $R$ a perfect $\Fp$-algebra, define the equivalence relation $\sim$ on $G_{n,m}(R)$ by declaring that $g_1 \sim g_2$
if there exists $h \in G_{n,m}(R)$ for which $h^{-1} g_1 \varphi(h) = g_2$.
Then there is a natural (in $G,R,n,m$) bijection
\[
G_{n,m}(R)/\!\sim \, \to H^1_{\et}(R, G_{n,m}(\Fp)).
\]
\end{prop}
\begin{proof}
For $g \in G_{n,m}(R)$, as in Lemma~\ref{L:DM relative}, we may construct a faithfully finite
\'etale $R$-algebra $S$ such that $g = h^{-1} \varphi(h)$ for some $h \in G_{n,m}(S)$.
The choice of $h$ then defines an element of $H^1_{\et}(R,G_{n,m}(\Fp))$
which depends only on $g$ up to equivalence. This gives the claimed map; its injectivity is straightforward.
Surjectivity comes down to the fact that
$H^1_{\et}(R, G_{n,m})$ is trivial as a pointed set, which holds because $G_{n,m}$ is unipotent.
\end{proof}

The operation of direct perfection can be extended to certain $p$-torsion-free rings in order to generate
strict $p$-rings.

\begin{defn} \label{D:perfection}
Let $A$ be a $p$-torsion-free ring with $A/pA$ reduced, equipped with an endomorphism
$\varphi_A: A \to A$ inducing the $p^r$-power Frobenius map on $A/pA$ for some $r>0$
and with an identification $(A/pA)^{\perf} \cong R$. Then
$\varphi$ induces a map $s_\varphi: A \to W(R)$ satisfying $\varphi^r \circ s_\varphi = s_\varphi \circ \varphi_A$; this may be seen by using the uniqueness property of $W(R)$ to identify it with the $p$-adic completion of
the limit of the direct system
\[
A \stackrel{\varphi_A}{\to} A \stackrel{\varphi_A}{\to} \cdots.
\]
(For more details,
we follow \cite[(1.3.16)]{illusie} in suggesting
the reference \cite[VII, \S 4]{lazard}.)
We describe $W(R)$ as the \emph{direct perfection} of $A$
with respect to $\varphi_A$.
\end{defn}

\begin{example}
Put $A = \ZZ[T], R = \Fp[\overline{T}]^{\perf}$, and identify $(A/pA)^{\perf}$ with $R$
by mapping the class of $T$ to $\overline{T}$.
For the endomorphism $\varphi_A: A \to A$ defined by $\varphi_A(T) = T^p$,
the map $s_\varphi$ takes $T$ to $[\overline{T}]$. However, note for instance that
$s_\varphi(T+1) \neq [\overline{T} + 1]$.
\end{example}

We next weaken the hypothesis on the modules over $W_n(R)$ in
Proposition~\ref{P:phi-modules}.
\begin{lemma} \label{L:phi-stable ideals}
Let $R$ be a perfect $\FF_p$-algebra and let $J$ be a finitely generated ideal of $R$ such that $\overline{\varphi}(J) = J$. Then $J$ is generated by an idempotent element of $R$; in particular, $R/J$ is a finite projective $R$-module.
\end{lemma}
The hypothesis that $J$ is finitely generated is crucial; otherwise,
one could choose any $x \in R$ and take $J$ to be the ideal generated by $\overline{\varphi}^n(x)$ for all $n \in \ZZ$.
\begin{proof}
Choose generators $x_1,\dots,x_n$ of $J$ and write
$x_i = \sum_j A_{ij} \overline{\varphi}(x_j)$ for some $A_{ij} \in R$.
Define the $n \times n$ matrix $B$ over $R$ by setting $B_{ii} = x_i^{p-1}$ and $B_{ij} =0$ for $i \neq j$;
then 
\[
0 = \sum_i (1 - AB)_{ij} x_j \qquad (j=1,\dots,n).
\]
For each prime ideal $\gothp$ of $R$, if $x_1,\dots,x_n$ all map to zero in $\kappa_{\gothp}$, then $B$ maps to the zero matrix over $\kappa_{\gothp}$ and so $\det(1-AB)$ maps to 1 in $\kappa_{\gothp}$;
otherwise, $1-AB$ maps to a matrix over $\kappa_{\gothp}$ with nontrivial kernel and so $\det(1-AB)$ maps to 0 in $\kappa_{\gothp}$. Since $R$ is reduced, this implies that $\det(1-AB)$ is an idempotent in $R$.

We may thus reduce to the cases where $\det(1-AB) \in \{0,1\}$. If
$\det(1-AB) = 1$, then by the previous paragraph $x_1,\dots,x_n$ map to zero in every $\kappa_{\gothp}$ and so $J = 0$. If $\det(1-AB) = 0$, then by the previous paragraph for every prime ideal $\gothp$ of $R$, at least one of $x_1,\dots,x_n$ has nonzero image in $\kappa_{\gothp}$ and so $J = R$.
\end{proof}

\begin{prop} \label{P:phi-modules are projective}
Let $R$ be a perfect $\FF_p$-algebra, let $n$ be a positive integer, and let $M$ be a finitely presented $W_n(R)$-module which is flat over $\ZZ/p^n \ZZ$ and admits a semilinear $\varphi$-action. Then $M$ is a finite projective $W_n(R)$-module.
\end{prop}
\begin{proof}
Since $M$ is flat over $\ZZ/p^n \ZZ$, we may reduce to the case $n=1$.
Since $M$ is finitely presented, we may define the Fitting ideals $\Fitt_I(M)$ as in
Definition~\ref{D:projective}.
Since $M$ admits a semilinear $\overline{\varphi}$-action,
we have $\overline{\varphi}(\Fitt_i(M)) = \Fitt_i(M)$ for all $i \geq 0$. By Lemma~\ref{L:phi-stable ideals}, each ideal $\Fitt_i(M)$ is generated by an idempotent element of $R$; we may thus reduce to the case where for some $n \geq 0$ we have $\Fitt_i(M) = 0$ for $i<n$ and
$\Fitt_n(M) = R$. In this case, as in Definition~\ref{D:projective}, $M$ is a finite projective $R$-module of constant rank $n$, as claimed.
\end{proof}

\subsection{Norms on strict \texorpdfstring{$p$}{p}-rings}
\label{subsec:Witt ring}

We now take a more metric look at strict $p$-rings.
\begin{hypothesis}
Throughout \S\ref{subsec:Witt ring},
let $R$ be a perfect
$\Fp$-algebra.
\end{hypothesis}

We introduce some operations
relating the spectra of $R$ and $W(R)$. For variants that do not require the norm on $R$
to be trivial, see Proposition~\ref{P:mu multiplicative}.
\begin{defn}
For $\alpha$ a submultiplicative (resp.\ power-multiplicative, multiplicative)
seminorm on $R$ bounded by the trivial norm,
\begin{equation} \label{eq:lambda}
\lambda(\alpha) \left( \sum_{i=0}^\infty p^i [\overline{x}_i] \right)
= \sup_i \{p^{-i} \alpha(\overline{x}_i)\}
\end{equation}
is a submultiplicative (resp.\ power-multiplicative, multiplicative)
seminorm on $W(R)$ bounded by the $p$-adic norm
\cite[Lemma~4.1]{kedlaya-witt}. For $\beta$ a submultiplicative
(resp.\ power-multiplicative, multiplicative) seminorm
on $W(R)$ bounded by the $p$-adic norm,
\begin{equation} \label{eq:mu}
\mu(\beta)(\overline{x}) = \beta([\overline{x}])
\end{equation}
is a submultiplicative (resp.\ power-multiplicative, multiplicative)
seminorm on $R$ bounded by the trivial norm
\cite[Lemma~4.2]{kedlaya-witt}.
\end{defn}

\begin{lemma} \label{L:lambda mu}
Equip $R$ with the trivial norm and $W(R)$ with the $p$-adic norm.
Then the functions $\lambda: \calM(R) \to \calM(W(R))$
and $\mu: \calM(W(R)) \to \calM(R)$ are continuous,
and satisfy $(\mu \circ \lambda)(\alpha) = \alpha$ and
$(\lambda \circ \mu)(\beta) \geq \beta$.
\end{lemma}
\begin{proof}
See \cite[Theorem~4.5]{kedlaya-witt}.
\end{proof}

\begin{defn} \label{D:primitive}
Suppose that $R$ is complete with respect to a power-multiplicative norm $\alpha$ bounded above by the trivial norm.
An element $z = \sum_{i=0}^\infty p^i [\overline{z}_i] \in W(R)$ is
\emph{primitive of degree $1$} if
\begin{equation} \label{eq:primitive}
\alpha(\overline{x} \overline{z}_0) = p^{-1} \alpha(\overline{x}) \qquad (x \in R)
\end{equation}
and $\overline{z}_1 \in R^\times$ (or equivalently $z - [\overline{z}_0] \in p W(R)^\times$). 
This implies that the principal ideal $(z)$ in $W(R)$ is closed
(see \cite[Theorem~5.11]{kedlaya-witt}.
Note that if \eqref{eq:primitive} holds, then $S = R[\overline{z}_0^{-1}]$ is a Banach algebra over the analytic field $\FF_p((\overline{z}_0))$. Conversely, if $R$ is contained in a Banach algebra over $\FF_p((\overline{z}_0))$, then \eqref{eq:primitive} holds if and only if 
$\alpha(\overline{z}_0) = p^{-1}$, as then
\[
\alpha(\overline{x} \overline{z}_0) \leq \alpha(\overline{x}) \alpha(\overline{z}_0)
= \alpha(\overline{x}) \alpha(\overline{z}_0^{-1})^{-1}
\leq \alpha(\overline{x} \overline{z}_0).
\]

The terminology is modeled on that of \cite{fargues-fontaine},
in which a result similar to our Theorem~\ref{T:quotient norm} can be found;
the wording is meant to evoke an analogy with the theory of Weierstrass preparation for
analytic power series. Note however that when $R = L$ is an analytic field,
our definition is more restrictive than that used in \cite{fargues-fontaine},
in which the condition $\alpha(\overline{z}_0) = p^{-1}$ is relaxed to
$\alpha(\overline{z}_0)< 1$.
\end{defn}

A key example of the previous definition is the following.
\begin{example} \label{exa:stable residue}
Suppose $R$ is a uniform Banach ring with spectral norm $\alpha$.
Choose $\overline{\pi} \in R^\times$ with
$\alpha(\overline{\pi}) = p^{-p/(p-1)}$
and $\alpha(\overline{\pi}^{-1}) = p^{p/(p-1)}$, and put
\[
z = \sum_{i=0}^{p-1} [\overline{\pi} + 1]^{i/p} = \sum_{i=0}^\infty p^i [\overline{z}_i].
\]
Then $\overline{z}_0 = \overline{\pi}^{(p-1)/p}$, so
$\alpha(\overline{z}_0) = p^{-1}$ and $\alpha(\overline{z}_0^{-1}) = p$.
We may check that $\overline{z_1} \in \gotho_R^\times$ by noting that
under the map $W(\Fp[\overline{\pi}]^{\perf}) \to W(\Fp)$
induced by reduction modulo $\overline{\pi}$,
the image of $\sum_{i=0}^{p-1} [\overline{\pi}+1]^{i/p}$ is
$\sum_{i=0}^{p-1} 1 = p$. Hence $z \in W(\gotho_R)$ is primitive of degree 1.
\end{example}

\begin{lemma} \label{L:stable residue}
Suppose that $R$ is complete with respect to a power-multiplicative norm $\alpha$
and that $z \in W(\gotho_R)$ is primitive of degree $1$.
Then any $x \in W(\gotho_R)$ is congruent modulo $z$ to some $y =
\sum_{i=0}^\infty p^i [\overline{y}_i] \in W(\gotho_R)$ with
$\alpha(\overline{y}_0) \geq \alpha(\overline{y}_i)$ for all $i > 0$.
\end{lemma}
\begin{proof}
See \cite[Lemma~5.5]{kedlaya-witt}.
\end{proof}

\begin{theorem} \label{T:quotient norm}
Take $R,z$ as in Lemma~\ref{L:stable residue}.
\begin{enumerate}
\item[(a)]
For each submultiplicative (resp.\ power-multiplicative, multiplicative)
seminorm $\gamma$ on $\gotho_R$ bounded by the trivial norm, the quotient seminorm $\sigma(\gamma)$
on $W(\gotho_R)/(z)$ induced by $\lambda(\gamma)$
is submultiplicative (resp.\ power-multiplicative, multiplicative)
and satisfies $\mu(\sigma(\gamma)) = \gamma$.
\item[(b)]
Equip $W(\gotho_R)$ with the power-multiplicative norm $\lambda(\alpha)$. Then the
map $\sigma: \calM(\gotho_R) \to \calM(W(\gotho_R))$ indicated by (a)
is a continuous section of $\mu$,
which induces a homeomorphism of $\calM(\gotho_R)$ with
$\calM(W(\gotho_R)/(z))$. Under this homeomorphism, a subspace of
$\calM(\gotho_R)$ is rational
if and only if the corresponding subspace of
$\calM(W(\gotho_R)/(z))$ is rational.
\item[(c)]
The homeomorphism of (b) induces a homeomorphism of $\calM(R)$
with $\calM(W(\gotho_R)[[\overline{z}]^{-1}]/(z))$
under which rational subspaces
again correspond.
\end{enumerate}
\end{theorem}
For more on the relationship between $\calM(R)$ and $\calM(W(\gotho_R)[[\overline{z}]^{-1}]/(z))$, see
\S\ref{subsec:perfectoid2} and
\S\ref{subsec:geometric}.

\begin{proof}
For (a), see \cite[Theorem~5.11(a)]{kedlaya-witt}
(which is itself an easy corollary of Lemma~\ref{L:stable residue}).
For (b), see \cite[Corollary~7.2]{kedlaya-witt}.
Note that for these results, $z$ need not be primitive of degree 1; it is enough to assume that
$\alpha(\overline{z}_0) \leq p^{-1}$ and $\overline{z}_1 \in \gotho_R^\times$.

By assuming that $z$ is primitive of degree 1, however, we ensure that
the quotient norm $\beta$
on $\calM(W(\gotho_R)/(z))$ has the property that $\beta(px) = p^{-1} \beta(x)$,
so that we may extend $\beta$ after inverting $p$.
We may then identify
\begin{align*}
\calM(R) &= \{ \gamma \in \calM(\gotho_R): \gamma(\overline{z}_0) \geq p^{-1} \} \\
\calM(W(\gotho_R)[[\overline{z}]^{-1}]/(z)) &= \{ \gamma \in \calM(W(\gotho_R)/(z)): \gamma([\overline{z}_0]) \geq p^{-1} \}.
\end{align*}
Since these are rational subspaces, we may deduce (c).
\end{proof}

\begin{example} \label{exa:lift analytic field}
Let $L$ be a perfect analytic field of characteristic $p$, let $\alpha$ be the norm on $L$,
and choose $z \in W(\gotho_L)$ which is primitive of degree 1.
By Theorem~\ref{T:quotient norm}(a), the quotient norm on $W(\gotho_L)/(z)$
induced by $\lambda(\alpha)$ is multiplicative.
Moreover, by Lemma~\ref{L:stable residue}, every nonzero element of $W(\gotho_L)/(z)$
can be lifted to an element of $W(\gotho_L)$ which becomes invertible in
$W(\gotho_L)[[\overline{z}]^{-1}]$. It follows that $W(\gotho_L)/(z)$ is the valuation
subring of an analytic field $F = W(\gotho_L)[[\overline{z}]^{-1}]/(z)$,
whose residue field is the same as that of $L$.
(In terms of the rings $\tilde{\calR}^{\inte,r}_L$ to be introduced in
Definition~\ref{D:extended Robba} below, we can also realize $F$ as $\tilde{\calR}^{\inte,r}_L/(z)$
for any $r\geq 1$. See Lemma~\ref{L:stable residue2}.)
Two key examples are the following.
\begin{itemize}
\item
For $L$ the completed perfection of $\Fp((\overline{\pi}))$
and $z$ as in Example~\ref{exa:stable residue}, $F$ is the completion of $\Qp(\mu_{p^\infty})$ for the
$p$-adic norm.
\item
For $L$ the completed perfection of $\Fp((\overline{\pi}))$
with $\alpha(\overline{\pi}) = p^{-1}$
and $z = [\overline{\pi}] - p$,
$F$ is the completion of $\Qp(p^{p^{-\infty}})$ for the
$p$-adic norm.
\end{itemize}
Note that
\[
\gotho_L/(\overline{z}) \cong W(\gotho_L)/(p,[\overline{z}])
= W(\gotho_L)/(p, z) = \gotho_F/(p).
\]
This implies that $\overline{\varphi}$ is surjective on $\gotho_F/(p)$ and that $\gotho_F$ is not discretely valued.
These conditions turn out to characterize the fields $F$ which arise in this manner; see
Lemma~\ref{L:deeply ramified}.
\end{example}

The following refinement of \cite[Lemma~5.16]{kedlaya-witt} is useful for some calculations.

\begin{lemma} \label{L:perfectoid calculation}
Take $R,z$ as in Lemma~\ref{L:stable residue}.
Then for any $\epsilon > 0$ and any nonnegative integer $m$,
every $x \in W(\gotho_R)[[\overline{z}]^{-1}]$ is congruent
modulo $z$ to some $y = \sum_{n=0}^\infty p^n [\overline{y}_n] \in W(\gotho_R)[[\overline{z}]^{-1}]$
such that for each $\alpha \in \calM(R)$,
\begin{align}
\label{eq:strong bound on higher terms1}
\alpha(\overline{y}_1) &\leq \max\{p^{-p^{-1}-\cdots-p^{-m}} \alpha(\overline{y}_0),\epsilon\} \\
\label{eq:strong bound on higher terms2}
\alpha(\overline{y}_n) &\leq \max\{\alpha(\overline{y}_0),\epsilon\}
\quad (n > 1).
\end{align}
\end{lemma}
\begin{proof}
Define the sequence $x = x_0, x_1, \dots$ as in the proof of \cite[Lemma~5.5]{kedlaya-witt}. That is,
let $w$ be the inverse of $p^{-1}(z - [\overline{z}])$ in $W(\gotho_R)$, then write
$x_i = \sum_{j=0}^{\infty} p^j [\overline{x}_{ij}]$ with $\overline{x}_{ij} \in R$
and put
\[
x_{i+1} = x_i - p^{-1} w (x_i - [\overline{x}_{i0}]) z
= [\overline{x}_{i0}] - p^{-1} w (x_i - [\overline{x}_{i0}]) [\overline{z}].
\]
The proof of \cite[Lemma~5.16]{kedlaya-witt} shows that there exists $i_0$ such that for each
$\alpha \in \calM(R)$,
\[
\alpha(\overline{x}_{ij}) \leq \max\{\alpha(\overline{x}_{i0}),\epsilon\}
\quad (i \geq i_0, j > 0).
\]
If we take $y = x_{i_0 + k}$ for some nonnegative integer $k$, then
\eqref{eq:strong bound on higher terms2} is satisfied.
If $\alpha(\overline{x}_{i_0 0}) \leq \epsilon$, then
\eqref{eq:strong bound on higher terms1} is also satisfied.

Suppose instead that $\alpha(\overline{x}_{i_0 0}) > \epsilon$; in this case, it will complete
the proof to show that \eqref{eq:strong bound on higher terms1} is satisfied whenever $k \geq m$.
It will suffice to check that for each nonnegative integer $k$,
\begin{equation} \label{eq:strong bound on higher terms3}
\alpha(\overline{x}_{(i_0+k)1}) \leq \max\{p^{-p^{-1}-\cdots-p^{-k}} |\overline{x}_{(i_0+k)0}|,\epsilon\}.
\end{equation}
We have this for $k=0$, so we may proceed by induction on $k$. Given
\eqref{eq:strong bound on higher terms3} for some $k$, write
\[
x_{i_0+k+1} \equiv [\overline{x}_{(i_0+k)0}] - w [\overline{x}_{(i_0+k)1} \overline{z}] + pw[\overline{x}_{(i_0+k)2} \overline{z}]
\pmod{p^3}
\]
and then deduce that
\begin{align*}
\overline{x}_{(i_0+k+1)0} &= \overline{w} \overline{x}_{(i_0+k)0} - \overline{w} \overline{x}_{(i_0+k)1} \overline{z} \\
\overline{x}_{(i_0+k+1)1} &= \overline{w} \overline{x}_{(i_0+k)2} \overline{z} + P((\overline{w} \overline{x}_{(i_0+k)0})^{1/p}, (\overline{w} \overline{x}_{(i_0+k)1} \overline{z})^{1/p} )
\end{align*}
for $P(x,y) = p^{-1} (x^p - y^p - (x-y)^p) \in \ZZ[x,y]$. From this we deduce
\begin{align*}
\alpha(\overline{x}_{(i_0+k+1)0}) &=\alpha( \overline{w} \overline{x}_{(i_0+k)0}) \\
\alpha(\overline{x}_{(i_0+k+1)1}) &\leq \max\{ \alpha(\overline{w} \overline{x}_{(i_0+k)2} \overline{z}),
\alpha(\overline{w} \overline{x}_{(i_0+k)0})^{(p-1)/p} \alpha(\overline{w} \overline{x}_{(i_0+k)1} \overline{z})^{1/p}\}.
\end{align*}
Since $\alpha(\overline{z}) = p^{-1}$, this yields the analogue of \eqref{eq:strong bound on higher terms3}
with $k$ replaced by $k+1$.
\end{proof}

\subsection{Inverse perfection}
\label{subsec:perfection}

We have already introduced one method for passing from an $\Fp$-algebra to a perfect $\Fp$-algebra,
that of \emph{direct perfection}. We now consider the dual operation of \emph{inverse perfection},
which has the advantage of capturing useful information from characteristic $0$.

\begin{defn} \label{D:inverse perfection}
For any ring $A$,
define the \emph{inverse perfection} $A^{\frep}$ of $A$ as
the inverse limit of the system
\[
\cdots \stackrel{\overline{\varphi}}{\to} A/pA
\stackrel{\overline{\varphi}}{\to} A/pA.
\]
This evidently gives a perfect $\Fp$-algebra.
There is a natural projection $A^{\frep} \to A/pA$ by projection onto
the last factor; this is surjective as long as
$\overline{\varphi}: A/pA \to A/pA$ is surjective.
\end{defn}

\begin{lemma} \label{L:same frep}
Let $A$ be a ring.
For any ideal $I$ of $A$ satisfying $I^m \subseteq (p) \subseteq I$ for some
positive integer $m$, the natural map $A^{\frep} \to (A/I)^{\frep}$
is an isomorphism.
\end{lemma}
\begin{proof}
We may assume $m = p^k$ for some positive integer $k$.
Let $y = (\dots,y_1,y_0)$ be an element of $A^{\frep}$ whose image in
$(A/I)^{\frep}$ is zero. For each nonnegative integer $n$, we then have
$y_{n+k} \equiv 0 \pmod{I}$, and so
\[
y_n \equiv y_{n+k}^{p^k} \equiv 0 \pmod{(p) + I^{p^k}}.
\]
Hence $y_n \equiv 0 \pmod{p}$, and so $y = 0$ in $A^{\frep}$.

Given $z = (\dots, z_1, z_0) \in (A/I)^{\frep}$, choose any lifts
$\tilde{z}_n \in A$ of $z_n$. Put $y_n = \tilde{z}_{n+k}^{p^k}$; then
the congruence $\tilde{z}_{n+k+1}^p \equiv \tilde{z}_{n+k} \pmod{I}$
implies $y_{n+1}^p \equiv y_n \pmod{(p) + I^{p^k}}$. Hence
$y = (\dots,y_1,y_0)$ forms an element of $A^{\frep}$ lifting $z$.
\end{proof}

\begin{defn} \label{D:frep hom}
Let $A$ be a ring,
and let $\widehat{A}$ denote the $p$-adic completion of $A$.
{}From the projection $A^{\frep} \to A/pA$, we obtain first a multiplicative map
$A^{\frep} \to \widehat{A}$ and then by Lemma~\ref{L:Teichmuller2} a homomorphism $\theta: W(A^{\frep}) \to \widehat{A}$.
Note that $\theta$ is surjective if and only if $\overline{\varphi}: A/pA \to A/pA$ is surjective.
\end{defn}

\begin{remark} \label{R:project norm}
Let $A$ be a $p$-adically separated ring,
let $\beta$ be any submultiplicative (resp.\ power-multiplicative,
multiplicative) seminorm on $A$ bounded by the
$p$-adic norm, and extend $\beta$ to $\widehat{A}$ by continuity.
Then $\alpha = \mu(\theta^*(\beta))$ is a submultiplicative
(resp.\ power-multiplicative, multiplicative)
seminorm on $A^{\frep}$ bounded by the trivial norm.
In particular, the map $\theta$ is bounded for the seminorm $\lambda(\alpha)$
on $W(A^{\frep})$ and the seminorm $\beta$ on $\widehat{A}$.
\end{remark}

\begin{lemma} \label{L:frep norm}
In Remark~\ref{R:project norm},
suppose that $\beta$ is power-multiplicative (resp.\ multiplicative) norm and that $A$ is complete under $\beta$.
Then $\alpha = \mu(\theta^*(\beta))$ is a power-multiplicative (resp.\ multiplicative) norm under which
$A^{\frep}$ is complete.
\end{lemma}
\begin{proof}
For $x = (\dots, \overline{x}_1, \overline{x}_0) \in A^{\frep}$
and any lifts $x_i \in A$ of $\overline{x}_i$,
we have
\[
\max\{\alpha(x), \beta(p)^{i}\} =
\max\{\beta(x_i)^{p^i}, \beta(p)^{i}\}
\]
for all $i$.
In particular, if $\alpha(x) \neq 0$, then
$x_i \neq 0$ for all sufficiently large $i$, so $x \neq 0$.
Hence $\alpha$ is a norm.

Let $y_0, y_1, \dots$ be a sequence in
$A^{\frep}$ which is Cauchy with respect to $\alpha$.
For each nonnegative integer $i$, the sequence $y_0^{p^{-i}}, y_1^{p^{-i}}, \dots$
is also Cauchy with respect to $\alpha$, so $[y_0^{p^{-i}}], [y_1^{p^{-i}}], \dots$
is Cauchy with respect to $\lambda(\alpha)$. The images of $[y_0^{p^{-i}}], [y_1^{p^{-i}}], \dots$
in $A$ then form a Cauchy sequence with respect to $\beta$, which by hypothesis has a limit
$\tilde{z}_i \in A$. Let $z_i \in A/pA$ be the image of $\tilde{z}_i$; then
$(\dots, z_1, z_0)$ forms an element of $A^{\frep}$ which is the limit of the $y_n$.
Hence $A^{\frep}$ is complete.
\end{proof}

\begin{remark} \label{R:Witt to frep}
For $R$ a perfect $\Fp$-algebra,
the natural map $W(R)^{\frep} \cong R$ is an isomorphism, as then is the map
$\theta: W(W(R)^{\frep}) \to W(R)$.
\end{remark}

\begin{lemma} \label{L:dense image}
Let $A_1$ be a $p$-adically separated ring written as a union $\cup_{i \in I} A_{1,i}$
such that for each $i \in I$, $\overline{\varphi}$ is surjective on $A_{1,i}/pA_{1,i}$.
Let $A_2$ be a $p$-adically separated ring equipped with a norm $\beta_2$
bounded by the $p$-adic norm. Let $\psi: A_1 \to A_2$ be a homomorphism.
Put $\alpha_2 = \mu(\theta^*(\beta_2))$.
\begin{enumerate}
\item[(a)]
Suppose that $\psi$ has dense image.
Then the induced map $\psi^{\frep}: \cup_{i \in I} A_{1,i}^{\frep} \to A_2^{\frep}$ has dense image.
\item[(b)]
Suppose that the image of $\psi$ has dense intersection with $\gothm_{A_2}$.
Then the image of $\psi^{\frep}$ has dense intersection with $\gothm_{A_2^{\frep}}$.
\end{enumerate}
\end{lemma}
\begin{proof}
Given $w = (\dots, w_1, w_0) \in A_2^{\frep}$,
if we choose a nonnegative integer $n$,
we can choose $i \in I$ and $\tilde{w}_n \in A_{1,i}$ so that $\alpha_2(\psi(\tilde{w}_n) - w_n) \leq p^{-1}$.
Since $\overline{\varphi}$ is surjective on
$A_{1,i}/pA_{1,i}$, we can find $x = (\dots, x_1, x_0) \in A_{1,i}^{\frep}$ with $x_n$ equal to the image of
$\tilde{w}_n$  in $A_{1,i}/pA_{1,i}$. Let $y = (\dots, y_1, y_0) \in A_2^{\frep}$ be the image of $x$
under $\psi^{\frep}$;
then $\alpha_2(y_n - w_n) \leq p^{-1}$, so $\alpha_2(y - w) \leq p^{-p^{n}}$.
Since this holds for any $n$ (for some $i,y$ depending on $n$),
it follows that $\cup_{i \in I} A_{1,i}^{\frep}$ has dense image  in $A_2^{\frep}$. This proves (a);
the proof of (b) is similar.
\end{proof}

\begin{lemma} \label{L:ker theta}
Let $A$ be a $p$-adically separated $p$-torsion-free ring complete under a power-multiplicative norm $\beta$
bounded by the $p$-adic norm, and put $\alpha = \mu(\theta^*(\beta))$.
Suppose that there exists $z \in W(A^{\frep})$ primitive of degree $1$ with $\theta(z) = 0$.
Extend $\theta$ to a map $W(A^{\frep})[[\overline{z}]^{-1}] \to A[\theta([\overline{z}])^{-1}]$.
\begin{enumerate}
\item[(a)]
The ideal $\ker(\theta) \subset W(A^{\frep})[[\overline{z}]^{-1}]$ is generated by $z$.
\item[(b)]
The extended map $\theta$ is optimal.
\item[(c)]
The map $\overline{\varphi}: A/(p) \to A/(p)$
is surjective if and only if $\theta$ has dense image in $A[\theta([\overline{z}])^{-1}]$.
\end{enumerate}
\end{lemma}
\begin{proof}
Given $x\in W(A^{\frep})$ not divisible by $z$,
choose $y = \sum_{i=0}^\infty p^i [\overline{y}_i]$
as in Lemma~\ref{L:stable residue}.
Then $\theta(x) = \theta(y) = \theta([\overline{y}_0]) +
\theta(y - [\overline{y}_0])$ and
\[
\beta(\theta([\overline{y}_0])) = \alpha(\overline{y}_0)
> \lambda(\alpha)(y - [\overline{y}_0]) \geq \beta(\theta(y -
[\overline{y}_0])).
\]
Consequently, $\beta(\theta(x)) = \alpha(\overline{y}_0) >
0$. This implies (a) and (b). To check (c), note that strictness of $\theta$ (from (b)) implies that $\theta$ is surjective if and only if it
has dense image.
\end{proof}

\begin{remark} \label{R:reduce surjectivity}
If $\overline{\varphi}: A/pA \to A/pA$ is surjective, then so is
$\overline{\varphi}: A/I \to A/I$ for any ideal $I$ for which $I^m \subseteq (p) \subseteq I$ for
some positive integer $m$. The converse is not true: e.g., take $A = \ZZ_p[\sqrt{p}]$ and $I = (\sqrt{p})$.

One correct partial converse is that if there exist $x,y \in A$ such that $\overline{\varphi}: A/(x,p) \to A/(x,p)$
is surjective, $x^m \in (p)$ for some positive integer $m$, and $y^p \equiv x \pmod{(x^2,p)}$, then
$\overline{\varphi}: A/pA \to A/pA$ is surjective. To see this, we prove by induction that $\overline{\varphi}: A/(x^i,p) \to A/(x^i,p)$ is surjective for $i=1,\dots,m$, the case $i=1$ being given and the case $i=m$ being the desired result.
Given the claim for some $i$, note first that $(y^p,p) = (x,p)$ and that $y^{pi} \equiv x^i \pmod{(x^{i+1}, p)}$.
for any $z \in A$, we can find $w_0, z_1 \in A$ with $z - w_0^p - x^iz_1 \in pA$.
We can then find $w_1 \in A$ with $z_1 - w_1^p \in (x^i,p)$; then $w = w_0 + y^i w_1$ satisfies
$w^p \equiv w_0^p + y^{pi} w_1^p \equiv w_0^p + y^{pi} z_1 \equiv w_0^p + x^i z_1 \equiv z \pmod{(x^{i+1},p)}$.
\end{remark}

\begin{remark} \label{R:full power inverse limit}
If $A$ is $p$-adically complete, then $\theta$ (surjective or not) induces an isomorphism of multiplicative monoids
\[
A^{\frep} \cong \varprojlim_{x \mapsto x^p} A
\]
whose inverse is reduction modulo $p$. This can be used to reformulate the perfectoid correspondence; see Proposition~\ref{P:power inverse limit as set}.
\end{remark}

\subsection{The perfectoid correspondence for analytic fields}

In order to bring nonabelian Artin-Schreier-Witt theory to bear upon $p$-adic Hodge theory,
one needs a link between \'etale covers of spaces of characteristic $0$ and characteristic $p$.
We first make this link at the level of analytic fields;
this extends the \emph{field of norms} correspondence introduced by Fontaine and Wintenberger \cite{fontaine-wintenberger},
upon which usual $p$-adic Hodge theory is based.  Similar results have been obtained by Scholze
\cite{scholze1} using a slightly different method; see Remark~\ref{R:gabber-ramero}.
(See also \cite{kedlaya-new-phigamma} for a self-contained presentation of the correspondence
following the approach taken here.)
See \S\ref{subsec:perfectoid2} for an extension to more general Banach algebras.

\begin{defn} \label{D:perfectoid field}
An analytic field $F$ is  \emph{perfectoid} if $F$ is of characteristic $0$, $\kappa_F$ is of characteristic $p$,
$F$ is not discretely valued,
and $\overline{\varphi}$ is surjective on $\gotho_F/(p)$. For example, any field $F$ appearing
in Example~\ref{exa:lift analytic field} is perfectoid; the converse is also true by
Lemma~\ref{L:deeply ramified} below.
\end{defn}

\begin{lemma} \label{L:deeply ramified}
Let $F$ be a perfectoid analytic field with norm $\beta$. Put $R = (\gotho_F/(p))^{\frep}$,
let $\theta: W(R) \to \gotho_F$ be the surjective homomorphism from Definition~\ref{D:frep hom},
and define the multiplicative norm $\alpha = \mu(\theta^*(\beta))$ on $R$ as in Remark~\ref{R:project norm}.
\begin{enumerate}
\item[(a)]
The ring $K = \Frac(R)$ is an analytic field under $\alpha$ which is perfect of characteristic $p$,
and $R = \gotho_K$.
\item[(b)]
We have $\beta(F^\times) = \alpha(K^\times)$.
\item[(c)]
For any $\overline{z} \in K$ with $\alpha(\overline{z}) = p^{-1}$ (which exists by (b)),
there is a natural (in $F$) isomorphism $\gotho_F/(p) \cong \gotho_K/(\overline{z})$.
In particular, we obtain a natural isomorphism $\kappa_F \cong \kappa_K$.
\item[(d)]
There exists $z \in W(\gotho_K)$ in $\ker(\theta)$ which is primitive of degree $1$. Consequently (by
Lemma~\ref{L:ker theta}), the kernel of $\theta: W(\gotho_K)[[\overline{z}]^{-1}] \to F$
is generated by $z$.
\end{enumerate}
\end{lemma}
\begin{proof}
Note that $R$ is already complete under $\alpha$ by Lemma~\ref{L:frep norm}. Hence to prove (a),
it suffices to check that for any nonzero $x,y \in R$ with $\alpha(x) \leq \alpha(y)$,
$x$ is divisible by $y$ in $R$. Write $x = (\dots, \overline{x}_1, \overline{x}_0),
y = (\dots, \overline{y}_1, \overline{y}_0)$ and lift $\overline{x}_n, \overline{y}_n$
to $x_n, y_n \in \gotho_F$. Choose $n_0 \geq 0$ so that $\alpha(x), \alpha(y) > p^{-p^{n_0}}$.
For $n \geq n_0$,
as in the proof of Lemma~\ref{L:frep norm} we have $\beta(x_n) = \alpha(x)^{p^{-n}},
\beta(y_n) = \alpha(y)^{p^{-n}}$, so $\beta(x_n) \leq \beta(y_n)$. Since $\gotho_F$
is a valuation ring, $z_n = x_n/y_n$ belongs to $\gotho_F$.
Since $\alpha(x_{n+1}^p - x_n),\alpha(y_{n+1}^p - y_n) \leq p^{-1}$, we have
$\alpha(z_{n+1}^p - z_n) \leq p^{-1} / \alpha(y_n)$. This last quantity is bounded away from 1
for $n \geq n_0$, so by Lemma~\ref{L:same frep}, the $z_n$ define an element $z \in R$ for which
$x = yz$.

To establish (b), note that the group $\alpha(K^\times)$ is $p$-divisible and
that $\alpha(K^\times) \cap (p^{-1}, 1) = \beta(F^\times) \cap (p^{-1}, 1)$ by (a).
Since $F$ is not discretely valued, we can choose $r \in \beta(F^\times) \cap (1, p^{1/p})$. For any such
$r$, we have $r^{-1}, p^{-1} r^p \in \beta(F^\times) \cap (p^{-1}, 1) \subseteq \alpha(K^\times)$,
so $p^{-1} \in \alpha(K^\times)$.

To establish (c), note that from the definition of the inverse perfection,
we obtain a homomorphism $\gotho_K \to \gotho_F/(p)$. By comparing norms, we see that the kernel of this
map is generated by $\overline{z}$.

To establish (d), keep notation as in (c). Note that $\theta([\overline{z}])$ is divisible by $p$ in $\gotho_F$
and that $\theta: W(\gotho_K) \to \gotho_F$ is surjective. We can thus find $z_1 \in W(\gotho_K)^\times$
with $\theta(z_1) = \theta([\overline{z}])/p$; we then take $z = [\overline{z}] - pz_1$.
\end{proof}
\begin{theorem}[Perfectoid correspondence] \label{T:perfectoid field}
The constructions
\[
F \rightsquigarrow (\Frac((\gotho_F/(p))^{\frep}), \ker(\theta)),
\qquad
(L,I) \rightsquigarrow \Frac(W(\gotho_L)/I)
\]
define a equivalence of categories between perfectoid analytic fields $F$ and pairs
$(L,I)$ in which $L$ is a perfect analytic field of characteristic $p$ and $I$ is a principal ideal of $W(\gotho_L)$
admitting a generator which is primitive of degree $1$.
\end{theorem}
\begin{proof}
This follows immediately from Example~\ref{exa:lift analytic field} and Lemma~\ref{L:deeply ramified}.
\end{proof}

We next study the compatibility of this correspondence with finite extensions of fields. Moving from characteristic
$p$ to characteristic $0$ turns out to be straightforward.
\begin{lemma} \label{L:transfer degree}
Fix $F$ and $(L,I)$ corresponding as in Theorem~\ref{T:perfectoid field}.
Then for any finite extension $M$ of $L$, the pair $(M, I W(\gotho_M))$ corresponds via
Theorem~\ref{T:perfectoid field} to a finite extension $E$ of $F$ with $[E:F] = [M:L]$.
\end{lemma}
\begin{proof}
Suppose first that $M$ is Galois over $L$, and put $G = \Gal(M/L)$. Since $I$ is a principal ideal,
averaging over $G$ induces a projection
\[
E = \frac{W(\gotho_M)[p^{-1}]}{I W(\gotho_M)[p^{-1}]} \to
\frac{W(\gotho_L)[p^{-1}]}{W(\gotho_L)[p^{-1}] \cap I W(\gotho_M)[p^{-1}]}
= \frac{W(\gotho_L)[p^{-1}]}{I W(\gotho_L)[p^{-1}]}
= F.
\]
Consequently, $E^G = F$, so by Artin's lemma, $E$ is a finite Galois extension of $F$
and $[E:F] = \#G = [M:L]$. This proves the claim when $M$ is Galois; the general case follows by Artin's lemma again.
\end{proof}

For the reverse direction, the crucial case is when the characteristic $p$ field is algebraically
closed.
\begin{lemma} \label{L:lift algebraically closed}
Fix $F$ and $(L,I)$ corresponding as in Theorem~\ref{T:perfectoid field}.
If $L$ is algebraically closed, then so is $F$.
\end{lemma}
\begin{proof}
Let $\beta$ denote the norm on $F$.
Let $P(T) \in \gotho_F[T]$ be an arbitrary monic polynomial of degree $d \geq 1$; it suffices to check that
$P(T)$ has a root in $\gotho_F$.
We will achieve this by exhibiting a sequence $x_0,x_1,\dots$ of elements of $\gotho_F$ such that
for all $n \geq 0$,  $\beta(P(x_n)) \leq p^{-n}$ and $\beta(x_{n+1} - x_n) \leq p^{-n/d}$.
This sequence will then have a limit $x \in \gotho_F$ which is a root of $P$.

To begin, take $x_0 = 0$. Given $x_n \in \gotho_F$ with $\beta(P(x_n)) \leq p^{-n}$,
write $P(T + x_n) = \sum_i Q_i T^i$.
If $Q_0 = 0$, we may take $x_{n+1} = x_n$, so assume hereafter that $Q_0 \neq 0$.
Put
\[
c = \min\{\beta(Q_0/Q_j)^{1/j}: j > 0, Q_j \neq 0\};
\]
by taking $j=d$, we see that $c \leq \beta(Q_0)^{1/d}$.
Also, $\beta(F^\times) = \alpha(L^\times)$
by Lemma~\ref{L:deeply ramified}, and the latter group is divisible because
$L$ is algebraically closed; we thus have $c = \beta(u)$ for some $u \in \gotho_F$.

Apply Lemma~\ref{L:deeply ramified} to construct $\overline{z} \in \gotho_F$ with $\alpha(\overline{z}) = p^{-1}$.
For each $i$, choose $\overline{R}_i \in \gotho_{L}$
whose image in $\gotho_{L}/(\overline{z}) \cong \gotho_F/(p)$ is the same as that of
$Q_i u^i/Q_0$. Define the polynomial $\overline{R}(T) = \sum_i \overline{R}_i T^i \in \gotho_{L}[T]$.
By construction, the largest slope in the Newton polygon of $\overline{R}$ is 0; by this observation
plus the fact that $L$ is algebraically closed, it follows that $\overline{R}(T)$ has a root
$y' \in \gotho_{L}^\times$. Choose $y \in \gotho_F^\times$ whose image in $\gotho_F/(p) \cong \gotho_{L}/(\overline{z})$
is the same as that of $y'$, and take $x_{n+1} = x_n + u y$.
Then $\sum_i Q_i u^i y^i / Q_0 \equiv 0 \pmod{p}$, so
$\beta(P(x_{n+1})) \leq p^{-1} \beta(Q_0) \leq p^{-n-1}$
and $\beta(x_{n+1} - x_n) = \beta(u) \leq \beta(Q_0)^{1/d} \leq p^{-n/d}$.
We thus obtain the desired sequence, proving the claim.
\end{proof}

\begin{theorem} \label{T:mixed lift field}
For $F$ and $(L,I)$ corresponding as in Theorem~\ref{T:perfectoid field},
the correspondence described in Lemma~\ref{L:transfer degree}
induces a tensor equivalence $\FEt(F) \cong \FEt(L)$. In particular, every finite extension of $F$
is perfectoid, and the absolute Galois groups of $F$ and $L$ are homeomorphic.
\end{theorem}
\begin{proof}
Let $M$ be the completion of an algebraic closure of $L$.
Via Theorem~\ref{T:perfectoid field},
$(M, I W(\gotho_M))$ corresponds to a perfectoid analytic field $E$, which by Lemma~\ref{L:lift algebraically closed} is
algebraically closed.

By Lemma~\ref{L:transfer degree}, each finite Galois extension of $L$ within $M$ corresponds to a finite
Galois extension of $F$ within $E$ which is perfectoid. The union of the latter is an algebraic extension
of $F$ whose closure is the algebraically closed field $E$; the union is thus forced to be separably closed by Krasner's lemma.
Since $F$ is of characteristic $0$ and hence perfect,
every finite extension of $F$ is thus forced to lie within a finite Galois extension which is perfectoid;
the rest follows from Theorem~\ref{T:perfectoid field}.
\end{proof}

\begin{remark}
Using Theorem~\ref{T:mixed lift field}, it is not difficult to show that the functor $F \rightsquigarrow L$
induced by Theorem~\ref{T:perfectoid field} by forgetting the ideal $I$ is not fully faithful.
For instance, as $F$ varies over finite totally ramified extensions of the completion of $\Qp(\mu_{p^\infty})$ of a fixed degree, the fields
$L$ are all isomorphic.
\end{remark}

Theorem~\ref{T:mixed lift field} implies that the perfectoid property moves up along finite extensions of
analytic fields. It also moves in the opposite direction.
(See Proposition~\ref{P:perfectoid ring descent} for a more general result.)

\begin{lemma} \label{L:trivial cocycle}
Let $K$ be a perfect analytic field of characteristic $p$, and let $G$ be a finite group that acts
faithfully on $K$ by isometric automorphisms. Then $H^1(G, 1 + \gothm_K) = 0$.
\end{lemma}
\begin{proof}
We start with an observation concerning additive Galois cohomology.
Since $K$ is an acyclic $K^{G}[G]$-module by the normal basis theorem, the complex
\[
K \to \Hom(G,K)\to \Hom(G^2,K) \to \cdots
\]
computing Galois cohomology is exact. Using the inverse of Frobenius as in Remark~\ref{R:perfect uniform strict}, we see
that this complex is in fact almost optimal exact for the supremum norm on each factor.

Now let $f: G \to 1 + \gothm_K$ be a 1-cocycle, and put $\delta = \max\{|f(g)-1|: g \in G\} < 1$.
If we view $f$ as an element of $\Hom(G,K)$, its image in $\Hom(G^2,K)$ has supremum norm at most $\delta^2$.
By the previous paragraph, we can modify $f$ by an element of $1 + \gothm_K$ of norm at most
$\delta^{1/2}$ to get a new multiplicative cocycle $f'$ such that
$\max\{|f'(g)-1|: g \in G\} \leq \delta^{3/2}$. By iterating the construction, we obtain the desired
conclusion.
\end{proof}

\begin{prop} \label{P:perfectoid field descent}
Let $E/F$ be a finite extension of analytic fields such that $E$ is perfectoid. Then $F$ is also perfectoid.
\end{prop}
\begin{proof}
By Theorem~\ref{T:mixed lift field}, we are free to enlarge $E$, so we may assume $E/F$ is Galois with group
$G$. Let $(L, I)$ be the pair corresponding to $E$ via Theorem~\ref{T:perfectoid field}; then $G$ acts on
both $L$ and $I$.

We first check that $I$ admits a $G$-invariant generator (this is immediate if $F$ is already known to contain
a perfectoid field, but not otherwise). Let $z \in I$ be any generator. Write $z$ as $[\overline{z}] + p z_1$
with $z_1 \in W(\gotho_E)^\times$; then $z_1^{-1} z$ is also a generator. Define the function $f: G \to  W(\gotho_E)^\times$ taking $g \in G$ to $g(z_1^{-1} z)/(z_1^{-1} z)$. The composition $G \to W(\gotho_E)^\times
\to W(\kappa_E)^\times$ is identically 1, so we may apply Lemma~\ref{L:trivial cocycle} to trivialize the 1-cocycle;
that is, there exists $y \in W(\gotho_E)^\times$ with $f(g) = g(y)/y$ for all $g \in G$. Then
$(yz_1)^{-1} z$ is a $G$-invariant generator of $I$.

Put $K = L^G$; by Artin's lemma, $L$ is Galois over $K$ of degree
$\#G = [E:F]$. Since $I$ admits a generator contained in $W(\gotho_K)$ (which is then primitive of degree 1),
we may apply Theorem~\ref{T:perfectoid field} to the pair $(K, I \cap W(\gotho_K))$ to obtain a perfectoid
field $F'$. By Lemma~\ref{L:transfer degree}, $E$ is Galois over $F'$ of degree $[L:K] = [E:F]$ with Galois group $G$;
consequently, $F' = E^G = F$. This proves the claim.
\end{proof}

\begin{defn}
An analytic field $K$ is \emph{deeply ramified} if for any finite extension $L$ of $K$,
$\Omega_{\gotho_L/\gotho_K} = 0$; that is, the morphism $\Spec(\gotho_L) \to \Spec(\gotho_K)$ is formally unramified.
(Beware that this morphism is usually not of finite type if $K$ is not discretely valued.)
\end{defn}

\begin{theorem} \label{T:formally unramified}
Any perfectoid analytic field is deeply ramified. (The converse is also true;
see \cite[Proposition~6.6.6]{gabber-ramero}.)
\end{theorem}
\begin{proof}
Let $F$ be a perfectoid field, and let $E$ be a finite extension of $F$. Since $E/F$ is separable,
$\Omega_{E/F} = 0$; it follows easily that $\Omega_{\gotho_E/\gotho_F}$ is killed by some nonzero element of $\gotho_F$.
On the other hand, since $E$ is perfectoid by Theorem~\ref{T:mixed lift field},
for any $x \in \gotho_E$, we can find $y \in \gotho_E$ for which $x \equiv y^p \pmod{p}$.
Hence $\Omega_{\gotho_E/\gotho_F} = p \Omega_{\gotho_E/\gotho_F}$; it now follows that
 $\Omega_{\gotho_E/\gotho_F} = 0$.
\end{proof}

\begin{remark}
Many cases of Theorem~\ref{T:mixed lift field} in which $F$ is the completion of an algebraic extension of $\Qp$
arise from the \emph{field of norms} construction of Fontaine and Wintenberger
\cite{fontaine-wintenberger, wintenberger}.
For instance, one may take any \emph{arithmetically profinite} extension of $\Qp$ thanks to
Sen's theory of ramification in $p$-adic Lie extensions \cite{sen-lie}.
The approach to Theorem~\ref{T:mixed lift field} instead requires only checking the perfectoid condition
for a single analytic field, as then it is transmitted along finite extensions. For example,
for $F$ the completion of $\Qp(\mu_{p^\infty})$, the perfectoid condition is trivial to check.
\end{remark}

\begin{remark} \label{R:gabber-ramero}
Theorems~\ref{T:mixed lift field} and~\ref{T:formally unramified} have also been obtained by Scholze
\cite{scholze1} using an analysis of valuation rings made by Gabber and Ramero \cite[Chapter~6]{gabber-ramero}
in the language of \emph{almost ring theory}.
This generalizes the alternate proof of the Fontaine-Wintenberger theorem introduced by Faltings;
see \cite[Exercise~13.7.4]{brinon-conrad}.
 Scholze uses the term \emph{tilting} to refer to the relationship
between $F$ and $L$, as well as to the corresponding relationship between Banach algebras introduced
in Theorem~\ref{T:perfectoid ring}. (The term \emph{perfectoid} is also due to Scholze.)
\end{remark}

\subsection{The perfectoid correspondence for adic Banach algebras}
\label{subsec:perfectoid2}

We next extend Theorem~\ref{T:perfectoid field} to a correspondence of adic Banach algebras. A parallel development appears in the work of Scholze \cite{scholze1}, but he fixes a pair of corresponding fields and works over these; our treatment does not require this, and gives rise to perfectoid algebras which need not be defined over a perfectoid field. 
Our treatment is much closer in spirit to that given in the Bourbaki seminar of Fontaine
\cite{fontaine-bourbaki}.
The development in \cite{gabber-ramero-arxiv} takes a similar (albeit even more general) approach, but in common with Scholze's treatment it depends heavily on almost ring theory, which ours does not; we achieve similar effects by keeping track of norms. If one is interested in the statements in their almost-ring-theoretic form, these can be recovered after the fact (see for example \S\ref{subsec:almost purity}).

The form of the following definition is taken from \cite{kedlaya-davis}, where the perfectoid condition is studied from a purely ring-theoretic point of view.

\begin{defn} \label{D:perfectoid Banach}
A uniform adic Banach $\Qp$-algebra $(A, A^+)$ is \emph{perfectoid} if 
$\overline{\varphi}: A^+/(p) \to A^+/(p)$ is surjective
and there exists $x \in A^+$ with $x^p \equiv p \pmod{p^2 A^+}$. 
A uniform Banach algebra $A$ is \emph{perfectoid} if $(A,A^\circ)$ is perfectoid.
Note that the condition of $(A,A^+)$ being perfectoid only depends on $A$
(see Proposition~\ref{P:perfectoid formulations}).
\end{defn}

It is worth pointing out some equivalent formulations of the perfectoid property.
\begin{prop} \label{P:perfectoid formulations}
Let $F$ be an analytic field containing $\Qp$ and let $(A,A^+)$ be a uniform adic Banach $F$-algebra.
\begin{enumerate}
\item[(a)]
If $\left| F^\times \right| \neq \left| \Qp^\times \right|$, then $(A,A^+)$ is perfectoid if and only if   $\overline{\varphi}: A^+/(p) \to A^+/(p)$ is surjective.
\item[(b)]
The field $F$ is perfectoid as a uniform Banach $\Qp$-algebra as per Definition~\ref{D:perfectoid Banach}
if and only if it is perfectoid as an analytic field as per Definition~\ref{D:perfectoid field}.
\item[(c)]
Equip $A$ with the spectral norm. Then $A$ is perfectoid if and only if there exists $c \in (0,1)$ such that
for every $x \in A$, there exists $y \in A$ with $|x-y^p| \leq c |x|$. Moreover, if this holds for some $c$,
it holds for any $c \in (p^{-1},1)$.
\item[(d)]
The ring $A$ is perfectoid if and only if $(A,A^+)$ is perfectoid.
\item[(e)]
The ring $A$ is perfectoid if and only if there exists a topologically nilpotent unit $\varpi \in A$ such that $\varpi^p$ divides $p$ in $A^+$ and $\overline{\varphi}: A^+/(\varpi) \to A^+/(\varpi^p)$ is surjective. (This criterion of Fontaine can be used to define perfectoid rings which are not $\Qp$-algebras; see Remark~\ref{R:Fontaine perfectoid}.)
\end{enumerate}
\end{prop}
\begin{proof}
To check (a), assume that $\overline{\varphi}: A^+/(p) \to A^+/(p)$ is surjective.
By hypothesis, there exists $a \in F$ with $p^{-1} < |a| < 1$,
and there exist $b,c \in A^+$ with $b^p \equiv a \mod{pA^+}$, $c^p \equiv (p/a) \pmod{pA^+}$.
In particular, $b^p/a$ and $ac^p/p$ are elements of $A^+$ congruent to $1$ modulo
$p/a$ and $a$, respectively, and so are units. We can thus find $d \in A^+$ with $d^p \equiv (b^p/a)^{-1} (ac^p/p)^{-1} \pmod{p A^+}$, and then
$x = bcd$ has the property that $x^p \equiv p \pmod{p^2 A^+}$. Hence $A$ is perfectoid.
This yields (a), from which (b) follows by taking $(A,A^\circ) = (F, \gotho_F)$.

To check (c), 
suppose first that the given condition holds for some $c \in (0,1)$.
We may then construct a sequence $x_1,x_2,\dots$ in $A$ such that $\left| p - x_1^p \right| \leq cp^{-1}$ and $\left| x_{n} - x_{n+1}^p \right| \leq cp^{p^{-n}}$ for $n \geq 1$. For $n$ sufficiently large, the conditions of Remark~\ref{R:reduce surjectivity} are satisfied for $x = x_n, y = x_{n+1}$, and so 
$\overline{\varphi}: \gotho_A/(p) \to \gotho_A/(p)$ is surjective.
We may then choose $y \in \gotho_A$ with $y^p \equiv p/x_1^p \pmod{p \gotho_A}$, and then $x = x_1 y$ satisfies $x^p \equiv p \pmod{p^2 \gotho_A}$. We conclude that $A$ is perfectoid.

Conversely, suppose that $A$ is perfectoid. Then for every nonnegative integer $m$, we can find $x_m \in A$ with
$x_m^{p^m}/p \in \gotho_A^\times$. Given $x \in A$ nonzero and $c \in (p^{-1},1)$, choose a nonnegative integer $m$ and
an integer $t$ such that $p^{-1}/c < |x/x_m^{pt}| \leq 1$. Since $A$ is perfectoid, we can find $w \in \gotho_A$
with $w^p \equiv (x/x_m^{pt}) \pmod{p}$; then $y = x_m^t w$ satisfies $|x-y^p| \leq c|x|$.

To check (d), suppose first that $(A,A^+)$ is perfectoid. To check that $A$ is perfectoid, we check the criterion of (c) for any $c \in (p^{-1},1)$. Choose $x_1,x_2,\dots \in A^+$ with $x_1^p \equiv p \pmod{p^2 A^+}$ and $x_{n+1}^p \equiv x_n \pmod{p A^+}$ for $n \geq 1$.
Given $x \in A$, we can find a positive integer $n$ and some integer $m$ such that
$\left| x_{n+1}^{pm} x \right| \in (p^{-1}/c, 1)$. 
Then $x_{n+1}^{pm} x \in A^+$, so we can find $y \in A^+$ with $y^p \equiv x_{n+1}^{pm} x \pmod{p A^+}$. We then have $\left|(y/x_{n+1}^m)^p - x \right| \leq c \left| x \right|$, verifying the criterion.

Conversely, suppose that $A$ is perfectoid. We then make the following observations.
\begin{enumerate}
\item[(i)]
There exists $x_1 \in A^\circ$ with $x_1^p \equiv p \pmod{p^2 A^\circ}$. Since $x_1$ and $(p/x_1)$ are topologically nilpotent, they belong to $A^+$.
\item[(ii)]
For any $y \in A^+$, there exists $z \in A^\circ$ with $z^p \equiv y \pmod{p A^\circ}$.
Since $z^p - y$ is topologically nilpotent, it belongs to $A^+$. Since $A^+$ is integrally closed, $z$ belongs to $A^+$. Consequently, $\overline{\varphi}: A^+/(x_1,p) \to A^+/(x_1,p)$ is surjective.
\item[(iii)]
By (ii), there exist $x_2, x_3 \in A^+$ with $x_2^p \equiv x_1 \pmod{pA^\circ}$,
$x_3^p \equiv x_2 \pmod{pA^{\circ}}$.
In particular, $x_3^p \equiv x_2 \pmod{A^+/(x_2,p)}$.
\item[(iv)]
By Remark~\ref{R:reduce surjectivity}, $\overline{\varphi}: A^+/(p) \to A^+/(p)$ is surjective.
\item[(v)]
By (iv), there exists $y \in A^+$ with $y^p \equiv x_1^p/p \pmod{pA^+}$. Then $x = x_1 y$ satisfies $x^p \equiv p \pmod{p^2 A^+}$. Hence $(A,A^+)$ is perfectoid.
\end{enumerate}

To check (e), note that if $A$ is perfectoid, then the stated criterion holds for $\varpi = x$ for any $x \in A^+$ with $x^p \equiv p \pmod{p^2 A^+}$. Conversely, if the criterion holds, then $\overline{\varphi}: A^+/(p) \to A^+/(p)$ is surjective by Remark~\ref{R:reduce surjectivity}.
We may thus choose $x_0 \in A^+$ with $x_0^p \equiv (p/\varpi^p) \pmod{pA^+}$, and then
$x_1 = \varpi x_0$ satisfies
$x_1^p \equiv p \pmod{p\varpi^p A^+}$.
In particular, $x_1^p/p$ is congruent to 1 modulo $\varpi^p A^+$ and hence is a unit in $A^+$. We may thus choose $x_2 \in A^+$ with $x_2^p \equiv (p/x_1^p) \pmod{p A^+}$,
and then $x = x_1 x_2$ satisfies $x^p \equiv p \pmod{p^2 A^+}$.
\end{proof}

\begin{lemma} \label{L:primitive generator}
Let $(A,A^+)$ be a perfectoid uniform Banach $\Qp$-algebra.
Then the homomorphism $\theta: W(A^{+,\frep}) \to A^+$ is surjective,
with kernel generated by an element $z$ which is primitive of degree $1$.
\end{lemma}
\begin{proof}
We may assume $A$ carries its spectral norm.
The map $\theta$ is surjective because $\overline{\varphi}$ is surjective on $A^+/(p)$.
Choose $x \in A^+$ with $x^p \equiv p \pmod{p^2}$.
Choose $\overline{z} = (\dots, \overline{z}_1, \overline{z}_0) \in A^{+,\frep}$
with $\overline{z}_1$ equal to the reduction of $x$ modulo $p$; then $\theta([\overline{z}]) \equiv p \pmod{p^2 A^+}$.
For $n \geq 1$, choose $z_n \in A^+$ lifting $\overline{z}_n$; then $z_n^{p^n} \equiv p \pmod{p^2 A^+}$, so $z_n \in A^\times$ and $|z_n y| = p^{p^{-n}} |y|$ for all $y \in A$.

Since $\theta([\overline{z}])$ is divisible by $p$ in $A^+$, we can find $t \in W(A^+)^\times$ with $\theta(t) = \theta([\overline{z}])/p$. Then $z = [\overline{z}] - pt$ is primitive of degree $1$ and belongs to the kernel of $\theta$.
By parts (a) and (b) of Lemma~\ref{L:ker theta}, $z$ generates the kernel of $\theta$ on $W(\gotho_A^{\frep})$. In particular, given $y \in \ker(\theta: W(A^{+,\frep}) \to A^+)$, there is a unique $x = \sum_{n=0}^\infty p^n [\overline{x}_n] \in W(\gotho_A^{\frep})$ satisfying $xz = y$.
But since $t \in W(A^+)^\times$, each $\overline{x}_n$ belongs to 
$A^{+,\frep} + \gothm_A^{\frep} = A^{+,\frep}$, so $y$ is divisible by $z$. This proves the claim.
\end{proof}

\begin{defn}
Let $(A,A^+)$ be a perfectoid uniform adic Banach $\Qp$-algebra.  By Lemma~\ref{L:primitive generator},
the kernel of $\theta: W(A^{+,\frep}) \to A^+$ is generated by some element $z$ which is primitive of degree $1$.
Let $\overline{z} \in A^{+,\frep}$ be the reduction of $z$, and define $R(A) = A^{+,\frep}[\overline{z}^{-1}]$ and $R^+(A^+) = A^{+,\frep}$.
Note that this construction does not depend on the choice of $z$ and that $R(A)$ does not depend on $A^+$ (hence the notation).
By Lemma~\ref{L:frep norm}, $(R(A),R^+(A^+))$ is a perfect uniform Banach $\Fp$-algebra.
Also write $I(A,A^+) = \ker(\theta) = z W(A^{+,\frep})$; we also write $I(A)$ for $I(A,A^\circ)$.

For $(R,R^+)$ a perfect uniform adic Banach $\Fp$-algebra and $I$ an ideal of $W(R^+)$ generated by an element $z$ which
is primitive of degree $1$, write $A(R, I) = (W(R^+)/I)[p^{-1}]$
and $A^+(R^+,I) = W(R^+)/I$.
By Lemma~\ref{L:stable residue}, the surjective map $W(R^+)[[\overline{z}]^{-1}] \to A(R,I)$ is optimal,
so $(A(R,I), A^+(R^+,I))$ is a perfectoid uniform adic Banach $\Qp$-algebra; we sometimes denote this object by $(A(R), A^+(R^+))$ when the choice of $I$ is to be understood.

Note that both of these constructions transfer idempotents to idempotents:
if $e \in A$ is idempotent, then $e \in A^+$ because $A^+$ is integrally closed, and so $(\dots, e,e)$ is an idempotent of $R^+$; conversely, if $\overline{e} \in R$ is idempotent, then so is $\theta([\overline{e}])$. Consequently, 
for $A$ and $(R,I)$ corresponding as in Theorem~\ref{T:perfectoid ring}, $\Spec(R)$ and $\Spec(A)$ below
have the same closed-open subsets;
however, they need not have the same irreducible components.
\end{defn}

\begin{theorem}[Perfectoid correspondence] \label{T:perfectoid ring}
The functors
\[
A \rightsquigarrow (R(A), I(A)),
\qquad
(R,I) \rightsquigarrow A(R,I)
\]
define an equivalence of categories between
perfectoid uniform Banach $\Qp$-algebras $A$ and
pairs $(R,I)$ in which $R$ is a perfect uniform Banach $\Fp$-algebra
and $I$ is a principal ideal of $W(\gotho_R)$ generated by an element which is primitive of degree $1$. Similarly, the functors
\[
(A,A^+) \rightsquigarrow ((R(A), R^+(A^+)), I(A,A^+)),
\qquad
((R,R^+),I) \rightsquigarrow (A(R, I), A^+(R^+, I))
\]
define an equivalence of categories between
perfectoid uniform adic Banach $\Qp$-algebras $(A,A^+)$ and
pairs $((R,R^+),I)$ in which $(R,R^+)$ is a perfect uniform adic Banach $\Fp$-algebra
and $I$ is a principal ideal of $W(R^+)$ generated by an element which is primitive of degree $1$.
\end{theorem}
\begin{proof}
The proofs of the two assertions are similar, so we give only the first one.
Given $A$ carrying its spectral norm, the surjectivity of $\theta: W(\gotho_A^{\frep}) \to \gotho_A$ provides a natural isomorphism
$A(R(A), I(A)) \cong A$. Conversely, given $(R,I)$ with $R$ carrying its spectral norm, note that $\gotho_{A(R,I)} = W(\gotho_R)/I$,
so $\gotho_{A(R,I)}/(p) = W(\gotho_R)/(p,I) = \gotho_R/(\overline{z})$
for any $z \in I$ which is primitive of degree 1.
This yields a natural isomorphism $\gotho_{A(R,I)}^{\frep} \cong \gotho_R$,
under which $I(A(R,I)) \subset W(\gotho_{A(R,I)}^{\frep})$ corresponds to $I \subset W(\gotho_R)$.
\end{proof}

We introduce a key example: the perfectoid analogue of a Tate algebra.
\begin{example} \label{exa:perfectoid Tate algebra}
Suppose that $A$ and $(R,I)$ correspond as in Theorem~\ref{T:perfectoid ring}, e.g., $A = F$ and $R = L$ where $F$ and $(L,I)$ correspond as in Theorem~\ref{T:perfectoid field}.
For $r_1,\dots,r_n > 0$, let $B, S$ be the completions of
\[
A\{T_1/r_1,\dots,T_n/r_n\}[T_1^{p^{-\infty}}, \dots, T_n^{p^{-\infty}}],
R\{T_1/r_1,\dots,T_n/r_n\}[T_1^{p^{-\infty}}, \dots, T_n^{p^{-\infty}}]
\]
under the extension of the weighted Gauss norm. That is, the norm of $\sum_{i_1,\dots,i_n} a_{i_1,\dots,i_n} T_1^{i_1} \cdots T_n^{i_n}$ is the maximum of the norm of $a_{i_1,\dots,i_n}$ times $r_1^{i_1} \cdots r_n^{i_n}$ over all $i_1,\dots,i_n \geq 0$. Then $B$ is perfectoid, $S$ is perfect, and $B$ corresponds to $(S, I W(\gotho_S))$ via Theorem~\ref{T:perfectoid field} with the element $(T_i, T_i^{1/p},\dots)$ of $\gotho_B^{\frep}$ corresponding to $T_i \in S$.
\end{example}

For some applications, it will be useful to have the following refinement of the criterion from Proposition~\ref{P:perfectoid formulations}(c).
\begin{cor} \label{C:perfectoid calculation}
Let $A$ be a perfectoid uniform Banach $\Qp$-algebra with spectral norm $|\cdot|$,
and let $R$ be the perfect uniform Banach $\Fp$-algebra corresponding to $A$
via Theorem~\ref{T:perfectoid ring}. Then for any $\epsilon > 0$, any nonnegative
integer $m$, and any $x \in A$, there exists $y \in A$ of the form $\theta([\overline{y}])$ for some
$\overline{y}\in R$,
such that
\begin{equation} \label{eq:perfectoid calculation1}
\beta(x - y^p) \leq \max\{p^{-1-p^{-1}-\cdots-p^{-m}} \beta(x), \epsilon\} \qquad (\forall \beta \in \calM(A)).
\end{equation}
In particular, we may choose $y$ such that
\begin{equation}\label{eq:perfectoid calculation}
|x - y^p| \leq p^{-1-p^{-1}-\cdots-p^{-m}} |x|.
\end{equation}
\end{cor}
\begin{proof}
Lift $x$ along $\theta$ to $\tilde{x} \in W(\gotho_R)[[\overline{z}]^{-1}]$
and then apply Lemma~\ref{L:perfectoid calculation} with $\tilde{x}$ playing the role of $x$.
Let $\sum_{n=0}^\infty p^n [\overline{y}_n]$ be the resulting element of $W(\gotho_R)[[\overline{z}]^{-1}]$;
then $y = \theta([\overline{y}_0^{1/p}])$ satisfies \eqref{eq:perfectoid calculation1}. To obtain \eqref{eq:perfectoid calculation}, take $y=0$ if $x=0$, and otherwise apply \eqref{eq:perfectoid calculation1} with
$\epsilon$ equal to the right side of \eqref{eq:perfectoid calculation}.
\end{proof}

\begin{remark}
The constant in \eqref{eq:perfectoid calculation} cannot be improved to $p^{-p/(p-1)}$. See
\cite[Example~5.9]{kedlaya-davis}.
\end{remark}

Using Theorem~\ref{T:perfectoid ring}, we may replicate the conclusions of Remark~\ref{R:perfect uniform strict}
with perfect uniform Banach $\Fp$-algebras replaced by perfectoid algebras, in the process obtaining compatibility
of the correspondence described in Theorem~\ref{T:perfectoid ring} with various natural operations on adic Banach rings.
We begin with a correspondence between strict maps that includes an analogue of
Remark~\ref{R:perfect uniform strict}(a) in characteristic $0$.
\begin{prop} \label{P:perfectoid uniform strict}
Keep notation as in Theorem~\ref{T:perfectoid ring}, and equip all uniform Banach rings with their spectral norms.
\begin{enumerate}
\item[(a)]
Let $\overline{\psi}: R \to S$ be a strict (and hence almost optimal, by Remark~\ref{R:perfect uniform strict})
homomorphism of perfect uniform Banach $\Fp$-algebras,
and apply the functor $A$ to obtain $\psi: A \to B$. Then $\psi$ is almost optimal (and surjective if $\psi$ is).
\item[(b)]
Let $\overline{\psi}_1, \overline{\psi}_2: R \to S$ be
homomorphisms of perfect uniform Banach $\Fp$-algebras,
and apply the functor $A$ to obtain $\psi_1, \psi_2: A \to B$. If $\overline{\psi}_1 - \overline{\psi_2}$
is strict surjective
(and hence almost optimal, by Remark~\ref{R:perfect uniform strict}), then
$\psi_1 - \psi_2$ is almost optimal and surjective.
\item[(c)]
Let $\psi: A \to B$ be a strict homomorphism of perfectoid uniform Banach $\Qp$-algebras.
Then $\psi$ is almost optimal and $\psi(A)$ is perfectoid.
\item[(d)]
With notation as in (c), apply the functor $R$ to obtain $\overline{\psi}: R \to S$.
Then $\overline{\psi}$ is also almost optimal (and surjective if $\psi$ is).
\end{enumerate}
\end{prop}
\begin{proof}
We first check (a) in case $\overline{\psi}$ is strict surjective.
By Remark~\ref{R:perfect uniform strict}, $\overline{\psi}$ is almost optimal;
in particular, every element of $S$ of norm strictly less than 1 lifts to an element of $R$
of norm strictly less than 1. By Lemma~\ref{L:stable residue},
every element of $B$ of norm strictly less than 1 lifts to an element of $A$ of norm strictly less than 1.
Consequently, $\psi$ is almost optimal and surjective. A similar argument yields (b).

We now check (a) in the general case. We may factor $\psi$ as a composition $R \to S_0 \to S$ with $R \to S_0$
strict surjective and $S_0 \to S$ an isometric injection (since $R$ and $S$ are uniform).
For $z$ a generator of $I$, we have $z W(\gotho_S) \cap W(\gotho_{S_0}) = z W(\gotho_{S_0})$,
so the map $S_0 \to S$ corresponds to an isometric injection $B_0 \to B$. We may thus deduce (a) from the previous paragraph.

To check (c), let $\alpha, \beta$ be the spectral norms on $A,B$.
Choose a constant $c \geq 1$ such that every $b \in \image(\psi)$ admits a lift $a \in A$
with $\alpha(a) \leq c \beta(b)$. We will then prove that the same conclusion holds with $c$ replaced by $c^{1/p}$;
this is enough to imply the desired result.

Suppose that $b_l \in \image(\psi)$ for some nonnegative integer $l$.
Lift $b_l^p$ to $a_l \in A$ with $\alpha(a_l) \leq c \beta(b_l^p) = c \beta(b_l)^p$.
Apply Lemma~\ref{L:perfectoid calculation} (or \cite[Lemma~5.16]{kedlaya-witt}) to find $\overline{x} \in R$
such that
\begin{equation} \label{eq:approximate with theta}
\gamma(a_l - \theta([\overline{x}])) \leq p^{-1} \max\{\gamma(a_l), \beta(b_l)^p\}
\qquad (\gamma \in \calM(A)).
\end{equation}
In particular,
\[
\gamma(\theta([\overline{x}])) \leq \max\{\gamma(a_l),p^{-1} \gamma(a_l),p^{-1} \beta(b_l)^p\}
\leq c \beta(b_l)^p.
\]
Put $u_l = \theta([\overline{x}^{1/p}])$, $v_l = \psi(u_l)$, and $b_{l+1} = b_l - v_l$;
note that $\alpha(u_l) \leq c^{1/p} \beta(b_l)$.

For each $\gamma \in \calM(B)$, by applying
\eqref{eq:approximate with theta} to the restriction of $\gamma$ to $A$,
we find that
$\gamma(b_l^p - v_l^p) \leq p^{-1} \max\{\gamma(b_l)^p, \beta(b_l)^p \} = p^{-1} \beta(b_l)^p$. We now consider three cases.
\begin{enumerate}
\item[(i)]
If $\gamma(b_l^p - v_l^p) > \gamma(b_l)^p$, then $\gamma(b_l^p - v_l^p) = \gamma(v_l)^p$, so
$\gamma(b_{l+1}) = \gamma(v_l) > \gamma(b_l)$. It follows that
$\gamma(b_{l+1}) = \gamma(b_l^p - v_l^p)^{1/p} \leq p^{-1/p} \beta(b_l)$.
\item[(ii)]
If $p^{-1} \gamma(b_l)^p \leq \gamma(b_l^p - v_l^p) \leq \gamma(b_l)^p$,
we may apply \cite[Lemma~10.2.2]{kedlaya-course} to deduce that
$\gamma(b_{l+1}) \leq \gamma(b_l^p - v_l^p)^{1/p} \leq p^{-1/p} \beta(b_l)$.
\item[(iii)]
If $\gamma(b_l^p - v_l^p) \leq p^{-1} \gamma(b_l)^p$, then by \cite[Lemma~10.2.2]{kedlaya-course} again,
$\gamma(b_{l+1}) \leq p^{-1/p} \gamma(b_l) \leq p^{-1/p} \beta(b_l)$.
\end{enumerate}
It follows that $\gamma(b_{l+1}) \leq p^{-1/p} \beta(b_l)$ for all $\gamma \in \calM(B)$,
and so $\beta(b_{l+1}) \leq p^{-1/p} \beta(b_l)$.

If we now start with $b = b_0 \in \image(\psi)$ and recursively define $b_l, u_l, v_l$ as above,
the $b_l$ converge to 0, so the series $\sum_{l=0}^\infty v_l$ converges to $b$.
Meanwhile, the series $\sum_{l=0}^\infty u_l$ converges to a limit $a \in A$ satisfying
$\psi(a) = b$ and $\alpha(a) \leq c^{1/p} \beta(b)$.
Hence $\psi$ is almost optimal; by Proposition~\ref{P:perfectoid formulations}(c), $\psi(A)$ is perfectoid.
This proves (c).

To obtain (d), by noting that a strict injection of uniform Banach rings is isometric and invoking (c),
we may reduce to the case where $\psi$ is strict surjective.
Let $\overline{\alpha}, \overline{\beta}$ be the norms on $R,S$.
Given $\overline{y} \in S$, by (c) we may lift $\theta([\overline{y}]) \in B$ to some $a \in A$
with $\alpha(a) \leq p^{1/2} \overline{\beta}(\overline{y})$.
By Lemma~\ref{L:perfectoid calculation} again, we may find $\overline{x} \in R$
such that $\gamma(a - \theta([\overline{x}])) \leq p^{-1} \max\{\gamma(a), \overline{\beta}(\overline{y})\}$
for all $\gamma \in \calM(A)$; in particular,
\[
\gamma(\theta([\overline{x}])) \leq \max\{\gamma(a), p^{-1} \gamma(a), p^{-1} \overline{\beta}(\overline{y})\} \leq p^{1/2} \overline{\beta}(\overline{y}).
\]
Let $\overline{z} \in S$ be the image of $\overline{x}$.
For all $\gamma \in \calM(B)$,
we have
$\gamma(\theta([\overline{z}])) \leq p^{1/2} \overline{\beta}(\overline{y})$
and
\[
\gamma(\theta([\overline{y}] - [\overline{z}])) \leq p^{-1} \max\{\gamma(\theta([\overline{y}])), \overline{\beta}(\overline{y})\} = p^{-1} \overline{\beta}(\overline{y}).
\]
Put $\overline{\gamma} = \mu(\theta^*(\gamma)) \in \calM(S)$; then $\overline{\gamma}(\overline{z}) \leq p^{1/2} \overline{\beta}(\overline{y})$.
If we expand $[\overline{y}] - [\overline{z}] = \sum_{i=0}^\infty p^i [\overline{w}_i]$,
for $i>0$ we have $\gamma(\theta( p^i [\overline{w}_i])) \leq p^{-1} \max\{\overline{\gamma}(\overline{y}),
\overline{\gamma}(\overline{z})\} \leq p^{-1} \overline{\beta}(\overline{y})$.
Since $\overline{w}_0 = \overline{y} - \overline{z}$, it follows that
$\overline{\gamma}(\overline{y} - \overline{z}) = \gamma(\theta([\overline{y} - \overline{z}]))
\leq p^{-1/2} \overline{\beta}(\overline{y})$. By iterating the construction as in the proof of (b), we see that
every $\overline{y} \in S$ admits a lift $\overline{x} \in R$ for which
$\overline{\alpha}(\overline{x}) \leq p^{1/2} \overline{\beta}(\overline{y})$.
Hence $\overline{\psi}$ is strict, and hence almost optimal by Remark~\ref{R:perfect uniform strict}.

We thus may deduce (d) except for the fact that if $\psi$ is surjective, then so is $\overline{\psi}$.
This follows from (c) plus Lemma~\ref{L:dense image}(b).
\end{proof}

\begin{remark}
Note that in Proposition~\ref{P:perfectoid uniform strict}(a),
strict surjectivity of $\overline{\psi}$ does not imply that $\gotho_R$ surjects onto $\gotho_S$.
Similarly, in part (c), strict surjectivity of $\psi$ does not imply that $\gotho_A$ surjects onto $\gotho_B$.
\end{remark}

We next establish compatibility of the correspondence with completed tensor products, and obtain
an analogue of
Remark~\ref{R:perfect uniform strict}(c).
\begin{prop} \label{P:perfectoid tensor}
Let $A \to B, A \to C$ be morphisms of perfectoid uniform Banach $\Qp$-algebras.
Let $(R,I)$ be the pair corresponding to $A$ via Theorem~\ref{T:perfectoid ring},
and put $S = R(B), T = R(C)$.
Then the completed tensor product $B \widehat{\otimes}_A C$ with the tensor product norm
is the perfectoid uniform Banach $\Qp$-algebra corresponding to $S \widehat{\otimes}_R T$. (Note that this immediately implies the corresponding statement for adic Banach rings.)
\end{prop}
\begin{proof}
Put $U = S \widehat{\otimes}_R T$,
which is a perfect uniform Banach $\Fp$-algebra by Remark~\ref{R:perfect uniform strict}(c).
Put $D = A(U)$; then $D$ is perfectoid.
It remains to check that the natural map $B \widehat{\otimes}_A C \to D$ is an isometric isomorphism
of Banach $\Qp$-algebras.
To see this, let $\alpha,\beta,\gamma,\delta,\overline{\alpha}, \overline{\beta}, \overline{\gamma}, \overline{\delta}$
denote the spectral norms on $A,B,C,D,R,S,T,U$, respectively.
Choose any $x \in D$ and any $\epsilon > 1$,
and apply Lemma~\ref{L:stable residue} to find
$y = \sum_{n=0}^\infty p^n [\overline{y}_n] \in W(\gotho_U)[[\overline{z}]^{-1}]$
with $\theta(y) = x$ and $\overline{\delta}(\overline{y}_0) \geq \overline{\delta}(\overline{y}_n)$ for all $n  > 0$.
Then write each $\overline{y}_n$ as a convergent sum $\sum_{i=0}^\infty \overline{s}_{ni} \otimes \overline{t}_{ni}$
with $\overline{s}_{ni} \in S$, $\overline{t}_{ni} \in T$
and $\overline{\beta}(\overline{s}_{ni}), \overline{\gamma}(\overline{t}_{ni}) < (\epsilon \overline{\delta}(\overline{y}_n))^{1/2}$.
(More precisely, by Remark~\ref{R:perfect uniform strict}(c) we can ensure that
$\overline{\beta}(\overline{s}_ni) \overline{\gamma}(\overline{t}_{ni}) < \epsilon \overline{\delta}(\overline{y}_n)$,
but then we can enforce the desired inequality by transferring a suitable power of $\overline{z}$ between the two terms.)
We can then write $[\overline{y}_n]$ as a convergent sum for the $(p, z)$-adic topology, each term of which is a power of $p$ times the Teichm\"uller lift of an element of $S$ times the Teichm\"uller lift of an element of $T$; moreover,
each of these terms has norm at most $\epsilon \overline{\delta}(\overline{y}_n)$. It follows that
$x$ is the image of an element of $B \widehat{\otimes}_A C$ of norm at most $\epsilon \delta(y)$;
since $\epsilon > 1$ was arbitrary, this yields the desired result.
\end{proof}

\begin{remark}
At this point, we have analogues for perfectoid algebras of parts (a) and (c) of
Remark~\ref{R:perfect uniform strict}. It would be useful to also have an analogue of part (b);
that is, if $\psi_1, \psi_2: A \to B$ are two homomorphisms of perfectoid uniform Banach $\Qp$-algebras
such that $\psi = \psi_1 - \psi_2$ is strict, one would expect that $\psi$ is almost optimal.
Unfortunately, the technique of proof of Proposition~\ref{P:perfectoid uniform strict}(c) does not
suffice to establish this, due to the fact that the image of $\psi$ is not closed under taking $p$-th powers.
\end{remark}

We next establish the compatibility of the perfectoid correspondence with rational localizations, starting with an explicit calculation in the special case of a simple Laurent covering. (See Remark~\ref{R:perfectoid strict presentation2} for a related observation.)
\begin{lemma} \label{L:perfectoid rational}
Suppose that $(A,A^+)$ and $((R,R^+),I)$ correspond as in Theorem~\ref{T:perfectoid
ring}. Choose $\overline{g} \in R$ and put $g = \theta([\overline{g}]) \in A$.
Put
\begin{gather*}
B_- = A\{T\}/(T-g), \quad B_+ = A\{U\}/(Ug-1), \\
S_- = R\{\overline{T}\}/(\overline{T} -\overline{g}), \quad S_+ = R\{\overline{U}\}/(\overline{U} \overline{g}-1).
\end{gather*}
Then there are $A$-linear isomorphisms $A(S_-, IW(S_-^\circ)) \cong B_-$,
$A(S_+, IW(S_+^\circ)) \cong B_+$ taking $[\overline{T}],[\overline{U}]$ to $T,U$.
\end{lemma}
\begin{proof}
Equip $A$ with the spectral norm.
For $r \in \ZZ[p^{-1}]_{\geq 0}$, put $g^r = \theta([\overline{g}^r])$.
By Lemma~\ref{L:principal ideal closed}, for each nonnegative integer $h$,
\begin{align*}
B_- &\cong A\{T^{p^{-h}}\}/(T^{p^{-h}}-g^{p^{-h}}),\\
B_+ &\cong A\{U^{p^{-h}}\}/(U^{p^{-h}} g^{p^{-h}} - 1).
\end{align*}
More precisely, for $y_- = \sum_i y_{-,i} \in A\{T^{p^{-h-1}}\},
y_+ = \sum_i y_{+,i} \in A\{U^{p^{-h-1}}\}$, put
\begin{align*}
z_- &= \sum_{i} y_{-,i} g^{i - p^{-h} \lfloor p^h i \rfloor} T^{p^{-h} \lfloor p^h i \rfloor} \in A\{T^{p^{-h}}\},\\
z_+ &= \sum_{i} y_{+,i} g^{p^{-h} \lceil p^h i \rceil-i} U^{p^{-h} \lceil p^h i \rceil} \in A\{U^{p^{-h}}\};
\end{align*}
then $y_*$ and $z_*$ represent the same class in $B_*$ and
\[
\left| z_* \right| \leq \max\{1, \left| g \right|^{p^{-h}} \} \left| y_* \right|.
\]
Write $\left| \bullet \right|_h$ for the quotient norms on $B_-, B_+$ induced from
$A\{T^{p^{-h}}\}$, $A\{U^{p^{-h}}\}$; then
\[
\left| x \right|_{h+1} \leq \left| x \right|_h \leq \max\{1, \left| g \right|\}^{p^{-h}} \left| x \right|_{h+1} \qquad (x \in B_*).
\]
In particular,
\begin{equation} \label{eq:perfectoid rational1}
\left| x \right|_{h} \leq \left| x \right|_0 \leq \max\{1, \left| g \right|\}^{p/(p-1)} \left| x \right|_{h} \qquad (x \in B_*; h=0,1,\dots).
\end{equation}
By Proposition~\ref{P:perfect uniform localization}, $S_*$ is perfect uniform.
Let $A'_-, A'_+$ be the completions of $A\{T\}[T^{p^{-\infty}}]$,
$A\{U\}[U^{p^{-\infty}}]$ for the Gauss norm.
Let $J_*$ be the closure in $A'_*$ of the ideal of $A\{T\}[T^{p^{-\infty}}]$ generated by $T^{p^{-h}} - g^{p^{-h}}$ (if $* = -$)
or the ideal of $A\{U\}[U^{p^{-\infty}}]$ generated by $U^{p^{-h}} g^{p^{-h}}-1$ 
(if $* = +$)
for $h=0,1,\dots$.
Note that $J_*$ is itself an ideal of $A'_*$, so we may form the quotient
$B'_* = A'_*/J_*$; the inclusions $A\{T\} \to A'_-, A\{U\} \to A'_+$ then induce maps $B_* \to B'_*$, which by \eqref{eq:perfectoid rational1} are isomorphisms.

From Example~\ref{exa:perfectoid Tate algebra},
we see that $A'_-, A'_+$ are perfectoid and we obtain 
identifications $R(A'_*) \cong S_*$ taking $(\dots,T^{1/p},T), (\dots,U^{1/p},U)$
to $\overline{T}, \overline{U}$.
The resulting map $A'_* \to A(S_*,IW(S_*^\circ))$ is surjective; its kernel $J'_*$  is the closure of the
ideal generated by $\theta([\overline{T}^{p^{-h}} - \overline{g}^{p^{-h}}])$
(if $* = -$) or $\theta([\overline{U}^{p^{-h}} \overline{g}^{p^{-h}} - 1])$
(if $* = +$) for all $h$. 

Under the map $A'_- \to A(S_-,IW(S_-^\circ))$, $T^{p^{-h}}$ and $g^{p^{-h}}$
both map to $[\overline{T}^{p^{-h}}] = [\overline{g}^{p^{-h}}]$; similarly,
under the map $A'_+ \to A(S_+,IW(S_+^\circ))$, $U^{p^{-h}} g^{-p^{-h}}$
maps to $[\overline{U}^{p^{-h}} \overline{g}^{p^{-h}}] = [1] = 1$.
This means that $J_* \subseteq J'_*$; we also have the reverse inclusion thanks to Remark~\ref{R:addition formula}. We conclude that the induced map
$B'_* \to A(S_*,IW(S_*^\circ))$ is an isomorphism.
\end{proof}

With this calculation in hand, we may treat the general case.
\begin{theorem} \label{T:perfectoid rational}
Suppose that $(A,A^+)$ and $((R,R^+),I)$ correspond as in Theorem~\ref{T:perfectoid
ring}.
\begin{enumerate}
\item[(a)]
The homeomorphism $\calM(A) \cong \calM(R)$ of Theorem~\ref{T:quotient norm} lifts to a functorial homeomorphism $\Spa(A,A^+) \cong \Spa(R,R^+)$.
\item[(b)]
The homeomorphism in (a) identifies rational subspaces. 
More precisely, for $\overline{f}_1,\dots,\overline{f}_n,\overline{g} \in R$, the rational subspace
\[
\{v \in \Spa(R,R^+): v(\overline{f}_i) \leq v(\overline{g}) \quad (i=1,\dots,n)\}.
\]
corresponds to
\begin{equation} \label{eq:lifted rational subspace}
\{v \in \Spa(A,A^+): v(f_i) \leq v(g) \quad (i=1,\dots,n)\}
\quad (f_i = \theta([\overline{f}_i]), g = \theta([\overline{g}])),
\end{equation}
and every rational subspace of $\Spa(A,A^+)$ can be written in the form.
\item[(c)]
Let $U \subseteq \Spa(A,A^+)$ and $V \subseteq \Spa(R,R^+)$ be rational subdomains corresponding as in (a).
 Let $(A,A^+) \to (B,B^+)$, $(R,R^+) \to (S,S^+)$ be the rational localizations representing $U$
and $V$, respectively. Then $(B,B^+)$ is again perfectoid,
and there are natural identifications $(S,S^+) \cong (R(B),R^+(B^+))$, $(B,B^+) \cong (A(S, I
W(S^+)), A^+(S^+, IW(S^+)))$. 
\end{enumerate}
\end{theorem}
\begin{proof}
To prove (a), we first define the map on points.
For $\alpha \in \calM(A)$ corresponding to $\beta \in \calM(B)$, we have a canonical isomorphism $\gotho_{\calH(\alpha)}/(p) \cong \gotho_{\calH(\beta)}/(\overline{z})$
for any $z$ as in Lemma~\ref{L:primitive generator}. 
We may thus identify points of $\Spa(A,A^+)$ lifting $\alpha$ with valuation rings of $\calH(\alpha)$ containing $\gothm_{\calH(\alpha)}$, then with valuation rings of
$\kappa_{\calH(\alpha)} \cong \kappa_{\calH(\beta)}$, then with valuation rings of $\calH(\beta)$ containing $\gothm_{\calH(\beta)}$, then with points of $\Spa(R,R^+)$ lifting $\beta$.

We now have a functorial bijection $\Spa(A,A^+) \to \Spa(R,R^+)$. To see that this is a homeomorphism, it suffices to prove (b), which we do by imitating
the proof of Theorem~\ref{T:quotient norm}.
For $\overline{f}_1,\dots,\overline{f}_n, \overline{g} \in R$, if we put $f_i = \theta([\overline{f}_i]), g = \theta([\overline{g}])$,
then $f_1,\dots,f_n,g$ generate the unit ideal in $A$ if and only if
$\overline{f}_1,\dots,\overline{f}_n, \overline{g}$ generate the unit ideal in $R$ (by applying Corollary~\ref{C:ideal from spectrum} in both $A$ and $R$). 
This means that rational subspaces of $\Spa(R,R^+)$ correspond to rational subspaces of $\Spa(A,A^+)$ as described. Conversely, given a rational subspace $U$ as in \eqref{eq:adic rational subspace}, pick $\epsilon > 0$ as in Remark~\ref{R:approximate rational},
then apply Lemma~\ref{L:perfectoid calculation} to find
$\overline{f}_1,\dots,\overline{f}_n,\overline{g} \in R$ such that
\[
\alpha(f_i - \theta([\overline{f}_i])) \leq p^{-1} \max \{\alpha(f_i), 
\epsilon\}, \quad
\alpha(g - \theta([\overline{g}])) \leq p^{-1} \max \{\alpha(g), \epsilon\}
\quad (\alpha \in U \cap \calM(A)).
\]
As in Remark~\ref{R:approximate rational}, 
we see that the rational subspace defined by $\theta([\overline{f}_1]),
\dots, \theta([\overline{f}_n]), \theta([\overline{g}])$ coincides with $U$.
This proves (b).

To prove (c), by Proposition~\ref{P:Tate reduction single} we may assume that $U$ is part of a simple Laurent covering. 
From the proof of (b), we may define this covering using a parameter of the form $g = \theta([\overline{g}])$ for some $\overline{g} \in R$.
By Proposition~\ref{P:perfect uniform localization}, $(S,S^+)$ is perfect uniform.
By Lemma~\ref{L:perfectoid rational}, $B$ is perfectoid; 
by Proposition~\ref{P:perfectoid formulations}, $(B,B^+)$ is also perfectoid.
The other identifications now follow from Theorem~\ref{T:perfectoid ring} and the universal property of a rational localization.
\end{proof}

As a corollary, we obtain the Tate and Kiehl properties for perfectoid algebras.
\begin{theorem} \label{T:Tate-Kiehl analogue2}
Any perfectoid adic Banach algebra is stably uniform and sheafy, and thus satisfies the Tate sheaf and Kiehl glueing properties
(Definition~\ref{D:Tate-Kiehl properties}).
\end{theorem}
\begin{proof}
The stably uniform proprety is immediate from Theorem~\ref{T:perfectoid rational}.
The sheafy property then follows from Theorem~\ref{T:uniform rational is sheafy};
alternatively, one may use Proposition~\ref{P:Tate reduction}
and Proposition~\ref{P:perfectoid tensor} to reduce to the case of a simple Laurent covering of $\Spa(A,A^+)$,
then argue as in Proposition~\ref{P:perfectoid uniform strict}(a,b) to derive the desired exact sequence from the corresponding exact sequence in characteristic $p$
(Theorem~\ref{T:Tate-Kiehl analogue1}). As usual, the Tate and Kiehl properties follow from the sheafy property via Theorem~\ref{T:tate to Kiehl}.
\end{proof}

We have the following analogue of Remark~\ref{R:perfect uniform localization1}.
\begin{remark}\label{R:perfectoid strict presentation2}
Set notation as in Theorem~\ref{T:perfectoid rational}, and equip $A$ with the spectral norm. For $r \in \ZZ[p^{-1}]_{\geq 0}$ and $* \in \{f_1,\dots,f_n,g\}$, write
$*^r$ for $\theta([\overline{*}^r])$.
Choose $h_1,\dots,h_n,k \in A$ such that $h_1 f_1 + \cdots + h_n f_n + kg = 1$.
Then 
\[
y = \sum_{i_1,\dots,i_n} y_{i_1,\dots,i_n} T_1^{i_1} \cdots T_n^{i_n} \in A\{T_1,\dots,T_n\}[T_1^{p^{-\infty}}, \dots, T_n^{p^{-\infty}}]
\]
represents the same element of $B$ as does
\[
z = (k + h_1 T_1 + \cdots + h_n T_n)^n \sum_{i_1,\dots,i_n} y_{i_1,\dots,i_n} f_1^{i_1 - \lfloor i_1 \rfloor}
\cdots f_n^{i_n - \lfloor i_n \rfloor} g^{n - (i_1 - \lfloor i_1\rfloor + \cdots + i_n - \lfloor i_n\rfloor)} T_1^{\lfloor i_1 \rfloor} \cdots T_n^{\lfloor i_n \rfloor},
\]
which satisfies
\[
\left| z \right| \leq c \left| y \right|, \qquad c  = \max\{\left| h_1 \right|,\dots, \left| h_n \right|, \left| k \right|\}^n \left|f_1 \right| \cdots \left| f_n \right| \left| g \right|^n.
\]
By replacing $\overline{f}_1,\dots,\overline{f}_n,\overline{g}$ by suitable $p$-power roots, for any given $c>1$ we may obtain a strict surjection $A\{T_1,\dots,T_n\} \to B$ in which the quotient norm is at most $c$ times the spectral norm.
\end{remark}

We next establish compatibility with formation of quotients, and more generally with passage along homomorphisms with dense image.
\begin{theorem} \label{T:perfectoid dense image}
Suppose that $A$ and $(R,I)$ correspond as in Theorem~\ref{T:perfectoid ring}.
\begin{enumerate}
\item[(a)]
Let $\overline{\psi}: R \to S$ be a bounded homomorphism of uniform Banach $\Fp$-algebras with dense image.
Then $S$ is also perfect, and the corresponding homomorphism $\psi: A \to B$ of perfectoid uniform Banach
$\Qp$-algebras also has dense image.
\item[(b)]
Let $B$ be a uniform Banach $\Qp$-algebra admitting
a bounded homomorphism $\psi: A \to B$ with dense image. Then $B$ is also perfectoid,
and the corresponding homomorphism $\overline{\psi}: R \to S$ of perfect uniform Banach $\Fp$-algebras
also has dense image.
\item[(c)]
In (a) and (b), $\overline{\psi}$ is surjective if and only if $\psi$ is.
\end{enumerate}
\end{theorem}
\begin{proof}
In the setting of (a), the ring $S$ is reduced and admits the dense perfect $\Fp$-subalgebra $\overline{\psi}(R)$,
so $S$ is also perfect. Let $\alpha, \beta, \overline{\alpha}, \overline{\beta}$ denote the spectral norms on $A,B,R,S$,
respectively. By Lemma~\ref{L:stable residue},
for any $x \in B$ and $\epsilon > 0$, we can find a finite sum
$\sum_{i=0}^n p^i [\overline{x}_i] \in W(S)$ such that $\beta(x - \theta(\sum_{i=0}^n p^i[\overline{x}_i])) < \epsilon$.
For $i=0,\dots,n$, if $\overline{x}_i = 0$, take $\overline{y}_i = 0$,
otherwise choose $\overline{y}_i \in R$ with
\[
\overline{\beta}(\overline{x}_i-\overline{\psi}(\overline{y}_i)) <
\inf\{ \epsilon^{p^j} p^{(i+j)p^{j}} \overline{\beta}(\overline{x}_i)^{1-p^j}: j=0,1,\dots\}.
\]
(Note that the sequence whose infimum is sought tends to $+\infty$ as $j \to \infty$, since it is
dominated by $p^{jp^j}$, so the infimum is positive.)
Put $y = \sum_{i=0}^n p^i \theta([\overline{y}_i]) \in A$; then
\begin{align*}
\beta\left(\sum_{i=0}^n p^i \theta([\overline{x}_i] - [\overline{\psi}(\overline{y}_i)])\right)
&\leq \max\{p^{-i} \beta(\theta([\overline{x}_i] - [\overline{\psi}(\overline{y}_i)])): i=0,\dots,n\} \\
&\leq \max\{p^{-i-j} \overline{\beta}(\overline{x}_i)^{1-p^{-j}} \overline{\beta}(\overline{x}_i-\overline{\psi}(\overline{y}_i))^{p^{-j}}: i=0,\dots,n; j=0,1,\dots\} \\
&< \epsilon.
\end{align*}
It follows that $\beta(x - \psi(\sum_{i=0}^n p^i [\overline{y}_i])) < \epsilon$, yielding (a).

In the setting of (b), let $\alpha, \beta$ be the norms on $A, B$.
Given $x \in \psi(A) \cap \gotho_B$, choose
$w \in \psi^{-1}(x)$. By Lemma~\ref{L:perfectoid calculation}, we can find $\overline{w} \in R$
such that
\[
\gamma(w - \theta([\overline{w}])) \leq p^{-1} \max\{\gamma(w), \beta(x)\} \qquad (\gamma \in \calM(A)).
\]
Put $y = \psi(\theta([\overline{w}^{1/p}]))$; then $\gamma(x - y^p) \leq p^{-1} \beta(x)$ for all $\gamma \in \calM(B)$,
so $\beta(x-y^p) \leq p^{-1} \beta(x)$. Since $\psi(A)$ is dense in $B$, $B$ is perfectoid.
Given $\overline{x} \in S$, choose $w \in A$ with
$\beta(\psi(w) - \theta([\overline{x}])) \leq p^{-1} \overline{\beta}(\overline{x})$,
then apply Lemma~\ref{L:perfectoid calculation} again to
choose $\overline{w} \in R$ such that
\[
\gamma(w - \theta([\overline{w}])) \leq p^{-1} \max\{\gamma(w), \overline{\beta}(\overline{x})\} \qquad (\gamma \in \calM(A)).
\]
Put $\overline{y} = \overline{\psi}(\overline{w})$; then $\gamma(\theta([\overline{x}] - [\overline{y}]))
\leq p^{-1} \overline{\beta}(\overline{x})$ for $\gamma \in \calM(B)$,
so $\gamma(\overline{x} - \overline{y}) \leq p^{-1} \overline{\beta}(\overline{x})$.
This yields (b).

In the setting of (c), if either $\overline{\psi}$ or $\psi$ is surjective, then
it is strict by the open mapping theorem (Theorem~\ref{T:open mapping}).
Consequently, (c) follows from Proposition~\ref{P:perfectoid uniform strict}(a,d).
\end{proof}

\begin{cor} \label{C:perfectoid tensor product over any base}
Let $\psi_1: C \to A$, $\psi_2: C \to B$ be bounded homomorphisms of
uniform Banach $\Qp$-algebras,
and let $D = (A \widehat{\otimes}_C B)^u$ denote the uniform completion of $A \otimes_C B$. If $A$ and
$B$ are perfectoid, then so is $D$.
\end{cor}
\begin{proof}
Let $\alpha,\beta$ be the spectral norms on $A,B$.
Let $N$ be the multiplicative monoid of $R(B)$.
Equip the monoid ring $A[N]$ with the weighted Gauss norm
\[
\left| \sum_i a_i [n_i] \right| = \max_i \{\alpha(a_i) \mu(\beta)(n_i)\}.
\]
Let $E$ be the completion of $A[N]$; it is a perfectoid algebra (compare
Example~\ref{exa:perfectoid Tate algebra}).
The formula
\[
\sum_i a_i [n_i] \mapsto \sum_i a_i \otimes \theta([n_i])
\]
defines a bounded homomorphism from $E$ to the ordinary completion of $A
\otimes_C B$ with dense image; consequently, the resulting homomorphism
$E \to D$ also has dense image. By
Theorem~\ref{T:perfectoid dense image}(b) again, $D$ is perfectoid.
\end{proof}

A related observation is that the perfectoid property is preserved under completions.
\begin{prop}
Let $A$ be a perfectoid uniform Banach $\Qp$-algebra equipped with its spectral norm. Let $J$ be a finitely generated
ideal of $\gotho_A$ which contains $p$. Equip each quotient $\gotho_A/J^n$ with the quotient norm, equip the inverse limit
$R$ with the supremum norm, and put $B = R[p^{-1}]$. Then $B$ is again a perfectoid uniform
Banach $\Qp$-algebra.
\end{prop}
\begin{proof}
Choose generators $x_1,\dots,x_m$ of $J$; then $R$ can also be written as the inverse limit of the quotients
$\gotho_A/(p^{n}, x_1^{pn}, \dots, x_m^{pn})$. Consequently,
each element $y$ of $R$ can be written as an infinite series
\[
\sum_{n=0}^\infty (a_{n0} p^n + a_{n1} x_1^{pn} + \cdots + a_{nm} x_m^{pn})
\]
with all of the $a_{ni}$ in $\gotho_A$. Choose $b_{ni} \in \gotho_A$ with $b_{ni}^p \equiv a_{ni} \pmod{p}$; then
the series
\[
b_{n0} + \sum_{n=0}^\infty (b_{n1}^p x_1^{n} + \cdots + b_{nm}^p x_m^{n})
\]
converges to an element $z$ of $R$ satisfying $z^p \equiv y \pmod{pR}$. From this, the claim follows at once.
\end{proof}

We finally establish compatibility with finite \'etale covers. As in the case of analytic fields,
the first step is to lift from characteristic $p$.
\begin{lemma} \label{L:lift finite etale perfectoid}
For $A$ and $(R,I)$ corresponding as in Theorem~\ref{T:perfectoid ring}
and $S \in \FEt(R)$, the perfectoid Banach algebra $B$ over $\Qp$ corresponding to $(S, IW(\gotho_S))$
via Theorem~\ref{T:perfectoid ring} belongs to $\FEt(A)$ and 
its norm is equivalent to any norm given by Proposition~\ref{P:finite etale Banach norm}.
\end{lemma}
\begin{proof}
We may assume that $S$ is of constant rank $d>0$ as an $R$-module.
Let $z$ be a generator of $I$ which is primitive of degree $1$.
Since $S$ is a finite $R$-module and $\overline{\varphi}$ is bijective on $S$,
we can find $\overline{x}_1, \dots, \overline{x}_n \in \gotho_S$ such that
$\gotho_S/(\overline{x}_1 \gotho_R + \cdots + \overline{x}_n \gotho_R)$ is killed by
$\overline{z}$.
Using Remark~\ref{R:addition formula}, it follows that
$W(\gotho_S)/([\overline{x}_1] W(\gotho_R) + \cdots + [\overline{x}_n] W(\gotho_R))$
is killed by $[\overline{z}]$. Quotienting by $z$ and then inverting $p$, we find that
$B$ is a finite $A$-module. By Proposition~\ref{P:finite generation2} and Lemma~\ref{L:transfer degree},
$B$ is locally free of constant rank $d$ as an $A$-module.

We now know that $B$ is a finite projective $A$-module.
To check that $B \in \FEt(A)$,
it remains to check that the map $B \to \Hom_A(B, A)$ taking $x$ to $y \mapsto \Trace_{B/A}(xy)$
is surjective. By Lemma~\ref{L:finite generation}, it suffices to check this pointwise;
we may thus apply Lemma~\ref{L:transfer degree} to conclude.

To conclude, note that the equivalence between the norm on $B$ and the one derived from Proposition~\ref{P:finite etale Banach norm} is a consequence of the open mapping theorem (Theorem~\ref{T:open mapping}).
\end{proof}

\begin{theorem} \label{T:mixed lift ring}
For $A$ and $(R,I)$ corresponding as in Theorem~\ref{T:perfectoid ring},
if we equip $B \in \FEt(A)$ with a Banach norm provided by Proposition~\ref{P:finite etale Banach norm}, then $B$ is a perfectoid uniform Banach $A$-algebra.
Moreover, the correspondence of Lemma~\ref{L:lift finite etale perfectoid} induces a tensor equivalence
$\FEt(A) \cong \FEt(R)$.
(As in Theorem~\ref{T:formally unramified}, it follows that
$\Omega_{\gotho_B/\gotho_A} = 0$.)
\end{theorem}
\begin{proof}
It suffices to check that any $B \in \FEt(A)$ arises as in Lemma~\ref{L:lift finite etale perfectoid}.
Extend $A$ to an adic Banach ring $(A,A^+)$ corresponding to $(R,R^+)$ via Theorem~\ref{T:perfectoid ring}.
Recall that by Theorem~\ref{T:quotient norm} and 
Theorem~\ref{T:perfectoid rational},
there are compatible homeomorphisms $\calM(A) \cong \calM(R)$,
$\Spa(A,A^+) \cong \Spa(R,R^+)$
matching up rational subdomains on both sides.
For each $\delta \in \calM(A)$, let $\gamma \in \calM(R)$ be the corresponding point.
Using Theorem~\ref{T:mixed lift field}, we may transfer
$B \otimes_A \calH(\delta) \in \FEt(\calH(\delta))$ to some $S(\gamma) \in \FEt(\calH(\gamma))$.
By Lemma~\ref{L:henselian direct limit}(a) (and Theorem~\ref{T:henselian}), there exists a rational localization $(R,R^+) \to (R_1, R_1^+)$
encircling $\gamma$ such that $S(\gamma)$ extends to $S_1 \in \FEt(R_1)$. Let $(A,A^+) \to (A_1,A_1^+)$ be the
rational localization corresponding to $(R,R^+) \to (R_1,R_1^+)$ via Theorem~\ref{T:perfectoid rational}.
Applying Lemma~\ref{L:lift finite etale perfectoid}, we may lift $S_1$ to $B_1 \in \FEt(A_1)$.
By Lemma~\ref{L:henselian direct limit}(a) again, by replacing $(A,A^+) \to (A_1,A_1^+)$ by another rational localization
encircling $\delta$, we can ensure that $B_1 \cong B \otimes_A A_1$. In particular, $B \otimes_A A_1$
is perfectoid.

By compactness, we obtain a strong rational covering $\{(A,A^+) \to (A_i,A_i^+)\}_i$
corresponding to a strong rational covering $\{(R,R^+) \to (R_i, R_i^+)\}$
as in Theorem~\ref{T:perfectoid rational},
such that for each $i$, $B \otimes_A A_i$ corresponds to some $S_i \in \FEt(R_i)$
as in Theorem~\ref{T:perfectoid ring}.
Put $(A_{ij}, A_{ij}^+) = (A_i,A_i^+) \widehat{\otimes}_{(A,A^+)} (A_j,A_j^+)$, so that $(A,A^+) \to (A_{ij}, A_{ij}^+)$ is the rational localization
corresponding to $\Spa(A_i,A_i^+) \cap \Spa(A_j,A_j^+)$ by Proposition~\ref{P:perfectoid tensor}.
Let $(R,R^+) \to (R_{ij}, R_{ij}^+)$ be the corresponding rational localization.
Using Theorem~\ref{T:perfectoid rational}, we may transfer the isomorphisms
$(B \otimes_A A_i) \otimes_{A_i} A_{ij} \cong (B \otimes_A A_j) \otimes_{A_j} A_{ij}$
to obtain isomorphisms $S_i \otimes_{R_i} R_{ij} \cong S_j \otimes_{R_j} R_{ij}$
satisfying the cocycle condition. By
Theorem~\ref{T:henselian direct limit3}, we may glue the $S_i$ to obtain $S \in \FEt(R)$.

Apply Lemma~\ref{L:lift finite etale perfectoid} to lift $S$ to $C \in \FEt(A)$.
By Theorem~\ref{T:perfectoid ring}, we have isomorphisms $B \otimes_A A_i \cong C \otimes_A A_i$
which again satisfy the cocycle condition. They thus glue to an isomorphism $B \cong C$ by Theorem~\ref{T:henselian direct limit3} again,
so $B$ is perfectoid as desired.
\end{proof}
As for fields (see Proposition~\ref{P:perfectoid field descent}),
we have the following converse result.
\begin{prop} \label{P:perfectoid ring descent}
Let $A \to B$ be a morphism of uniform Banach algebras
over $\Qp$ such that $B$ is perfectoid.
\begin{enumerate}
\item[(a)]
If $A \to B$ is faithfully finite \'etale, then $A$ is perfectoid.
\item[(b)]
Let $k \subseteq \ell$ be perfect fields of characteristic $p$ and put $K = W(k)[p^{-1}], L = W(\ell)[p^{-1}]$. Suppose that $A$ is a Banach algebra over $K$ and that $B = (A \widehat{\otimes}_K L)^u$. Then $A \widehat{\otimes}_K L = B$ and $A$ is perfectoid.
\end{enumerate}
\end{prop}
\begin{proof}
Equip all uniform Banach rings in this argument with their spectral norms.
To prove (a), we may assume that $B$ is finite \'etale over $A$ of constant degree $d>0$.
Let $(S,J)$ be the pair corresponding to $B$ via Theorem~\ref{T:perfectoid ring}. Since the two natural morphisms $\iota_1, \iota_2: B \to B \otimes_A B$ are both finite \'etale, by Theorem~\ref{T:mixed lift ring}, on one hand $B \otimes_A B$ is perfectoid; on the other hand, if we put $T = R(B \otimes_A B)$, then
$\iota_1, \iota_2$ correspond to two finite \'etale morphisms $\overline{\iota}_1, \overline{\iota}_2: S \to T$.

Since $\Spec(B) \to \Spec(A)$ is \'etale, the surjection $B \otimes_{\ZZ} B \to B \otimes_A B$ defines a closed immersion
$\Spec B \times_{\Spec A} \Spec B \to \Spec B \times_{\Spec \ZZ} \Spec B$ which is a finite \'etale equivalence relation on $\Spec(B)$ over $\Spec(\ZZ)$. This transfers to a finite \'etale equivalence relation on $\Spec(S)$ over $\Spec(\ZZ)$. However, by \cite[Tag 07S5]{stacks-project}, any such equivalence relation on an affine scheme admits a quotient in the category of schemes; that is, for $R$ the equalizer of $\overline{\iota}_1, \overline{\iota}_2$, the morphism $R \to S$ is finite \'etale of constant degree $d$ and the induced map $S \otimes_R S \to T$ is an isomorphism.

Since $R$ is the equalizer of $\overline{\iota}_1, \overline{\iota}_2: R \to S$, $W(R)$ is the equalizer of $W(\overline{\iota}_1), W(\overline{\iota}_2): W(R) \to W(S)$.
Consequently, the image of $W(\gotho_R)$ under $\theta$ belongs to the equalizer of $\iota_1, \iota_2$, which is $A$.

For $x \in \gotho_A$, the image of $x$ in $\gotho_B$ belongs to the equalizer of $\iota_1$ and $\iota_2$. Consequently, the image of $x$ in $\gotho_S/(p)$ lifts to an element $\overline{x} \in \gotho_S$ for which $\left| \overline{\iota}_1(x) - \overline{\iota}_2(x) \right| \leq p^{-1}$. By Remark~\ref{R:perfect uniform strict},
for any $\epsilon > 0$
there exists $\overline{y} \in \gotho_R$ satisfying
$\left| \overline{y} - \overline{x} \right| \leq p^{-1+\epsilon}$.
If we put $y = \theta([\overline{y}^{1/p}])$, then
$y \in \gotho_A$ satisfies $\left| x - y^p \right| \leq p^{-1+\epsilon}$. It follows that $\overline{\varphi}: \gotho_A/(p) \to \gotho_A/(p)$ is surjective.

To prove that $A$ is perfectoid, it now suffices to produce $x \in \gotho_A$ with $x^p \equiv p \pmod{p^2 \gotho_A}$. To do this, choose an integer $n > d$ and choose integers $a,b$ satisfying $pa + bd/p^n = 1$. Let $y = [\overline{y}] + p y_1$ be a generator of $J$, and put
\[
x_0 = p^a \theta([\Norm_{S/R} (-\overline{y}/\overline{y}_1)^{bp^{-n-1}}]);
\]
then $x_0 \in A$ and $\left| x_0^p - p \right| \leq p^{-1-p^{-n}}$. By the previous paragraph, we can find $x_1 \in \gotho_A$ such that
$\left| x_1^p - x_0^p/p \right| \leq p^{-1}$; we may then take
$x = x_0/x_1$. This proves (a).

To prove (b), we follow the proof of Lemma~\ref{L:unramified uniform extension}.
Let $S$ be a basis of $\ell$ over $k$ containing 1; then each $y \in \gotho_B$ has a unique convergent representation as $\sum_{s \in S} y_s \otimes [s]$ with $y_s \in A$,
and $\left| y \right| = \max\{\left| y_s \right|: s \in S\}$.
In particular,
$B = A \widehat{\otimes}_K L$; more precisely,
\[
\gotho_B/p \gotho_B \cong \gotho_A/p\gotho_A \otimes_{W(k)/pW(k)} W(\ell)/pW(\ell);
\]
since $W(\ell)/pW(\ell) \cong \ell$ is flat over $W(k)/pW(k) \cong k$,
the surjectivity of $\overline{\varphi}$ on $\gotho_A/p\gotho_A$ follows from the corresponding property on $\gotho_B/p\gotho_B$ by faithfully flat descent.
To find $x \in \gotho_A$ with $x^p \equiv p \pmod{p^2 \gotho_A}$,
first choose $y \in \gotho_B$ with $y^p \equiv p \pmod{p^2 \gotho_B}$.
Then $\left| y_s \right| \leq p^{-1/p}$ for all $s \in S$, so
\[ 
\sum_{s \in S} y_s^p \otimes [s^p] \equiv p \pmod{p^2 \gotho_B}.
\]
Since $k$ and $\ell$ are perfect, $\{s^p: s \in S\}$
is also a basis of $\ell$ over $k$, so we may take $x = y_1$ to deduce (b).
\end{proof}

\begin{remark} \label{R:perfectoid spaces}
Scholze observes \cite{scholze1} that for $A$ and $(R,I)$ corresponding as in Theorem~\ref{T:perfectoid ring},
Theorem~\ref{T:perfectoid rational}
and Theorem~\ref{T:mixed lift ring} imply
that the small \'etale sites of the adic spaces associated to $R$ and $A$
are naturally equivalent. This observation is the point of departure of his theory of
\emph{perfectoid spaces}, which casts the aforementioned results in more geometric terms; we will make contact with this construction in \S\ref{sec:adic}.

Besides the expected consequences for relative $p$-adic Hodge theory, as in the study of
relative comparison isomorphisms between \'etale and de Rham cohomology \cite{scholze2} and in our own work in this paper,
the perfectoid correspondence are some unexpected consequences. For instance, Scholze uses it to derive some new cases
of the weight-monodromy conjecture in \'etale cohomology \cite{scholze1}. There may also be consequences
in the direction of Hochster's direct summand conjecture in commutative algebra, as in the work of Bhatt
\cite{bhatt} (see also the discussion in \cite{gabber-ramero-arxiv}).
For further discussion, see \cite{scholze-icm}.
\end{remark}

\begin{remark} \label{R:almost later}
While the correspondence described above is sufficient for some applications, relative $p$-adic Hodge theory tends to requires
somewhat more refined information. The most common approach to getting this extra information is through variants of the
\emph{almost purity theorem} of Faltings \cite{faltings-purity1, faltings-almost}.
We will instead take an alternative approach based on relative Robba rings,
as introduced in \S\ref{sec:relative extended}.
For the relationship between the two points of view, see \S\ref{subsec:almost purity}.
\end{remark}

We mention the characterization of the perfectoid correspondence used in
\cite{scholze1}, which avoids any reference to Witt vectors, as well as a related criterion not found in \cite{scholze1}.

\begin{prop} \label{P:power inverse limit as set}
The following statements are true.
\begin{enumerate}
\item[(a)]
For $A$ and $(R,I)$ corresponding as in Theorem~\ref{T:perfectoid ring}, there is an isomorphism
\[
R \to \varprojlim_{x \mapsto x^p} A, \qquad \overline{x} \mapsto (\theta([\overline{x}]), \theta([\overline{x}^{1/p}]), \dots).
\]
\item[(b)]
For $A$ perfectoid, $\varprojlim_{x \mapsto x^p} A$ generates a dense $\ZZ$-subalgebra of $A$.
\item[(c)]
If $A$ is a uniform Banach algebra over some perfectoid algebra $B$
and $\varprojlim_{x \mapsto x^p} A$ generates a dense $B$-subalgebra of $A$, then $A$ is perfectoid.
\end{enumerate}
\end{prop}
Note that the converse of (b) would imply (c), but such a converse would require a different proof technique than that of (c).
\begin{proof}
Put $S = \varprojlim_{x \mapsto x^p}A$. Let $\alpha$ be the norm on $S$ which is defined to be the norm of the first element in the inverse limit.
Part (a) is immediate from Remark~\ref{R:full power inverse limit}.
To prove (b), note that if $A$ is perfectoid, then finite sums of images of
$S$ form a dense subring of $A$.
To prove (c), note that the quotient of $B\{t_s/\alpha(s): s \in S\}$ by the closure of the ideal $(t_s^p - t_{s^p}: s \in S)$ is a perfectoid algebra mapping to $A$ with dense image. 
By Theorem~\ref{T:perfectoid dense image}, $A$ is perfectoid.
\end{proof}

The following argument is a slight modification of an argument of Colmez \cite[Proposition~4.8]{scholze2}.
\begin{lemma} \label{L:perfectoid neighborhood}
Let $A$ be a uniform Banach algebra over $\Qp$.
Let $U$ be a subset of $\gothm_A$ which generates a dense $\Qp$-subalgebra of $A$.
For each finite subset $T$ of $U$ and each nonnegative integer $n$, equip the finite \'etale $A$-algebra
\[
B_{T,n} =A[x_{t,n}: t \in T]/(x_{t,n}^{p^n} - 1-t: t \in T)
\otimes_{\Qp} \Qp(\mu_{p^n})
\]
with the spectral seminorm. View these as a directed system running over $T$ and $n$ by identifying $x_{t,n}$ with $x_{t,n+1}^p$.
Then the completed direct limit $B$ of the $B_{T,n}$
is a perfectoid uniform Banach algebra over $\Qp$ and the morphism $\calM(B) \to \calM(A)$ is surjective; consequently (by Remark~\ref{R:transform uniform}), the map $A \to B$ is isometric.
\end{lemma}
\begin{proof}
Let $F$ be the completed direct limit of $\bigcup_{n=0}^\infty \Qp(\mu_{p^n})$; this is a perfectoid field, as then is the completion $C$ of
$F[x_t^{p^{-n}}: t \in U]$ for the Gauss norm.
However, $C$ admits a bounded homomorphism to $B$ taking $x_t$ to $1+t$. Since this homomorphism has dense image, by
Theorem~\ref{T:perfectoid dense image}(b) $B$ is perfectoid. The surjectivity of $\calM(B) \to \calM(A)$ follows by viewing $\calM(B)$ as the inverse limit of the $\calM(B_{T,n})$ and applying Lemma~\ref{L:finite etale surjective spectrum}.
\end{proof}

\begin{remark} \label{R:locally perfectoid is perfectoid}
By analogy with Remark~\ref{R:locally perfect is perfect}, one may ask whether an adic Banach algebra $(A,A^+)$ admitting a rational covering by perfectoid algebras is itself perfectoid. For Banach algebras over a perfectoid field, this conjecture is made in 
\cite[Conjecture~2.16]{scholze1}, but one can exhibit a counterexample against this conjecture by adapting the construction of \cite[Proposition~13]{buzzard-verberkmoes}.

By analogy with Proposition~\ref{P:locally perfect is perfect}, one can ask whether one can salvage the conjecture by adding the assumption that $(A,A^+)$ is uniform or even stably uniform. This is quite unclear: to analogize the proof of Proposition~\ref{P:locally perfect is perfect},
one would need to show that an exact sequence as in Corollary~\ref{C:principal ideal closed} is not only strict (as would follow from the open mapping theorem), but 
strict with suitably small factors between the quotient and subspace norms. 
(In other words, one needs to control the cohomology of the integral structure sheaf.)
\end{remark}

\begin{remark} \label{R:Fontaine perfectoid}
In \cite{fontaine-bourbaki}, Fontaine defines a \emph{perfectoid algebra}
to be a uniform f-adic ring $A$ containing a topologically nilpotent unit $\varpi$ such that $\varpi^p$ divides $p$ in $A^\circ$ (so in particular $p$ is topologically nilpotent) and $\overline{\varphi}: A^\circ/(\varpi) \to A^\circ/(\varpi^p)$ is surjective. By Proposition~\ref{P:perfectoid formulations}(e), a uniform Banach algebra over $\Qp$ is perfectoid in our sense if and only if it is perfectoid in Fontaine's sense. However, Fontaine's definition includes perfect uniform Banach $\Fp$-algebras as a special case, as well as some rings which are not algebras over either $\Fp$ or $\Qp$.

Although we will not do so in this paper, it is not difficult to extend the arguments used here to cover perfectoid algebras and spaces in the sense of Fontaine, modulo the formal adjustment of working with topological rings rather than Banach rings: dropping the requirement of working over $\Qp$ means that not every continuous homomorphism is bounded (Remark~\ref{R:Tate f-adic}). One important nonformal change is to replace
Lemma~\ref{L:perfectoid neighborhood} by an alternate construction using lifts of Artin-Schreier extensions instead of Kummer extensions; such a construction has recently been suggested by Scholze, but we will not include the details here.
\end{remark}

\subsection{Preperfectoid and relatively perfectoid algebras}

We next consider some Banach algebras closely related to perfectoid algebras. Although these do not correspond directly to objects in characteristic $p$, they are close enough to perfectoid algebras to inherit some of their most useful properties.

\begin{defn}
Consider the following properties of a uniform Banach algebra $A$ over $\Qp$.
\begin{enumerate}
\item[(a)]
For some perfectoid field $K$, $A \widehat{\otimes}_{\Qp} K$ is uniform and perfectoid.
\item[(b)]
For every perfectoid field $K$, $A \widehat{\otimes}_{\Qp} K$ is uniform and perfectoid.
\item[(c)]
For every perfectoid algebra $B$, $A \widehat{\otimes}_{\Qp} B$ is uniform and perfectoid.
\end{enumerate}
In case (a), we say that $A$ is \emph{preperfectoid}; this definition and some examples are due to Scholze and Weinstein \cite{scholze-weinstein}.
In case (b), we say that $A$ is \emph{strongly preperfectoid}. In case (c), we say that  $C$ is \emph{relatively perfectoid}. Note that perfectoid algebras are generally not preperfectoid; see Remark~\ref{R:uniformity unstable} for a typical example.

We say that a uniform adic Banach algebra $(A,A^+)$ is \emph{preperfectoid}, \emph{strongly preperfectoid}, or \emph{relatively perfectoid} if $A$ has the corresponding property.
\end{defn}

A typical example is the following analogue of Example~\ref{exa:perfectoid Tate algebra}.
\begin{example}
For any $r_1,\dots,r_n > 0$,
the completion of $\Qp\{T_1/r_1,\dots,T_n/r_n\}[T_1^{p^{-\infty}}, \dots, T_n^{p^{-\infty}}]$
for the weighted Gauss norm is relatively perfectoid.
\end{example}

Although we cannot prove that a preperfectoid algebra is relatively perfectoid, we do have the following result.
\begin{prop} \label{P:preperfectoid to relatively perfectoid}
Let $A$ be a preperfectoid Banach algebra. Then for any uniform Banach algebra $B$ over $\Qp$,
$A \widehat{\otimes}_{\Qp} B$ is uniform.
\end{prop}
\begin{proof}
Choose a perfectoid field $K$ such that $A \widehat{\otimes}_{\Qp} K$ is perfectoid.
Apply Lemma~\ref{L:perfectoid neighborhood} to construct an isometric morphism $B \to C$ of uniform Banach algebras with $C$ perfectoid.
Put $D = (K \widehat{\otimes}_{\Qp} C)^u$; it is perfectoid by 
Corollary~\ref{C:perfectoid tensor product over any base}.
By Proposition~\ref{P:perfectoid tensor}, $A \widehat{\otimes}_{\Qp} D = (A \widehat{\otimes}_{\Qp} K) \widehat{\otimes}_K D$ is perfectoid.
Since $C$ and $D$ are uniform and $\calM(D) \to \calM(C)$ is surjective by Lemma~\ref{L:fibre product}, by Remark~\ref{R:transform uniform} the bounded homomorphism $C \to D$ is isometric. Therefore the composition $B \to C \to D$ is isometric; by Lemma~\ref{L:inject tensor}, it follows that the map $A \widehat{\otimes}_{\Qp} B \to A \widehat{\otimes}_{\Qp} D$ is a strict injection. This yields the claim.
\end{proof}

\begin{theorem} \label{T:Kiehl for preperfectoid}
Let $(A,A^+)$ be a preperfectoid adic Banach algebra. 
Then any rational localization of $(A,A^+)$ is again preperfectoid.
In particular, $(A,A^+)$ is stably uniform,
sheafy (by Theorem~\ref{T:uniform rational is sheafy}),
and satisfies the Tate sheaf and Kiehl glueing properties
(by Theorem~\ref{T:tate to Kiehl}).
\end{theorem}
\begin{proof}
Choose a perfectoid field $K$ such that $A_K = A \widehat{\otimes}_{\Qp} K$ is perfectoid. By Theorem~\ref{T:perfectoid rational}, $A_K$ is stably uniform; by
Remark~\ref{R:check uniform by analytic}, so is $A$.
\end{proof}

\begin{prop}
Let $A$ be a preperfectoid Banach algebra, and view $B \in \FEt(A)$ as a Banach algebra as per Proposition~\ref{P:finite etale Banach norm}. Then $B$ is again preperfectoid.
\end{prop}
\begin{proof}
Choose a perfectoid field $K$ such that $A_K = A \widehat{\otimes}_{\Qp} K$ is perfectoid. By the open mapping theorem (Theorem~\ref{T:open mapping}), the 
tensor product norm on $B_K= B \widehat{\otimes}_{\Qp} K$ is equivalent to the norm obtained by viewing $B_K$ as an object of $\FEt(A_K)$ and applying Proposition~\ref{P:finite etale Banach norm}. We may thus conclude by
applying Theorem~\ref{T:mixed lift ring} to deduce that $B_K$ is perfectoid.
\end{proof}

\begin{prop} \label{P:perfectoid tensor product}
Let $(A_1, A_1^+) \to (A_2, A_2^+)$, $(A_1, A_1^+) \to (A_3, A_3^+)$ be morphisms of strongly preperfectoid (resp.\ relatively perfectoid) adic Banach algebras over $\Qp$. Then the tensor product $(A_2, A_2^+) \widehat{\otimes}_{(A_1, A_1^+)} (A_3, A_3^+)$ is strongly preperfectoid (resp.\ relatively perfectoid).
\end{prop}
\begin{proof}
This is immediate from Proposition~\ref{P:perfectoid tensor}. Note that the argument does not work in the preperfectoid case because we need to know that the same perfectoid field can be used for all three rings.
\end{proof}

\section{Robba rings and \texorpdfstring{$\varphi$}{phi}-modules}

An important feature of the approach to $p$-adic Hodge theory used in this series of
papers (and in the work of Berger and others) is the theory of slopes of Frobenius actions
on modules over certain rings. This bears some resemblance to the theory of slopes for
vector bundles on Riemann surfaces, including the relationship of those slopes to unitary
representations of fundamental groups.
(A more explicit link to vector bundles appears in the work of Fargues
and Fontaine \cite{fargues-fontaine}; see \S\ref{subsec:vector bundles}.)
We review here some of the principal results of the first author which are pertinent to $p$-adic
Hodge theory, mostly omitting proofs. Besides serving as a model,
some of these results provide key inputs into our work on relative $p$-adic Hodge theory.

\subsection{Slope theory over the Robba ring}
\label{subsec:slope Robba}

We begin by introducing several key rings used in $p$-adic Hodge theory, and the basic
theory of slopes of Frobenius actions. Our description here is rather minimal;
see \cite{kedlaya-relative} for a more detailed discussion.

\begin{hypothesis} \label{H:slope Robba}
Throughout \S\ref{subsec:slope Robba}, put $K = \Frac(W(k))$ for some perfect field
$k$ of characteristic $p$. Equip $K$ with the $p$-adic norm and the Frobenius
lift $\varphi_K$ induced by Witt vector functoriality.
Fix also a choice of $\omega \in (0,1)$.
\end{hypothesis}

\begin{defn} \label{D:Robba rings}
For $r>0$, put
\[
\calR_K^r = \left\{\sum_{i \in \mathbb{Z}} c_i T^i: c_i \in K, \lim_{i \to \pm \infty} |c_i| \rho^i = 0
\quad (\rho \in [\omega^r,1)) \right\}.
\]
In other words, $\calR_K^r$ consists of formal sums $\sum_{i \in \mathbb{Z}} c_i T^i$
in the indeterminate $T$ with coefficients in $K$ which
converge on the annulus $\omega^r \leq |T| < 1$.
The set $\calR_K^r$ forms a ring under formal series addition and multiplication;
let $\calR_K^{\inte,r}$ be the subring of $\calR_K^r$ consisting of series whose coefficients
have norm at most 1, and put $\calR_K^{\bd,r} = \calR_K^{\inte,r}[p^{-1}]$. Put
\[
\calR_K^{\inte} = \bigcup_{r>0} \calR_K^{\inte,r}, \quad
\calR_K^{\bd} = \bigcup_{r>0} \calR_K^{\bd,r}, \quad
\calR_K = \bigcup_{r>0} \calR_K^r.
\]
The ring $\calR_K^{\inte}$ is a
local ring with residue field $k(( \overline{T}))$
which is not complete but is henselian.
(See for instance \cite[Lemma~3.9]{kedlaya-annals}. For a similar argument, see Proposition~\ref{P:perfect henselian}.)
The completion of the field $\calR_K^{\bd}$ is the field
\[
\calE_K = \left\{\sum_{i \in \mathbb{Z}} c_i T^i: c_i \in K, \sup_i \{|c_i|\} < +\infty,
\lim_{i \to - \infty} |c_i| = 0 \right\}.
\]
The units of $\calR_K$ are precisely
the nonzero elements of $\calR_K^{\bd}$, as may be seen
by considering Newton polygons \cite[Corollary~3.23]{kedlaya-annals}.
\end{defn}

\begin{defn} \label{D:formal to convergent}
We will need to consider several different topologies on the rings
described above.
\begin{enumerate}
\item[(a)]
Those rings contained in $\calE_K$ carry both a \emph{$p$-adic topology}
(the metric topology defined by the Gauss norm)
and a \emph{weak topology} (in which a sequence
converges if it is bounded for the Gauss norm and
converges $T$-adically modulo any fixed power of $p$).
For both topologies, $\calE_K$ is complete.
\item[(b)]
Those rings contained in $\calR^r_K$ carry a \emph{Fr\'echet topology},
in which a sequence converges if and only if it converges under
the $\omega^s$-Gauss norm for all $s \in (0,r]$.
For this topology, $\calR^r_K$ is complete.
\item[(c)]
Those rings contained in $\calR_K$ carry a \emph{limit-of-Fr\'echet topology},
or \emph{LF topology}. This topology is defined on $\calR_K$ by taking the
locally convex direct limit (in the sense of \cite[\S II.4]{bourbaki-evt})
of the $\calR^r_K$ (each equipped with the Fr\'echet topology).
In particular, a sequence converges in $\calR_K$ if and only if it
is a convergent sequence in $\calR^r_K$ for some $r>0$.
\end{enumerate}
\end{defn}

\begin{remark} \label{R:formal to convergent}
The convergence of the formal expression $x = \sum_i c_i T^i$ for various of
the topologies described in Definition~\ref{D:formal to convergent}
is useful for defining operations such as Frobenius lifts
(see Definition~\ref{D:slopes} below). In $\calE_K$, the formal sum
converges for the weak topology but not the $p$-adic topology.
In $\calR^r_K$, the sum converges for the Fr\'echet topology.
In $\calR_K$, the sum converges for the LF topology.
\end{remark}

\begin{remark} \label{R:bounded weak}
Note that a sequence of elements of $\calR_K^{\bd,r}$
which is $p$-adically bounded and convergent under the
$\omega^r$-Gauss norm also converges in the weak topology.
\end{remark}

\begin{defn} \label{D:slopes}
A \emph{Frobenius lift} $\varphi$ on $\calR_K$ is an endomorphism defined by the formula
\[
\varphi\left( \sum_{i \in \ZZ} c_i T^i \right) = \sum_{i \in \ZZ} \varphi_K(c_i) u^i
\]
for some $u \in \calR_K^{\inte}$ with $|u - T^{p}| < 1$,
where the right side may be interpreted as a convergent sum using
Remark~\ref{R:formal to convergent}.
Such an endomorphism also acts on $\calR_K^{\inte},
\calR_K^{\bd}, \calE_K$, but not on
$\calR_K^r$ for any individual $r>0$; rather, for $r>0$ sufficiently small,
$\varphi$ carries $\calR_K^r$ into $\calR_K^{r/p}$.
The action of $\varphi$ is continuous for each of the topologies
described in Definition~\ref{D:formal to convergent}.

Choose a Frobenius lift $\varphi$ on $\calR_K$.
For $R \in \{\calR_K^{\inte}, \calR_K^{\bd}, \calR_K, \gotho_{\calE_K},\calE_K\}$,
a \emph{$\varphi$-module} over $R$ is a finite free $R$-module $M$
equipped with a semilinear $\varphi$-action
(i.e., an $R$-linear isomorphism $\varphi^* M \to M$).
Since the action of $\varphi$ takes any basis of $M$
to another basis, the $p$-adic valuation of the matrix
via which $\varphi$ acts on a basis of $M$ is both finite and independent of the choice of the basis.
We call the negative of this quantity the \emph{degree} of $M$, denoted $\deg(M)$. For $M$ nonzero, we define the
\emph{slope} of $M$ to be $\mu(M) = \deg(M)/\rank(M)$.

For $s \in \QQ$, we say $M$ is \emph{pure of slope $s$}
if for some (hence any) $c,d \in \ZZ$ with $d>0$ and $c/d = s$,
$p^{c} \varphi^d$ acts on $M$ via a matrix $U$ such that the entries of $U$ and $U^{-1}$
all have Gauss norm at most 1.
This evidently implies $\mu(M) = s$. If $s = 0$,
we also say $M$ is \emph{\'etale}.
(Note that our definitions force any nonzero $\varphi$-module over $\calR_K^{\inte}$ or $\gotho_{\calE_K}$ to be \'etale.)
\end{defn}

\begin{remark} \label{R:dual slopes}
For $M$ a $\varphi$-module over a ring $R$, 
view the dual module $M^\dual = \Hom_R(M,R)$ as a $\varphi$-module
as in Remark~\ref{R:dual semilinear action}.
If $M$ is pure of slope $s$, then
$M^\dual$ is pure of slope $-s$.
\end{remark}

\begin{prop} \label{P:fully faithful1}
For any $s \in \QQ$, we have the following.
\begin{enumerate}
\item[(a)]
The functor $M \rightsquigarrow M \otimes_{\calR_K^{\bd}} \calE_K$ gives a fully faithful functor from $\varphi$-modules over $\calR_K^{\bd}$ which are pure of slope $s$ to $\varphi$-modules
over $\calE_K$ which are pure of slope $s$.
\item[(b)]
The functor $M \rightsquigarrow M \otimes_{\calR_K^{\bd}} \calR_K$ gives an equivalence of categories
between $\varphi$-modules over $\calR_K^{\bd}$ which are pure of slope $s$ and $\varphi$-modules
over $\calR_K$ which are pure of slope $s$.
\end{enumerate}
\end{prop}
\begin{proof}
See \cite[Proposition~1.2.7]{kedlaya-relative} for (a)
and \cite[Theorem~1.6.5]{kedlaya-relative} for (b).
For (b), see also Remark~\ref{R:etale descent}.
\end{proof}

The main theorem about slopes of $\varphi$-modules over $\calR_K$ can be formulated
in several ways. This formulation
asserts that $M$ is pure if and only if $M$ is semistable with respect
to slope in the sense of geometric invariant theory \cite{git}.
\begin{theorem} \label{T:slope filtration1}
Let $M$ be a nonzero
$\varphi$-module over $\calR_K$ with $\mu(M) = s$.
Then $M$ is pure of slope $s$ if and only if there exists no
nonzero proper $\varphi$-submodule of $M$ of slope greater than $s$.
\end{theorem}
\begin{proof}
See \cite[Theorem~1.7.1]{kedlaya-relative}.
\end{proof}
An essentially equivalent formulation, incorporating an analogue of the Harder-Narasimhan
filtration for vector bundles, is the following.
(See Remark~\ref{R:Harder-Narasimhan} for further discussion of this analogy.)
\begin{theorem} \label{T:slope filtration explicit1}
Let $M$ be a nonzero $\varphi$-module over $\calR_K$.
Then there exists a unique filtration $0 = M_0 \subset \cdots \subset M_l = M$ by
saturated $\varphi$-submodules, such that $M_1/M_0, \dots, M_l/M_{l-1}$ are pure $\varphi$-modules and
$\mu(M_1/M_0) > \cdots > \mu(M_l/M_{l-1})$.
\end{theorem}
\begin{proof}
See again \cite[Theorem~1.7.1]{kedlaya-relative}.
\end{proof}

\begin{remark}
Beware that Proposition~\ref{P:fully faithful1}
does not imply anything about maps between $\varphi$-modules which are pure
of different slopes. For instance, it is common in $p$-adic Hodge theory
(in the study of trianguline representations; see for instance \cite{colmez-trianguline})
to encounter
short exact sequences of the form $0 \to M_1 \to M \to M_2 \to 0$ of
$\varphi$-modules over $\calR_K$
in which
for some positive integer $m$, $M_1$ has rank 1 and slope $-m$,
$M_2$ has rank 1 and slope $m$, and $M$ is \'etale.
While each term individually descends uniquely to $\calR_K^{\bd}$ by
Proposition~\ref{P:fully faithful1}(b), the sequence cannot descend
because there are no nonzero maps
between pure $\varphi$-modules over $\calR_K^{\bd}$ of different slopes.
(Note that $\mu(M_1) < \mu(M)$, so there is no contradiction with
Theorem~\ref{T:slope filtration1}.)
\end{remark}

\begin{remark}
The sign convention used here for degrees of $\varphi$-modules is opposite to that used
in previous work of the first author
\cite{kedlaya-annals, kedlaya-revisited, kedlaya-relative, kedlaya-course}.
We have changed it in order to match the sign convention used in geometric invariant theory
\cite{git},
in which the ample line bundle $\calO(1)$ on any projective space has degree $+1$.
This choice of sign also creates agreement with the work of Hartl and Pink \cite{hartl-pink}
and of Fargues and Fontaine \cite{fargues-fontaine}.
\end{remark}

\begin{convention} \label{conv:power slope}
It will be convenient at several points to speak also of $\varphi^d$-modules for $d$ an arbitrary
positive integer. We adopt the convention that the degree of a $\varphi^d$-module is defined by
computing the $p$-adic valuation of the determinant of the matrix on which $\varphi^d$ acts on some basis,
then dividing by $d$. This has the advantage that the degree is preserved upon restriction of a
$\varphi^d$-module to a $\varphi^{de}$-module. (One can even replace $\varphi$ with a map on $\tilde{\calR}_K$
lifting the
$p^d$-power absolute Frobenius on $k((\overline{T}))$, not necessarily given by raising a
$p$-power Frobenius lift to the $d$-th power, by modifying Definition~\ref{D:slopes} in the obvious way.
We will not need this extra generality.)
\end{convention}

\subsection{Slope theory and Witt vectors}
\label{subsec:slopes}

One can generalize the slope theory for Frobenius modules over the Robba ring
by first making explicit the role of the $\overline{T}$-adic norm on the
residue field of $\calR_K^{\bd}$, then replacing this norm
with something more general.
This second step turns out to be a bit subtle unless we first perfect the
residue field of
$\calR_K^{\bd}$ and pass to Witt vectors (as in Definition~\ref{D:perfection});
this gives a slope theory introduced in \cite{kedlaya-revisited} and
reviewed here. In fact, this study is integral to the slope theory
over the Robba ring itself; see Remark~\ref{R:pure semistable}.
We take up the relative version of this story starting in \S\ref{sec:relative extended}.

\begin{hypothesis}
Throughout \S\ref{subsec:slopes}, let $L$ be a perfect analytic field of
characteristic $p$ with norm $\alpha$.
\end{hypothesis}

\begin{defn} \label{D:extended Robba}
For $r > 0$, let $\tilde{\calR}^{\inte,r}_L$ be the set of $x
= \sum_{i=0}^\infty p^i [\overline{x}_i] \in W(L)$ for which
$\lim_{i \to \infty} p^{-i} \alpha(\overline{x}_i)^r = 0$.
Thanks to the homogeneity property of Witt vector addition (Remark~\ref{R:addition formula}),
this set forms a ring on which the formula
\[
\lambda(\alpha^s)(x) = \max_i \{p^{-i} \alpha(\overline{x}_i)^s\}
\]
defines a multiplicative norm $\lambda(\alpha^s)$ on $\tilde{\calR}^{\inte,r}_L$ for each $s \in [0,r]$.
(For an explicit argument, see Proposition~\ref{P:mu multiplicative}.)

Put $\tilde{\calR}^{\bd,r}_L = \tilde{\calR}^{\inte,r}_L[p^{-1}]$.
Let $\tilde{\calR}^r_L$ be the Fr\'echet completion of $\tilde{\calR}^{\bd,r}_L$ under the norms
$\lambda(\alpha^s)$ for $s \in (0,r]$.
Put
\[
\tilde{\calR}^{\inte}_L = \cup_{r>0} \tilde{\calR}^{\inte,r}_L, \quad
\tilde{\calR}^{\bd}_L = \cup_{r>0} \tilde{\calR}^{\bd,r}_L, \quad
\tilde{\calR}_L = \cup_{r>0} \tilde{\calR}^r_L.
\]
Again, $\tilde{\calR}^{\inte}_L$ is an incomplete
but henselian (see \cite[Lemma~2.1.12]{kedlaya-revisited}
or Proposition~\ref{P:perfect henselian})
local ring with residue field $L$;
the completion of the field $\tilde{\calR}^{\bd}_L$ is simply
$\tilde{\calE}_L = W(L)[p^{-1}]$.
(For paralellism, we write $\tilde{\calE}_L^{\inte}$ as another notation for $W(L)$.)
One can again identify the units of $\tilde{\calR}_L$ as the
nonzero elements of $\tilde{\calR}_L^{\bd}$; see Corollary~\ref{C:units}.

We call $\tilde{\calR}_L$ the
\emph{extended Robba ring} with residue field $L$. This
terminology is not used in \cite{kedlaya-revisited}, but is suggested in
\cite{kedlaya-relative}.
\end{defn}

\begin{lemma} \label{L:hadamard}
For $x \in \tilde{\calR}^r_L$, the function
$s \mapsto \log \lambda(\alpha^s)(x)$ is continuous and convex on $(0,r]$.
\end{lemma}
\begin{proof}
The function is affine in case $x = p^i [\overline{x}]$ for some
$i \in \ZZ, \overline{x} \in L$. In case $x$ is a finite sum of such terms,
the function is the maximum of finitely many affine functions, and hence is convex.
Since such finite sums are dense in $\tilde{\calR}^r_L$, the general case follows.
\end{proof}

\begin{lemma} \label{L:bounded}
For $x \in \tilde{\calR}_L$, we have $x \in \tilde{\calR}^{\bd}_L$
if and only if
for some $r>0$, $\lambda(\alpha^s)(x)$ is bounded for $s \in (0,r]$.
\end{lemma}
\begin{proof}
If $x \in \tilde{\calR}^{\bd}_L$, then as $s$ tends to 0,
$\lambda(\alpha^s)(x)$ tends to the $p$-adic norm of $x$.
Hence for some $r>0$, $\lambda(\alpha^s)(x)$ is bounded for $s \in (0,r]$.
Conversely, suppose that
for some $r>0$, $\lambda(\alpha^s)(x)$ is bounded for $s \in (0,r]$.
To prove that $x \in \tilde{\calR}^{\bd}_L$, we may first multiply by
a power of $p$; we may thus ensure that for some $r>0$,
$x \in \tilde{\calR}^r_L$ and $\lambda(\alpha^s)(x) \leq 1$ for $s \in (0,r]$.
We will show in this case that $x \in \tilde{\calR}^{\inte,r}_L$.

Write $x$ as the limit of a sequence $x_0,x_1, \dots$
with $x_i \in \tilde{\calR}^{\bd,r}_L$.
For each positive integer $j$, we can find $N_j>0$ such that
\[
\lambda(\alpha^s)(x_i-x) \leq p^{-j} \qquad (i \geq N_j,\, s \in [p^{-j}r,r]).
\]
Write $x_i = \sum_{l=m(i)}^\infty p^l [\overline{x}_{il}]$, and put
$y_i = \sum_{l=0}^\infty p^l [\overline{x}_{il}] \in \tilde{\calR}^{\inte,r}_L$. For $i \geq N_j$,
we have $\lambda(\alpha^{p^{-j}r})(x_i) \leq 1$ and so
\[
\alpha(\overline{x}_{il}) \leq p^{lp^j/r}
\qquad (i \geq N_j, \,l < 0).
\]
Since $p^{-l} p^{lp^j} \leq p^{1-p^j}$ for $l \leq -1$,
we deduce that $\lambda(\alpha^r)(x_i - y_i) \leq p^{1-p^j}$.
Consequently, the sequence $y_0,y_1,\dots$ converges to
$x$ under $\lambda(\alpha^r)$, and hence under $\lambda(\alpha^s)$ for $s \in (0,r]$
by Lemma~\ref{L:hadamard}; it follows that $x \in \tilde{\calR}^{\inte,r}_L$
as desired.
(See also \cite[Corollary~2.5.6]{kedlaya-revisited} for a slightly
different argument.)
\end{proof}
\begin{cor} \label{C:units}
The units in $\tilde{\calR}_L$ are precisely the nonzero elements of
$\tilde{\calR}^{\bd}_L$.
\end{cor}
\begin{proof}
Suppose $x \in \tilde{\calR}_L$ is a unit with inverse $y$.
Choose $r>0$ so that $x,y \in \tilde{\calR}^r_L$.
Then the functions $\log \lambda(\alpha^s)(x), \log \lambda(\alpha^s)(y)$
are convex by Lemma~\ref{L:hadamard}, but their sum is the constant function $0$.
Hence both functions are affine in $s$; in particular, $\lambda(\alpha^s)(x)$
is bounded for $s \in (0,r]$.
By Lemma~\ref{L:bounded}, this forces $x \in \tilde{\calR}^{\bd}_L$.
(Compare \cite[Lemma~2.4.7]{kedlaya-revisited}.)
\end{proof}

\begin{lemma} \label{L:Bezout domain}
The rings $\tilde{\calR}^r_L$ and $\tilde{\calR}_L$ are B\'ezout domains, i.e., integral domains in which every
finitely generated ideal is principal.
\end{lemma}
\begin{proof}
See \cite[Theorem~2.9.6]{kedlaya-revisited}.
\end{proof}

\begin{remark} \label{R:Robba PID}
A fact closely related to Lemma~\ref{L:Bezout domain} is that for $0 < s \leq r$, the completion of $\tilde{\calR}^r_L$
with respect to the norm $\max\{\lambda(\alpha^r), \lambda(\alpha^s)\}$ is a principal ideal domain, and even a Euclidean domain.
See \cite[Proposition~2.6.8]{kedlaya-revisited}.
(This ring will later be denoted $\tilde{\calR}^{[s,r]}_L$; see Definition~\ref{D:relative extended Robba}.)
\end{remark}

\begin{defn} \label{D:no common slopes}
To each nonzero $x \in \tilde{\calR}^r_L$ is associated its
\emph{Newton polygon}, i.e., the convex dual of the function $f_x(t) = \log \lambda(\alpha^t)(x)$
(which is convex by Lemma~\ref{L:hadamard}). The slopes of the Newton polygon of $x$ are the
values $t$ where $f_x$ changes slope; for short, we call these the \emph{slopes of $x$}.
The \emph{multiplicity} of a slope $s$ is the width of the corresponding segment of the
Newton polygon. Note that $r$ fails to receive a multiplicity under this definition; to correct
this, choose any $s>r$ and any $y \in \tilde{\calR}^s_L$ with $\lambda(\alpha^r)(x-y) < \lambda(\alpha^r)(x)$,
then define the multiplicity of $r$ as a slope of $x$ to be its multiplicity as a slope of $y$.
This does not depend on the choice of $y$.
(See \cite[Definition~2.4.4]{kedlaya-revisited} for an alternate definition.)

For our present purposes, the most important properties of slopes are the following.
\begin{enumerate}
\item[(a)]
If $x = yz$, then the slopes of $x$ are precisely the slopes of $y$ and $z$. More precisely,
the multiplicity of any $s \in (0,r]$ as a slope of $x$
equals the sum of its multiplicities as a slope of $y$ and of $z$.
\item[(b)]
If $x \in \tilde{\calR}^r_L$ is a unit, then by (a) it has no slopes.
The converse is also true, by the following argument.
If $x$ has no slopes, then $x \in \tilde{\calR}^{\bd,r}_L$
by Lemma~\ref{L:bounded}. Write $x = p^m y$ with $m \in \ZZ$ and $y \in \tilde{\calR}^{\inte,r}_L$
not divisible by $p$; we must then have $\lambda(\alpha^r)(y - [\overline{y}]) < \lambda(\alpha^r)(y)$
by definition of the multiplicity of $r$.
It follows that $x$ is a unit in $\tilde{\calR}^r_L$.
\end{enumerate}
\end{defn}

\begin{defn} \label{D:phi-modules extended}
The Frobenius lift $\varphi$ on $W(L)$ acts on $\tilde{\calR}^{\inte}_L$, and extends by continuity
to $\tilde{\calR}^{\bd}_L, \tilde{\calR}_L, \tilde{\calE}_L$. Note the following
useful identity:
\begin{equation} \label{eq:Frobenius norms}
\lambda(\alpha^s)(\varphi(x)) = \lambda(\alpha^{ps})(x) \qquad (x \in
\tilde{\calR}_L, \,s > 0).
\end{equation}
We define $\varphi$-modules over these rings,
degrees, slopes,
and the pure and \'etale conditions as in Definition~\ref{D:slopes}.
(We also consider $\varphi^d$-modules for $d$ a positive integer, keeping in mind Convention~\ref{conv:power slope}.)
\end{defn}

\begin{lemma} \label{L:big Robba invariants}
We have $\tilde{\calE}_L^{\varphi} = \Qp$ and $\tilde{\calR}_L^{\varphi} = \Qp$.
More generally, for any positive integer $a$, the elements of
$\tilde{\calE}_L$ and $\tilde{\calR}_L$ fixed by $\varphi^a$ constitute the
unramified extension of $\Qp$ with residue field $\FF_{p^a}$.
\end{lemma}
\begin{proof}
The first equality holds because $\tilde{\calE}_L = W(L)[p^{-1}]$,
so $\tilde{\calE}_L^{\varphi} = W(L^{\varphi})[p^{-1}] = \Qp$.
For the second equality, suppose $x \in \tilde{\calR}_L^{\varphi}$ is nonzero;
then by \eqref{eq:Frobenius norms},
$\lambda(\alpha^{s})(x) = \lambda(\alpha^{ps})(x)$ for all $s>0$.
It follows that $\lambda(\alpha^r)(x)$ is bounded over all $r>0$,
and so by Lemma~\ref{L:bounded},
$x \in (\tilde{\calR}_L^{\bd})^{\varphi} \subseteq \tilde{\calE}_L^{\varphi} = \Qp$.
The final assertion is proved similarly.
\end{proof}

We have the following analogue of Proposition~\ref{P:fully faithful1}.
One key difference is that the functor in part (a) can be shown to be essentially surjective; see Theorem~\ref{T:perfect equivalence2}.
\begin{prop} \label{P:fully faithful2}
For any $s \in \QQ$, we have the following.
\begin{enumerate}
\item[(a)]
The functor $M \rightsquigarrow M \otimes_{\tilde{\calR}_L^{\bd}} \tilde{\calE}_L$
gives a fully faithful
functor from $\varphi$-modules over $\tilde{\calR}_L^{\bd}$ which are pure of slope $s$ to $\varphi$-modules
over $\tilde{\calE}_L$ which are pure of slope $s$.
\item[(b)]
The functor $M \rightsquigarrow M \otimes_{\tilde{\calR}_L^{\bd}} \tilde{\calR}_L$ gives an equivalence of categories
between $\varphi$-modules over $\tilde{\calR}_L^{\bd}$ which are pure of slope $s$ and $\varphi$-modules
over $\tilde{\calR}_L$ which are pure of slope $s$.
\end{enumerate}
\end{prop}
\begin{proof}
See \cite[Theorem~6.3.3(a,b)]{kedlaya-revisited}
or Remark~\ref{R:etale descent}.
\end{proof}

We also have analogues of Theorem~\ref{T:slope filtration1}
and Theorem~\ref{T:slope filtration explicit1}.
\begin{theorem} \label{T:slope filtration2}
Let $M$ be a nonzero $\varphi$-module over $\tilde{\calR}_L$ with $\mu(M) = s$.
Then $M$ is pure of slope $s$ if and only if there exists no
nonzero proper $\varphi$-submodule of $M$ of slope greater than $s$.
\end{theorem}
\begin{proof}
See \cite[Proposition~6.3.5, Corollary~6.4.3]{kedlaya-revisited}.
\end{proof}
\begin{theorem} \label{T:slope filtration explicit2}
Let $M$ be a nonzero $\varphi$-module over $\tilde{\calR}_L$.
Then there exists a unique filtration $0 = M_0 \subset \cdots \subset M_l = M$ by
saturated $\varphi$-submodules, such that $M_1/M_0, \dots, M_l/M_{l-1}$ are pure and
$\mu(M_1/M_0) > \cdots >\mu(M_l/M_{l-1})$.
\end{theorem}
\begin{proof}
See \cite[Theorem~6.4.1]{kedlaya-revisited}.
\end{proof}

\begin{cor} \label{C:purity base extension}
Let $M$ be a nonzero $\varphi$-module over $\tilde{\calR}_L$ with $\mu(M) = s$.
Let $L'$ be a perfect analytic field containing $L$ with compatible norms.
Then $M$ is pure of slope $s$
if and only if $M \otimes_{\tilde{\calR}_L} \tilde{\calR}_{L'}$ is pure of
slope $s$.
\end{cor}
\begin{proof}
{}From the definition of purity, it is clear that
if $M$ is pure of slope $s$, then so is
$M \otimes_{\tilde{\calR}_L} \tilde{\calR}_{L'}$. On the other hand,
from the alternate criterion for purity given by
Theorem~\ref{T:slope filtration2}, it is also clear that
if $M$ fails to be pure of slope $s$, then so does
$M \otimes_{\tilde{\calR}_L} \tilde{\calR}_{L'}$.
\end{proof}

In addition, in case $L$ is algebraically closed, we get an analogue of
Manin's classification of rational Dieudonn\'e modules.
\begin{prop} \label{P:DM}
Suppose that $L$ is algebraically closed.
Let $M$ be a $\varphi$-module over $\tilde{\calE}_L$ (resp.\
$\tilde{\calR}^{\bd}_L$) which is pure of
slope $s$.
Then for any $c,d \in \ZZ$ with $d>0$
and $c/d = s$,
and any basis of $M$ on which $p^{c} \varphi^d$ acts via an invertible matrix
over $W(L)$ (resp.\ $\tilde{\calR}^{\inte}_L$), the
$W(L)$-span (resp.\ $\tilde{\calR}^{\inte}_L$-span) of this basis admits another basis
fixed by $p^{c} \varphi^d$.
\end{prop}
\begin{proof}
Both assertions reduce easily to the case $s=0$ provided that we allow $\varphi$ to be replaced by a
power (which does not affect the argument).
The assertion about $\tilde{\calE}_L$ is fairly standard; see
for instance \cite[Proposition~A1.2.6]{fontaine-phigamma}.
The assertion about $\tilde{\calR}^{\bd}_L$ follows from the assertion
about $\tilde{\calE}_L$ as in \cite[Proposition~2.5.8]{kedlaya-relative}.
See also Proposition~\ref{P:DM relative} for a stronger statement.
\end{proof}

One has an analogue of Proposition~\ref{P:DM} for $\varphi$-modules over $\tilde{\calR}_L$
which need not be pure.
\begin{prop} \label{P:DM not pure}
Suppose that $L$ is algebraically closed.
Let $M$ be a $\varphi$-module over $\tilde{\calR}_L$.
Then for some positive integer $d$, there exists a basis of $M$ on which $\varphi^d$
acts via a diagonal matrix with diagonal entries in $p^\ZZ$.
\end{prop}
\begin{proof}
Using Theorem~\ref{T:slope filtration explicit2} and Proposition~\ref{P:DM},
this reduces to the assertion that for any positive integers $c,d$, the map
$x \mapsto x^{\varphi^d} - p^c x$ on $\tilde{\calR}_L$ is surjective.
For this, see \cite[Proposition~3.3.7(c)]{kedlaya-revisited}.
See also \cite[Proposition~4.5.3]{kedlaya-revisited} for a detailed proof of the
original statement.
\end{proof}

\begin{remark}
The use of growth conditions to cut out subrings of the rings of Witt vectors
also appears in the work of Fargues and Fontaine \cite{fargues-fontaine}
with which we make contact later (\S\ref{subsec:vector bundles}),
as well as in the approach to constructing $p$-adic cohomology via the
\emph{overconvergent de Rham-Witt complex} of Davis, Langer, and Zink \cite{davis-langer-zink,
davis-langer-zink2}.
\end{remark}

\begin{remark} \label{R:Harder-Narasimhan}
Theorems~\ref{T:slope filtration2} and~\ref{T:slope filtration explicit2} together mean that the \emph{slope filtration} of a $\varphi$-module $M$, as described in Theorem~\ref{T:slope filtration explicit2}, coincides with the \emph{Harder-Narasimhan filtration} of $M$ in the category of $\varphi$-modules for the degree function
$M \mapsto \mu(\wedge^{\rank(M)} M)$. That is to say, $M_1$ is the maximal nonzero $\varphi$-submodule of $M$ of maximal slope,
$M_2/M_1$ is the maximal nonzero $\varphi$-submodule of $M$ of maximal slope, and so on. This equality implies among other things that the
Harder-Narasimhan filtration is multiplicative (because the tensor product of pure $\varphi$-modules of slopes $s_1, s_2$ is pure of slope $s_1+s_2$).

The analogy with stability of vector bundles will become even more explicit when we relate $\varphi$-modules to vector bundles on the relative Fargues-Fontaine curve. See
\S\ref{subsec:vector bundles} and \S\ref{subsec:relative FF}.
\end{remark}

\begin{remark}
We take this opportunity to record some corrections to \cite[\S 2.5--2.6]{kedlaya-revisited}
not included in the printed erratum. Thanks to Max Bender for reporting these.
\begin{itemize}

\item
Lemma 2.5.3: in (a), $i \in \mathbb{Z}$ should be $i \geq 0$.
In the last line of the proof of (a), both instances of $n$ should be $i$.
In the last line of the proof of (c), $n \to \infty$ should be $i \to \infty$.

\item
Lemma 2.5.4: it should be assumed that $r>0$.

\item
Corollary 2.5.6: $v_{j,n}$ should be $v_{j,r}$.

\item
Lemma 2.5.11: it should also be assumed that $r \in I$.
The statement is also correct when $r \notin I$ provided that the
right side of the inequality is finite, but this is not used anywhere.

\item
Lemma 2.6.3: 
The first displayed equation should read
\[
v_{n,r}(x-x') \geq w_r(x) + (1-r/r_0) \qquad (n \geq m).
\]
In the second displayed equation, $x'/x$ should be $x/x'$.
In the following line, the inequality
\[
w_r(z_l \pi^m(1 - x'/x)) \geq w_r(y_l) + (1-r/r_0)
\]
should instead assert that
\[
v_{n,r}(z_l \pi^m(1 - x/x')) \geq \min\{w_r(y_l) + (1-r/r_0), 
\min_{n'>n}\{v_{n',r}(y_l)\}\} \,\, \text{for} \,\, n \geq m.
\]
Similarly, after the third displayed equation, the inequality
\[
v_{n,r}(y_{l+1}) \geq w_r(y_l) + (1-r/r_0)
\]
should read
\[
v_{n,r}(y_{l+1}) \geq \min\{w_r(y_l) + (1-r/r_0), 
\min_{n'>n}\{v_{n',r}(y_l)\}\}.
\]
The first two sentences of the last paragraph (from ``It follows that...") 
must be replaced by the following:
\begin{verse}
``It follows that, for $n \geq m$, we have 
\[
v_{n,r}(y_{l+1}) 
\geq \min\{w_r(y_l) + (1-r/r_0),
\min_{n'>n}\{v_{n',r}(y_l)\}\}
\]
We may assume that $y_{l+1}, y_{l+2}, \dots$ also have height at least $m$, 
in which case $w_r(y_{l+h}) \geq w_r(y_l)$ for all $h>0$
and $w_r(y_{l+h}) \geq w_r(y_l) + (1-r/r_0)$ for some $h>0$
(because the maximum value of $n$ for which $v_{n,r}(y_{l+h}) <
w_r(y_l) + (1-r/r_0)$ decreases as $h$ increases).''
\end{verse}
\item
Remark 2.6.4: ``[discreteness of the valuation] on $K$'' should be ``[...] on $\mathcal{O}$''.

\item
Lemma 2.6.7: The sentence starting ``Moreover, if it is ever less than
$\min_{n<0} \{v_{n,r'}(u_l x)\} + c$,'' should continue ``then the smallest 
value of $n$ for which 
$v_{n,r'}(u_{l+1} x) \leq \min_{n<0} \{ v_{n,r'}(u_l x)\} + c$
is strictly greater than the smallest value of $n$ for which 
$v_{n,r'}(u_l x) \leq \min_{n<0} \{v_{n,r'}(u_l x)\} + c$.''

\item
Proposition 2.6.8: the reference to Proposition 2.6.8 in the proof should be to Proposition 2.6.5.
\end{itemize}
\end{remark}

\subsection{Comparison of slope theories}
\label{subsec:comparison}

The slope theories for $\varphi$-modules over $\calR_K$ and $\tilde{\calR}_L$
can be related as follows. Throughout \S\ref{subsec:comparison},
retain Hypothesis~\ref{H:slope Robba}.
\begin{defn} \label{D:perfection Robba}
Equip the field $k((\overline{T}))$
with the $\overline{T}$-adic norm $\alpha$ for the normalization
$\alpha(\overline{T}) = \omega$.
Let $L$ be the completed perfection of
$k((\overline{T}))$ for the unique multiplicative extension of $\alpha$.
Proceeding as in Definition~\ref{D:perfection},
we obtain a map
$s_\varphi: \calE_K \to \tilde{\calE}_L$;
more precisely,
$\tilde{\calE}_L$ is the completion
of the direct perfection of
$\calE_K$ for the weak topology.
For $r>0$ small enough that the $\omega^{r/p}$-Gauss norm of $\varphi(T)/T^{p} - 1$ is less than 1,
$s_\varphi$ takes $\calR_K^{\inte,r}$ into $\tilde{\calR}^{\inte,r}_L$. In fact, this map is isometric
for the $\omega^r$-Gauss norm on the source and the norm $\lambda(\alpha^r)$
on the target \cite[Lemma~2.3.5]{kedlaya-revisited}.
We thus obtain a $\varphi$-equivariant homomorphism $\calR_K \to \tilde{\calR}_L$.
\end{defn}

\begin{example} \label{exa:isometric}
For the Frobenius lift $\varphi(T) = (T+1)^p - 1$ and $\omega = p^{-p/(p-1)}$,
$s_\varphi$ is isometric for the $\omega^r$-Gauss norm for
$r \in (0,1)$.
\end{example}

We can use the extended rings to trivialize $\varphi$-modules over the
smaller rings, as follows.
\begin{prop} \label{P:trivialize etale}
Let $L'$ be the completed direct perfection of $\gotho_{\widehat{\calE_K^{\unr}}}/(p)$ (which is algebraically closed). Identify the completion of the maximal unramified extension $\widehat{\calE^{\unr}_K}$ of $\calE_K$
with a subring of $\tilde{\calE}_{L'}$.
\begin{enumerate}
\item[(a)]
Let $M$ be an \'etale $\varphi$-module over $\gotho_{\calE_K}$.
Then the $\Zp$-module
\[
V = (M \otimes_{\gotho_{\calE_K}} \gotho_{\widehat{\calE^{\unr}_K}})^{\varphi}
\]
has the property that the natural map
\[
V \otimes_{\Zp} \gotho_{\widehat{\calE^{\unr}_K}} \to M \otimes_{\gotho_{\calE_K}} \gotho_{\widehat{\calE^{\unr}_K}}
\]
is an isomorphism.
\item[(b)]
Let $M$ be an \'etale $\varphi$-module over $\calR_K^{\inte}$.
Then the $\Zp$-module
\[
V = (M \otimes_{\calR_K^{\inte}} (\gotho_{\widehat{\calE_K^{\unr}}} \cap \tilde{\calR}_{L'}^{\inte}))^{\varphi}
\]
has the property that the natural map
\[
V \otimes_{\Zp} (\gotho_{\widehat{\calE_K^{\unr}}} \cap \tilde{\calR}_{L'}^{\inte})
 \to M \otimes_{\calR_K^{\inte}} (\gotho_{\widehat{\calE_K^{\unr}}} \cap \tilde{\calR}_{L'}^{\inte})
\]
is an isomorphism.
Moreover,
\[
V \otimes_{\Zp} \Qp = (M \otimes_{\calR_K^{\inte}} \tilde{\calR}_{L'})^{\varphi}
= (M \otimes_{\calR_K^{\inte}} \tilde{\calE}_{L'})^{\varphi}.
\]
\end{enumerate}
\end{prop}
\begin{proof}
Both parts follow from Proposition~\ref{P:DM} plus Lemma~\ref{L:big Robba invariants}.
\end{proof}

\begin{remark} \label{R:etale descent}
For any $\varphi$-modules $M_1, M_2$ over a ring $R$,
there is a natural identification
\[
\Hom_R(M_1,M_2) = M_1^\dual \otimes_R M_2.
\]
If $M_1, M_2$ are both pure of slope $s$, then $M_1^\dual \otimes_R M_2$ is
\'etale. By this reasoning,
Proposition~\ref{P:fully faithful1} reduces to Proposition~\ref{P:trivialize etale},
while Proposition~\ref{P:fully faithful2} reduces to a similar consequence of Proposition~\ref{P:DM}
(derived using Lemma~\ref{L:big Robba invariants}).
\end{remark}

\begin{remark} \label{R:pure semistable}
Note that any pure $\varphi$-module over $\calR_K$ remains pure upon
base extension to $\tilde{\calR}_L$, while any $\varphi$-module over $\calR_K$ whose base extension to
$\tilde{\calR}_L$ is
\emph{semistable}, i.e., which
does not have
any nonzero proper $\varphi$-submodule of larger slope,
is also itself semistable. The reverse implications
also hold, and in fact form part of the proof of Theorem~\ref{T:slope filtration1}
(in the form of a reduction to the somewhat more tractable Theorem~\ref{T:slope filtration2}).
One approach to the reverse implications is to make somewhat careful calculations,
as in \cite{kedlaya-revisited}; a simpler approach is to
use faithfully flat descent, as in \cite[\S 3]{kedlaya-relative}.
\end{remark}

\begin{remark} \label{R:relative}
On the topic of descent, we record some minor inaccuracies in the
statement and proof of \cite[Proposition~3.3.2]{kedlaya-relative}.
\begin{enumerate}
\item[(a)]
The module $M$ should not be assumed to be a $\varphi$-module over $R$, but only
an $R$-module equipped with an isomorphism $\varphi^* M \cong M$. That is,
we should not assume $M$ is finite free over $R$. That is because
in the proof of \cite[Theorem 3.1.3]{kedlaya-relative},
we need to take $R = \mathcal{R}^{\mathrm{bd}}$ and
$S = \tilde{\mathcal{R}}_L^{\mathrm{bd}}$, and to take $M$ to be
the restriction of scalars of a $\varphi$-module over $\mathcal{R}$.

\item[(b)]
The conclusion should not state that $N$ is a $\varphi$-module over $R$, only a
finite \emph{locally free} $R$-module equipped with an isomorphism
$\varphi^* N \cong N$. The proof of \cite[Proposition~3.2.2]{kedlaya-relative}
invokes \cite[Expos\'e~VIII, Corollaire~1.3]{sga1}, which only
guarantees the existence and uniqueness of the module $N$.
It should instead invoke \cite[Expos\'e~XIII, Th\'eor\`eme~1.1]{sga1}
(i.e., Theorem~\ref{T:descent modules}(a)) to recover both
$N$ and the isomorphism $\varphi^* N \cong N$, plus
\cite[Expos\'e~VIII, Proposition~1.10]{sga1}
(i.e., Theorem~\ref{T:descent finite locally free})
 to deduce that
$N$ is finite locally free over $R$.
\end{enumerate}

Note that the modified statement suffices for the applications to
\cite[Theorems~3.1.2 and~3.1.3]{kedlaya-relative} because in those cases
$R$ is a B\'ezout domain (Lemma~\ref{L:Bezout domain}), over which any finite locally free module
is free \cite[Remark 1.1.2]{kedlaya-relative}.
\end{remark}

\section{Relative Robba rings}
\label{sec:relative extended}

We now begin in earnest to consider the relative setting.
Although for some applications it is necessary to consider analogues of the
Robba ring itself, these are not so straightforward to construct, and we leave them
to a subsequent paper. Here, we treat only
the analogue of the extended Robba ring
in which the field of positive characteristic
(over which we define Witt vectors)
is replaced by a more general ring.
As noted in the introduction, this pertains to a ``geometric''
relativization of slope theory, which is rather different
from an ``arithmetic'' relativization in which one works with power series
over a more general ring. (See Remark~\ref{R:arithmetic families}
for some discussion of the latter.)

\setcounter{theorem}{0}
\begin{hypothesis} \label{H:relative extended}
For the remainder of the paper,
let $(R,R^+)$ be a perfect uniform adic Banach algebra over $\Fp$ with spectral norm $\alpha$, such that $R$ is a Banach algebra over some analytic field.
(The condition that $R$ be defined over an analytic field is no restriction at all if we are willing to modify the norm on $R$ without changing the norm topology; see Remark~\ref{R:warp norm}.)
When $R$ has been assumed to be an analytic field, we
conventionally change its name from $R$ to $L$, but this change is pointed out explicitly in each instance.

It is possible to further weaken the hypothesis on $R$ in some of the results;
see for example Remark~\ref{R:global invariants free of trivial spectrum}.
However, this extra generality is of no use to us: our ultimate goal is to consider perfectoid algebras, which always give rise to perfect uniform Banach algebras over analytic fields (see Definition~\ref{D:primitive}).
\end{hypothesis}

\begin{remark} \label{R:valuation semicontinuous}
For $x = \sum_{i=m}^\infty p^i [\overline{x}_i] \in W(R)[p^{-1}]$, for each $h \in \ZZ$, the set
\[
\{\beta \in \calM(R): \beta(\overline{x}_i) = 0 \mbox{ for all } i \leq h\}
\]
is closed in $\calM(R)$. Consequently, the $p$-adic absolute value of
the image of $x$ in $W(\calH(\beta))[p^{-1}]$ is a lower semicontinuous function
of $\beta \in \calM(R)$. When $x$ is a unit, this function is seen to be continuous
by applying the same argument to $x^{-1}$.
\end{remark}

\subsection{Relative extended Robba rings}
\label{subsec:relative extended}

We start by generalizing the definition of the extended Robba rings
and their subrings.

\begin{defn} \label{D:relative extended Robba}
For $*\in\{R, R^+\}$, define the rings $\tilde{\calE}^{\inte}_*, \tilde{\calE}_*, \tilde{\calR}^{\inte,r}_*,
\tilde{\calR}^{\inte}_*, \tilde{\calR}^{\bd,r}_*, \tilde{\calR}^{\bd}_*,
\tilde{\calR}^r_*, \tilde{\calR}_*$ by changing $L$ to $*$ in
Definition~\ref{D:extended Robba}. That is, for $r>0$,
put $\tilde{\calE}^{\inte}_* = W(*)$ and $\tilde{\calE}_*= W(*)[p^{-1}]$,
let $\tilde{\calR}^{\inte,r}_*$ be the ring of $x
= \sum_{i=0}^\infty p^i [\overline{x}_i] \in W(*)$ for which
$\lim_{i \to \infty} p^{-i} \alpha(\overline{x}_i)^r = 0$,
and extend $\lambda(\alpha^s)$ to a power-multiplicative norm on
$\tilde{\calR}^{\inte,r}_*$ for $s \in (0,r]$
by putting
\[
\lambda(\alpha^s) \left( \sum_{i=0}^\infty p^i [\overline{x}_i] \right)
= \max_i \{p^{-i} \alpha(\overline{x}_i)^s\}.
\]
(See Proposition~\ref{P:mu multiplicative}(a) for more details about the case $*=R$.)
Put $\tilde{\calR}^{\bd,r}_*= \tilde{\calR}^{\inte,r}_*[p^{-1}]$,
let $\tilde{\calR}^r_*$ be the Fr\'echet completion of
$\tilde{\calR}^{\bd,r}_*$ under $\lambda(\alpha^s)$ for $s \in (0,r]$,
and drop $r$ from the superscript to
indicate the union over all $r>0$.

For $0 < s \leq r$, let $\tilde{\calR}^{[s,r]}_*$ be the Fr\'echet completion of
$\tilde{\calR}^{\bd,r}_*$ under the norms $\lambda(\alpha^t)$ for $t \in [s,r]$;
it will follow from Lemma~\ref{L:hadamard relative} below that $\tilde{\calR}^{[s,r]}_*$
is also complete under $\max\{\lambda(\alpha^r), \lambda(\alpha^s)\}$,
and so is a Banach ring.
Note that the rings $\tilde{\calE}^{\inte}_{R^+}, \tilde{\calR}^{\inte,r}_{R^+},
\tilde{\calR}^{\inte}_{R^+}$ (resp. $\tilde{\calE}^{\bd}_{R^+}, \tilde{\calR}^{\bd,r}_{R^+},
\tilde{\calR}^{\bd}_{R^+}$) are all equal to $W(R^+)$ (resp. $W(R^+)[1/p]$); we also denote them by $\tilde{\calR}^{\inte,+}_R$ (resp. $\tilde{\calR}^{\bd,+}_R)$ later on. 
Let $\tilde{\calR}^+_R$ be the Fr\'echet completion of $\tilde{\calR}^{\bd,+}_R$
under $\lambda(\alpha^s)$ for all $s>0$.
Note that ring is in general properly contained in $\tilde{\calR}^{\infty}_R = \cap_{r>0} \tilde{\calR}^r_R$.
\end{defn}

We need the following mild extension of the basic constructions of \cite[\S 4]{kedlaya-witt}.
For more discussion of topological aspects (e.g., continuity of $\lambda$ and $\mu$), see \S\ref{subsec:geometric}.
\begin{prop} \label{P:mu multiplicative}
Choose $0 < s \leq r$.
\begin{enumerate}
\item[(a)]
The set $\tilde{\calR}^{\inte,r}_R$ is a ring on which $\lambda(\alpha^s)$ is a power-multiplicative norm.
Moreover, $\lambda(\alpha^s)$ is multiplicative in case $\alpha$ is.
\item[(b)]
For $\beta$ a submultiplicative (resp.\ power-multiplicative, multiplicative) (semi)norm on $R$
dominated by $\max\{\alpha^s, \alpha^r\}$, the formula
\[
\lambda(\beta) \left( \sum_{i=0}^\infty p^i [\overline{x}_i] \right)
= \max_i \{p^{-i} \beta(\overline{x}_i)\}
\]
defines a submultiplicative (resp.\ power-multiplicative, multiplicative) (semi)norm on $\tilde{\calR}^{\inte,r}_R$
dominated by $\max\{\lambda(\alpha^s), \lambda(\alpha^r)\}$.
\item[(c)]
In (b), if $\beta$ is power-multiplicative (resp.\ multiplicative), then
$\lambda(\beta)$ extends to a power-multiplicative (resp.\ multiplicative) (semi)norm on $\tilde{\calR}^{\bd,r}_R$,
and then extends further by continuity to $\tilde{\calR}^{[s,r]}_R$.
\item[(d)]
For $\gamma$ a power-multiplicative (resp.\ multiplicative) (semi)norm on $\tilde{\calR}^{\inte,r}_R$
dominated by $\max\{\lambda(\alpha^s), \lambda(\alpha^r)\}$, the formula
\[
\mu(\gamma)( \overline{x} )  = \gamma([\overline{x}])
\]
defines a power-multiplicative (resp.\ multiplicative) (semi)norm on $R$
dominated by $\max\{\alpha^s, \alpha^r\}$. Moreover, $\gamma$ is dominated by $\lambda(\mu(\gamma))$.
\end{enumerate}
\end{prop}
\begin{proof}
To check (a), we follow the argument of \cite[Lemma~4.1]{kedlaya-witt}, omitting those details which remain unchanged.
Closure of $\tilde{\calR}^{\inte,r}_R$ under addition and the
 inequality $\lambda(\alpha)(x+y) \leq \max\{\lambda(\alpha)(x), \lambda(\alpha)(y)\}$ follow
from the homogeneity of the Witt vector addition formula \cite[Remark~3.7]{kedlaya-witt}
(see also Remark~\ref{R:addition formula}).
This easily implies that $\tilde{\calR}^{\inte,r}_R$ is closed under multiplication and that
$\lambda(\alpha)$ is a submultiplicative norm, as in \cite[Lemma~4.1]{kedlaya-witt}.
To check that $\lambda(\alpha)$ is multiplicative whenever $\alpha$ is, it is enough to check that
$\lambda(\alpha)(xy) \geq \lambda(\alpha)(x) \lambda(\alpha)(y)$ in case the right side of this inequality
is positive. Write $x = \sum_{i=0}^\infty p^i [\overline{x}_i]$, $y = \sum_{i=0}^\infty p^i [\overline{y}_i]$.
Let $j,k$ be the largest indices maximizing $p^{-j} \alpha(\overline{x}_j)$, $p^{-k} \alpha(\overline{y}_k)$.
As in \cite[Lemma~4.1]{kedlaya-witt}, we use the fact that $\lambda(\alpha)$ is a submultiplicative norm
to reduce to the case where $\overline{x}_i = 0$ for $i < j$ and $\overline{y}_i = 0$ for $i < k$.
Then $xy = \sum_{i=j+k}^\infty p^i [\overline{z}_i]$ with $\overline{z}_{j+k} = \overline{x}_j \overline{y}_k$,
proving the desired inequality.
To check that $\lambda(\alpha)$ is power-multiplicative whenever $\alpha$ is,
one makes the same argument with $y = x$. This yields (a); we may check (b) by imitating the proof of (a), and (c) is clear.

To check (d), we introduce an alternate proof of \cite[Lemma~4.4]{kedlaya-witt}.
Again from \cite[Remark~3.7]{kedlaya-witt}, we deduce that
for $\overline{x}, \overline{y} \in R$, $\gamma([\overline{x} + \overline{y}])
\leq \max\{\gamma([\overline{x}]), \gamma([\overline{y}])\}$. (Note that this requires at least
power-multiplicativity, not just submultiplicativity.)
By rewriting this inequality as
$\mu(\gamma)(\overline{x} + \overline{y}) \leq \max\{\mu(\gamma)(\overline{x}), \mu(\gamma)(\overline{y})\}$,
we see that $\mu(\gamma)$ is a (semi)norm.
The power-multiplicativity or multiplicativity of $\mu(\gamma)$ follows from
the corresponding property of $\gamma$. The fact that $\gamma$ is dominated by $\lambda(\mu(\gamma))$
follows as in \cite[Theorem~4.5]{kedlaya-witt}.
\end{proof}

\begin{defn} \label{D:formal to convergent2}
As in Definition~\ref{D:formal to convergent}, we impose
topologies on the aforementioned rings as follows.
\begin{enumerate}
\item[(a)]
Those rings contained in $\tilde{\calE}_R$ carry both a \emph{$p$-adic topology}
(the metric topology defined by the Gauss norm)
and a \emph{weak topology} (in which a sequence
converges if it is bounded for the Gauss norm and
converges under $\lambda(\alpha)$ modulo any fixed power of $p$).
For both topologies, $\tilde{\calE}_R$ is complete.
\item[(b)]
Those rings contained in $\tilde{\calR}^r_R$ carry a \emph{Fr\'echet topology},
in which a sequence converges if and only if it converges under
$\lambda(\alpha^s)$ for all $s \in (0,r]$.
For this topology, $\tilde{\calR}^r_R$ is complete.
\item[(c)]
Those rings contained in $\tilde{\calR}_R$
carry a \emph{limit-of-Fr\'echet topology},
or \emph{LF topology}. This topology is defined by taking the
locally convex direct limit
of the $\tilde{\calR}^r_R$ (each equipped with the Fr\'echet topology).
\end{enumerate}
The analogue of Remark~\ref{R:bounded weak} is true: a sequence in $\calR_R^{\bd,r}$
which is $p$-adically bounded and convergent under $\lambda(\alpha^r)$ also converges in the weak topology.
\end{defn}

\begin{remark} \label{R:trivial norm topologies}
If one extends the definitions of $\tilde{\calE}_R, \tilde{\calR}^{\bd}_R,
\tilde{\calR}_R$ to the case where $\alpha$ is the trivial norm, then in this case these rings all coincide. This makes it possible to abbreviate some arguments.
\end{remark}

\begin{remark}\label{R:phi-invariant}
Recall that $\overline{\varphi}$ acts as the identity map on $R$ if and only if $R$ is generated over $\Fp$
by idempotent elements (Lemma~\ref{L:idempotents}). In this case, the power-multiplicative norm $\alpha$
on $R$ must be trivial, so by Remark~\ref{R:trivial norm topologies}, all of the
topologies in Definition~\ref{D:formal to convergent2} coincide with the $p$-adic
topology.
\end{remark}

\begin{remark}
All of the constructions in Definition~\ref{D:relative extended Robba}
are functorial with respect to bounded homomorphisms $\psi: R \to S$
in which $S$ is another perfect uniform Banach $\Fp$-algebra with norm $\beta$.
If $\psi$ is strict injective,
then $\psi$ is isometric by Remark~\ref{R:perfect uniform strict}, and it is evident that
the functoriality maps induced by $\psi$
are strict injective, for all of the topologies named in
Definition~\ref{D:formal to convergent2}.
(The case when $\psi$ is injective but not strict is more subtle; we do not treat it here.)

Similarly, if $\psi$ is strict surjective, then the functoriality maps induced by $\psi$
are again strict surjective, by the following argument.
Choose $c>0$ such that any $\overline{y} \in S$ admits a lift $\overline{x} \in R$
with $\alpha(\overline{x}) \leq c \beta(\overline{y})$.
(In fact any $c>1$ has this property by Remark~\ref{R:perfect uniform strict}, but we do not need this here.)
By lifting
each Teichm\"uller element separately, we can lift each
$y \in \tilde{\calR}^{\inte,r}_S$ to some $x \in \tilde{\calR}^{\inte,r}_S$
for which $\lambda(\alpha^s)(x) \leq c^r \lambda(\beta^s)(y)$ for all $s \in (0,r]$.
{}From this, the claim follows. (See Lemma~\ref{L:lift surjective} for a similar argument.)
\end{remark}

\begin{lemma} \label{L:extend generators}
For some $0 < s \leq r$, let $M$ be a finite projective module over $\tilde{\calR}^{[s,r]}_R$.
Choose $\beta \in \calM(R)$, and choose $\be_1,\dots,\be_n \in M$ to form a set of module generators of
$M \otimes_{\tilde{\calR}^{[s,r]}_R} \tilde{\calR}^{[s,r]}_{\calH(\beta)}$.
Then there exists a rational localization $(R,R^+) \to (R',R^{\prime +})$ encircling $\beta$ such that $\be_1,\dots,\be_n$
also form a set of module generators of
$M \otimes_{\tilde{\calR}^{[s,r]}_R} \tilde{\calR}^{[s,r]}_{R'}$.
\end{lemma}
\begin{proof}
For each $\gamma \in \calM(\tilde{\calR}^{[s,r]}_{\calH(\beta)})$,
by Nakayama's lemma, $\be_1,\dots,\be_n$ form a set of module generators of
$M \otimes_{\tilde{\calR}^{[s,r]}_R} S_\gamma$ for some
rational localization $\tilde{\calR}^{[s,r]}_R \to S_\gamma$ encircling $\gamma$.
We may cover $\calM(\tilde{\calR}^{[s,r]}_{\calH(\beta)})$ with finitely many of the
$\calM(S_\gamma)$; by Remark~\ref{R:compact spaces}(b), these also cover
$\calM(\tilde{\calR}^{[s,r]}_{R'})$ for some rational localization $R \to R'$ encircling $\beta$.
It follows that $\be_1,\dots,\be_n$
also form a set of module generators of
$M \otimes_{\tilde{\calR}^{[s,r]}_R} \calH(\gamma)$
for each $\gamma \in \calM(\tilde{\calR}^{[s,r]}_{R'})$;
this implies the claim using
Lemma~\ref{L:finite generation}.
\end{proof}

\begin{remark} \label{R:Robba is Tate}
By construction, the ring $\tilde{\calR}^{[s,r]}_R$ is a Banach algebra over the analytic field $\QQ_p$. By contrast, the ring $\tilde{\calR}^{\inte,r}_R$ is complete with respect to the norm $\lambda(\alpha^r)$, but is not a Banach algebra over any analytic field. Nonetheless, it is a Banach ring according to our conventions:
for any topologically nilpotent (resp.\ uniform) unit $\overline{z} \in R$,
$z = [\overline{z}]$ is a topologically nilpotent unit (resp. uniform unit) in $\tilde{\calR}^{\inte,r}_R$.
\end{remark}

\begin{remark} \label{R:Frechet-Stein}
The ring $\tilde{\calR}^r_R$ is by construction the inverse limit of the rings $\tilde{\calR}^{[s,r]}_R$ for all $s \in (0,r]$. As such, it behaves much like a \emph{Fr\'echet-Stein algebra} in the sense of Schneider and Teitelbaum \cite{schneider-teitelbaum}; in particular, it enjoys some cohomological properties more typical of Banach algebras than of general Fr\'echet algebras.
\end{remark}

\subsection{Reality checks}

The operation of Fr\'echet completion in
Definition~\ref{D:relative extended Robba} leaves the structure of the resulting
rings a bit mysterious. To clarify these, we make some calculations akin to
the \emph{reality checks} of \cite[\S 2.5]{kedlaya-revisited}.

\begin{lemma} \label{L:hadamard relative}
For each $x \in \tilde{\calR}_R^{[s,r]}$, the function
$t \mapsto \log \lambda(\alpha^t)(x)$ is continuous and convex.
In particular, $\max\{\lambda(\alpha^r), \lambda(\alpha^s)\} = \sup\{\lambda(\alpha^t): t \in [s,r]\}$.
\end{lemma}
\begin{proof}
As in Lemma~\ref{L:hadamard}.
\end{proof}

\begin{lemma} \label{L:bounded extended}
For $x \in \tilde{\calR}_R$, we have $x \in \tilde{\calR}^{\bd}_R$
if and only if
for some $r>0$, $\lambda(\alpha^s)(x)$ is bounded for $s \in (0,r]$.
\end{lemma}
\begin{proof}
As in Lemma~\ref{L:bounded}.
\end{proof}
Using this criterion, we obtain a generalization of Corollary~\ref{C:units}.
\begin{cor} \label{C:units relative}
Any unit in $\tilde{\calR}_R$ is also a unit in $\tilde{\calR}^{\bd}_R$.
\end{cor}
\begin{proof}
Suppose $x \in \tilde{\calR}_R$ is a unit with inverse $y$.
Choose $r>0$ so that $x,y \in \tilde{\calR}^r_R$.
For each $\beta \in \calM(R)$, the function $\log \lambda(\beta^s)(x)$
is affine in $s$, as in the proof of Corollary~\ref{C:units}.
Write this affine function as $a_\beta s + b_\beta$; we then have
\[
b_\beta =
2 \log \lambda(\beta^{r/2})(x) - \log \lambda(\beta^r)(x)
=
2 \log \lambda(\beta^{r/2})(x) + \log \lambda(\beta^r)(y)
\]
and
\[
a_\beta = \frac{1}{r} \left( \log \lambda(\beta^r)(x)
- b_\beta \right) =
\frac{2}{r} \log \lambda(\beta^{r/2})(y) + \frac{2}{r} \log \lambda(\beta^r)(x)
\]
and so
\begin{align*}
\log \lambda(\beta^s)(x) &= a_\beta s + b_\beta  \\
&\leq \left( \frac{2}{r} \log \lambda(\alpha^{r/2})(y) + \frac{2}{r} \log \lambda(\alpha^r)(x) \right) s +
2 \log \lambda(\alpha^{r/2})(x) + \log \lambda(\alpha^r)(y).
\end{align*}
Taking suprema over $\calM(R)$ yields a similar upper bound for $\log \lambda(\alpha^s)(x)$, so $x \in \tilde{\calR}^{\bd}_R$ by Lemma~\ref{L:bounded extended}.
Similarly, $y \in \tilde{\calR}^{\bd}_R$, so $x$ is a unit in $\tilde{\calR}^{\bd}_R$ as desired.
\end{proof}

\begin{cor} \label{C:extended Robba invariants}
We have $\tilde{\calE}_R^{\varphi} = \tilde{\calR}_R^{\varphi} =
W(R^{\overline{\varphi}})[p^{-1}]$.
In particular, by Corollary~\ref{C:idempotents},
$W(R^{\overline{\varphi}})[p^{-1}] = \Qp$ if and only if $R$ is connected.
\end{cor}
\begin{proof}
It is clear that $\tilde{\calE}_R^{\varphi} = W(R^{\overline{\varphi}})[p^{-1}]
\subseteq \tilde{\calR}_R^{\varphi}$.
We have $\tilde{\calR}_R^{\varphi} = (\tilde{\calR}_R^{\bd})^{\varphi} \subseteq \tilde{\calE}_R^{\varphi}$
by Lemma~\ref{L:bounded extended},
as in the proof of Lemma~\ref{L:big Robba invariants}.
\end{proof}

\begin{lemma}
For $0<s_1\leq s_2\leq r_2\leq r_1$, the natural restriction map $\tilde{\calR}^{[s_1,r_1]}_R\to\tilde{\calR}^{[s_2,r_2]}_R$ is injective.  
\end{lemma}
\begin{proof}
Since $R$ is uniform, the spectral norm $\alpha$ is equal to the supremum norm on $\prod_{\beta\in\calM(R)}\calH(\beta)$ by Theorem~\ref{T:transform}.  Thus it is straightforward to see that the natural map
\[
\tilde{\calR}^{I}_R\to\prod_{\beta\in\calM(R)}\tilde{\calR}^{I}_{\calH(\beta)}
\]
is injective for any closed interval $I\subset(0,\infty)$. Thus it reduces to show the lemma in the case when $R=L$ is an analytic field. We deduce by Lemma~\ref{L:hadamard} that for $x\in\tilde{\calR}_L^{[s,r]}$, the function
$t \mapsto \log \lambda(\alpha^t)(x)$ is continuous and convex on $[s,r]$. This implies that if 
$\lambda(\alpha^{t})(x)=0$ for some $t\in [s,r]$, then $\lambda(\alpha^t)(x)=0$ for all $t\in[s,r]$; thus $x=0$. The lemma then follows. 
\end{proof}

\begin{lemma} \label{L:integral intersection}
For $0 < r \leq r'$, inside $\tilde{\calR}^{[r,r]}_R$
we have
\[
\tilde{\calR}_R^{\inte,r}
\cap \tilde{\calR}_R^{[r,r']} = \tilde{\calR}_R^{\inte,r'}.
\]
\end{lemma}
\begin{proof}
The case $r'=r$ is trivial, so assume that $r' > r$.
Take $x \in \tilde{\calR}_R^{\inte,r}
\cap \tilde{\calR}_R^{[r,r']}$,
and write $x$ as the limit in $\tilde{\calR}_R^{[r,r']}$
of a sequence $x_0,x_1,\dots$ with
$x_i \in \tilde{\calR}_R^{\bd,r'}$.
For each positive integer $j$, we can find $N_j > 0$ such that
\[
\lambda(\alpha^s)(x_i - x) \leq p^{-j} \qquad (i \geq N_j, \, s \in [r,r']).
\]
Write $x_i = \sum_{l=m(i)}^\infty p^l [\overline{x}_{il}]$
and put $y_i = \sum_{l=0}^\infty p^l [\overline{x}_{il}]
\in \tilde{\calR}_R^{\inte,r'}$.
For $i \geq N_j$, having $x \in \tilde{\calR}_R^{\inte,r}$ and
$\lambda(\alpha^r)(x_i -x) \leq p^{-j}$ implies that
$\lambda(\alpha^r)(p^l [\overline{x}_{il}]) \leq p^{-j}$ for $l<0$.
That is,
\[
\alpha(\overline{x}_{il}) \leq p^{(l-j)/r} \qquad (i \geq N_j, \,l < 0).
\]
Since $p^{-l} p^{(l-j)r'/r} \leq p^{1+(1-j)r'/r}$ for $l \leq -1$,
we deduce that $\lambda(\alpha^{r'})(x_i - y_i) \leq p^{1+(1-j)r'/r}$
for $i \geq N_j$.
Consequently, the sequence $y_0,y_1,\dots$ converges to $x$
under $\lambda(\alpha^{r'})$; it follows that $x \in \tilde{\calR}_R^{\inte,r'}$.
This proves the claim.
\end{proof}

\begin{remark} \label{R:splitting}
In both Lemma~\ref{L:bounded} and Lemma~\ref{L:integral intersection},
the key step was to split an element of $\tilde{\calR}_R^{\bd}$ into
what one might call an \emph{integral part} and a \emph{fractional part}.
One cannot directly imitate the construction for elements of
$\tilde{\calR}^{[s,r]}_R$ because they cannot be expressed as sums of
Teichm\"uller elements. One can give presentations of a slightly less
restrictive form with which one can make similar arguments
(the \emph{semiunit presentations} of \cite[\S 2]{kedlaya-revisited});
the \emph{stable presentations} of \cite[\S 5]{kedlaya-witt} are similar.
We will instead rely on Lemma~\ref{L:decompose}
(see below) to simulate splittings into integral and fractional parts.
\end{remark}

As noted in Remark~\ref{R:splitting}, the following lemma
extends to $\tilde{\calR}_R^{[s,r]}$ the splitting argument for
elements of $\tilde{\calR}^{\bd}_R$ used previously. Its formulation
is modeled on \cite[Lemma~2.5.11]{kedlaya-revisited}.
\begin{lemma} \label{L:decompose}
For $0 < s \leq r$ and $n \in \ZZ$,
any $x \in \tilde{\calR}^{[s,r]}_R$ can be written as
$y+z$ with $y \in p^n \tilde{\calR}^{\inte,r}_R$, $z \in \cap_{r' \geq r} \tilde{\calR}^{[s,r']}_R$, and
\begin{equation} \label{eq:decompose}
\lambda(\alpha^{t})(z) \leq p^{(1-n)(1-t/r)} \lambda(\alpha^r)(x)^{t/r} \qquad
(t \geq r).
\end{equation}
\end{lemma}
\begin{proof}
In case $x \in \tilde{\calR}^{\bd}_R$, write
$x = \sum_{i=m(x)}^\infty p^i[\overline{x}_i]$,
and put $y = \sum_{i=n}^\infty p^i[\overline{x}_i]$ and $z = y-x$.
This works because for $i \leq n-1$ and $t \geq r$,
\begin{equation} \label{eq:decompose2}
\lambda(\alpha^t)(p^i [\overline{x}_i])
= p^{-i} \alpha(\overline{x}_i)^t
= p^{-i(1-t/r)} \lambda(\alpha^r)(p^i [\overline{x}_i])^{t/r}
\leq p^{(1-n)(1-t/r)} \lambda(\alpha^r)(p^i [\overline{x}_i])^{t/r}.
\end{equation}
To handle the general case,
choose $x_0,x_1,\ldots \in \tilde{\calR}^{\bd,r}_R$ so that
\[
\lambda(\alpha^t)(x - x_0 - \cdots - x_i)
\leq p^{-i-1} \lambda(\alpha^t)(x) \qquad (i=0,1,\dots; \,t \in [s,r]).
\]
The series $\sum_{i=0}^\infty x_i$ converges to $x$
under $\lambda(\alpha^t)$ for $t \in [s,r]$, and $\lambda(\alpha^t)(x_i) \leq p^{-i}
\lambda(\alpha^t)(x)$ for $i=0,1,\dots$ and $t \in [s,r]$.
Split each $x_i$ as $y_i + z_i$ as above. Since the sum
$\sum_{i=0}^\infty y_i$ converges under $\lambda(\alpha^r)$
and consists of elements of $\tilde{\calR}^{\inte,r}_R$,
it converges under $\lambda(\alpha^t)$ for all $t \in (0,r]$
and defines an element $y$ of $\tilde{\calR}^{\inte,r}_R$.
Put $z = y - x$; then the series $\sum_{i=0}^\infty z_i$ converges
to $z$ under $\lambda(\alpha^t)$ for $t \in [s,r]$.
On the other hand, for $t \geq r$,
by \eqref{eq:decompose2} we have
\[
\lambda(\alpha^{t})(z_i) \leq p^{(1-n)(1-t/r)} \lambda(\alpha^r)(x_i)^{t/r}
\leq p^{(1-n)(1-t/r)} p^{-i(t/r)} \lambda(\alpha^r)(x)^{t/r}.
\]
Consequently, $\sum_{i=0}^\infty z_i$ also converges to $z$ under
$\lambda(\alpha^t)$, and \eqref{eq:decompose} holds.
\end{proof}

We will also have use for the following variant, where we separate in terms of $\alpha$ rather than
the $p$-adic norm.
\begin{lemma} \label{L:decompose variant}
For $c < 1$ and $0 < s \leq r$, each $x \in \tilde{\calR}^{[s,r]}_R$ can be written
as $y + z$ with $y \in \tilde{\calR}^{\bd,r}_R$, $z \in \tilde{\calR}^{[s,r]}_{R^+}$ and 
\begin{align*} 
\lambda(\overline{\alpha}^t)(y), \lambda(\overline{\alpha}^t)(z) \leq \lambda(\overline{\alpha}^t)(x) \qquad &(t \in [s,r]) \\
y \in p^n \tilde{\calR}^{\inte,r}_R \qquad& (n = \lceil - \log_p (\lambda(\overline{\alpha}^s)(x) c^{-s}) \rceil) \\
\lambda(\overline{\alpha}^t)(z) \leq c^{t-r} \lambda(\overline{\alpha}^r)(z) \qquad &(t > r).
\end{align*}
In particular, for any positive integer $a$, if we put $q= p^a$, then for all $t \in [s,r]$,
\begin{align*}
\lambda(\alpha^t)(y), \lambda(\alpha^t)(z) &\leq \lambda(\alpha^t)(x), \\
\lambda(\alpha^t)(\varphi^{-a}(y)) &\leq c^{-(q-1)t/q} \lambda(\alpha^t)(x), \\
\lambda(\alpha^t)(\varphi^a(z)) &\leq c^{(q-1)t} \lambda(\alpha^t)(x).
\end{align*}
\end{lemma}
\begin{proof}
By approximating $x$ as in the proof of Lemma~\ref{L:decompose}, it is enough to consider
the case $x \in \tilde{\calR}^{\bd,r}_R$. In this case, write $x = \sum_{i=m}^\infty p^i [\overline{x}_i]$,
let $y$ be the sum of $p^i [\overline{x}_i]$ over all indices $i$ for which $\alpha(\overline{x}_i) > c$,
and put $z = x-y$.
\end{proof}

As an immediate application of Lemma~\ref{L:decompose}, we extend Lemma~\ref{L:integral intersection}
as follows.
\begin{lemma} \label{L:intersection}
For $0 < s \leq s' \leq r \leq r'$, inside $\tilde{\calR}^{[s',r]}_R$
we have
\[
\tilde{\calR}_R^{[s,r]} \cap \tilde{\calR}_R^{[s',r']} = \tilde{\calR}_R^{[s,r']}.
\]
\end{lemma}
\begin{proof}
Given $x$ in the intersection, apply Lemma~\ref{L:decompose}
to write $x = y+z$ with $y \in \tilde{\calR}_R^{\inte,r}$,
$z \in \tilde{\calR}_R^{[s,r']}$. By Lemma~\ref{L:integral intersection},
\[
y = z - x \in \tilde{\calR}_R^{\inte,r}
\cap \tilde{\calR}_R^{[s',r']} = \tilde{\calR}_R^{\inte,r'}
\subseteq \tilde{\calR}_R^{[s,r']},
\]
so $x \in \tilde{\calR}_R^{[s,r']}$ as desired.
\end{proof}

\begin{lemma} \label{L:in plus ring}
Suppose that $R^+ = \gotho_R$.
\begin{enumerate}
\item[(a)]
An element $x \in \tilde{\calR}^\infty_R$ belongs to $\tilde{\calR}^+_R$ if and only if
\[
\limsup_{r \to +\infty} \lambda(\alpha^r)(x)^{1/r} \leq 1.
\]
\item[(b)]
For $r>0$ and $x \in \tilde{\calR}^r_R$, we have $x = y+z$ for some
$y \in \tilde{\calR}^{\bd,r}_R$ and $z \in \tilde{\calR}^+_R$.
\item[(c)]
We have
$\tilde{\calR}_R^{\inte,r} \cap \tilde{\calR}^+_R = \tilde{\calR}^{\inte,+}_R$.
\end{enumerate}
\end{lemma}
\begin{proof}
We first prove (a) for $x = \sum_{i=m}^\infty p^i [\overline{x}_i] \in \tilde{\calR}^{\bd}_R \cap \tilde{\calR}^\infty_R$, by observing that
\[
\limsup_{r \to +\infty} \lambda(\alpha^r)(x)^{1/r}
= \limsup_{r \to +\infty} \,\sup_i \{p^{-i/r} \alpha(\overline{x}_i)\}.
\]
If $x \in \tilde{\calR}^{\bd,+}_R$, then for some $m$, we can bound the quantity
$p^{-i/r} \alpha(\overline{x}_i)$ from above by $p^{m/r}$, and so the limit superior
in question is at most 1. Conversely, if $x \notin \tilde{\calR}^{\bd,+}_R$,
then there exists an index $i$ for which $\alpha(\overline{x}_i) > 1$; we can then find $\epsilon > 0$
so that $p^{-i/r} \alpha(\overline{x}_i) > 1 + \epsilon$ for $r$ large, so the limit superior
is at least $1 + \epsilon$.
This proves (a) for such $x$.

We next prove (b).
For $r>0$ and $x \in \tilde{\calR}^r_R$,
choose $n \in \ZZ$ so that $n<1$ and $p^{(n-1)/(2r)} \lambda(\alpha^r)(x)^{1/r} < 1$.
Set notation as in the proof of Lemma~\ref{L:decompose}; for $t \geq 2r$, we have
$\lambda(\alpha^t)(z_i)^{1/t} \leq p^{(1-n)(1/t - 1/r)} \lambda(\alpha^r)(x)^{1/r} \leq 1$,
so $z_i \in \tilde{\calR}^{\bd,+}_R$ by the previous paragraph.
Hence $z \in \tilde{\calR}^+_R$ as needed.

We next return to (a).
By the first paragraph, if $x \in \tilde{\calR}^+_R$, then
$\limsup_{r \to +\infty} \lambda(\alpha^r)(x)^{1/r} \leq 1$.
Conversely, if $x \in \tilde{\calR}^\infty_R$
and
$\limsup_{r \to +\infty} \lambda(\alpha^r)(x)^{1/r} \leq 1$,
apply (b) to write $x = y+z$ with $y \in \tilde{\calR}^{\bd}_R$ and $z \in \tilde{\calR}^+_R$.
We may then apply the first paragraph to $y$ to deduce that $x \in \tilde{\calR}^+_R$.

To deduce (c), use Lemma~\ref{L:integral intersection} to deduce that
$\tilde{\calR}_R^{\inte,r} \cap \tilde{\calR}^+_R \subseteq \cap_{s>0} \tilde{\calR}^{\inte,s}_R$,
then argue as in the proof of (a).
\end{proof}

\begin{cor} \label{C:twisted invariants}
For $n$ a nonnegative integer, $d$ a positive integer, $q = p^d$, and $r>0$, the inclusions
\begin{gather*}
\{x \in \tilde{\calR}^+_R: \varphi^d(x) = p^n x\}
\subseteq
\{x \in \tilde{\calR}_R: \varphi^d(x) = p^n x\}, \\
\{x \in \tilde{\calR}^+_R: \varphi^d(x) = p^n x\}
\subseteq
\{x \in \tilde{\calR}^{[r/q,r]}_R: \varphi^d(x) = p^n x\}
\end{gather*}
are bijective.
\end{cor}
\begin{proof}
Suppose first that $x \in \tilde{\calR}^\infty_R$ and $\varphi^d(x) = p^n x$.
For each $r>0$,
\[
\lambda(\alpha^{rq})(x) = \lambda(\alpha^r)(\varphi^d(x)) =
p^{-n} \lambda(\alpha^r)(x).
\]
It follows that for any fixed $s>0$,
\[
\limsup_{r \to +\infty} \lambda(\alpha^r)(x)^{1/r}
\leq \sup_{r \in [s, qs]} \{\limsup_{n \to +\infty}
p^{-n/(rq^n)} \lambda(\alpha^r)(x)^{1/(rq^n)} \} \leq 1,
\]
so if $R^+ = \gotho_R$ then
$x \in \tilde{\calR}^+_R$ by Lemma~\ref{L:in plus ring}(a).
To treat the general case, note that $x$ is now known to belong to the
Fr\'echet completion of $W(\gotho_R)[p^{-1}]$. Under the projection from
this ring to $W(\kappa_R)[p^{-1}]$, $x$ maps to zero if $n>0$ and to
$W(\kappa_R^{\overline{\varphi}^d})[p^{-1}]$ if $n = 0$; in either case
it follows that $x \in \tilde{\calR}^+_R$.

Given $x \in \tilde{\calR}_R$ for which $\varphi^d(x) = p^n x$, there exists $r>0$ for which $x \in \tilde{\calR}^r_R$,
but then $x = p^{-n} \varphi^{-d}(x) \in \tilde{\calR}^{rq}_R$. Consequently,
$x \in \tilde{\calR}_R^\infty$, so by the previous paragraph, $x \in \tilde{\calR}^+_R$.

Given $x \in \tilde{\calR}^{[r/q,r]}_R$ for which $\varphi^d(x) = p^n x$,
for each positive integer $m$ we also have $x \in \tilde{\calR}^{[r/q^m,r]}_R$.
Namely, this holds for $m=1$, and given the statement for some $m$,
we also have $x = p^{-n} \varphi^d(x) \in \tilde{\calR}^{[r/q^{m+1},r/q]}_R$,
so $x \in \tilde{\calR}^{[r/q^{m+1},r]}_R$ by Lemma~\ref{L:intersection}.
It follows that $x \in \tilde{\calR}^r_R$, so by the previous paragraph, $x \in \tilde{\calR}^+_R$.
\end{proof}

\begin{remark} \label{R:intersection}
Let $R \subseteq S' \subseteq S$ and $R \subseteq S'' \subseteq S$ be strict (and hence isometric, by Remark~\ref{R:perfect uniform strict})
inclusions
of perfect uniform Banach $\Fp$-algebras, and suppose that within $S$ we have
$S' \cap S'' = R$. It is easy to see that
\[
*_{S'} \cap *_{S''} = *_R
\qquad * = \tilde{\calE}, \tilde{\calR}^{\inte,r}, \tilde{\calR}^{\inte,+}, \tilde{\calR}^{\inte},
\tilde{\calR}^{\bd,r}, \tilde{\calR}^{\bd,+}, \tilde{\calR}^{\bd},
\]
with the intersection taking place within $*_S$;
namely, $W(S') \cap W(S'') = W(R)$ within $W(S)$.

Now suppose additionally that the morphism $S' \oplus S'' \to S$ taking $(s',s'')$ to $s'-s''$ is strict. Then 
\[
*_{S'} \cap *_{S''} = *_R
\qquad * =
\tilde{\calR}^{[s,r]}, \tilde{\calR}^{r}, \tilde{\calR}^+, \tilde{\calR}.
\]
If the inclusions are not strict, it is unclear
whether such maps as $*_R \to *_{S'}$ are even injective.
\end{remark}

\begin{remark}
It would be useful to have the following refinement of 
Lemma~\ref{L:bounded extended}: an element $x \in \tilde{\calR}_R$ belongs to $\tilde{\calR}^{\bd}_R$ if and only if for each $\beta \in \calM(R)$, the image of $x$ in $\tilde{\calR}_{\calH(\beta)}$ belongs to $\tilde{\calR}^{\bd}_{\calH(\beta)}$. However, it is unclear to us whether to expect this to hold.
\end{remark}

\subsection{Sheaf properties}
\label{subsec:sheaf properties}

We next investigate the sheaf-theoretic properties of the preceding constructions.

\begin{defn} \label{D:robba presheaves}
For $(R,R^+)$ a perfect uniform adic Banach algebra and
\[
* = \tilde{\calE}^{\inte}, \tilde{\calE}, \tilde{\calR}^{\inte,r}, \tilde{\calR}^{\inte,+}, \tilde{\calR}^{\inte},
\tilde{\calR}^{\bd,r}, \tilde{\calR}^{\bd,+}, \tilde{\calR}^{\bd},
\tilde{\calR}^{[s,r]}, \tilde{\calR}^{r}, \tilde{\calR}^+, \tilde{\calR},
\]
construct the presheaf $*$ on $\Spa(R,R^+)$ assigning to each open subset $U$
the inverse limit of $*_S$ over each rational localization $(R,R^+) \to (S,S^+)$
for which $\Spa(S,S^+) \subseteq U$.
\end{defn}

\begin{lemma} \label{L:Tate lemma}
With notation as in Lemma~\ref{L:Tate lemma1}, the sequence
\begin{equation} \label{eq:strict exact sequence}
0 \to R \to R_1 \oplus R_2 \to R_{12} \to 0
\end{equation}
remains exact, and the morphisms remain almost optimal, when $R_*$ is replaced by any
of $\tilde{\calE}^{\inte}_{R_*}$ or $\tilde{\calE}_{R_*}$ (for the $p$-adic norm),
$\tilde{\calR}^{\inte,r}_{R_*}$ (for the norm $\lambda(\alpha^r)$),
$\tilde{\calR}^{\bd,r}_{R_*}$ (for the maximum of $\lambda(\alpha^r)$ and the $p$-adic norm),
$\tilde{\calR}^{[s,r]}_{R_*}$ (for the norm $\max\{\lambda(\alpha^r), \lambda(\alpha^s)\}$),
or $\tilde{\calR}^r_{R_*}$ (omitting the statement about norms).
\end{lemma}
\begin{proof}
This is a straightforward consequence of Lemma~\ref{L:Tate lemma1} except for the case of $\tilde{\calR}^r_{R_*}$, for which we must separately check exactness on the right.
This calculation may be viewed as giving a  vanishing of a $\varprojlim^1$ term,
as suggested by Remark~\ref{R:Frechet-Stein}.

Given $x \in \tilde{\calR}^r_{R_{12}}$, we may construct elements $x_{n,j} \in \tilde{\calR}^{\bd,r}_{R_j}$ for $n=0,1,\dots$ and $j=1,2$ such that for
$y_n = x - \sum_{m=0}^{n-1} (x_{m,1} + x_{m,2})$, we have
\begin{align*}
\lambda(\alpha^s)(y_n) &\leq p^{-n} \lambda(\alpha^s)(x) \qquad (s \in [p^{-n} r, r]) \\
\lambda(\alpha^s)(x_{n,j}) &\leq (1 + p^{-n}) \lambda(\alpha^s)(y_n) \qquad (s \in [p^{-n} r, r]; j \in \{1,2\}). 
\end{align*}
Namely, given $y_n$, we split $y_n$ in $\tilde{\calR}^{[p^{-n-1}r,r]}_{R_*}$
using Lemma~\ref{L:Tate lemma1},
then approximate the results suitably well with elements of $\tilde{\calR}^{\bd, r}_{R_*}$.
The sums $\sum_{n=0}^\infty x_{n,j}$ for $j=1,2$ then converge and give the desired splitting of $x$.
\end{proof}

\begin{theorem} \label{T:robba presheaves}
All of the presheaves defined in Definition~\ref{D:robba presheaves} are 
sheaves. Moreover, for $* = \tilde{\calE}^{\inte}, \tilde{\calE}, \tilde{\calR}^{\inte,r}, \tilde{\calR}^{\inte},
\tilde{\calR}^{\bd,r}, \tilde{\calR}^{\bd},
\tilde{\calR}^{[s,r]}, \tilde{\calR}^{r}, \tilde{\calR}$, the resulting sheaf is acyclic (i.e., satisfies the Tate sheaf property).
\end{theorem}
\begin{proof}
By Proposition~\ref{P:acyclicity template}, we may deduce the claim from
Lemma~\ref{L:Tate lemma}.
\end{proof}

The Kiehl property is somewhat more elusive; we only obtain it for a few of the sheaves in question.
\begin{lemma} \label{L:Robba ring cover}
Choose $0 < s \leq r$.
\begin{enumerate}
\item[(a)]
A multiplicative seminorm $\beta$ on $\tilde{\calR}^{[s,r]}_R$ is dominated by $\max\{\lambda(\alpha^s), \lambda(\alpha^r)\}$ if and only if it is dominated by
$\lambda(\alpha^t)$ for some unique $t \in [s,r]$.
\item[(b)]
For $\beta,t$ as in (a), put
$\gamma = \mu(\beta)^{1/t} \in \calM(R)$. Then $\beta$ extends uniquely to a multiplicative seminorm on
$\tilde{\calR}^{[s,r]}_{\calH(\gamma)}$.
\end{enumerate}
\end{lemma}
\begin{proof}
We first address (a).
If $\beta$ is dominated by $\lambda(\alpha^t)$ for some $t \in [s,r]$,
then $\beta$ is dominated by $\max\{\lambda(\alpha^s), \lambda(\alpha^r)\}$
by Lemma~\ref{L:hadamard relative}.
Conversely, suppose that $\beta$ is dominated by $\max\{\lambda(\alpha^s), \lambda(\alpha^r)\}$. Write $R$ as a Banach algebra over some analytic field $K$,
and pick $\overline{\pi} \in K^\times$ with $\alpha(\overline{\pi}) < 1$. Then
$\beta([\overline{\pi}]) \in [\alpha^r(\overline{\pi}), \alpha^s(\overline{\pi})]$,
so there exists
$t \in [s,r]$ such that $\beta([\overline{\pi}]) = \alpha^t(\overline{\pi})$.
For $\overline{x} \in R$ such that $\alpha(\overline{x}) \leq 1$, we also have
$\beta([\overline{x}]) \leq 1$. If we take $\overline{x} = \overline{y}^m \overline{\pi}^{-n}$
for $m$ a positive integer and $n$ an arbitrary integer, we deduce that
if $\alpha(\overline{y})^m \alpha(\overline{\pi})^{-n} \leq 1$,
then $\beta([\overline{y}])^m \beta([\overline{\pi}])^{-n} \leq 1$.
That is, if $\alpha(\overline{y}) \leq \alpha(\overline{\pi})^{n/m}$,
then $\beta([\overline{y}]) \leq \alpha(\overline{\pi})^{nt/m}$.
Since $n/m$ can be chosen to be any rational number,
it follows that $\beta([\overline{y}]) \leq \alpha(\overline{y})^t$.

To deduce (b), note first that $\beta$ extends uniquely to the localization of $\tilde{\calR}^{[s,r]}_R$
at the multiplicative set consisting of $[\overline{x}]$ for each $x \in R \setminus \gothp_{\gamma}$,
and that this extension is dominated by $\max\{\lambda(\gamma^s), \lambda(\gamma^r)\}$. Then observe
that the separated completion under $\max\{\lambda(\gamma^s), \lambda(\gamma^r)\}$ of this localization
is precisely $\tilde{\calR}^{[s,r]}_{\calH(\gamma)}$.
\end{proof}

\begin{lemma} \label{L:Kiehl lemma perfect}
With notation as in Lemma~\ref{L:Tate lemma1}, for any $r>0$, the diagrams
\[
\xymatrix{ R \ar[r] \ar[d] & R_1 \ar[d] \\
R_2 \ar[r] & R_{12} 
}
\qquad
\xymatrix{ \tilde{\calR}^{\inte,r}_R \ar[r] \ar[d] & \tilde{\calR}^{\inte,r}_{R_1} \ar[d] \\
\tilde{\calR}^{\inte,r}_{R_2} \ar[r] & \tilde{\calR}^{\inte,r}_{R_{12}}
}
\]
are glueing squares
in the sense of Definition~\ref{D:glueing pair}. (Note that the rings in the second diagram are Banach rings by virtue of Remark~\ref{R:Robba is Tate}.)
\end{lemma}
\begin{proof}
Note that $R_2 \to R_{12}$ has dense image because $f$ is already invertible in $R_2$,
and that by construction, $\calM(R_1 \oplus R_2) \to \calM(R)$ is surjective.
Hence $R \to R_1, R \to R_2, R \to R_{12}$ form a glueing square by Lemma~\ref{L:Tate lemma}.
The other assertion follows similarly, keeping in mind Lemma~\ref{L:Robba ring cover} in order to get surjectivity
of $\calM(\tilde{\calR}^{\inte,r}_{R_1} \oplus \tilde{\calR}^{\inte,r}_{R_2}) \to \calM(\tilde{\calR}^{\inte,r}_{R})$.
\end{proof}

\begin{theorem} \label{T:Kiehl for integral}
The sheaves $\tilde{\calE}^{\inte}$ and $\tilde{\calR}^{\inte,r}$
satisfy the Kiehl glueing property.
\end{theorem}
\begin{proof}
We check both claims using Proposition~\ref{P:Tate reduction}.
The case of $\tilde{\calR}^{\inte,r}$ follows from
Lemma~\ref{L:Kiehl lemma perfect} and Proposition~\ref{P:glue projective}.
The case of $\tilde{\calE}^{\inte}$ does not follow in this manner (see Remark~\ref{R:no glueing}); we instead directly check the hypotheses of Lemma~\ref{L:Kiehl generic2}(b) as follows. Let $M_1, M_2, M_{12}$ be a finite projective glueing datum. Reducing modulo $p$, we obtain a finite projective glueing datum
for the rings $R_*$, which by Theorem~\ref{T:Tate-Kiehl analogue1}
arises from a finite projective module $\overline{M}$ over $R$. 
For any $\overline{\bv} \in \overline{M}$, we may lift $\overline{\bv}$ to an element $\bv$ in the kernel of $M_1 \oplus M_2 \to M_{12}$ by successive approximations: if $\bv \in M_1 \oplus M_2$ maps to $p^n \bw \in M_{12}$, we lift $\bw$ to $\bv' \in M_1 \oplus M_2$ and replace $\bv$ with $\bv - p^n \bv'$. In particular, we may lift a finite generating set of $\overline{M}$ to a generating set of $M_1$.
This verifies that the hypothesis of Lemma~\ref{L:Kiehl generic1} holds for finite projective glueing data; the fact that $\Spec(W(R_1) \oplus W(R_2))$ covers $\Maxspec(W(R))$ follows from Lemma~\ref{L:Kiehl lemma perfect} and 
the identification $\Maxspec(W(*)) \cong \Maxspec(*)$.
\end{proof}

\begin{remark} \label{R:no glueing}
The reason the proof of Theorem~\ref{T:Kiehl for integral} for the sheaf $\tilde{\calE}^{\inte}$ fails to follow the model of $\tilde{\calR}^{\inte,r}$ is that $\tilde{\calE}^{\inte}_{R_2}$ is not dense in $\tilde{\calE}^{\inte}_{R_{12}}$ for the $p$-adic topology, only for the weak topology. It should be possible to adapt the concept of a glueing square in Definition~\ref{D:glueing pair} to apply to a topology not specified in terms of a norm, but we did not verify this.

However, no such adaptation can exist for $\tilde{\calE}$ or $\tilde{\calR}^{\bd,r}$, because the Kiehl property fails in these cases;  see Example~\ref{exa:Tate curve}.
Somewhat confusingly, the rings $\tilde{\calE}_R$ and $\tilde{\calR}^{\bd,r}_R$ are themselves sheafy (Theorem~\ref{T:Kiehl for Robba}), but there is no inconsistency because localizations of these rings do not correspond directly to localizations of $R$ (again because of the density issue).
\end{remark}

While the proof of Theorem~\ref{T:Kiehl for integral} can be carried through for
$\tilde{\calR}^{[s,r]}$, it is more useful to deduce the corresponding statement from a much stronger glueing property. Theorem~\ref{T:Kiehl for Robba} and its proof are taken from the PhD thesis of Ryan Rodriguez \cite{rodriguez}.
\begin{lemma} \label{L:orthogonal Schauder basis}
Let $A$ be a Banach algebra over $\Qp$ and let $K$ be the completion of $\Qp(p^{p^{-\infty}})$. Then each element $x$ of $A \widehat{\otimes}_{\Qp} K$ admits a unique presentation as a convergent sum $\sum_i x_i \otimes p^i$, where $i$ runs over $\ZZ[p^{-1}] \cap [0,1)$; moreover, the tensor product norm of $x$ can be computed as
$\max_i \{\left| x_i \right| p^{-i} \}$.
\end{lemma}
\begin{proof}
This is a straightforward consequence of the fact that the $p^i$ form an orthogonal topological basis of $K$ over $\Qp$.
\end{proof}

\begin{theorem}[Rodriguez] \label{T:Kiehl for Robba}
For $0 < s \leq r$, the Banach rings $\tilde{\calE}_R$, $\tilde{\calR}^{\bd,r}_R$, and $\tilde{\calR}^{[s,r]}_R$ are relatively perfectoid.
In particular, by Theorem~\ref{T:Kiehl for preperfectoid},
any extension of one of these rings to an adic Banach ring is stably uniform and sheafy
and satisfies the Tate sheaf and Kiehl glueing properties.
\end{theorem}
\begin{proof}
We treat only the case of $\tilde{\calR}^{[s,r]}_R$ in detail, the other cases being easier (and not needed in what follows).
Let $K$ be the completion of $\Qp(p^{p^{-\infty}})$,
which is a perfectoid field (see Example~\ref{exa:lift analytic field}). We first check that $S = \tilde{\calR}^{[s,r]}_R \widehat{\otimes}_{\Qp} K$ is uniform.
For $t \in [s,r]$, let $\left| \bullet \right|_t$ be the tensor product
norm on $S$ induced by $\lambda(\alpha^t)$; we first check that $\left| \bullet \right|_t$ is power-multiplicative. 
It suffices to check the inequality $\left| x^2 \right|_t \geq \left|x \right|^2_t$ for $x$ running over a dense subset of $S$. We may thus assume that $x$ has the form $\sum_{i=1}^n [\overline{x}_i] \otimes p^{j_i}$ for some $\overline{x}_i \in R$, $j_i \in \ZZ[p^{-1}]$ such that the $\overline{x}_i$ are nonzero and the $j_i$ are pairwise distinct.
Note that in this case, $\left| \bullet \right|_u$ is well-defined for all $u>0$, not just $u \in [s,r]$. Let $T_x$ be the set of $u>0$ for which $\lambda(\alpha^u)(\overline{x}_i) p^{-j_i} = \lambda(\alpha^u)(\overline{x}_{i'}) p^{-j_{i'}}$ for some $i \neq i'$; this set is finite.
For each $u>0$ with $u \notin T_x$, there is a unique index $i$ maximizing
$\lambda(\alpha)^u(\overline{x}_i) p^{-j_i}$; if we put
$y = [\overline{x}_i] \otimes p^{j_i}$, then
$\left| y^2 \right|_u = \left| y \right|^2_u$
and $\left| y \right|_u > \left| x-y \right|_u$.
Using Lemma~\ref{L:orthogonal Schauder basis}, it follows easily that
\[
\left| x^2 \right|_u = \left| y^2 + (x+y)(x-y) \right|_u
= \left|y \right|^2_u = \left| x \right|^2_u.
\]
This proves the claim for the given $x$ and all $u \notin T_x$; however, 
Lemma~\ref{L:hadamard relative} implies that $\log \left| x \right|_u$ is a convex and hence continuous function of $u$, so we may interpolate the claim for $u = t$ even if $t \in T_x$.
We thus conclude that $\left| \bullet \right|_t$ is power-multiplicative for each $t \in [s,r]$, and so $\max\{\left| \bullet \right|_s, \left| \bullet \right|_r\}$ 
(which by log-convexity again is the norm induced by $\max\{\lambda(\alpha^s), \lambda(\alpha^r)\}$)
is also power-multiplicative. 
 In particular, $S$ is uniform.

We next check that for any perfectoid algebra $A$ for which $T = \tilde{\calR}^{[s,r]}_R \widehat{\otimes}_{\Qp} A$ is uniform, $T$ is also perfectoid. By 
Proposition~\ref{P:power inverse limit as set}, it suffices to show that 
$\varprojlim_{x \mapsto x^p} T$ generates a dense subring of $T$. 
Since $A$ is perfectoid, $\varprojlim_{x \mapsto x^p} A$ generates a dense subring of $A$; we obtain a dense subring of $T$ by adding the additional generator
$\theta([\overline{x}])$ for each $\overline{x} \in R$.

To conclude, note that $S$ is uniform (by the first paragraph) and hence perfectoid (by the second paragraph). For any perfectoid algebra $A$, it follows that $T$ is uniform (by Proposition~\ref{P:preperfectoid to relatively perfectoid}) and hence perfectoid (by the second paragraph again).
 \end{proof}

We can use Theorem~\ref{T:Kiehl for Robba} to deduce the Kiehl property for the sheaf $\tilde{\calR}^{[s,r]}$ by relating rational localizations of $R$ and $\tilde{\calR}^{[s,r]}_R$.

\begin{defn} \label{D:valuation projection}
Recall that
Lemma~\ref{L:Robba ring cover} defines a map $\calM(\tilde{\calR}^{[s,r]}_R) \to \calM(R)$. We lift this to a map of adic spectra as follows.

For $v$ a semivaluation on $\tilde{\calR}^{[s,r]}_R$, let $\beta$ be the associated seminorm and define $t$ and $\gamma$ as in Lemma~\ref{L:Robba ring cover}. Since $\beta$ extends uniquely to a multiplicative seminorm on $\tilde{\calR}^{[s,r]}_{\calH(\gamma)}$, $v$ extends to a semivaluation on this ring.
The set of $\overline{x} \in \calH(\gamma)$ for which $v([\overline{x}]) \leq 1$ is a valuation ring; we thus obtain a valuation $w$ on $R$.

Let $\tilde{\calR}^{[s,r],+}_R$ be the completion with respect to 
$\max\{\lambda(\alpha^s), \lambda(\alpha^r)\}$ of the 
subring of $\tilde{\calR}^{[s,r]}_R$
generated by those $x$ for which $\max\{\lambda(\alpha^s), \lambda(\alpha^r)\}(x) < 1$
and $[\overline{y}]$ for $\overline{y} \in R^+$. The previous paragraph then defines a map
$\mu: \Spa(\tilde{\calR}^{[s,r]}_R, \tilde{\calR}^{[s,r],+}_R) \to \Spa(R,R^+)$.
\end{defn}

\begin{lemma} \label{L:lift rational to Robba}
For $0 < s \leq r$ and $(R,R^+) \to (S,S^+)$ a rational localization,
the morphism $(\tilde{\calR}^{[s,r]}_R, \tilde{\calR}^{[s,r],+}_R) \to (\tilde{\calR}^{[s,r]}_S, \tilde{\calR}^{[s,r],+}_S)$ is a rational localization representing the inverse image of $\Spa(S,S^+)$ under the morphism
of Definition~\ref{D:valuation projection}.
\end{lemma}
\begin{proof}
Represent the rational subdomain represented by $(R,R^+) \to (S,S^+)$
as 
\[
\{v \in \Spa(R,R^+): v(\overline{f}_i) \leq v(\overline{g}) \quad (i=1,\dots,n)\}.
\]
Form the rational subspace
\[
U= \{v \in \Spa(\tilde{\calR}^{[s,r]}_R, \tilde{\calR}^{[s,r],+}_R): 
v([\overline{f}_i]) \leq v([\overline{g}_i]) \quad (i=1,\dots,n)\}.
\]
Let $\Spa(\tilde{\calR}^{[s,r]}_R, \tilde{\calR}^{[s,r],+}_R) \to (T,T^+)$ be
the corresponding rational localization. 
The image of $\Spa(\tilde{\calR}^{[s,r]}_S, \tilde{\calR}^{[s,r],+}_S)
\to \Spa(\tilde{\calR}^{[s,r]}_R, \tilde{\calR}^{[s,r],+}_R)$ is equal to $U$,
so by the universal property of rational localizations, we obtain a morphism
$(T,T^+) \to  (\tilde{\calR}^{[s,r]}_S, \tilde{\calR}^{[s,r],+}_S)$ which induces a bijection of adic spectra. 
By Theorem~\ref{T:Kiehl for Robba}, $(T,T^+)$ is uniform.
By Theorem~\ref{T:transform}, the morphism
$T \to \tilde{\calR}^{[s,r]}_S$ is isometric for the spectral norms;
however, the image of this morphism contains the dense subring generated over $\Qp$ by $[\overline{x}]$ for $\overline{x} \in S$.
This yields the claim.
\end{proof}

\begin{cor} \label{C:lift covering family}
For $0 < s \leq r$, if $\{(R,R^+) \to (R_i,R_i^+)\}_i$ 
is a rational covering, 
then $\{(\tilde{\calR}^{[s,r]}_R, \tilde{\calR}^{[s,r],+}_R) \to (\tilde{\calR}^{[s,r]}_{R_i}, \tilde{\calR}^{[s,r],+}_{R_i})\}_i$ is also a
rational covering.
\end{cor}

We now recover an analogue of Theorem~\ref{T:Kiehl for integral}.
\begin{theorem} \label{T:Kiehl for Robba2}
The sheaf $\tilde{\calR}^{[s,r]}$ on $\Spa(R,R^+)$ satisfies the Kiehl glueing property.
\end{theorem}
\begin{proof}
This is immediate from Theorem~\ref{T:Kiehl for Robba} and
Corollary~\ref{C:lift covering family}.
\end{proof}

From Theorem~\ref{T:Kiehl for Robba}, we also obtain a glueing result with respect to the interval $[s,r]$.

\begin{lemma} \label{L:interval localization}
Choose $0 < s \leq r$.
Let $K$ be a perfectoid analytic field containing $\Qp(p^{p^{-\infty}})$
with $|K^\times| = \RR^+$.
\begin{enumerate}
\item[(a)]
The tensor product norm on $\tilde{\calR}^{[s,r]}_R \widehat{\otimes}_{\Qp} K$
is power-multiplicative, and $\tilde{\calR}^{[s,r]}_R \widehat{\otimes}_{\Qp} K$ is perfectoid.
\item[(b)]
Choose $\overline{z} \in R$ with $\alpha(\overline{z}) < 1$ and $\alpha(\overline{z}) \alpha(\overline{z}^{-1}) = 1$ (possible because $R$ is a Banach algebra over an analytic field). 
Then for $0<s \leq s' \leq r' \leq r$,
$(\tilde{\calR}^{[s,r]}_R, \tilde{\calR}^{[s,r],+}_R) \widehat{\otimes}_{\Qp} K \to
(\tilde{\calR}^{[s',r']}_R, \tilde{\calR}^{[s',r'],+}_R) \widehat{\otimes}_{\Qp} K$
is the rational localization corresponding to
\[
\{v \in \Spa((\tilde{\calR}^{[s,r]}_R, \tilde{\calR}^{[s,r],+}_R) \widehat{\otimes}_{\Qp} K):
v([\overline{z}]) \in [\alpha(\overline{z})^{r'}, \alpha(\overline{z})^{s'}]
\}.
\]
\end{enumerate}
\end{lemma}
\begin{proof}
Part (a) follows from Theorem~\ref{T:Kiehl for Robba} and
Proposition~\ref{P:perfectoid tensor}. Part (b) follows by a similar argument as in the proof of Lemma~\ref{L:lift rational to Robba}.
\end{proof}
\begin{cor} \label{C:interval localization}
For $K$ as in Lemma~\ref{L:interval localization}, for $I, I_1,\dots,I_n$ closed subintervals of $(0, +\infty)$ satisfying $I = I_1 \cup \cdots \cup I_n$, 
\[
\{(\tilde{\calR}^{I}_R, \tilde{\calR}^{I,+}_R) \widehat{\otimes}_{\Qp} K \to
(\tilde{\calR}^{I_i}_R, \tilde{\calR}^{I_i,+}_R) \widehat{\otimes}_{\Qp} K\}_{i=1}^n
\]
is a rational covering.
\end{cor}

\begin{theorem} \label{T:Kiehl for interval}
For $I, I_1,\dots,I_n$ closed subintervals of $(0, +\infty)$ satisfying $I = I_1 \cup \cdots \cup I_n$, the morphism $\tilde{\calR}^{I}_R \to \oplus_i \tilde{\calR}^{I_i}_R$
is an effective descent morphism for the category of finite projective modules over uniform Banach rings.
\end{theorem}
\begin{proof}
Using Lemma~\ref{L:inject tensor}, we can tensor over $\Qp$ with a perfectoid field $K$
containing $\Qp(p^{p^{-\infty}})$ with $|K^\times| = \RR^+$; we may then deduce the claim from Theorem~\ref{T:Kiehl for Robba} and
Corollary~\ref{C:interval localization}.
Alternatively, one can reduce to the case $n=2$ and check directly that one gets a glueing square in the sense of Definition~\ref{D:glueing pair}:
condition (a) holds by Theorem~\ref{T:transform} (for exactness at the left), Lemma~\ref{L:intersection} (for exactness at the middle),
and Lemma~\ref{L:decompose variant} (for exactness at the right);
condition (b) is straightforward;
condition (c) holds by Lemma~\ref{L:Robba ring cover}.
\end{proof}

\subsection{Some geometric observations}
\label{subsec:geometric}

We mention some observations concerning the geometry
of the spaces $\calM(\tilde{\calR}^{[s,r]}_R)$, in the spirit of \cite{kedlaya-witt}.
These results will be used later to build relative Fargues-Fontaine curves; see \S\ref{subsec:relative FF}.

\begin{theorem}
Define $\lambda: \calM(R) \to \calM(\tilde{\calR}^{\inte,1}_R)$,
$\mu: \calM(\tilde{\calR}^{\inte,1}_R) \to \calM(R)$ as in Proposition~\ref{P:mu multiplicative}.
\begin{enumerate}
\item[(a)]
The maps $\lambda$ and $\mu$ are continuous. Moreover, the inverse image under either map of a finite
union of Weierstrass (resp.\ Laurent, rational) subdomains has the same form.
\item[(b)] For all $\beta \in \calM(R)$, $(\mu \circ \lambda)(\beta) = \beta$.
\item[(c)] For all $\gamma \in \calM(W(R))$, $(\lambda \circ \mu)(\gamma) \geq \gamma$.
\end{enumerate}
\end{theorem}
\begin{proof}
The proof of \cite[Theorem~4.5]{kedlaya-witt} carries over without change.
\end{proof}

\begin{lemma} \label{L:surjective Robba}
Let $R \to S$ be a bounded homomorphism of perfect uniform Banach $\Fp$-algebras such that
$\calM(S) \to \calM(R)$ is surjective. Then for any $r>0$, the map
$\calM(\tilde{\calR}^{\inte,r}_S) \to \calM(\tilde{\calR}^{\inte,r}_R)$ is also surjective.
\end{lemma}
\begin{proof}
Equip $R[T]$ and $S[T]$ with the $p^{-1}$-Gauss norm,
and let $R'$ and $S'$ be the completions of $R[T]^{\perf}$ and $S[T]^{\perf}$.
We may then identify $S'$ with $S \widehat{\otimes}_R R'$; since $\calM(S) \to \calM(R)$ is surjective, so is
$\calM(S') \to \calM(R')$ by \cite[Lemma~1.20]{kedlaya-witt}.

Given $\gamma \in \calM(\tilde{\calR}^{\inte,r}_R)$,
put $\beta = \mu(\gamma)$ and $\gotho = \gotho_{\calH(\beta)}$,
extend $\gamma$ to $\tilde{\calR}^{\inte,r}_{\calH(\beta)}$ by continuity,
then restrict to $W(\gotho)$.
Let $\gotho'$ be the completed perfect closure of $\gotho[T]$ for the $p^{-1}$-Gauss norm;
as in \cite[Definition~7.5]{kedlaya-witt}, we may extend $\gamma$ from
$W(\gotho)$ to a seminorm $\gamma'$ on $W(\gotho')$ in such a way that $\gamma'(p-T) = 0$.
By \cite[Remark~5.14]{kedlaya-witt}, this extension computes the quotient norm
on $W(\gotho')/(p-T)$ induced by $\lambda(\beta')$ for $\beta' = \mu(\gamma')$.

Choose $\tilde{\beta} \in \calM(S)$ lifting $\beta$, put
$\tilde{\gotho} = \gotho_{\calH(\tilde{\beta})}$,
and let $\tilde{\gotho}'$ be the completion of $\tilde{\gotho}[T]$ for the $p^{-1}$-Gauss norm.
We may then identify $\tilde{\gotho}'$ with $\tilde{\gotho} \widehat{\otimes}_{\gotho} \gotho'$; by
\cite[Lemma~1.20]{kedlaya-witt}, the map $\calM(\tilde{\gotho}') \to \calM(\gotho')$ is surjective.

We can thus lift $\beta'$ to a seminorm $\tilde{\beta}'$ on $\tilde{\gotho'}$.
Let $\tilde{\beta}'$ be the quotient norm on $W(\tilde{\gotho}')/(p-T)$ induced by
$\lambda(\tilde{\beta}')$, viewed as a seminorm on $W(\tilde{\gotho}')$. We may then restrict
$\tilde{\beta}'$ to a seminorm $\tilde{\beta}$ on $W(\tilde{\gotho})$,
extend multiplicatively to $\tilde{\calR}^{\inte,r}_{\calH(\tilde{\beta})}$,
then restrict to $\tilde{\calR}^{\inte,r}_S$. This proves the claim.
\end{proof}

\begin{defn}
Choose $r>0$ and $\gamma \in \tilde{\calR}^{\inte,r}_R$. For $\beta = \mu(\gamma)^{1/r}$,
we may extend $\gamma$ to $\tilde{\calR}^{\inte,r}_{\calH(\beta)}$ and then restrict
to $W(\gotho_{\calH(\beta)})$. We may then define the multiplicative seminorm $H(\gamma,t)$
on $W(\gotho_{\calH(\beta)})$ as in \cite[Definition~7.5]{kedlaya-witt},
extend multiplicatively to $\tilde{\calR}^{\inte,r}_{\calH(\beta)}$,
then restrict back to $\tilde{\calR}^{\inte,r}_R$. From \cite[Theorem~7.8]{kedlaya-witt},
the construction has the following properties.
\begin{enumerate}
\item[(a)]
We have $H(\gamma,0) = \gamma$.
\item[(b)]
We have $H(\gamma,1) = (\lambda \circ \mu)(\gamma)$.
\item[(c)]
For $t\in [0,1]$, $\mu(H(\gamma,t)) = \mu(\gamma)$.
\item[(d)]
For $t,u\in [0,1]$, $H(H(\gamma,t),u) = H(\gamma,\max\{t,u\})$.
\end{enumerate}
\end{defn}

\begin{theorem} \label{T:homotopy}
For any $r>0$, the map $H: \calM(\tilde{\calR}^{\inte,r}_R) \times [0,1] \to \calM(\tilde{\calR}^{\inte,r}_R)$ is continuous.
\end{theorem}
\begin{proof}
Equip $R$ with the spectral norm.
For the trivial norm on $\gotho_R$, the map $H: \calM(W(\gotho_R)) \times [0,1] \to \calM(W(\gotho_R))$ is continuous by
\cite[Theorem~7.8]{kedlaya-witt}. 
By identifying $\calM(\tilde{\calR}^{\inte,r}_R)$
with a closed subspace of $\calM(W(\gotho_R))$, we deduce the claim.
\end{proof}

\begin{remark}
One can go further with analysis of this sort; for instance, one can show that the fibres of $\mu$ bear a strong resemblance to the spectra of one-dimensional
affinoid algebras over an analytic field. See \cite[\S 8]{kedlaya-witt}.
\end{remark}

\begin{prop} \label{P:radius fibration}
Define the topological space
\[
T_R = \bigcup_{0 < s < r} \calM(\tilde{\calR}^{[s,r]}_R).
\]
\begin{enumerate}
\item[(a)]
For each $\beta \in T_R$, there is a unique value $t \in (0, +\infty)$ for which $\alpha^t$ dominates $\mu(\beta)$
(or equivalently, $\lambda(\alpha^t)$ dominates $\beta$).
\item[(b)]
Let $t: T_R \to (0, +\infty)$ be the map described in (a).
Then the formula $\beta \mapsto (\mu(\beta)^{1/t(\beta)}, t(\beta))$ defines a continuous map
$T_R \to \calM(R) \times (0,+\infty)$. In particular, $t$ is continuous.
\item[(c)]
The group $(\varphi^*)^\ZZ$ acts properly discontinuously on $T_R$
with compact quotient $X_R$. (Note that the map in (b) induces a continuous map
$X_R \to \calM(R) \times S^1$.)
\item[(d)]
The map $T_R \to \calM(R) \times (0,+\infty)$
is a strong deformation retract, and induces a strong deformation retract
$X_R \to \calM(R) \times S^1$.
\end{enumerate}
\end{prop}
\begin{proof}
To check (a), we appeal to Lemma~\ref{L:Robba ring cover}.
To check (b), choose $\beta_0 \in T_R$ and put $t_0 = t(\beta_0)$.
Let $U$ be any open neighborhood of $\mu(\beta_0)$ of the form
$\{\gamma \in \calM(R): \gamma(\overline{f}_1) \in I_1,\dots, \gamma(\overline{f}_n) \in I_n\}$
for some
$\overline{f}_1,\dots,\overline{f}_n \in R$ and some open intervals $I_1,\dots,I_n$.
Let $(a,b)$ be any open subinterval of
$(0, +\infty)$ containing $t_0$.
Choose $\overline{z} \in R$ for which
$0 < \mu(\beta_0)(\overline{z}) < 1$, and put $z = [\overline{z}]$.
Choose $\delta > 1$ such that $a < t_0/\delta, t_0 \delta < b$.
For $i=1,\dots,n$, choose an open neighborhood $J_i$ of
$\beta_0([\overline{f}_i])$ such that
for all $x \in J_i$ and all $u \in [1/\delta,\delta]$,
$x^{1/(ut_0)} \in I_i$. Put
\[
V = \{\gamma \in T_R: \gamma([\overline{f}_1]) \in J_1,\dots, \gamma([\overline{f}_n]) \in J_n,
\gamma(z) \in (\beta_0(z)^{t_0/\delta}, \beta_0(z)^{t_0 \delta})\};
\]
this is an open subset of $T_R$ with the property that for any $\gamma \in V$,
$\mu(\gamma)^{1/t(\gamma)} \in U$ and $t(\gamma) \in (a,b)$.
This gives the desired continuity.

To check (c), first apply (b) after observing that for all $\beta \in T_R$,
$t(\varphi^*(\beta)) = p t(\beta)$.
We see from this that
the action is properly discontinuous, so $X_R$ is Hausdorff.
Then note that for any $r>0$, the projection
$\calM(\tilde{\calR}^{[r/p,r]}_R) \to X_R$ is surjective.
Since $X_R$ receives a surjective continuous map from a compact space,
it is quasicompact (Remark~\ref{R:compact spaces}(a)) and hence compact.

To check (d), argue as in the proof of Theorem~\ref{T:homotopy} to produce
a continuous map $H: \calM(\tilde{\calR}^{\inte,r}_R) \times [0,1] \to \calM(\tilde{\calR}^{\inte,r}_R)$.
Then observe by Lemma~\ref{L:Robba ring cover} that the image of $\calM(\tilde{\calR}^{\inte,r}_R) \times \{1\}$
may be identified with $\calM(R) \times (0,r]$ by mapping $(\gamma,s)$ to $\lambda(\gamma^s)$,
and that the resulting map $\calM(\tilde{\calR}^{\inte,r}_R) \to \calM(R) \times (0,r]$  is precisely $T_R$.
\end{proof}

\begin{remark} \label{R:relative FF curve quotient}
The space $X_R$ will later appear as the maximal Hausdorff quotient of the relative Fargues-Fontaine curve over $R$; see \S\ref{subsec:relative FF}. For now, we note that when $\calM(R)$ is contractible,
Proposition~\ref{P:radius fibration} asserts that $X_R$ has the homotopy type of a circle,
$T_R$ is the universal covering space of $X_R$, and
$\varphi^*$ acts on $T_R$ as a deck transformation generating the fundamental group.
For instance, this is the case when $R = L$ is an analytic field;
in this case, the profinite \'etale fundamental group of $X_R$ is the product of the absolute Galois groups of $\Qp$ and $L$ (see \cite{weinstein}), but the \'etale fundamental group (once it is suitably defined, which we will not do here) should differ from this. In particular, the $\ZZ$-covering coming from $T_R \to X_R$ corresponds to the subgroup of the unramified Galois group of $\Qp$ generated by Frobenius.

When $L$ is a finite extension of $\Fp((\overline{\pi}))$, this suggests
a relationship with the Weil group of the field
$\tilde{\calR}^{\inte,1}_L/(z)$ for $z = \sum_{i=0}^{p-1} [1 + \overline{\pi}]^{i/p}$
(this field being a finite extension of the completion of $\Qp(\mu_{p^\infty})$).
However, this relationship remains to be clarified; see
Remark~\ref{R:relative FF curve quotient2} for further discussion.
\end{remark}

\subsection{Compatibility with finite \'etale extensions}
\label{subsec:almost purity}

We next establish a compatibility between the construction of extended Robba rings and formation of
finite \'etale ring extensions.
As promised earlier, this yields a variant of Faltings's almost purity theorem,
thus refining the perfectoid correspondence introduced in \S\ref{subsec:perfectoid2}.

\begin{convention} \label{conv:extend norm}
For $S \in \FEt(R)$, we will always view $S$ as a finite Banach $R$-algebra
as per Proposition~\ref{P:finite etale Banach norm}.
By Lemma~\ref{L:finite uniform Banach}, $S$ is then also a perfect uniform Banach $\Fp$-algebra.
\end{convention}

\begin{lemma} \label{L:lift surjective}
Let $\psi: R \to S$ be a bounded homomorphism from $R$ to a perfect uniform $\Fp$-algebra $S$ with spectral norm $\beta$. Use $\psi$ to view $S$ as an $R$-algebra.
Let $x_1,\dots,x_n$ be elements of $\tilde{\calR}^{\inte}_S$ whose reductions
$\overline{x}_1,\dots,\overline{x}_n$ modulo $p$ generate $S$ as an $R$-module.
Then for all sufficiently small $r>0$,
$x_1,\dots,x_n$ generate $\tilde{\calR}^{\inte,r}_S$ as a module over
$\tilde{\calR}^{\inte,r}_R$.
\end{lemma}
\begin{proof}
It is harmless to assume that $\overline{x}_1,\dots,\overline{x}_n$ are all nonzero.
Since the surjection $R^n \to S$ is strict by Theorem~\ref{T:open mapping}, we can find $c \geq 1$ such that for each
$\overline{z} \in S$, there exist $\overline{a}_1,\dots,\overline{a}_n \in R$
for which $\overline{z} = \sum_{i=1}^n \overline{a}_i \overline{x}_i$
and $\alpha(\overline{a}_i) \beta(\overline{x}_i)
\leq c \beta(\overline{z})$ for $i=1,\dots,n$.

Given $z \in \tilde{\calR}^{\inte}_S$, for $l=0,1,\dots$
we choose $z_l \in \tilde{\calR}^{\inte}_S$ and $a_{l,1},\dots,a_{l,n} \in
\tilde{\calR}^{\inte}_R$ as follows. Put $z_0 = 0$. Given $z_l$,
let $\overline{z}_l$ be its reduction modulo $p$, and invoke the previous paragraph to
construct
$\overline{a}_{l,1},\dots,\overline{a}_{l,n} \in R$
with $\alpha(\overline{a}_{l,i}) \beta(\overline{x}_i)
\leq c \beta(\overline{z}_l)$ for $i=1,\dots,n$ such that
$\overline{z}_l = \sum_{i=1}^n \overline{a}_{l,i} \overline{x}_i$.
Then put $a_{l,i} = [\overline{a}_{l,i}]$ and $z_{l+1} = p^{-1}(z_l - \sum_{i=1}^n a_{l,i} x_i)$.

Choose $r>0$ such that $x_1,\dots,x_n \in \tilde{\calR}^{\inte,r}_S$
and $\lambda(\beta^r)(x_i - [\overline{x}_i]) < \beta(\overline{x}_i)^r$.
For $z \in \tilde{\calR}^{\inte,r}_S$,
we then have
$\lambda(\beta^r)(z_{l+1}) \leq c^r p \lambda(\beta^r)(z_l)$,
and so $\lambda(\beta^r)(z_l) \leq (c^r p)^l \lambda(\beta^r)(z_0)$.
In particular,
\[
\alpha(\overline{a}_{l,i}) \leq c \beta(\overline{x}_i)^{-1} \lambda(\beta^r)(z_0)^{1/r}
(c p^{1/r} )^l.
\]
For $s$ sufficiently small (depending on $r$), we have
$c^s p^{s/r - 1} < 1$, and so the series $\sum_{l=0}^\infty p^l a_{l,i}$
converges under $\lambda(\alpha^s)$. 

We now specialize the previous construction to the case $z = [\overline{z}]$.
In this case, we can write $[\overline{z}] = \sum_{i=1}^n a_i x_i$ with $a_i \in \tilde{\calR}^{\inte,s}_R$ and
$\lambda(\alpha^s)(a_i) \leq c_1 \alpha(\overline{z})^s$ for some $c_1>0$ not depending on $\overline{z}$. By writing a general element of $\tilde{\calR}^{\inte,s}_R$ as
$z = \sum_{n=0}^\infty p^n [\overline{z}_n]$, we see that $x_1,\dots,x_n$ generate $\tilde{\calR}^{\inte,s}_S$ as a module over $\tilde{\calR}^{\inte,s}_R$.
This proves the claim.
\end{proof}

\begin{prop} \label{P:perfect henselian}
Choose any $r>0$.
\begin{enumerate}
\item[(a)]
The base extension functors
\[
\varphi^{-1}\mbox{-}\FEt(\tilde{\calR}^{\inte,r}_R) \to \FEt(\tilde{\calR}^{\inte}_R) \to \FEt(W(R)) \to \FEt(R)
\]
are tensor equivalences, where $\varphi^{-1}$-$\FEt(*)$ denote the category of finite \'etale $*$-algebras equipped with isomorphisms with their $\varphi^{-1}$-pullbacks.
\item[(b)]
The composition $\varphi^{-1}\mbox{-}\FEt(\tilde{\calR}^{\inte,r}_R) \to \FEt(R)$ admits a quasi-inverse taking $S$ to 
$\tilde{\calR}^{\inte,r}_S$.
\end{enumerate}
\end{prop}
\begin{proof}
We first prove a weak version of (a): the functors $\FEt(\tilde{\calR}^{\inte}_R) \to \FEt(W(R)) \to \FEt(R)$
are tensor equivalences.
For each $r>0$, $\tilde{\calR}^{\inte,r}_R$ is complete with respect to
the maximum of $\lambda(\alpha^r)$ and the $p$-adic norm.
For these norms, the maps $\tilde{\calR}^{\inte,r}_R \to \tilde{\calR}^{\inte,s}_R$
for $0 < s \leq r$ are submetric by Lemma~\ref{L:hadamard relative}.
Consequently, Lemma~\ref{L:henselian direct limit}(b)
implies that the pair $(\tilde{\calR}^{\inte}_R, (p))$ is henselian.
We may thus conclude using Theorem~\ref{T:henselian}.

We next prove a weak version of (b): for $S \in \FEt(R)$, the corresponding element $U_S$ of $\FEt(\tilde{\calR}^{\inte}_R)$ may be identified with
$\tilde{\calR}^{\inte}_S$.
We may identify $U_S/(p)$ and $\tilde{\calR}^{\inte}_S/(p)$ with $S$;
the $p$-adic completions of $U_S$ and $\tilde{\calR}^{\inte}_S$ may then be
identified with $W(S)$ by the uniqueness property of the latter.

Let $\pi_1, \pi_2: S \to S \otimes_R S$ denote the
structure morphisms.
Put $V = U_S \otimes_{\tilde{\calR}^{\inte}_R} \tilde{\calR}^{\inte}_S$,
and let $\tilde{\pi}_1: U_S \to V$ and $\tilde{\pi}_2: \tilde{\calR}^{\inte}_S \to V$ denote the
structure morphisms. Note that $\tilde{\pi}_2$ is the distinguished lift of
$\pi_2$ from $\FEt(S)$ to $\FEt(\tilde{\calR}^{\inte}_S)$ constructed above.
Consequently,
if we view the multiplication map $\mu: S \otimes_R S \to S$ as a map in $\FEt(S)$
by equipping $S \otimes_R S$ with the structure morphism $\pi_2$,
then (a) provides a lift $\tilde{\mu}$
of $\mu$ to $\FEt(\tilde{\calR}^{\inte}_S)$.
The composition $\psi = \tilde{\mu} \circ \tilde{\pi}_1: U_S \to \tilde{\calR}^{\inte}_S$
lifts the identity map modulo $p$.
As noted above, the injection $U_S \to W(S)$ factors through $\psi$,
so $\psi$ is injective; since $\psi$ is $\tilde{\calR}^{\inte}_R$-linear,
it is also surjective by Lemma~\ref{L:lift surjective}
and Convention~\ref{conv:extend norm}. This proves the claim.

We next verify that for $S \in \FEt(R)$, we have $\tilde{\calR}^{\inte,r}_S 
\in \FEt(\tilde{\calR}^{\inte,r}_R)$. We first observe that $\tilde{\calR}^{\inte,r}_S$ is finitely generated as a module over $\FEt(\tilde{\calR}^{\inte,r}_R)$: this holds for small $r>0$ by Lemma~\ref{L:lift surjective}, and hence for all $r>0$ by repeated application of $\varphi^{-1}$. We next note that given any $\tilde{\calR}^{\inte,r}_R$-linear surjection of a finite free module onto $\tilde{\calR}^{\inte,r}_S$, using the fact that $\tilde{\calR}^{\inte}_S \in \FEt(\tilde{\calR}^{\inte}_R)$ we can find a $\tilde{\calR}^{\inte,s}_R$-linear splitting for some $0 < s \leq r$.
It follows that $\tilde{\calR}^{\inte,r}_S$ is finite projective as a module over
$\tilde{\calR}^{\inte,r}_R$ for $r>0$ small, and hence again for all $r>0$ using $\varphi^{-1}$. To conclude, it remains to check that the natural map
\[
\tilde{\calR}^{\inte,r}_S \otimes_{\tilde{\calR}^{\inte,r}_R} \tilde{\calR}^{\inte,r}_S
\to \tilde{\calR}^{\inte,r}_{S \otimes_R S}
\]
is an isomorphism (as then we may conclude that $\tilde{\calR}^{\inte,r}_S$ is also finite projective over $\tilde{\calR}^{\inte,r}_S \otimes_{\tilde{\calR}^{\inte,r}_R} \tilde{\calR}^{\inte,r}_S$); this follows from it being a map with dense image between finite projective $\tilde{\calR}^{\inte,r}_R$-modules of the same rank.

We now note that $\varphi^{-1}\mbox{-}\FEt(\tilde{\calR}^{\inte,r}_R) \to \FEt(\tilde{\calR}^{\inte}_R)$ is fully faithful (because any morphism in
$\FEt(\tilde{\calR}^{\inte}_R)$ is automatically $\varphi$-equivariant by virtue of the equivalence with $\FEt(R)$)
and has a right quasi-inverse (by the previous paragraph). This completes the proof of both (a) and (b).
\end{proof}

\begin{prop} \label{P:perfect mixed lift}
Let $S$ be a (faithfully) finite \'etale $R$-algebra. Then for
\[
* \in \{\tilde{\calE}, \tilde{\calR}^{\inte,r}, \tilde{\calR}^{\inte},
\tilde{\calR}^{\bd,r}, \tilde{\calR}^{\bd}, \tilde{\calR}^{[s,r]}, \tilde{\calR}^r, \tilde{\calR}\},
\]
the natural homomorphism $*_R \to *_S$
is (faithfully) finite \'etale.
\end{prop}
\begin{proof}
The cases $* = \tilde{\calE}, \tilde{\calR}^{\inte,r}, \tilde{\calR}^{\inte},
\tilde{\calR}^{\bd}, \tilde{\calR}^{\bd}$ follow at once from
Proposition~\ref{P:perfect henselian}.
To handle the cases $* = \tilde{\calR}^{[s,r]}, \tilde{\calR}^r, \tilde{\calR}$, it suffices to check that the natural map
$\tilde{\calR}^r_R \otimes_{\tilde{\calR}^{\bd,r}_R} \tilde{\calR}^{\bd,r}_S
\to \tilde{\calR}^r_S$ is an isometric isomorphism with respect to $\lambda(\alpha^r)$.
We proceed by first noticing that
$\tilde{\calR}^r_R \otimes_{\tilde{\calR}^{\bd,r}_R} \tilde{\calR}^{\bd,r}_S
= \tilde{\calR}^r_R \widehat{\otimes}_{\tilde{\calR}^{\bd,r}_R} \tilde{\calR}^{\bd,r}_S$
because $\tilde{\calR}^{\bd,r}_S$ is a finite projective $\tilde{\calR}^{\bd,r}_R$-module
(see Definition~\ref{D:finite etale category}).
We may then argue as in Lemma~\ref{L:lift surjective} that
$\tilde{\calR}^r_R \otimes_{\tilde{\calR}^{\bd,r}_R} \tilde{\calR}^{\bd,r}_S
\to \tilde{\calR}^r_S$ is a strict surjection for sufficiently small $r>0$, and hence for all $r>0$
by applying $\varphi^{-1}$ as needed.
In particular, there exists $c>0$ (depending on $r$)
for which any element $z \in \tilde{\calR}^r_S$
can be lifted to $\sum_i x_i \otimes y_i \in \tilde{\calR}^r_R \otimes_{\tilde{\calR}^{\bd,r}_R} \tilde{\calR}^{\bd,r}_S$
with $\max_i \{\lambda(\alpha^r)(x_i y_i)\} \leq c \lambda(\alpha^r)(z)$.
Finally, given a nonzero element $z \in \tilde{\calR}^r_S$, choose $z_0 \in \tilde{\calR}^{\bd,r}_S$ with
$\lambda(\alpha^r)(z-z_0) < c^{-1} \lambda(\alpha^r)(z)$, and lift $z - z_0$ to $\sum_i x_i \otimes y_i$
with $\max_i \{\lambda(\alpha^r)(x_i y_i)\} \leq c \lambda(\alpha^r)(z-z_0)$.
The representation $1 \otimes z_0 + \sum_i x_i \otimes y_i$ of $z$ then shows that
the map $\tilde{\calR}^r_R \otimes_{\tilde{\calR}^{\bd,r}_R} \tilde{\calR}^{\bd,r}_S
\to \tilde{\calR}^r_S$, which is evidently submetric, is in fact isometric.
In particular, it is injective, completing the argument.
\end{proof}

To link these results to almost purity, we use the following extension of Lemma~\ref{L:stable residue}.
\begin{lemma} \label{L:stable residue2}
Suppose that $z \in W(R^+)$ is primitive of degree $1$ and that $z - [\overline{z}] \in p W(R^+)^\times$. Choose any $r \geq 1$.
\begin{enumerate}
\item[(a)]
Any $x \in \tilde{\calR}^{\inte,r}_R/(z)$ lifts to $y = \sum_{i=0}^\infty p^i [\overline{y}_i]
\in W(R)$ with $\alpha(\overline{y}_0) \geq \alpha(\overline{y}_i)$ for all $i$.
\item[(b)]
For any closed interval $I$ of $(0, +\infty)$ containing $1$ and any positive integer $m$, each of the arrows in the diagram
\[
\xymatrix{
W(R^+)[[\overline{z}]^{-1}] \ar@{=}[rr] \ar[d] & &  W(\gotho_R)[[\overline{z}]^{-1}] \ar[d] 
\ar[rr] & & \tilde{\calR}^{\inte,r}_R \ar[dd] \\
W(R^+)[[\overline{z}]^{-1}, p^{-1}] \ar@{=}[rr] & & W(\gotho_R)[[\overline{z}]^{-1}, p^{-1}] \ar[rrd] \\
& \tilde{\calR}^{\bd,r}_{R^+} \ar[rr] 
& & \tilde{\calR}^{\bd,r}_{\gotho_R} \ar[r] & \tilde{\calR}^{\bd,r}_R \ar@{-->}[dd] \\
W(R^+)[p^{-1}] \ar[rr] \ar[uu] \ar@{=}[ru] \ar[dr] & & W(\gotho_R)[p^{-1}] \ar[uu] \ar@{=}[ur] \ar[dr] \\
& \tilde{\calR}^I_{R^+} \ar[rr] & & \tilde{\calR}^I_{\gotho_R} \ar[r] & \tilde{\calR}^I_R
}
\]
(including $\tilde{\calR}^{\bd,r}_R \to \tilde{\calR}^I_R$ only if $I \subseteq (0,r]$)
becomes an isomorphism upon quotienting by the ideal $(z^m)$.
\end{enumerate}
\end{lemma}
\begin{proof}
By hypothesis, we have $p = y(z - [\overline{z}])$ for some unit $y \in W(R^+)$.
Given $x = \sum_{i=0}^\infty p^i [\overline{x}_i] \in \tilde{\calR}^{\inte,r}_R$,
the series $\sum_{i=0}^\infty (-y)^i [\overline{x}_i \overline{z}^i]$ converges to an element of
$W(\gotho_R)[[\overline{z}]^{-1}]$ congruent to $x$ modulo $z$. This proves surjectivity of the map
\[
W(\gotho_R)[[\overline{z}]^{-1}]/(z) \to \tilde{\calR}^{\inte,r}_R/(z);
\] 
the existence of lifts as in (a) follows by applying Lemma~\ref{L:stable residue} to
$x [\overline{z}]^n$ for a large positive integer $n$.

To prove (b), note that since $z$ is not a zero-divisor in any of the rings appearing in the diagram (by consideration of Newton polygons as in Definition~\ref{D:no common slopes}), the claim reduces at once to the case $m=1$. 
Since $y \in W(R^+)$, when working modulo $z$, inverting $[\overline{z}]$ is equivalent to inverting $p$.
Consequently, all of the vertical arrows in the first and third column induce isomorphisms modulo $z$, as does the arrow $\tilde{\calR}^{\inte,r}_R \to \tilde{\calR}^{\bd,r}_R$.

We next check that $W(R^+)[p^{-1}] \to \tilde{\calR}^I_{R^+}$ induces a surjection modulo $z$; this will imply the same for $W(\gotho_R)[p^{-1}] \to \tilde{\calR}^{I}_{\gotho_R}$.
Any element of $\tilde{\calR}^I_{R^+}$ can be written as a convergent sum $\sum_{n=0}^\infty [\overline{x}_n] p^{-i_n}$ for some $\overline{x}_n \in R^+$ and $i_n \in \ZZ$. Convergence with respect to $\lambda(\alpha)$ means that there exists some $n_0 \geq 0$ such that for all $n \geq n_0$, $\alpha(\overline{x}_n) p^{i_n} < 1$.
Let $S$ be the set of integers $n \geq n_0$ for which $i_n > 0$,
and let $T$ be the complement of $S$ in $\{0,1,\dots\}$. We may thus represent the same class in $\tilde{\calR}^I_{R^+}/(z)$ by the sum
\[
\sum_{n \in S} (-y)^{-i_n} [\overline{x}_n \overline{z}^{-i_n}] + \sum_{n \in T} x_n p^{-i_n}
\]
which converges in $W(R^+)[p^{-1}]$ with respect to $\lambda(\alpha)$.

We next check that $\tilde{\calR}^I_{R^+} \to \tilde{\calR}^I_{R}$ induces a surjection modulo $z$; this will imply the same for $\tilde{\calR}^I_{\gotho_R} \to \tilde{\calR}^I_{R}$. 
By Lemma~\ref{L:decompose variant}, for $t$ the right endpoint of $I$,
any class
in $\tilde{\calR}^I_{R^+}/(z)$ can be represented as the sum of a class arising from
$\tilde{\calR}^I_{R^+}/(z)$ and a class arising from $\tilde{\calR}_R^{\bd,t}/(z)$. By what we have already shown, the latter class
can also be found in $W(R^+)[p^{-1}]/(z)$ and hence in
$\tilde{\calR}^I_{R^+}/(z)$. This proves the claim.

To conclude, it suffices to check that $W(R^+)[p^{-1}] \to \tilde{\calR}^I_R$
induces an injection modulo $z$ (as this will formally imply the same for $W(R^+)[p^{-1}] \to \tilde{\calR}^I_{R^+}$, and similarly with $R^+$ replaced by $\gotho_R$).
Choose $x \in W(R^+)[p^{-1}]$
which maps to zero in $\tilde{\calR}^I_R/(z)$. 
We wish to check that $x$ is divisible by $z$ in $W(R^+)[p^{-1}]$.
For this purpose, there is no harm to assume that $x \in W(R^+)[p^{-1}]$ or to
modify $x$ by a multiple of $z$; by (a),  we may thus assume that
$x = \sum_{n=0}^\infty [\overline{x}_n] p^n$ with $\overline{\alpha}(\overline{y}_0) \geq \overline{\alpha}(\overline{x}_n)$ for all $n \geq 0$.
If $x$ is nonzero, we can choose $\beta \in \calM(R)$ such that
$\beta(\overline{y}_0) > p^{-1} \alpha(\overline{y}_0)$; then the divisibility of $x$ by $z$ in $\tilde{\calR}^I_{\calH(\beta)}$ yields a contradiction by consideration of Newton polygons.
\end{proof}

\begin{cor} \label{C:refine perfectoid correspondence}
For any $z \in W(\gotho_R)$ primitive of degree $1$ and any $r \geq 1$,
for $A = W(\gotho_R)[[\overline{z}]^{-1}]/(z) = \tilde{\calR}^{\inte,r}_R/(z)$
(by Lemma~\ref{L:stable residue2}),
we have a 2-commuting diagram
\begin{equation} \label{eq:perfectoid diagram}
\xymatrix{
\varphi^{-1}\mbox{-}\FEt(\tilde{\calR}^{\inte,r}_R) \ar[r] \ar[d] & \FEt(A) \ar@{-->}[ld] \\
\FEt(R) &
}
\end{equation}
in which the solid arrows are base extensions and the dashed arrow is the one provided by
Theorem~\ref{T:mixed lift ring}. In particular, each of these is a tensor equivalence.
\end{cor}
\begin{proof}
The commutativity comes from Lemma~\ref{L:stable residue2}.
The dashed arrow is a tensor equivalence by Theorem~\ref{T:mixed lift ring},
while the vertical arrow is an equivalence by Proposition~\ref{P:perfect henselian}.
\end{proof}

\begin{remark} \label{R:Robba to perfectoid}
Corollary~\ref{C:refine perfectoid correspondence} makes it possible to study the effect of Frobenius
on the perfectoid correspondence, leading to the almost purity theorem
(Theorem~\ref{T:almost purity}). Arthur Ogus has asked about an alternate approach to Theorem~\ref{T:mixed lift ring} obtained by directly establishing essential surjectivity of
$\FEt(\tilde{\calR}^{\inte,r}_R) \to \FEt(A)$
using lifting arguments for smooth algebras, as in the work of Elkik \cite{elkik} and Arabia
\cite{arabia}; however, it is not immediately clear how to obtain $\varphi^{-1}$-equivariance in such an approach.
\end{remark}

To assert an almost purity theorem, we need a few definitions from almost ring theory; for these we follow \cite{gabber-ramero}.
\begin{defn} \label{D:almost}
Let $A$ be a uniform Banach algebra equipped with its spectral norm. Suppose that for each $\epsilon > 1$, there exists $\lambda \in A$ with $1 < \left| \lambda \right| < \epsilon$ and $\left| \lambda \right| \left| \lambda^{-1} \right| = 1$. For instance, this holds if $A$ is a Banach algebra over an analytic field with nondiscrete norm.

An $\gotho_A$-module is \emph{almost zero} if it is killed by $\gothm_A$.
The category of \emph{almost modules} over $\gotho_A$ is the localization of the category of
$\gotho_A$-modules at the set of morphisms with almost zero kernel and cokernel.

An $A$-module $B$ is \emph{almost finite projective}
if for each $t \in \gothm_A$, there exist a finite free $A$-module $F$
and some morphisms $B \to F \to B$ of $A$-modules whose composition is multiplication by $t$. We say that $B$ is 
\emph{uniformly almost finite projective} if there is a positive integer $m$ so that we can always choose $F$ to be free of rank $m$. (See \cite[Lemma~2.4.15]{gabber-ramero} for some equivalent formulations.)

For $B$ an $A$-algebra whose underlying $A$-module is almost finite projective,
there is a well-defined trace map $\Trace: B \to A$ in the category of almost modules over $\gotho_A$
\cite[\S 4.1.7]{gabber-ramero};
we say that $B$ is \emph{almost finite \'etale} if the trace pairing induces an almost isomorphism $B \to \Hom_A(B,A)$.
This is not the definition used in \cite{gabber-ramero}, but is equivalent to it via \cite[Theorem~4.1.14]{gabber-ramero}.
\end{defn}

We are now ready to fulfill the promise made in Remark~\ref{R:almost later}.
The key new ingredient provided by relative Robba rings is an action of Frobenius (or more
precisely its inverse) in characteristic 0, which can be used in much the same
way that Frobenius can be used to give a cheap proof of almost purity in
positive characteristic (see Remark~\ref{R:almost perfect} and then \cite[Chapter~3]{gabber-ramero}).

\begin{theorem}[Almost purity] \label{T:almost purity}
Let $A$ be a perfectoid algebra. Then for any $B \in \FEt(A)$,
$\gotho_B$ is uniformly almost finite projective and almost \'etale over $\gotho_A$.
\end{theorem}
\begin{proof}
By Theorem~\ref{T:mixed lift ring}, $B$ is also perfectoid.
Equip $A$ and $B$ with their spectral norms.
Apply Lemma~\ref{L:primitive generator} to construct $z \in W(\gotho_A^{\frep})$ primitive of degree 1 generating the kernel of $\theta: W(\gotho_A^{\frep}) \to \gotho_A$. For $n$ a nonnegative integer, put $y_n = [\overline{z}^{-p^{-n}}] z$.
Let $R,S$ correspond to $A,B$ via Theorem~\ref{T:perfectoid ring},
so that $S \in \FEt(A)$ by Theorem~\ref{T:mixed lift ring} again.
Write $\gotho^r_*$ as shorthand for $\gotho_{\tilde{\calR}^{\inte,r}_*}$;
by Lemma~\ref{L:stable residue2},
for each nonnegative integer $n$, we have isomorphisms
$\gotho^{p^n}_R/y_n \gotho^{p^n}_R \cong \gotho_A$, $\gotho^{p^n}_S/y_n \gotho^{p^n}_S \cong \gotho_B$.

By Proposition~\ref{P:perfect henselian},
$\tilde{\calR}^{\inte,1}_S$ is the object of $\varphi^{-1}$-$\FEt(\tilde{\calR}^{\inte,1}_R)$ corresponding
to $S$. In particular, $\tilde{\calR}^{\inte,1}_S$ is a finite projective $\tilde{\calR}^{\inte,1}_R$-module,
so for some positive integer $m$ there exist morphisms $\tilde{\calR}^{\inte,1}_S \to (\tilde{\calR}^{\inte,1}_R)^m \to \tilde{\calR}^{\inte,1}_S$
of $\tilde{\calR}^{\inte,1}_R$-modules whose composition is the identity.
For a suitable $\lambda \in \gotho_R$ as in Definition~\ref{D:almost},
we may multiply through to obtain morphisms
$\gotho^1_S \to (\gotho^1_R)^m \to \gotho^1_S$ of $\gotho^1_R$-modules whose composition is multiplication
by $[\lambda]$. By applying $\varphi^{-n}$ and then quotienting by $y_n$, we obtain morphisms
$\gotho_B \to \gotho_A^m \to \gotho_B$ of $\gotho_A$-modules
whose composition is multiplication by $\theta([\lambda^{p^{-n}}])$.
Since the norm of $\theta([\lambda^{p^{-n}}])$ tends to 1 as $n \to \infty$, it follows that
$\gotho_B$ is uniformly almost finite projective over $\gotho_A$. Similarly, starting from the perfectness of the
trace pairing on $\tilde{\calR}^{\inte,1}_S$ over $\tilde{\calR}^{\inte,1}_R$, then applying $\varphi^{-n}$,
and finally quotienting by $y_n$, we deduce that the trace pairing defines an almost isomorphism $\gotho_B
\to \Hom_{\gotho_A}(\gotho_B, \gotho_A)$.
\end{proof}

\begin{remark}
The original almost purity theorem of
Faltings \cite{faltings-purity1, faltings-almost}
differs a bit in form from Theorem~\ref{T:almost purity}, in that it refers to a specific construction
to pass from a suitable affinoid algebra over a complete discretely valued field of mixed characteristics
to a perfectoid algebra. We will encounter this construction, which uses toric local coordinates,
in a subsequent paper.

After Faltings introduced the concept of almost purity, and the broader context of almost ring theory,
an abstract framework for such results has been introduced by Gabber and Ramero
\cite{gabber-ramero}, and used by them to establish certain generalizations of
Faltings's almost purity theorem \cite{gabber-ramero-arxiv}.

Theorem~\ref{T:almost purity} appears to be much stronger than the main result of
\cite{gabber-ramero-arxiv}, and the proof is simpler. In place of some complicated analysis
in the style of Grothendieck's proof of Zariski-Nagata purity (as in the original work of Faltings),
the proof of Theorem~\ref{T:almost purity} ultimately rests on the local nature of the perfectoid correspondence.

An independent derivation of Theorem~\ref{T:almost purity}, based on the same set of ideas, has been
given by Scholze \cite{scholze1}. Scholze goes further, extending the perfectoid correspondence
and the almost purity theorem to a certain class of adic analytic spaces; he then uses these
to establish relative versions of the de Rham-\'etale comparison isomorphism in $p$-adic Hodge theory
\cite{scholze2}. We will make contact with the latter results later in this series.
\end{remark}

\section{\texorpdfstring{$\varphi$}{phi}-modules}
\label{sec:phi-modules}

We now introduce $\varphi$-modules over the rings introduced in \S\ref{sec:relative extended}.
In order to avoid some headaches later when working in the relative setting, we expend some energy
here to relate $\varphi$-modules to more geometrically defined concepts, including a relative
analogue of the vector bundles considered by Fargues and Fontaine \cite{fargues-fontaine}.

\setcounter{theorem}{0}
\begin{hypothesis} \label{H:add positive integer}
Throughout \S\ref{sec:phi-modules}, continue to retain Hypothesis~\ref{H:relative extended}.
In addition, let $a$ denote a positive integer, and put $q = p^a$.
\end{hypothesis}

\subsection{\texorpdfstring{$\varphi$}{phi}-modules and \texorpdfstring{$\varphi$}{phi}-bundles}

\begin{defn} \label{D:phi-module relative}
A \emph{$\varphi^a$-module} over $W(R)$ (resp.\ $\tilde{\calE}_R$,
$\tilde{\calR}^{\inte}_R$, $\tilde{\calR}^{\bd}_R$,
$\tilde{\calR}_R$, $\tilde{\calR}_R^+$, $\tilde{\calR}_R^{\infty}$) is a
finite locally free module $M$ equipped with a semilinear $\varphi^a$-action.
For example, one may take the direct sum of one or more copies of the base ring and use the action of
$\varphi^a$ on the ring; any such $\varphi^a$-module is said to be \emph{trivial}.

For $* \in \{\tilde{\calR}^{\inte}, \tilde{\calR}^{\bd}, \tilde{\calR}\}$
and $r>0$, note that any $\varphi^a$-module over $*_R$ descends
uniquely to a finite locally free module $M_r$ over $*^r_R$ equipped with an isomorphism
$(\varphi^a)^* M_r \cong M_r \otimes_{*^r_R} *^{r/q}_R$ of modules over $*^{r/q}_R$. (The argument is by applying $\varphi^{-1}$ as in the proof of
Proposition~\ref{P:perfect henselian}.) We call $M_r$ the \emph{model} of $M$ over $*^r_R$.
\end{defn}

Unfortunately, it is not straightforward to deal with $\varphi^a$-modules
over $\tilde{\calR}_R$ because of the complicated nature of the base ring.
We are thus forced to introduce an auxiliary definition with a more geometric flavor.
\begin{defn}
For $0 < s \leq r/q$, a \emph{$\varphi^a$-module} over $\tilde{\calR}^{[s,r]}_R$
is a finite locally free module $M$ equipped with an isomorphism
$(\varphi^a)^* M \otimes_{\tilde{\calR}^{[s/q,r/q]}_R} \tilde{\calR}^{[s,r/q]}_R
\cong M \otimes_{\tilde{\calR}^{[s,r]}_R} \tilde{\calR}^{[s,r/q]}_R$ of
modules over $\tilde{\calR}^{[s,r/q]}_R$.
A \emph{$\varphi^a$-bundle} over $\tilde{\calR}_R$
consists of a $\varphi^a$-module $M_{I}$ over $\tilde{\calR}^{I}_R$
for every interval $I = [s,r]$ with $0 < s \leq r/q$, together with isomorphisms
$\psi_{I,I'}: M_I \otimes_{\tilde{\calR}^{I}_R} \tilde{\calR}^{I'}_R \cong M_{I'}$
for every pair of intervals $I,I'$ with $I' \subseteq I$, satisfying the cocycle condition
$\psi_{I',I''} \circ \psi_{I,I'} = \psi_{I,I''}$.
We refer to $M_I$ as the \emph{model} of the $\varphi^a$-bundle over $\tilde{\calR}^I_R$;
we may freely pass between $\varphi^a$-modules and models using Lemma~\ref{L:model to bundle} below.
We define base extensions and exact sequences of $\varphi^a$-modules in terms of models.

For $M = \{M_I\}$ a $\varphi^a$-bundle over $\tilde{\calR}_R$ and any interval $I'=[s,r]$ with $0<s<r<+\infty$ (not necessarily satisfying $s\leq r/q$), define $M_{I'}=M_{I}\otimes_{\tilde{\calR}^{I}_R}\tilde{\calR}^{I'}_R$
where $M_{I}$ is a model of $M$ with $I'\subseteq I$; this is independent of the choice of $M_{I}$.
Moreover, it is clear that one gets the isomorphisms $\psi_{I,I'}: M_I \otimes_{\tilde{\calR}^{I}_R} \tilde{\calR}^{I'}_R \cong M_{I'}$
for every pair of intervals $I,I'$
with $I' \subseteq I$, satisfying the cocycle condition, and the
isomorphisms between  $(\varphi^a)^* M_I$ and $M_{I/q}$ which commute with the $\psi_{I,I'}$.
A \emph{global section} of $M$ consists of an element
$\bv_I \in M_I$ for each $I$ such that $\psi_{I,I'}(\bv_I) = \bv_{I'}$; note that $\varphi^a$ acts on
the module of global sections.
\end{defn}

\begin{remark} \label{R:kernels}
The kernel of a surjective morphism
of finite projective modules over a ring is itself a finite projective module.
Consequently, the kernel of a surjective morphism of $\varphi^a$-modules or $\varphi^a$-bundles
(where surjectivity in the latter case means that each map of models is surjective)
is again a $\varphi^a$-module or $\varphi^a$-bundle, respectively.
\end{remark}

The following lemma makes it unambiguous to say that a $\varphi^a$-bundle is generated by a given
finite set of global sections. The proof is loosely modeled on that of \cite[Satz~2.4]{kiehl}.
\begin{lemma} \label{L:module of sections}
Let $M = \{M_I\}$ be a $\varphi^a$-bundle over $\tilde{\calR}_R$.
Suppose that $\bv_1,\dots,\bv_n$ are global sections of $M$ which generate $M_I$ as a module over
$\tilde{\calR}^I_R$ for every closed interval $I \subset (0, +\infty)$.
Then $\bv_1,\dots,\bv_n$ also generate the set of global sections of $M$ as a module over $\tilde{\calR}^\infty_R$.
\end{lemma}
\begin{proof}
For each nonnegative integer $l$,
choose a morphism $\psi_l: M_{[p^{-l},p^l]} \to (\tilde{\calR}^{[p^{-l},p^{l}]}_R)^n$
of $\tilde{\calR}^{[p^{-l},p^l]}_R$-modules whose composition with the map
$(\tilde{\calR}^{[p^{-l},p^l]}_R)^n \to M_{[p^{-l},p^l]}$ defined by $\bv_1,\dots,\bv_n$
is the identity.
By Lemma~\ref{L:finite projective Banach}, we can choose $c_l>0$ such that
the subspace norm on $M_{[p^{-l},p^l]}$ defined by $\psi_{l+1}$ (or rather its base extension
from $\tilde{\calR}^{[p^{-l+1},p^{l+1}]}_R$ to $\tilde{\calR}^{[p^{-l},p^l]}_R$)
and the quotient norm on $M_{[p^{-l},p^l]}$
differ by a multiplicative factor of at most $c_l$.

Given a global section $\bw$ of $M$, we choose elements $a_{il} \in \tilde{\calR}^{[p^{-l},p^l]}_R$,
$b_{il} \in \tilde{\calR}^{\infty}_R$ for $i=1,\dots,n$, $l=0,1,\dots$ as follows.
\begin{itemize}
\item
Given the $b_{ij}$ for $j<l$, use $\psi_l$ to construct $a_{il}$ so that $\bw - \sum_i \sum_{j<l} b_{ij} \bv_i = \sum_i a_{il} \bv_i$.
\item
Given the $a_{il}$, choose the $b_{il}$ so that $\lambda(\alpha^t)(b_{il} - a_{il}) \leq p^{-1} c_l^{-1}
\lambda(\alpha^t)(a_{il})$ for $i=1,\dots,n$ and $t \in [p^{-l},p^l]$.
\end{itemize}
Note that for $t \in [p^{-l},p^l]$, we have $\max_i \{\lambda(\alpha^t)(a_{i(l+1)})\} \leq p^{-1} \max_i \{\lambda(\alpha^t)(a_{il})\}$. Consequently, the series $\sum_l a_{il}$ converges to a limit $a_i \in \tilde{\calR}^\infty_R$
satisfying $\bw = \sum_i a_i \bv_i$; this proves the claim.
\end{proof}

\begin{lemma} \label{L:model to bundle}
For $0 < s \leq r/q$, the projection functor from
$\varphi^a$-bundles over $\tilde{\calR}_R$ to $\varphi^a$-modules over $\tilde{\calR}^{[s,r]}_R$
is a tensor equivalence.
\end{lemma}
\begin{proof}
For each nonnegative integer $n$,
we may uniquely lift a $\varphi^a$-module over $\tilde{\calR}^{[s,r]}_R$ to
$\tilde{\calR}^{[sq^{-n},rq^{n}]}_R$ by pulling back along positive and
negative powers of $\varphi^a$, then glueing using Theorem~\ref{T:Kiehl for interval}.
The claim follows at once.
\end{proof}

\begin{remark} \label{R:module to bundle}
There is a natural functor from $\varphi^a$-modules over $\tilde{\calR}_R$ to
$\varphi^a$-bundles over $\tilde{\calR}_R$: given a $\varphi^a$-module over $\tilde{\calR}_R$,
form its model $M_r$ over $\tilde{\calR}^r_R$ and then base extend to $\tilde{\calR}^{[s,r]}_R$
to obtain a $\varphi^a$-module over $\tilde{\calR}^{[s,r]}_R$.
We may recover $M_r$ as the set of $(0,r]$-sections of the resulting $\varphi^a$-bundle,
so this functor is fully faithful.
It also turns out to be essentially surjective; see Theorem~\ref{T:vector bundles} below.
\end{remark}

\subsection{Construction of \texorpdfstring{$\varphi$}{phi}-invariants}

We next introduce some calculations that allow us to construct $\varphi^a$-invariants.
These are relative
analogues of results from \cite[\S 4]{kedlaya-revisited} which were used as part of the construction
of slope filtrations.

\begin{defn}
For $M$ a $\varphi^a$-module or $\varphi^a$-bundle and $n \in \ZZ$, define the twist
$M(n)$ of $M$ to be the same underlying module or bundle with the $\varphi^a$-action
multiplied by $p^{-n}$.
\end{defn}

\begin{prop} \label{P:H1}
Let $M = \{M_I\}$ be a $\varphi^a$-bundle over $\tilde{\calR}_R$.
Then there exists an integer $N$ such that for
$n \geq N$ and $0 < s \leq r$, the map
$\varphi^a - 1: M_{[s,rq]}(n) \to M_{[s,r]}(n)$ is surjective. Moreover, if $M$ arises from a
$\varphi^a$-module over $\tilde{\calR}^{\inte}_R$, we may take $N = 1$.
\end{prop}
\begin{proof}
We first assume that $r/s \leq q^{1/2}$; by applying a suitable power of $\varphi^a$ as needed, we may also reduce to the
case where $r \in [1,q]$. Choose module generators $\bv_1,\dots,\bv_m$ of $M_{[s/q, rq]}$
and representations $\varphi^{-a}(\bv_j) = \sum_i A_{ij} \bv_i$,
$\varphi^a(\bv_j) = \sum_i B_{ij} \bv_i$ with $A_{ij} \in \tilde{\calR}^{[s,rq]}_R$,
$B_{ij} \in \tilde{\calR}^{[s/q,r]}_R$.
Put
\[
c_1 = \sup\{\lambda(\alpha^t)(A): t \in [s,rq]\}, \qquad
c_2 = \sup\{\lambda(\alpha^t)(B): t \in [s/q,r]\}.
\]
We take $N$ large enough so that
\begin{equation} \label{eq:choose N to open range}
p^{-N} c_1 < 1, \qquad
p^{N(1-q^{1/2})} c_1^{q^{1/2}} c_2 < 1;
\end{equation}
note that we may take $N=1$ if $M$ arises from a $\varphi^a$-module over $\tilde{\calR}^{\inte}_R$, as then we can ensure that $c_1 = c_2 = 1$.
The choice of $N$ ensures that for $n \geq N$, we can choose $c \in (0,1)$ so that
\begin{equation} \label{eq:choose c in range}
\epsilon =
\max\{p^{-n} c_1 c^{-(q-1)r/q},
p^n c_2 c^{(q-1)s}\} < 1.
\end{equation}
Given $(x_1,\dots,x_m) \in (\tilde{\calR}^{[s,r]}_R)^m$,
apply Lemma~\ref{L:decompose variant} to write $x_i = y_i + z_i$
with $y_i \in \tilde{\calR}^{[s/q,r]}_R$, $z_i \in \tilde{\calR}^{[s,rq]}_R$
such that for $t \in [s,r]$,
\begin{align*}
\lambda(\alpha^t)(y_i), \lambda(\alpha^t)(z_i) &\leq \lambda(\alpha^t)(x_i), \\
\lambda(\alpha^t)(\varphi^{-a}(y_i)) &\leq c^{-(q-1)t/q} \lambda(\alpha^t)(x_i), \\
\lambda(\alpha^t)(\varphi^a(z_i)) &\leq c^{(q-1)t} \lambda(\alpha^t)(x_i).
\end{align*}
Put
\[
x'_i = p^{n} \sum_j A_{ij} \varphi^{-a}(y_j) + p^{-n} \sum_j B_{ij} \varphi^a(z_j),
\]
so that
\[
\sum_i x'_i \bv_i = p^{n} \varphi^{-a}\left(\sum_i y_i \bv_i\right) + p^{-n} \varphi^a\left(\sum_i z_i \bv_i\right)
\]
and
\[
\max_i \{\lambda(\alpha^t)(x'_i)\} \leq \epsilon \max_i \{\lambda(\alpha^t)(x_i)\}
\qquad (t \in [s,r]).
\]
Let us view $y = (y_1,\dots,y_m)$, $z = (z_1,\dots,z_m)$, and $x' = (x'_1,\dots,x'_m)$
as functions of $x = (x_1,\dots,x_m)$.
Given $x_{(0)} \in (\tilde{\calR}^{[s,r]}_R)^m$, define $x_{(l+1)} = x'(x_{(l)})$,
and put
\begin{equation} \label{eq:produce horizontal element}
\bv = \sum_{l=0}^\infty \left( -p^n \varphi^{-a}\left(\sum_i y(x_{(l)})_i \bv_i\right) + \sum_i z(x_{(l)})_i \bv_i \right).
\end{equation}
This series converges to an element of $M_{[s, rq]}$ satisfying $\bv - p^{-n} \varphi^a(\bv) =\bw$ for
$\bw = \sum_{i=1}^m x_{(0),i} \bv_i \in M_{[s, r]}$. More precisely, from the choice of the $y_i$ and $z_i$,
$\sum_l \varphi^{-a} (y(x_{(l)})_i)$ converges under $\lambda(\alpha^t)$ in the ranges
$t \in [s,r]$ and $t \in [sq,rq]$, hence for $t \in [s,rq]$ by Lemma~\ref{L:hadamard relative}; a similar argument
applies to $\sum_l z(x_{(l)})_i$.

For general $r,s$, given $\bw \in M_{[s,r]}$,
the previous paragraph produces $\bv \in M_{[t,rq]}$ with $t = \max\{rq^{-1/2},s\}$
such that $\bv - p^{-n} \varphi^a(\bv) =\bw$.
By rewriting this equation as $\bw + p^{-n} \varphi^a(\bv) = \bv$
and invoking Lemma~\ref{L:intersection}, we find that
$\bv \in M_{[t',r]} \cap M_{[t,rq]} = M_{[t',rq]}$ for $t' = \max\{t/q, s\}$.
Repeating this argument, we eventually obtain $\bv \in M_{[s,rq]}$ as desired.
\end{proof}
\begin{cor} \label{C:exact invariants}
Let $0 \to M_1 \to M \to M_2 \to 0$ be an exact sequence of $\varphi^a$-bundles over $\tilde{\calR}_R$. Then there
exists an integer $N$ such that for
$n \geq N$, the sequence
\[
0 \to M_1(n)^{\varphi^a} \to M(n)^{\varphi^a} \to M_2(n)^{\varphi^a} \to 0
\]
is again exact.
\end{cor}
\begin{proof}
Apply Proposition~\ref{P:H1} and the snake lemma.
\end{proof}

\begin{prop} \label{P:global invariants}
Let $M = \{M_I\}$ be a $\varphi^a$-bundle over $\tilde{\calR}_R$.
For $N$ as in Proposition~\ref{P:H1} and $n \geq N$,
there exist finitely many $\varphi^a$-invariant global sections of
$M(n)$ which generate $M$.
(If $M$ is obtained from a $\varphi^a$-module over $\tilde{\calR}^{\inte}_R$ generated by $m$ elements, then we may take $N=1$
and use only $2m$ global sections.)
\end{prop}
\begin{proof}
Pick any $r>0$, and set notation as in the proof of Proposition~\ref{P:H1} with $s = rq^{-1/2}$.
By hypothesis, $R$ is a Banach algebra over some analytic field $L$; choose 
$\overline{\pi} \in L$ with $0 < \alpha(\overline{\pi}) < 1$.
For any $n \geq N$, we can find a positive rational number $u \in \ZZ[p^{-1}]$
so that $c = \alpha(\overline{\pi}^{u})$ satisfies \eqref{eq:choose c in range}.
For $i=1,\dots,m$, define $\bw_i$ to be the sum of a series as in \eqref{eq:produce horizontal element}
in which
\[
x_{(0)} = 0, \qquad (y(x_{(0)})_j, z(x_{(0)})_j) = \begin{cases}
(-[\overline{\pi}^{u}], [\overline{\pi}^{u}]) & (j = i) \\
(0,0) & (j \neq i);
\end{cases}
\]
this gives an element of $M_{[rq^{-1/2}, rq]}$ killed by $\varphi^a-1$, and hence a
$\varphi^a$-invariant global section of $M$. We can write $\bw_j = [\overline{\pi}^s]\bv_j + \sum_i X_{ij} \bv_i$ with
$\lambda(\alpha^t)(X_{ij}) \leq \epsilon \alpha(\overline{\pi})^{st}$ for all $i,j$
and all $t \in [rq^{-1/2},r]$.
It follows that the matrix $1 + X$ is invertible over $\tilde{\calR}^{[rq^{-1/2},r]}_{R}$,
so $\bw_1,\dots,\bw_m$ generate $M_{[rq^{-1/2},r]}$.

By repeating the argument with $r$ replaced by $rq^{-1/2}$, we obtain $\varphi^a$-invariant global sections
$\bw'_1,\dots,\bw'_m$ which generate
$M_{[rq^{-1},rq^{-1/2}]}$.
By applying powers of $\varphi^a$ and invoking 
Lemma~\ref{L:finite generation} and
Lemma~\ref{L:Robba ring cover},
we see that $\bw_1, \dots, \bw_m, \bw'_1,\dots,\bw'_m$ generate $M_{I}$ for any $I$.
\end{proof}

\begin{remark} \label{R:no direct invariants}
One cannot hope to refine the calculation in Proposition~\ref{P:H1} to cover the entire interval
$[r/q,r]$ and thus prove Proposition~\ref{P:global invariants} in one step. This approach is obstructed
by the following observation: take $R = L$ and obtain $M$ from a trivial $\varphi^a$-module
with $\varphi^a$-invariant basis $\bv_1,\dots,\bv_m$.
If it were possible to refine the construction in Proposition~\ref{P:H1} so that the
elements $\bw_1,\dots,\bw_m$ produced in Proposition~\ref{P:global invariants} were generators of $M$,
we would have produced two isomorphic $\varphi$-modules with different degrees, a contradiction.
\end{remark}

\begin{remark} \label{R:global invariants free of trivial spectrum}
The proof of Proposition~\ref{P:global invariants} uses in a crucial way the running hypothesis that $R$ is a Banach algebra over an analytic field. However, the proof can be modified to also treat the case where $R$ is free of trivial spectrum (see 
Remark~\ref{R:small norm elements}): one produces elements which generate the fiber of $M$ over a point of $\calM(R)$, notes that these also generate the fibers in some neighborhood, and argues by compactness. Similarly, the results of \S\ref{subsec:vector bundles} can be extended to the case where $R$ is free of trivial spectrum.
\end{remark}

\subsection{Vector bundles \`a la Fargues-Fontaine}
\label{subsec:vector bundles}

We now make contact with the new perspective on $p$-adic Hodge theory provided by the
work of Fargues and Fontaine \cite{fargues-fontaine} (see Remark~\ref{R:vector bundles}). After we introduce adic spaces, we will be able to restate these results: see \S\ref{subsec:relative FF}.

\begin{defn} \label{D:vector bundles}
Define the reduced graded ring $P = \oplus_{n=0}^\infty P_n$ by
\[
P_n = \{x \in \tilde{\calR}^+_R: \varphi^a(x) = p^n x \}
= \{x \in \tilde{\calR}_R: \varphi^a(x) = p^n x \}
\qquad (n=0,1,\dots).
\]
The last equality holds by Corollary~\ref{C:twisted invariants}.
(We will write $P_R$ instead of $P$ in case it becomes necessary to specify $R$.)
For $d > 0$ and $f \in P_d$, let $P[f^{-1}]_0$ denote the degree zero subring of $P[f^{-1}]$;
the affine schemes $D_+(f) = \Spec(P[f^{-1}]_0)$ glue to define a reduced separated scheme $\Proj(P)$
as in \cite[Proposition~2.4.2]{ega2}. The points of $\Proj(P)$ may be identified with
the homogeneous prime ideals of $P$ not containing $P_+ = \oplus_{n>0} P_n$,
with $D_+(f)$ consisting of those ideals not containing $f$.
\end{defn}

\begin{defn}
For $f \in P_d$ for some $d>0$,
for $M = \{M_I\}$ a $\varphi^a$-bundle over $\tilde{\calR}_R$, define
\begin{equation} \label{eq:union of twists}
M_f = \bigcup_{n \in \ZZ} f^{-n} M(dn)^{\varphi^a}
\end{equation}
as a module over $P[f^{-1}]_0$.
(In other words, $M_f = M[f^{-1}]^{\varphi^a}$.)
For any closed interval $I \subset (0, +\infty)$,
we have a natural map
\begin{equation} \label{eq:natural map}
M_f \otimes_{P[f^{-1}]_0} \tilde{\calR}^I_R[f^{-1}] \to M_I \otimes_{\tilde{\calR}^I_R} \tilde{\calR}^I_R[f^{-1}].
\end{equation}
\end{defn}

\begin{lemma} \label{L:f surjective}
Choose $f \in P_d$ for some $d>0$.
Let $0 \to M_1 \to M \to M_2 \to 0$ be a short exact sequence of
$\varphi^a$-bundles over $\tilde{\calR}_R$. Then the sequence
\begin{equation} \label{eq:f surjective}
0 \to M_{1,f} \to M_f \to M_{2,f} \to 0
\end{equation}
is exact.
\end{lemma}
\begin{proof}
This follows from \eqref{eq:union of twists} and Corollary~\ref{C:exact invariants}.
\end{proof}
\begin{cor} \label{C:f sums}
For $f \in P_d$ for some $d>0$, $M$ a $\varphi^a$-bundle over $\tilde{\calR}_R$, and
$M_1, M_2$ two $\varphi^a$-subbundles of $M$ for which $M_1 + M_2$ is again a $\varphi^a$-subbundle,
the natural inclusion $M_{1,f} + M_{2,f} \to (M_1 + M_2)_f$ within $M_f$ is an equality.
\end{cor}
\begin{proof}
Apply Lemma~\ref{L:f surjective} to the surjection $M_1 \oplus M_2 \to M_1 + M_2$
(after extending this to an exact sequence using Remark~\ref{R:kernels}).
\end{proof}

\begin{defn} \label{D:Prufer domain}
For $A$ an integral domain, the following conditions are equivalent
\cite[IV.2, Exercise 12]{bourbaki-ac}.
\begin{enumerate}
\item[(a)]
Every finitely generated ideal of $A$ is projective.
\item[(b)]
Every finitely generated torsion-free module over $A$ is projective.
\item[(c)]
For all ideals $I_1,I_2,I_3$ of $A$, the inclusion $I_1 \cap I_2 + I_1 \cap I_3 \to I_1 \cap (I_2 + I_3)$
is an equality. Note that it is sufficient to test finitely generated ideals.
\end{enumerate}
An integral domain satisfying any of these conditions is called a \emph{Pr\"ufer domain}.
For example, any B\'ezout domain is a Pr\"ufer domain (but not conversely).
Note that a noetherian Pr\"ufer domain is a Dedekind domain, analogously to the fact that
a noetherian B\'ezout domain is a principal ideal domain.
\end{defn}

\begin{lemma} \label{L:Prufer domain}
Suppose that $R = L$ is an analytic field.
Then for $f \in P_d$ for some $d>0$, the ring $P[f^{-1}]_0 = (\tilde{\calR}_L[f^{-1}])^{\varphi^a}$
is a Pr\"ufer domain.
\end{lemma}
\begin{proof}
We check criterion (c) of Definition~\ref{D:Prufer domain}.
Let $I_1, I_2, I_3$ be three finitely generated ideals of $P[f^{-1}]_0$.
For $j = 1,2,3$, $I_j \otimes_{P[f^{-1}]_0} \tilde{\calR}_L[f^{-1}]$
can be generated by a finite set of elements $x_{j,1},\dots,x_{j,m} \in \tilde{\calR}_L$
such that for each $i$, $\varphi^a(x_{j,i}) = p^h x_{j,i}$ for some $h \in \ZZ$.
Let $M_j$ be the ideal of $\tilde{\calR}_L$ generated by $x_{j,1},\dots,x_{j,m}$;
it is principal (because $\tilde{\calR}_L$ is a B\'ezout domain by Lemma~\ref{L:Bezout domain}) and $\varphi^a$-stable
(because each $x_{j,i}$ generates a $\varphi^a$-stable ideal), and hence a $\varphi^a$-module
over $\tilde{\calR}_L$. Since $\tilde{\calR}_L$ is a B\'ezout domain and hence a Pr\"ufer domain,
by Definition~\ref{D:Prufer domain}, the inclusion
\[
M_1 \cap M_2 + M_1 \cap M_3 \to M_1 \cap (M_2 + M_3)
\]
is surjective. By Lemma~\ref{L:f surjective}, the map
\[
(M_1 \cap M_2 + M_1 \cap M_3)_f \to (M_1 \cap (M_2 + M_3))_f
\]
is also surjective. The operation $M \mapsto M_f$ clearly distributes across
intersections; it also distributes across sums thanks to Corollary~\ref{C:f sums}
and the fact that the sum of two $\varphi^a$-submodules of $\tilde{\calR}_L$ is again
a $\varphi^a$-submodule (by the B\'ezout property again).
By identifying $I_j$ with $M_{j,f}$, we verify criterion (c) of Definition~\ref{D:Prufer domain},
so $P[f^{-1}]_0$ is a Pr\"ufer domain as desired.
\end{proof}

\begin{lemma} \label{L:ideal-trivial}
The following statements are true.
\begin{enumerate}
\item[(a)]
For any $d>0$, $P_d$ generates the unit ideal in
$\tilde{\calR}^\infty_R$. In particular, we may choose $f_1,\dots,f_m \in P_d$ which generate the unit ideal in
$\tilde{\calR}^\infty_R$.
\item[(b)]
For any such elements,
the ideal in $P$ generated by $f_1,\dots,f_m$ is saturated (i.e., its radical equals $P_+$).
Consequently, the schemes $D_+(f_1), \dots, D_+(f_m)$ cover $\Proj(P)$,
so $\Proj(P)$ is quasicompact.
\end{enumerate}
\end{lemma}
\begin{proof}
To prove (a), apply Proposition~\ref{P:global invariants} and Lemma~\ref{L:module of sections}
to $\tilde{\calR}^\infty_R$ viewed as a trivial $\varphi^a$-bundle.
To prove (b), choose any homogeneous $f \in P_+$, then apply Lemma~\ref{L:f surjective}
to show that the map $P[f^{-1}]_0^m \to P[f^{-1}]_0$ defined by
$f_1,\dots,f_m$ contains 1 in its image. We then
obtain an expression for some power of $f$ as an element of the ideal of $P$ generated by $f_1,\dots,f_m$,
proving the claim.
\end{proof}

\begin{lemma} \label{L:pick out ideal}
Choose $r>0$ and $f \in P_d$ for some $d>0$.
Let $\gothp$ be any maximal ideal of $P[f^{-1}]_0$, and let $\gothq$ be the corresponding
homogeneous prime ideal of $P$ not containing $f$.
\begin{enumerate}
\item[(a)]
The ideal in $\tilde{\calR}^{[r/q,r]}_R$ generated by $\gothq$ and $f$ is trivial.
\item[(b)]
The ideal in $\tilde{\calR}^{[r/q,r]}_R$ generated by $\gothq$ is not trivial.
\end{enumerate}
\end{lemma}
\begin{proof}
The homogeneous ideal in $P$ generated by $\gothq$ and $f$ contains $P_{dn}$
for some $n>0$. By Lemma~\ref{L:ideal-trivial}, $\gothq$ and $f$ generate the trivial ideal in
$\tilde{\calR}^{\infty}_R$ and hence also in $\tilde{\calR}^{[r/q,r]}_R$. This yields (a).

Suppose now that $\gothq$ contains elements $g_1,\dots,g_m$ which generate the unit ideal in $\tilde{\calR}^{[r/q,r]}_R$.
We may as well assume $g_1,\dots,g_m \in P_{dn}$ for some $n>0$; then
these elements define a map $(\tilde{\calR}^\infty_R(-dn))^{\oplus m} \to \tilde{\calR}^\infty_R$
of $\varphi$-modules. This map is surjective by Lemma~\ref{L:model to bundle};
by Remark~\ref{R:kernels} and Lemma~\ref{L:f surjective}, we again get a surjective map upon inverting $f$
and taking $\varphi$-invariants. But this implies that $f \in \gothq$, a contradiction. This yields (b).
\end{proof}

\begin{theorem} \label{T:phi-module to vb}
Choose $f \in P_d$ for some $d>0$.
Let $M = \{M_I\}$ be a $\varphi^a$-bundle over $\tilde{\calR}_R$. Then $M_f$ is a finite projective module
over $P[f^{-1}]_0$ and \eqref{eq:natural map} is bijective for every interval $I$.
\end{theorem}
\begin{proof}
Apply Proposition~\ref{P:global invariants}
to construct an integer $n$ and a finite set $\bw_1,\dots,\bw_m$ of $\varphi^a$-invariant
global sections of $M(dn)$ which generate $M$.
We may then view $f^{-n} \bw_1,\dots,f^{-n}\bw_m$ as elements of $M_f$;
this implies that \eqref{eq:natural map} is surjective.

Let $M'$ be the $\varphi^a$-bundle associated to the $\varphi^a$-module $(\tilde{\calR}_R(-dn))^{\oplus m}$,
so that $\bw_1,\dots,\bw_m$ define a surjection $M'(dn) \to M(dn)$ and hence a surjection $M' \to M$.
By Remark~\ref{R:kernels}, we obtain an exact sequence $0 \to M'' \to M' \to M \to 0$ of $\varphi^a$-bundles.
By Lemma~\ref{L:f surjective}, the sequence
\begin{equation} \label{eq:phi-module to vb}
0 \to M''_f \to M'_f \to M_f \to 0
\end{equation}
is exact. Consequently, $M_f$ is a finitely generated
module over $P[f^{-1}]_0$. For any interval $I$, we obtain a commuting diagram
\[
\xymatrix{
& M''_f \otimes_{P[f^{-1}]_0} \tilde{\calR}^I_R[f^{-1}]
\ar[r] \ar[d]& M'_f \otimes_{P[f^{-1}]_0} \tilde{\calR}^I_R[f^{-1}]
\ar[r] \ar[d] & M_f \otimes_{P[f^{-1}]_0} \tilde{\calR}^I_R[f^{-1}]
\ar[r] \ar[d] & 0 \\
0 \ar[r] & M''_I \otimes_{\tilde{\calR}^I_R} \tilde{\calR}^I_R[f^{-1}]\ar[r] & M'_I
\otimes_{\tilde{\calR}^I_R} \tilde{\calR}^I_R[f^{-1}]
 \ar[r] & M_I \otimes_{\tilde{\calR}^I_R} \tilde{\calR}^I_R[f^{-1}] \ar[r] & 0
}
\]
with exact rows. (The left exactness in the last row follows from the exactness of localization.)
Since the left vertical arrow is surjective (by the first part of the proof) and the middle vertical arrow is bijective,
by the five lemma the right vertical arrow is injective. Hence \eqref{eq:natural map} is a bijection.
We may also repeat the arguments with $M$ replaced by $M''$ to deduce that $M_f$ is finitely presented.

The exact sequence
$0 \to M'' \to M' \to M \to 0$ corresponds to an element of $H^1_{\varphi^a}(M^\dual \otimes M'')$.
By Proposition~\ref{P:H1}, for $m$ sufficiently large, we have $H^1_{\varphi^a}(M^\dual \otimes M''(dm)) = 0$. That is,
if we form the commutative diagram
\[
\xymatrix{
0 \ar[r] & M'' \ar[r] \ar[d] & M' \ar[r] \ar[d] & M \ar[r] \ar[d] & 0 \\
0 \ar[r] & f^{-m} M'' \ar[r] & N \ar[r] & M \ar[r] & 0
}
\]
by pushing out, then the exact sequence in the bottom row splits in the category of $\varphi^a$-bundles.
By Lemma~\ref{L:f surjective}, we obtain a split exact sequence
\[
0 \to (f^{-m} M'')_f \to N_f \to M_f \to 0
\]
of modules over $P[f^{-1}]_0$; however, $M'[f^{-1}] \cong N[f^{-1}]$ and so $M'_f = M'[f^{-1}]^\varphi \cong  N[f^{-1}]^\varphi = N_f$.
Since the construction of $M'$ guarantees that $M'_f$ is a free module over $P[f^{-1}]_0$, the same is
true of $N_f$; it follows that $M_f$ is a projective module over $P[f^{-1}]_0$, as desired.
\end{proof}

Theorem~\ref{T:phi-module to vb} can be reformulated in terms of an equivalence of categories between
$\varphi^a$-bundles and vector bundles on $\Proj(P)$.
\begin{defn} \label{D:sheaf to module}
Let $V$ be a quasicoherent finite locally free sheaf on $\Proj(P)$. For each homogeneous $f \in P_+$, form the
module $\Gamma(D_+(f), V) \otimes_{P[f^{-1}]_0} \tilde{\calR}^\infty_R[f^{-1}]$.
Since $P_+$ generates the unit ideal in $\tilde{\calR}^\infty_R$ by Lemma~\ref{L:ideal-trivial},
we may glue to obtain a quasicoherent finite locally free sheaf on $\Spec(\tilde{\calR}^\infty_R)$.
Let $M(V)$ be the module of global sections of this sheaf; then $V \rightsquigarrow M(V)$ defines an exact functor
from quasicoherent finite locally free sheaves on
$\Proj(P)$ to $\varphi^a$-modules over $\tilde{\calR}^\infty_R$.
\end{defn}

\begin{defn} \label{D:module to sheaf}
Let $M$ be a $\varphi^a$-bundle over $\tilde{\calR}_R$.
By Theorem~\ref{T:phi-module to vb}, for each $f \in P_+$,  $M_f$ is
a finite locally free module over $P[f^{-1}]_0$ and the natural map \eqref{eq:natural map}
is an isomorphism. In particular, the $M_f$ glue to define a quasicoherent finite locally free sheaf
$V(M)$ on $\Proj(P)$ (which one might call a \emph{vector bundle} on $\Proj(P)$).
\end{defn}

\begin{theorem} \label{T:vector bundles}
The following tensor categories are equivalent.
\begin{enumerate}
\item[(a)]
The category of quasicoherent finite locally free sheaves on $\Proj(P)$.
\item[(b)]
The category of $\varphi^a$-modules over $\tilde{\calR}^\infty_R$.
\item[(c)]
The category of $\varphi^a$-modules over $\tilde{\calR}_R$.
\item[(d)]
The category of $\varphi^a$-bundles over $\tilde{\calR}_R$.
\end{enumerate}
More precisely, the functor from (a) to (b) is the functor $V \rightsquigarrow M(V)$ given in
Definition~\ref{D:sheaf to module}, the functor from (b) to (c) is base extension,
the functor from (c) to (d) is the one indicated in
Remark~\ref{R:module to bundle}, and the functor from (d)
to (a) is the functor $M \rightsquigarrow V(M)$ given in Definition~\ref{D:module to sheaf}.
\end{theorem}
\begin{proof}
This is immediate from Theorem~\ref{T:phi-module to vb}.
\end{proof}

\begin{cor} \label{C:phi-modules glueing}
For any rational covering $\{(R,R^+) \to (R_i,R_i^+)\}_i$,
the morphism $R \to \oplus_i R_i$ is an effective descent morphism for $\varphi^a$-modules over $\tilde{\calR}_*$. 
\end{cor}
\begin{proof}
This is immediate from Theorem~\ref{T:Kiehl for Robba2} and Theorem~\ref{T:vector bundles}.
\end{proof}

It is worth mentioning the following refinement of Theorem~\ref{T:vector bundles} in the case of an analytic field.
\begin{theorem}  \label{T:vector bundles field}
Suppose that $R = L$ is an analytic field. Then the construction of
Definition~\ref{D:sheaf to module} defines an equivalence of categories between
the category of coherent sheaves on $\Proj(P)$ and the category of finitely presented
$\tilde{\calR}_L$-modules equipped with semilinear $\varphi^a$-actions.
\end{theorem}
\begin{proof}
Since $\tilde{\calR}_L$ is a B\'ezout domain by Lemma~\ref{L:Bezout domain} and hence a Pr\"ufer domain,
any finitely presented module over $\tilde{\calR}_L$ is automatically coherent.
The claim thus follows from Theorem~\ref{T:vector bundles}.
\end{proof}

\begin{remark}
The obstruction to generalizing Theorem~\ref{T:vector bundles field} is that
it is unclear whether the rings
$P[f^{-1}]_0$ have the property that every finitely generated ideal is finitely presented
(i.e., whether these rings are \emph{coherent}).
This is most likely not true in general; however,
we do not know what to expect if $R$ is restricted to being the completed perfection of an affinoid
algebra over an analytic field.
\end{remark}

\begin{remark} \label{R:not coherent}
One can improve the formal analogy between $\Proj(P)$ and the projective line over a field
by defining $\calO(n)$ for $n \in \ZZ$ as the invertible sheaf on $\Proj(P)$
corresponding via Theorem~\ref{T:vector bundles} to the $\varphi^a$-module over $\tilde{\calR}_R$
free on one generator $\bv$ satisfying $\varphi^a(\bv) = p^{-n} \bv$.
For $V$ a quasicoherent sheaf on $\Proj(P)$, write $V(n)$ for $V \otimes_{\calO} \calO(n)$; we may then naturally identify $M(V(n))$ with $M(V)(n)$. For $V$ a quasicoherent finite locally free sheaf on $\Proj(P)$,
$V(n)$ is generated by finitely many global sections for $n$ large (by
Theorem~\ref{T:vector bundles} and Proposition~\ref{P:global invariants}).
For a vanishing theorem for $H^1$ in the same vein, see Corollary~\ref{C:vanishing of H1}.
\end{remark}

It will be useful for subsequent developments to explain how to add topologies to
both types of objects appearing in Theorem~\ref{T:vector bundles}.
\begin{lemma} \label{L:graded Banach norm}
Let $M = \{M_I\}$ be a $\varphi^a$-bundle over $\tilde{\calR}_R$.
Choose $r>0$, and induce from $\lambda(\alpha^r)$ a norm on $M_{[r/q,r]}$ as in
Lemma~\ref{L:finite projective Banach}. Then for each $n \in \ZZ$,
the equivalence class of the restriction of this norm
to $\{\bv \in M_{[r/q,r]}: \varphi^a(\bv) = p^{-n} \bv\} \cong M(n)^{\varphi^a}$
is independent of $r$ and of the choice of the norm on $M$.
(However, the construction is not uniform in $n$.)
\end{lemma}
\begin{proof}
It is clear that the choice of the norm on $M$ makes no difference up to equivalence,
so we need only check the dependence on $r$.
For any $s \in (0,r/q]$, by fixing a set of generators for $M_{[s,r]}$, we obtain norms
$|\cdot|_t$ induced by $\lambda(\alpha^t)$ for all $t \in [s,r]$.
We can choose $c_1, c_2 > 0$ so that these norms
satisfy $c_1 |\bv|_t \leq |\varphi^a(\bv)|_{t/q} \leq c_2 |\bv|_t$ for all $\bv \in  M_{[s,r]}$
and all $t \in [sq,r]$.

For $\bv \in M_{[r/q,r]}$ with $\varphi^a(\bv) = p^{-n} \bv$, we have
$|\varphi^a(\bv)|_{r/q} = p^{n} |\bv|_{r/q}$. Consequently, $|\cdot|_r$ and $|\cdot|_{r/q}$
have equivalent restrictions. By induction, we see that $r$ and $rq^{m}$ give equivalent
norms for any $m \in \ZZ$.
To complete the proof, it is enough to observe that for $t$ in the interval between $r$
and $rq^{m}$ (inclusive),
we have $|\cdot|_t \leq \max\{|\cdot|_r, |\cdot|_{rq^{m}} \}$ by Lemma~\ref{L:hadamard relative};
this then implies that the norm induced by $r$ dominates the norm induced by $s$,
and vice versa by symmetry.
\end{proof}

\begin{defn} \label{D:continuous action on vector bundle}
Let $G$ be a profinite group acting continuously on $\tilde{\calR}^r_R$ for each $r >0$
and commuting with $\varphi^a$; then $G$ also acts on $P$ (but not continuously).
Let $M$ be a $\varphi^a$-module over $\tilde{\calR}_R$, and apply Theorem~\ref{T:vector bundles}
to construct a corresponding quasicoherent finite locally free sheaf $V$ on $\Proj(P)$.
Then the following conditions on an action of $G$ on $V$ (or equivalently on $M$) are equivalent.
\begin{enumerate}
\item[(a)]
The action of $G$ on $M$ is continuous for the LF topology.
\item[(b)]
For each $n \in \ZZ$, the action of $G$ on $M(n)^{\varphi^a} = \Gamma(\Proj(P), V(n))$ is continuous
for any norm as in Lemma~\ref{L:graded Banach norm}.
(This implies (a) by Proposition~\ref{P:global invariants}.)
\end{enumerate}
If these equivalent conditions are satisfied, we say the action is \emph{continuous}.
\end{defn}

We record a consequence of Theorem~\ref{T:vector bundles} for the cohomology of $\varphi$-modules.
\begin{prop} \label{P:truncate cohomology1}
Let $M$ be a $\varphi^a$-module over $\tilde{\calR}^\infty_R$.
\begin{enumerate}
\item[(a)]
For $r>0$, put $M_r = M \otimes_{\tilde{\calR}^\infty_R} \tilde{\calR}^r_R$. Then the vertical arrows in the diagram
\[
\xymatrix{
0 \ar[r] & M \ar^{\varphi^a-1}[r] \ar[d] & M \ar[r] \ar[d] & 0 \\
0 \ar[r] & M_r \ar^{\varphi^a-1}[r]  & M_{r/q} \ar[r] & 0
}
\]
induce an isomorphism on the cohomology of the horizontal complexes.
In particular, the lower complex computes $H^i_{\varphi^a}(M)$.
\item[(b)]
The map $M \to M \otimes_{\tilde{\calR}^\infty_R} \tilde{\calR}_R$ induces an isomorphism on cohomology.
\item[(c)]
For $r,s$ with $0 < s \leq r/q$, put $M_{[s,r]} = M \otimes_{\tilde{\calR}^\infty_R} \tilde{\calR}^{[s,r]}_R$. Then the vertical arrows in the diagram
\[
\xymatrix{
0 \ar[r] & M \ar^{\varphi^a-1}[r] \ar[d] & M \ar[r] \ar[d] & 0 \\
0 \ar[r] & M_{[s,r]} \ar^{\varphi^a-1}[r]  & M_{[s,r/q]} \ar[r] & 0
}
\]
induce an isomorphism on the cohomology of the horizontal complexes.
In particular, the lower complex computes $H^i_{\varphi^a}(M)$.
\item[(d)]
In (c), the map $\varphi^a-1: M_{[s,r]} \to M_{[s,r/q]}$ is strict.
Consequently, for $i=0,1$, $H^i_{\varphi^a}(M)$ admits the structure of a Banach space over $\Qp$.
\end{enumerate}
\end{prop}
\begin{proof}
Note that (b) follows from (a) by taking direct limits, so we need only treat (a) and (c).
Write $H^i_{\varphi^a}(M_r)$ and $H^i_{\varphi^a}(M_{[s,r]})$ as shorthand for the kernel and cokernel of the second row in (a)
and (c), respectively.
Since $M$ is finite projective over $\tilde{\calR}^\infty_R$, the maps $M \to M_r$, $M \to M_{[s,r]}$
are injective; consequently, the maps $H^0_{\varphi^a}(M) \to H^0_{\varphi^a}(M_r)$,
$H^0_{\varphi^a}(M) \to H^0_{\varphi^a}(M_{[s,r]})$ are injective.
Conversely, for $\bv \in H^0_{\varphi^a}(M_r)$, we also have
$\bv = \varphi^{-a}(\bv) \in M_{rq}$. By induction, we have $\bv \in M_{rq^n}$
for all $n$ and so $\bv \in M$; that is, $H^0_{\varphi^a}(M) \to H^0_{\varphi^a}(M_r)$ is surjective.
Similarly, for $\bv \in H^0_{\varphi^a}(M_{[s,r]})$, we may apply powers
of $\varphi^a$ and invoke Lemma~\ref{L:intersection} to deduce that $\bv \in M$;
that is, $H^0_{\varphi^a}(M) \to H^0_{\varphi^a}(M_{[s,r]})$ is surjective.

By similar reasoning, if $\bv \in M_r$ (resp.\ $\bv \in M_{[s,r]}$) is such that $(\varphi^a-1)(\bv) \in M_{r/q}$,
then $\bv \in M$. Consequently, the maps $H^1_{\varphi^a}(M) \to H^1_{\varphi^a}(M_r)$,
$H^1_{\varphi^a}(M) \to H^1_{\varphi^a}(M_{[s,r]})$ are injective.
To see that $H^1_{\varphi^a}(M) \to H^1_{\varphi^a}(M_r)$ is surjective, note that any class
in the target defines an extension of $\varphi^a$-modules over $\tilde{\calR}_R$, which lifts to
an extension of $\varphi^a$-modules over $\tilde{\calR}^{\infty}_R$
by Theorem~\ref{T:vector bundles}.
The argument for $H^1_{\varphi^a}(M) \to H^1_{\varphi^a}(M_{[s,r]})$ is similar.

To prove (d), apply Proposition~\ref{P:H1} to find a nonnegative integer $n$ such that $\varphi^a-1: M(n) \to M(n)$ is surjective. 
Using Proposition~\ref{P:global invariants} we may construct a strict injective morphism $M \to M(n)$ of $\varphi^a$-modules; we then obtain a commutative diagram
\[
\xymatrix{
0 \ar[r] & M_{[s,r]} \ar[d] \ar[r] & M_{[s,r]}(n) \ar[d] \ar[r] & M_{[s,r]}(n)/M_{[s,r]} \ar[d] \ar[r] & 0 \\
0 \ar[r] & M_{[s,r/q]} \ar[r] & M_{[s,r/q]}(n) \ar[r] & M_{[s,r/q]}(n)/M_{[s,r/q]} \ar[r] & 0 
}
\]
of Banach spaces in which the vertical arrows are induced by $\varphi^a-1$. The second vertical arrow is surjective by the choice of $n$ plus (c), and hence strict by the open mapping theorem (Theorem~\ref{T:open mapping}). Consequently, the connecting homomorphism
\[
\ker(\varphi^a-1:  M_{[s,r]}(n)/M_{[s,r]}  \to  M_{[s,r/q]}(n)/M_{[s,r/q]} )
\to \coker(\varphi^a-1: M_{[s,r]} \to M_{[s,r/q]})
\]
is also strict surjective. By (c) again, it follows that $H^i_{\varphi^a}(M)$ is a Banach space over $\Qp$ for $i=1$; this is also clear for $i=0$. By the open mapping theorem again, we deduce that $\varphi^a-1: M_{[s,r]} \to M_{[s,r/q]}$ is strict.
\end{proof}

\begin{remark} \label{R:vector bundles}
Theorem~\ref{T:vector bundles field} is essentially due to Fargues and Fontaine
\cite{fargues-fontaine}, who have further studied the structure of the scheme $\Proj(P)$ when
$R = L$ is an analytic field
(see also \cite{fargues, fargues-fontaine-durham}).
They show that it is a \emph{complete absolute curve} in the sense of being
noetherian, connected, separated, and regular of dimension 1, with each
closed point having a well-defined degree and the total degree of any principal divisor being 0.
(If $L$ is algebraically closed, then the degrees are all equal to 1.) One corollary is that the rings
$P[f^{-1}]_0$ for $f \in P_+$ homogeneous are not just Pr\"ufer domains but Dedekind domains.

Many of the basic notions in $p$-adic Hodge theory
can be interpreted in terms of the theory of vector bundles on $\Proj(P_L)$; this is the viewpoint
developed in \cite{fargues-fontaine}, which we find appealing and suggestive.
For instance, the slope polygon of a $\varphi^a$-module over $\tilde{\calR}_L$ can be interpreted
as the Harder-Narasimhan polygon of the corresponding vector bundle, with \'etale $\varphi^a$-modules
corresponding to semistable vector bundles of degree 0
via the slope filtration theorem
over $\tilde{\calR}_L$ (Theorem~\ref{T:slope filtration2}).
The correspondence between \'etale $\varphi^a$-modules
and \'etale local systems then bears a remarkable formal similarity to the correspondence
between stable vector bundles on compact Riemann surfaces and irreducible unitary representations
of the fundamental group, due to Narasimhan and Seshadri \cite{narasimhan-seshadri}.
(A materially equivalent construction was given by Berger \cite{berger-b-pairs} in the somewhat less geometric
language of \emph{$B$-pairs}.)

For general $R$, we expect the scheme $\Proj(P)$ to exhibit much less favorable behavior
(see for instance Remark~\ref{R:not coherent}).
However, it may still be profitable to view the relationship between \'etale local systems
and $\varphi$-modules through the optic of vector bundles over $\Proj(P)$. One possibly surprising
aspect we will encounter later (see \S\ref{subsec:relative FF} for further discussion)
is that \'etale $\Qp$-local systems on the analytic space associated to $R$, which need not descend to $\Spec(R)$
(as in Example~\ref{exa:Tate curve}),
will nonetheless give rise to algebraic vector bundles on $\Proj(P)$ via Theorem~\ref{T:vector bundles2a} and
Corollary~\ref{C:vector bundles2b}.
\end{remark}

\section{Slopes in families}
\label{sec:slope theory}

When one considers $\varphi$-modules over relative Robba rings, one has not one slope polygon but a whole
family of polygons indexed by the base analytic space. We now study the variation of the slope polygon
in such families. Throughout \S\ref{sec:slope theory}, continue to retain Hypothesis~\ref{H:add positive integer}.

\subsection{An approximation argument}

Much of our analysis of slopes in families depends on the following argument for spreading out certain bases of $\varphi$-modules, modeled on \cite[Lemma~6.1.1 and~Proposition~6.2.2]{kedlaya-revisited}.

\begin{lemma} \label{L:get basis}
Let $M$ be a $\varphi^a$-module over $\tilde{\calR}_R$. For $r>0$, let $M_r$ be the model of $M$ over
$\tilde{\calR}^r_R$.
Let $\{M_I\}$ be the $\varphi^a$-bundle associated to $M$.
Suppose that there exists a basis $\bv_1,\dots,\bv_n$ of
$M_{[r/q,r]}$ on which
$\varphi$ acts via an invertible matrix $F$ over
$\tilde{\calR}^{r/q}_R$.
Then $\bv_1,\dots,\bv_n$ is a basis of $M_r$.
\end{lemma}
\begin{proof}
As in Lemma~\ref{L:module of sections},
it suffices to prove that $\bv_1,\dots,\bv_n$ is a basis of
$M_{[r/q^{l+1},r]}$ for each nonnegative integer $l$.
As the case $l=0$ is given, we may proceed by induction on $l$.

Suppose that $l>0$ and that the claim is known for $l-1$,
so that $\bv_1,\dots,\bv_n$ form a basis of $M_{[r/q^{l},r]}$.
Then $\varphi^a(\bv_1),\dots,\varphi^a(\bv_n)$ is a basis of
$M_{[r/q^{l+1},r/q]}$.
By hypothesis, $\bv_j$ can be written as a $\tilde{\calR}^{r/q}_R$-linear combination of
the $\varphi^a(\bv_i)$ and vice versa, so the $\bv_j$ also form a basis of
$M_{[r/q^{l+1},r/q]}$.
By Lemma~\ref{L:finite generation} and Lemma~\ref{L:Robba ring cover},
 $\bv_1,\dots,\bv_n$ form a basis of
$M_{[r/q^{l+1},r]}$ as desired.
\end{proof}

\begin{lemma} \label{L:approximation lemma}
Let $M$ be a $\varphi^a$-module over $\tilde{\calR}_R$. For $r>0$, let $M_r$ be the model of $M$ over
$\tilde{\calR}^r_R$. Let $\{M_I\}$ be the associated $\varphi^a$-bundle.
Suppose that there exist a nonnegative integer $h$, a diagonal matrix $D$ with diagonal
entries $p^{d_1},\dots,p^{d_n}$ for some $d_1,\dots,d_n \in \ZZ$ no two of
which differ by more than $h$, and a basis $\be_1,\dots,\be_n$ of
$M_{[r/q,r]}$ on which
$\varphi^a$ acts via a matrix $F$ over $\tilde{\calR}^{[r/q,r/q]}_R$
for which $\lambda(\alpha^{r/q})(FD -1) < p^{-h}$.
Then there exists a basis $\bv_1,\dots,\bv_n$ of $M_r$
on which $\varphi^a$ acts via a matrix $F'$ over $\tilde{\calR}^{[r/q,r/q]}_R$
such that $F'D - 1$ has entries in $p\tilde{\calR}^{\inte,r/q}_R$,
and for which the invertible matrix $U$ over $\tilde{\calR}^{[r/q,r]}_R$
defined by $\bv_j = \sum_i U_{ij} \be_i$ satisfies $\lambda(\alpha^{r/q})(U-1),
\lambda(\alpha^r)(D^{-1}U D - 1) < p^{-h}$.
\end{lemma}
\begin{proof}
Put $c_0 = p^h \lambda(\alpha^{r/q})(FD - 1) < 1$.
We construct a sequence of invertible $n \times n$ matrices $U_0,U_1,\dots$
over $\tilde{\calR}_R^{[r/q,r]}$ such that the following conditions
hold for $l = 0,1,\dots$.
\begin{enumerate}
\item[(a)]
We have $\lambda(\alpha^{r/q})(U_l - 1),
\lambda(\alpha^r)(D^{-1} U_l D -1) \leq c_0 p^{-h}$.
\item[(b)]
For $F_l = U_l^{-1} F \varphi^a(U_l)$
(which has entries in $\tilde{\calR}^{[r/q,r/q]}_R$
and satisfies $\lambda(\alpha^{r/q})(F_l D - 1) \leq c_0 p^{-h}$),
there exists a matrix $X_l$ over $\tilde{\calR}^{[r/q,r/q]}_R$
such that $F_lD - X_l-1$ has entries in $p\tilde{\calR}^{\inte,r/q}_R$
and $\lambda(\alpha^{r/q})(X_l) \leq c_0^{l+1} p^{-h}$.
\end{enumerate}
For $l=0$, we may take $U_0 = 1$ and $X_0 = F_0 D-1$.
Given $U_l$ for some $l \geq 0$,
by applying Lemma~\ref{L:decompose} to the entries of $p^{-1} X_l$, we construct
a matrix $Z_l$ over $\tilde{\calR}^{[r/q,r]}_R$ such that
$F_l D - Z_l-1$ has entries in $p\tilde{\calR}^{\inte,r/q}_R$,
$\lambda(\alpha^{r/q})(Z_l) \leq c_0^{l+1} p^{-h}$,
and $\lambda(\alpha^{r})(Z_l) \leq c_0^{q(l+1)} p^{-qh}$.
Since
\begin{equation} \label{eq:phi yl bound}
\lambda(\alpha^{r/q})(D^{-1} \varphi^a(Z_l) D) \leq p^h \lambda(\alpha^{r})(Z_l) \leq
p^h c_0^{q(l+1)} p^{-qh} \leq
c_0^{l+2} p^{-h},
\end{equation}
both $1 + Z_l$ and $1 + D^{-1} \varphi^a(Z_l) D$ are invertible over
$\tilde{\calR}^{[r/q,r/q]}_R$. We may thus put
\begin{align*}
U_{l+1} &= U_l(1 + Z_l) \\
F_{l+1} &= U_{l+1}^{-1} F \varphi^a(U_l) = (1 + Z_l)^{-1} F_l (1 + \varphi^a(Z_l)) \\
X_{l+1} &= F_{l+1}D -1- (F_lD - Z_l-1),
\end{align*}
so that $F_{l+1}D -X_{l+1}-1 = F_l D - Z_l-1$ has entries in $p\tilde{\calR}_R^{\inte,r/q}$.
By writing
\[
X_{l+1} = (1 + Z_l)^{-1} (F_l D) (D^{-1} \varphi^a(Z_l) D) + Z_l^2 (1+Z_l)^{-1}F_l D+ Z_l(1 - F_lD),
\]
we see that $\lambda(\alpha^{r/q})(X_{l+1}) \leq c_0^{l+2} p^{-h}$.
Hence $U_{l+1}$ has the desired properties.

The matrices $U_l$ converge to an invertible matrix $U$ over
$\tilde{\calR}_R^{[r/q,r]}$ for which the matrix
$G = U^{-1} F \varphi^a(U) D$ has entries in $\tilde{\calR}^{\inte,r/q}_R$
and $G - 1$ has entries in $p\tilde{\calR}^{\inte,r/q}_R$
(because $G$ is the limit of the Cauchy sequence
$F_l D$ with respect to $\lambda(\alpha^{r/q})$).
If we put $\bv_j = \sum_i U_{ij} \be_i$,
we obtain a basis of
$M_{[r/q,r]}$
on which $\varphi^a$ acts via $G D^{-1}$.
By Lemma~\ref{L:get basis}, $\bv_1,\dots,\bv_n$ form a basis of $M_r$,
proving the desired result.
\end{proof}

\subsection{Rank, degree, and slope}

\begin{defn} \label{D:phi-module relative1}
Define the \emph{rank} and
\emph{degree} of a $\varphi^a$-module $M$ over $W(R)$ (resp.\ $\tilde{\calE}_R$,
$\tilde{\calR}^{\inte}_R$, $\tilde{\calR}^{\bd}_R$, $\tilde{\calR}_R$) as the functions
$\rank(M, \cdot): \calM(R) \to \ZZ$  and
$\deg(M, \cdot): \calM(R) \to \frac{1}{a} \ZZ$ whose values at $\beta \in \calM(R)$ are the rank and
degree, respectively, of
the $\varphi^a$-module obtained by base extension from $M$ by passing from $R$
to $\calH(\beta)$
(recalling Convention~\ref{conv:power slope} in case $a>1$).
Note that by our definitions, the degree of a $\varphi^a$-module over $W(R)$ or
$\tilde{\calR}^{\inte}_R$ is identically zero.
\end{defn}

\begin{lemma} \label{L:degree continuous}
The rank and degree of a $\varphi^a$-module over any of the rings allowed in
Definition~\ref{D:phi-module relative}
are continuous on $\calM(R)$. In other words, the set of all points at which the
rank or degree takes any given value is closed and open in $\calM(R)$.
\end{lemma}
\begin{proof}
The rank function is continuous for the Zariski topology on $\Spec(R)$,
and hence also for the topology on $\calM(R)$. Continuity of the degree follows
from Corollary~\ref{C:units relative} and Remark~\ref{R:valuation semicontinuous}.
\end{proof}

\begin{defn}
Let $M$ be a $\varphi^a$-module over one of the rings allowed in Definition~\ref{D:phi-module relative},
of nowhere zero rank (that is, $\rank(M, \cdot)$ never takes the value $0$).
The \emph{slope} of $M$ is then defined as the function
$\mu(M, \cdot): \calM(R) \to \QQ$ given by $\mu(M, \beta) = \deg(M, \beta)/\rank(M,\beta)$.
By Lemma~\ref{L:degree continuous}, $\mu(M, \cdot)$ is continuous
for the discrete topology on $\QQ$.

The \emph{pure locus} (resp.\ \emph{\'etale locus}) of $M$ is
the set of $\beta \in \calM(R)$ for which $M$
becomes pure (resp.\ \'etale) upon passing from $R$ to $\calH(\beta)$.
If this locus is all of $\calM(R)$, we say that $M$ is
\emph{pointwise pure} (resp.\ \emph{pointwise \'etale}).
These conditions have a more global interpretation; see
Corollary~\ref{C:pointwise etale is etale}.
\end{defn}

\begin{convention} \label{conv:every slope}
At a point where a $\varphi^a$-module has rank 0, it is considered to be pointwise
pure of \emph{every} slope.
This is the correct convention for defining the categories of pure and \'etale $\varphi$-modules,
so that they admit kernels for surjective morphisms (by Remark~\ref{R:kernels}).
\end{convention}

\subsection{Pure models}
\label{subsec:pure models}

\begin{defn} \label{D:pure model}
Fix integers $c,d$ with $d$ a positive multiple of $a$.
Let $M$ be a $\varphi^a$-module over $\tilde{\calE}_R$ (resp.\ $\tilde{\calR}^{\bd}_R$, $\tilde{\calR}_R$).
A \emph{$(c,d)$-pure model} of $M$ is a $W(R)$-submodule
(resp.\ $\tilde{\calR}^{\inte}_R$-submodule, $\tilde{\calR}^{\inte}_R$-submodule) $M_0$ of $M$
which is bounded (i.e., there exists a finitely generated submodule $N_0$ of $M$ over the same subring such that $p^n M_0 \subseteq N_0, p^n N_0 \subseteq M_0$ for some $n \geq 0$)
with $M_0 \otimes_{W(R)} \tilde{\calE}_R \cong M$
(resp.\ with $M_0 \otimes_{\tilde{\calR}^{\inte}_R} \tilde{\calR}^{\bd}_R \cong M$,
with $M_0 \otimes_{\tilde{\calR}^{\inte}_R} \tilde{\calR}_R \cong M$),
such that the $\varphi^a$-action on $M$ induces an isomorphism
$(p^{c} \varphi^d)^* M_0 \cong M_0$.
The existence of such a model implies that $M$ is pointwise pure
of constant slope $c/d$.
A pure model is \emph{locally free} or \emph{free}
if its underlying module is finite locally free or finite free;
 by Proposition~\ref{P:phi-modules are projective}, any finitely presented pure model is locally free.

For $\beta \in \calM(R)$, a \emph{(locally free, free)
local $(c,d)$-pure model} of $M$ at $\beta$
consists of a rational localization $R \to R'$ encircling $\beta$
and a (locally free, free) $(c,d)$-pure
model of $M \otimes_{\tilde{\calE}_R} \tilde{\calE}_{R'}$
(resp.\ $M \otimes_{\tilde{\calR}^{\bd}_R} \tilde{\calR}^{\bd}_{R'}$,
$M \otimes_{\tilde{\calR}_R} \tilde{\calR}_{R'}$).

A $(0,d)$-pure model will also be called an \emph{\'etale model}, and likewise
with the modifiers \emph{local}, \emph{locally free}, or \emph{free} in place.
\end{defn}

\begin{remark}
Note that if $M$ is a $\varphi^a$-module over $\tilde{\calR}_R$ and $M_0$ is a $(c,d)$-pure model,
it is only assumed that $M_0[p^{-1}]$ is stable under the action of $\varphi^d$, rather than $\varphi^a$.
For a locally free $(c,d)$-pure model, stability under $\varphi^a$ will follow later from Theorem~\ref{T:perfect equivalence2a},
which will imply that the isomorphism $(\varphi^a)^* M \cong M$ descends to $M_0[p^{-1}]$.
Until then, we will not assume this stability.
\end{remark}

\begin{lemma} \label{L:locally free to free}
Keep notation as in Definition~\ref{D:pure model}.
\begin{enumerate}
\item[(a)]
The $\varphi^a$-module $M$ admits a locally free local $(c,d)$-pure model at $\beta$ if and only if
it admits a free local $(c,d)$-pure model at $\beta$.
\item[(b)]
If $M$ is a $\varphi^a$-module over $\tilde{\calR}^{\bd}_R$, then $M$ admits a free local $(c,d)$-pure model
at $\beta$ if and only if $M \otimes_{\tilde{\calR}^{\bd}_R} \tilde{\calE}_R$ does.
\end{enumerate}
\end{lemma}
\begin{proof}
Part (a) follows at once from the fact that the direct limit of
$W(S)$ (resp.\ $\tilde{\calR}^{\inte}_S$) over all rational localizations
$R \to S$ encircling $\beta$ is a local ring: namely, it
contains $p$ in its Jacobson radical, and its quotient by the ideal $(p)$ is the local ring
$R_\beta$ (see Lemma~\ref{L:henselian local ring}).

To deduce (b), we may assume that $M' = M \otimes_{\tilde{\calR}^{\bd}_R} \tilde{\calE}_R$
admits a free $(c,d)$-pure model $M'_0$. Let $\be_1,\dots,\be_n$ be a basis of $M'_0$,
let $\bv_1,\dots,\bv_m$ be generators of $M$, and write
$\be_j = \sum_i B_{ij} \bv_i$, $\bv_j = \sum_i C_{ij} \be_i$
with $B_{ij}, C_{ij} \in \tilde{\calE}_R$. Note that $CB = 1$ and that  $\sum_i (1-BC)_{ij} \bv_i = 0$.

Since $\tilde{\calR}^{\bd}_R$ is dense in $\tilde{\calE}_R$ for the $p$-adic topology,
we can find elements $B'_{ij} \in \tilde{\calR}^{\bd}_R$ so that
$C(B'-B)$ has entries in $p W(R)$, and so $X = 1 + C(B'-B)$ is invertible over $W(R)$.
Define elements $\be'_1,\dots,\be'_n$ of $M$ by the formula
$\be'_j = \sum_i B'_{ij} \bv_i$.
Then $\be'_j = \sum_i (CB')_{ij} \be_i = \sum_i X_{ij} \be_i$, so
$\be'_1,\dots,\be'_n$ form another basis of $M'_0$ and hence of $M'$.

Note that for each maximal ideal $\gothm$ of $\tilde{\calR}^{\bd}_R$,
we can find a maximal ideal $\gothm'$ of $\tilde{\calE}_R$ containing $\gothm$.
(Otherwise, $\gothm$ would generate the unit ideal in $\tilde{\calE}_R$,
and so would contain an element of $\tilde{\calR}^{\inte}_R$ congruent to 1 modulo $p$.
But the latter would be a unit, contradiction.)
Since $\bv_1,\dots,\bv_m$ generate $M$, there exists an $n$-element subset $J$
of $\{1,\dots,m\}$ such that the $\bv_j$ for $j \in J$ form a basis of
$M/\gothm M$.
Since $\be'_1,\dots,\be'_n$ form a basis of $M'$, the maximal minor of $B'$
corresponding to $J$ is nonzero in $\tilde{\calE}_R/\gothm'$ and hence
also in $\tilde{\calR}^{\bd}_R/\gothm$.
By Nakayama's lemma, $\be'_1,\dots,\be'_n$ generate the localization of $M$
at $\gothm$; since this holds for all $\gothm$, $\be'_1,\dots,\be'_n$
generate $M$. This proves the desired result.
\end{proof}

\begin{defn}
Let $M$ be a $\varphi^a$-module
over one of $\tilde{\calE}_R, \tilde{\calR}^{\bd}_R, \tilde{\calR}_R$.
For $\beta \in \calM(R)$ and $s \in \QQ$, we say
$M$ is \emph{pure (of slope $s$)} at $\beta$ if $M$ admits a
locally free local $(c,d)$-pure model at $\beta$ for some pair $(c,d)$ of integers
with $d$ a positive multiple of $a$ and $c/d = s$. This forces $s = \mu(M)$ if
$\rank(M, \beta) > 0$; if $\rank(M,\beta) = 0$, then $M$ is pure of \emph{every}
slope at $\beta$ (see Convention~\ref{conv:every slope}).
We say $M$ is \emph{pure} if it is pure
at each $\beta \in \calM(R)$; in this case, we can cover all $\beta \in \calM(R)$
using finitely many local pure models thanks to the compactness of $\calM(R)$.
In these definitions, we regard \emph{\'etale} as a synonym for \emph{pure of slope $0$}.

In case we need to be precise about the choice of $c$ and $d$, we will say that a module
is \emph{$(c,d)$-pure} rather than \emph{pure of slope $s$}.
This will ultimately be rendered unnecessary by the observation that purity for one pair
$(c,d)$ with $d>0$ and $c/d = s$ implies the same for any other pair (Corollary~\ref{C:same purity}).

We also say that $M$ is \emph{globally pure/\'etale} if it admits a locally free pure/\'etale model.
For more on this condition, see Remark~\ref{R:globally pure}.
\end{defn}

\begin{remark} \label{R:globally pure}
Consider the following conditions on a $\varphi^a$-module $M$
over one of $\tilde{\calE}_R, \tilde{\calR}^{\bd}_R, \tilde{\calR}_R$.
\begin{enumerate}
\item[(a)]
The $\varphi^a$-module $M$ is globally pure (i.e., admits a locally free pure model).
\item[(b)]
The $\varphi^a$-module $M$ admits a pure model.
\item[(c)]
The $\varphi^a$-module $M$ is pure (i.e., admits locally free local pure models).
\item[(d)]
The $\varphi^a$-module $M$ admits local pure models.
\item[(e)]
The $\varphi^a$-module $M$ is pointwise pure.
\end{enumerate}
There are some trivial implications among these conditions:  (a) implies (b) and (c), which in turn each imply (d), which in turn implies (e). We will see eventually that there are more implications as follows.
\begin{itemize}
\item
Over $\tilde{\calE}_R$ or $\tilde{\calR}^{\bd}_R$,
Corollary~\ref{C:bounded pointwise pure} will show that (e) implies (c), while Example~\ref{exa:banana} will show that (c) does not imply (a). Also,
(d) implies (b): given some local pure models on a finite covering of
$\calM(R)$, the elements of $M$ which restrict into each local pure model form a pure model. 
Consequently, (a) strictly implies (b) and (b)-(e) are equivalent.
\item
Over $\tilde{\calR}_R$,
Corollary~\ref{C:pointwise etale is etale} will show that
(e) implies (c), while Example~\ref{exa:banana} will show that (c) does not imply (a) and Example~\ref{exa:Tate curve} will show that (c) does not imply (b).
Consequently, (a) strictly implies (b), (b) strictly implies (c), and (c)-(e) are equivalent.
\end{itemize}
\end{remark}

\begin{prop} \label{P:DM relative}
Let $M$ be a $\varphi^a$-module over $\tilde{\calE}_R$
(resp. $\tilde{\calR}^{\bd}_R$, $\tilde{\calR}_R$) admitting a free $(c,d)$-pure model $M_0$
for some $c,d \in \ZZ$ with $d$ a positive multiple of $a$.
Then there exists an $R$-algebra $S$ which is the completed direct limit of
some faithfully finite \'etale $R$-subalgebras, such that
$M_0 \otimes_{W(R)} W(S)$
(resp.\
$M_0 \otimes_{\tilde{\calR}^{\inte}_R} \tilde{\calR}^{\inte}_{S}$,
$M_0 \otimes_{\tilde{\calR}^{\inte}_R} \tilde{\calR}^{\inte}_{S}$)
admits a basis fixed by $p^{c} \varphi^d$.
\end{prop}
\begin{proof}
The assertion over $\tilde{\calE}_R$ follows from Lemma~\ref{L:DM relative}.
To handle the other cases, invoke the case $n=1$ of Lemma~\ref{L:DM relative}
to reduce to the case where
$M_0$ admits a basis $\be_1,\dots,\be_n$ on which $p^c \varphi^d$ acts via a matrix $A$
over $\tilde{\calR}^{\inte}_R$ congruent to 1 modulo $p$.
Choose $S$ so that
$M_0 \otimes_{\tilde{\calR}^{\inte}_R} \tilde{\calE}_{S}$
contains a basis $\bv_1,\dots,\bv_n$ fixed by $p^c \varphi^d$ with $\bv_i \equiv \be_i \pmod{p}$.
Since $A-1$ has $p$-adic absolute value less than 1,
we have $\lambda(\alpha^r)(A-1) < 1$ for all sufficiently small $r>0$.
For any such $r$, the matrix $U$ over $W(S)$ defined by $\bv_j = \sum_i U_{ij} \be_i$ is congruent to 1 modulo $p$
and satisfies $A \varphi^d(U) = U$.

We now argue as in \cite[Proposition~2.5.8]{kedlaya-relative}.
We may assume $a=d$, so $q = p^d$.
Put $C = \max\{p^{-1}, \lambda(\alpha^r)(A-1)\} < 1$.
We prove by induction that for each positive integer $m$, $U$ is congruent modulo $p^m$ to an invertible matrix
$V_m$ over $\tilde{\calR}^{\inte,rq}_S$ with
\[
\lambda(\alpha^r)(V_m-1),
\lambda(\alpha^{rq})(V_m-1) \leq C.
\]
This is obvious for $m=1$ by taking $V_m = 1$. Given the claim for some $m$,
$U$ is congruent modulo $p^{m+1}$ to a matrix $V_m + p^m X$ in which each entry $X_{ij}$ is the Teichm\"uller
lift of some $\overline{X}_{ij} \in S$. We have
\[
\varphi^d(X) - X \equiv p^{-m}(V_m - \varphi^d(V_m) - (A-1) \varphi^d(V_m)) \pmod{p},
\]
from which it follows that
\[
\alpha(\overline{X}_{ij})^r \leq \max\{1, (p^m C)^{q^{-1}}\} = (p^m C)^{q^{-1}}.
\]
If we put $V_{m+1} = V_m + p^m X$, then
\begin{align*}
\lambda(\alpha^r)(V_{m+1}-1) &\leq \max\{C,
p^{-m} (p^m C)^{q^{-1}}\} \leq C\\
\lambda(\alpha^{rq})(V_{m+1}-1) &\leq \max\{C, p^{-m} (p^m C)\} = C
\end{align*}
as desired.

{}From the previous induction, we conclude that $U$ has entries in $\tilde{\calR}^{\inte,rq}_{S}$
and satisfies $\lambda(\alpha^r)(U-1), \lambda(\alpha^{rq})(U-1) < 1$.
It is thus invertible over $\tilde{\calR}^{\inte}_{S}$,
so $\bv_1,\dots,\bv_n$ form a basis of
$M_0 \otimes_{\tilde{\calR}^{\inte}_R} \tilde{\calR}^{\inte}_{S}$ as desired.
\end{proof}

\begin{theorem} \label{T:spread etale}
Let $M$ be a $\varphi^a$-module over $\tilde{\calR}_R$ of nowhere zero rank.
Choose $\beta \in \calM(R)$ and choose $c,d \in \ZZ$ with $d$ a positive multiple of $a$ and $c/d = \mu(M,\beta)$.
Suppose that $\beta$ belongs to the pure locus of $M$.
Then any $(c,d)$-pure model of
$M \otimes_{\tilde{\calR}_R} \tilde{\calR}_{\calH(\beta)}$
extends to a free local $(c,d)$-pure model of $M$ at $\beta$.
\end{theorem}
\begin{proof}
We may assume $d=a$, so $q= p^d$.
Choose a $(c,d)$-pure model $M_{0,\beta}$ of
$M \otimes_{\tilde{\calR}_R} \tilde{\calR}_{\calH(\beta)}$.
By Proposition~\ref{P:DM},
for $L$ a completed algebraic closure of $\calH(\beta)$,
for some choice of $r$,
there exists a basis $\be_1,\dots,\be_n$ of
$M_r \otimes_{\tilde{\calR}^r_R} \tilde{\calR}^{[r/q,r]}_{L}$ on which $p^{c} \varphi^d$ acts via
the identity matrix. We may also ensure that
$\be_1,\dots,\be_n$ also form a basis of $M_{0,\beta} \otimes_{\tilde{\calR}^{\inte}_{\calH(\beta)}} \tilde{\calR}^{\inte}_L$.

By Lemma~\ref{L:nearby generators}, any elements $\be'_1,\dots,\be'_n$ of
$M_r \otimes_{\tilde{\calR}^{r}_{R}} \tilde{\calR}^{[r/q,r]}_L$
which are sufficiently close to $\be_1,\dots,\be_n$ also generate
$M_r \otimes_{\tilde{\calR}^{r}_{R}} \tilde{\calR}^{[r/q,r]}_L$.
Since the separable closure of
$\calH(\beta)$ in $L$ is dense,
we can take $\be'_1,\dots,\be'_n$ to be generators of
$M_r \otimes_{\tilde{\calR}^{r}_{R}} \tilde{\calR}^{[r/q,r]}_E$
for some finite Galois extension $E$ of $\calH(\beta)$ such that $\be'_j=\sum_i C_{ij} \be_i$
for some invertible matrix $C$ over $\tilde{\calR}^{[r/q,r]}_E$ satisfying
$\lambda(\beta^{r/q})(C-1),\lambda(\beta^r)(C-1)<1$.

Since $R_\beta$ is a henselian local ring (by Lemma~\ref{L:henselian local ring}),
by Theorem~\ref{T:henselian} we can find a rational localization $R \to R'$
encircling $\beta$ and a faithfully finite \'etale $R'$-algebra $S$
such that $S$ is Galois over $R'$,
$S$ admits a unique extension $\gamma$ of $\beta$,
and such extension has residue field $E$.
By Lemma~\ref{L:extend generators}, for a suitable choice of $R'$,
we can take $\be'_1,\dots,\be'_n$ to be a basis of
$M_r \otimes_{\tilde{\calR}^{r}_{R}} \tilde{\calR}^{[r/q,r]}_S$
on which the action of $p^{c} \varphi^d$ is via
a matrix $F$ for which $\lambda(\alpha^{r/q})(F-1) < 1$.

In this setting,
Lemma~\ref{L:approximation lemma}
produces
a basis $\be''_1,\dots,\be''_n$ of $M \otimes_{\tilde{\calR}_R} \tilde{\calR}_{S}$
on which $p^{c} \varphi^d$ acts via an invertible matrix over $\tilde{\calR}^{\inte,r/q}_S$
congruent to 1 modulo $p$.
More precisely, we have $\be''_j = \sum_i U_{ij} \be'_i$ for some invertible matrix
$U$ over $\tilde{\calR}^{[r/q,r]}_S$ for which $\lambda(\alpha^{r/q})(U-1),
\lambda(\alpha^r)(U-1) < 1$.
Starting from this new basis and then applying Proposition~\ref{P:DM relative},
we obtain,
for some $S$-algebra $S'$ which is the completed
direct limit of faithfully finite \'etale subalgebras (and which maps to $L$),
a basis $\be'''_1, \dots, \be'''_n$ of $M \otimes_{\tilde{\calR}_R} \tilde{\calR}_{S'}$
fixed by $p^{c} \varphi^d$. More precisely, we have $\be'''_j = \sum_i V_{ij} \be''_i$ for some invertible matrix
$V$ over $\tilde{\calR}^{\inte,r}_{S'}$ for which $\lambda(\alpha^{r})(V-1)<  1$
(this bound following from the proof of Proposition~\ref{P:DM relative}). We may also choose $S'$ to have
an automorphism lifting each element of $G = \Gal(E/\calH(\beta))$.

We thus have $\be'''_j = \sum_i (CUV)_{ij} \be_i$,
and the matrix $CUV$ over $\tilde{\calR}^{[r/q,r]}_L$ satisfies
$\lambda(\beta^{r/q})(CUV-1), \lambda(\beta^r)(CUV-1) < 1$.
However, both $\be_i$ and $\be'''_j$ are fixed by $p^{c} \varphi^d$,
so the entries of $CUV$ must be fixed by $\varphi^d$.
By Lemma~\ref{L:big Robba invariants}, $CUV$ has entries in the field
$W(\FF_q)[p^{-1}]$. Since $\lambda(\beta^{r/q})(CUV-1) < 1$, we conclude that
$CUV-1$ has entries in $p W(\FF_q)$. In particular, $\be''_1,\dots,\be''_n$
form a basis of $M_{0,\beta} \otimes_{\tilde{\calR}^{\inte}_{\calH(\beta)}} \tilde{\calR}^{\inte}_L$
and hence also a basis of $M_{0,\beta} \otimes_{\tilde{\calR}^{\inte}_{\calH(\beta)}} \tilde{\calR}^{\inte}_{E}$.

This last result implies that the action of any element $\tau \in G$ on $S$
induces an automorphism of the $\tilde{\calR}^{\inte}_{E}$-span of $\be''_1,\dots,\be''_n$.
For each $\tau$, choose a lift of $\tau$ to $S'$ and
define the matrix $T_\tau$ over $\tilde{\calR}_{S'}$ by $\tau(\be'''_j) = \sum_i (T_\tau)_{ij} \be'''_i$.
Again, since $\be'''_i$ and $\tau(\be'''_j)$ are fixed by $p^{c} \varphi^d$,
the entries of $T_\tau$ are forced to belong to the $\varphi^d$-fixed subring of
$\tilde{\calR}^{\inte}_{S'}$, which Corollary~\ref{C:extended Robba invariants}
identifies as $W((S')^{\overline{\varphi}^d})[p^{-1}]$. By construction,
the images of these entries in $\tilde{\calR}_L$ belong to
$\tilde{\calR}^{\inte}_L$, and hence to $\gotho_L^{\varphi^d}$.
By Remark~\ref{R:phi-invariant} and Remark~\ref{R:valuation semicontinuous},
the condition of an element of $W((S')^{\overline{\varphi}^d})[p^{-1}]$
belonging to $W(S')$ is closed \emph{and open} on $\calM(S')$. 
By Remark~\ref{R:compact spaces}(b) and Lemma~\ref{L:henselian local ring},
the natural map $p: \calM(S') \to \calM(R')$ is open.
We can thus force the entries of $T_\tau$
into $W((S')^{\overline{\varphi}^d})$ by shrinking $\calM(R')$ again (in a manner dependent on $\tau$).
(Note that this last step fails if we try to argue directly with the $\be''_i$ rather than
with the $\be'''_i$.)

After the resulting shrinking of $\calM(R')$,
the $\tilde{\calR}^{\inte}_S$-span of $\be''_1,\dots,\be''_n$
admits an action of $G$,
 so we can apply faithfully flat descent
(Theorem~\ref{T:descent modules}) to descend it to a local $(c,d)$-pure model
of $M$ at $\beta$.
(We are here using Proposition~\ref{P:perfect henselian}(b) to deduce that
$\tilde{\calR}^{\inte}_S$ is faithfully finite \'etale over $\tilde{\calR}^{\inte}_{R'}$.)
 This local model is locally free by Theorem~\ref{T:descent finite locally free};
we obtain a free local model at $\beta$ by applying Lemma~\ref{L:locally free to free}.
\end{proof}
\begin{cor} \label{C:etale locus is open}
For any $\varphi^a$-module $M$ over $\tilde{\calR}_R$,
the pure locus  and \'etale locus of $M$ are open.
\end{cor}
\begin{proof}
The openness of the pure locus is immediate from Theorem~\ref{T:spread etale}.
The \'etale locus is open because it is the intersection of the pure locus with the
closed and open (by Lemma~\ref{L:degree continuous}) subset of $\calM(R)$ on which
the degree of $M$ is zero.
\end{proof}
\begin{cor} \label{C:pointwise etale is etale}
For any $\varphi^a$-module $M$ over $\tilde{\calR}_R$,
$M$ is \'etale (resp.\ pure) if and only if $M$ is pointwise \'etale (resp.\ pointwise pure).
\end{cor}

\begin{cor} \label{C:local model not locally free}
For any $\varphi^a$-module $M$ over $\tilde{\calR}_R$
and any $\beta \in \calM(R)$,
$M$ is pure at $\beta$ if and only if $M$ admits
a (not necessarily locally free) local pure model at $\beta$.
\end{cor}
\begin{proof}
If $M$ admits a local pure model $M_0$ at $\beta$, then $M_0$ also generates a pure model of
$M \otimes_{\tilde{\calR}_R} \tilde{\calR}_{\calH(\beta)}$, which is necessarily free because
$\tilde{\calR}^{\inte}_{\calH(\beta)}$ is a principal ideal domain. Hence
$\beta$ belongs to the pure locus of $M$, which by
Theorem~\ref{T:spread etale} implies that $M$ is pure at $\beta$.
\end{proof}

\begin{remark}
Let $M$ be a $\varphi^a$-module over $\tilde{\calR}^{\bd}_R$.
Whereas Lemma~\ref{L:locally free to free} implies that $M$ is pure if and only if
$M \otimes_{\tilde{\calR}^{\bd}_R} \tilde{\calE}_R$ is pure,
it is not the case that purity of $M$ can be deduced from purity of
$M \otimes_{\tilde{\calR}^{\bd}_R} \tilde{\calR}_R$.
\end{remark}

\subsection{Slope filtrations in geometric families}

At this point, it is natural to discuss generalizations to the relative case of the
existence of slope filtrations for Frobenius modules over the Robba ring
(Theorem~\ref{T:slope filtration explicit1}) or the extended Robba ring
(Theorem~\ref{T:slope filtration explicit2}). We do not have in mind an explicit
use for these in $p$-adic Hodge theory, but we expect them to become relevant
in the same way that slope theory over the Robba
ring appears in the work of Colmez on trianguline representations
\cite{colmez-trianguline}.

We first give a brief review of the formalism of slope filtrations and slope polygons.
See \cite[\S 3.5]{kedlaya-revisited} for a more thorough discussion, keeping in mind the
change in sign convention. (To compensate for that change, we have also swapped the order of slopes in the slope polygon in order to
preserve the convex shape of the polygon.)

\begin{defn}
Suppose that $R = L$ is an analytic field, and let $M$ be a
$\varphi^a$-module over $\tilde{\calR}_L$. Let $0 = M_0 \subset \cdots \subset M_l = M$
be the filtration provided by Theorem~\ref{T:slope filtration explicit2},
in which $M_1/M_0,\dots,M_l/M_{l-1}$ are pure and $\mu(M_1/M_0) > \cdots > \mu(M_l/M_{l-1})$.
Define the \emph{slope polygon} of $M$ to be the polygonal line
starting at $(0,0)$ and consisting of, for $i=l,\dots,1$ in order,
a segment of horizontal width $\rank(M_i/M_{i-1})$ and slope $\mu(M_i/M_{i-1})$.
Note that the right endpoint of the polygon is $(\rank(M), \deg(M))$.

For $M$ a $\varphi^a$-module over $\tilde{\calR}^{\bd}_L$, there are two natural ways to
associate a slope polygon to $M$. One is to first extend scalars to $\tilde{\calR}_L$
and use the definition given in the previous paragraph;
this gives the \emph{special slope polygon} of $M$. The other is to first extend scalars
to $\tilde{\calE}_L$, identify the latter with $\tilde{\calR}_L$ for the trivial norm on $L$,
then invoke the previous paragraph. This gives the \emph{generic slope polygon} of $M$.
(It is equivalent to define the generic slope polygon using the usual Dieudonn\'e-Manin definition of slopes;
see for instance \cite[Chapter~14]{kedlaya-course}.)
\end{defn}

\begin{remark} \label{R:polygon base extension}
For $R = L$ an analytic field and $L'$ a complete extension of $L$, passing from $L$ to $L'$
does not change slope polygons, by virtue of Corollary~\ref{C:purity base extension}
and the uniqueness of the slope filtration in Theorem~\ref{T:slope filtration explicit2}.
\end{remark}

\begin{prop} \label{P:special above generic}
Let $R = L$  be an analytic field, and let $M$ be a $\varphi^a$-module over $\tilde{\calR}^{\bd}_L$.
\begin{enumerate}
\item[(a)]
The special slope polygon lies on or above the generic slope polygon, with the same endpoints.
\item[(b)]
If the two polygons coincide, then the slope filtration of $M \otimes_{\tilde{\calR}^{\bd}_L}
\tilde{\calR}_L$
(Theorem~\ref{T:slope filtration explicit2}) descends to $M$.
\end{enumerate}
\end{prop}
\begin{proof}
For (a), see \cite[Proposition~5.5.1]{kedlaya-revisited}. For (b),
see \cite[Theorem~5.5.2]{kedlaya-revisited}.
\end{proof}

\begin{lemma} \label{L:read off generic}
Let $M$ be a $\varphi^a$-module over $\tilde{\calR}^{\bd}_R$
admitting a basis on which $\varphi^a$ acts via a matrix of the form $AD$, where
$D$ is a diagonal matrix with diagonal
entries in $p^{\ZZ}$, and $A$ is a square matrix such that $A-1$ has entries in
$p \tilde{\calR}^{\inte}_R$.
Then there exist an $R$-algebra $S$ which is the union of faithfully finite \'etale $R$-subalgebras and an
invertible matrix $U$ over $W(S)$ congruent to $1$ modulo $p$,
such that $U^{-1} AD \varphi^a(U) = D$.
In particular, for each $\beta \in \calM(R)$,
the generic slopes of $M \otimes_{\tilde{\calR}^{\bd}_R} \tilde{\calR}^{\bd}_{\calH(\beta)}$
are the negatives of the $p$-adic valuations of the entries of $D$, divided by $a$.
\end{lemma}
\begin{proof}
We can proceed as in \cite[Lemma~5.9]{kedlaya-annals}. Note that although the analogous statement in our setup would be taking $S$ to be the completed union of faithfully finite \'etale $R$-subalgebras, the argument of \cite[Lemma~5.9]{kedlaya-annals} actually implies that one can take $S$ to be the union of faithfully finite \'etale $R$-subalgebras. This refinement will be useful for Theorem \ref{T:split polygon at point}.
\end{proof}

\begin{theorem} \label{T:open locus}
For any $\varphi^a$-module $M$ over $\tilde{\calR}_R$,
the function mapping $\beta \in \calM(R)$ to the slope polygon of
$M \otimes_{\tilde{\calR}_R} \tilde{\calR}_{\calH(\beta)}$ is lower
semicontinuous. In other words, if $\rank(M)$ is constant (which is
true locally by Lemma~\ref{L:degree continuous}), for any $x \in [0,\rank(M)]$,
the $y$-coordinate of the point of the slope polygon of
$M \otimes_{\tilde{\calR}_R} \tilde{\calR}_{\calH(\beta)}$ is a
lower semicontinuous function of $\beta$. (Note that this function is
locally constant for $x = \rank(M)$, by Lemma~\ref{L:degree continuous} again.)
\end{theorem}
\begin{proof}
Choose $\beta \in \calM(R)$. Let $L$ be a completed algebraic closure of
$\calH(\beta)$. By Proposition~\ref{P:DM not pure}, for some positive multiple $d$ of $a$,
$M \otimes_{\tilde{\calR}_R} \tilde{\calR}_L$ admits a basis on which $\varphi^d$
acts via a diagonal matrix $D$ with entries in $p^\ZZ$.
We now proceed as in Theorem~\ref{T:spread etale}, by
applying Lemma~\ref{L:approximation lemma} to a suitably good approximation of this
basis. As a result, we obtain a
rational localization $R \to R'$ encircling $\beta$, a faithfully
finite \'etale $R'$-algebra $S$, and a basis $\bv_1,\dots,\bv_n$ of
$M \otimes_{\tilde{\calR}_R} \tilde{\calR}_S$ on which $\varphi^d$ acts
via an invertible matrix over $\tilde{\calR}^{\bd}_S$ of the form
$FD$, where $F-1$ has entries in $p \tilde{\calR}^{\inte}_S$.
Let $N$ be the $\varphi^d$-module over $\tilde{\calR}^{\bd}_S$ spanned by $\bv_1,\dots,\bv_n$.

By Remark~\ref{R:polygon base extension}, the slope polygon of $M$ at a given point of
$\calM(R')$ is the same as at any point of $\calM(S)$ restricting to the given point.
For one, this means that the negatives of the $p$-adic valuations of the diagonal entries of $D$
give the slopes in the slope polygon of $M$ at $\beta$. By Lemma~\ref{L:read off generic},
this polygon also computes the generic slope polygon of $N$ at each
$\gamma \in \calM(S)$. By Proposition~\ref{P:special above generic}(a),
we conclude that the special slope polygon of $N$ at each $\gamma \in \calM(S)$,
or in other words the slope polygon of $M$ at $\gamma$, lies on or above the
slope polygon of $M$ at $\beta$.
By Remark~\ref{R:polygon base extension} again, this implies that the slope polygon
is lower semicontinuous as a function on $\calM(R)$, as desired.
\end{proof}

In addition to semicontinuity, we have the following boundedness property for  the slope polygon.
\begin{prop} \label{P:slopes bounded}
For any $\varphi^a$-module $M$ over $\tilde{\calR}_R$,
the function mapping $\beta \in \calM(R)$ to the slope polygon of
$M \otimes_{\tilde{\calR}_R} \tilde{\calR}_{\calH(\beta)}$ is bounded above and below.
\end{prop}
\begin{proof}
Choose $N$ as in Proposition~\ref{P:global invariants};
we then obtain a surjection $\tilde{\calR}_R^m(-N) \to M$ of $\varphi^a$-modules for some nonnegative integer $m$.
For each $\beta \in \calM(R)$, for $s$ the smallest slope in the slope polygon of $M$
at $\beta$,
$M \otimes_{\tilde{\calR}_R} \tilde{\calR}_{\calH(\beta)}$ surjects onto a nonzero $\varphi^a$-module
over $\tilde{\calR}_{\calH(\beta)}$ of slope $s$, as then does $\tilde{\calR}_R^m(-N)$. This forces
$-N/a \leq s$ by (the easy direction of) Theorem~\ref{T:slope filtration2}; consequently, all of the slopes of $M$
at $\beta$ are bounded below by $-N/a$. Since the sum of the slopes is a continuous
(by Lemma~\ref{L:degree continuous})
and hence bounded function
on $\calM(R)$, the slopes of $M$ at  $\beta$ are also bounded above. This proves the claim.
\end{proof}

\begin{cor}
Let $M$ be a $\varphi^a$-module over $\tilde{\calR}_R$. Then there exists an open dense subset $U$
of $\calM(R)$ on which the slope polygon of $M$ is locally constant.
\end{cor}
\begin{proof}
It suffices to check this property locally around some $\beta \in \calM(R)$.
By Lemma~\ref{L:degree continuous}, we may assume that $M$ has constant rank and degree over
$\calM(R)$.

By Proposition~\ref{P:slopes bounded}, the slope polygons of $M$ are limited to a finite set $S$.
Let $T$ be the subset of $S$ consisting of those polygons which occur in every neighborhood of $\beta$; this
set is nonempty because it contains the slope polygon at $\beta$.
Note that we can find a neighborhood $U$ of $\beta$ on which no polygon outside $T$ occurs, by eliminating
elements of $S \setminus T$ one at a time.

Let $P$ be a maximal polygon in $T$, i.e., one which does not lie on or below any other element of $T$.
By lower semicontinuity (Theorem~\ref{T:open locus}), for every neighborhood $V$ of $\beta$ within $U$,
there is a nonempty open subset on which the slope polygon is identically equal to $P$.
This proves the claim.
\end{proof}

If the slope polygon does indeed vary, it is unclear whether one can expect to construct
a slope filtration. One does get a result in case the polygon is locally constant, or more generally if one of its vertices is locally constant.

\begin{lemma} \label{L:pointwise splitting criterion}
Let $A$ be an $n \times n$ matrix
over $\tilde{\calR}^{\inte}_R$ which is invertible over $\tilde{\calR}^{\bd}_R$, fix $x_1,\dots,x_n \in \tilde{\calR}^{\bd}_R$, and choose $y_1,\dots,y_n \in \tilde{\calE}_R$ so that
\begin{equation} \label{eq:splitting equation for filtration}
y_i -x_i = \sum_j A_{ij} \varphi^a(y_j) \qquad (i=1,\dots,n).
\end{equation}
Then $y_i \in \tilde{\calR}^{\bd}_R$ for $i=1,\dots,n$ if and only if
for each $\beta \in \calM(R)$, the image of $y_i$ in $\tilde{\calE}_{\calH(\beta)}$ belongs to $\tilde{\calR}^{\bd}_{\calH(\beta)}$ for $i=1,\dots,n$.
\end{lemma}
\begin{proof}
We may assume without loss of generality that $x_1,\dots,x_n \in \tilde{\calR}^{\inte}_R$ and $y_1,\dots,y_n \in W(R)$.
Fix $r > 0$ for which $x_i, A_{ij} \in \tilde{\calR}^{\inte,r}_R$
and $(A^{-1})_{ij} \in \tilde{\calR}^{\bd,r}_R$,
and put $x'_i = \sum_j (A^{-1})_{ij} x_j$.

Suppose that $y_1,\dots,y_n \in \tilde{\calR}^{\inte}_R$.
We may then choose some $s>0$ so that $y_1,\dots,y_n \in \tilde{\calR}^{\inte,s}_R$. By \eqref{eq:splitting equation for filtration}, we must also have $y_1,\dots,y_n \in \tilde{\calR}^{\inte,s'}_R$ for $s' = \min\{rp^a, sp^a\}$; it follows that $y_1,\dots,y_n \in \tilde{\calR}^{\inte,rp^a}_R$.
In particular,
\[
\lambda(\alpha^r)(y_i) \leq 
\lambda(\alpha^{rp^a})(y_i)^{p^{-a}}
= \lambda(\alpha^r)(\varphi^a(y_i))^{p^{-a}}.
\]
Consequently, if $\lambda(\alpha^r)(\varphi^a(y_i))
> \lambda(\alpha^r)(A^{-1})^{1/(1-p^{-a})}$ for some $i$,
then
\[
\lambda(\alpha^r)(\varphi^a(y_i)) > \lambda(\alpha^r)(A^{-1})
\lambda(\alpha^r)(y_i)
\]
and rewriting \eqref{eq:splitting equation for filtration} 
as $\varphi^a(y_i) + x'_i = \sum_j (A^{-1})_{ij} y_j$
yields
\[
\max_i \{\lambda(\alpha^r)(\varphi^a(y_i))\} =  \lambda(\alpha^r)(x').
\]
To summarize, if $y_1,\dots,y_n \in \tilde{\calR}^{\inte}_R$,
then $y_1,\dots,y_n \in \tilde{\calR}^{\inte,rp^a}_R$
and 
\begin{equation} \label{eq:uniform bound}
\lambda(\alpha^{rp^a})(y_i) \leq \max\{\lambda(\alpha^r)(x'),
\lambda(\alpha^r)(A^{-1})^{1/(1-p^{-a})}\}.
\end{equation}
In particular, if the images of $y_1,\dots,y_n$ in $W(\calH(\beta))$
belong to $\tilde{\calR}^{\inte}_{\calH(\beta)}$ for all $\beta \in \calM(R)$,
then these images belong to $\tilde{\calR}^{\inte,rp^a}_{\calH(\beta)}$
and $\lambda(\beta^{rp^a})(y_i)$ is bounded uniformly over $\beta$.
This yields the desired result.
\end{proof}

\begin{theorem} \label{T:split polygon at point}
Let $M$ be a $\varphi^a$-module over $\tilde{\calR}_R$ of constant rank $n$. Let $\mu_1(M,\beta) \geq \cdots \geq \mu_n(M,\beta)$ be the slopes of $M$ at $\beta \in \calM(R)$ listed with multiplicity. Suppose that for some $m \in \{1,\dots,n-1\}$, for all $\beta \in \calM(R)$, $\mu_m(M,\beta) > \mu_{m+1}(M,\beta)$ and $\mu_1(M,\beta) + \cdots + \mu_m(M,\beta)$ is equal to a constant value.
Then there exists a unique $\varphi^a$-submodule $N$ of $M$ of rank $m$ such that $M/N$ is a $\varphi^a$-module and for each $\beta \in \calM(R)$, the slopes of $N$ at $\beta$ are $\mu_1(M,\beta),\dots,\mu_m(M,\beta)$
while the slopes of $M/N$ at $\beta$ are $\mu_{m+1}(M,\beta),\dots,\mu_n(M,\beta)$.
\end{theorem}
\begin{proof}
It suffices to check the claim locally around a point $\beta \in \calM(R)$,
since $\varphi^a$-modules can be glued 
by Corollary~\ref{C:phi-modules glueing}.
Set notation as in the proof of Theorem~\ref{T:open locus};
using faithfully flat descent, we may reduce to the case $R = R' = S$.
Let $N$ be the $\tilde{\calR}^{\bd}_R$-span of $\bv_1, \dots,
\bv_n$, and let $N_0\subset N$ be the $\tilde{\calR}^{\inte}_R$-span of $\bv_1, \dots,
\bv_n$. 

By Lemma~\ref{L:read off generic},
for some $S$ which is the union of faithfully finite
\'etale $R$-subalgebras $\{S_i\}_{i \in I}$, we may split
$N_0 \otimes_{\tilde{\calR}^{\inte}_R} W(S)$ uniquely as a direct sum of
submodules whose base extensions to $\tilde{\calE}_R$ are globally pure
$\varphi$-submodules. From the construction, there is a profinite group $G$ acting on $S$
and each $S_i$ in such a way that $S_i^G = R$.

Each projector in the splitting can be viewed as an element $\bv$ of
$N_0^\dual \otimes N_0 \otimes_{\tilde{\calR}^{\inte}_R} W(S)$ which is
$G$-invariant. This forces $\bv \in N_0^\dual \otimes N_0
\otimes_{\tilde{\calR}^{\inte}_R} W(R)$; that is, the projectors are all
defined over $W(R)$. In particular,
$N\otimes_{\tilde{\calR}^{\bd}_R} \tilde{\calE}_R$ splits uniquely into
globally pure $\varphi$-submodules.

We will show that the splitting separating the first $m$ slopes from the others descends to $N$; for this, we may
by the proof of \cite[Proposition~5.4.5]{kedlaya-revisited}
(plus a descent argument as in the previous paragraph) reduce to the
case where
$N$ admits a basis $\be_1,\dots,\be_n$ on which $\varphi^a$ acts via a
matrix
of the form $UD$, where $U$ is an upper triangular unipotent matrix
congruent to 1 modulo $p$,
and $D_{ii} = p^{c_i}$ with $c_1 \geq \cdots \geq c_n$.
By Lemma~\ref{L:pointwise splitting criterion},
we reduce to checking the splitting pointwise on $\calM(R)$.
But for $R = L$ an analytic field, the splitting in question
is given by the proof of \cite[Theorem~5.5.2]{kedlaya-revisited}.
This completes the proof.
\end{proof}

\begin{cor} \label{C:slope filtration}
Let $M$ be a $\varphi^a$-module over $\tilde{\calR}_R$ such that
the slope polygon function of $M$ is constant on $\calM(R)$.
Then there exists a unique filtration
$0 = M_0 \subset \cdots \subset M_l = M$ of $M$ by $\varphi^a$-submodules
such that $M_1/M_0,\dots,M_l/M_{l-1}$ are
$\varphi^a$-modules which are pure of constant slope,
and $\mu(M_1/M_0) > \cdots > \mu(M_l/M_{l-1})$.
\end{cor}
\begin{proof}
By Lemma~\ref{L:degree continuous},
we may assume that $M$ is of constant rank.
We may then induct on the rank using Theorem~\ref{T:split polygon at point}.
\end{proof}

\begin{cor} \label{C:local H1 injection}
Let $M$ be a $\varphi^a$-module over $\tilde{\calR}_R$
with everywhere negative slopes.
Then $H^0_{\varphi^a}(M) = 0$,
$H^0_{\varphi^a}(M \otimes_{\tilde{\calR}_R} \tilde{\calR}_{\calH(\beta)}) = 0$
for all $\beta \in \calM(R)$, and the map
\begin{equation} \label{eq:local H1 injection}
H^1_{\varphi^a}(M) \to \prod_{\beta \in \calM(R)} H^1_{\varphi^a}(M \otimes_{\tilde{\calR}_R} \tilde{\calR}_{\calH(\beta)})
\end{equation}
is injective.
\end{cor}
\begin{proof}
By Theorem~\ref{T:slope filtration2}, $H^0_{\varphi^a}(M \otimes_{\tilde{\calR}_R} \tilde{\calR}_{\calH(\beta)}) = 0$
for all $\beta \in \calM(R)$; since the map
\[
H^0_{\varphi^a}(M) \to \prod_{\beta \in \calM(R)} H^0_{\varphi^a}(M \otimes_{\tilde{\calR}_R} \tilde{\calR}_{\calH(\beta)})
\]
is evidently injective, it follows that $H^0_{\varphi^a}(M) = 0$. If $x \in H^1_{\varphi}(M)$ has zero image in $H^1_{\varphi^a}(M \otimes_{\tilde{\calR}_R} \tilde{\calR}_{\calH(\beta)})$ for each $\beta \in \calM(R)$, then $x$ defines an extension
\[
0 \to M \to P \to \tilde{\calR}_R \to 0
\]
whose base extension to $\tilde{\calR}_{\calH(\beta)}$ splits for each $\beta \in \calM(R)$.
This implies that the largest slope of $P$ is identically $0$
and the second-largest slope is always negative, so Theorem~\ref{T:split polygon at point} implies that $P \cong M \oplus \tilde{\calR}_R$. It follows that $x = 0$, completing the proof.
\end{proof}

\begin{remark}
For any given $\varphi^a$-module over $\tilde{\calR}_R$, the conclusion of Corollary~\ref{C:local H1 injection} holds for $M(n)$ for $n$ sufficiently small, since Proposition~\ref{P:slopes bounded}
ensures that the hypothesis of Corollary~\ref{C:local H1 injection} is satisfied. On the other hand, the injectivity of \eqref{eq:local H1 injection} also holds for $M(n)$ for $n$ sufficiently large, as in that case $H^1_{\varphi^a}(M(-n)) = 0$ by Proposition~\ref{P:H1}.
\end{remark}

\begin{remark} \label{R:arithmetic families}
As noted earlier, there is a generalization of slope
theory for $\varphi$-modules orthogonal to the one given here, where
one continues to work with rings of power series but with coefficients
in more general rings (such as affinoid algebras over $\Qp$).
These are called \emph{arithmetic families} in \cite{kedlaya-liu},
where they are distinguished from the \emph{geometric families} arising
here.
Unfortunately, it seems difficult to achieve any results in the context of arithmetic
families as complete as those given here, in no small part because such results would most likely
require a
heretofore nonexistent slope theory for Frobenius modules
over a Robba ring consisting of Laurent series over a \emph{nondiscretely} valued field.
One does however get some important information by working in neighborhoods of rigid analytic points;
for instance, one can construct global slope filtrations in such neighborhoods
\cite{liu}, which is relevant for applications to $p$-adic automorphic forms via the study of eigenvarieties \cite{liu-triangulation}.
\end{remark}

\section{Perfectoid spaces}
\label{sec:adic}

Up to this point, our constructions have generally taken as input a Banach algebra or an adic Banach algebra. In this sense they are \emph{local}; our next step is to parlay this work into some results of a more global nature,
including the relationship between $\varphi$-modules and \'etale local systems. The appropriate category of geometric spaces to use here is
the category of \emph{perfectoid spaces}, obtained by glueing together the adic spectra of perfectoid algebras using Huber's formalism of \emph{adic spaces}.

\subsection{Some topological properties}
\label{subsec:adic topological properties}

We begin by recalling some properties of topological spaces relevant to the study of adic spaces. The key definition of a \emph{spectral space}, and the basic properties of that definition, are due to Hochster \cite{hochster}.
\begin{defn}
A topological space $X$ is \emph{sober} if every irreducible closed subset has a unique generic point. This implies that $X$ is $T_0$: no two distinct points belong to exactly the same closed subsets of $X$.

A topological space is \emph{quasiseparated} if the intersection of any two quasicompact open subsets is again quasicompact. We will write \emph{qcqs} as an abbreviation for \emph{quasicompact and quasiseparated}.

A \emph{spectral space} is a topological space $X$ which is sober and qcqs 
and admits a neighborhood basis consisting of quasicompact open subsets.

A \emph{locally spectral space} is a topological space admitting a neighborhood basis consisting of open spectral subspaces. Such a space is spectral if and only if it is qcqs.
\end{defn}

A number of equivalent characterizations of spectral spaces can be found in \cite{hochster}, including the following.
\begin{lemma} \label{L:inverse limit}
A topological space is spectral if and only if it is an inverse limit of finite $T_0$ spaces.
\end{lemma}
\begin{proof}
See \cite[\S 13, Proposition~10]{hochster}.
\end{proof}
\begin{cor} \label{C:inverse limit}
Any inverse limit of spectral spaces is again a spectral space.
\end{cor}

\begin{defn} \label{D:patch topology}
For $X$ a spectral space, the \emph{patch topology} on $X$ is the topology generated by quasicompact open subsets for the original topology (which we also call the \emph{spectral topology} to clarify the distinction) and their complements.
Beware that one cannot recover the spectral topology from the patch topology alone; for example, there is another spectral topology whose quasicompact open subsets are the complements of the quasicompact open subsets for the original topology
 \cite[Proposition~8]{hochster}.
\end{defn}

\begin{lemma} \label{L:spectral quasicompact intersections}
Let $X$ be a spectral space. Then $X$ is
compact for the patch topology.
\end{lemma}
\begin{proof}
See \cite[Theorem~1]{hochster}.
\end{proof}

\begin{defn}
A map $f: Y \to X$ between locally spectral spaces is \emph{spectral} if it is continuous and for any qcqs open subsets $U \subseteq X, V \subseteq Y$ with $f(V) \subseteq U$, the induced map $V \to U$ is quasicompact (that is, the inverse image of any quasicompact open subset is again quasicompact). Equivalently, the inverse image of any quasicompact open subset is a quasicompact open subset. In particular,
any spectral morphism is continuous for the patch topologies.
\end{defn}

\begin{remark}
Any scheme is a locally spectral space. Any scheme which is qcqs in the sense of algebraic geometry (i.e., its underlying topological space is quasicompact and the absolute diagonal morphism is quasicompact) is spectral.
Any morphism of schemes is spectral.
\end{remark}

We have the following refinement of Theorem~\ref{T:spectra are spectral1}.
\begin{theorem} \label{T:spectra are spectral}
For any adic Banach ring $(A,A^+)$, $\Spa(A,A^+)$ is a spectral space.
For any morphism $(A,A^+) \to (B,B^+)$ of adic Banach rings, the map
$\Spa(B,B^+) \to \Spa(A,A^+)$ is spectral.
\end{theorem}
\begin{proof}
See again \cite[Theorem~3.5(i,ii)]{huber1}.
\end{proof}

\begin{remark}
Remark~\ref{R:Berkovich rational subspace} and Corollary~\ref{C:naive rational covering} together can be reinterpreted as follows. For $(A,A^+)$ an adic Banach algebra, the image of $\calM(A)$ in $\Spa(A,A^+)$ is not necessarily dense for the patch topology;
however, if $A$ is an affinoid algebra over an analytic field and $A^+ = A^\circ$,
then already $\Maxspec(A)$ is dense in $\Spa(A,A^+)$ for the patch topology.
\end{remark}

The comparison of rigid and Berkovich analytic spaces with adic spaces involves the following definition.
\begin{defn} \label{D:taut}
Let $X$ be a locally spectral topological space. We say $X$ is \emph{taut} if $X$ is quasiseparated and the closure of every quasicompact subset of $X$ is again quasicompact. For example, if $X$ is qcqs, or more generally if $X$ is quasiseparated and admits a locally finite covering by quasicompact open subsets, then $X$ is taut.

A spectral morphism between locally spectral topological spaces is \emph{taut} if the inverse image of every taut open subspace is taut.
For example, any qcqs morphism is taut. For more basic properties,
see \cite[Lemma~5.1.3]{huber}.
\end{defn}

\begin{remark}
For $k$ a field, glueing finitely many copies of $\Spec k[t]$ along $\Spec k[t,t^{-1}]$ gives a taut locally spectral space, but glueing infinitely many copies does not. By contrast, for $K$ an analytic field, glueing even two copies of $\Spa(K\{T\}, K\{T\}^\circ)$ along the complement of the origin does not give a taut space, because the complement of the origin is not quasicompact.
\end{remark}

\subsection{Adic spaces}
\label{subsec:adic spaces}

We now construct spaces out of adic Banach rings, following Huber. Beware that there does not yet seem to be a consensus about terminology concerning adic spaces, so one must check carefully when comparing results across sources.

We begin with an enhancement of the concept of a locally ringed space.

\begin{defn}
We define the category of \emph{locally valuation-ringed spaces}
as follows.
\begin{itemize}
\item
The objects are triples $X = (|X|, \calO_X, (v_x)_{x \in X})$ in which $|X|$ is a topological space,
$\calO_X$ is a sheaf of complete topological rings, and each $v_x$ is a semivaluation on the set-theoretic stalk $\calO_{X,x}$.
\item
The morphisms from $X = (|X|, \calO_X, (v_x)_{x \in X})$ to $Y = (|Y|, \calO_Y, (v_y)_{y \in Y})$
are pairs $(f, \varphi)$ where $f: |X| \to |Y|$ is a continuous map and $\varphi: \calO_Y \to f_* \calO_X$
is a morphism of sheaves of topological rings such that for each $x \in X$,
the restriction of $v_x$ along $\varphi_x: \calO_{Y,f(x)} \to \calO_{X,x}$
is equivalent to $v_{f(x)}$.
\end{itemize}
We also refer to these spaces for short as
\emph{locally v-ringed spaces}.
\end{defn}

\begin{defn}
For $(A,A^+)$ a sheafy adic Banach ring, we view $X = \Spa(A,A^+)$ as a locally v-ringed space where $\calO_X$ is the structure sheaf on $X$ (Definition~\ref{D:adic structure presheaf}) and $v_x$ is the restriction to $\calO_{X,x}$ of the canonical valuation on $\calH(x)$ (Definition~\ref{D:adic field}).
Any locally v-ringed space of this form is called an \emph{adic affinoid space}.
An \emph{adic space} is a locally v-ringed space admitting an open covering by adic affinoid spaces.

A \emph{morphism} of adic spaces is just a morphism of underlying locally v-ringed spaces. 
Note that any such morphism is automatically an \emph{adic morphism} in the sense of \cite[\S 3]{huber2} because our Banach rings are required to contain topologically nilpotent units; see \cite[Proposition~3.2]{huber2}. Note also that if $X$ is an adic space and $U \to X$ is an open immersion of locally ringed spaces, then $U$ naturally acquires the structure of an adic space in such a way that $U \to X$ becomes a morphism of adic spaces. We refer to any such morphism as an \emph{open immersion} of adic spaces.
\end{defn}

\begin{defn} \label{D:preadic}
Since the property of being sheafy is not known to be stable under various natural operations (e.g., formation of finite \'etale extensions), we are forced to associate spaces also to general adic Banach rings. Perhaps the simplest way to do this is to sheafify in the style of Scholze-Weinstein \cite[Definition~2.1.5]{scholze-weinstein}.
Let $\AdBan$ be the category of adic Banach rings. View $\AdBan^{\op}$ as a site in which the coverings are rational coverings. Let $(\AdBan^{\op})^{\tilde{}}$ be the associated topos (consisting of set-valued sheaves on $\AdBan^{\op}$). For $(A,A^+) \in \AdBan$, let $\widetilde{\Spa}(A,A^+) \in (\AdBan^{\op})^{\tilde{}}$ be the sheafification of the presheaf
\[
(B,B^+) \to \Hom_{\AdBan}((A,A^+), (B,B^+));
\]
any such object is called a \emph{preadic affinoid space}.
A \emph{rational subspace} of $\widetilde{\Spa}(A,A^+)$ is a morphism of the form $\widetilde{\Spa}(B,B^+) \to \widetilde{\Spa}(A,A^+)$ for some rational localization $(A,A^+) \to (B,B^+)$.
An \emph{open immersion} in $(\AdBan^{\op})^{\tilde{}}$ is a morphism
$f: \calF \to \calG$ such that for all $(A,A^+) \in \AdBan$ and all morphisms
$\widetilde{\Spa}(A,A^+) \to \calG$ in $(\AdBan^{\op})^{\tilde{}}$, there is an open subset $U \subseteq \Spa(A,A^+)$ such that
\[
\calF \times_{\calG} \widetilde{\Spa}(A,A^+) = \varinjlim_{V \subseteq U, V \text{ rational}}
\widetilde{\Spa}(\calO_{\Spa(A,A^+)}(V), \calO_{\Spa(A,A^+)}^+(V)).
\]
A \emph{preadic space} (called an \emph{adic space} in \cite{scholze-weinstein}) is a functor $\calF \in (\AdBan^{\op})^{\tilde{}}$ such that
\[
\calF = \varinjlim_{\widetilde{\Spa}(A,A^+) \to \calF \text{ open}} \widetilde{\Spa}(A,A^+).
\]
There are natural functors from adic spaces to preadic spaces and from preadic spaces to locally v-ringed spaces, whose composition is the full embedding of adic spaces into locally v-ringed spaces. This allows us to regard adic spaces as a full subcategory of preadic spaces;
in particular, when $(A,A^+)$ is sheafy we will freely confuse $\Spa(A,A^+)$ with
$\widetilde{\Spa}(A,A^+)$.
We may also associate to any preadic space $X$ an underlying topological space $\left| X \right|$ in such a way that $\widetilde{\Spa}(A,A^+)$ has underlying topological space $\Spa(A,A^+)$.
\end{defn}

\begin{remark}
Beware that a preadic affinoid space can be an adic space without being an adic affinoid space. See Remark~\ref{R:locally perfectoid is perfectoid}.
\end{remark}

\begin{remark} \label{R:ignorance}
The categories of locally v-ringed spaces and preadic spaces admit fibred products.
However, it is unknown whether the fibred product of adic spaces (over an adic space) is again an adic space.
A counterexample would necessarily involve nonnoetherian Banach rings thanks to Proposition~\ref{P:strongly noetherian}; on the other hand, the example of perfectoid spaces (\S\ref{subsec:perfectoid})
shows that failure of the noetherian property alone is not sufficient.
\end{remark}

\begin{defn}
Let $X$ be a preadic space. We say that $X$ is \emph{quasicompact} if it admits a finite covering by open immersions of preadic affinoid spaces.
We say that $X$ is \emph{quasiseparated} if for any two 
open immersions $U_1 \to X, U_2 \to X$ with $U_1, U_2$ quasicompact,
$U_1 \times_X U_2$ is also quasicompact. We say that $X$ is \emph{taut} if
$X$ is quasiseparated and
for any open immersion $U \to X$ with $U$ quasicompact, there exists a covering $V_1 \to X, V_2 \to X$ by open immersions such that $V_1$ is quasicompact, 
$U \to X$ factors through $V_1$,
and $U \times_X V_2 = \emptyset$.
Note that for $X$ an adic space, $X$ is quasicompact, quasiseparated, or taut if and only if $\left| X \right|$ has the same property.
\end{defn}

\begin{defn}
The assignments $(A,A^+) \mapsto A, (A,A^+) \mapsto A^+$ give rise to sheaves
$\calO_X, \calO_X^+$ of rings on any preadic space $X$. In case $X$ is an adic space,
the sheaf $\calO_X$ coincides with the structure sheaf defined previously,
while $\calO_X^+$ coincides with 
the subsheaf of $\calO_X$ such that for each open subset $U$ of $X$,
a section $f \in \calO_X(U)$ belongs to $\calO_X^+(U)$ if and only if for each $x \in X$, the image of $f$ in $\calH(x)$ belongs to $\calH(x)^+$.
The stalk $\calO_{X,x}^+$ of $\calO_X^+$ at $x \in X$ may then be identified with the inverse image of $\calH(x)^+$ in $\calO_{X,x}$.

For $(A,A^+)$ an adic Banach ring and $X = \widetilde{\Spa}(A,A^+)$, by \eqref{eq:adic transform} we have $A^+ = \calO_X^+(X)$.
\end{defn}

\begin{remark}
For $X$ an adic affinoid space, the sheaf $\calO_X^+$ need not be acyclic. However, for perfectoid spaces, $\calO_X^+$ is acyclic as a sheaf of almost modules; 
see Proposition~\ref{P:almost sheaf property}.
\end{remark}

\begin{lemma} \label{L:adic morphisms}
Let $(A,A^+)$ be an adic Banach ring. Then for each adic space $X$, the global sections functor induces a bijection between morphisms $X \to \widetilde{\Spa}(A,A^+)$
of preadic spaces and morphisms $A \to \calO_X(X)$ of topological rings
taking $A^+$ into $\calO_X^+(X)$.
\end{lemma}
\begin{proof}
See \cite[Proposition~2.1(ii)]{huber2}.
\end{proof} 

\begin{remark}
The statement of Lemma~\ref{L:adic morphisms} also holds when $X$ is a preadic space, but in that case it is purely formal.
The essential content of the lemma is to define the functor from adic spaces to preadic spaces.
\end{remark}

The construction of Definition~\ref{D:Berkovich rational subspace} extends to preadic and adic spaces as follows. (See also \cite[\S 8.1]{huber}.)
\begin{defn}
Recall that for any adic Banach ring $(A,A^+)$ over an analytic field, there is a
continuous map $\Spa(A,A^+) \to \calM(A)$ constructed in Definition~\ref{D:Berkovich rational subspace} and a canonical but discontinuous section $\calM(A) \to \Spa(A,A^+)$. One way to interpret these constructions is that the first map takes each $x \in \Spa(A,A^+)$ to the unique $y \in \calM(A)$ belonging to the intersection of all open neighborhoods of $x$ in $\Spa(A,A^+)$, and that the topology on $\calM(A)$ coincides with the quotient topology (not the subspace topology).

Now let $X$ be a preadic space (over an analytic field). Let $\overline{X}$ be the set of $x \in \left| X \right|$ for which $v_x$ is equivalent to a real semivaluation. Then for any $x \in X$, the intersection of all open neighborhoods of $x$ in $X$ contains a unique point $y$ of $\overline{X}$
(by the previous paragraph). We thus obtain a set-theoretic map $\left| X \right| \to \overline{X}$; we equip $\overline{X}$ with the quotient topology and call it the \emph{real quotient} of $X$. 
An open subset of $X$ is \emph{partially proper} if it arises as the inverse image of a (necessarily open) subset of $\overline{X}$.
\end{defn}

\begin{lemma} \label{L:taut Hausdorff quotient}
Let $X$ be an adic space over an analytic field.
\begin{enumerate}
\item[(a)]
If $X$ is quasicompact, then so is $\overline{X}$.
\item[(b)]
If $X$ is taut, then $\overline{X}$ is Hausdorff (and thus is the maximal Hausdorff quotient of $X$).
\item[(c)]
If $X$ is taut, then any partially proper open subset of $X$ is also taut.
\end{enumerate}
\end{lemma}
\begin{proof}
Part (a) is trivial. Part (b) follows from the proof of \cite[Lemma~8.1.8(ii)]{huber}.
To prove (c), let $U$ be a partially proper open subset of $X$ and let $V$ be a quasicompact open subset of $U$. Let $\overline{V}$ be the image of $V$ in $\overline{X}$; it is a closed subset by (a) and (b).
Let $W$ be the inverse image of $\overline{V}$ in $\overline{X}$; then $W$ is the closure of $V$ in $X$. Since $W \subseteq U$, $W$ is also the closure of $V$ in $U$;
this yields (c).
\end{proof}

Lemma~\ref{L:taut Hausdorff quotient}(b) fails without the taut condition as follows.
\begin{example}
Let $K$ be an analytic field. Let $X$ be the adic space obtained by glueing two copies of the closed unit disc $\Spa(K\{T\}, K\{T\}^\circ)$ along the complement of the origin. Then $\overline{X}$ is obtained by glueing two copies of $\calM(K\{T\})$ along the complement of the origin, and thus is not Hausdorff.
\end{example}

Using Lemma~\ref{L:taut Hausdorff quotient}, one can formulate the relationship between adic spaces and Berkovich spaces.
\begin{prop} \label{P:comparison to Berkovich}
The real quotient functor defines an equivalence of categories between
taut adic spaces locally of finite type over an analytic field $K$ (i.e., covered by the adic spectra of affinoid algebras over $K$) and
Hausdorff strictly $K$-analytic spaces in the sense of Berkovich
\cite{berkovich2}.
\end{prop}
\begin{proof}
This follows from Lemma~\ref{L:taut Hausdorff quotient} and \cite[Proposition~8.3.1]{huber}.
\end{proof}

\begin{remark}
In practice, the adic spaces arising in applications are almost always taut. However, nontaut adic spaces do arise in \emph{arithmetic} relative $p$-adic Hodge theory, where one considers Galois representations
not on $\Qp$-vector spaces but on vector bundles over more general analytic spaces.
In that context, it has been shown by us \cite{kedlaya-liu} and Hellmann  \cite{hellmann1} that the analogue of the \'etale locus
for an arithmetic family of $(\varphi, \Gamma)$-modules is in general a nontaut adic space.
\end{remark}

In \cite[Chapter~1]{huber}, Huber introduces a large number of properties of morphisms of adic spaces, but only under some noetherian hypotheses which are too restrictive to apply to perfectoid spaces. We thus introduce some \emph{ad hoc} definitions that agree with Huber's definitions when the latter are applicable, as in \cite[Definition~7.1]{scholze1}.

\begin{defn} \label{D:properties of morphisms}
Let $\psi: Y \to X$ be a morphism of preadic
spaces. We have already defined what it means for $\psi$ to be an \emph{open immersion} (see Definition~\ref{D:preadic}).
\begin{itemize}
\item
We say that $\psi$ is \emph{surjective} if for any morphism $Z \to X$, the map
$\left| Y \times_X Z \right| \to \left| Z \right|$ is surjective. If $X$ is an adic space, this is the same as saying that $\left| Y \right| \to \left| X \right|$ is surjective 
(by Lemma~\ref{L:fibre product}).
\item
We say that $\psi$ is \emph{finite \'etale}
if locally on the target, $\psi$ corresponds to a morphism of the form $\widetilde{\Spa}(B,B^+) \to \widetilde{\Spa}(A,A^+)$ where $B$ is a finite \'etale $A$-algebra and $B^+$ is the integral closure of $A^+$ in $B$.
\item
We say that $\psi$ is \emph{\'etale} if for each $y \in Y$, there exists an open neighborhood $U$ of $y$ in $Y$ such that the restriction of $\psi$ to $U$ factors as 
an open immersion followed by a finite \'etale morphism followed by another open immersion.
\end{itemize}
These properties are evidently stable under base extension.
\end{defn}

\begin{lemma} \label{L:finite etale from algebra}
The following statements are true.
\begin{enumerate}
\item[(a)]
For $(A,A^+)$ an adic Banach ring, the global sections functor induces an equivalence of categories between finite \'etale morphisms to $\widetilde{\Spa}(A,A^+)$ and $\FEt(A)$.
\item[(b)]
Any \'etale morphism $Y \to X$ of preadic spaces with $Y$ quasicompact factors uniquely as $Y \to Z \to X$ where $Y \to Z$ is surjective, $Z$ is quasicompact, and $Z \to X$ is an open immersion.
\item[(c)]
The properties introduced in Definition~\ref{D:properties of morphisms} are
stable under compositions and fibred products.
\end{enumerate}
\end{lemma}
\begin{proof}
Part (a) is immediate from Theorem~\ref{T:henselian direct limit3} and the formal properties of the functor $\widetilde{\Spa}$.
To prove (b)--(c), we may use
Lemma~\ref{L:construct affinoid system} and 
Proposition~\ref{P:henselian direct limit2} to reduce to the case where all of the spaces involved are classical affinoid spaces,
for which we may appeal to \cite[Proposition~3.1.7]{dejong-vanderput} for (b)
and \cite[Proposition~1.7.5]{huber} for (c).
\end{proof}

\begin{remark} \label{R:etale over acyclic}
When applying Lemma~\ref{L:finite etale from algebra}(a), beware that we do not know that  a finite \'etale extension of a sheafy adic Banach algebra is sheafy, or even that a finite \'etale cover of an adic affinoid space is itself an adic space.
In particular, these statements will not follow from 
Theorem~\ref{T:adic vector bundle}.
\end{remark}

\begin{defn} \label{D:adic etale topology}
A family $\{Y_i \to X\}_i$ of morphisms of preadic spaces is a \emph{set-theoretic covering} if
the morphism from the disjoint union of the $Y_i$ to $X$ is surjective.
Note that if $Y_i \to X$ is \'etale, we may factor $Y_i \to Z_i \to X$ as in Lemma~\ref{L:finite etale from algebra}(b), and then the original family is a set-theoretic covering if and only if the family $\{Z_i \to X\}$ is.

We may define the \emph{small finite \'etale site} (resp.\ the \emph{small \'etale site}) on a preadic space $X$ over an analytic field as the site $X_{\fet}$ (resp.\ $X_{\et}$)
whose objects consist of finite \'etale (resp.\ \'etale) morphisms $Y \to X$ and whose coverings are set-theoretic coverings.

We define a \emph{stable basis} of $X_{\et}$ to be a basis $\calB$ of $X_{\et}$ consisting of adic affinoid spaces such that for any morphism $Y' \to Y$ in $X_{\et}$ which is either finite \'etale or a rational subdomain embedding, if $Y \in \calB$ then $Y' \in \calB$.
(Note that the basis property also includes stability under formation of fibred products.) We say that $X$ is \emph{stably adic} if there exists a stable basis of $X_{\et}$. For instance, rigid analytic spaces are stably adic; we will see later that perfectoid spaces are also stably adic (Definition~\ref{D:perfectoid space}).
\end{defn}

Many properties of the \'etale topology may be verified by making the corresponding verifications for the finite \'etale topology and the adic topology separately.
In this process, the role of Tate's reduction argument for the adic topology (Proposition~\ref{P:Tate reduction}) will be played by the following argument,
adapted from de Jong and van der Put \cite[Proposition~3.2.2]{dejong-vanderput}.

\begin{prop} \label{P:djvdp}
Let $X$ be a preadic space, and let $\calB$ be a basis of $X_{\et}$ consisting of preadic affinoid subspaces which is closed under formation of finite \'etale extensions and rational subdomain embeddings (e.g., a stable basis).
Let $\calP$ be a property of coverings in $X_{\et}$ of and by elements of $\calB$, and assume that the following conditions hold.
\begin{enumerate}
\item[(a)]
Any covering admitting a refinement having property $\calP$ also has property $\calP$.
\item[(b)]
Any composition of coverings having property $\calP$ also has property $\calP$.
\item[(c)]
For any $Y \in \calB$, any rational covering of $Y$ has property $\calP$.
\item[(d)]
For any $Y \in \calB$, any faithfully finite \'etale morphism $Y' \to Y$, viewed as a covering, has property $\calP$. 
\end{enumerate}
Then every covering in $X_{\et}$ of and by elements of $\calB$ has property $\calP$.
\end{prop}
\begin{proof}
We first note that using (a) and Lemma~\ref{L:finite etale from algebra}(b), we may formally extend (c) to any covering for the adic topology.
We will use (c) in this stronger form without further comment.

We next establish the following extra condition.
\begin{enumerate}
\item[(e)]
Let $Z \to U$ be a surjective morphism between elements of $\calB$
which factors as $Z \to V \to U$ with $V \to U$ finite \'etale and $Z \to V$ an open immersion. Then $Z \to U$, viewed as a covering, has property $\calP$.
\end{enumerate}
We induct on the maximum degree $d$ of $V \to U$.
The case $d=1$ holds because in this case $Z \to U$ is a surjective open immersion and hence an isomorphism (so for instance (d) applies). For $d>1$, let $W$ be the complement of the diagonal in $V \times_U V$; then the second projection $\pr_2: W \to V$ is finite \'etale of maximum degree $d-1$.
Put $Z' = (Z \times_U V) \times_{V \times_U V} W$ 
and $U' = \pr_2(Z')$; these spaces are both quasicompact
by Lemma~\ref{L:finite etale from algebra}(b). 
We may thus find a finite covering $\{U'_i \to U'\}_i$ for the adic topology by  elements of $\calB$.
Put $Z'_i = Z' \times_{U'} U'_i$; the surjective \'etale morphism
$\phi': Z'_i \to U'_i$ induced by $\pr_2$ factors through $V'_i \to U'_i$ for
$V'_i = \pr_2^{-1}(U'_i) \times_{V \times_U V} W$. The latter morphism is finite \'etale of maximum degree $d-1$, so the following coverings have property $\calP$:
\begin{itemize}
\item
$\{Z'_i \to U'_i\}$, by the induction hypothesis; 
\item
$\{Z \to V\} \cup \{U'_i \to V\}_i$, by (c);
\item
$\{V \to U\}$, by (d);
\item
$\{Z \to U\} \cup \{Z'_i \to U\}_i$, by (b);
\item
$\{Z \to U\}$, by (a).
\end{itemize}
This completes the induction and hence the proof of (e).

Given (e), let $\{Y_i \to Y\}_i$ be any covering in $X_{\et}$ of and by elements of $\calB$. For each $i$, we can find a covering $\{Y_{ij} \to Y_i\}_j$ such that the composition $Y_{ij} \to Y_i \to Y$ factors as $Y_{ij} \to Z_{ij} \to Y$ where $Z_{ij}$ is finite \'etale and $Y_{ij}$ is an open immersion. 
By Lemma~\ref{L:finite etale from algebra}(b), we may write the image of $Y_{ij} \to Y$ as a finite union $\{U_{ijk}\}_k$ of elements of $\calB$. Put
$Y_{ijk} = Y_{ij} \times_Y U_{ijk}$. The following coverings then have property $\calP$:
\begin{itemize}
\item
$\{Y_{ijk} \to U_{ijk}\}$, by (e);
\item
$\{U_{ijk} \to Y\}$, by (c);
\item
$\{Y_{ijk} \to Y\}$, by (b);
\item
$\{Y_i \to Y\}$, by (a).
\end{itemize}
This completes the proof.
\end{proof}

This gives rise to the following analogue of Proposition~\ref{P:acyclicity template}.
\begin{prop} \label{P:acyclicity template etale}
For $X, \calB$ as in Proposition~\ref{P:djvdp},
let $\calF$ be a presheaf of abelian groups on $X_{\et}$ such that for every $Y = \widetilde{\Spa}(A, A^+) \in \calB$
and every covering $\gothV$ of one of the following forms:
\begin{enumerate}
\item[(a)]
a simple Laurent covering, or
\item[(b)]
a faithfully finite \'etale morphism;
\end{enumerate}
we have
\[
\check{H}^0(Y, \calF; \gothV) = \calF(Y), \qquad \text{resp.} \qquad
\check{H}^i(Y, \calF; \gothV) = \begin{cases} \calF(Y) & i=0 \\
0 & i>0. \end{cases}
\]
Then for every covering $\gothV$ of $Y \in \calB$ by elements of $\calB$, 
\[
H^0(Y, \calF) = \check{H}^0(Y, \calF; \gothV) = \calF(Y), \quad \text{resp.} \quad
H^i(Y, \calF) = \check{H}^i(Y, \calF; \gothV) = \begin{cases} \calF(Y) & i=0 \\
0 & i>0. \end{cases}
\]
In particular, on $Y$, $\calF$ takes the same value as its sheafification.
\end{prop}
\begin{proof}
As in the proof of Proposition~\ref{P:acyclicity template}, using
Proposition~\ref{P:Tate reduction} and Proposition~\ref{P:djvdp} we may successively verify that:
\begin{itemize}
\item $\calF(Y) \to \check{H}^0(Y, \calF; \gothV)$ is injective;
\item $\calF(Y) \to \check{H}^0(Y, \calF; \gothV)$ is bijective;
\item in the second situation, $\gothV$ is universally \v{C}ech-acyclic (its pullback along any morphism $Y' \to Y$ of elements of $\calB$ is \v{C}ech-acyclic);
\end{itemize}
and then check the claims for $H^i(Y, \calF)$ by standard homological algebra.
\end{proof}

For vector bundles on adic spaces, one has analogues of the theorems of Tate and Kiehl.
Recall that we cannot handle general coherent sheaves because rational localizations are in general not flat. (We will return to this issue in a subsequent paper.)
\begin{theorem} \label{T:adic vector bundle}
Let $X = \Spa(A,A^+)$ be an adic affinoid space.
\begin{enumerate}
\item[(a)]
For any finite projective $A$-module $M$,
the presheaf $\tilde{M}$ with $\tilde{M}(U) = M \otimes_A \calO_X(U)$ is an acyclic sheaf for the adic topology and the finite \'etale topology.
\item[(b)]
The categories of finite projective $A$-modules, finite locally free $\calO_X$-modules,
and finite locally free $\calO_{X_{\fet}}$-modules are equivalent.
\item[(c)]
Suppose that $X_{\et}$ admits a stable basis $\calB$. Then
for any finite projective $A$-module $M$, the presheaf $\tilde{M}$ on $X_{\et}$ defined as in (a) is an acyclic sheaf on $\calB$, meaning that for $Y \in \calB$ we have
\[
H^i(Y, \tilde{M}) = \begin{cases} \tilde{M}(Y) & i=0 \\
0 & i>0. \end{cases}
\]
\item[(d)]
Suppose that $X$ is stably adic. Then
the categories of finite projective $A$-modules and finite locally free $\calO_{X_{\et}}$-modules are equivalent.
\end{enumerate}
\end{theorem}
\begin{proof}
To prove (a), note that the adic case follows from Theorem~\ref{T:Tate sheaf property for structure sheaf} and the finite \'etale case follows from
Lemma~\ref{L:finite etale from algebra}(a).
To prove (b), note that the equivalence between the first and second categories follows from Theorem~\ref{T:tate to Kiehl}, while the equivalence between the first and third categories follows from faithfully flat descent
for finite \'etale morphisms
(Theorem~\ref{T:descent modules} and Theorem~\ref{T:descent finite locally free}).
Given (a), we may deduce (c) using Proposition~\ref{P:acyclicity template etale}.
Given (b), we may deduce (d) using Proposition~\ref{P:djvdp}.
\end{proof}

\subsection{Perfectoid spaces}
\label{subsec:perfectoid}

We next globalize the theory of perfectoid algebras to obtain perfectoid spaces,
following Scholze \cite{scholze1} (plus some minor modifications to avoid having to work over a perfectoid field). We also introduce the related notions of \emph{preperfectoid spaces} and \emph{relatively perfectoid spaces}, following Scholze--Weinstein \cite{scholze-weinstein}.

\begin{defn} \label{D:perfectoid space}
A \emph{perfect uniform/perfectoid/preperfectoid/strongly preperfectoid/relatively perfectoid affinoid space} is a preadic affinoid space of the form $\widetilde{\Spa}(A,A^+)$ where $(A,A^+)$ is a perfect uniform/perfectoid/preperfectoid/strongly preperfectoid/relatively perfectoid adic Banach algebra. Any such space is in fact an adic affinoid space by Theorem~\ref{T:Tate-Kiehl analogue1} (in the perfect case),
Theorem~\ref{T:Tate-Kiehl analogue2} (in the perfectoid case),
and Theorem~\ref{T:Kiehl for preperfectoid} (in the remaining cases).

A \emph{perfect/perfectoid/preperfectoid/strongly preperfectoid/relatively perfectoid space} is a preadic space covered by open subspaces which are
perfect/perfectoid/preperfectoid/strongly preperfectoid/relatively perfectoid affinoid spaces. Any such space is in fact an adic space, and even a stably adic space:
thanks to Theorem~\ref{T:perfectoid rational} and Theorem~\ref{T:mixed lift ring},
the subspaces which are themselves perfect/perfectoid/preperfectoid/strongly preperfectoid/relatively perfectoid form a basis for the \'etale topology.
\end{defn}

\begin{prop} \label{P:almost sheaf property}
Let $X = \Spa(A,A^+)$ be a perfect uniform or perfectoid affinoid space, and equip $A$ with the spectral norm. 
\begin{enumerate}
\item[(a)]
We have $H^0(X, \calO) = H^0(X_{\et}, \calO) = A$.
\item[(b)]
For $i>0$, $H^i(X, \calO) = H^i(X_{\et}, \calO) = 0$.
\item[(c)]
For $i>0$, the groups $H^i(X, \calO_X^+), H^i(X_{\et}, \calO_X^+)$ are annihilated by $\gothm_{A}$ (i.e., they are \emph{almost zero}). 
\end{enumerate}
\end{prop}
\begin{proof}
The first two assertions follow from Theorem~\ref{T:uniform rational is sheafy},
Theorem~\ref{T:Tate sheaf property for structure sheaf},
and Theorem~\ref{T:adic vector bundle}.
The third assertion follows from the first two assertions plus Remark~\ref{R:perfect uniform strict}
(in the perfect uniform case)
or Proposition~\ref{P:perfectoid uniform strict}
(in the perfectoid case).
\end{proof}

\begin{remark}
By Proposition~\ref{P:locally perfect is perfect}, any adic affinoid space which is also a perfect uniform space is in fact a perfect uniform affinoid space. The corresponding statement for perfectoid spaces is unknown; see 
Remark~\ref{R:locally perfectoid is perfectoid}.
\end{remark}

\begin{defn}
For $X$ a perfect adic space, we may define sheaves of rings $*_X$ for
\[
* = \tilde{\calE}^{\inte}, \tilde{\calE}, \tilde{\calR}^{\inte,r}, \tilde{\calR}^{\inte,+}, \tilde{\calR}^{\inte},
\tilde{\calR}^{\bd,r}, \tilde{\calR}^{\bd,+}, \tilde{\calR}^{\bd},
\tilde{\calR}^{[s,r]}, \tilde{\calR}^{r}, \tilde{\calR}^+, \tilde{\calR},
\]
by glueing together the presheaves defined in Definition~\ref{D:robba presheaves},
which are sheaves by Theorem~\ref{T:robba presheaves}.
We may also view these as presheaves on $X_{\et}$; by Proposition~\ref{P:acyclicity template etale}, sheafifying these presheaves does not change their values
on any perfect adic affinoid space.
\end{defn}

We globalize the perfectoid correspondence as follows.
\begin{theorem} \label{T:perfectoid correspondence}
There is an equivalence of categories between perfectoid adic spaces and pairs $(X, \calI)$ in which $X$ is a perfect adic space and $\calI$ is an ideal subsheaf of $ \tilde{\calR}^{\inte,+}_X$ which is locally generated by a single element which is primitive of degree $1$. (In the latter category, morphisms have the form $(Y, \calJ) \to (X, \calI)$ where $Y \to X$ is a morphism of perfect uniform adic spaces and $\calJ$ is isomorphic to the pullback of $\calI$.) This equivalence is compatible with rational localizations, fibred products, and \'etale morphisms; moreover, corresponding spaces are functorially homeomorphic.
\end{theorem}
\begin{proof}
This follows by combining Theorem~\ref{T:perfectoid ring},
Theorem~\ref{T:perfectoid rational},
and Theorem~\ref{T:mixed lift ring}.
\end{proof}

This recovers Scholze's tilting correspondence from \cite{scholze1}.
\begin{cor}
Suppose that $K$ is a perfectoid analytic field of characteristic $0$,
and let $K'$ be the corresponding perfect analytic field of characteristic $p$
from Theorem~\ref{T:perfectoid field} (i.e., the \emph{tilt} of $K$ in the sense of Scholze).
Then there is a canonical equivalence of categories between
perfectoid adic spaces over $K$ and perfect uniform adic spaces over $K'$
which is compatible with fibred products.
Moreover, corresponding spaces have homeomorphic underlying topological spaces
and isomorphic \'etale topoi.
\end{cor}

\subsection{\'Etale local systems on adic spaces}
\label{subsec:spectra local systems}

We now wish to define and study \'etale local systems on adic spaces.
For this, we must clarify the distinction between \'etale local systems on the Zariski spectrum and the adic spectrum of an adic Banach ring.
We begin with a refinement of Lemma~\ref{L:henselian local ring}.

\begin{lemma} \label{L:henselian local ring plus}
Let $(A,A^+)$ be an adic Banach ring.
Let $\Spec(A') \to \Spec(A)$ be a surjective \'etale morphism of schemes. Then for any $\alpha \in \calM(A)$, there exists a rational localization $(A,A^+) \to (B,B^+)$
encircling $\alpha$ such that $A' \otimes_A B$ splits as a direct sum of subrings, at least one of which is faithfully finite \'etale over $B$.
\end{lemma}
\begin{proof}
By hypothesis, we can choose some prime ideal $\gothq$ of $A'$ above $\gothp_\alpha$. The norm on $\kappa(\gothp_\alpha)$ induced from $\calH(\alpha)$ then extends to $\kappa(\gothq)$ and thus defines a point $\beta \in \calM(A')$. By the Jacobian criterion, we can write $A'_\gothq$ as a complete intersection $A_{\gothp_\alpha}[x_1,\dots,x_n]/(f_1,\dots,f_n)$ with invertible Jacobian determinant $J$.
We may then choose a rational localization $(A,A^+) \to (B,B^+)$ encircling $\alpha$ so that $f_1,\dots,f_n$ have coefficients in $B$
and that $J$ is invertible in
\[
S[x_1,\dots,x_n]/(f_1,\dots,f_n).
\]
This
yields the desired result.
\end{proof}

\begin{lemma} \label{L:isogeny descent analytic}
Let $(A,A^+)$ be an adic Banach ring. Let $V$ be an \'etale $\Qp$-local system over $\Spec(A)$. Then for any $\alpha \in \calM(A)$, there exists a rational localization $(A,A^+) \to (B,B^+)$ encircling $\alpha$ such that $V \times_{\Spec(A)} \Spec(B)$ is an isogeny $\Zp$-local system on $\Spec(B)$.
\end{lemma}
\begin{proof}
This follows by combining Lemma~\ref{L:isogeny descent}
with Lemma~\ref{L:henselian local ring plus}.
\end{proof}

\begin{defn}
Let $X$ be a preadic space. For each covering of $X$ by
preadic affinoid spaces $U_i = \widetilde{\Spa}(A_i, A_i^+)$, construct the categories of descent data for \'etale $\Zp$-local
systems and etale $\Qp$-local systems for the covering: that is, specify a local system $V_i$ over $\Spec(A_i)$ for each $i$,
then cover each intersection $U_i \cap U_j$ with preadic affinoid spaces $\widetilde{\Spa}(B_k, B_k^+)$ and 
define isomorphisms of the restrictions of $V_i$ and $V_j$ to $B_k$ satisfying the cocycle condition.
Then form the 2-limit (as in Remark~\ref{R:fet direct limit}) over all
covering families; we call the resulting categories the categories of
\emph{\'etale $\Zp$-local systems over $X$}
and
\emph{\'etale $\Qp$-local systems over $X$}.
When $X = \Spa(A,A^\circ)$ for $A$ an affinoid algebra over an analytic field, these categories are equivalent to the corresponding categories
defined by de Jong \cite{dejong-etale} in terms of \'etale covering spaces.
\end{defn}

\begin{remark}
Thanks to the local factorization of \'etale morphisms, one gets the same categories of local systems if one takes coverings in the \'etale topology rather than the adic topology.
\end{remark}

\begin{remark} \label{R:analytic Zp equivalence}
The natural functor from
\'etale $\Zp$-local systems over $\Spec(A)$ to \'etale $\Zp$-local systems over $\widetilde{\Spa}(A,A^+)$
is an equivalence of categories, since $\Zp$-local systems are determined by finite \'etale algebras
(Remark~\ref{R:local systems rings equivalence}) and these glue over covering families of rational localizations
(Theorem~\ref{T:henselian direct limit3}).
The corresponding statement for $\Qp$-local systems is false; see Remark~\ref{R:different local systems}.
\end{remark}

We have the following variant of Lemma~\ref{L:isogeny descent analytic}, with essentially the same proof.
\begin{prop} \label{P:glueing weak covering}
Let $(A,A^+)$ be an adic Banach ring. Let $V$ be an \'etale $\Qp$-local system over $\widetilde{\Spa}(A,A^+)$. Then for any $\alpha \in \calM(A)$, there exists a rational localization $(A,A^+) \to (B,B^+)$ encircling $\alpha$ such that the restriction of $V$ to $\widetilde{\Spa}(B,B^+)$ is an isogeny $\Zp$-local system.
\end{prop}
\begin{proof}
By Remark~\ref{R:lattice space} 
and Lemma~\ref{L:finite etale from algebra}(a),
given a preadic space $X$ and an object $T \in \ZpLoc(X)$, the functor taking a morphism $Y \to X$ to the pairs $(T', \iota)$ in which $T' \in \ZpLoc(Y)$ and $\iota: T_Y \otimes \Qp \to T' \otimes \Qp$ is an isomorphism such that $p^m \iota \in \Mor(T_Y, T')$, $p^m \iota^{-1} \in \Mor(T', T_Y)$
is representable by a finite \'etale morphism $L_m(T) \to X$.
By this construction plus Theorem~\ref{T:henselian} and Lemma~\ref{L:henselian local ring}, we may find a rational localization $(A, A^+) \to (B,B^+)$ encircling $\alpha$, a faithfully finite \'etale morphism $(B, B^+) \to (C,C^+)$, and a $\Zp$-lattice in $V$
over $\widetilde{\Spa}(C,C^+)$ admitting a descent datum relative to $\widetilde{\Spa}(B,B^+)$. We may thus 
invoke Lemma~\ref{L:isogeny descent} to conclude.
\end{proof}

\begin{cor} \label{C:glueing weak covering}
Let $(A,A^+)$ be an adic Banach ring. Then any \'etale $\Qp$-local system over $\widetilde{\Spa}(A,A^+)$ can be realized as a descent datum for isogeny $\Zp$-local systems for some strong covering family of rational localizations.
\end{cor}
\begin{proof}
This is immediate from Proposition~\ref{P:glueing weak covering} and the compactness of $\calM(A)$.
\end{proof}

\begin{remark} \label{R:different local systems}
The natural functor from \'etale $\Qp$-local systems over $\Spec(A)$
to \'etale $\Qp$-local systems over $\widetilde{\Spa}(A,A^+)$ is fully faithful, but
it need not be essentially surjective even when $A$ is a reduced affinoid algebra over an analytic field, as observed by de Jong.
For instance, suppose that $A$ is an integral affinoid algebra over an analytic field. Then on one hand $\Spec(A)$ is normal and noetherian, so $\ZpILoc(\Spec(A))
= \QpLoc(\Spec(A))$ by Remark~\ref{R:normal noetherian local systems}.
On the other hand,
\'etale $\Qp$-local systems over $\Spa(A, A^+)$
correspond to continuous representations of the \'etale fundamental group of $\calM(A)$
(as defined in \cite[\S 2.6]{dejong-etale}) on finite-dimensional $\Qp$-vector spaces,
and such representations can fail to have compact image.
Typical examples arise from instances of $p$-adic uniformization, such as the Tate
uniformization of an elliptic curve of split multiplicative reduction;
see Example~\ref{exa:Tate curve}.
More examples of this sort arise from Rapoport-Zink period morphisms; see \cite[\S 7]{dejong-etale}.
\end{remark}

\begin{remark} \label{R:extend Qp-local systems}
For any adic Banach ring $(A,A^+)$ and any isogeny $\Zp$-local systems $V_1, V_2$ on $\Spec(A)$,
any extension $0 \to V_1 \to V \to V_2 \to 0$ in the category of \'etale $\Qp$-local systems on $\widetilde{\Spa}(A,A^+)$
descends to an extension of isogeny $\Zp$-local systems on $\Spec(A)$. To see this, start with identifications
$V_i = T_i \otimes_{\Zp} \Qp$ for some \'etale $\Zp$-local systems on $\Spec(A)$. The extension
$0 \to V_1 \to V \to V_2 \to 0$ then corresponds \'etale locally on $\widetilde{\Spa}(A,A^+)$ to a class in
$\Ext(T_2, T_1) \otimes_{\Zp} \Qp$;
by Corollary~\ref{C:glueing weak covering},
we can find a strong covering family
$\{(A,A^+) \to (B_i, B_i^+)\}_i$ of rational localizations such that
the extension corresponds to a class $x_i$ in $\Ext(T_2, T_1) \otimes_{\Zp} \Qp$ over $\Spec(B_i)$.
Put $B_{ij} = B_i \widehat{\otimes}_A B_j$; then $x_i - x_j$ vanishes as an element of
$\Ext(T_2, T_1) \otimes_{\Zp} \Qp$ over $\Spec(B_{ij})$.
We may rescale $T_1$ by a power of $p$ first to force each $x_i$ into $\Ext(T_2, T_1)$ over $\Spec(B_i)$,
then to force $x_i - x_j$ to vanish in $\Ext(T_2, T_1)$ over $\Spec(B_{ij})$.
The extensions then glue to define an extension of \'etale $\Zp$-local systems on $\widetilde{\Spa}(A,A^+)$.
We may now invoke Remark~\ref{R:analytic Zp equivalence} to conclude.
\end{remark}

\subsection{\texorpdfstring{$\varphi$}{phi}-modules and local systems}
\label{subsec:phi-modules}

We now relate pure and \'etale $\varphi$-modules to \'etale local systems on perfectoid adic spaces.

\begin{hypothesis} \label{H:phi-modules}
Throughout \S\ref{subsec:phi-modules},
let $d$ be a positive integer. Write $\QQ_{p^d}$ for the finite unramified extension of $\Qp$ of degree $d$ and $\ZZ_{p^d}$ for the valuation subring of $\QQ_{p^d}$.
\end{hypothesis}

We begin with the case of $\Zp$-local systems.

\begin{lemma} \label{L:small translate}
Let $R$ be a perfect Banach algebra over $\FF_p$ with spectral norm $\alpha$.
For any $c > 1$, any positive integer $d$, and any $x \in R$, there exists $y \in R$ with
$\alpha(x - y + y^{p^d}) < c$.
\end{lemma}
\begin{proof}
If $\alpha(x) \leq 1$, we may take $y = 0$. Otherwise, we may take
$y = -(x^{p^{-d}} + \cdots + x^{p^{-md}})$ for any positive integer $m$ which is large
enough that $\alpha(x)^{p^{-md}} < c$, as then $x - y + y^{p^d} = x^{p^{-md}}$.
\end{proof}

\begin{theorem} \label{T:perfect equivalence2}
For $(R,R^+)$ a perfect uniform adic Banach algebra over $\FF_{p^d}$,
the following categories are equivalent.
\begin{enumerate}
\item[(a)] The category of \'etale $\ZZ_{p^d}$-local systems over $\Spec(R)$ (or equivalently $\Spa(R,R^+)$).
\item[(b)] The category of \'etale $\ZZ_{p^d}$-local systems over $\Spec(R_0)$ (or equivalently $\Spa(R_0,R_0^+)$) for any complete subring $R_0$ of $R$ whose completed direct perfection is equal to $R$, taking $R_0^+ = R_0 \cap R^+$.
\item[(c)] The category of \'etale $\ZZ_{p^d}$-local systems over $\Spec(A)$ (or equivalently $\Spa(A,A^+)$) for 
$(A,A^+)$ corresponding to $(R,R^+)$ as in Theorem~\ref{T:perfectoid ring}.
\item[(d)] The category of $\varphi^d$-modules over $\tilde{\calE}^{\inte}_{R}$.
\item[(e)] The category of $\varphi^d$-modules over $\tilde{\calR}^{\inte}_R$.
\end{enumerate}
More precisely, the functor from (e) to (d) is base extension.
\end{theorem}
\begin{proof}
The equivalences between (a) and (b) and between (a) and (c)
follow from Remark~\ref{R:local systems rings equivalence} combined with
Theorem~\ref{T:perfect etale} and Theorem~\ref{T:mixed lift ring},
respectively. The equivalence between (a) and (d) follows immediately from Proposition~\ref{P:phi-modules}.

We next check that the base extension functor from (e) to (d) is fully faithful.
As in Remark~\ref{R:etale descent},
it suffices to check that for
$M$ a $\varphi^d$-module over $\tilde{\calR}^{\inte}_R$,
any $\varphi^d$-stable element $\bv \in M \otimes_{\tilde{\calR}^{\inte}_R} W(R)$
belongs to $M$. This claim may be checked locally on $\calM(R)$,
so we may assume that $M$ is free over $\tilde{\calR}^{\inte}_R$.
By Proposition~\ref{P:DM relative}, we can choose an $R$-algebra $S$ which is a completed
direct limit of faithfully finite \'etale $R$-subalgebras, in such a way that
$M \otimes_{\tilde{\calR}^{\inte}_R} \tilde{\calR}^{\inte}_S$ admits a $\varphi^d$-invariant
basis $\be_1,\dots,\be_n$. If we write $\bv = \sum_i x_i \be_i$ with $x_i \in W(S)$,
then $x_i \in W(S)^{\varphi^d} = W(S^{\overline{\varphi}^d})$,
and the latter ring is contained in $\tilde{\calR}^{\inte}_S$ by
Remark~\ref{R:phi-invariant}. Since $W(R) \cap \tilde{\calR}^{\inte}_S = \tilde{\calR}^{\inte}_R$,
it follows that $\bv \in M$ as desired.

We finally check that the functor from (e) to (d) is essentially surjective.
Let $M$ be a $\varphi^d$-module over $W(R)$.
By faithfully flat descent (Theorem~\ref{T:descent modules} and Theorem~\ref{T:descent finite locally free}),
to check that $M$ arises by base extension from $\tilde{\calR}^{\inte}_R$,
it suffices to do so after replacing $R$ with a faithfully finite \'etale extension.
Since (a) and (d) are equivalent, we may reduce to the case where $M$
admits a basis $\be_1,\dots,\be_n$ on which $\varphi^d$ acts via a matrix
$F$ for which $F-1$ has entries in $pW(R)$.

Let $\alpha$ be the spectral norm on $R$.
We define matrices $F_n, G_n$ for each positive integer $n$ such that $F_1 = F$, $G_1 = 1$,
$F_n-1$ has entries in $p W(R)$, $G_n$ has entries in
$\tilde{\calR}^{\inte,1}_R$, $\lambda(\alpha)(G_n - 1) < 1$,
and $X_n = p^{-n}(F_n - G_n)$ has entries in $W(R)$.
Namely, given $F_n$ and $G_n$, apply Lemma~\ref{L:small translate} to construct
a matrix $\overline{Y}_n$ over $R$ so that
$\alpha(\overline{X}_n - \overline{Y}_n + \overline{\varphi}^d(\overline{Y}_n)) < p^{n/2}$.
Then put $U_{n} = 1 + p^n [\overline{Y}_n]$ (the entries of $[\overline{Y}_n]$ are Teichm\"{u}ller liftings of corresponding entries of $\overline{Y}_n$), $F_{n+1} = U_{n}^{-1} F_n \varphi^d(U_{n})$,
$G_{n+1} = G_n + p^n [\overline{X}_n - \overline{Y}_n + \overline{\varphi}^d(\overline{Y}_n)]$.
The product $U_1 U_2 \cdots$ converges to a matrix $U$ so that
$U^{-1} F \varphi^d(U)$ is equal to the $p$-adic limit of the $G_n$, which is invertible over $\tilde{\calR}^{\inte,r}_R$
for any $r \in (0,1)$.
Consequently, the $\tilde{\calR}^{\inte}_R$-span of the vectors
$\bv_1,\dots,\bv_n$ defined by $\bv_j = \sum_i U_{ij} \be_i$ gives a
$\varphi^d$-module $N$ over $\tilde{\calR}^{\inte}_R$ for which
$N \otimes_{\tilde{\calR}^{\inte}_R} W(R) \cong M$.
\end{proof}

\begin{remark}
In Theorem~\ref{T:perfect equivalence2}, the equivalence between \'etale $\Zp$-local systems over
$R$ and $\varphi$-modules over $W(R)$ can also be interpreted as a form of nonabelian Artin-Schreier theory,
using \emph{Lang torsors}.
For example,
see \cite[Proposition~4.12]{milne} for a derivation of ordinary Artin-Schreier theory in this framework.
\end{remark}

Thanks to Theorem~\ref{T:perfectoid correspondence},
Theorem~\ref{T:perfect equivalence2} immediately globalizes as follows.
\begin{theorem} \label{T:perfect equivalence2 global}
Let $X$ be a perfectoid adic space over $\QQ_{p^d}$ and let $X'$ be the corresponding perfect uniform adic space over $\FF_{p^d}$.
Then the following categories are equivalent.
\begin{enumerate}
\item[(a)] The category of \'etale $\ZZ_{p^d}$-local systems over $X$.
\item[(b)] The category of \'etale $\ZZ_{p^d}$-local systems over $X'$.
\item[(c)]
The category of \'etale $\ZZ_{p^d}$-local systems over $X'_0$
for any adic space $X'_0$ whose inverse perfection (i.e., its inverse limit along absolute Frobenius) is isomorphic to $X'$.
\item[(d)] The category of $\varphi^d$-modules over $\tilde{\calE}^{\inte}_{X'}$.
\item[(e)] The category of $\varphi^d$-modules over $\tilde{\calR}^{\inte}_{X'}$.
\end{enumerate}
\end{theorem}

We next consider isogeny $\Zp$-local systems.

\begin{theorem} \label{T:perfect equivalence2a}
For $(R,R^+)$ a perfect uniform adic Banach algebra over $\FF_{p^d}$, the following categories are equivalent.
\begin{enumerate}
\item[(a)] The category of isogeny $\ZZ_{p^d}$-local systems over $\Spec(R)$ (or equivalently $\Spa(R,R^+)$).
\item[(b)] The category of isogeny $\ZZ_{p^d}$-local systems over $\Spec(R_0)$ (or equivalently $\Spa(R_0, R_0^+))$) for any complete subring $R_0$ of $R$ whose completed direct perfection is equal to $R$, taking $R_0^+ = R_0 \cap R^+$.
\item[(c)] The category of isogeny $\ZZ_{p^d}$-local systems over $\Spec(A)$ (or equivalently $\Spa(A,A^+)$)
for $(A,A^+)$ corresponding to $(R,R^+)$ as in Theorem~\ref{T:perfectoid ring}.
\item[(d)] The category of globally \'etale  $\varphi^d$-modules over $\tilde{\calE}_R$.
\item[(e)] The category of globally \'etale $\varphi^d$-modules over $\tilde{\calR}^{\bd}_R$.
\item[(f)] The category of globally \'etale $\varphi^d$-modules over $\tilde{\calR}_R$.
\end{enumerate}
More precisely, the functors from (e) to (d) and (f) are base extensions.
\end{theorem}
\begin{proof}
The equivalences between (a) and (b) and between (a) and (c)
again follow from Remark~\ref{R:local systems rings equivalence} combined with
Theorem~\ref{T:perfect etale} and Theorem~\ref{T:mixed lift ring},
respectively.

The functor from (a) to (e) is constructed as follows. Let $V$ be an isogeny $\ZZ_{p^d}$-local
system over $\Spec(R)$; we may write $V = T \otimes_{\Zp} \Qp$ for some $\ZZ_{p^d}$-local system $T$
on $\Spec(R)$. The latter corresponds to a $\varphi^d$-module $M_{0,R}$ over
$\tilde{\calR}^{\inte}_R$ by Theorem~\ref{T:perfect equivalence2}.
The assignment $V \to M_{0,R} \otimes_{\tilde{\calR}^{\inte}_R} \tilde{\calR}^{\bd}_R$
then defines a fully faithful functor by Corollary~\ref{C:extended Robba invariants};
by the same reasoning, the resulting functors from (a) to (d) and from (a) to (f) are fully faithful.

To construct the functor from (e) back to (a), given a $\varphi^d$-module $M$ over
$\tilde{\calR}^{\bd}_R$ admitting a locally free \'etale model $M_0$,
apply Theorem~\ref{T:perfect equivalence2} to convert $M_0$ into a $\ZZ_{p^d}$-local system $T$
on $\Spec(R)$. The assignment $M \to T \otimes_{\Zp} \Qp$ defines a quasi-inverse to the functor
from (a) to (e). By similar reasoning, the functors from (a) to (d) and from (a) to (f) are equivalences
of categories.
\end{proof}

\begin{defn} \label{D:etale twist}
For $c,d \in \ZZ$ with $d>0$, define an \emph{isogeny $(c,d)$-$\Zp$-local system}
on a scheme $X$ to be an isogeny $\ZZ_{p^d}$-local system $V$ on $X$ equipped with a semilinear
action of the Frobenius automorphism $\tau$ of $\QQ_{p^d}$ on sections,
such that $p^c \tau^d$ acts as the identity. For $c',d' \in \ZZ$ with $d'>0$ and $c'/d' = c/d$,
the categories of isogeny $(c,d)$-$\Zp$-local systems and isogeny $(c',d')$-$\Zp$-local systems
on any scheme are naturally equivalent: this reduces to the case where $c' = ce, d' = de$ for some
positive integer $e$, in which case the claim is an easy exercise using Hilbert's Theorem 90.
(See \cite[Th\'eor\`eme~3.2.3]{berger-b-pairs} for a similar construction.)
We may similarly define \emph{\'etale $(c,d)$-$\Qp$-local systems}.
\end{defn}

\begin{theorem} \label{T:perfect equivalence2a1}
For $(R,R^+)$ a perfect uniform adic Banach algebra over $\FF_{p^d}$ and $c \in \ZZ$,
the following categories are equivalent.
\begin{enumerate}
\item[(a)] The category of isogeny $(c,d)$-$\Zp$-local systems over $\Spec(R)$.
\item[(b)] The category of isogeny $(c,d)$-$\Zp$-local systems over $\Spec(R_0)$ for any subring $R_0$ of $R$ whose completed direct perfection is equal to $R$.
\item[(c)] The category of isogeny $(c,d)$-$\Zp$-local systems over $\Spec(A)$ for $A = \tilde{\calR}_R^{\inte,1}/(z)$
for any $z \in W(R^+)$ which is primitive of degree $1$.
\item[(d)] The category of globally $(c,d)$-pure $\varphi$-modules over $\tilde{\calE}_R$.
\item[(e)] The category of globally $(c,d)$-pure $\varphi$-modules over $\tilde{\calR}^{\bd}_R$.
\item[(f)] The category of globally $(c,d)$-pure $\varphi$-modules over $\tilde{\calR}_R$.
\end{enumerate}
More precisely, the functors from (e) to (d) and (f) are base extensions.
\end{theorem}
\begin{proof}
This is immediate from Theorem~\ref{T:perfect equivalence2a}.
\end{proof}

We finally pass to $\Qp$-local systems on adic spaces.
\begin{defn}
Let $X$ be a perfect uniform adic space over $\FF_{p}$ and let $M$ be a $\varphi$-module over one of $\tilde{\calE}_X$, $\tilde{\calR}^{\bd}_X$, $\tilde{\calR}_X$.
For $x \in X$, we say that $M$ is \emph{$(c,d)$-pure at $x$}
if $M$ is $(c,d)$-pure at the rank 1 seminorm induced by $x$.
We say $M$ is \emph{$(c,d)$-pure} if it is $(c,d)$-pure
at each $x\in X$. Similarly, we define the \emph{slope polygon} of $M$ at $x$ by passing to the induced rank 1 seminorm; this leads to corresponding definitions of the \emph{pure locus} and \emph{\'etale locus} of $M$. By definition, these sets are pullback from subsets of the real quotient $\overline{X}$ when the latter is defined.
\end{defn}

\begin{remark} \label{R:purity local}
Let $(R,R^+)$ be a perfect uniform Banach algebra over $\FF_{p}$ and put $X = \Spa(R,R^+)$.
For $* = \tilde{\calE}, \tilde{\calR}^{\bd}, \tilde{\calR}$,
there is a natural functor from $\varphi^d$-modules  over $*_R$ to
$\varphi^d$-modules over $*_X$; we sometimes refer to the latter as \emph{local $\varphi^d$-modules} over $*_R$.
These functors are fully faithful by Theorem~\ref{T:robba presheaves}.
The functors for $\tilde{\calE}$ and $\tilde{\calR}^{\bd}$
are not equivalences of categories (see Example~\ref{exa:Tate curve}),
but the functor for $\tilde{\calR}$ is an equivalence of categories
by Corollary~\ref{C:phi-modules glueing}.
In any case, thanks to the local nature of the pure and \'etale conditions,
one sees easily that a $\varphi^d$-module over $*_R$ is pure or \'etale if and only if the
corresponding $\varphi^d$-module over $*_X$ has this property.
\end{remark}

\begin{lemma} \label{L:pure locus is open taut}
Let $X$ be a perfect uniform adic space over $\FF_{p^d}$ and let $M$ be a $\varphi^d$-module over $\tilde{\calR}_X$. Then the pure locus and the \'etale locus of $\tilde{\calR}_X$ are open and partially proper. In particular, by Lemma~\ref{L:taut Hausdorff quotient}, if $X$ is taut, then so are the pure locus and the \'etale locus.
\end{lemma}
\begin{proof}
This is immediate from Corollary~\ref{C:etale locus is open}.
\end{proof}

\begin{theorem} \label{T:perfect equivalence2b}
Let $X$ be a perfectoid adic space over $\QQ_{p^d}$ and let $X'$ be the corresponding perfect uniform adic space over $\FF_{p^d}$.
Then the following categories are equivalent.
\begin{enumerate}
\item[(a)] The category of \'etale $(c,d)$-$\QQ_{p}$-local systems over $X$.
\item[(b)] The category of \'etale $(c,d)$-$\QQ_{p}$-local systems over $X'$.
\item[(c)]
The category of \'etale $(c,d)$-$\QQ_{p}$-local systems over $X'_0$
for any adic space $X'_0$ whose inverse perfection  is isomorphic to $X'$.
\item[(d)] The category of $(c,d)$-pure $\varphi$-modules over $\tilde{\calE}_{X'}$.
\item[(e)] The category of $(c,d)$-pure $\varphi$-modules over $\tilde{\calR}^{\bd}_{X'}$.
\item[(f)] The category of $(c,d)$-pure $\varphi$-modules over $\tilde{\calR}_{X'}$.
\end{enumerate}
\end{theorem}
\begin{proof}
This is immediate from
Theorem~\ref{T:perfectoid correspondence}
and Theorem~\ref{T:perfect equivalence2a1}.
\end{proof}

\begin{cor} \label{C:same purity}
Let $X$ be a perfect uniform adic space over $\FF_{p^d}$.
Then a $\varphi^d$-module over $\tilde{\calE}_X$, $\tilde{\calR}^{\bd}_X$, $\tilde{\calR}_X$ is pure of slope $s$
at some $x \in X$ if and only if it is $(c',d')$-pure at $x$ for every (not just one)
pair $c',d'$ of integers for which $d'$ is a positive multiple of $d$ and $c'/d' = s$.
\end{cor}
\begin{proof}
This follows from Theorem~\ref{T:perfect equivalence2a1}
plus the corresponding equivalence on the side of local
systems (Definition~\ref{D:etale twist}).
\end{proof}

\begin{cor} \label{C:bounded pointwise pure}
Let $M$ be a $\varphi^d$-module over $\tilde{\calE}_R$, $\tilde{\calR}^{\bd}_R$, or $\tilde{\calR}_R$. If $M$ is pointwise pure, then $M$ is pure.
\end{cor}
\begin{proof}
Over $\tilde{\calR}_R$ this is immediate from
Corollary~\ref{C:pointwise etale is etale}, so from now on we assume that $M$ is a $\varphi^d$-module over $\tilde{\calE}_R$ (resp.\
$\tilde{\calR}^{\bd}_R$).
Choose $\beta \in \calM(R)$, put $n= \rank(M, \beta)$, and choose $c',d' \in \ZZ$ with $d'$ a positive multiple of $d$ and $c'/d' = \mu(M,\beta)$.
Put $M_{\beta} = M \otimes_{\tilde{\calE}_R} \tilde{\calE}_{\calH(\beta)}$ 
(resp.\ $M_\beta = M \otimes_{\tilde{\calR}^{\bd}_R} \tilde{\calR}^{\bd}_{\calH(\beta)}$).

Suppose first that $M_{\beta}$ admits a \emph{cyclic vector}, i.e.,
an element $\be$ such that $(p^{c'} \varphi^{d'})^i(\be)$ for $i=0,\dots,n-1$ are linearly independent. 
Then any element of $M_{\beta}$ sufficiently close to $\be$ for the weak topology (resp.\ the LF topology) is also a cyclic vector, so we may choose $\be \in M_S$ for
$R \to S$ a rational localization encircling $\beta$ and $M_S = M \otimes_{\tilde{\calE}_R} \tilde{\calE}_S$ (resp.\ $M_S = M \otimes_{\tilde{\calR}^{\bd}_R} \tilde{\calR}^{\bd}_S$). For a suitable choice of $S$, $(p^{c'} \varphi^{d'})^i(\be)$ for $i=0,\dots,n-1$ form a basis of $M_S$. For such $S$, the $W(S)$-submodule (resp.\ $\tilde{\calR}^{\inte}_S$-submodule) of $M$ spanned by this basis forms a free pure model
by \cite[Lemma~5.2.4]{kedlaya-revisited}.

To handle the general case, write $R$ as a Banach algebra over a perfect analytic field $L$. Choose $\overline{x} \in L^\times$ of norm less than 1.
Let $\be_1,\dots,\be_n$ be a basis of $M_{\beta}$.
Then for some $r_1,\dots,r_n \in \ZZ[p^{-1}]$,
$[\overline{x}]^{r_1} \be_1 + \cdots + [\overline{x}]^{r_n} \be_n$ is a cyclic vector of $M_\beta$, so the previous paragraph shows that $M$ is pure.
\end{proof}

\begin{cor} \label{C:purity base extension extended}
Let $(R,R^+) \to (S,S^+)$ be a bounded homomorphism of perfect uniform adic Banach algebras over $\FF_{p^d}$ for which $\Spa(S,S^+) \to \Spa(R,R^+)$
is surjective. Let $M$ be a local $\varphi^d$-module over $\tilde{\calE}_R$ (resp.\ $\tilde{\calR}^{\bd}_R$, $\tilde{\calR}_R$).
Then $M$ is pure if and only if $M \otimes_{\tilde{\calE}_R} \tilde{\calE}_S$ (resp.\ $M \otimes_{\tilde{\calR}^{\bd}_R} \tilde{\calR}^{\bd}_S$,
$M \otimes_{\tilde{\calR}_R} \tilde{\calR}_S$) is pure.
\end{cor}
\begin{proof}
By Corollary~\ref{C:purity base extension},
$M$ is pointwise pure. By Corollary~\ref{C:bounded pointwise pure},
$M$ is pure.
\end{proof}
\begin{cor}
Let $Y \to X$ be a surjective morphism of perfectoid adic spaces.
Let $M$ be a $\varphi^d$-module over $\tilde{\calR}_X$. Then
$M$ is pure (resp.\ \'etale) if and only if the pullback of $M$ to
$\tilde{\calR}_Y$ is pure (resp.\ \'etale).
\end{cor}

Here are some examples to illustrate the difference between $\varphi$-modules and local $\varphi$-modules, and between globally \'etale and \'etale $\varphi$-modules.
\begin{example} \label{exa:Tate curve}
Put $K = \FF_p((q))$ for an arbitrary normalization $|q| = \omega < 1$ of the $q$-adic norm,
and define the strictly affinoid algebras
\[
B = K\{\omega^2/T, T,U/\omega^{-2}\}/(U(T-q)-1),  \qquad
B_1 = K\{\omega^2/T, T/\omega^2\}, \qquad
B_2 = K\{1/T, T\}.
\]
over $K$. In words, $\Spa(B,B^\circ)$ is the annulus $\omega^2 \leq |T| \leq 1$ minus the open disc
$|T-q| < \omega^2$, and
$\Spa(B_1,B_1^\circ)$ and $\Spa(B_2, B_2^\circ)$ are the boundary circles $|T| = \omega^2$ and $|T| = 1$, respectively,
within $\Spa(B, B^\circ)$.
Let $\sigma_q: B_2 \to B_1$ be the substitution $T \mapsto q^2 T$.
If we quotient $\Spa(B,B^\circ)$ by the identification $\Spa(B_2, B_2^\circ) \cong \sigma_q^* \Spa(B_1, B_1^\circ)$,
we obtain a strictly affinoid subspace $\Spa(A, A^\circ)$ of the
Tate curve over $X$ for the parameter $q^2$. The latter is the analytification of a
smooth projective curve over $K$ of genus 1; see for instance \cite[Theorem~V.3.1]{silverman-aec2}
for explicit equations.

We may construct an \'etale $\Qp$-local system $V$ on $\Spa(A,A^\circ)$ as follows.
Let $\tilde{V}$ be the trivial $\Qp$-local system on $\Spa(B,B^\circ)$, equipped with the distinguished generator $1$.
Let $\tilde{V}_1, \tilde{V_2}$ be the restrictions of $\tilde{V}$ to $\Spa(B_1,B_1^\circ), \Spa(B_2,B_2^\circ)$,
respectively. To specify $V$, it suffices to specify an isomorphism $\tilde{V}_1 \cong \sigma_q^* \tilde{V}_2$;
we choose the isomorphism matching $1 \in \tilde{V}_1$ with $p \in \sigma_q^* \tilde{V}_2$.

Let $R, S, S_1, S_2$ be the completed perfections of $A,B,B_1,B_2$, respectively. We claim that $V$ cannot correspond to an \'etale $\varphi$-module
$M$ over $\tilde{\calE}_R$ (or over the subring $\tilde{\calR}^{\bd}_R$ thereof).
To check this, suppose the contrary, and choose any nonzero element $\bv \in M$.
The pullback of $M$ to $\tilde{\calE}_S$ can be identified with the trivial $\varphi$-module
$\tilde{\calE}_S$ itself, and $\bv$ must correspond to an element $x \in \tilde{\calE}_S$.
Let $x_1,x_2$ be the images of $x$ in $\tilde{\calE}_{S_1}, \tilde{\calE}_{S_2}$, respectively.
We must then have
\[
x_2 = p \sigma_q(x_1) \in \tilde{\calE}_{S_2} = W(S_2)[p^{-1}].
\]
However, this is impossible: the maps $S \to S_1, S_2$ are injective, so an element
of $W(S)[p^{-1}]$ which maps to $W(S_1)$ or $W(S_2)$ must itself belong to $W(S)$.
Thus if $x \in p^m W(S)$ for some $m \in \ZZ$, then also $x \in p^{m+1}W(S)$, which cannot hold
for all $m$ if $x \neq 0$.

By contrast, by Theorem~\ref{T:perfect equivalence2b},
$V$ does correspond to
an \'etale $\varphi$-module over $\tilde{\calR}_R$,
and to \'etale $\varphi$-modules over $\tilde{\calE}_X$ and $\tilde{\calR}^{\bd}_X$. By the previous paragraph, however, the latter do not descend to \'etale $\varphi$-modules over  $\tilde{\calE}_R$ or $\tilde{\calR}^{\bd}_R$; in particular, we obtain an obstruction to glueing finite projective modules over these rings
as indicated in Remark~\ref{R:no glueing}.
This lack of descent in turn provides an example of an \'etale $\varphi$-module over $\tilde{\calR}_R$ which does not admit a (not necessarily finite locally free) \'etale model.
\end{example}

\begin{example} \label{exa:banana}
Let $K$ be an algebraically closed analytic field of characteristic $p \neq 2$. 
Let $R$ be the completed perfection of $K\{X,Y\}/(y^2-(x^2-1)^2)$.
Put $X = \Spa(R,R^\circ)$. Then $X$ admits a \'etale cover consisting of a doubly infinite chain of copies of $\Spa(K\{X\}, K\{X\}^\circ)$; as in
Example~\ref{exa:Tate curve}, this gives rise to an \'etale $\Qp$-local system $V$ on $X$ which is not an isogeny $\Zp$-local system. 
By Theorems~\ref{T:perfect equivalence2a1}
and~\ref{T:perfect equivalence2b},
$V$ corresponds to \'etale $\varphi$-modules over $\tilde{\calE}_X$, $\tilde{\calE}^{\bd}_X, \tilde{\calR}_X$ which are not globally \'etale.
\end{example}

\begin{remark} \label{R:local systems affinoid2}
If $A$ is a connected affinoid algebra over an analytic field $K$, then an \emph{\'etale fundamental group}
of $\calM(A)$ has been defined by de Jong \cite{dejong-etale}; its continuous representations on
finite-dimensional $\Qp$-vector spaces correspond precisely to \'etale $\Qp$-local systems on $\Spa(A,A^\circ)$
in our sense
\cite[Lemma~2.6]{dejong-etale}.
It should be possible to show using Theorem~\ref{T:quotient norm}(c) and Theorem~\ref{T:mixed lift ring}
that $\calM(A)$ and $\calM(\tilde{\calR}^{\inte,1}_R/(z))$ have the same \'etale fundamental group;
the only serious issue is that $\tilde{\calR}^{\inte,1}_R/(z)$ need not be an affinoid algebra over an analytic field,
so some work is needed to define the \'etale fundamental group and check some basic properties.
\end{remark}

\subsection{A bit of cohomology}
\label{subsec:bit of cohomology}

We next relate the cohomology of $\Zp$-local systems and isogeny $\Zp$-local systems with $\varphi$-modules; again, this comes down to nonabelian Artin-Schreier-Witt theory.
The corresponding statements for adic spaces require additional work even to assert what is meant by \'etale cohomology; this work is carried out in the next section.

\begin{hypothesis}
Throughout \S\ref{subsec:bit of cohomology}, retain Hypothesis~\ref{H:phi-modules}, and in addition let $R$ be a perfect uniform Banach algebra over $\FF_{p^d}$.
\end{hypothesis}

\begin{theorem} \label{T:Galois cohomology2}
Let $T$ be an \'etale $\ZZ_{p^d}$-local system on $\Spec(R)$.
Let $M$ be the $\varphi^d$-module over $\tilde{\calR}^{\inte}_R$ or $W(R)$ corresponding to $T$
via Theorem~\ref{T:perfect equivalence2}.
Then there are natural (in $T$ and $R$) bijections
$H^i_{\et}(X, T) \cong H^i_{\varphi^d}(M)$ for all $i \geq 0$.
\end{theorem}
\begin{proof}
Suppose first that $M$ is defined over $W(R)$.
For each positive integer $n$, view $T/p^n T$ as a locally constant \'etale sheaf on $\Spec(R)$,
and let $\tilde{M}_n$ be the \'etale sheaf on $\Spec(R)$ corresponding to the quasicoherent sheaf
on $\Spec(W(R)/(p^n))$ with global sections $M/p^n M$. We then have an exact sequence
\[
0 \to T/p^n T \to \tilde{M}_n \stackrel{\varphi^d-1}{\to} \tilde{M}_n \to 0,
\]
where exactness at the right is given by Theorem~\ref{T:perfect equivalence2}
(or more directly by Proposition~\ref{P:DM relative}).
We see by induction on $n$ that $\tilde{M}_n$ is acyclic: it is enough to check that
$\ker(\tilde{M}_n \to \tilde{M}_{n-1})$ is acyclic, which it is because it arises from a quasicoherent sheaf on an affine scheme.
Taking the long exact sequence
in cohomology thus yields the desired result.

Suppose next that $M$ is defined over $\tilde{\calR}^{\inte}_R$; by the previous paragraph, we need only
check the cases $i=0,1$. For these, interpret $H^i_{\varphi^d}(M)$ as an extension group as in
Definition~\ref{D:varphi cohomology}, then note that
any extension of two $\varphi^d$-modules is again a $\varphi^d$-module.
By Theorem~\ref{T:perfect equivalence2}, the 0th and 1st extension groups do not change upon base extension to $W(R)$;
we may thus deduce the claim from the previous paragraph.
\end{proof}

\begin{lemma} \label{L:two of three}
Let $0 \to M_1 \to M \to M_2 \to 0$ be a short exact sequence of $\varphi$-modules over $\tilde{\calR}_R$.
If any two of $M, M_1, M_2$ are $(c,d)$-pure, then so is the third.
\end{lemma}
\begin{proof}
By Corollary~\ref{C:pointwise etale is etale},
it suffices to treat the case that $R = L$ is an analytic field.
In this case, the lemma follows immediately from Theorem~\ref{T:slope filtration2}.
\end{proof}

We have a similar result for $\Qp$-local systems. Note that this result can also be formulated in
terms of $\varphi$-bundles using Proposition~\ref{P:truncate cohomology1}.
\begin{theorem} \label{T:Galois cohomology2a}
Suppose that $R$ is an $\FF_{p^d}$-algebra for a positive integer $d$.
Let $E$ be an isogeny $\ZZ_{p^d}$-local system on $\Spec(R)$.
Let $M$ be the globally \'etale $\varphi^d$-module over $\tilde{\calR}^{\bd}_R$, $\tilde{\calE}_R$, or $\tilde{\calR}_R$
corresponding to $E$ via Theorem~\ref{T:perfect equivalence2a}.
Then there are natural (in $E$ and $R$) bijections
$H^i_{\et}(X, E) \cong H^i_{\varphi^d}(M)$ for all $i \geq 0$.
\end{theorem}
\begin{proof}
We first treat the case over $\tilde{\calR}^{\bd}_R$, the case over $\tilde{\calE}_R$ being similar.
Let $T$ be an \'etale $\ZZ_{p^d}$-local system on $X$ for which $E = T \otimes_{\Zp} \Qp$.
Let $M_0$ be the $\varphi^d$-module over $\tilde{\calR}^{\inte}_R$ corresponding to $T$.
By Theorem~\ref{T:Galois cohomology2}, we have natural (in $T$ and $A$) bijections
$H^i_{\et}(X, T) \cong H^i_{\varphi^d}(M_0)$ for all $i \geq 0$. By definition, we may identify $H^i_{\et}(X,E)$ with
$H^i_{\et}(X, T) \otimes_{\Zp} \Qp$; in particular, it is zero for $i>1$. On the other hand, for $i=0,1$, we may identify
$H^i_{\varphi^d}(M)$ with $H^i_{\varphi^d}(M_0) \otimes_{\Zp} \Qp$ by identifying
$M$ with $M_0 \otimes_{\Zp} \Qp = \cup_{n=0}^\infty p^{-n} M_0$ and noting that the computation
of $H^i_{\varphi^d}$ commutes with direct limits. This proves the claim in this case.

We next treat the case over $\tilde{\calR}_R$. Put
$M_1 = M_0 \otimes_{\Zp} \Qp$, so that we may identify $M$ with $M_1 \otimes_{\tilde{\calR}^{\bd}_R}
\tilde{\calR}_R$. It follows from
Theorem~\ref{T:perfect equivalence2a} that the natural map
$H^0_{\varphi^d}(M_1) \to H^0_{\varphi^d}(M)$ is bijective; hence $H^0_{\varphi^d}(M)$ is naturally isomorphic to $H^0_{\et}(X, E)$. Recall that by Remark~\ref{R:extend Qp-local systems}, the extension of two isogeny $\ZZ_{p^d}$-local systems on
$\Spec(R)$ in the category of \'etale $\QQ_{p^d}$-local systems on $\Spa(R,R^+)$
descends to an extension of isogeny $\ZZ_{p^d}$-local systems on $\Spec(R)$. That is, we may compute $H^1_{\et}(X,E)$ as an extension group in the category of \'etale local systems over $\Spa(R,R^+)$. We then obtain an isomorphism
between this group and $H^1_{\varphi^d}(M)$ by applying Theorem 6.2.9 and noting that any extension of \'etale $\varphi^d$-modules over $\tilde{\calR}_R$ is again \'etale by Lemma~\ref{L:two of three}.
\end{proof}

\begin{remark}
One can formulate an analogue of Theorem~\ref{T:Galois cohomology2a} for isogeny $(c,d)$-local systems
comparing a suitably modified \'etale cohomology to $H^1_{p^c \varphi^d}$ of the corresponding $\varphi$-module.
As this statement is a formal consequence of the one given, and we have no particular use for it, we omit
further details.
\end{remark}

\subsection{The relative Fargues-Fontaine curve}
\label{subsec:relative FF}

We have already seen (Theorem~\ref{T:vector bundles}) that
for $R$ a perfect Banach algebra over $\FF_p$, the $\varphi$-modules over $\tilde{\calR}_R$ can be described in terms of a vector bundle on a certain scheme $\Proj(P)$,
which in the case of an analytic field is a Fargues-Fontaine curve. In order to globalize this construction, we must replace the scheme with what amounts to an \emph{analytification} thereof, although the latter construction does not come equipped with a universal property like the one for analytification of schemes of finite type over a field \cite[Expos\'e~XII]{sga1}.

\begin{hypothesis} \label{H:relative FF}
Throughout \S\ref{subsec:relative FF}, fix a positive integer $a$
and put $q = p^a$.
Let $(A,A^+)$ be a perfectoid adic Banach algebra over $\Qp$ and let $(R,R^+)$ be the perfect uniform adic Banach algebra over $\Fp$ associated to $(A,A^+)$ via the perfectoid correspondence (Theorem~\ref{T:perfectoid ring}). Let $X$ be a perfectoid adic space over $\Qp$ and let $X'$ be the corresponding perfect uniform adic space over $\Fp$
associated to $X$ via the global perfectoid correspondence (Theorem~\ref{T:perfectoid correspondence}).
\end{hypothesis}

\begin{remark} \label{R:over analytic field}
Since $(R,R^+)$ arises from $(A,A^+)$ via the perfectoid correspondence, there exists $z \in W(R^+)$ which is primitive of degree 1 and generates the kernel of $\theta: W(R^+) \to A^+$.
\end{remark}

\begin{defn}
For $0 < s \leq r$, the ring $\tilde{\calR}^{[s,r]}_R$ is a relatively perfectoid Banach ring by Theorem~\ref{T:Kiehl for Robba}.
We promote it to an adic Banach ring as in Definition~\ref{D:valuation projection}.
\end{defn}

\begin{defn}
Define the space
\[
U_R = \bigcup_{0 < s < r} \Spa(\tilde{\calR}^{[s,r]}_R, \tilde{\calR}^{[s,r],+}_R)
\]
whose maximal Hausdorff quotient is the space $T_{R}$
considered in Proposition~\ref{P:radius fibration}.
By Theorem~\ref{T:Kiehl for Robba}, $U_R$ is a relatively perfectoid space. By Proposition~\ref{P:radius fibration},
$\varphi^{a*}$ acts properly discontinuously on $U_R$,
so we may form the orbit space $\FFC_R$ which again is a relatively perfectoid space. 

Using Lemma~\ref{L:lift rational to Robba}, 
we may glue to obtain a relatively perfectoid space $\FFC_{X'}$ with the property that for $X' = \Spa(R,R^+)$, we have a natural isomorphism
$\FFC_{X'} \cong \FFC_R$. We also have a natural map $\left| \FFC_{X'} \right| \to \left| X' \right|$ of topological spaces. However, this map cannot arise from a morphism of adic spaces due to the mismatch in characteristics.

We will use the notation $\FFC_A$ to denote the space $\FFC_{R}$ additionally equipped with the morphism $\Spa(A,A^+) \to \FFC_R$ induced by the map $\theta: \tilde{\calR}^{[1,1]}_R \to A$ (see Lemma~\ref{L:stable residue2}).
We again glue to obtain a space $\FFC_X$ which is naturally isomorphic to $\FFC_{X'}$ but additionally is equipped with a distinguished morphism $X \to \FFC_X$.
The induced map $\left| X \right| \to \left| \FFC_X \right|$ is a continuous section of the projection map $\left| \FFC_X \right| \cong \left| \FFC_{X'} \right| \to \left| X' \right| \cong \left| X \right|$. However, despite there no longer being a mismatch of characteristics, the map $\left| \FFC_X \right| \to \left| X \right|$ is still not induced by a morphism $\FFC_X \to X$ of adic spaces.
\end{defn}

\begin{defn}
Define the graded ring $P_R$ as in Definition~\ref{D:vector bundles}.
The natural morphism $P_R \to \tilde{\calR}^{[s,r]}_R$ then defines a morphism $U_R \to \Proj(P_R)$ of locally ringed spaces, which factors through a morphism $\FFC_R \to \Proj(P_R)$.
For $n \in \ZZ$, we define the line bundle $\calO(n)$ on $\FFC_R$
by pulling back the line bundle $\calO(n)$ on $\Proj(P_R)$ defined in
Remark~\ref{R:not coherent}.
\end{defn}

\begin{remark} \label{R:two unit generators}
Let $L$ be a perfect analytic field of characteristic $p$ over which $R$ is a Banach algebra. (For example,
with notation as in Remark~\ref{R:over analytic field},
we may take $L$ to be the completed perfect closure of $\FF_p((\overline{z}))$.)
By Proposition~\ref{P:global invariants} there exist nonzero homogeneous elements
$f_1, f_2 \in P_{L,+}$ which
generate the unit ideal in $\tilde{\calR}_L$, and hence also in $\tilde{\calR}_R$.
(In fact, one can even take $f_1, f_2 \in P_{L,1}$.)
This has the following further consequences.
\begin{enumerate}
\item[(a)]
For any positive integers $m,n$ such that $f_1^m$ and $f_2^n$ have the same degree $d$,
if we write $P_{R,(d)} = \oplus_{h=0}^\infty P_{R,hd}$, then
the scheme $\Proj(P_R/(f_1^m))$ is isomorphic to
\[
\Spec (P_{R,(d)}[f_2^{-n}]_0/(f_1^m f_2^{-n})),
\]
and hence is affine.
\item[(b)]
By Lemma~\ref{L:ideal-trivial},
$\Proj(P_R)$ is covered by the two open affine subsets $D_+(f_1), D_+(f_2)$.
Since $\Proj(P_R)$ is separated, we may use \v{C}ech cohomology for this covering to compute
sheaf cohomology for quasicoherent sheaves
\cite[Proposition~1.4.1]{ega3-1}, so the cohomology of any quasicoherent sheaf on $\Proj(P_R)$
vanishes in degree greater than 1.
\end{enumerate}
See Remark~\ref{R:flat covering} for a related observation.
\end{remark}

The morphism $\FFC_R \to \Proj(P_R)$ leads to a GAGA-style extension of Theorem~\ref{T:vector bundles}.
\begin{theorem} \label{T:vector bundles2}
Pullback along the morphism $\FFC_R \to \Proj(P_R)$ of locally ringed spaces
defines an equivalence of categories between vector bundles on $\Proj(P_R)$ and on $\FFC_R$. Consequently, by Theorem~\ref{T:vector bundles},
the latter is equivalent to the category of $\varphi^a$-modules over $\tilde{\calR}_{R}$ (which is independent of $R^+$).
\end{theorem}
\begin{proof}
By Theorem~\ref{T:adic vector bundle},
pulling back along the functor $U_R \to \FFC_R$ defines an equivalence of categories between vector bundles on $\FFC_R$ and $\varphi$-bundles on $\tilde{\calR}_{R}$. Since the latter category is equivalent to the category of vector bundles on $\Proj(P_R)$ by Theorem~\ref{T:vector bundles}, we deduce the desired result.
\end{proof}

This result immediately globalizes as follows.
\begin{theorem} \label{T:vector bundles2a}
The category of vector bundles on $\FFC_X$ is functorially equivalent
to the category of $\varphi^a$-modules over $\tilde{\calR}_X$.
\end{theorem}
\begin{cor}
For $n \in \ZZ$, we have
\[
H^0(\Proj(P_R), \calO(n)) = H^0(\FFC_R, \calO(n)) = \begin{cases}
P_{R,n} & n \geq 0 \\ 0 & n < 0. \end{cases}
\]
\end{cor}

\begin{cor} \label{C:vector bundles2b}
Fix a morphism $X \to \Spa(\QQ_{p^a}, \ZZ_{p^a})$. Then the category of \'etale $\QQ_{p^a}$-local systems on $X$ is functorially equivalent to the category of
vector bundles on $\FFC_X$ which are \emph{pointwise semistable} of degree $0$,
i.e., whose restriction to $\FFC_{\calH(x)}$ is semistable of degree $0$ for any $x \in X$.
\end{cor}
\begin{proof}
This is immediate from Remark~\ref{R:Harder-Narasimhan},
Theorem~\ref{T:perfect equivalence2b},
and Theorem~\ref{T:vector bundles2a}.
\end{proof}

We will also need some higher-rank vector bundles on $\FFC_R$.
\begin{defn}
Write $P_{R,n}^{(a)}, P_R^{(a)}, U_R^{(a)}, \FFC^{(a)}_R, \FFC^{(a)}_X$ instead of $P_{R,n}, P_R, U_R, \FFC_R, \FFC_X$ to record the dependence on the positive integer $a$.
Then for any positive integer $d$, we have $P_{R,n}^{(d)} \subseteq P_{R,nd}^{(ad)}$ as subsets of $\tilde{\calR}_R$. We thus get a morphism $P_{R}^{(d)} \to P_R^{(ad)}$ of graded rings and hence a morphism $\Proj(P_R^{(ad)}) \to \Proj(P_R^{(a)})$ of schemes.
This morphism is finite \'etale of degree $d$: it is the quotient by the automorphism induced by $\varphi^{a}$. Similarly, the corresponding morphism $U_{R}^{(ad)} \to U_R^{(a)}$ is an isomorphism,
so the induced morphism $\FFC^{(ad)}_R \to \FFC^{(a)}_R$ is finite \'etale of degree $d$ and is the quotient by the action of $\varphi^{a*}$.

For $n,d \in \ZZ$ with $\gcd(n,d)= 1$ and $d>0$, let $\calO(n,d)$ be the vector bundle of rank $d$ on $\FFC^{(a)}_R$ obtained by pushing forward the bundle $\calO(n)$ on $\FFC^{(ad)}_R$.
\end{defn}

\begin{cor}
For $R = L$ an algebraically closed analytic field, every vector bundle on either $\Proj(P_R)$ or $\FFC_R$ is (nonuniquely) isomorphic to a direct sum $\oplus_i \calO(n_i,s_i)$ for some $n_i, s_i \in \ZZ$
with $\gcd(n_i,s_i) = 1$ and $s_i >0$.
\end{cor}
\begin{proof}
Immediate from Theorem~\ref{T:vector bundles2} and
Proposition~\ref{P:DM not pure}.
\end{proof}

Continuing in the GAGA vein, we have the following comparison result. We will later add \'etale cohomology once we make sense of it in \S\ref{sec:perfect}.

\begin{theorem} \label{T:compare cohomology}
Let $M$ be a $\varphi^a$-module over $\tilde{\calR}_R$ corresponding to the vector bundle $V$ on $\FFC_R$ via Theorem~\ref{T:vector bundles2a}. Then there are natural (in $M$ and $R$) isomorphisms $H^i_{\varphi^a}(M) \cong H^i(\Proj(P_R), V) \cong H^i(\FFC_R, V)$ for $i \geq 0$.
\end{theorem}
\begin{proof}
For $Z$ a scheme which is quasicompact and semiseparated
(i.e., $Z$ is covered by finitely many open affine subschemes, any two of which have
affine intersection), the category of quasicoherent sheaves on $Z$ has enough injectives,
and the resulting derived functors agree with sheaf cohomology as defined using the full category
of sheaves of abelian groups \cite[Proposition~B.8]{thomason-trobaugh}.
These results apply to $\Proj(P_R)$ because this scheme is quasicompact
(by Lemma~\ref{L:ideal-trivial}) and separated.

By the previous paragraph plus \cite[Theorem~IV.9.1]{hilton-stammbach},
we may identify the sheaf cohomology groups $H^i(\Proj(P_R), V)$ with the Yoneda
extension groups $\Ext^i(\calO, V)$ in the category of quasicoherent sheaves.
For $i=0,1$, this computation involves only quasicoherent finite locally free sheaves,
so we may apply Theorem~\ref{T:vector bundles} to obtain the desired isomorphisms.
For $i \geq 2$, $H^i(\Proj(P_R), V) = 0$ by Remark~\ref{R:two unit generators}(b)
while $H^i_{\varphi^a}(M) = 0$ by definition, so we again obtain an isomorphism.

Similarly, thanks to Theorem~\ref{T:Kiehl for Robba}
we may compute $H^i(\FFC_R, V)$ using \v{C}ech cohomology for the covering of
$\FFC_R$ by $\Spa(\tilde{\calR}^{[rq^{-1/2},r]}_R,\tilde{\calR}^{[rq^{-1/2},r],+}_R)$
and $\Spa(\tilde{\calR}^{[rq^{-1},rq^{-1/2}]}_R, \tilde{\calR}^{[rq^{-1},rq^{-1/2}],+}_R)$ for any fixed $r>0$. This immediately yields $H^i(\FFC_R,V) = 0$ for $i \geq 2$,
and again the comparison for $i=0,1$ follows by comparing Yoneda extension groups using Theorem~\ref{T:vector bundles} and Theorem~\ref{T:vector bundles2a}.
\end{proof}

\begin{cor} \label{C:vanishing of H1}
Let $V$ be a quasicoherent finite locally free sheaf on $\Proj(P_R)$. Then there exists $N \in \ZZ$
such that for all $n \geq N$,
$H^0(\Proj(P_R), V(n))$ generates $V(n)$ and
$H^1(\Proj(P_R), V(n)) = H^1(\FFC_R, V(n)) = 0$.
\end{cor}
\begin{proof}
This follows from Theorem~\ref{T:compare cohomology}
plus Propositions~\ref{P:global invariants}
and~\ref{P:H1}.
\end{proof}

We next consider the compatibility of the functor $\FFC$ with \'etale morphisms.
\begin{lemma} \label{L:FFC finite}
For $Y \to X$ a morphism of perfectoid adic spaces which is \'etale (resp.\ finite \'etale, faithfully finite \'etale), the induced morphism $\FFC_Y \to \FFC_X$ is also \'etale (resp.\ finite \'etale, faithfully finite \'etale).
\end{lemma}
\begin{proof}
By the perfectoid correspondence, the morphism $Y' \to X'$ in characteristic $p$ is also \'etale (resp.\ finite \'etale, faithfully finite \'etale). By Proposition~\ref{P:perfect mixed lift},
if $Y' \to X'$ is finite \'etale then so is $\FFC_Y \to \FFC_X$;
the other cases follow immediately.
\end{proof}

\begin{remark} \label{R:relative FF curve quotient2}
Lemma~\ref{L:FFC finite} provides a promotion of the morphism $\FFC_X \to X$ of topological spaces to a morphism of \'etale topoi. 
This morphism behaves in many ways like a circle bundle; 
see Remark~\ref{R:relative FF curve quotient}.
\end{remark}

\subsection{Ampleness on relative curves}
\label{subsec:ampleness}

We now make a more detailed study of positivity of vector bundles on relative Fargues-Fontaine curves; the analogy with vector bundles on projective curves turns out to be rather fruitful.
Throughout \S\ref{subsec:ampleness}, continue to retain
Hypothesis~\ref{H:relative FF}.

\begin{defn}
Throughout \S\ref{subsec:ampleness},
take $L, f_1, f_2$ as in Remark~\ref{R:two unit generators};
we can and will assume that $f_1, f_2 \in P_{L,1}$.
Let $Z_1$ be the zero locus of $f_1$ and put $U_1 = \Proj(P_R) \setminus Z_1$; by
Remark~\ref{R:two unit generators}, both $Z_1$ and $U_1$ are affine schemes.
\end{defn}

\begin{defn} \label{D:ample on schemes}
A vector bundle $\calF$ on $\Proj(P_R)$ is \emph{globally ample} if for every quasicoherent sheaf of finite type $\calG$ on $\Proj(P_R)$, there exists $n_0 \in \ZZ$ such that $\calF^{\otimes n} \otimes \calG$ is generated by global sections for all $n \geq n_0$. By Corollary~\ref{C:vector bundle quotient} below, it will suffice to check the condition for $\calG = \calO(e)$ for all $e \in \ZZ$.

We say that $\calF$ is \emph{ample} if there exists a strong rational covering $\{(R,R^+) \to (R_i, R_i^+)\}$ such that the pullback of $\calF$ to $\Proj(P_{R_i})$ is globally ample for each $i$.
\end{defn}

\begin{lemma} \label{L:ampleness of a multiple}
Let $\calF$ be a vector bundle on $\Proj(P_R)$ and let $m$ be a positive integer. Then $\calF$ is (globally) ample if and only if $\calF^{\otimes m}$ is (globally) ample.
\end{lemma}
\begin{proof}
It is evident that if $\calF$ is globally ample, then so is $\calF^{\otimes m}$. On the other hand, if $\calF^{\otimes m}$ is globally ample, then for any given $\calG$, there exists $n_0 \in \ZZ$ such that for $n \geq n_0$, each of
the bundles $(\calF^{\otimes m})^{\otimes n} \otimes (\calF^{\otimes i} \otimes \calG)$ for $i=0,\dots,m-1$ is generated by global sections. Consequently, $\calF^{\otimes n} \otimes \calG$ is generated by global sections for $n \geq m n_0$, so $\calF$ is globally ample. The ample case is similar.
\end{proof}

\begin{lemma} \label{L:ample}
The vector bundles $\calO(e)$ on $\Proj(P_R)$ are globally ample for all $e>0$.
\end{lemma}
\begin{proof}
Let $\calG$ be a quasicoherent sheaf of finite type on $\Proj(P_R)$.
Since $U_1 = \Spec(P_R[f_1^{-1}]_0)$ is affine,
$H^0(U_1, \calG)$ is a finitely generated $P_R[f_1^{-1}]_0$-module 
(see Remark~\ref{R:finite type}), and so $\calG(md)$ is generated by global sections for any sufficiently large $m$. By Lemma~\ref{L:ampleness of a multiple}, this proves the claim.
\end{proof}
\begin{cor} \label{C:vector bundle quotient}
Any quasicoherent sheaf of finite type on $\Proj(P_R)$ is a quotient of a vector bundle of the form $\oplus_{i=1}^{m} \calO(e_i)$
for some $m \geq 0$ and $e_1,\dots,e_m \in \ZZ$.
\end{cor}

We have the following variant of the cohomological criterion for ampleness for projective schemes \cite[Proposition~III.5.3]{hartshorne}.
\begin{prop} \label{P:H1 from ample from H1}
For $\calF$ a vector bundle on $\Proj(P_R)$, the following conditions are equivalent.
\begin{enumerate}
\item[(a)]
The vector bundle $\calF$ is globally ample.
\item[(b)]
For every quasicoherent sheaf of finite type $\calG$ on $\Proj(P_R)$, there exists $n_0 \in \ZZ$ such that for all $n \geq n_0$, $H^1(\Proj(P_R), \calF^{\otimes n} \otimes \calG) = 0$.
\item[(c)]
For each $e \in \ZZ$, there exists $n_0 \in \ZZ$ such that for all $n \geq n_0$, $H^1(\Proj(P_R), \calF^{\otimes n}(e)) = 0$.
\end{enumerate}
\end{prop}
\begin{proof}
We first observe that (b) and (c) are equivalent. Indeed, (b) trivially implies (c), whereas (c) implies (b) by Corollary~\ref{C:vector bundle quotient} and Remark~\ref{R:two unit generators}(b).

We next check that (a) implies (c). 
By Lemma~\ref{L:ample}, there exists $e'>0$ such that $H^1(\Proj(P_R), \calO(e')) = 0$. (In fact we may take $e'=1$ by Proposition~\ref{P:H1}, but this is not crucial here.)
Given (a), we may choose $n_0$ so that for $n \geq n_0$, $\calF^{\otimes n}(e-e')$ is generated by global sections; then for each such $n$, there exists a surjective homomorphism $\calO(e')^{\oplus m} \to \calF^{\otimes n}(e)$ for some $m$ (depending on $e$). 
By Remark~\ref{R:two unit generators}(b), the long exact sequence in cohomology yields $H^1(\Proj(P_R), \calF^{\otimes n}(e)) = 0$. Hence (a) implies (b).

We finally check that (b) implies (a). Fix $e \in \ZZ$.
For any $n \in \ZZ$, we have an exact sequence
\[
H^0(\Proj(P_R), \calF^{\otimes n}(e)) \to H^0(Z_1, \calF^{\otimes n}(e)) \to H^1(\Proj(P_R), \calF^{\otimes n}(e-1)).
\]
By Remark~\ref{R:finite type},
$H^0(Z_1, \calF^{\otimes n}(e))$ is a finitely generated module over the coordinate ring of $Z_1$. 
Choose $n$ sufficiently large
so that $H^1(\Proj(P_R), \calF^{\otimes n}(e-1)) = 0$;
there then exist finitely many sections $s_1,\dots,s_m \in H^0(\Proj(P_R), \calF^{\otimes n}(e))$ which generate the restriction of $\calF^{\otimes n}(e)$ to $Z_1$.

Let $\calG$ be the subsheaf of $\calF^{\otimes n}(e)$ generated by 
$s_1,\dots,s_m$. 
Let $Z'$ be the support of $\calF^{\otimes n}(e)/\calG$; it is disjoint from $Z_1$ and hence is a closed subscheme of $U_1$. Since $U_1$ is affine, so then is $Z'$. Put $U' = \Proj(P_R) \setminus Z'$, which is open in $\Proj(P_R)$ but not necessarily affine.

Since $\calF^{\otimes n}(e)/\calG$ is finitely presented, $Z'$ can also be realized as the support of a finitely generated ideal sheaf $\calI$ (e.g., using Fitting ideals).
For $n'$ sufficiently large, we have 
\[
H^1(\Proj(P_R), \calF^{\otimes (n+n')}(e) \otimes \calI) = 
H^1(\Proj(P_R), \calF^{\otimes n'} \otimes (\calF^{\otimes n}(e) \otimes \calI)) = 0
\]
and hence another exact sequence
\[
H^0(\Proj(P_R), \calF^{\otimes (n+n')}(e)) \to H^0(Z', \calF^{\otimes n+n'}(e) \otimes (\calO/\calI)) \to 0.
\]
Since $Z'$ is affine, by Remark~\ref{R:finite type}
$\calF^{\otimes (n+n')}(e) \otimes (\calO/\calI)$ corresponds to a finitely generated module over the coordinate ring of $Z'$.
Consequently, for $n'$ sufficiently large, 
there exist finitely many global sections of $\calF^{\otimes (n+n')}(e)$ which generate the restriction of $\calF^{\otimes (n+n')}(e)$ to $Z'$. 

In case $e=0$, we also know that for $n'$ divisible by $n$, the restriction of $\calF^{\otimes (n+n')}$ to $U'$ is also generated by finitely many global sections (namely the $(n'/n+1)$-fold products of $s_1,\dots,s_m$). Consequently, for all $n'$ sufficiently large and divisible by $n$, $\calF^{\otimes (n+n')}$ is generated by finitely many global sections.

For general $e$, for $n'$ sufficiently large and divisible by $n$, 
the restrictions of $\calF^{\otimes n}(e)$ and
$\calF^{\otimes n'}$ to $U'$ are both generated by finitely many global sections
(the latter thanks to the previous paragraph).
For such $n'$, $\calF^{\otimes (n+n')} (e)$ is generated by finitely many global sections, so $\calF^{\otimes n}$ is globally ample; 
by Lemma~\ref{L:ampleness of a multiple}, $\calF$ is also globally ample.
\end{proof}

\begin{cor} \label{C:globally etale to ample}
Let $\calF$ be a globally \'etale vector bundle on $\Proj(P_R)$. Then for any positive integer $n$, $H^1(\Proj(P_R), \calF(n)) = 0$ and $\calF(n)$ is globally ample.
\end{cor}
\begin{proof}
The equality $H^1(\Proj(P_R), \calF(n)) = 0$ is immediate from Proposition~\ref{P:H1}.
Given this equality (for all $\calF$ and $n$), we may check criterion (c) of Proposition~\ref{P:H1 from ample from H1} to deduce that $\calF(n)$ is globally ample.
\end{proof}

\begin{lemma} \label{L:affine subscheme from section}
Let $\calF$ be a globally ample line bundle on $\Proj(P_R)$.
Then for any $s \in H^0(\Proj(P_R), \calF)$, the open subscheme $U_s$ of $\Proj(P_R)$ on which $s$ generates $\calF$ is affine.
\end{lemma}
\begin{proof}
Since $\Proj(P_R)$ is quasicompact and quasiseparated
by Remark~\ref{R:two unit generators}, so then is $U_s$.
By the cohomological criterion for affinity 
\cite[Tag~01XG]{stacks-project}, it suffices to check that $H^1(U_s, \calG) = 0$ where $\calG$ is an arbitrary quasicoherent sheaf of finite type on $U_s$. Let $j: U_s \to \Proj(P_R)$ be the canonical inclusion, so that $H^1(U_s, \calG) = H^1(\Proj(P_R), j_* \calG)$.
We may then write $j_* \calG$ as the direct limit of $\calH \otimes 
\calF^{\otimes d}$ as $d \to \infty$ for some quasicoherent sheaf of finite type $\calH$ on $\Proj(P_R)$. Since cohomology commutes with direct limits, we deduce the claim from Proposition~\ref{P:H1 from ample from H1}.
\end{proof}
\begin{cor}
Let $\calF$ be a globally ample line bundle on $\Proj(P_R)$.
There is then a natural isomorphism
\[
\Proj(P_R) \cong \Proj(P_\calF), \qquad P_{\calF} = \bigoplus_{n=0}^\infty H^0(\Proj(P_R), \calF^{\otimes n}).
\]
\end{cor}

We next establish a pointwise interpretation of ampleness.
\begin{defn}
Let $\calF$ be a vector bundle on $\Proj(P_R)$.
Let $M$ be the $\varphi^a$-module over $\tilde{\calR}_R$ corresponding to $\calF$ via Theorem~\ref{T:vector bundles}. We define the \emph{slope polygon} of $\calF$ as a function on $\calM(R)$ (and on $\Spa(R,R^+)$ by retraction) as the fiberwise Harder-Narasimhan polygon; this agrees with the slope polygon of $M$ thanks to Remark~\ref{R:Harder-Narasimhan}.
We say that $\calF$ is \emph{pointwise ample} at $\beta \in \calM(R)$ if the slopes of $\calF$ at $\beta$ are everywhere positive; by Theorem~\ref{T:open locus}, this
is an open condition on $\calM(R)$.
If this condition holds for all $\beta$, we say that $\calF$ is \emph{pointwise ample}.
\end{defn}

\begin{lemma} \label{L:twist is pointwise ample}
For any pointwise ample vector bundle $\calF$ on $\Proj(P_R)$
and any vector bundle $\calG$ on $\Proj(P_R)$, there exists $n_0 \in \ZZ$ such that $\calF^{\otimes n} \otimes \calG$ is pointwise ample for all $n \geq n_0$.
\end{lemma}
\begin{proof}
We first observe that this is true if $R = L$ is an analytic field. Namely, if the slopes of $\calF$ are 
$\mu_1,\dots,\mu_m$ and the slopes of $\calG$ are $\nu_1,\dots,\nu_l$,
then for $n > 0$ each slope of $\calF^{\otimes n} \otimes \calG$ is equal to some $\nu_i$ plus an $n$-fold sum of the $\mu_j$. If $\calF$ is pointwise ample, then $\mu_1,\dots,\mu_m$ are positive, so for $n$ large enough the slopes of $\calF^{\otimes n} \otimes \calG$ are all positive.

To extend this argument to the general case, it suffices to note that on one hand, by Proposition~\ref{P:slopes bounded} the slopes of $\calF$ and $\calG$ are bounded below; on the other hand, the slopes of $\calF$ are limited to a discrete subset of $\QQ$, and hence bounded away from 0.
\end{proof}

\begin{lemma} \label{L:pointwise}
Suppose that $R = L$ is an analytic field.
Let $\calF$ be an ample vector bundle on $\Proj(P_R)$.
\begin{enumerate}
\item[(a)]
We have $H^1(\Proj(P_R), \calF) = 0$.
\item[(b)]
The bundle $\calF$ is generated by $H^0(\Proj(P_R), \calF)$.
\end{enumerate}
\end{lemma}
\begin{proof}
Let $M$ be the $\varphi^a$-module over $\tilde{\calR}_R$ associated to $\calF$ via Theorem~\ref{T:vector bundles}.
To prove (a), by Theorem~\ref{T:slope filtration explicit2}
we may reduce to the case where $M$ is pure of some slope $s>0$.
By Theorem~\ref{T:compare cohomology}, we must check that $H^1_{\varphi^a}(M) = 0$. There is no harm in enlarging $a$, so we
may assume that $as \in \ZZ$; this case follows from
Corollary~\ref{C:globally etale to ample}.

To prove (b), by (a) and Theorem~\ref{T:slope filtration explicit2} we may again reduce to the case where $M$ is pure of some slope $s>0$.
By Theorem~\ref{T:compare cohomology}, we must check that $H^0_{\varphi^a}(M)$ generates $M$. There is no harm in enlarging $a$, so we may assume that $as \in \ZZ$; this case follows from
Proposition~\ref{P:global invariants}.
\end{proof}

To relate ampleness to pointwise ampleness, we use a d\'evissage argument in the style of \cite{liu-herr}.
\begin{lemma} \label{L:single step devissage}
Let $\calF$ be a vector bundle on $\Proj(P_R)$ whose slopes at some $\beta \in \calM(R)$ are all nonnegative but not all zero. 
Then there exists a short exact sequence
\[
0 \to \calO(-1) \to \calG \to \calF \to 0
\]
of vector bundles on $\Proj(P_R)$ such that the slopes of $\calG$ at $\beta$ are also all nonnegative.
\end{lemma}
\begin{proof}
Let $i: Z_1 \to \Proj(P_R)$ denote the canonical inclusion; we then have an exact sequence
\[
0 \to \calO(-1) \to \calO \to i_* i^* \calO \to 0
\]
of sheaves on $\Proj(P_R)$ 
in which the map $\calO(-1) \to \calO$ is multiplication by $f_1$. Tensoring by $\calF^{\dual}$ yields another exact sequence
\[
0 \to \calF^\dual(-1) \to \calF^\dual \to i_* i^* \calF^\dual \to 0.
\]
Write $Z_{1,\beta}$ for the zero locus of $f_1$ on $\Proj(P_{\calH(\beta)})$;
we write $i$ also for the canonical inclusion $Z_{1,\beta} \to \Proj(P_{\calH(\beta)})$.

By Theorem~\ref{T:slope filtration explicit2}, there exists a short exact sequence
\[
0 \to \calF_+ \to \calF_\beta \to \calF_0 \to 0
\]
of vector bundles on $\Proj(P_{\calH(\beta)})$
such that $\calF_{+} \neq 0$ has all positive slopes and $\calF_0$ has all zero slopes.
In the exact sequence
\[
H^0(\Proj(P_{\calH(\beta)}), \calF_+^\dual) \to 
H^0(\Proj(P_{\calH(\beta)}), i_* i^* \calF_+^\dual) \to
H^1(\Proj(P_{\calH(\beta)}), \calF_+^\dual(-1)),
\]
the first term vanishes because $\calF_+^\dual$ has all slopes negative.
We thus have a commutative diagram
\[
\xymatrix{
H^0(Z_1, i^* \calF^\dual) \ar[r] \ar@{=}[d]& 
H^0(Z_{1,\beta}, i^* \calF^\dual) \ar[r] \ar@{=}[d]& H^0(Z_{1,\beta}, i^* \calF_+^\dual) \ar@{=}[d]\\
H^0(\Proj(P_R), i_* i^* \calF^\dual) \ar[r] \ar[d] & H^0(\Proj(P_{\calH(\beta)}), i_* i^* \calF_\beta^\dual) \ar[r] \ar[d]
& H^0(\Proj(P_{\calH(\beta)}), i_* i^* \calF_+^\dual) \ar@{^{(}->}[d] \\
H^1(\Proj(P_R), \calF^\dual(-1)) \ar[r] & H^1(\Proj(P_{\calH(\beta)}), \calF_\beta^\dual(-1)) \ar[r] & H^1(\Proj(P_{\calH(\beta)}), \calF_+^\dual(-1))
}
\]
Since $Z_1$ and $Z_{1,\beta}$ are affine,
we have
\[
H^0(Z_{1,\beta}, i^* \calF^\dual) = H^0(Z_1, i^* \calF^\dual) \otimes_{\calO(Z_1)} \calO(Z_{1,\beta}).
\]
Since $\calF_+ \neq 0$, we can find $x \in H^0(Z_1, i^* \calF^\dual)$ whose image in $H^0(Z_{i,\beta}, i^* \calF_+^\dual)$ is nonzero.
Map $x$ to $y \in H^1(\Proj(P_R), \calF^\dual(-1))$; then the image of $y$ in 
$H^0(\Proj(P_{\calH(\beta)}), \calF_+^\dual(-1))$ is also nonzero.

We claim that the short exact sequence defined by $y$ has the desired property.
To check the claim, we may reduce to the case where $R = \calH(\beta)$, so that $\calF_\beta = \calF$.
Let $\calG_+$ be the inverse image of $\calF_+$ in $\calG$; 
we then have an exact sequence
\begin{equation} \label{eq:single step devissage1}
0 \to \calO(-1) \to \calG_+ \to \calF_+ \to 0
\end{equation}
which by construction is not split. For any vector bundle quotient $\calH$ of $\calG_+$, we get an exact sequence
\[
0 \to \calH_1 \to \calH \to \calH_2 \to 0
\]
in which $\calH_1$ is the image of $\calO(-1)$ in $\calH$. Note that on one hand,
$\deg(\calH_1) \geq -1$ with equality only if $\calH_1 = \calO(-1)$;
on the other hand, $\deg(\calH_2) \geq 0$ with equality only if $\calH_2 = 0$.
We cannot have both equalities simultaneously because \eqref{eq:single step devissage1} does not split; hence $\deg(\calH) = \deg(\calH_1) + \deg(\calH_2) \geq 0$, and by Theorem~\ref{T:slope filtration explicit2} it follows that $\calG_+$ has all slopes nonnegative. From the exact sequence
\[
0 \to \calG_+ \to \calG \to \calF_0 \to 0,
\]
we see that $\calG$ also has all slopes nonnegative.
\end{proof}

\begin{cor} \label{C:single step devissage}
Let $\calF$ be a vector bundle on $\Proj(P_R)$ whose slopes at some $\beta \in \calM(R)$ are all nonnegative. Then there exists a short exact sequence
\[
0 \to \calH \to \calG \to \calF \to 0
\]
of vector bundles on $\Proj(P_R)$ such that $\calG$ is \'etale at $\beta$.
\end{cor}
\begin{proof}
This follows from Lemma~\ref{L:single step devissage} by induction on $\deg(\calF_\beta)$, plus Theorem~\ref{T:spread etale}.
\end{proof}

\begin{theorem} \label{T:pointwise ample is ample}
A vector bundle $\calF$ on $\Proj(P_R)$ is ample if and only if it is pointwise ample.
\end{theorem}
\begin{proof}
Suppose that $\calF$ is ample; to prove that $\calF$ is pointwise ample, we may even assume that $\calF$ is globally ample. Choose any $\beta \in \calM(R)$
and let $\alpha_1 \geq \dots \geq \alpha_m$ be the slopes of $\calF$
at $\beta$ listed with multiplicity. For $n > 0$, the slopes of $\calF^{\otimes n}$ at $\beta$ are the $n$-fold sums of $\alpha_1,\dots,\alpha_m$. For some $n$, $\calF^{\otimes n}(-1)$ is generated by global sections, and so $n \alpha_m - 1 \geq 0$ and $\alpha_m > 0$. 
Hence $\calF$ is pointwise ample.

Conversely, suppose that $\calF$ is pointwise ample.
To prove that $\calF$ is ample, it is harmless to first replace $R$ by a rational localization encircling $\beta$.
By Lemma~\ref{L:twist is pointwise ample}, there exists $n_0 \in \ZZ$ such that for $n \geq n_0$, $\calF^{\otimes n}(-1)$ is pointwise ample at $\beta$. 
By Corollary~\ref{C:single step devissage}, there exists a short exact sequence
\[
0 \to \calH \to \calG \to \calF^{\otimes n}(-1) \to 0
\]
of vector bundles on $\Proj(P_R)$ such that $\calG$ is \'etale at $\beta$.
By replacing $R$ by a suitable rational localization, we may ensure that $\calG$ is globally \'etale; by Corollary~\ref{C:globally etale to ample}, $\calG(1)$ is globally ample, as then is its quotient $\calF^{\otimes n}$.
By Lemma~\ref{L:ampleness of a multiple}, $\calF$ is ample.
\end{proof}

\begin{remark} \label{R:not globally etale}
We do not know whether pointwise ample (or equivalently ample) implies globally ample.
For example, we do not know whether $\calF = \calL^{\otimes n}(e)$ is necessarily globally ample in case $e,n>0$ and $\calL$ is \'etale but not globally \'etale.
\end{remark}

We next globalize the construction.
\begin{defn}
Let $\calF$ be a vector bundle on $\FFC_X$.
We say that $\calF$ is \emph{ample} if for every choice of $A$ and $R$
and every morphism $f: \Spa(A,A^+) \to X$, 
the vector bundle $f^*\calF$ on $\FFC_R$ corresponds 
via Theorem~\ref{T:vector bundles2} to an ample vector bundle on $\Proj(P_R)$.
By Theorem~\ref{T:pointwise ample is ample}, this is equivalent to
requiring that the slopes of $\calF$ (as functions on $X$)
be everywhere positive; this equivalence has the following consequences.
\begin{itemize}
\item
If $X = \Spa(A,A^+)$, then a vector bundle on $\Proj(P_R)$ is ample if and only if the corresponding vector bundle on $\FFC_X$ is ample.
\item
Ampleness descends along surjective morphisms on the base. That is,
if $f: Y \to X$ is a surjective morphism of perfectoid adic spaces,
$\calF$ is a vector bundle on $\FFC_X$, and $f^* \calF$ is ample, then $\calF$ is ample. In particular, ampleness is local on the base.
\item
By Theorem~\ref{T:open locus},
ampleness is an open condition on the base, and even on the real quotient of the base. That is, if $\calF$ is a vector bundle on $X$ and the restriction of $\calF$ to $\FFC_{\calH(x)}$ is ample for some $x \in X$, then there exists a partially proper open neighborhood $U$ of $x$ in $X$ such that the restriction of $\calF$ to $\FFC_U$ is also ample.
\end{itemize}
\end{defn}

We finally introduce a key example of an ample line bundle.
\begin{defn} \label{D:canonical line bundle}
For $z$ as in Remark~\ref{R:over analytic field},
the inclusion of modules $\tilde{\calR}_R^{[q^{-1/2}, q^{1/2}]} \to z^{-1} \tilde{\calR}_R^{[q^{-1/2}, q^{1/2}]}$
induces an inclusion $\calO \to \calL$ of vector bundles on $\FFC_R$.
Since this construction is canonically independent of $z$,
it globalizes to define an inclusion $\calO_X \to \calL_X$ of line bundles on $\FFC_X$ corresponding to an element $t_X \in H^0(\FFC_X, \calL_X)$ whose divisor is the image of the canonical section $X \to \FFC_X$.
Note that $\calL_X$ is pure of slope 1 and hence ample
by Theorem~\ref{T:pointwise ample is ample}.
\end{defn}

\begin{lemma} \label{L:line bundle is globally pure}
For $X = \Spa(A,A^+)$, the $\varphi^a$-module corresponding to $\calL_X$ is globally pure of slope $1$.
\end{lemma}
\begin{proof}
Write $z = [\overline{z}] + p z_1$.
Let $M$ be the $\varphi^a$-module over $\tilde{\calR}_R$ free on a single generator $\bv$ satisfying $\varphi^a(\bv) = z_1^{-1} z \bv$; it is evidently globally \'etale.
Define the convergent product
\[
u = \prod_{n=0}^\infty \varphi^{an}(1 + p^{-1} z_1^{-1} [\overline{z}])
\in \tilde{\calR}^+_R.
\]
In $\tilde{\calR}_R$, we have $\varphi^a(u) = p z_1 z^{-1} u$;
consequently, $u \bv$ defines an inclusion $\tilde{\calR}_R \to M(1)$
of $\varphi^a$-modules. Computing in $\tilde{\calR}_R^{[q^{-1/2}, q^{1/2}]}$ shows that $M(1)$ must be the $\varphi^a$-module corresponding to $\calL_X$. This proves the claim.
\end{proof}

In the previous example, when $X = \Spa(A,A^+)$ the zero locus of the section $t_X$ is isomorphic to $\Spec(A)$. This suggests the following conjectures.

\begin{conj}\label{conj:affine locus perfectoid}
Let $\calF$ be a line bundle on $\Proj(P_R)$
such that $\deg(\calF) > 0$ and  $\calF(-\deg(\calF))$ is globally \'etale.
(Note that $\calF(-\deg(\calF))$ makes sense because $\deg(\calF): X \to \ZZ$ is continuous for the discrete topology on $\ZZ$. Note also that $\calF$ is globally ample
by Corollary~\ref{C:globally etale to ample}.)
\begin{enumerate}
\item[(a)]
Choose $t \in H^0(\Proj(P_R),\calF)$ whose restriction to $\FFC_x$ is nonzero for all $x \in \Spa(R,R^+)$. Then the closed subscheme $Z$ of $\Proj(P_R)$ cut out by $t$ is affine. (This would follow from Lemma~\ref{L:affine subscheme from section}
given the existence of a second section $s$ whose zero locus is disjoint from that of $t$.)
\item[(b)]
With notation as in (a), equip the coordinate ring of $Z$
with a uniform norm by identifying it with the global sections of the subspace of $\FFC_R$ cut out by $t$. Then this Banach algebra over $\Qp$ is perfectoid.
\end{enumerate}
\end{conj}

\begin{remark}
In the case where $R = L$ is an analytic field, Conjecture~\ref{conj:affine locus perfectoid}
is established in \cite{fargues-fontaine}. 
In the case where $\deg(\calF) = 1$, it should be possible to
argue by reversing the proof of Lemma~\ref{L:line bundle is globally pure} and rescaling the norm on $R$ as in Remark~\ref{R:transform}(d).
It is less clear what should happen if $\calF(-\deg(\calF))$, which is necessarily \'etale, fails to be globally \'etale; compare Remark~\ref{R:not globally etale}.
\end{remark}

\subsection{\texorpdfstring{$B$}{B}-pairs}
\label{subsec:B-pairs}

We next make contact with another interpretation of the functor $\FFC$ inspired by Berger's construction of \emph{$B$-pairs} \cite{berger-b-pairs}.

\begin{hypothesis}
Throughout \S\ref{subsec:B-pairs}, retain Hypothesis~\ref{H:relative FF}, but assume in addition that $X = \Spa(A,A^+)$.
\end{hypothesis}

\begin{convention}
We use the notation $\calL_X$ to represent not only the line bundle on
$\FFC_X$ described in Definition~\ref{D:canonical line bundle}, but also the line bundle on $\Proj(P_R)$ giving rise to it via
Theorem~\ref{T:vector bundles2}.
\end{convention}

\begin{lemma} \label{L:affine pieces}
Let $Z$ be the image of the canonical section $\Spec(A) \to \Proj(P_R)$.
\begin{enumerate}
\item[(a)]
The open subscheme $\Proj(P_R)-Z$ of $\Proj(P_R)$ is affine.
\item[(b)]
The closed subscheme $Z$ of $\Proj(P_R)$ is a Cartier divisor contained in an open affine subscheme of $\Proj(P_R)$.
\end{enumerate}
\end{lemma}
\begin{proof}
Thanks to the interpretation of $Z$ as
the divisor of the section $t_X$ of the line bundle $\calL_X$,
we may invoke Lemma~\ref{L:affine subscheme from section} to deduce (a).
To prove (b), define $L$ as in Remark~\ref{R:two unit generators};
it then suffices to exhibit some $t_L \in P_{L,1}$  whose support in
$\Proj(P_R)$ is disjoint from $Z$. For this, it suffices to follow the proof of Proposition~\ref{P:global invariants} to force $t$ not to have the slope 1 in its Newton polygon.
\end{proof}

\begin{defn} \label{D:B-pair components}
By Lemma~\ref{L:affine pieces}(a), the complement of $Z$ in $\Proj(P_R)$ is an affine scheme $\Spec(R_1)$. By
Lemma~\ref{L:affine pieces}(b), the completion of $\Proj(P_R)$ along $Z$ is another affine scheme $\Spec(R_2)$,
and $\Spec(R_1) \times_{\Proj(P_R)} \Spec(R_2)$ is yet another affine scheme $\Spec(R_3)$. One can also identify $R_2$ with the $\ker(\theta)$-adic completion of $\tilde{\calR}_R^{\inte,1}$
and $R_3$ with $R_2[z^{-1}]$ for some $z$ generating $\ker(\theta)$.
\end{defn}

\begin{remark} \label{R:flat covering}
In case $R = L$, the morphism $\Spec(R_1 \oplus R_2) \to \Proj(P_R)$ is faithfully flat, so it can be used to
define quasicoherent sheaves and compute their cohomology.
In general, one might expect the same to hold, but one
cannot quite prove it using faithfully flat descent because it is unclear whether $ \Spec(R_2) \to\Proj(P_R)$
is a flat morphism.
Nonetheless, we can salvage something; see Theorem~\ref{T:glueing flat covering}.
\end{remark}

\begin{theorem} \label{T:glueing flat covering}
Set notation as in Remark~\ref{R:flat covering}.
\begin{enumerate}
\item[(a)]
For any flat quasicoherent sheaf $V$ on $\Proj(P_R)$, the cohomology of the complex
\[
0 \to \Gamma(\Spec(R_1), V) \oplus \Gamma(\Spec(R_2), V) \to \Gamma( \Spec(R_3), V) \to 0
\]
where the arrow is given by the difference between the two natural restriction maps,
may be naturally identified with $H^i(\Proj(P_R), V)$.
\item[(b)]
The morphism $\Spec(R_1 \oplus R_2) \to \Proj(P_R)$ is an effective descent morphism for the category
of quasicoherent finite locally free sheaves over schemes (after reversing all arrows).
\item[(c)]
The category of vector bundles on $\Proj(P_R)$ is equivalent to the category of triples $(V_1, V_2, \iota)$ where $V_1$ is a finite projective $R_1$-module, $V_2$ is a finite projective $R_2$-module, and $\iota: V_1 \otimes_{R_1} R_3 \cong V_2 \otimes_{R_2} R_3$ is an isomorphism of $R_3$-modules.
\end{enumerate}
\end{theorem}
\begin{proof}
This follows from Proposition~\ref{P:reduced descent}.
\end{proof}

\section{Relative \texorpdfstring{$(\varphi, \Gamma)$}{(phi, Gamma)}-modules}
\label{sec:perfect}

To conclude, we indicate how to use the preceding constructions to describe \'etale local systems on arbitrary adic spaces; this involves certain sheaves for Scholze's \emph{pro-\'etale topology}. These sheaves generalize the extended Robba ring in ordinary $p$-adic Hodge theory (as considered in \S\ref{subsec:slopes}) but not the Robba ring itself (as considered in \S\ref{subsec:slope Robba}). Generalizing the latter involves passing from the \emph{perfect period rings} that we consider to certain \emph{imperfect period rings} whose construction is somewhat less functorial;
we defer discussion of imperfect period rings to a subsequent paper.

\subsection{The pro-\'etale topology for adic spaces}
\label{subsec:proetale}

In $p$-adic Hodge theory, one studies the Galois theory of a $p$-adic field using certain highly ramified infinite algebraic extensions.
To carry out relative $p$-adic Hodge theory, one needs an analogous geometric construction; one convenient mechanism for this is the \emph{pro-\'etale topology} introduced by Scholze  \cite[\S 3]{scholze2}.

\begin{defn} \label{D:pro-etale}
For $\calC$ a category, a \emph{pro-object} over $\calC$ consists of a pair
$(I, F)$ in which $I$ is a directed poset and $F$ is a
contravariant functor from $I$ to $\calC$.
The pro-objects over $\calC$ form a category $\widehat{\calC}$
in which
\[
\Hom((I,F), (I',F')) =
\varprojlim_{i' \in I'} \varinjlim_{i \in I} \Hom(F(i), F'(i'))
\]
(see \cite[Expos\'e~I, \S 8.10]{sga4-1}).
By design, the category $\widehat{\calC}$ (the \emph{pro-category} associated to $\calC$)
admits inverse limits (modulo set-theoretic difficulties
which we gloss over here).
There is a natural embedding $\calC \to \widehat{\calC}$
taking an object $X \in \calC$ to the pair $(\{0\}, F)$ in which $\{0\}$ is the singleton poset and
$F$ takes $0$ to $X$.

In the cases we are considering, the category $\calC$ admits a forgetful functor to topological spaces
denoted $X \mapsto \left|X\right|$. We extend this to a forgetful functor from $\widehat{\calC}$ to topological spaces
by setting $\left|\varprojlim_{i \in I} X_i\right| = \varprojlim_{i \in I} \left|X_i\right|$.
\end{defn}

\begin{defn}
For $X$ a preadic space, a morphism $U \to V$ in $\widehat{X}_{\et}$
is \emph{\'etale} (resp.\ \emph{finite \'etale}, \emph{faithfully finite \'etale})
if it arises by base extension from an \'etale (resp.\ finite \'etale, faithfully finite \'etale) morphism
$Y_0 \to X_0$ in $X_{\et}$.
One checks formally that these properties are stable under composition 
(as in \cite[Lemma~3.10(ii)]{scholze2})
and base change (as in \cite[Lemma~3.10(i)]{scholze2}).

A morphism $U \to V$ in $\widehat{X}_{\et}$ is \emph{pro-\'etale} if
$U$ admits a \emph{pro-\'etale presentation}
as a cofiltered inverse limit $\varprojlim U_i$ of objects which are \'etale over $V$,
such that $U_i \to U_j$ is faithfully finite \'etale for sufficiently large $j$.
\end{defn}

\begin{lemma} \label{L:proetale properties1}
Let $X$ be a preadic space.
\begin{enumerate}
\item[(a)]
Let $U \to V$ be a pro-\'etale morphism in
$\widehat{X}_{\et}$ and let $W \to V$ be an arbitrary morphism in $\widehat{X}_{\et}$.
Then $U \times_V W \to W$ is pro-\'etale, and the map $|U \times_V W| \to |U| \times_{|V|} |W|$ is surjective.
\item[(b)]
For $U \to V \to W$ pro-\'etale morphisms in $\widehat{X}_{\et}$, the composition $U \to W$ is pro-\'etale.
\end{enumerate}
\end{lemma}
\begin{proof}
To prove (a), we follow the proof of \cite[Lemma~3.10(i)]{scholze2}.
In case $U \to V$ is \'etale, we may assume that $U, V \in X_{\et}$ and realize $W \to V$ using a compatible system of morphisms $W_i \to V$ in $X_{\et}$. We then have maps
\[
\left| U \times_V W \right| = \varprojlim_i \left| U \times_V W_i \right|
\to \varprojlim_i \left| U \right| \times_{\left| V \right|} \left| W_i \right| = 
\left| U \right| \times_{\left| V \right|} \left| W \right|.
\]
If we put the discrete topology on each fibre over $W$, then the central arrow is a surjective map of compact spaces by Remark~\ref{R:compact spaces}(c).
In the general case, choose a pro-\'etale presentation
$U = \varprojlim_i U_i \to V$; we then have maps
\[
\left| U \times_V W \right| = \varprojlim_i \left| U_i \times_V W \right|
\to \varprojlim_i \left| U_i \right| \times_{\left| V \right|} \left| W \right| = 
\left| U_i \right| \times_{\left| V \right|} \left| W \right|.
\]
The central arrow is again surjective by Remark~\ref{R:compact spaces}(c).

To prove (b), we follow the proof of \cite[Lemma~3.10(vi)]{scholze2}.
It suffices to check the case where $U \to V$ is \'etale.
In this case, $U \to V$ is the pullback of some \'etale morphism $U_0 \to V_0$ in $X_{\et}$ along some morphism $V \to V_0$ in $\widehat{X}_{\et}$. 
Choose a pro-\'etale presentation $V = \varprojlim_i V_i \to W$; 
then $V \to V_0$ arises from a compatible family of morphisms $V_i \to V_0$ in $X_{\et}$ (for $i$ large). Hence
$U = \varprojlim_i U_0 \times_{V_0} V_i$ is pro-\'etale over $W$.
\end{proof}

\begin{defn}
Let $X_{\proet}$ denote the full subcategory of $\widehat{X}_{\et}$ consisting of objects which
are pro-\'etale over $X$.
By Lemma~\ref{L:proetale properties1}, we may view $X_{\proet}$ as a site by taking coverings to be families $\{U_i \to Y\}_i$
such that for some (hence any) pro-\'etale presentations $\varprojlim_j U_{i,j}$ of $U_i$ over $Y$ and some (hence any) choice of indices $j = j(i)$ such that $U_{j''} \to U_{j'}$ is faithfully finite \'etale for any $j' \geq j(i)$, the family
$\{U_{i,j(i)} \to Y\}_i$ is a set-theoretic covering.
The resulting site maps to the usual \'etale site $X_{\et}$ via the embedding
$X_{\et} \to \widehat{X}_{\et}$.

In case $X$ is stably adic, we may characterize coverings in $X_{\proet}$ more simply: they are the families $\{U_i \to Y\}_i$ such that the maps $\{\left| U_i \right| \to \left| Y \right| \}$ form a set-theoretic covering.
\end{defn}

\begin{remark} \label{R:tower}
Recall that an inverse limit of spectral spaces is a spectral space by Corollary~\ref{C:inverse limit}.
Consequently, for any preadic space $X$ and any $Y \in X_{\proet}$ admitting a pro-\'etale presentation in which the underlying space
of each term is qcqs, $\left| Y \right|$ is a spectral space.
It follows that for any $Y \in X_{\proet}$,
$|Y|$ is a locally spectral space.
\end{remark}

In case $X$ is stably adic, we can emulate more of \cite[Lemma~3.10]{scholze2}.
\begin{lemma} \label{L:proetale properties2}
Let $X$ be a stably adic space.
\begin{enumerate}
\item[(a)]
For $U \in \widehat{X}_{\et}$ and $W \subseteq |U|$ a quasicompact open set, there exist $V \in \widehat{X}_{\et}$
and an \'etale map $V \to U$ such that $|V| \to |U|$ induces a homeomorphism $|V| \cong W$.
Moreover, if $U \in X_{\proet}$, one can take $V \in X_{\proet}$, and then any morphism
$V' \to U$ in $X_{\proet}$ with image contained in $|W|$ factors through $V$.
\item[(b)]
For any pro-\'etale morphism $U \to V$ in $\widehat{X}_{\et}$, the map $|U| \to |V|$ is open.
\item[(c)]
Any (finite) \'etale map $U \to V$ in $\widehat{X}_{\et}$ with $V \in X_{\proet}$ and $\left| U \right| \to \left| V \right|$ surjective
is the base extension of some surjective (finite) \'etale morphism in $X_{\et}$.
\end{enumerate}
\end{lemma}
\begin{proof}
Given Lemma~\ref{L:proetale properties1}
and the fact that \'etale maps are open (Lemma~\ref{L:finite etale from algebra}(b)),
we may deduce (a)--(c) as in \cite[Lemma~3.10(iii)--(v)]{scholze2}.
\end{proof}

\begin{remark} \label{R:proetale equalizers}
In \cite{scholze2}, the only preadic spaces considered are adic spaces which are \emph{locally  noetherian} (i.e., they are covered by the adic spectra of strongly noetherian adic Banach rings, which are sheafy by 
Proposition~\ref{P:strongly noetherian}). This hypothesis is used in an essential way in \cite[Lemma~3.10(vii)]{scholze2},
which asserts that $X_{\proet}$ admits arbitrary finite projective limits (not just fibred products, which exist by virtue of Lemma~\ref{L:proetale properties1}).
Namely, it is necessary to ensure that objects of $X_{\et}$
locally have only finitely many connected components, so that arbitrary intersections of closed-open subsets are
again closed-open. In the absence of a noetherian hypothesis, we may work with the pro-\'etale site without incident, but we cannot freely apply topos-theoretic machinery
as in \cite[Proposition~3.12]{scholze2}.
\end{remark}

\begin{remark}
One can similarly define a pro-\'etale topology
for rigid analytic spaces over $K$, Berkovich analytic spaces, or schemes.
However, in the cases of rigid or Berkovich analytic spaces, one must take suitable care with the definition of coverings.
In the case of schemes, one can take advantage of the properties of flatness to introduce a simpler variant of the pro-\'etale topology; see Definition~\ref{D:pro-etale schemes}.
\end{remark}

Using the pro-\'etale topology, we may reinterpret the definition of \'etale local systems.

\begin{lemma}
Let $X$ be a preadic space.
For any topological space $T$, the functor $\calF_T: Y \mapsto \Map_{\cont}(\left| Y \right|, T)$ is a sheaf on $X_{\proet}$. In particular, if $T$ is discrete, then $\calF_T$ is the usual constant sheaf associated to $T$.
\end{lemma}
\begin{proof}
It is clear that $\calF_T$ is a sheaf on $X$. To verify that $\calF_T$ is a sheaf on $X_{\et}$, by Proposition~\ref{P:acyclicity template etale} it suffices to observe that for $(A,A^+) \to (B,B^+)$ a faithfully finite \'etale morphism,
by Lemma~\ref{L:finite etale from algebra}(b) the morphism
$\Spa(B,B^+) \to \Spa(A,A^+)$ of topological spaces is open and hence a quotient map.
To complete the proof, it suffices to observe that for $Y \in X_{\proet}$ a tower of faithfully finite \'etale morphisms, the morphism $\left| Y \right| \to \left| X \right|$ is an inverse limit of surjective open morphisms, so it is again a quotient map.
\end{proof}

\begin{defn}
By an \emph{\'etale $\Zp$-local system} (resp.\ an \emph{\'etale $\Qp$-local system}) on $X_{\proet}$, we mean a sheaf in $\Zp$-modules (resp.\ in $\Qp$-vector spaces) locally of the form $\calF_T$ for $T$ a finite free $\Zp$-module (resp.\ finite-dimensional $\Qp$-vector space) carrying its usual $p$-adic topology.
\end{defn}
\begin{lemma} \label{L:proetale local systems}
For $X$ a preadic space, the categories of \'etale $\Zp$-local systems and $\Qp$-local systems on $X$ are equivalent to the corresponding categories on $X_{\proet}$.
\end{lemma}
\begin{proof}
The functors from local systems on $X$ to local systems on $X_{\proet}$ are defined
by pullback as in Definition~\ref{D:pro-etale schemes}.
To check that these are fully faithful, we may reduce to considering local systems on
$\widetilde{\Spa}(A,A^+)$ which become constant on some faithfully finite \'etale tower; however, such towers descend to $\Spec(A)$ by Lemma~\ref{L:finite etale from algebra}(a), so we may appeal to Theorem~\ref{T:proetale local systems schemes}.
By a similar argument, we may also deduce essential surjectivity from
Theorem~\ref{T:proetale local systems schemes}.
(See also \cite[Proposition~8.2]{scholze2}.)
\end{proof}

\begin{defn} \label{D:etale cohomology}
For $X$ a preadic space and $V$ an \'etale $\Zp$-local system or $\Qp$-local system on $X$, following \cite{scholze2} we define the \emph{pro-\'etale cohomology} of $V$ to be the cohomology of the corresponding \'etale local system on $X_{\proet}$, and denote it by $H^i_{\proet}(X, V)$.
\end{defn}

\subsection{Perfectoid subdomains}
\label{subsec:perfectoid subdomains}

We next identify the \emph{perfectoid subdomains} of the pro-\'etale site of a preadic space, which will be used to construct period sheaves.

\begin{hypothesis} \label{H:perfectoid subdomains}
Throughout \S\ref{subsec:perfectoid subdomains},
let $X$ be a preadic space over an analytic field $K$ of residue characteristic $p$.
\end{hypothesis}

\begin{defn} \label{D:structure sheaf}
We define the \emph{structure presheaf} $\calO$ and the sub-presheaves $\calO^{\circ}, \calO^+$ on $X_{\proet}$
as follows. For $Y = \varprojlim_i Y_i \in X_{\proet}$,
put
\begin{align*}
\calO(Y) &= \varinjlim_i \Gamma(Y_i, \calO_{Y_i}) \\
\calO^\circ(Y) &= \varinjlim_i \Gamma(Y_i, \calO^\circ_{Y_i}) \\
\calO^+(Y) &= \varinjlim_i \Gamma(Y_i, \calO^+_{Y_i}).
\end{align*}
We define the \emph{spectral seminorm} on $\calO(Y)$ as follows.
Choose $j$ so that the maps $Y_i \to Y_j$ are faithfully finite \'etale for all $i \geq j$.
Then for each $f \in \calO(Y)$, choose $i \geq j$ for which $f \in \calO_{Y_i}(Y_i)$ and
define the spectral seminorm of $f$ in $\calO(Y)$ to be the spectral seminorm of $f$ in $\calO_{Y_i}(Y_i)$.
\end{defn}

\begin{remark} \label{R:spectral seminorm}
Another way to interpret the spectral seminorm introduced in Definition~\ref{D:structure sheaf}
is to observe that each element of $|Y|$ defines
a multiplicative seminorm on $\calO(Y)$, and that the spectral seminorm is the supremum of these.
\end{remark}

To define the sheaves we are interested in, we use a special neighborhood basis for the pro-\'etale topology.
\begin{defn}
An element $Y$ of $X_{\proet}$ is a \emph{perfectoid subdomain} if it admits a pro-\'etale presentation
$\varprojlim_i Y_i$ satisfying the following conditions.
\begin{enumerate}
\item[(a)]
There exists an index $j \in I$ such that the maps $Y_i \to Y_j$ are faithfully finite \'etale for all $i \geq j$ and the space $Y_j$
is a preadic affinoid space over $K$ (as then are the spaces $Y_i$ for all $i \geq j$).
\item[(b)]
The completion of $\calO(Y)$ for the spectral seminorm is a perfectoid
(if $K$ is of characteristic $0$) or perfect (if $K$ is of characteristic $p$) Banach algebra over $K$.
\end{enumerate}
These satisfy the following properties.
\begin{enumerate}
\item[(i)]
If $Y$ is a perfectoid subdomain and $Z \to Y$ is a morphism in $X_{\proet}$ which is the pullback of
a rational subdomain embedding of elements of $X_{\et}$, then $Z$ is also a perfectoid subdomain (Theorem~\ref{T:perfectoid rational}).
\item[(ii)]
Any finite \'etale cover of a perfectoid subdomain is a perfectoid subdomain
(Theorem~\ref{T:mixed lift ring}).
\item[(iii)]
The fibred product of two perfectoid subdomains is a perfectoid subdomain
(by (i) and (ii) plus the local factorization of \'etale morphisms).
\end{enumerate}
\end{defn}

\begin{lemma} \label{L:neighborhood basis adic}
If $X = \Spa(A,A^+)$ for some adic Banach algebra $(A,A^+)$ over $K$, then
there exists $Y = \varprojlim_i Y_i \in X_{\proet}$ which is a pro-finite \'etale covering of $X$
and is a perfectoid subdomain of $X_{\proet}$.
Consequently, for arbitrary $X$, the perfectoid subdomains of $X_{\proet}$ form a neighborhood basis of $X_{\proet}$.
\end{lemma}
\begin{proof}
Apply Lemma~\ref{L:perfectoid neighborhood} to $A^u$.
\end{proof}

We have the following analogue of Proposition~\ref{P:djvdp}.
\begin{prop} \label{P:djvdp proetale}
Let $(A,A^+)$ be a perfectoid adic Banach algebra over $\Qp$
and put $X = \Spa(A,A^+)$.
Let $\calB$ be the collection of perfectoid subdomains in $X_{\proet}$ which can written as faithfully finite \'etale towers over perfectoid subodmains in $X_{\et}$.
(By Lemma~\ref{L:neighborhood basis adic} and the local factorization of \'etale morphisms, these form a stable basis for $X_{\proet}$.)
Let $\calP$ be a property of coverings in $X_{\proet}$ of and by elements of $\calB$, and assume that the following conditions hold.
\begin{enumerate}
\item[(a)]
Any covering admitting a refinement having property $\calP$ also has property $\calP$.
\item[(b)]
Any composition of coverings having property $\calP$ also has property $\calP$.
\item[(c)]
For any $Y \in \calB$, any rational covering of $Y$ has property $\calP$.
\item[(d)]
For any $Y \in \calB$, any tower $Y' \to Y$ of faithfully finite \'etale morphisms, viewed as a covering, has property $\calP$. 
\end{enumerate}
Then every covering in $X_{\proet}$ of and by elements of $\calB$ has property $\calP$.
\end{prop}
\begin{proof}
By Proposition~\ref{P:djvdp}, any covering in $X_{\et}$ of and by elements of $\calB$ has property $\calP$. The claim then follows by (b) and (d).
\end{proof}

\begin{defn}
We define the \emph{completed structure presheaf} $\widehat{\calO}$
(or $\widehat{\calO}_X$ for clarity)
on $X_{\proet}$ as follows: for $Y \in X_{\proet}$, let $\widehat{\calO}(Y)$ be the completion of $\calO(Y)$ for the spectral seminorm. We similarly define the subpresheaf $\widehat{\calO}^+$ of $\widehat{\calO}$.
\end{defn}

\begin{lemma} \label{L:perfectoid proetale sheaf}
Let $(A,A^+)$ be a perfectoid adic Banach algebra over $\Qp$ or a perfect uniform Banach algebra over $\Fp$ and put $X = \Spa(A,A^+)$.
\begin{enumerate}
\item[(a)]
We have  $H^0(X_{\proet}, \widehat{\calO}) = A$.
\item[(b)]
For $i>0$, $H^i(X_{\proet}, \widehat{\calO}) = 0$.
\item[(c)]
For $i>0$, the group $H^i(X_{\proet}, \widehat{\calO}^+)$ is annihilated by $\gothm_A$.
\end{enumerate}
\end{lemma}
\begin{proof}
Again as in the proof of Proposition~\ref{P:acyclicity template},
it suffices to check universal (almost) \v{C}ech-acyclicity for coverings of and by elements of a suitable basis. Using Proposition~\ref{P:djvdp proetale} and Theorem~\ref{T:Tate sheaf property for structure sheaf}, we reduce to checking
\v{C}ech-acyclicity for a tower of \'etale surjective morphisms of affinoid perfectoid spaces. Let $\{Y_i\}_i$ be the terms in such a tower and put $B_i = \calO(Y_i)$.

Suppose first that we are in the perfect case. To clarify notation, we write $R,S_i$ in place of $A, B_i$. For each $i$, the \v{C}ech sequence
\[
0 \to R \to S_i \to S_i \otimes_R S_i \to \cdots
\]
is strict exact and hence almost optimal exact
(Remark~\ref{R:perfect uniform strict}); we may thus take completed direct limits to obtain another strict exact sequence
\begin{equation} \label{eq:perfectoid proetale sheaf1}
0 \to R \to S_\infty \to S_\infty \widehat{\otimes}_R S_\infty \to \cdots,
\end{equation}
where $S_\infty$ is the completed direct limit of the $S_i$.

In the perfectoid case, we apply the perfectoid correspondence to each morphism in the tower to obtain a tower of \'etale surjective morphism of affinoid perfect uniform spaces. By
applying Proposition~\ref{P:perfectoid uniform strict}(b) to \eqref{eq:perfectoid proetale sheaf1}, we deduce the desired result.
\end{proof}

\begin{cor} \label{C:completed structure presheaf}
The completed structure presheaf on $X_{\proet}$ agrees with its sheafification on perfectoid subdomains.
\end{cor}

\begin{remark}
Lemma~\ref{L:perfectoid proetale sheaf}(c) implies that the subsheaf $\widehat{\calO}^+$ of $\widehat{\calO}$
is \emph{almost acyclic}
in the sense of Proposition~\ref{P:almost sheaf property}.
\end{remark}

\begin{defn}
For $K$ of characteristic zero,
we define the sheaves $\overline{\calO}, \overline{\calO}^\circ, \overline{\calO}^+$ on $X_{\proet}$
via the perfectoid correspondence: for $Y \in X_{\proet}$ a perfectoid subdomain, $(\overline{\calO}(Y), \overline{\calO}^{+}(Y))$ is the perfect uniform adic Banach algebra of characteristic $p$
corresponding to $(\widehat{\calO}(Y),\widehat{\calO}^{+}(Y))$
via Theorem~\ref{T:perfectoid ring}.
\end{defn}

\begin{defn}
Let $\nu_X: X_{\proet} \to X_{\et}$ be the natural morphism.
We say that $X = \Spa(A,A^+)$ is \emph{pro-sheafy} if $X$ is uniform and sheafy
and the morphism $\calO_X \to \nu_{X*} \widehat{\calO}_X$ is an isomorphism. For example, if $A$ is perfectoid, then $X$ is pro-sheafy by Lemma~\ref{L:perfectoid proetale sheaf}.
For another example, if $A = K$, then $X$ is pro-sheafy
by the Ax-Sen-Tate theorem \cite{ax}.
\end{defn}

\begin{remark}
Suppose that $X = \Spa(A,A^+)$ and that there exists a perfectoid Banach algebra $B$ over $K$ admitting a bounded homomorphism $A \to B$  which splits in the category of Banach modules over $A$.
Then $X$ is pro-sheafy. For instance, this is the case if $A$ is strongly preperfectoid
(by choosing a perfectoid field which admits a Schauder basis over $\QQ_p$).
Also, if $A$ has this property, then so does $A\{T_1,\dots,T_n\}$.
\end{remark}

\begin{remark}
In case $X = \Spa(A,A^+)$ for $A$ a reduced normal affinoid algebra over $K$,
one can prove that $X$ is pro-sheafy by using Temkin's resolution of singularities for affinoid algebras \cite{temkin-functorial}
to reduce to the case where $A$ is smooth over $K$,
then making a construction of imperfect period rings generalizing that of Andreatta-Brinon \cite{andreatta-brinon}. We will discuss this point in a subsequent paper.
\end{remark}

One has a Kiehl glueing property for the pro-\'etale topology.
\begin{theorem} \label{T:vector bundles on perfectoid}
Suppose $X$ is perfectoid or perfect uniform. Then pullback of finite locally free $\calO_X$-modules to finite locally free $\widehat{\calO}_X$-modules 
defines an equivalence of categories.
\end{theorem}
\begin{proof}
We treat the case where $K$ is of characteristic $0$, the characteristic $p$ case being similar but easier. Full faithfulness of the pullback functor is immediate from Lemma~\ref{L:perfectoid proetale sheaf}, so we need only check essential surjectivity.
By Theorem~\ref{T:adic vector bundle}
and Proposition~\ref{P:djvdp},
it suffices to check
descent in case $X = \Spa(A,A^+)$ for some perfectoid adic Banach algebra $(A,A^+)$
and $Y \to X$ is a tower of faithfully finite \'etale morphisms; moreover, we may formally reduce to the case of a countable tower. For $B = \calO(Y)$, it then suffices to construct
an $A$-linear splitting $B \to A$ of the inclusion $A \to B$. 

Write $B$ as the completed direct limit of an increasing
sequence of faithfully finite \'etale $A$-subalgebras $B_i$.
By Theorem~\ref{T:almost purity}, one can find a $B_i$-linear splitting of $B_i \to B_{i+1}$ such that the operator norm of $B_{i+1} \to B_i \to B_{i+1}$ is at most $p^{p^{-i}}$. Chaining these together gives the desired splitting $B \to A$.
\end{proof}

\begin{remark}
For $X$ perfectoid or perfect uniform, Scholze has suggested a variant of the pro-\'etale topology more in the spirit of the definition for schemes (Definition~\ref{D:pro-etale schemes}).
In this approach, one says that a morphism $f: Y \to X$ of perfectoid spaces is \emph{pro-\'etale} if $Y$ is ``similar to'' an inverse limit of \'etale spaces over $X$
in the sense of \cite[\S 2.4]{scholze-weinstein}. More precisely, one posits the existence of a cofiltered inverse system $\{Y_i\}_{i \in I}$ in $X_{\et}$ and a compatible family of morphisms $Y \to Y_i$ with the following properties.
\begin{enumerate}
\item[(a)]
The transition morphisms $Y_j \to Y_i$ are qcqs (i.e., for any morphism $U \to V_i$ with $U$ a preadic affinoid space, $Y_j \times_{Y_i} U$ is qcqs).
\item[(b)]
The induced map $\left| Y \right| \to \varprojlim_i \left| Y_i \right|$ is a homeomorphism.
\item[(c)]
There exists a covering of $Y$ by perfectoid affinoid subspaces $\Spa(A, A^+)$,
each with the following property: for each $i \in I$, consider all 
perfectoid affinoid subspaces $\Spa(A_{ij}, A_{ij}^+)$ of $Y_i$ through which $\Spa(A,A^+) \to Y_i$ factors. Then the direct limit of the maps $A_{ij} \to A$ has dense image.
(Note that $Y$ then admits a basis of such subspaces.)
\end{enumerate}
One then defines a \emph{pro-\'etale covering} of a perfectoid space $X$ to be a family of pro-\'etale morphisms $\{U_i \to X\}_{i \in I}$ with the property that for any quasicompact open subspace $V$ of $X$, there exist a finite subset $J$ of $I$ 
and some quasicompact open subsets $V_j$ of $U_j \times_X V$ for each $j \in J$ such that $\{V_j \to V\}_{j \in J}$ is a set-theoretic covering. (The auxiliary finiteness condition on coverings is needed because the analogue of Lemma~\ref{L:proetale properties2} does not hold; this is 
typical for ``large'' topologies such as the fpqc topology on schemes.)

With this definition, it is not difficult to check that all of the acyclicity assertions we make about the pro-\'etale topology on perfectoid or perfectoid uniform spaces, such as  Lemma~\ref{L:perfectoid proetale sheaf} and Lemma~\ref{L:perfectoid proetale period sheaf}, remain true for this finer topology. However, it is not immediately clear how to adapt the proof of 
Theorem~\ref{T:vector bundles on perfectoid} to the finer topology.
\end{remark}

\subsection{\texorpdfstring{$\varphi$}{Phi}-modules and local systems}

Using the completed structure sheaf, we proceed to construct perfect period sheaves on preadic spaces over $\Qp$.
We then relate $\varphi$-modules over these period sheaves to \'etale local systems on the spaces. Even over perfectoid spaces, this adds to the discussion in \S\ref{sec:adic} because now we can also say something about the pro-\'etale cohomology of local systems.

\begin{remark}
The suite of notations introduced below is similar to the ``Colmez style'' of notations in $p$-adic Hodge theory (as distinguished from the ``Fontaine style'' of notations).
One key difference is that those notations primarily distinguish between rings with and without integral structure, by basing their notations on the letters $\bA$ and $\bB$ respectively. We prefer to further emphasize the difference between bounded and unbounded rings without integral structure (e.g., the bounded Robba ring $\calR^{\bd}$ versus the full Robba ring $\calR$), so we derive our notations for the latter from the letter $\bC$.
\end{remark}

\begin{hypothesis}
For the remainder of \S\ref{sec:perfect}, let $X$ be a preadic
space over $\QQ_{p^d}$ for some positive integer $d$,
and let $Y \in X_{\proet}$ denote an arbitrary perfectoid subdomain.
Let $Y'$ be the perfect uniform adic space over $\FF_{p^d}$ associated to $Y$ via the perfectoid correspondence.
\end{hypothesis}

\begin{defn} \label{D:perfect period sheaves}
Define sheaves on $X_{\proet}$ by the following formulas:
\begin{gather*}
\tilde{\bA}_X(Y) = W(\overline{\calO}(Y)), \qquad
\tilde{\bA}^{+}_X(Y) = \tilde{\calR}_{\overline{\calO}(Y)}^{\inte,+}, \qquad
\tilde{\bA}^{\dagger,r}_X(Y) = \tilde{\calR}_{\overline{\calO}(Y)}^{\inte,r}, \qquad
\tilde{\bA}^{\dagger}_X(Y) = \tilde{\calR}_{\overline{\calO}(Y)}^{\inte}, \\
\tilde{\bB}_X(Y) = W(\overline{\calO}(Y))[p^{-1}], \qquad
\tilde{\bB}^{\dagger,r}_X(Y) = \tilde{\calR}_{\overline{\calO}(Y)}^{\bd,r}, \qquad
\tilde{\bB}^{\dagger}_X(Y) = \tilde{\calR}_{\overline{\calO}(Y)}^{\bd}, \\
\tilde{\bC}^+_X(Y) = \tilde{\calR}_{\overline{\calO}(Y)}^{+}, \qquad
\tilde{\bC}^r_X(Y) = \tilde{\calR}_{\overline{\calO}(Y)}^r, \qquad
\tilde{\bC}^{I}_X(Y) = \tilde{\calR}_{\overline{\calO}(Y)}^{I}, \qquad
\tilde{\bC}_X(Y) = \tilde{\calR}_{\overline{\calO}(Y)}.
\end{gather*}
\end{defn}

\begin{lemma} \label{L:perfectoid proetale period sheaf}
The sheaves defined in Definition~\ref{D:perfect period sheaves} are acyclic on $Y$.
\end{lemma}
\begin{proof}
Imitate the proof of Lemma~\ref{L:perfectoid proetale sheaf}.
\end{proof}

\begin{defn}
The sheaves $\tilde{\bA}_X, \tilde{\bA}^+_X, \tilde{\bA}^{\dagger}_X,
\tilde{\bB}_X, \tilde{\bB}^{\dagger}_X, \tilde{\bC}^+_X, \tilde{\bC}_X$ carry
actions of $\varphi$; we refer to the sheaves collectively as
\emph{perfect period sheaves}.
By a \emph{$\varphi^d$-module} over one of these sheaves, we will mean a sheaf of finite projective modules over the corresponding sheaf of rings equipped with an isomorphism with its $\varphi^d$-pullback.

We say that a $\varphi^d$-module $M$ over $\tilde{\bC}_X$ is \emph{pure} (resp.\ \emph{\'etale}) at a point $x \in X$ if its restriction to some perfectoid subdomain is pure (resp.\ \'etale) at some lift of $x$. The same is then true for any other lift of $x$ to any other perfectoid subdomain, by virtue of the pointwise nature of the pure and \'etale conditions (Corollary~\ref{C:local model not locally free}).
\end{defn}

\begin{theorem}
For $M$ a $\varphi^d$-module over $\tilde{\bC}_X$, the pure locus (resp. the \'etale locus) of $M$ is a partially proper open subset of $X$.
In particular, by Lemma~\ref{L:taut Hausdorff quotient}, if $X$ is taut, then so are the pure locus and the \'etale locus.
\end{theorem}
\begin{proof}
The claim is local, so we may assume $X = \Spa(A, A^+)$.
In this case, we deduce the claim by constructing $A \to B$
as in Lemma~\ref{L:perfectoid neighborhood},
applying Corollary~\ref{C:etale locus is open} to $\calM(B)$,
and applying Remark~\ref{R:compact spaces}(b) to the map $\calM(B) \to \calM(A)$.
\end{proof}

\begin{theorem} \label{T:proetale equivalence1 global}
The following categories are equivalent via functors which preserve rank and are natural for pullbacks on $X$.
\begin{enumerate}
\item[(a)]
The category of \'etale $\ZZ_{p^d}$-local systems over $X$.
\item[(b)]
The category of $\varphi^d$-modules over $\tilde{\bA}_X$.
\item[(c)]
The category of $\varphi^d$-modules over $\tilde{\bA}^\dagger_X$.
\end{enumerate}
\end{theorem}
\begin{proof}
Immediate from Theorem~\ref{T:perfect equivalence2 global}  and Lemma~\ref{L:proetale local systems}.
\end{proof}
\begin{cor} \label{C:zp etale on perfectoid subdomain}
For $X = Y$,
the base extension functor from $\varphi^d$-modules over
$\tilde{\calE}^{\inte}_{\overline{\calO}(Y)}$ (resp.\ $\tilde{\calR}^{\inte}_{\overline{\calO}(Y)}$)
to $\varphi^d$-modules over $\tilde{\bA}_X$ (resp.\ $\tilde{\bA}^{\dagger}_X$) is an equivalence of categories.
\end{cor}
\begin{proof}
Combine Theorem~\ref{T:perfect equivalence2}
with Theorem~\ref{T:proetale equivalence1 global}.
\end{proof}

In our description of \'etale $\Qp$-local systems, we would like to include ``vector bundles over $\FFC_X$'' but it is not so straightforward to make sense of what $\FFC_X$ would be. Instead, we settle for some indirect descriptions of the category of vector bundles.
\begin{defn} \label{D:proetale vector bundle}
Define the sheaf of graded rings $P_X = \oplus_{n=0}^\infty P_{X,n}$ by the formula
\[
P_{X,n}(Y) = \{x \in \tilde{\bC}_X(Y): \varphi^d(x) = p^n x\}.
\]
By a \emph{vector bundle over $\FFC_X$}, we will mean a sheaf
of graded $P_X$-modules whose restriction to each $Y$ has the form
$\oplus_{n=0}^\infty \Gamma(\FFC_Y, V^{\otimes n})$ for some vector bundle $V$ on $\FFC_Y$. Morphisms of vector bundles are morphisms of graded modules whose restrictions to each $Y$ arise from morphisms of vector bundles.
\end{defn}

\begin{remark} \label{R:same vector bundles}
In case $X = Y$, then $\FFC_Y$ is a well-defined adic space and we have already introduced the category of vector bundles over $\FFC_Y$.
In order to make sense of Definition~\ref{D:proetale vector bundle},
one needs to check that base extension of true vector bundles over $\FFC_Y$ to vector bundles over $\FFC_X$ in this sense is an equivalence of categories. Fortunately, this follows from Theorem~\ref{T:vector bundles on perfectoid}
by Theorem~\ref{T:Kiehl for Robba} and a splitting argument.
\end{remark}

\begin{defn}
Define the sheaf $\bB^{+}_{\mathrm{dR},X}$ by
setting $\bB^{+}_{\mathrm{dR},X}(Y)$ to be the $\ker(\theta)$-adic completion of $\tilde{\bB}^{\dagger,1}_X(Y)$.
Let $\bB_{\mathrm{dR},X}$ be the localization of
$\bB^{+}_{\mathrm{dR},X}$ obtained by inverting a generator of $\ker(\theta)$.

Define the sheaf $\bB_{\mathrm{e},X}$ by setting
$\bB_{\mathrm{e},X}(Y)$ to be the coordinate ring of the open affine subscheme of $\Proj(P_X(Y))$ obtained by removing the zero locus of the canonical section $Y \to \FFC_Y$. The rings
$\bB_{\mathrm{e},X}(Y), \bB^{+}_{\mathrm{dR},X}(Y),
\bB_{\mathrm{dR},X}(Y)$
correspond to the rings $R_1, R_2, R_3$ introduced in
Definition~\ref{D:B-pair components}.

By a \emph{B-pair} over $X$, we will mean a triple $(M_{\mathrm{e}}, M^+_{\mathrm{dR}}, M_{\mathrm{dR}})$
in which $M_{\mathrm{e}}$ is a finite projective $\bB_{\mathrm{e},X}$-module, $M^+_{\mathrm{dR}}$ is a finite projective $\bB_{\mathrm{dR},X}^+$-module, and
$M_{\mathrm{dR}}$ is a finite projective $\bB_{\mathrm{dR},X}$-module
equipped with isomorphisms
\[
M_{\mathrm{dR}}\cong M_{\mathrm{e}} \otimes_{\bB_{\mathrm{e},X}} \bB_{\mathrm{dR},X} \cong  M_{\mathrm{dR}}^+ \otimes_{\bB_{\mathrm{dR},X}^+} \bB_{\mathrm{dR},X}.
\]
\end{defn}

\begin{theorem} \label{T:vector bundles global}
The categories of $\varphi^d$-modules over $\tilde{\bC}_X$, vector bundles on $\FFC_X$, and $B$-pairs on $X$ are functorially equivalent.
\end{theorem}
\begin{proof}
Immediate from Theorems~\ref{T:vector bundles2a}, \ref{T:glueing flat covering}.
\end{proof}

\begin{theorem} \label{T:proetale equivalence2 global}
The following categories are equivalent via functors which preserve rank and are natural for pullbacks on $X$.
\begin{enumerate}
\item[(a)]
The category of \'etale $\QQ_{p^d}$-local systems over $X$.
\item[(b)]
The category of \'etale $\varphi^d$-modules over $\tilde{\bB}_X$.
\item[(c)]
The category of \'etale $\varphi^d$-modules over $\tilde{\bB}^\dagger_X$.
\item[(d)]
The category of \'etale $\varphi^d$-modules over $\tilde{\bC}_X$.
\item[(e)]
The category of vector bundles on $\FFC_X$ which are pointwise semistable of degree $0$.
\end{enumerate}
\end{theorem}
\begin{proof}
Immediate from Theorem~\ref{T:perfect equivalence2b}, 
Theorem~\ref{T:vector bundles2} and Lemma~\ref{L:proetale local systems}.
\end{proof}
\begin{cor} \label{C:qp etale on perfectoid subdomain}
For $X = Y$,
the base extension functor from $\varphi^d$-modules over  $\tilde{\calR}_{\overline{\calO}(Y)}$
to $\varphi^d$-modules over $\tilde{\bC}_X$
is an equivalence of categories.
\end{cor}
\begin{proof}
Immediate from Theorem~\ref{T:proetale equivalence2 global}
and Remark~\ref{R:same vector bundles}.
\end{proof}

\begin{cor} \label{C:pure is adic local}
For $X = Y$, the base extension functor from
pure $\varphi^d$-modules over $\tilde{\calE}_{Y'}$
(resp.\ $\tilde{\calR}^{\bd}_{Y'}$) to pure $\varphi^d$-modules
over $\tilde{\bB}_X$ (resp.\ $\tilde{\bB}^\dagger_X$)
is an equivalence of categories.
\end{cor}
\begin{proof}
Immediate from Theorem~\ref{T:proetale equivalence2 global}
and the fact that \'etale $\Qp$-local systems arise adic-locally from isogeny $\Zp$-local systems.
\end{proof}

\begin{remark}
The analogue of Fontaine's functor $D_{\mathrm{dR}}$ in this 
context is the functor taking a $B$-pair $M$ to the global sections of 
$M_{\mathrm{dR}}$; for instance, this functor appears in Scholze's
approach to the comparison isomorphism in \cite{scholze2}. We will return to this point, and to the analogues of the functors $D_{\mathrm{crys}}$ and $D_{\mathrm{st}}$, in a subsequent paper.
\end{remark}

\begin{remark}
Thanks to Lemma~\ref{L:perfectoid proetale period sheaf},
any $\varphi^d$-module over any of the sheaves described in 
Definition~\ref{D:perfect period sheaves}
may be evaluated at any perfectoid space mapping to $X$, not just perfectoid subdomains.
\end{remark}

\subsection{Comparison of cohomology}

We next compare the pro-\'etale cohomology of local systems with $\varphi$-cohomology.
\begin{defn}
For $M$ a $\varphi^d$-module over a perfect period sheaf,
define the cohomology groups $H^i_{\varphi^d}(M)$ as the hypercohomology groups of the complex
\[
0 \to M \stackrel{\varphi^d - 1}{\to} M \to 0
\]
of sheaves on $X_{\proet}$.
\end{defn}

\begin{theorem} \label{T:proetale cohomology1}
Let $E$ be an \'etale $\ZZ_{p^d}$-local system on $X$.
Apply Theorem~\ref{T:proetale equivalence1 global} to produce
a $\varphi^d$-module $M$ over $\tilde{\bA}_X$ and
a $\varphi^d$-module $M^\dagger$ over $\tilde{\bA}^\dagger_X$.
Then for $i \geq 0$, there are functorial (in $E$ and $X$)
isomorphisms
\[
H^i_{\proet}(X, E) \to H^i_{\varphi^d}(M^\dagger) \to H^i_{\varphi^d}(M).
\]
\end{theorem}
\begin{proof}
The morphism $H^i_{\varphi^d}(M^\dagger) \to H^i_{\varphi^d}(M)$
is induced by the inclusion $M^\dagger \to M$. By Theorem~\ref{T:Galois cohomology2}, this morphism induces quasi-isomorphisms of the kernel and cokernel sheaves of $\varphi^d-1$.
By computing the hypercohomology groups using matching spectral sequences, we see that $H^i_{\varphi^d}(M^\dagger) \to H^i_{\varphi^d}(M)$ is itself an isomorphism.

To produce the isomorphism $H^i_{\proet}(X, E) \to H^i_{\varphi^d}(M)$,
we reinterpret the proof of Theorem~\ref{T:Galois cohomology2}
in the pro-\'etale context using the interpretation of $E$ as a locally constant sheaf $\calF_T$ on $X_{\proet}$ in the sense of Lemma~\ref{L:proetale local systems}.
To begin with, we may view $M$ as the sheafification of the presheaf taking $Y$ to $W(\overline{\calO}(Y)) \widehat{\otimes}_{\Zp} E(Y)$.
In this interpretation, the arrow $E \to M$ in the sequence
\begin{equation} \label{eq:proetale ASW sequence}
0 \to E \to M \stackrel{\varphi^d-1}{\to} M \to 0
\end{equation}
is the natural map $1 \otimes \id$, so it is clear that we obtain a complex.

To prove the theorem, it suffices to check that the sequence
\eqref{eq:proetale ASW sequence} is exact. It is enough to compare sections over a perfectoid subdomain $Y$ on which $E \cong \calF_T$ for some finite free $\Zp$-module $T$.
By the perfectoid correspondence, we may identify
$E(Y)$ not only with $\Map_{\cont}(\left| Y \right|, T)$
but also with $\Map_{\cont}(\left| Y' \right|, T)$;
the injectivity and the exactness at the middle are now immediate from
Corollary~\ref{C:idempotents}. For any $\bv\in M(Y)$, by Proposition~\ref{P:DM relative}, there exists a perfectoid subdomain $Z$ which is pro-\'etale over $Y$ 
(namely, it is a tower of faithfully finite \'etale morphisms over $Y$)
such that $\bv$ can be lifted to $M(Z)$ via the map $\varphi^d-1$. This proves the surjectivity.  
 \end{proof}

\begin{remark} \label{R:proetale cohomology global1}
One aspect of Theorem~\ref{T:Galois cohomology2} that is hidden in the statement and proof of Theorem~\ref{T:proetale cohomology1}  is that
while $E$ is determined by its sections only on sufficiently small perfectoid subdomains, $M$ is determined by its sections on arbitrary perfectoid subdomains by Corollary~\ref{C:zp etale on perfectoid subdomain}. The same will happen in Theorem~\ref{T:proetale cohomology2b} by virtue of Corollary~\ref{C:qp etale on perfectoid subdomain}.
\end{remark}

\begin{defn}
For $M = (M_{\mathrm{e}}, M^+_{\mathrm{dR}}, M_{\mathrm{dR}})$
a $B$-pair over $X$, define the cohomology groups $H^i_{B}(M)$ as the hypercohomology groups of the complex
\[
0 \to M_{\mathrm{e}} \times M_{\mathrm{dR}}^+ \to M_{\mathrm{dR}} \to 0 \qquad (x,y) \mapsto x-y
\]
of sheaves on $X_{\proet}$.
\end{defn}

\begin{theorem} \label{T:proetale cohomology2b}
Let $V$ be an \'etale $\QQ_{p^d}$-local system over $X$.
Apply Theorem~\ref{T:proetale equivalence2 global} to produce
an \'etale $\varphi^d$-module $M$ over $\tilde{\bB}_X$,
an \'etale $\varphi^d$-module $M^\dagger$ over $\tilde{\bB}^\dagger_X$,
an \'etale $\varphi^d$-module $M_{\bC}$ over $\tilde{\bC}_X$,
and a B-pair $M_B = (M_{\mathrm{e}}, M^+_{\mathrm{dR}}, M_{\mathrm{dR}})$
over $X$.
Then for $i \geq 0$, there are functorial (in $E$ and $X$)
isomorphisms
\[
H^i_{\proet}(X, V) \to H^i_{\varphi^d}(M) \to H^i_{\varphi^d}(M^\dagger)
\to H^i_{\varphi^d}(M_{\bC})
 \to H^i_B(M_B).
\]
\end{theorem}
\begin{proof}
We reduce at once to the case where $V$ is an isogeny $\ZZ_{p^d}$-local system. The comparisons among
\[
H^i_{\proet}(X, V), H^i_{\varphi^d}(M), H^i_{\varphi^d}(M^\dagger), H^i_{\varphi^d}(M_{\bC})
\]
are then achieved by constructing exact sequences
\begin{gather*}
0 \to V \to M \stackrel{\varphi^d-1}{\to} M \to 0 \\
0 \to V \to M^\dagger \stackrel{\varphi^d-1}{\to} M^\dagger \to 0 \\
0 \to V \to M_{\bC} \stackrel{\varphi^d-1}{\to} M_{\bC} \to 0
\end{gather*}
analogous to \eqref{eq:proetale ASW sequence};
for the exactness we depend on Corollary~\ref{C:extended Robba invariants} in addition to
Corollary~\ref{C:idempotents} and Proposition~\ref{P:DM relative}.
The comparison between $H^i_{\varphi^d}(M_{\bC})$ and $H^i_B(M_B)$ follows from Theorem~\ref{T:compare cohomology} and
Theorem~\ref{T:glueing flat covering}(a).
\end{proof}

\subsection{Comparison with classical \texorpdfstring{$p$}{p}-adic Hodge theory}
\label{subsec:compare classical}

To conclude, we compare this construction to the classical theory of $(\varphi, \Gamma)$-modules.
\begin{defn}
With notation as in \S\ref{subsec:slope Robba}, define the rings
\begin{gather*}
\bA_{\Qp} = \gotho_{\calE_{\Qp}}, \, \bA^{\dagger}_{\Qp} = \calR^{\inte}_{\Qp},  \, \bB_{\Qp} = \calE_{\Qp},\,  \bB^{\dagger}_{\Qp} = \calR^{\bd}_{\Qp}, \, \bC_{\Qp} = \calR_{\Qp}, \\
\tilde{\bA}_{\Qp} = \gotho_{\tilde{\calE}_{\Qp}}, \, \tilde{\bA}^{\dagger}_{\Qp} = \tilde{\calR}^{\inte}_{\Qp},  \, \tilde{\bB}_{\Qp} =
\tilde{\calE}_{\Qp},\,  \tilde{\bB}^{\dagger}_{\Qp} = \tilde{\calR}^{\bd}_{\Qp}, \, \tilde{\bC}_{\Qp} = \tilde{\calR}_{\Qp}.
\end{gather*}
In addition to the endomorphism $\varphi$, these rings also carry an action of the group $\Gamma = \Zp^\times$ characterized by the formula
\[
\gamma(1 + T) = (1+T)^\gamma = \sum_{n=0}^\infty \binom{\gamma}{n} T^n.
\]
\end{defn}

\begin{defn}
Let $F$ be the completion of $\Qp(\mu_{p^\infty})$. This analytic field is perfectoid: it is of characteristic $0$, not discretely valued, and $\overline{\varphi}$ is surjective on
\[
\gotho_F/(p) \cong \ZZ_p[\mu_{p^\infty}]/(p)
\cong \FF_p[T_0, T_1, \dots]/(T_0-1, T_1^p-T_0, \dots).
\]
By the henselian property of $\bA^{\dagger}_{\Qp}$ (see Definition~\ref{D:Robba rings}),
Lemma~\ref{L:descend etale on field} (or Krasner's lemma),
and the perfectoid correspondence for analytic fields
(Theorem~\ref{T:perfectoid field}), we have distinguished equivalences of tensor categories
\[
\FEt(\Qp(\mu_{p^\infty})) \cong \FEt(F) \cong \FEt(\tilde{\bA}^\dagger_{\Qp}/(p)) \cong
\FEt(\bA^\dagger_{\Qp}/(p)) \cong \FEt(\bA^{\dagger}_{\Qp}).
\]
Via these equivalences, for any finite extension $K$ of $\Qp$, the finite extension $K \otimes_{\Qp} \Qp(\mu_{p^\infty})$ of $\Qp(\mu_{p^\infty})$ corresponds to a finite \'etale extension $\bA^{\dagger}_K$ of $\bA^{\dagger}_{\Qp}$.

For $* \in \{\bA, \bA^\dagger, \bB, \bB^\dagger, \bC, \tilde{\bA}, \tilde{\bA}^{\dagger}, \tilde{\bB}, \tilde{\bB}^\dagger, \tilde{\bC}\}$, put $*_K = *_{\Qp} \otimes_{\bA^{\dagger}_{\Qp}} \bA^{\dagger}_K$.
These rings admit extensions of the actions of $\varphi$ and $\Gamma$.
A \emph{$(\varphi, \Gamma)$-module} over one of these rings is a $\varphi$-module equipped with an action of $\Gamma$ which is continuous for the appropriate topology (the weak topology on $\bA_{K}, \tilde{\bA}_{K}, \bB_{K}, \tilde{\bB}_{K}$, or the LF-topology on the other rings).
\end{defn}

\begin{remark}
Our convention for $(\varphi, \Gamma)$-modules for $K \neq \Qp$ is not the standard one: it is more common to take $*_K$ to be a connected component of the ring we are considering, in which case one gets an action not of $\Gamma$ but only its subgroup $\Gamma_K$ corresponding to $\Gal(K(\mu_{p^\infty})/K)$ via the cyclotomic character.
However, results formulated using one convention convert easily to the other via induction and restriction between $\Gamma_K$ and $\Gamma$.
\end{remark}

\begin{theorem} \label{T:compare to classical integral}
For $K$ a finite extension of $\Qp$,
the following categories are canonically equivalent.
\begin{enumerate}
\item[(a)]
The category of continuous representations of $G_{K}$ on finite
free $\Zp$-modules.
\item[(b)]
The category of $(\varphi, \Gamma)$-modules over $\bA_{K}$.
\item[(c)]
The category of $(\varphi, \Gamma)$-modules over $\tilde{\bA}_{K}$.
\item[(d)]
The category of $(\varphi, \Gamma)$-modules over $\bA^\dagger_{K}$.
\item[(e)]
The category of $(\varphi, \Gamma)$-modules over $\tilde{\bA}^\dagger_{K}$.
\item[(f)]
The category of $\varphi$-modules over $\tilde{\bA}_{\Spa(K, \gotho_K)}$.
\item[(g)]
The category of $\varphi$-modules over $\tilde{\bA}^\dagger_{\Spa(K, \gotho_K)}$.
\end{enumerate}
\end{theorem}
\begin{proof}
The equivalences among (a) and (b)--(e) include Fontaine's original theory of $(\varphi, \Gamma)$-modules and its refinement by Cherbonnier and Colmez; see \cite[\S 2]{kedlaya-new-phigamma}.
The equivalences among (a) and (f)--(g) follow from
Theorem~\ref{T:proetale equivalence1 global}.
\end{proof}

\begin{defn}
For $* = \bB, \bB^\dagger, \bC, \tilde{\bB}, \tilde{\bB}^\dagger, \tilde{\bC}$, we say that a $(\varphi, \Gamma)$-module over $*_K$ is \emph{\'etale} if its underlying $\varphi$-module is \'etale,
i.e., it descends to a $\varphi$-module over $\bA_K, \bA^\dagger_K, \bA^\dagger_K, \tilde{\bA}_K, \tilde{\bA}^\dagger_K, \tilde{\bA}^\dagger_K$, respectively. Note that the $\Gamma$-action does act on some such descent, but not necessarily on all of them.
\end{defn}

\begin{theorem} \label{T:compare to classical rational}
For $K$ a finite extension of $\Qp$,
the following categories are canonically equivalent.
\begin{enumerate}
\item[(a)]
The category of continuous representations of $G_{K}$ on finite-dimensional $\Qp$-vector spaces.
\item[(b)]
The category of \'etale $(\varphi, \Gamma)$-modules over $\bB_{K}$.
\item[(c)]
The category of \'etale $(\varphi, \Gamma)$-modules over $\bB^\dagger_{K}$.
\item[(d)]
The category of \'etale $(\varphi, \Gamma)$-modules over $\bC_{K}$.
\item[(e)]
The category of \'etale $(\varphi, \Gamma)$-modules over $\tilde{\bB}_{K}$.
\item[(f)]
The category of \'etale $(\varphi, \Gamma)$-modules over $\tilde{\bB}^\dagger_{K}$.
\item[(g)]
The category of \'etale $(\varphi, \Gamma)$-modules over $\tilde{\bC}_{K}$.
\item[(h)]
The category of \'etale $\varphi$-modules over $\tilde{\bB}_{\Spa(K,\gotho_K)}$.
\item[(i)]
The category of \'etale $\varphi$-modules over $\tilde{\bB}^\dagger_{\Spa(K,\gotho_K)}$.
\item[(j)]
The category of \'etale $\varphi$-modules over $\tilde{\bC}_{\Spa(K,\gotho_K)}$.
\end{enumerate}
\end{theorem}
\begin{proof}
The equivalences among (a), (b)--(c), (e)--(f) is immediate from
Theorem~\ref{T:compare to classical integral}.
The equivalence between (c) and (d) follows from
Proposition~\ref{P:fully faithful1}.
The equivalence between (f) and (g) follows from
Theorem~\ref{T:perfect equivalence2a}.
The equivalences among (a) and (h)--(j) follow from
Theorem~\ref{T:proetale equivalence2 global}.
\end{proof}

We next introduce Berger's concept of a $B$-pair
from \cite{berger-b-pairs}.
\begin{defn}
Let $\Cp$ be a completed algebraic closure of $\Qp$.
For $K$ a finite extension of $\Qp$ within $\Cp$,
a \emph{$B$-pair over $K$} is a $B$-pair over $\Spa(\Cp, \gotho_{\Cp})$
equipped with a continuous semilinear action of $G_K$.
\end{defn}

\begin{theorem}\label{T:compare to classical B-pairs}
Let $F$ be the completion of $\Qp(\mu_{p^\infty})$.
For $K$ a finite extension of $\Qp$,
the following categories are canonically equivalent.
\begin{enumerate}
\item[(a)]
The category of $(\varphi, \Gamma)$-modules over $\bC_{K}$.
\item[(b)]
The category of $(\varphi, \Gamma)$-modules over $\tilde{\bC}_{K}$.
\item[(c)]
The category of $B$-pairs over $K$.
\item[(d)]
The category of continuous $\Gamma$-equivariant vector bundles over $\FFC_{\Spa(F \otimes_{\Qp} K, (F \otimes_{\Qp} K)^{\circ})}$.
\item[(e)]
The category of continuous $G_K$-equivariant vector bundles over $\FFC_{\Spa(\Cp, \gotho_{\Cp})}$.
\item[(f)]
The category of $\varphi$-modules over $\tilde{\bC}_{\Spa(K, \gotho_{K})}$.
\item[(g)]
The category of $B$-pairs over $\Spa(K, \gotho_{K})$.
\end{enumerate}
\end{theorem}
\begin{proof}
The equivalence between (a) and (c) is a theorem of Berger
\cite[Th\'eor\`eme~2.2.7]{berger-b-pairs}; the same argument gives the equivalence between (b) and (c).
The equivalences between (b) and (d) and between (c) and (e) follow from Theorem~\ref{T:vector bundles2}.
The equivalence between (b) and (f) follows from Corollary~\ref{C:qp etale on perfectoid subdomain}.
The equivalence between (f) and (g) follows from
Theorem~\ref{T:glueing flat covering}.
\end{proof}

\begin{remark}
Using Remark~\ref{R:proetale cohomology global1},
one can similarly check that our computation of pro-\'etale cohomology in terms of $\varphi$-cohomology agrees with the analogous computations in the language of classical $(\varphi, \Gamma)$-modules made by
Herr \cite{herr, herr-tate} and the second author \cite{liu-herr}.
\end{remark}

\begin{remark}
Theorems~\ref{T:compare to classical integral},
\ref{T:compare to classical rational},
and~\ref{T:compare to classical B-pairs}
 can be extended to the
relative $(\varphi, \Gamma)$-modules of Andreatta--Brinon
\cite{andreatta-brinon}. We will discuss this and related generalizations in a subsequent paper.
\end{remark}

\end{document}